\documentclass[12pt,a4paper]{article}

\usepackage{makeidx}

\makeindex
\makeatletter
\newcommand{\DefOrd}[2][**]{\DoIndex#1{#2}\textbf{#2}}
\newcommand{\emDefOrd}[2][**]{\DoIndex#1{#2}\emph{#2}}
\newcommand{\DoIndex}{%
   \@ifstar{\@ifstar\index\indexAndGobble}\DoIndexX
}    
\newcommand{\indexAndGobble}[2]{\index{#1}}
\newcommand{\DoIndexX}[1]{%
   \def\search@key{#1}%
   \@ifstar{\joinIndexArgs{}}{\joinIndexArgs\@gobble}%
}
\newcommand{\joinIndexArgs}[2]{\index{\search@key @#2}#1}
\let\gobblepage=\@firstoftwo
\makeatother

\usepackage{amsmath,amsfonts,amssymb,amscd}

\usepackage{amsthm}
\theoremstyle{plain}
\providecommand*{\theoremincountername}{section}
\newtheorem{theorem}{Theorem}[\theoremincountername]
\newtheorem{lemma}[theorem]{Lemma}
\newtheorem{corollary}[theorem]{Corollary}
\theoremstyle{definition}
\newtheorem{definition}[theorem]{Definition}
\newtheorem{construction}[theorem]{Construction}
\newtheorem{example}[theorem]{Example}
\theoremstyle{remark}
\newtheorem*{remark}{Remark}

\numberwithin{equation}{section}

\usepackage{mpembed}

\begin{mpmaterial}
  input proppy.mp
\end{mpmaterial}

\usepackage{ifpdf}
\ifpdf
\else
\fi

\makeatletter
\DeclareRobustCommand\SMC{%
  \ifx\@currsize\normalsize\small\else
   \ifx\@currsize\small\footnotesize\else
    \ifx\@currsize\footnotesize\scriptsize\else
     \ifx\@currsize\large\normalsize\else
      \ifx\@currsize\Large\large\else
       \ifx\@currsize\LARGE\Large\else
        \ifx\@currsize\scriptsize\tiny\else
         \ifx\@currsize\tiny\tiny\else
          \ifx\@currsize\huge\LARGE\else
           \ifx\@currsize\Huge\huge\else
            \small\SMC@unknown@warning
 \fi\fi\fi\fi\fi\fi\fi\fi\fi\fi
}
\newcommand{\SMC@unknown@warning}{\TBWarning{\string\SMC: unrecognised
    text font size command -- using \string\small}}
\makeatother
\DeclareTextFontCommand{\smaller}{\SMC}
\DeclareTextFontCommand{\smalltexttt}{\SMC\ttfamily}

\newcommand{\PROP}{\texorpdfstring{\smaller{PROP}}{PROP}}

\newcommand{\mc}{\mathcal}
\newcommand{\vek}{\mathbf}

\newcommand{\cross}[2]{{}^{#1}\mathsf{X}^{#2}}
\newcommand{\same}[1]{\mathsf{I}^{#1}}

\newcommand*{\feedback}[1]{{\uparrow}^{#1}}
\newcommand*{\rightfeedback}[1]{\right\uparrow^{#1}}
\newcommand*{\join}[1][]{%
   \vphantom{\stackrel{#1}{\Join}}%
   \mathbin{\smash{\stackrel{#1}{\Join}}}%
}

\newcommand{\mOp}{\mu} 
\mpvarmaterial{verbatimtex \def\protect\mOp{\mOp} etex}
\begin{mpmaterial}
  verbatimtex
    \ifx\varkappa\undefined
       \font\tenmsb=msbm10
       \font\sevenmsb=msbm7
       \font\fivemsb=msbm5
       \textfont"F=\tenmsb
       \scriptfont"F=\sevenmsb
       \scriptscriptfont"F=\fivemsb
       \mathchardef\varkappa="0F7B
    \fi
    \ifx\eval\undefined
       \def\eval{{\fam=0 eval}}
    \fi
  etex
\end{mpmaterial}

\makeatletter
\newcommand*{\Norm}[2][\@gobble]{\left\|#1. #2 \right\|}
\newcommand*{\norm}[2][\@gobble]{\left|#1. #2 \right|}
\newcommand*{\setOf}[3][\@gobble]{%
   \left\{ \, #2 \,\,\vrule\relax#1.\,\@displayfalse\, #3 \, \right\}%
}
\makeatother

\newcommand*{\fuse}[4][\big]{%
   \mathopen{#1[} {}_{#2}^{#4} \,#1\vert\, {#3} \mathclose{#1]}%
}
\newcommand*{\Nwfuse}[4][\big]{%
   \Nw\mathopen{#1(} {}_{#2}^{#4} \,#1\vert\, {#3} \mathclose{#1)}%
}

\newcommand{\restr}[2]{#1\rvert_{#2}}
\newcommand{\card}{\norm}

\newcommand{\psm}[1]{%
  \left(%
    \begin{smallmatrix}%
      #1%
    \end{smallmatrix}%
  \right)%
}
\newcommand{\transpose}[1]{#1^{\mathrm{T}}}

\newcommand{\setmap}[1]{\wp(#1)}
\newcommand{\setinv}[1]{\overline{\wp}(#1)}
\DeclareMathOperator{\setim}{im}

\newcommand{\Epil}{\quad\Longleftrightarrow\quad}
\newcommand{\Fpil}{\longrightarrow}
\newcommand{\Ipil}{\quad\Longrightarrow\quad}

\newcommand{\B}{\mathbb{B}}
\newcommand{\C}{\mathbb{C}}
\newcommand{\N}{\mathbb{N}}
\newcommand{\Q}{\mathbb{Q}}
\newcommand{\R}{\mathbb{R}}

\newcommand{\Z}{\mathbb{Z}}
\newcommand{\Zp}{\Z_+}

\newcommand{\ve}{\varepsilon}

\newcommand*{\pin}[1]{
   \mathchoice{%
      \mathrel{\mathrm{in}}%
   }{%
      \mathrel{\mathrm{in}}%
   }{%
      \mathop{\mathrm{in}}%
   }{%
      \mathop{\mathrm{in}}%
   }#1%
}

\DeclareMathOperator{\eval}{eval}
\DeclareMathOperator{\tr}{tr}
\DeclareMathOperator{\Tr}{Trf}
\DeclareMathOperator{\Span}{Span}

\DeclareMathOperator{\Ker}{Ker}

\newcommand{\Hopf}{\mathcal{H}\mkern-2mu\mathit{opf}}
\newcommand{\HomPROP}{\mathcal{H}\mkern-2mu\mathit{om}}
\newcommand{\Baff}{\mathrm{Baff}}

\newcommand{\Nw}{\mathrm{Nw}}
\newcommand{\Nwt}{\widetilde{%
   \vphantom{\mathchoice{N}{\vrule width0pt height1.3ex}{N}{N}}%
   \smash{\Nw}%
}}

\newcommand{\Red}{\mathrm{Red}}
\newcommand{\Irr}{\mathrm{Irr}}
\newcommand{\DSM}{\mathrm{DSM}}
\newcommand{\DIS}{\mathrm{DIS}}

\newcommand{\canmake}{\curvearrowright} 

\providecommand{\PROP}{PROP}
\providecommand{\PROPs}{\PROP s}
\providecommand*{\DefOrd}[2][]{\textbf{#2}}
\providecommand*{\emDefOrd}[2][]{\emph{#2}}

\providecommand*{\Dash}{%
   \nobreak\hspace{0.166667em}\textemdash\hspace{0.166667em}%
}
\providecommand*{\Ldash}{%
   \hspace{0.166667em}\textemdash\nobreak\hspace{0.166667em}%
}
\providecommand*{\Rdash}{\Dash}

\newcommand*{\parenthetic}[1]{\/\textup{(#1)}}
\providecommand{\textprime}{\('\)}

\usepackage[%
   pdfborder={0 0 0},%
   pdfstartview={XYZ null null null}%
]{hyperref}
\providecommand*{\texorpdfstring}[2]{#1}

\begin{document}

\title{Network Rewriting I: The Foundation}
\author{Lars Hellstr\"om\footnote{
  Institutionen f\"or matematik och matematisk statistik,
  Ume\r{a} University,
  901\,87 Ume\r{a}, SWEDEN.
  E-mail: \texttt{lars.hellstrom}(at)\texttt{residenset.net}
}}
\maketitle

\begin{abstract}
  A theory is developed which uses ``networks'' (directed acyclic 
  graphs with some extra structure) as a formalism for expressions in 
  multilinear algebra. It is shown that this formalism is valid for 
  arbitrary \PROPs~(short for `PROducts and Permutations category'), 
  and conversely that the \PROP\ axioms are implicit in the concept 
  of evaluating a network. Ordinary terms and operads constitute the 
  special case that the graph underlying the network is a rooted tree.
  
  Furthermore a rewriting theory for networks is developed. Included 
  in this is a subexpression concept for which is given both 
  algebraic and effective graph-theoretical characterisations, a 
  construction of reduction maps from rewriting systems, and an 
  analysis of the obstructions to confluence that can occur. Several 
  Diamond Lemmas for this rewriting theory are given.
  
  In addition there is much supporting material on various related 
  subjects. In particular there is a toolbox for the construction of 
  custom orders on the free \PROP, so that an order can be tailored 
  to suit a specific rewriting system. Other subjects treated are the 
  abstract index notation in a general \PROP\ context and the use of 
  feedbacks (sometimes called ``traces'') in \PROPs.
%
\end{abstract}

\section{Introduction}

The `algebra generated by [certain elements] which 
satisfy [some relations]' is a staple concept in 
mathematics and perhaps \emph{the} key concept in universal algebra. 
Like most fruitful notions in mathematics, it can be defined from a 
wide variety of perspectives, including the classically algebraic 
(quotient ring by ideal), the abstract categorical (free object in 
suitable category), the universal algebraic (quotient, set of terms 
by congruence relation), and the rewriting (terms transformed according 
to rules) perspectives. This text will mainly focus on the latter two, 
but the reader who so wishes should have no problem to view the results 
from any of these perspectives.


An idea that is common to the universal algebraic and rewriting 
perspectives is that of the \emph{term}, or (less technically) the 
``general expression'', which is viewed and studied as a mathematical 
object in its own right. Most formalisations of this concept equips 
it with an underlying tree structure, where the results from 
subexpressions become arguments to the function in the root node 
for the expression as a whole; for example in `$f\bigl( 
g(x), 2, h(x,y) \bigr)$' the subexpressions `$g(x)$', `$2$', and 
`$h(x,y)$' serve as arguments of the function $f$ which constitutes 
the root node of the expression tree. Syntactic tree structures, with 
context-free languages and BNF grammars as their perhaps most 
iconic exponents, are so prevalent within computer science that it is 
hard to imagine how things could be otherwise. And yet, there are 
algebraic structures for which trees \emph{simply are not general 
enough!}

Examples of such structures have begun to accumulate over the last 
couple of decades\Ldash two that should definitely be mentioned are 
quantum logical circuits and Hopf algebras\Rdash but in order to see 
what it is about them that cannot be captured in terms of tree-like 
expressions, one really has to sit down and work through the gory 
details of some examples. In many cases the crux of the 
matter that makes tree-like expressions insufficient is that the 
tensor product $U \otimes V$ of two vector spaces $U$ and $V$ is not 
the same thing as the cartesian product $U \times V$.\footnote{
  This may seem as a trivial observation\Ldash they are two different 
  products of vector spaces, so of course they ought to be 
  different\Rdash but the catch is that when they serve as the domain 
  of a multilinear map they are not so different: $\hom(U 
  \otimes\nobreak V, W)$ is pretty much by definition isomorphic to 
  $\hom(U \times\nobreak V, W)$. Hence when elements are paired only 
  as a preparation for feeding them into a function, it need not be 
  clear which concept of pairing that is the most natural.
} To illustrate the differences, one may consider creating pairs of 
vectors in the standard basis 
$\{\vek{e}_1,\vek{e}_2,\vek{e}_3,\vek{e}_4\}$ of $\R^4$. A cartesian 
way of pairing two vectors would be to put them as separate columns 
of a matrix, whereas the tensor way of pairing two column vectors 
$\vek{u}$ and $\vek{v}$ would be to form the matrix-valued product 
$\vek{u} \transpose{\vek{v}}$\Dash hence the pairings of $\vek{e}_1$ 
and $\vek{e}_2$ come out as
\[
  \vek{e}_1 \times \vek{e}_2 =
  \bigl[ \vek{e}_1, \vek{e}_2 \bigr] = 
  \begin{pmatrix} 1 & 0 \\ 0 & 1 \\ 0 & 0 \\ 0 & 0 \end{pmatrix}
  \quad\text{and}\quad
  \vek{e}_1 \otimes \vek{e}_2 =
  \vek{e}_1 \transpose{\vek{e}_2} =
  \begin{pmatrix}
    0 & 1 & 0 & 0 \\
    0 & 0 & 0 & 0 \\
    0 & 0 & 0 & 0 \\
    0 & 0 & 0 & 0 
  \end{pmatrix}
\]
respectively. These two modes of pairing behave very differently 
with respect to superposition. In the cartesian case
\begin{multline*}
  \bigl[ \vek{e}_1, \vek{e}_2 \bigr] + 
  \bigl[ \vek{e}_3, \vek{e}_4 \bigr]
  = \\ =
  \begin{pmatrix} 1 & 0 \\ 0 & 1 \\ 0 & 0 \\ 0 & 0 \end{pmatrix} +
  \begin{pmatrix} 0 & 0 \\ 0 & 0 \\ 1 & 0 \\ 0 & 1 \end{pmatrix}
  =
  \begin{pmatrix} 1 & 0 \\ 0 & 1 \\ 1 & 0 \\ 0 & 1 \end{pmatrix}
  =
  \begin{pmatrix} 1 & 0 \\ 0 & 0 \\ 0 & 0 \\ 0 & 1 \end{pmatrix} +
  \begin{pmatrix} 0 & 0 \\ 0 & 1 \\ 1 & 0 \\ 0 & 0 \end{pmatrix}
  = \\ =
  \bigl[ \vek{e}_1, \vek{e}_4 \bigr] + 
  \bigl[ \vek{e}_3, \vek{e}_2 \bigr]
  \text{,}
\end{multline*}
so there is no way of knowing from the superposition whether it was 
$\vek{e}_1$ or $\vek{e}_3$ that was paired with $\vek{e}_2$. However, 
in the tensor case
\begin{multline*}
  \vek{e}_1 \transpose{\vek{e}_2} + \vek{e}_3 \transpose{\vek{e}_4}
  =
  \begin{pmatrix}
    0 & 1 & 0 & 0 \\
    0 & 0 & 0 & 0 \\
    0 & 0 & 0 & 0 \\
    0 & 0 & 0 & 0 
  \end{pmatrix} +
  \begin{pmatrix}
    0 & 0 & 0 & 0 \\
    0 & 0 & 0 & 0 \\
    0 & 0 & 0 & 1 \\
    0 & 0 & 0 & 0 
  \end{pmatrix}
  =
  \begin{pmatrix}
    0 & 1 & 0 & 0 \\
    0 & 0 & 0 & 0 \\
    0 & 0 & 0 & 1 \\
    0 & 0 & 0 & 0 
  \end{pmatrix}
  \neq \\ \neq
  \begin{pmatrix}
    0 & 0 & 0 & 1 \\
    0 & 0 & 0 & 0 \\
    0 & 1 & 0 & 0 \\
    0 & 0 & 0 & 0 
  \end{pmatrix}
  =
  \begin{pmatrix}
    0 & 0 & 0 & 1 \\
    0 & 0 & 0 & 0 \\
    0 & 0 & 0 & 0 \\
    0 & 0 & 0 & 0 
  \end{pmatrix} +
  \begin{pmatrix}
    0 & 0 & 0 & 0 \\
    0 & 0 & 0 & 0 \\
    0 & 1 & 0 & 0 \\
    0 & 0 & 0 & 0 
  \end{pmatrix}
  =
  \vek{e}_1 \transpose{\vek{e}_4} + \vek{e}_3 \transpose{\vek{e}_2}
  \text{!}
\end{multline*}
On the other hand, it is in a cartesian product possible to attach 
separate amplitudes $r$ and $s$ to the components of a pair\Ldash
\[
  \bigl[ r\vek{e}_1, s\vek{e}_2 \bigr] 
  = 
  \begin{pmatrix} r & 0 \\ 0 & s \\ 0 & 0 \\ 0 & 0 \end{pmatrix}
  \neq
  \begin{pmatrix} s & 0 \\ 0 & r \\ 0 & 0 \\ 0 & 0 \end{pmatrix}
  =
  \bigl[ s\vek{e}_1, r\vek{e}_2 \bigr] 
\]
\Rdash but not so in a tensor product, as
\[
  (r \vek{e}_1) \transpose{( s \vek{e}_2 )}
  =
  \begin{pmatrix}
    0 & rs & 0 & 0 \\
    0 & 0 & 0 & 0 \\
    0 & 0 & 0 & 0 \\
    0 & 0 & 0 & 0 
  \end{pmatrix} 
  =
  (s \vek{e}_1) \transpose{( r \vek{e}_2 )}
  \text{.}
\]
It is not always the case that one can take a superposition of 
tensor products apart\Ldash just like it can happen that one cannot 
tell from a product of two elements in a noncommutative ring whether 
that product was $ab$ or $ba$ when $a$ and $b$ happen to 
commute\Rdash but many interesting structures exploit the possibility 
of this and other phenomena that are not seen in cartesian products.

Wrapping one's head around what this implies for the syntax of 
expressions can however be mind-boggling at first.
Sticking with the expression tree terminology, where it is the 
``children'' of a function node that provide input to it, whereas 
output is passed on to the ``parent'', it is clear that in order for 
the expression to not simply become a tree, there must be nodes with 
more than one parent\Dash say for simplicity that there is one type 
of function node which has two parents: one slightly to the left and 
the other slightly to the right. This means the function produces two 
results, one of which (the left) is sent to one parent, and the other 
of which (the right) is sent to the other. 
\begin{figure}
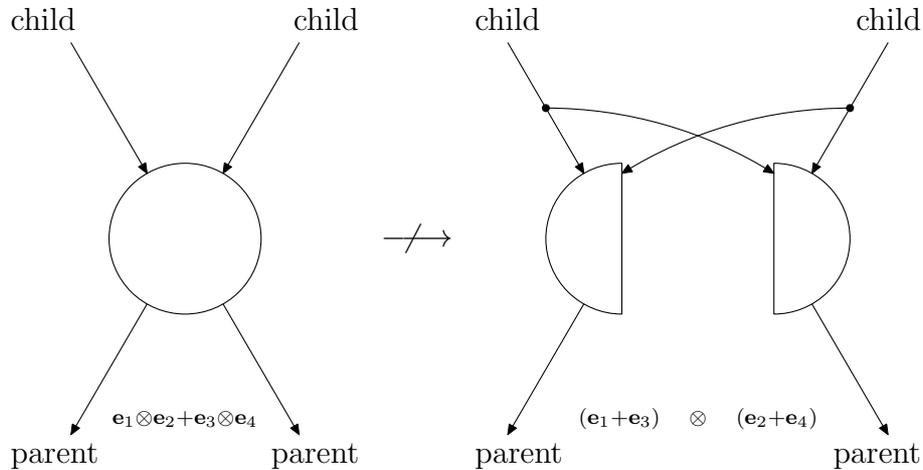

  \[
    \begin{array}{c}
      \text{child}\hspace{0pt plus 1filll}\text{child}
      \\
      \quad
      \begin{mpgraphics*}{-94}
        u := 1cm;
        draw fullcircle scaled 2u;
        drawarrow 3u*dir 120 -- u*dir 120;
        drawarrow 3u*dir 60 -- u*dir 60;
        drawarrow u*dir -120 -- 3u*dir -120;
        drawarrow u*dir -60 -- 3u*dir -60;
      \end{mpgraphics*}
      \quad\\
      \text{parent}\hspace{0pt plus 1filll}\smash{
        \begin{array}[b]{c}
          \scriptstyle
          \vek{e}_1 \otimes \vek{e}_2 + 
            \vek{e}_3 \otimes \vek{e}_4\\
          \quad
        \end{array}
      }\hspace{0pt plus 1filll}\text{parent}
    \end{array}
    \relbar\joinrel
    \not\longrightarrow
    \begin{array}{c}
      \text{child}\hspace{0pt plus 1filll}\text{child}
      \\
      \quad
      \begin{mpgraphics*}{-95}
        u := 1cm;
        pair d; d = (2u,0);
        draw (halfcircle -- cycle) rotated 90 scaled 2u;
        draw (halfcircle -- cycle) rotated -90 scaled 2u shifted d;
        drawarrow 3u*dir 120 -- u*dir 120;
        drawarrow (3u*dir 60 -- u*dir 60) shifted d;
        drawarrow u*dir -120 -- 3u*dir -120;
        drawarrow (u*dir -60 -- 3u*dir -60) shifted d;
        fill fullcircle scaled 3pt shifted (2u*dir120);
        drawarrow 2u*dir120 {right} .. (d + (0,u*sind 60));
        fill fullcircle scaled 3pt shifted (2u*dir60 + d);
        drawarrow (2u*dir60 + d) {left} .. (0,u*sind 60);
      \end{mpgraphics*}
      \quad\\
      \text{parent}\hspace{0pt plus 1filll}\smash{
        \begin{array}[b]{c}
          \scriptstyle
          (\vek{e}_1 + \vek{e}_3)\\
          \,
        \end{array}
      }\hspace{0pt plus 1filll}\smash{
        \begin{array}[b]{c}
          \scriptstyle \otimes \\
          \,
        \end{array}
      }\hspace{0pt plus 1filll}\smash{
        \begin{array}[b]{c}
          \scriptstyle
          (\vek{e}_2 + \vek{e}_4)\\
          \,
        \end{array}
      }\hspace{0pt plus 1filll}\text{parent}
    \end{array}
  \]
  \caption{A whole need not reduce to two halves}
\end{figure}
The tree-minded response to this is often ``Well, then split that 
function into two: one part which produces the left result and 
another which produces the right; then you can put the expression on 
tree form.'' The catch is, that this only works if the pair of 
results is viewed as living in the cartesian product of the ranges 
for the left and right respectively results; an element of a cartesian 
product is uniquely determined by what its first and second parts 
are, and this is the property that characterises cartesian products. 
Tensor products, as demonstrated above, permit values that constitute 
an entanglement of the left and right parts; much information lies in 
how values in one part corresponds to values in the other. Therefore 
both parts have to be \emph{created} as a unit, even though 
there is then no need to \emph{use} them as a unit; an entanglement 
will simply be passed on to the next generation of results.
(Yes, it seems odd, but it actually works that way.)

When venturing to explore unfamiliar mathematical phenomena, it can 
be a great help to employ a notation system that is equipped to deal 
with the oddities that lie ahead; the right notation system for the 
task may even serve as one's guide, in that it makes it easy to take 
the steps that are relevant for one's goals, as those are all steps 
that come naturally in the good notations. Conjuring up a notation 
system is however not a trivial task, as minor notational innovations 
can have major technical implications and an unfortunate choice can 
seriously restrict the applicability of one's labours. Therefore the 
first part of this text is spent on \emph{formally} developing not 
just one, but two notation systems (although they are closely related) 
for the kind of expressions that go beyond mere trees. It is only when 
a concept of expression has been fixed that rewriting can start in 
earnest.

\subsection{Structure of this text}

The contents of this text partition roughly into three themes. The 
first theme, which is the subject of Sections~\ref{Sec:PROPpar}, 
\ref{Sec:AIN}, and~\ref{Sec:Natverk}, is to establish the concept of 
\PROP\ and various expression notations for it. The second theme, 
which is the subject of Sections~\ref{Sec:Deluttryck}, 
\ref{Sec:FriPROP}, and~\ref{Sec:Omskrivning}, is a theory of 
rewriting in \PROPs. The third theme, which by elimination is the 
subject of Sections~\ref{Sec:Ordning}, \ref{Sec:Matrisordning}, 
and~\ref{Sec:Feedbacks}, is the construction of custom orders on 
\PROPs, to the end of providing consistent methods for orienting 
one's rewrite rules. Originally there was also going to be the fourth 
theme of applying the other three to establish a confluent and 
terminating set of rewrite rules for the $\Hopf$ \PROP\ (i.e., the 
\PROP\ encoding the operations and axioms of a Hopf algebra), but the 
foundation consisting of the first three turned out to be so voluminous 
that this latter theme was better split off; remnants of it do 
however show up in examples and the like.

Within the first theme, Section~\ref{Sec:PROPpar} introduces the 
concept of \PROP\ using traditional mathematical notation and gives a 
variety of examples of the same. Section~\ref{Sec:AIN} sets up the 
alternative \emph{Abstract Index Notation}, which is a derivative of 
the Einstein notation (also known as the Einstein summation convention) 
for tensors, and proves that it is well-defined for arbitrary \PROPs. 
Section~\ref{Sec:Natverk} sets up graphical \emph{networks} as a 
third notation for \PROP\ expressions, and proves that it is 
\emph{canonical} in the sense that (i)~a mathematical structure 
supports this notation if \emph{and only if} it is a \PROP, and 
(ii)~the equality modulo \PROP\ axioms of two network notation 
expressions is strictly a matter of network (essentially graph) 
isomorphism.

The degree of novelty of the material within this first theme varies 
from item to item, and is also a matter very much in the eye of the 
beholder. The examples in Section~\ref{Sec:PROPpar}, save perhaps 
that of the biaffine \PROP, should all be standard (even though I 
suspect they may be unusually concrete and elementary for examples 
in this area of mathematics), but when it comes to Sections~\ref{Sec:AIN} 
and~\ref{Sec:Natverk} the results are rather within the gray area of 
things that everyone sufficiently versed in the field quickly 
realises are true, but which it is less clear whether anyone has 
actually  
taken the time to write out in detail before. The abstract index 
notation\Ldash the idea that Einstein's summation convention has a 
coordinate-free interpretation\Rdash is definitely not new, but 
formalisations of it tend to rely on trace maps (also known as 
feedbacks; see Section~\ref{Sec:Feedbacks}), which reduces the range 
of their formal applicability to \PROPs\ with feedbacks. Similarly 
the network notation has many precedents, which on closer examination 
may turn out to be something which formally is slightly different. 
Several authors have employed ``shorthand diagrams'' 
(e.g.~\cite{Majid}) that more or less conform to the network 
notation used here, but perhaps rather meant that the \emph{actual} 
proof is what one gets by transcribing these diagrams into more 
traditional notation, than claiming 
that already the diagrams as such constitute a rigorous proof. Some 
authors~\cite{Penrose,Stefanescu} treat the diagrams as rigorous 
mathematical expressions, but require feedbacks for their 
interpretations of them. Finally it may be argued that the networks 
are merely a concrete realisation of the free \PROP, and a free 
object may more or less by definition be used to encode arbitrary 
expressions within the category in which it is free. It should 
however be observed that the networks used here are discrete 
mathematical objects immediately encodable in a computer, whereas 
traditionally realisations of the free \PROP~(e.g.~\cite{JoyalStreet}) 
rather tend to be topological and rely heavily on properties of 
continuous space. While that may be advantageous for visualisation 
and can at times help to prove specific results, it seems highly 
undesirable for a piece of formal mathematical notation.

The formal construction of the free \PROP\ is here the subject of 
Section~\ref{Sec:FriPROP}; this includes not only the construction as 
such, but also results establishing various additional structures 
within the free \PROP\Dash structures which the rewriting formalism 
makes heavy use of. In particular it defines the operation of 
symmetric join $\join$ of elements of the free \PROP; despite 
operating in a setting which is more general than associative 
polynomials, this operation manages to unify things in a manner which 
is reminiscent of what one sees in the theory for commutative 
polynomials!

The origin for the additional structures defined in 
Section~\ref{Sec:FriPROP} is the subexpression concept established in 
Section~\ref{Sec:Deluttryck}. It is based on abstract index notation 
rather than the traditional notation, and it is more powerful than 
the ``convex'' (cf.~Definition~\ref{Def:KonvextDeluttryck}) 
subexpression concept that one might get from the traditional 
notation in that it can make do with infinitely fewer rewrite rules 
(it may take infinitely many rules with convex left hand sides to 
express what one rule with nonconvex left hand side can do). The fact 
that the chosen subexpression concept has both algebraic and 
graph-theoretical characterisations makes it practical both for 
theory and computer implementation~\cite{cmplutil}.

Section~\ref{Sec:Omskrivning} is where the actual rewriting formalism 
is set up. This is an application of the multi-sorted generic 
framework of~\cite{rpaper}, but rather than following the obvious 
approach of having one sort per \PROP\ component, this set-up takes 
the more refined approach of having a separate sort for each 
transference type. This allows for rules that replace nonconvex 
subexpressions, and may for derived rules more accurately capture the 
requirements of their derivations. Several Diamond Lemmas for 
$\mc{R}$-linear \PROPs\ (and operads) are given, in particular 
Theorem~\ref{S:PROP-skarp-DL} that is directly applicable for the 
``missing fourth theme'' of describing the $\Hopf$ \PROP, since it 
comes with detailed account of what ambiguities need to be explicitly 
resolved.

Still, it must be stated up front that the rewriting formalism of 
Section~\ref{Sec:Omskrivning} may still be a bit too coarse to really 
support a fully general Diamond Lemma for $\mc{R}$-linear \PROPs; 
rather than being the finished product, it should be taken as a first 
working prototype. Limitations are encountered in the analysis of 
minimal sets of ambiguities (``critical pairs''), in that it is 
sometimes not possible to simplify an ambiguity by peeling away parts 
of it which are not in the left hand side of either rule. Even though 
conditions can be given (namely, that the rules are \emph{sharp}) under 
which all ambiguities formed by a pair of rules can be simplified to 
one from a small finite set that can be constructed through a 
straightforward combinatorial search, not all rewriting systems of 
interest satisfy this condition. The detailed analysis given of how 
things can go wrong suggests that there really are new phenomena 
that arise in the \PROP\ setting, and that the classical 
classification of ambiguities as being either inclusions or overlaps 
may have to be extended with a new class that are neither.

In the third theme, we may again encounter the gray area of things 
possibly known but probably not written down; the construction of 
strict orders on \PROPs\ turns out to be surprisingly difficult once 
one goes beyond the basic `compare counts of vertices of each type' 
and seeks to find something that takes into account the 
structure of the expression. Section~\ref{Sec:Ordning} sets up some 
basic theory and provides the connectivity 
\PROP~(Example~\ref{Ex:Sammanhangande}) as one fundamental example, 
whereas Section~\ref{Sec:Matrisordning} focuses on orders obtained by 
comparing matrices. With the notable exception of the connectivity 
\PROP, pretty much every order I have made use of has turned out to 
be a special case of the standard order on some biaffine \PROP~(see 
Corollary~\ref{Kor:BiaffinOrdning}), which can be both a blessing (it 
is very versatile) and a curse (if this seems reluctant to give 
direction to a rewrite rule, then we're pretty much out of options). 
Research to the end of constructing further examples of strict \PROP\ 
orders would be greatly appreciated.

Section~\ref{Sec:Feedbacks} is rather about providing tools to 
demonstrate that specific orders on \PROPs\ are compatible with the 
rewriting framework. To that end, it is convenient to introduce the 
concept of (formal) feedbacks on \PROPs, since these are on one hand 
related to the symmetric join operation, and can on the other hand be 
shown to preserve strict inequalities under the known strict \PROP\ 
orders.

%

\subsection{Preview of results on \texorpdfstring{$\Hopf$}{Hopf}}

The next couple of pages make up a sort of preview of the main results 
of \emph{Network Rewriting II}, the missing ``fourth theme'' of the 
present text. It is included primarily as a motivation for the 
foundational material that follows, to offer some problem which 
really exercises the various parts of the machinery that is being 
built, for the reader to better comprehend why there is a point in 
doing these things. Many parts of this text can of course be appreciated 
as interesting pieces of mathematics in their own right, but what 
motivates collecting \emph{these particular} pieces together is that 
they combine into a machinery for doing rewriting, and for 
appreciating \emph{that} it may be of useful to actually have a 
problem to which that machinery may be applied.

For any Hopf algebra $\mc{H}$ over a field $\mc{K}$, the five basic 
operations (and co-operations) are\Ldash for some pair $(m,n)$ of 
natural numbers\Dash elements in the set $\hom_\mc{K}(\mc{H}^{\otimes n}, 
\mc{H}^{\otimes m})$ of $\mc{K}$-linear maps from $\mc{H}^{\otimes n}$ 
to $\mc{H}^{\otimes m}$. Namely:
\begin{center}
  \begin{tabular}{r c lc}
    Symbol& Signature& Name& $(m,n)$\\
    $\mOp$& \(\mc{H} \otimes \mc{H} \Fpil \mc{H}\)&
    multiplication& $(1,2)$\\
    $\eta$& \(\mc{K} \Fpil \mc{H}\)& unit& $(1,0)$\\
    $\Delta$& \(\mc{H} \Fpil \mc{H} \otimes \mc{H}\)&
    coproduct& $(2,1)$\\
    $\ve$& \(\mc{H} \Fpil \mc{K}\)& counit& $(0,1)$\\
    $S$& \(\mc{H} \Fpil \mc{H}\)& antipode& $(1,1)$
  \end{tabular}
\end{center}
The entire family 
$\bigl\{ \hom_\mc{K}(\mc{H}^{\otimes n}, \mc{H}^{\otimes m}) 
\bigr\}_{m,n\in\N}$ of maps make up an algebraic structure known as a 
\PROP~(with composition of maps, tensor product of maps, and 
permutation of factors as operations), and the five 
given elements generate a sub-\PROP\ thereof, here denoted 
$\Hopf(\mc{H})$. It may be observed that all Hopf algebra 
axioms can be stated as identities in $\hom_\mc{K}(\mc{H}^{\otimes n}, 
\mc{H}^{\otimes m})$ for suitable $m$ and $n$, for example the 
antipode axioms are that \(\mOp \circ (S \otimes\nobreak \mathrm{id}) 
\circ \Delta = \eta \circ \ve = \mOp \circ (\mathrm{id} 
\otimes\nobreak S) \circ \Delta\) in $\hom_\mc{K}(\mc{H},\mc{H})$, 
and this makes it possible to find all consequences of these axioms.

As is generally the case in universal algebra, there is a free 
$\mc{K}$-linear \PROP\ $\mc{K}\{\Omega\}$ of which every 
$\Hopf(\mc{H})$ is a homomorphic image (the particular signature 
$\Omega$ consists of the symbols $\mOp$, $\eta$, $\Delta$, $\ve$, $S$ 
having coarities $1$, $1$, $2$, $0$, $1$ and arities $2$, $0$, 
$1$, $1$, $1$ respectively). There is for every Hopf algebra $\mc{H}$ 
a unique \PROP~homomorphism \(f_{\mc{H}}\colon \mc{K}\{\Omega\} \Fpil 
\Hopf(\mc{H})\) mapping the elements of $\Omega$ onto the generators 
of $\Hopf(\mc{H})$, and the kernel of this map defines a congruence 
relation $C_\mc{H}$ on $\mc{K}\{\Omega\}$. Exactly which elements 
this relation considers to be congruent depends on $\mc{H}$, but it 
is known that it at least considers the left hand sides of all Hopf 
algebra axioms congruent to their respective right hand sides, e.g. 
\(\mOp \circ (S \otimes\nobreak \mathrm{id}) \circ \Delta \equiv 
\eta \circ \ve \pmod{C_\mc{H}}\). Since there for every set of such 
identities is a minimal congruence which satisfies them all, there is 
in particular a congruence $C_{\mathrm{Hopf}}$ on $\mc{K}\{\Omega\}$ 
which is generated by only the Hopf axioms, and hence 
\(C_{\mathrm{Hopf}} \subseteq C_{\mc{H}}\) for every Hopf algebra 
$\mc{H}$. It follows that every $f_\mc{H}$ splits over the 
\PROP~\(\Hopf := \mc{K}\{\Omega\} \big/ C_{\mathrm{Hopf}}\), so in 
particular every $\Hopf(\mc{H})$ is an image of $\Hopf$. The main aim 
of \emph{Network Rewriting II} will be to describe the structure of 
this general $\Hopf$ \PROP, and in particular provide an algorithm 
for deciding equality in $\Hopf$.

Technically this is done by turning the Hopf algebra axioms into 
rewrite rules, and completing the corresponding rewrite system. The 
suitable instance (Theorem~\ref{S:PROP-skarp-DL}) of the Diamond Lemma then 
gives an isomorphism between $\Hopf$ and a vector subspace of 
$\mc{K}\{\Omega\}$, and a projection of $\mc{K}\{\Omega\}$ onto this 
subspace which maps elements to the same thing if and only if they are 
congruent modulo $C_{\mathrm{Hopf}}$. Since this projection is 
computationally effective, an equality algorithm is merely to apply it 
to the difference between the two elements to compare, and check 
whether the projection maps this difference to zero.

The completed rewrite system consists of $16$ infinite families of 
rules and $14$ isolated rules, stated here in abstract index 
notation (see Section~\ref{Sec:AIN} and page~\pageref{Sid:AIN-regel}). 
The $14$ isolated rules are:
\begin{subequations} \label{Eq:SporadiskaRegler}
  \begin{align}
    \label{Rule:mu-eta-left}
    \mOp^a_{bc} \, \eta^b &{}\longmapsto \delta^a_c
    \text{,}\\
    \label{Rule:mu-eta-right}
    \mOp^a_{bc} \, \eta^c &{}\longmapsto \delta^a_b
    \text{,}\\
    \label{Rule:mu-assoc}
    \mOp^a_{bc} \, \mOp^c_{de} &{}\longmapsto
    \mOp^a_{ce} \, \mOp^c_{bd}
    \text{,}\displaybreak[0]\\
    \label{Rule:Delta-ve-left}
    \ve_a \, \Delta^{ab}_c &{}\longmapsto \delta^b_c
    \text{,}\\
    \label{Rule:Delta-ve-right}
    \ve_b \, \Delta^{ab}_c &{}\longmapsto \delta^a_c
    \text{,}\\
    \label{Rule:Delta-coassoc}
    \Delta^{bc}_d \, \Delta^{ad}_e &{}\longmapsto
    \Delta^{ab}_d \, \Delta^{dc}_e
    \text{,}\displaybreak[0]\\
    \label{Rule:Delta-eta}
    \Delta^{ab}_c \, \eta^c &{}\longmapsto \eta^a \, \eta^b
    \text{,}\\
    \label{Rule:ve-eta}
    \ve_a \, \eta^a &{}\longmapsto 1
    \text{,}\\
    \label{Rule:ve-mu}
    \ve_a \, \mOp^a_{bc} &{}\longmapsto \ve_b \, \ve_c
    \text{,}\\
    \label{Rule:Delta-mu}
    \Delta^{ab}_c \, \mOp^c_{gh} &{}\longmapsto
    \mOp^a_{cd} \, \mOp^b_{ef} \, \Delta^{ce}_g \, \Delta^{df}_h
    \text{,}\displaybreak[0]\\
    \label{Rule:S-eta}
    S^a_b \, \eta^b &{}\longmapsto \eta^a
    \text{,}\\
    \label{Rule:S-mu}
    S^a_b \, \mOp^b_{de} &{}\longmapsto
    \mOp^a_{bc} \, S^b_e \, S^c_d
    \text{,}\\
    \label{Rule:ve-S}
    \ve_a \, S^a_b &{}\longmapsto \ve_b
    \text{,}\\
    \label{Rule:Delta-S}
    \Delta^{ab}_c \, S^c_e &{}\longmapsto
    S^a_c \, S^b_d \, \Delta^{dc}_e
    \text{;}
  \end{align}
\end{subequations}
the first $10$ of which are the axioms for a bialgebra\Dash $3$ 
axioms for a unital associative algebra, $3$ axioms for a counital 
coassociative coalgebra, and $4$ axioms stating the compatibility of 
the algebra and coalgebra structures. The last four rules encode 
the facts that the antipode $S$ is an algebra and coalgebra 
antihomomorphism; these are usually not 
taken as axioms, since they are implied by the ``formal group 
inverse'' axioms \(\mOp^a_{bc} \, S^b_d \, \Delta^{dc}_e = 
\eta^a \, \ve_e = \mOp^a_{bc} \, S^c_d \, \Delta^{bd}_e\) for the 
antipode, which are the \(N=0\) instances of \eqref{Rule:1-1-0-1} and 
\eqref{Rule:1-1-1-0}.

The set \eqref{Eq:SporadiskaRegler} of rules constitute a complete 
rewrite system which, although weaker than the full set of Hopf 
algebra axioms, still suffice for transforming any monomial element of 
$\mc{K}\{\Omega\}$ to the normal form \(A \circ B \circ C\), where $A$ 
is constructed from the algebra 
operations $\mOp$ and $\eta$, $C$ is constructed from the coalgebra 
cooperations $\Delta$ and $\ve$, and $B$ is constructed from 
(permutations and) the antipode $S$. The remaining $16$ families of 
rules all remove certain patterns that may be present in such an 
\(A \circ B \circ C\) composition, and the reason they are families 
is that there is no bound on the number of antipodes that may be 
present on each string in the $B$ part; the family parameter $N$ is 
mostly a counter for how many multiples of $S^{\circ 2}$ that are 
common to the $B$ part strings in the rule:
\begin{subequations} \label{Eq:Regelfamiljer}
  \begin{align}
    \label{Rule:1-1-0-1} 
    \mOp^b_{cd} \, \left(S^{\circ(2N)}\right)^c_e \, 
      \left(S^{\circ(2N+1)}\right)^d_f \, \Delta^{ef}_g &{}\longmapsto
    \eta^b \, \ve_g
    \\
    \label{Rule:1-1-1-0}
    \mOp^b_{cd} \, \left(S^{\circ(2N+1)}\right)^c_e \, 
      \left(S^{\circ(2N)}\right)^d_f \, \Delta^{ef}_g &{}\longmapsto
    \eta^b \, \ve_g
    \\
    \mOp^b_{cd} \, \left(S^{\circ(2N+2)}\right)^d_e \, 
      \left(S^{\circ(2N+1)}\right)^c_f \, \Delta^{ef}_g &{}\longmapsto
    \eta^b \, \ve_g
    \\
    \mOp^b_{cd} \, \left(S^{\circ(2N+1)}\right)^d_e \, 
      \left(S^{\circ(2N+2)}\right)^c_f \, \Delta^{ef}_g &{}\longmapsto
    \eta^b \, \ve_g
    \displaybreak[0]\\
    \label{Rule:1-2-1-0}
    \mOp^a_{bc} \, \mOp^b_{de} \, \left(S^{\circ(2N+1)}\right)^c_f \, 
      \left(S^{\circ(2N)}\right)^e_g \, \Delta^{gf}_h &{}\longmapsto
    \delta^a_d \, \ve_h
    \\
    \mOp^a_{bc} \, \mOp^b_{de} \, \left(S^{\circ(2N)}\right)^c_f \, 
      \left(S^{\circ(2N+1)}\right)^e_g \, \Delta^{gf}_h &{}\longmapsto
    \delta^a_d \, \ve_h
    \\
    \mOp^a_{bc} \, \mOp^b_{de} \, \left(S^{\circ(2N+2)}\right)^c_f \, 
      \left(S^{\circ(2N+1)}\right)^e_g \, \Delta^{fg}_h &{}\longmapsto
    \delta^a_d \, \ve_h
    \\
    \mOp^a_{bc} \, \mOp^b_{de} \, \left(S^{\circ(2N+1)}\right)^c_f \, 
      \left(S^{\circ(2N+2)}\right)^e_g \, \Delta^{fg}_h &{}\longmapsto
    \delta^a_d \, \ve_h
    \displaybreak[0]\\
    \mOp^b_{cd} \, \left(S^{\circ(2N+1)}\right)^c_e \, 
      \left(S^{\circ(2N)}\right)^d_f \, \Delta^{ae}_g \, \Delta^{gf}_h 
      &{}\longmapsto
    \delta^a_h \, \eta^b
    \\
    \mOp^b_{cd} \, \left(S^{\circ(2N)}\right)^c_e \, 
      \left(S^{\circ(2N+1)}\right)^d_f \, \Delta^{ae}_g \, \Delta^{gf}_h 
      &{}\longmapsto
    \delta^a_h \, \eta^b
    \\
    \mOp^b_{cd} \, \left(S^{\circ(2N+1)}\right)^d_e \, 
      \left(S^{\circ(2N+2)}\right)^c_f \, \Delta^{ae}_g \, \Delta^{gf}_h 
      &{}\longmapsto
    \delta^a_h \, \eta^b
    \\
    \mOp^b_{cd} \, \left(S^{\circ(2N+2)}\right)^d_e \, 
      \left(S^{\circ(2N+1)}\right)^c_f \, \Delta^{ae}_g \, \Delta^{gf}_h 
      &{}\longmapsto
    \delta^a_h \, \eta^b
    \displaybreak[0]\\
    \mOp^b_{cd} \, \mOp^c_{ef} \, \left(S^{\circ(2N)}\right)^d_g \, 
      \left(S^{\circ(2N+1)}\right)^f_h \, \Delta^{ah}_i \, \Delta^{ig}_j 
      &{}\longmapsto
    \delta^a_j \, \delta^b_e
    && \text{where \(a \canmake e\),}
    \\
    \mOp^b_{cd} \, \mOp^c_{ef} \, \left(S^{\circ(2N+1)}\right)^d_g \, 
      \left(S^{\circ(2N)}\right)^f_h \, \Delta^{ah}_i \, \Delta^{ig}_j 
      &{}\longmapsto
    \delta^a_j \, \delta^b_e
    && \text{where \(a \canmake e\),}
      \label{Eq:MinstaFeedback-exempel}
    \\
    \mOp^b_{cd} \, \mOp^c_{ef} \, \left(S^{\circ(2N+2)}\right)^d_g \, 
      \left(S^{\circ(2N+1)}\right)^f_h \, \Delta^{ag}_i \, \Delta^{ih}_j 
      &{}\longmapsto
    \delta^a_j \, \delta^b_e
    && \text{where \(a \canmake e\),}
    \\
    \mOp^b_{cd} \, \mOp^c_{ef} \, \left(S^{\circ(2N+1)}\right)^d_g \, 
      \left(S^{\circ(2N+2)}\right)^f_h \, \Delta^{ag}_i \, \Delta^{ih}_j 
      &{}\longmapsto
    \delta^a_j \, \delta^b_e
    && \text{where \(a \canmake e\).}
  \end{align}
\end{subequations}
While these $16$ families are distinct, it should also be clear that 
they are very much variations on a single theme, and indeed there are 
many ways in which the rules of one family logically implies the 
rules in another when \eqref{Eq:SporadiskaRegler} is given. What they 
all amount to is basically that two strings of antipodes cancel each 
other out if (i)~they are adjacent both in the $A$ and $C$ parts of 
the expression, (ii)~the number of antipodes on them differ by 
exactly~$1$, and (iii)~they cross each other iff the string with the 
lower number of antipodes has an odd number of antipodes. 

%

\subsection{Notation}

The set $\N$\index{N@$\N$} of natural numbers is considered to 
include $0$. The set of positive integers is written $\Zp$. A 
\emDefOrd{semiring} is like a ring, except that elements need not 
have an additive inverse; $\N$ and $\Z$ are (associative and 
commutative unital) semirings, but $\Zp$ is not a semiring (on 
account of not containing a zero element).

Matrices with parenthesis delimiters are ordinary matrices. Matrices 
with bracket delimiters are \emDefOrd{block matrices}, where entries 
may themselves be matrices, and it is then the elements of those 
entries that are elements of the block matrix. As usual, the $(i,j)$ 
entry of a matrix $A$ may be denoted $A_{ij}$, $a_{ij}$, $A_{i,j}$, 
or $a_{i,j}$ depending on which is most clear in context. Note, 
however, that e.g.~$A_{ij}$ might often denote some block in the 
matrix $A$; in such cases it is explicitly defined as such. $J_{m 
\times n}$\index{J m n@$J_{m \times n}$} denotes the $m \times n$ 
matrix of ones, i.e., the $m \times n$ matrix where all elements 
are~$1$.

Brackets are also used in the notation $[a]_\simeq$ for the 
equivalence class of $a$ (with respect to some equivalence relation 
$\simeq$), and in the notation $[n]$ for $\{1,\dotsc,n\}$; on the 
latter, see Definition~\ref{Def:Permutationer}.

Composition of maps has the inner function on the right, so 
\index{\circ@$\circ$!function composition}
\((f \circ\nobreak g)(x) = f\bigl( g(x) \bigr)\). In some places, 
$(f,g)(x)$ is used as a shorthand for $\bigl( f(x), g(x) \bigr)$. 
Given a function \(f \colon A \Fpil B\), the corresponding direct and 
inverse set maps are defined by%
\index{P@$\setmap{}$}\index{P bar@$\setinv$}
\begin{align*}
  \setmap{f}(X) ={}& \setOf[\big]{ f(x) }{ x \in X }
    && \text{for \(X \subseteq A\),}\\
  \setinv{f}(Y) ={}& \setOf[\big]{ x \in A }{ f(x) \in Y}
    && \text{for \(Y \subseteq B\);}
\end{align*}
the $\setmap{}$ notation is because making $\setmap{f}$ and 
$\setinv{f}$ out of $f$ is precisely what the co- and contravariant 
respectively power set functors do with morphisms. For 
$\setmap{f}(A)$, the shorter notation $\setim f$\index{im@$\setim$} 
is frequently used. The restriction of $f$ to \(X \subseteq A\) is 
denoted $\restr{f}{X}$. The identity function\slash map is generally 
denoted $\mathrm{id}$.\index{id@$\mathrm{id}$}

\section{\PROPs}
\label{Sec:PROPpar}

\PROP~is short for `PROducts and Permutations category', a special 
case of symmetric monoidal category, but this etymological background 
is of little relevance for their use here. Rather, \PROPs\ will be 
viewed as just another type of abstract algebraic structure, and 
incidentally one for which matrix arithmetic provides several 
elementary examples. An instructive \PROP\ is that which
consists of \emph{all matrices regardless of sides} over any fixed 
(unital associative semi-)ring, since the syntactic restrictions on 
the composition operation in a \PROP\ are the same as for matrix 
multiplication.

Concretely, a \PROP\ $\mc{P}$ is usually defined as a family 
\(\bigl\{ \mc{P}(m,n) \bigr\}_{m,n\in\N}\) of sets (the 
\emph{components} of the \PROP, which in the categorical formalism 
are precisely the $\hom$-sets) together with some operations 
satisfying certain axioms. In the \PROP\ $\mc{R}^{\bullet\times\bullet}$ 
of all matrices over some ring $\mc{R}$, the family of components is 
thus $\left\{ \mc{R}^{m \times n} \right\}_{m,n\in\N}$, i.e., for 
every number of rows $m$ and every number of columns $n$ there is the 
component $\mc{R}^{m \times n}$ of $m \times n$ matrices over $\mc{R}$. 
In the abstract setting, the indices $m$ and $n$ are called the 
\emph{coarity} and \emph{arity} respectively, and $\mc{P}(m,n)$ is 
simply the set of elements in $\mc{P}$ which have arity~$n$ and 
coarity~$m$. 


The first operation `$\circ$' is called \emph{composition}. It is 
similar to matrix multiplication in that it is a map \(\mc{P}(l,m) 
\times \mc{P}(m,n) \Fpil \mc{P}(l,n)\) for any \(l,m,n\in\N\), and 
in the \PROP\ $\mc{R}^{\bullet\times\bullet}$ it 
really is matrix multiplication. Like matrix multiplication it is 
associative and has identities\Dash a separate identity in every 
$\mc{P}(n,n)$.

The second operation $\otimes$ is called the tensor product (although 
in the case of $\mc{R}^{\bullet\times\bullet}$ this is perhaps not 
the first name one would choose for this operation), and has 
signature \(\mc{P}(k,l) \otimes \mc{P}(m,n) \Fpil \mc{P}(k +\nobreak 
m, l +\nobreak n)\) for all \(k,l,m,n \in \N\). Intuitively it is 
often a way of putting the left factor beside the right factor 
without making them interact, and in the 
$\mc{R}^{\bullet\times\bullet}$ \PROP\ it merely amounts to putting 
the two factors as diagonal blocks in an otherwise $0$ matrix, e.g.
\begin{equation*}
  \begin{pmatrix} a_{11} & a_{12} \\ a_{21} & a_{22} \end{pmatrix}
  \otimes
  \begin{pmatrix} b_{11} & b_{12} \\ b_{21} & b_{22} \end{pmatrix}
  = \begin{pmatrix}
    a_{11} & a_{12} & 0 & 0 \\
    a_{21} & a_{22} & 0 & 0 \\
    0 & 0 & b_{11} & b_{12} \\
    0 & 0 & b_{21} & b_{22}
  \end{pmatrix}
  \text{.}
\end{equation*}
(There are other matrix \PROPs\ in which $\otimes$ really ends up 
multiplying matrix elements, e.g.~when it is the Kronecker matrix 
product, but we'll get to that later.) The tensor product is 
associative and has a unique identity element (with arity and 
coarity~$0$).

The final operations are the actions of permutations on \PROP\ 
elements, which in the case of $\mc{R}^{\bullet\times\bullet}$ are 
permutation of rows (in the case of left action) and permutation of 
the columns (in the case of the right action). Alternatively (and with 
slightly less machinery) this can be formalised as a family 
$\{\phi_n\}_{n\in\N}$ of maps from the permutation group $\Sigma_n$ 
on $n$ elements to $\mc{P}(n,n)$, such that the actions are given 
by \(\sigma a = \phi_m(\sigma) \circ a\) and \(a \tau = a \circ 
\phi_n(\tau)\) for all \(a \in \mc{P}(m,n)\), \(\sigma \in 
\Sigma_m\), and \(\tau \in \Sigma_n\). In $\mc{R}^{\bullet\times\bullet}$, 
the matrix $\phi_m(\sigma)$ is thus the permutation matrix corresponding 
to $\sigma$, with \(\phi_m(\sigma)_{i,j} = 1\) iff \(i = \sigma(j)\).

Besides these, many \PROPs\ come with a linear structure such that 
each $\mc{P}(m,n)$ is a module over some ring\Ldash $\mc{R}^{m \times n}$ 
is for example an $\mc{R}$-module\Dash but this is not part of the 
basic concept. As is common in diamond lemma\slash Gr\"obner 
basis\slash rewriting theory, there will be a need both for a 
structure of ``monomials'' and for a structure of ``polynomials'', 
and in this case both of these will be \PROPs, just like both the 
monomials and polynomials of commutative Gr\"obner basis theory 
constitute monoids. The structure that relates to a basic \PROP\ as 
an $\mc{R}$-algebra does to a monoid is the 
\emDefOrd[*{linear@$\mc{R}$-linear!PROP@\PROP}]{$\mc{R}$-linear \PROP}, 
and one can always construct an $\mc{R}$-linear \PROP\ from 
formal $\mc{R}$-linear combinations of elements from another \PROP; 
Construction~\ref{Kons:R-linjar} gives the details.

\subsection{Formal definition and basic examples}

\begin{definition}
  An \DefOrd[*{N2-graded@$\N^2$-graded!set}]{$\N^2$-graded set} $\mc{P}$ 
  is a set together with two functions 
  \index{alpha@$\alpha$}\index{omega@$\omega$}%
  \(\alpha_{\mc{P}},\omega_{\mc{P}}\colon \mc{P} \Fpil \N\). 
  In this context, the value $\alpha_{\mc{P}}(b)$ is called the 
  \DefOrd{arity} of $b$ and the value $\omega_{\mc{P}}(b)$ is called 
  the \DefOrd{coarity} of $b$; the dependence on $\mc{P}$ is 
  usually elided. Denote by 
  \index{(m,n)@$\cdot(m,n)$!$\N^2$-graded set component}$\mc{P}(m,n)$ 
  the set of those elements in $\mc{P}$ which have coarity $m$ and 
  arity $n$.
\end{definition}

Modulo some formal nonsense regarding the distinction between 
formally disjoint union and plain set-theoretic union, any family 
$\bigl\{ \mc{P}(m,n) \bigr\}_{m,n\in\N}$ of sets parametrised by two 
natural numbers may be regarded as an $\N^2$-graded set, so in 
particular any \PROP\ may viewed this way. This simplifies the 
definition of many basic concepts such as sub-\PROP, \PROP\ morphism, 
and generating set, but concrete constructions often tend to describe 
each component separately, and it frequently happens that a strict 
set-theoretic interpretation of these descriptions would make some 
components non-disjoint; in such cases it is to be understood that 
elements from different components of the \PROP\ are distinct when 
the \PROP\ is viewed as an $\N^2$-graded set.

\begin{definition}
  Let $\mc{P}$ and $\mc{Q}$ be $\N^2$-graded sets. 
  An \DefOrd[*{N2-graded@$\N^2$-graded!set morphism}]{$\N^2$-graded 
  set morphism} 
  \(f\colon \mc{P} \Fpil \mc{Q}\) is a map \(\mc{P} \Fpil \mc{Q}\) 
  which preserves arity and coarity, i.e., \(\alpha_{\mc{Q}}\bigl( 
  f(b) \bigr) = \alpha_{\mc{P}}(b)\) and \(\omega_{\mc{Q}}\bigl( 
  f(b) \bigr) = \omega_{\mc{P}}(b)\) for all \(b \in \mc{P}\). 
  More generally, a map \(f\colon \mc{P} \Fpil \mc{Q}\) is said to 
  have \DefOrd{degree} $(k,l)$ if \(\omega_{Q}\bigl( f(b) \bigr) = 
  k + \omega_{\mc{P}}(b)\) and \(\alpha_{Q}\bigl( f(b) \bigr) = l + 
  \alpha_{\mc{P}}(b)\) for all \(b \in \mc{P}\).
\end{definition}

As might be expected, \PROP\ homomorphisms are $\N^2$-graded set 
morphisms which satisfy some additional conditions. Various kinds of 
derivatives are often convenient to formalise as degree $(0,1)$ or 
$(1,0)$ maps, but those concepts are more natural to discuss using 
the abstract index notation introduced in Section~\ref{Sec:AIN}.

Permutations are an important ingredient in the \PROP\ concept, so 
before giving the formal definition of the latter, it is necessary to 
introduce some notations for the former; some of this should be 
familiar, but other parts are uncommon outside discussions of \PROPs.

\begin{definition} \label{Def:Permutationer}
  For all \(n \in \N\), let $[n]$ denote the set $\{1,\dotsc,n\}$, 
  and in particular \([0] = \varnothing\). Then define the 
  \DefOrd{permutation group} $\Sigma_n$\index{Sigma n@$\Sigma_n$} to 
  be the group of all bijections \([n] \Fpil [n]\). In particular, 
  $\Sigma_0$ is considered to be a group with one element.
  Explicit permutations may be written in \DefOrd{relation 
  notation}, where \(\sigma = \left( \begin{smallmatrix} 1& 2& 
  \hdots& n \\ \sigma(1) & \sigma(2) & \hdots & \sigma(n) 
  \end{smallmatrix} \right)\), but more commonly in \DefOrd{cycle 
  notation}, where $(k_1\,k_2\,\ldots\,k_r)$ denotes the permutation 
  $\sigma$ which satisfies \(\sigma(k_1) = k_2\), \(\sigma(k_2) = 
  k_3\), \dots, \(\sigma(k_{r-1}) = k_r\), \(\sigma(k_r) = k_1\) 
  and leaves all other elements fixed. 
  
  For the \PROP\ axioms, it is convenient to introduce the 
  notation \index{X@$\cross{k}{m}$}$\cross{k}{m}$ for the element 
  of \(\Sigma_{k+m}\) which exchanges a left block of $k$ things 
  with a right block of $m$ things, i.e.,
  \begin{equation}
    \cross{k}{m} (i) = \begin{cases}
      i+m& \text{if \(i \leqslant k\),}\\
      i-k& \text{if \(i > k\).}
    \end{cases}
  \end{equation}
  One may think of this as a crossing of two flat ``cables'', one 
  containing $k$ wires and the other containing $m$ wires. 
  Similarly the identity element in $\Sigma_n$ will be written as 
  \index{I@$\same{n}$ (identity permutation)}$\same{n}$.
  
  The \PROP\ axioms also use the permutation juxtaposition 
  product \index{\star@$\star$}\((\sigma,\tau) \mapsto 
  \sigma \star \tau : \Sigma_m \times \Sigma_n \Fpil \Sigma_{m+n}\), 
  which is defined by
  \begin{equation}
    (\sigma \star \tau)(i) = \begin{cases}
      \sigma(i)& \text{if \(i \leqslant m\),}\\
      \tau(i-m)+m& \text{if \(i>m\).}
    \end{cases}
  \end{equation}
  The group product of permutations, i.e., the function composition, 
  will be written as $\circ$ if there is a need to emphasise this 
  operation, but usually the permutations are just written next to 
  each other.
\end{definition}

\begin{definition} \label{Def:PROP}
  A \DefOrd[{PROP}*]{\PROP} $\mc{P}$ is an $\N^2$-graded set together 
  with a family \(\{\phi_n\}_{n\in\N}\)\index{phi@$\phi$} of maps 
  \(\phi_n\colon \Sigma_n \Fpil \mc{P}(n,n)\) and two binary operations 
  `$\circ$'\index{\circ@$\circ$!PROP composition@\PROP\ composition} 
  (composition) and `$\otimes$' (tensor product). \emph{The $\otimes$ 
  operation is considered to have higher precedence (binding strength) 
  than the $\circ$ operation.} For any \(a,b \in \mc{P}\), 
  \(\alpha(a \nobreak\otimes b) = \alpha(a) + \alpha(b)\) 
  and \(\omega(a \nobreak\otimes b) = \omega(a) + \omega(b)\). 
  Composition is a partial operation, such that $a \circ b$ is defined 
  if and only if \(\alpha(a) = \omega(b)\), in which case \(\alpha(a 
  \circ\nobreak b) = \alpha(b)\) and \(\omega(a \circ\nobreak b) = 
  \omega(a)\).
  
  A \PROP\ must furthermore satisfy the following axioms:
  \begin{description}
    \item[composition associativity]
      \((a \circ b) \circ c = a \circ (b \circ c)\) for all \(a \in 
      \mc{P}(k,l)\), \(b \in \mc{P}(l,m)\), and \(c \in \mc{P}(m,n)\).
    \item[composition identity]
      \(\phi_m(\same{m}) \circ a = a = a \circ \phi_n(\same{n})\) 
      for all \(a \in \mc{P}(m,n)\).
    \item[tensor associativity]
      \((a \otimes b) \otimes c = a \otimes (b \otimes c)\) for all 
      \(a \in \mc{P}(k,l)\), \(b \in \mc{P}(m,n)\), and \(c \in 
      \mc{P}(r,s)\).
    \item[tensor identity]
      \(\phi_0(\same{0}) \otimes a = a = a \otimes \phi_0(\same{0})\) 
      for all \(a \in \mc{P}(m,n)\).
    \item[composition--tensor compatibility]
      \((a \circ b) \otimes (c \circ d) = (a \otimes c) \circ (b 
      \otimes d)\) for all \(a \in \mc{P}(l,m)\), \(b \in 
      \mc{P}(m,n)\), \(c \in \mc{P}(k,r)\), and \(d \in \mc{P}(r,s)\).
    \item[permutation composition]
      \(\phi_n(\sigma) \circ \phi_n(\tau) = \phi_n(\sigma\tau)\) for 
      all \(\sigma,\tau \in \Sigma_n\). In other words, each $\phi_n$ 
      is a group homomorphism.
    \item[permutation juxtaposition]
      \(\phi_m(\sigma) \otimes \phi_n(\tau) = \phi_{m+n}(\sigma \star 
      \tau)\) for all \(\sigma \in \Sigma_m\) and \(\tau \in \Sigma_n\). 
    \item[tensor permutation]
      \(\phi_{k+m}(\cross{k}{m}) \circ (a \otimes b) = (b \otimes a) 
      \circ \phi_{l+n}(\cross{l}{n})\) for all \(a \in \mc{P}(k,l)\) 
      and \(b \in \mc{P}(m,n)\).
  \end{description}
  Since the index $n$ of a $\phi_n$ map is clear from the permutation 
  fed into it, this part of the notation is often omitted. (In practical 
  applications, it is common to leave the $\phi_n$ maps out 
  altogether.) 
  
  A \DefOrd[*{PROP@\PROP!homomorphism}]{\PROP\ (homo)morphism} 
  \(f\colon \mc{P} \Fpil \mc{Q}\) is 
  an $\N^2$-graded set morphism which preserves $\circ$, $\otimes$, 
  and $\phi$, i.e., \(f(a \circ_{\mc{P}}\nobreak b) = f(a) 
  \circ_{\mc{Q}} f(b)\), \(f(a \otimes_{\mc{P}}\nobreak b) = f(a) 
  \otimes_{\mc{Q}} f(b)\), and \(f \circ \phi_{\mc{P}} = 
  \phi_{\mc{Q}}\).
  
  A \DefOrd[{sub-PROP}*]{sub-\PROP} $\mc{Q}$ of $\mc{P}$ is a subset 
  of $\mc{P}$ which contains $\phi_n(\Sigma_n)$ for all $n$ and is 
  closed under the composition and tensor product operations of 
  $\mc{P}$. An $\N^2$-graded set \(\Omega \subseteq \mc{P}\) is said to 
  \DefOrd{generate} $\mc{P}$ if the only sub-\PROP\ of $\mc{P}$ that 
  contains $\Omega$ is $\mc{P}$ itself. The 
  \DefOrd[*{sub-PROP@sub-\PROP!generated by}]{sub-\PROP\ generated by} 
  $\Omega$ is the intersection of all sub-\PROPs\ of $\mc{P}$ that 
  contains~$\Omega$.
  
  A \PROP\ $\mc{P}$ is said to be 
  \DefOrd[*{linear@$\mc{R}$-linear!PROP@\PROP}]{$\mc{R}$-linear} 
  for some ring $\mc{R}$ if every $\mc{P}(m,n)$ is an $\mc{R}$-module 
  and the two operations $\circ$ and $\otimes$ are $\mc{R}$-bilinear. 
  A \emDefOrd[*{linear@$\mc{R}$-linear!PROP homomorphism@\PROP\ 
  homomorphism}]{homomorphism} \(f\colon \mc{P} \Fpil \mc{Q}\) is 
  $\mc{R}$-linear if $\mc{P}$ and $\mc{Q}$ are both $\mc{R}$-linear and 
  the restriction of $f$ to every component of $\mc{P}$ is 
  $\mc{R}$-linear.
\end{definition}

Theorem~\ref{S:PROP2-definition} can be taken as an alternative 
definition of \PROP, as an $\N^2$-graded set which supports (and is 
closed under) evaluation of networks. Although that requires a 
heavier machinery to set up (particularly for defining `network'), 
it is far more intuitive and not as seemingly arbitrary as 
Definition~\ref{Def:PROP}. The natural next step is however to give 
examples of structures that satisfy the set of axioms given above.


\begin{example}[Permutations \PROP]
  The set of all permutations constitute a \PROP\ $\mc{P}$ with 
  the group operation(s) and $\star$ as composition and tensor 
  product respectively;
  \begin{equation*}
    \mc{P}(m,n) = \begin{cases}
      \Sigma_n& \text{if \(m=n\),}\\
      \varnothing& \text{otherwise,}\\
    \end{cases}
  \end{equation*}
  \(\sigma \circ \tau := \sigma\tau\), \(\sigma \otimes \tau := 
  \sigma \star \tau\), and \(\phi(\sigma) := \sigma\). Several axioms 
  are mere trivialities for this \PROP, and even those that may be 
  non-obvious (e.g.~the tensor associativity and tensor permutation 
  axioms) are straightforward to verify through explicit calculation.
  
  If $\mc{Q}$ is another \PROP\ then $\phi_{\mc{Q}}$ is a \PROP\ 
  morphism \(\mc{P} \Fpil \mc{Q}\). Hence the permutations \PROP\ 
  \emph{is} a free \PROP, although for an empty generating set. A 
  general construction of the free \PROP\ is the subject of 
  Section~\ref{Sec:FriPROP}.
\end{example}

The next example demonstrates that one doesn't have to keep the 
permutations distinct.

\begin{example}[Ring \PROP]
  Let $\mc{R}$ be an associative and commutative unital ring. Let 
  $\mc{P}$ be the $\N^2$-graded set where
  \begin{equation*}
    \mc{P}(m,n) = \begin{cases}
      \mc{R}& \text{if \(m=n\),}\\
      \{0\} \subset \mc{R}& \text{otherwise,}
    \end{cases}
  \end{equation*}
  let \(\phi_n(\sigma) = 1 \in \mc{R} = \mc{P}(n,n)\) for all 
  \(\sigma \in \Sigma_n\), and let $\circ$ and $\otimes$ on $\mc{P}$ 
  both be the multiplication in $\mc{R}$. Then $\mc{P}$ is an 
  $\mc{R}$-linear \PROP.
  
  It follows from the tensor permutation axiom that \(ab = 1 \circ 
  a \otimes b = \phi_2(\cross{1}{1}) \circ a \otimes b =
  b \otimes a \circ \phi_2(\cross{1}{1}) = b \otimes a \circ 1 = ba\) 
  for all \(a,b \in \mc{P}(1,1) = \mc{R}\), so this example requires 
  $\mc{R}$ to be commutative.
\end{example}

There is however a way to turn noncommutative rings into \PROPs, by 
only putting them in the $(1,1)$ component. This is essentially the 
same construction as that of an operad from a ring, where only the 
arity $1$ component is nontrivial.

\begin{example} \label{Ex:Algebra-PROP}
  Let $\mc{R}$ be an associative and commutative unital ring. Let 
  $\mc{A}$ be an associative unital $\mc{R}$-algebra. Let 
  $\mc{P}$ be the $\N^2$-graded set where
  \begin{equation*}
    \mc{P}(m,n) = \begin{cases}
      \mc{R}& \text{if \(m=n=0\),}\\
      \mc{A}& \text{if \(m=n=1\),}\\
      \{0\}& \text{otherwise.}
    \end{cases}
  \end{equation*}
  Define $\circ$ to be:
  \begin{itemize}
    \item
      the multiplication in $\mc{R}$ when mapping \(\mc{P}(0,0) 
      \times \mc{P}(0,0) \Fpil \mc{P}(0,0)\),
    \item
      the multiplication in $\mc{A}$ when mapping \(\mc{P}(1,1) 
      \times \mc{P}(1,1) \Fpil \mc{P}(1,1)\), and
    \item
      the constant $0$ map otherwise.
  \end{itemize}
  Define $\otimes$ to be:
  \begin{itemize}
    \item
      the multiplication in $\mc{R}$ when mapping \(\mc{P}(0,0) 
      \times \mc{P}(0,0) \Fpil \mc{P}(0,0)\),
    \item
      the action of $\mc{R}$ on $\mc{A}$ when mapping \(\mc{P}(0,0) 
      \times \mc{P}(1,1) \Fpil \mc{P}(1,1)\) or \(\mc{P}(1,1) 
      \times \mc{P}(0,0) \Fpil \mc{P}(1,1)\), and
    \item
      the constant $0$ map otherwise.
  \end{itemize}
  Finally define \(\phi_0(\same{0}) = 1 \in \mc{R}\), 
  \(\phi_1(\same{1}) = 1 \in \mc{A}\), and \(\phi_n(\sigma) = 0\) for 
  all \(\sigma \in \Sigma_n\) and \(n>1\). Then $\mc{P}$ is an 
  $\mc{R}$-linear \PROP.
\end{example}

There are also constructions along the same line of \PROPs\ from 
arbitrary (symmetric) operads where the coarity $1$ components of the 
\PROP\ are exactly the components of the operad, but these get more 
technical\Dash in part because the axioms for an operad, though 
fewer, are much messier than their \PROP\ counterparts. It is easier 
to \emph{define} an operad as the coarity~$1$ components of a \PROP, 
and introduce the structure map as
\begin{equation*}
  \gamma_{n_1,\dotsc,n_m}(a,b_1,\dotsc,b_m) =
  a \circ b_1 \otimes \dotsb \otimes b_m
\end{equation*}
or $i$'th composition as
\begin{equation*}
  a \circ_i b = 
  a \circ \phi(\same{i-1}) \otimes b \otimes \phi(\same{\alpha(a)-i})
  \text{.}
\end{equation*}
A better first example of a \PROP\ with nontrivial off-diagonal 
components is the \PROP\ of matrices, which was sketched above but 
can do with getting the corner cases straightened out.

\begin{example}[Matrix \PROP]
  For every associative unital semiring $\mc{R}$, there is a \PROP\ 
  $\mc{R}^{\bullet\times\bullet}$ such that:
  \begin{itemize}
    \item
      \(\mc{R}^{\bullet\times\bullet}(m,n) = \mc{R}^{m \times n}\) 
      ($m$-by-$n$ matrices whose entries are elements of $\mc{R}$). 
      In particular, each of the degenerate sets $\mc{R}^{0 \times 
      n}$ and $\mc{R}^{m \times 0}$ is considered to have exactly 
      one element, which is denoted $0$.
    \item
      The composition \(\circ\colon \mc{R}^{l \times m} \times 
      \mc{R}^{m \times n} \Fpil \mc{R}^{l \times n}\) is 
      multiplication of an $l$-by-$m$ matrix with an $m$-by-$n$ 
      matrix. In the case that \(m=0\), the result is the $l$-by-$n$ 
      all zeroes matrix.
    \item
      The tensor product $\otimes$ constructs a block matrix from its 
      factors, by putting these on the main diagonal and filling the 
      rest of the matrix with zeroes:
      \begin{equation*}
        A \otimes B = \begin{bmatrix} A & 0 \\ 0 & B \end{bmatrix}
        \text{.}
      \end{equation*}
      (This is often called the direct sum of two matrices, and is 
      then typically denoted $\oplus$. A \PROP\ where $\otimes$ is 
      the ``usual'' (Kronecker) tensor product of matrices can be 
      found in Example~\ref{Ex:MatrisHomPROP}.)
    \item
      The $\phi_n$ maps map permutation to corresponding 
      permutation matrices, i.e.,
      \begin{equation*}
        \phi_n(\sigma)_{i,j} = \begin{cases}
          1& \text{if \(i = \sigma(j)\),}\\
          0& \text{otherwise.}
        \end{cases}
      \end{equation*}
  \end{itemize}
  The \PROP\ axioms are either well-known properties (e.g.~matrix 
  multiplication is associative) or easily verified through direct 
  calculations. The tensor permutation axiom is for example verified 
  by observing that
  \begin{multline*}
     \phi_{k+m}(\cross{k}{m}) \circ (A \otimes B) = 
     \begin{bmatrix} 0 & I_m \\ I_k & 0 \end{bmatrix} \circ 
       \begin{bmatrix} A & 0 \\ 0 & B \end{bmatrix} =
     \begin{bmatrix} 0 & B \\ A & 0 \end{bmatrix} = \\ =
     \begin{bmatrix} B & 0 \\ 0 & A \end{bmatrix} \circ
       \begin{bmatrix} 0 & I_n \\ I_l & 0 \end{bmatrix} =
     (B \otimes A) \circ \phi_{l+n}(\cross{l}{n})    
     \text{,}
  \end{multline*}
  where $I_n$ denotes the $n \times n$ identity matrix.
  
  
  In $\mc{R}^{\bullet\times\bullet}$, every component is an 
  $\mc{R}$-module, and if $\mc{R}$ is commutative then composition is 
  $\mc{R}$-bilinear, but the \PROP\ as a whole is typically not 
  $\mc{R}$-linear. The reason for this is that its tensor product 
  fails to be bilinear; e.g.~\((2A) \otimes (2B) = 2(A \otimes\nobreak 
  B)\), whereas the left hand side would be equal to 
  $4(A \otimes\nobreak B)$ in a $\Z$-linear \PROP.
\end{example}


However, when \PROPs\ occur in higher algebra, it is usually first as 
some variant of the following, where the operations are bilinear by 
design.

\subsection{The \texorpdfstring{$\HomPROP$}{Hom} \PROP\ and friends}

\begin{example}[The $\HomPROP$ \PROP] \label{Ex:HomPROP}
  Let $\mc{R}$ be an associative and commutative unital ring. Let $V$ 
  be an $\mc{R}$-module. Let $\HomPROP_V(m,n)$\index{Hom@$\HomPROP$} 
  be defined for all \(m,n \in \N\) by
  \begin{equation}
    \HomPROP_V(m,n) = \hom_{\mc{R}}( V^{\otimes n}, V^{\otimes m} )
    \text{,}\qquad
    \text{where \(V^{\otimes n} = 
    \underbrace{V \otimes_{\mc{R}} \dotsb \otimes_{\mc{R}} V }_{ 
      \text{$n$ factors}}\)}
  \end{equation}
  and $\hom_{\mc{R}}(U,W)$ is the space of all $\mc{R}$-linear maps 
  from $U$ to $W$; this space is an $\mc{R}$-module under pointwise 
  operations. Let the composition and tensor product on $\HomPROP_V$ 
  be the ordinary compositions and tensor products of $\mc{R}$-linear 
  maps, i.e.,
  \begin{align}
    (a \circ b)(\vek{u}) ={}& a\bigl( b(\vek{u}) \bigr)
      && \text{for all \(\vek{u} \in V^{\otimes m}\),}\\
    (a \otimes c)(\vek{u} \otimes \vek{v}) ={}&
      a(\vek{u}) \otimes c(\vek{v}) &&
      \text{for all \(\vek{u} \in V^{\otimes l}\) and 
      \(\vek{v} \in V^{\otimes n}\),}
      \label{Eq3:HomPROP}
  \end{align}
  for all \(a \in \HomPROP_V(k,l)\), \(b \in \HomPROP_V(l,m)\), 
  \(c \in \HomPROP_V(m,n)\), and \(k,l,m,n \in \N\). 
  (It may be remarked that the argument $\vek{u} \otimes \vek{v}$ of 
  $a \otimes c$ in \eqref{Eq3:HomPROP} is not a general element of 
  $V^{\otimes(l+n)}$, but since $V^{\otimes(l+n)}$ is spanned by 
  vectors on this form and $a \otimes c$ is an element of 
  $\hom_{\mc{R}}( V^{\otimes (l+n)}, V^{\otimes (k+m)} )$ it is 
  nonetheless sufficient. The same argument applies in 
  \eqref{Eq4:HomPROP}.)
  Finally let \(\phi_n\colon \Sigma_n \Fpil \HomPROP_V(n,n)\) be 
  defined by
  \begin{equation} \label{Eq4:HomPROP}
    \phi_n(\sigma)(\vek{u}_1 \otimes \dotsb \otimes \vek{u}_n) =
    \vek{u}_{\sigma^{-1}(1)} \otimes \dotsb \otimes
      \vek{u}_{\sigma^{-1}(n)}
  \end{equation}
  for all \(\vek{u}_1, \dotsc, \vek{u}_n \in V\) and \(\sigma \in 
  \Sigma_n\). Then $\HomPROP_V$ is an $\mc{R}$-linear \PROP.
  
  The most reasonable method for proving that $\HomPROP_V$ is a 
  \PROP\ is probably to fall back on the category theory from which 
  the name `\PROP' originated, and consider the category whose objects 
  are $\{ V^{\otimes n} \}_{n\in\N}$ and whose hom-sets are those given 
  by $\hom_{\mc{R}}$; this is a small full subcategory of the 
  category of $\mc{R}$-modules. In that context, the associativity and 
  identity of composition in the \PROP\ are simply these properties 
  with respect to morphisms in a category. The composition--tensor 
  compatibility axiom must be satisfied because ${-}\otimes{-}$ is 
  a bifunctor in this category, and the tensor product associativity 
  and identity axioms must be satisfied because this bifunctor is a 
  product, whose neutral object is \(V^{\otimes 0} = \mc{R}\). 
  As for the permutation axioms however, it might be less work to 
  verify these through explicit calculations. In the case of 
  permutation composition, this amounts to
  \begin{multline*}
    \phi_n(\sigma\tau)(\vek{u}_1 \otimes \dotsb \otimes \vek{u}_n)
    =
    \vek{u}_{(\sigma\tau)^{-1}(1)} \otimes \dotsb \otimes
      \vek{u}_{(\sigma\tau)^{-1}(n)}
    = \\ =
    \vek{u}_{\tau^{-1}(\sigma^{-1}(1))} \otimes \dotsb \otimes
      \vek{u}_{\tau^{-1}(\sigma^{-1}(n))}
    \stackrel{(*)}=
    \phi_n(\sigma)(
      \vek{u}_{\tau^{-1}(1)} \otimes \dotsb \otimes
      \vek{u}_{\tau^{-1}(n)}
    )
    = \\ =
    \phi_n(\sigma)\bigl( \phi_n(\tau)(
      \vek{u}_1 \otimes \dotsb \otimes \vek{u}_n
    ) \bigr)
    =
    \bigl( \phi_n(\sigma) \circ \phi_n(\tau) \bigr)(
      \vek{u}_1 \otimes \dotsb \otimes \vek{u}_n
    ) \text{,}
  \end{multline*}
  where the step marked $(*)$ may seem mysterious, but should after 
  defining \(\vek{v}_i = \vek{u}_{\tau^{-1}(i)}\) be easily 
  recognisable as an instance of \eqref{Eq4:HomPROP}.
  Proving bilinearity of $\circ$ and $\otimes$ as defined above is a 
  standard exercise.
\end{example}



\begin{example} \label{Ex:MatrisHomPROP}
  In the case that \(V = \mc{R}^k\) for some integer $k$, 
  $V^{\otimes n}$ is isomorphic to $\mc{R}^{k^n}$, and 
  $\HomPROP_V(m,n)$ may thus be viewed as a set of $k^m \times k^n$ 
  matrices. A convenient choice of isomorphism is that which, given 
  that $\{\vek{e}_i\}_{i=1}^N$ denotes the standard basis of 
  $\mc{R}^N$, maps
  \begin{equation}
    \vek{e}_{d_1} \otimes \dotsb \otimes \vek{e}_{d_n}
    \mapsto \vek{e}_D
    \qquad\text{where \(D-1 = \sum_{i=1}^n (d_i-1) k^{i-1}\)}
  \end{equation}
  for all \(d_1,\dotsc,d_n \in [k]\); effectively the numbers $d_i-1$ are 
  interpreted as ``radix-$k$ digits'' of $D-1$.\footnote{
    This is much less cumbersome to express using zero-based 
    indexing\Ldash numbering the basis vectors of $\mc{R}^N$ from 
    $\vek{e}_0$ to $\vek{e}_{N-1}$\Rdash but violating the convention 
    about matrix indices beginning at $1$ would also be a source of 
    confusion.
  }
  Under this isomorphism, the permutations are again mapped to 
  permutation matrices, but rather than permuting rows or columns it 
  is the ``radix-$k$ digits'' in row or column indices that are permuted; 
  \(\phi_n(\sigma)_{i,j} = 1\) if and only if \(i-1 = \sum_{r=1}^n 
  d_{\sigma(r)} k^{r-1}\) for \(d_1,\dotsc,d_n \in 
  \{0,\dotsc, k -\nobreak 1\}\) such that \(j-1 = \sum_{r=1}^n d_r 
  k^{r-1}\). Hence
  \begin{align*}
    \phi(\cross{1}{1}) ={}& \begin{pmatrix}
      1& 0& 0& 0\\
      0& 0& 1& 0\\
      0& 1& 0& 0\\
      0& 0& 0& 1
    \end{pmatrix}
    &&\text{for \(k=2\),}\displaybreak[0]\\
    \phi(\cross{1}{1}) ={}& \begin{pmatrix}
      1& 0& 0& 0& 0& 0& 0& 0& 0\\
      0& 0& 0& 1& 0& 0& 0& 0& 0\\
      0& 0& 0& 0& 0& 0& 1& 0& 0\\
      0& 1& 0& 0& 0& 0& 0& 0& 0\\
      0& 0& 0& 0& 1& 0& 0& 0& 0\\
      0& 0& 0& 0& 0& 0& 0& 1& 0\\
      0& 0& 1& 0& 0& 0& 0& 0& 0\\
      0& 0& 0& 0& 0& 1& 0& 0& 0\\
      0& 0& 0& 0& 0& 0& 0& 0& 1
    \end{pmatrix}
    &&\text{for \(k=3\).}
  \end{align*}
  
  As usual $\circ$ is matrix multiplication, but $\otimes$ is 
  the Kronecker matrix product given by
  \begin{equation*}
    A \otimes
    \begin{pmatrix}
      b_{11}& \hdots & b_{1n} \\
      \vdots & \ddots & \vdots \\
      b_{m1}& \hdots & b_{mn}
    \end{pmatrix} :=
    \begin{bmatrix}
      A b_{11}& \hdots & A b_{1n} \\
      \vdots & \ddots & \vdots \\
      A b_{m1}& \hdots & A b_{mn}
    \end{bmatrix}
  \end{equation*}
  where the right hand side block matrix is $km \times ln$ when $A$ is 
  $k \times l$.
\end{example}

\begin{example}[Quantum gate array \PROP]
  Let \(\mc{P} = \HomPROP_V\) in the case that \(V = \C^2\) (as a 
  complex vector space). In this case, $\mc{P}(m,n)$ may be viewed as 
  a set of $2^m \times 2^n$ matrices. Let
  \begin{equation*}
    \mc{Q} = \setOf{ A \in \mc{P} }{ 
      \text{$A$ is unitary, i.e., \(AA^\dag = I = A^\dag A\)} }
  \end{equation*}
  where $A^\dag$ denotes the Hermitian conjugate of $A$. 
  Then $\mc{Q}$ is a \PROP\ (with \(\mc{Q}(m,n)=\varnothing\) for 
  \(m \neq n\)), since \(\phi_\mc{Q}(\sigma)^\dag = 
  \phi_\mc{Q}(\sigma)^{\mathrm{T}} = \phi_\mc{Q}(\sigma^{-1})\), 
  \((AB)(AB)^\dag = A B B^\dag A^\dag = A I A^\dag = I\), and 
  \((A \otimes\nobreak B)(A \otimes\nobreak B)^\dag =
  (A \otimes\nobreak B)(A^\dag \otimes\nobreak B^\dag) =
  (AA^\dag) \otimes (BB^\dag) = I \otimes I = I\) for all permutations 
  $\sigma$ and \(A,B \in \mc{Q}\).
  
  By definition, a qubit is a quantum system whose state space is 
  $\C^2$, and a register of $n$ qubits therefore has state space 
  $(\C^2)^{\otimes n}$. Any unitary $2^n \times 2^n$ matrix 
  corresponds to a quantum gate operating on a register of $n$ 
  qubits. Therefore, $\mc{Q}$ may be called the \emph{quantum gate array 
  \cite{qgate} \PROP}.
  
  The interesting point here is not so much that the unitary matrices 
  constitute a \PROP\Ldash similar arguments can be carried out for 
  many other well-known sets of matrices, and the tensor product of 
  $\mc{R}^{\bullet\times\bullet}$ may actually be easier to reason 
  about\Rdash but that the basic operations in a \PROP\ correspond 
  directly to basic operations for composing quantum gate arrays. 
  $\mc{Q}(n,n)$ is the set of operations that can be applied to an 
  $n$-qubit register. Composition is to first apply one operation and 
  then the other. The tensor product says what happens when you act 
  on some qubits of a register with one operation and on the 
  remaining qubits with another; if \(A \in \mc{Q}(m,m)\) and \(B \in 
  \mc{Q}(n,n)\) then $A \otimes B$ applies $A$ to the first $m$ 
  qubits and $B$ to the remaining $n$ qubits of a register of $m+n$ 
  qubits.
  
  Permutations are less common in presentations of quantum circuits, 
  but they are used implicitly when one lets gates act upon arbitrary 
  lists of qubits rather than just those adjacent to each other. For 
  example, if a \(\mathrm{CNOT} \in \mc{Q}(2,2)\) gate is to be used to 
  let qubit $3$ of a $7$~qubit register control whether to negate qubit 
  $6$, then one might express the corresponding element of 
  $\mc{Q}(7,7)$ as
  $$
    \phi(\same{3} \star \cross{2}{2}) \circ
    \phi(\same{2}) \otimes \mathrm{CNOT} \otimes \phi(\same{3}) \circ
    \phi(\same{3} \star \cross{2}{2})
  $$
  or, if cycle notation for permutations is preferred,
  $$
    \phi\bigl( (1\,3)(2\,6) \bigr) \circ
    \mathrm{CNOT} \otimes \phi(\same{5}) \circ
    \phi\bigl( (1\,3)(2\,6) \bigr) \text{.}
  $$
  Conversely, if \(\mathrm{NOTC} \in \mc{Q}(2,2)\) is $\mathrm{CNOT}$ 
  where qubit $2$ controls the negation of qubit $1$, then 
  \(\phi(\cross{1}{1}) = \mathrm{CNOT} \circ \mathrm{NOTC} \circ 
  \mathrm{CNOT}\), so general permutations can be realised by a 
  quantum gate array even if the qubits are in fixed positions. 
  (Whether it would be \emph{practical} to do so under a specific qubit 
  register implementation is of course a separate matter.)
\end{example}

Another example of physics relevance is the \PROP\ of tensor fields 
on a manifold. Given an associative and commutative unital ring 
$\mc{R}$ that constitutes the set of one's ``scalar fields'', and an 
$\mc{R}$-module $V$ of derivations on $\mc{R}$ that constitutes one's 
``vector fields'', the elements of the $\HomPROP_V$ \PROP\ will be 
the corresponding general ``tensor fields''.

\begin{example}[Physicist's tensor fields] \label{Ex:Tensorfalt}
  Let $M$ be a real, smooth $k$-dimensional manifold, and let $\mc{R}$ 
  be the ring of infinitely differentiable functions \(M \Fpil \R\); 
  elements of $\mc{R}$ may be called (smooth) 
  \emDefOrd[*{scalar field}]{scalar fields} (where `field' is used in 
  the physics sense rather than the abstract algebra sense). Let $V$ 
  be the set of all $\R$-linear functions \(X\colon \mc{R} \Fpil 
  \mc{R}\) which satisfy \(X(f \cdot g) = X(f) \cdot g + f \cdot X(g)\) 
  for all \(f,g \in \mc{R}\).
  Such functions are called (smooth) \emDefOrd[*{vector field}]{vector 
  fields} on $M$; they are necessarily local in the sense that if 
  \(U \subseteq M\) is open and \(f \in \mc{R}\) is such that \(f(p)=0\) 
  for all \(p \in U\) then \(X(f)(p)=0\) for all \(p \in U\).
  Note the explicit function application parenthesis; 
  in differential geometry, it is otherwise common that the result of 
  applying a vector field $X$ to a scalar field $f$ is just 
  written~$Xf$. The ``value at a point \(p \in M\)'' of a vector 
  field $X$ is the function \(\mc{R} \Fpil \R\) given by \(f 
  \mapsto X(f)(p)\); differential geometers may be more comfortable 
  with the notation $X_p f$ for $X(f)(p)$ (the value of $X(f)$ at~$p$).
  
  The set $V$ is an $\mc{R}$-module, with \((fX)(g) = f \cdot X(g)\) for 
  all \(f,g \in \mc{R}\) and \(X \in V\). This being an $\mc{R}$-module, 
  one can form the $n$-fold tensor product $V^{\otimes n}$ over 
  $\mc{R}$ of $V$ with itself. There is a canonical $\mc{R}$-linear map 
  $\theta$ from $V^{\otimes n}$ to \(\mc{R}^n \Fpil \mc{R}\), with
  \begin{equation}
    \theta(X_1 \otimes_\mc{R} \dotsb \otimes_\mc{R} X_n)
      (f_1,\dotsc,f_n) =
    \prod_{i=1}^n X_i(f_i)
  \end{equation}
  for all \(X_1,\dotsc,X_n \in V\) and \(f_1,\dotsc,f_n \in \mc{R}\). 
  
  A (smooth) \emDefOrd{tensor field} on $M$ is an element of 
  $\HomPROP_V$, where \(\HomPROP_V(m,n) = \hom_\mc{R}(V^{\otimes m}, 
  V^{\otimes n})\) as in Example~\ref{Ex:HomPROP}. Since \(V^{\otimes 
  0} = \mc{R}\) and an $\mc{R}$-linear map \(a\colon \mc{R} \Fpil 
  V^{\otimes m}\) is uniquely determined by $a(1)$, it follows that 
  \(\HomPROP_V(0,0) \cong \mc{R}\) and \(\HomPROP_V(1,0) \cong V\), 
  as expected. $\HomPROP_V(0,1)$, the set of covector fields, consists 
  of maps from $V$ to $\mc{R}$, and any \(f \in \mc{R}\) defines a 
  differential \(\mathrm{d}f \in \HomPROP_V(0,1)\) via 
  \(\mathrm{d}f(X) = X(f)\) for all \(X \in V\). 
  A tensor field \(a \in \HomPROP_V(m,n)\) is 
  \emDefOrd[*{symmetric tensor field}]{symmetric} if \(\phi(\sigma) 
  \circ a \circ \phi(\tau) = a\) and \emDefOrd[*{antisymmetric tensor 
  field}]{antisymmetric} if \(\phi(\sigma) \circ a \circ \phi(\tau) = 
  (-1)^{\sigma} (-1)^\tau a\) for all \(\sigma \in \Sigma_m\) and 
  \(\tau \in \Sigma_n\). An \emDefOrd[{form field}*]{$n$-form field} 
  is an antisymmetric element of $\HomPROP_V(0,n)$.
  
  Since $\mc{R}$ contains ``bump functions''\Ldash for any open \(U 
  \subseteq M\) and \(p \in U\), there exists some open neighbourhood 
  \(U' \subset U\) of $p$ and \(B \in \mc{R}\) such that \(B(q)=1\) 
  for all \(q \in U'\) and \(B(q) = 0\) for all \(q \in M \setminus 
  U\)\Rdash it is possible to realise quantities locally defined on 
  some neighbourhood of a point $p$ as fields defined 
  on the whole of $M$, in such a way that these global fields at 
  least coincide with the original local field on a neighbourhood of 
  $p$. In particular, one can for every chart $(U,\varphi)$ (where 
  \(U \subseteq M\) is open and \(\varphi\colon U \Fpil \R^k\) is 
  injective) of the manifold $M$ define a set \(\{x_i\}_{i=1}^k 
  \subseteq \mc{R}\) of coordinate functions valid around \(p \in U\) 
  by
  \begin{equation}
    x_i(q) = \begin{cases}
      B(q) \cdot (u_i \circ \varphi)(q)& \text{for \(q \in U\),}\\
      0& \text{for \(q \in M \setminus U\),}
    \end{cases}
  \end{equation}
  where \(B \in \mc{R}\) is some suitable bump function with 
  \(B(p)=1\) supported within $U$ and \(u_i\colon \R^k \Fpil \R\) 
  denotes projection onto the $i$th component. A corresponding set 
  \(\{E_i\}_{i=1}^k \subseteq V\) of standard vectors can be defined 
  by
  \begin{equation}
    E_i(f)(q) = \begin{cases}
      B(q) \cdot 
        \dfrac{\partial (f \circ \varphi^{-1})}{\partial u_i}
        \bigl( \varphi(q) \bigr)&
      \text{for \(q \in U\),}\\
      0 & \text{for \(q \in M \setminus U\),}
    \end{cases}
  \end{equation}
  where $B$ is the same bump function and $\frac{\partial}{\partial 
  u_i}$ is the partial derivative with respect to the $i$ component 
  of $\R^k$. These scalar and vector fields locally constitute a 
  basis in the sense that
  \begin{equation}
    Y(f)(q) = \sum_{i=1}^k Y(x_i)(q) \cdot E_i(f)(q)
    \qquad\text{where \(B(q)=1\)}
  \end{equation}
  for arbitrary \(f \in \mc{R}\) and \(Y \in V\). This extends to 
  recovering the old ``multidimensional array of numbers'' view of 
  tensors, since any tensor field \(a \in \HomPROP_V(m,n)\) locally 
  determines component scalar fields via
  \begin{equation}
    a_{r_1,\dotsc,r_m,s_1,\dotsc,s_n} =
    \bigotimes_{i=1}^m \mathrm{d}x_{r_i} \circ a \circ 
    \bigotimes_{i=1}^n E_{s_i}
    \quad\text{where \(r_1,\dotsc,r_m,s_1,\dotsc,s_n \in [k]\)}
  \end{equation}
  and conversely
  \begin{equation}
    a = \sum_{r_1,\dotsc,r_m,s_1,\dotsc,s_n=1}^k 
      a_{r_1,\dotsc,r_m,s_1,\dotsc,s_n}
      \bigotimes_{i=1}^m E_{r_i} \otimes 
      \bigotimes_{i=1}^n \mathrm{d}x_{s_i}
    \qquad\text{near $p$}
  \end{equation}
  in the sense that
  \begin{multline*}
    \theta\Bigl( 
      a\left( \textstyle\bigotimes_{i=1}^n Y_i \right)
    \Bigr)(f_1,\dotsc,f_m)(q) 
    = \\ =
    \sum_{r_1,\dotsc,r_m,s_1,\dotsc,s_n=1}^k \!\!\!\!\!
      a_{r_1,\dotsc,r_m,s_1,\dotsc,s_n}(q) \cdot
      \prod_{i=1}^m E_{r_i}(f_i)(q) \cdot
      \prod_{i=1}^n Y_i(x_{s_i})(q)
  \end{multline*}
  for $q$ near $p$, all \(f_1,\dotsc,f_m \in \mc{R}\), and all 
  \(Y_1,\dotsc,Y_n \in V\).

\end{example}

\subsection{More \PROP\ examples}

The rewriting set-up works with linear \PROPs, since many theories 
require a linear structure, even those that do not often benefit 
theoretically from having one available, and in those remaining cases 
where it really isn't of any use it can mostly be ignored (just like 
word problems in semigroups can be studied in associative algebras as 
problems regarding ideals generated by binomial relations). This, of 
course, calls for a method of equipping arbitrary \PROPs\ with a 
linear structure.
The following is a generalisation of the construction of 
a group algebra, so it is natural that one uses the same notation for 
both, albeit with a \PROP\ in the place of the group.

\begin{construction} \label{Kons:R-linjar}
  Let a \PROP\ $\mc{P}$ and an associative commutative unital ring 
  $\mc{R}$ be given. Let $\mc{RP}$ be the $\N^2$-graded set for which 
  each $\mc{RP}(m,n)$ is the set of all formal $\mc{R}$-linear 
  combinations of elements of $\mc{P}(m,n)$. Extend $\circ$ and 
  $\otimes$ to $\mc{RP}$ by $\mc{R}$-bilinearity. Then $\mc{RP}$ is an 
  $\mc{R}$-linear \PROP, and each $\mc{RP}(m,n)$ is a free 
  $\mc{R}$-module with basis $\mc{P}(m,n)$.
  If $\mc{Q}$ is an $\mc{R}$-linear \PROP\ and \(f\colon \mc{P} \Fpil 
  \mc{Q}\) is a \PROP\ homomorphism, then $f$ has a unique extension 
  to an $\mc{R}$-linear \PROP\ homomorphism \(\mc{RP} \Fpil \mc{Q}\).
\end{construction}
\begin{proof}
  The construction of formal linear combinations of a set of 
  elements, and extension to it by linearity of operations, is well 
  known and requires no explanation here. What is less clear is 
  perhaps that the axioms survive this extension, but this is 
  straightforward to verify; e.g.~three general elements of 
  $\mc{RP}(k,l)$, $\mc{RP}(l,m)$, and $\mc{RP}(m,n)$ respectively 
  have the forms $\sum_{i=1}^A r_i a_i$, $\sum_{i=1}^B s_i b_i$, and 
  $\sum_{i=1}^C t_i c_i$ respectively, where
  \(\{a_i\}_{i=1}^A \subseteq \mc{P}(k,l)\), \(\{b_i\}_{i=1}^B 
  \subseteq \mc{P}(l,m)\), \(\{c_i\}_{i=1}^C \subseteq \mc{P}(m,n)\), 
  and \(\{r_i\}_{i=1}^A, \{s_i\}_{i=1}^B, \{t_i\}_{i=1}^C \subseteq 
  \mc{R}\). Hence it follows from
  \begin{multline*}
    \Biggl(
    \biggl( \sum_{i=1}^A r_i a_i \biggr) \circ 
    \biggl( \sum_{j=1}^B s_j b_j \biggr) 
    \Biggr) \circ
    \biggl( \sum_{h=1}^C t_h c_h \biggr) 
    = \\ =
    \biggl(
    \sum_{i=1}^A \sum_{j=1}^B r_i s_j ( a_i \circ  b_j ) 
    \biggr) \circ
    \biggl( \sum_{h=1}^C t_h c_h \biggr) 
    = \displaybreak[0]\\ =
    \sum_{i=1}^A \sum_{j=1}^B \sum_{h=1}^C r_i s_j t_h
    \bigl( ( a_i \circ  b_j ) \circ c_h \bigr) 
    =
    \sum_{i=1}^A \sum_{j=1}^B \sum_{h=1}^C r_i s_j t_h
    \bigl( a_i \circ ( b_j \circ c_h ) \bigr) 
    = \displaybreak[0]\\ =
    \biggl( \sum_{i=1}^A r_i a_i \biggr) \circ
    \biggl(
    \sum_{j=1}^B \sum_{h=1}^C s_j t_h ( b_j \circ c_h ) 
    \biggr) 
    = \\ =
    \biggl( \sum_{i=1}^A r_i a_i \biggr) \circ
    \Biggl(
    \biggl( \sum_{j=1}^B s_j b_j \biggr) \circ 
    \biggl( \sum_{h=1}^C t_h c_h \biggr) 
    \Biggr) 
  \end{multline*}
  that the composition associativity axiom holds also in $\mc{RP}$ as 
  a whole.
  
  Similarly, any map \(\mc{P}(m,n) \Fpil \mc{Q}(m,n)\) extends 
  uniquely to an $\mc{R}$-linear map \(\mc{RP}(m,n) \Fpil 
  \mc{Q}(m,n)\) by the basis property, so the extension of \(f\colon 
  \mc{P} \Fpil \mc{Q}\) to the whole of $\mc{RP}$ is uniquely 
  determined already by the $\mc{R}$-linearity condition. Checking 
  that this extension of $f$ continues to satisfy the definition of 
  a \PROP\ homomorphism is straightforward.
\end{proof}

Thus having some examples of $\mc{R}$-linear \PROPs, one may also 
take Example~\ref{Ex:Algebra-PROP} the other way and observe that, 
for any $\mc{R}$-linear \PROP\ $\mc{P}$:
\begin{itemize}
  \item
    for each \(n\in\N\), the component $\mc{P}(n,n)$ is an 
    associative unital $\mc{R}$-algebra with $\circ$ as 
    multiplication and $\phi(\same{n})$ as unit (except that 
    unitality fails where \(\mc{P}(n,n) = \{0\}\) as it would 
    require \(0 \neq \phi(\same{n})\));
  \item
    for each \(n\in\N\), the component $\mc{P}(n,n)$ for \(n\in\N\) 
    is an associative unital $\mc{P}(0,0)$-algebra with $\circ$ as 
    multiplication and $\otimes$ as scalar multiple operation 
    (except, as above, for the unitality catch should \(\mc{P}(n,n) = 
    \{0\}\));
  \item
    for all \(m,n\in\N\), the component $\mc{P}(m,n)$ is a left 
    $\mc{P}(m,m)$-module, right $\mc{P}(n,n)$-module, and 
    $\bigl( \mc{P}(m,m), \mc{P}(n,n) \bigr)$-bimodule 
    with $\circ$ as scalar multiple operation.
\end{itemize}
That composition on $\mc{P}(n,n)$ is $\mc{P}(0,0)$-bilinear is not 
directly a \PROP\ axiom, but is still quickly verified, as
\begin{multline*}
  (r \otimes a) \circ b = 
  r \otimes a \circ \phi(\same{0}) \otimes b =
  \bigl( r \circ \phi(\same{0}) \bigr) \otimes (a \circ b) =
  r \otimes (a \circ b) = \\ =
  \bigl( \phi(\same{0}) \circ r \bigr) \otimes (a \circ b) =
  \phi(\same{0}) \otimes a \circ r \otimes b =
  a \circ (r \otimes b)
\end{multline*}
for all \(r \in \mc{P}(0,0)\) and \(a,b \in \mc{P}(n,n)\).

Algebras can also be built using the tensor product of a \PROP. For 
any $\mc{R}$-linear \PROP\ $\mc{P}$, the three sets
\begin{align*}
  \mc{A} ={}& \bigoplus_{n\in\N} \mc{P}(n,0) \text{,}&
  \mc{B} ={}& \bigoplus_{n\in\N} \mc{P}(n,n) \text{,}&
  \mc{C} ={}& \bigoplus_{n\in\N} \mc{P}(0,n)
\end{align*}
are all associative, unital, graded $\mc{R}$-algebras with $\otimes$ 
as multiplication operation. The issue of whether $\mc{A}$ above is 
generated in degree $1$ figures in the following theorem.

\begin{theorem}
  Let $\mc{R}$ be an associative and commutative unital ring. Let 
  $\mc{P}$ be an $\mc{R}$-linear \PROP\ with elements \(m \in 
  \mc{P}(1,2)\), \(u \in \mc{P}(1,0)\), \(D \in \mc{P}(2,1)\), 
  \(e \in \mc{P}(0,1)\), and \(A \in \mc{P}(1,1)\) 
  such that
  \begin{align*}
    m \circ m \otimes \phi(\same{1}) ={}&
      m \circ \phi(\same{1}) \otimes m &
    D \otimes \phi(\same{1}) \circ D ={}&
      \phi(\same{1}) \otimes D \circ D \\
    m \circ u \otimes \phi(\same{1}) ={}& \phi(\same{1}) &
    e \otimes \phi(\same{1}) \circ D ={}& \phi(\same{1})\\
    m \circ \phi(\same{1}) \otimes u ={}& \phi(\same{1}) &
    \phi(\same{1}) \otimes e \circ D ={}& \phi(\same{1})\\
    D \circ m ={}&
      m \otimes m \circ 
      \phi(\same{1} \star \cross{1}{1} \star \same{1}) \circ
      D \otimes D 
      \mkern -70mu & \mkern +70mu
    e \circ u ={}& \phi(\same{0}) \\
    u \circ e ={}&
      m \circ \phi(\same{1})\otimes A \circ D &
    m \circ A \otimes \phi(\same{1}) \circ D ={}&
      u \circ e
    \text{.}
  \end{align*}
  If $\mc{P}$ is such that the $\mc{R}$-algebra $\bigoplus_{n\in\N} 
  \mc{P}(n,0)$ with $\otimes$ as multiplication is generated in 
  degree $1$, then $\mc{P}(1,0)$ is a Hopf algebra over $\mc{P}(0,0)$ 
  with operations
  \begin{align*}
    a \cdot b :={}& m \circ a \otimes b &&
      \text{for all \(a,b \in \mc{P}(1,0)\),}\\
    ra := {}& r \otimes a &&
      \text{for all \(r \in \mc{P}(0,0)\) and \(a \in \mc{P}(1,0)\),}\\
    \eta :={}& u \text{,}\\
    \Delta(a) :={}& D \circ a \in \mc{P}(1,0) \otimes \mc{P}(1,0) &&
      \text{for all \(a \in \mc{P}(1,0)\),}\\
    \ve(a) :={}& e \circ a \in \mc{P}(0,0) &&
      \text{for all \(a \in \mc{P}(1,0)\),}\\
    S(a) :={}& A \circ a &&
      \text{for all \(a \in \mc{P}(1,0)\).}
  \end{align*}
\end{theorem}
\begin{proof}
  Since $\bigoplus_{n\in\N} \mc{P}(n,0)$ is generated in degree 
  (i.e., coarity) $1$, it follows in particular that every element of 
  $\mc{P}(2,0)$ can be written as $\sum_{k=1}^n b_k \otimes c_k$ for 
  some \(\{b_k\}_{k=1}^n, \{c_k\}_{k=1}^n \subseteq \mc{P}(1,0)\), 
  which means \(\mc{P}(2,0) = \mc{P}(1,0) \otimes \mc{P}(1,0)\). 
  Hence $\Delta$ indeed maps $\mc{P}(1,0)$ into $\mc{P}(1,0) \otimes 
  \mc{P}(1,0)$, as a coproduct is supposed to.
  
  That the axioms for a Hopf algebra hold follow directly from 
  conditions in the theorem.
\end{proof}

Is this ``theorem'' anything but a mildly contorted restatement of 
the definition of a Hopf algebra, though? Perhaps not, but there is a 
point to it: \emph{the techniques presented in this paper make it 
trivial to construct any number of \PROPs\ $\mc{P}$ satisfying the 
list of equations in the theorem, and any list of additional 
equations that one might want to impose,} but this need not yield a 
Hopf algebra in the classical sense if it fails to meet the final 
condition about the nullary components being generated in coarity~$1$.

\medskip


One more \PROP\ that will be of interest is the following 
\emDefOrd[{biaffine PROP}{biaffine \PROP}]{biaffine} \PROP, so named 
because it contains both the \PROP\ of affine transformations (like some 
$\mc{R}^{\bullet\times\bullet}$ is 
the \PROP\ of linear transformations) and its dual (which could perhaps 
be called a \PROP\ of ``coaffine transformations'', even though that 
begs the question of what is being transformed). This extension of the 
ordinary matrix \PROP\ introduces nontrivial elements with $0$ arity 
or coarity.

\begin{example}[Biaffine \PROP]
  Let $\mc{R}$ be an associative unital semiring. The 
  \DefOrd[{biaffine PROP}{biaffine \PROP}]{biaffine \PROP} 
  \index{Baff(R)@$\Baff(\mc{R})$}$\Baff(\mc{R})$ over $\mc{R}$ can be 
  defined as the subset of $\mc{R}^{\bullet\times\bullet}$ of 
  matrices on the form
  \begin{equation} \label{Eq1:Biaffin}
    \begin{bmatrix}
      1 & d & \transpose{\vek{c}} \\
      0 & 1 & 0 \\
      0 & \vek{b} & A
    \end{bmatrix}
    \quad\text{where \(A \in \mc{R}^{m \times n}\), \(\vek{b} \in 
    \mc{R}^m\), \(\vek{c} \in \mc{R}^n\), and \(d \in \mc{R}\),}
  \end{equation}
  i.e.,
  \begin{equation}
    \Baff(\mc{R}) = \setOf{ T \in \mc{R}^{\bullet\times\bullet} }{
      \begin{aligned}
        \alpha_{\mc{R}^{\bullet\times\bullet}}(T) \geqslant{}& 2 
          \text{,}&
        T_{11} ={}& 1\text{,} &
        T_{i1} ={}& 0\text{ if \(i \neq 1\),}\\ 
        \omega_{\mc{R}^{\bullet\times\bullet}}(T) \geqslant{}& 2 
          \text{,}&
        T_{22} ={}& 1\text{,} &
        T_{2j} ={}& 0\text{ if \(j \neq 2\)}
      \end{aligned}
    }
    \text{.}
  \end{equation}
  The $A$, $\vek{b}$, $\vek{c}$, and $d$ parts of \eqref{Eq1:Biaffin} 
  may be described as the \emDefOrd[*{matrix part}]{matrix}, 
  \emDefOrd[*{vector part}]{vector}, 
  \emDefOrd[*{covector part}]{covector}, and 
  \emDefOrd[*{scalar part}]{scalar} respectively parts of this \PROP\ 
  element. 
  
  $\Baff(\mc{R})$ is not a sub-\PROP\ of $\mc{R}^{\bullet\times\bullet}$, 
  since the arities and coarities are distinct. Concretely,
  \begin{align}
    \alpha_{\Baff(\mc{R})}(T) ={}& 
      \alpha_{\mc{R}^{\bullet\times\bullet}}(T)-2\text{,} &
    \omega_{\Baff(\mc{R})}(T) ={}& 
      \omega_{\mc{R}^{\bullet\times\bullet}}(T)-2
  \end{align}
  for all \(T \in \Baff(\mc{R})\). The composition operation on 
  $\Baff(\mc{R})$ is however the same as in 
  $\mc{R}^{\bullet\times\bullet}$:
  \begin{equation}
    \begin{bmatrix}
      1 & d_1 & \transpose{\vek{c}_1} \\
      0 & 1 & 0 \\
      0 & \vek{b}_1 & A_1
    \end{bmatrix} 
    \circ
    \begin{bmatrix}
      1 & d_2 & \transpose{\vek{c}_2} \\
      0 & 1 & 0 \\
      0 & \vek{b}_2 & A_2
    \end{bmatrix}
    =
    \begin{bmatrix}
      1 & d_1 + \transpose{\vek{c}_1} \vek{b}_2 + d_2 
      & \transpose{\vek{c}_1} A_2 + \transpose{\vek{c}_2} \\
      0 & 1 & 0 \\
      0 & \vek{b}_1 + A_1\vek{b}_2 & A_1A_2
    \end{bmatrix}
    \text{;}
  \end{equation}
  and the permutation maps for $\Baff(\mc{R})$ can be defined in 
  terms of those for $\mc{R}^{\bullet\times\bullet}$:
  \begin{equation}
    \phi_{\Baff(\mc{R})}(\sigma) = 
    \phi_{\mc{R}^{\bullet\times\bullet}}(\same{2} \star \sigma) =
    \begin{bmatrix}
      1 & 0 & 0 \\
      0 & 1 & 0 \\
      0 & 0 & \phi_{\mc{R}^{\bullet\times\bullet}}(\sigma)
    \end{bmatrix}
  \end{equation}
  for all permutations $\sigma$. What is not possible to state purely 
  in terms of its $\mc{R}^{\bullet\times\bullet}$ counterpart is the 
  tensor product, since the definition
  \begin{equation} \label{Eq:Biaffin-tensor}
    \begin{bmatrix}
      1 & d_1 & \transpose{\vek{c}_1} \\
      0 & 1 & 0 \\
      0 & \vek{b}_1 & A_1
    \end{bmatrix} 
    \otimes
    \begin{bmatrix}
      1 & d_2 & \transpose{\vek{c}_2} \\
      0 & 1 & 0 \\
      0 & \vek{b}_2 & A_2
    \end{bmatrix}
    =
    \begin{bmatrix}
      1 & d_1 + d_2 & \transpose{\vek{c}_1}
      & \transpose{\vek{c}_2} \\
      0 & 1 & 0 & 0\\
      0 & \vek{b}_1 & A_1 & 0 \\
      0 & \vek{b}_2 & 0 & A_2
    \end{bmatrix}
  \end{equation}
  of this puts more than one nonzero block in some rows and columns. 
  There are nonetheless strong similarities between the two 
  operations, and if one denotes the tensor product of 
  $\mc{R}^{\bullet\times\bullet}$ by $\oplus$ then it may in 
  particular be observed that \(\Psi(A) = \psm{1&0\\0&1} \oplus A\) 
  is an embedding of $\mc{R}^{\bullet\times\bullet}$ in 
  $\Baff(\mc{R})$; \(\Psi(A_1 \oplus\nobreak A_2) = \psm{1&0\\0&1} 
  \oplus A_1 \oplus A_2 = \Psi(A_1) \otimes \Psi(A_2)\) and 
  \(\Psi(A_1 \circ\nobreak A_2) = \psm{1&0\\0&1} \oplus A_1A_2 = 
  \Psi(A_1) \circ \Psi(A_2)\). This suffices for verifying the 
  permutation composition and juxtaposition axioms.
  
  The composition associativity and identity axioms follow from the 
  fact that these hold in $\mc{R}^{\bullet\times\bullet}$. The tensor 
  product associativity and identity axioms are obvious from 
  \eqref{Eq:Biaffin-tensor}. The two remaining axioms are 
  straightforward to verify through explicit calculations, e.g.
  \begin{multline*}
    \left(
      \begin{bmatrix}
        1 & d_1 & \transpose{\vek{c}_1} \\
        0 & 1 & 0 \\
        0 & \vek{b}_1 & A_1
      \end{bmatrix}
      \circ
      \begin{bmatrix}
        1 & d_2 & \transpose{\vek{c}_2} \\
        0 & 1 & 0 \\
        0 & \vek{b}_2 & A_2
      \end{bmatrix}
    \right) \otimes \left(
      \begin{bmatrix}
        1 & d_3 & \transpose{\vek{c}_3} \\
        0 & 1 & 0 \\
        0 & \vek{b}_3 & A_3
      \end{bmatrix}
      \circ
      \begin{bmatrix}
        1 & d_4 & \transpose{\vek{c}_4} \\
        0 & 1 & 0 \\
        0 & \vek{b}_4 & A_4
      \end{bmatrix}
    \right)
    = \\ =
    \begin{bmatrix}
      1 & d_1 + \transpose{\vek{c}_1} \vek{b}_2 + d_2 
      & \transpose{\vek{c}_1} A_2 + \transpose{\vek{c}_2} \\
      0 & 1 & 0 \\
      0 & \vek{b}_1 + A_1\vek{b}_2 & A_1A_2
    \end{bmatrix}
    \otimes
    \begin{bmatrix}
      1 & d_3 + \transpose{\vek{c}_3} \vek{b}_4 + d_4 
      & \transpose{\vek{c}_3} A_4 + \transpose{\vek{c}_4} \\
      0 & 1 & 0 \\
      0 & \vek{b}_3 + A_3\vek{b}_4 & A_3A_4
    \end{bmatrix}
    = \displaybreak[0]\\ =
    \begin{bmatrix}
      1 & \begin{aligned}[c]
          d_1 +{}& d_3 + \transpose{\vek{c}_1} \vek{b}_2 + \\
          &\transpose{\vek{c}_3} \vek{b}_4 + d_2 + d_4
        \end{aligned}
      & \transpose{\vek{c}_1} A_2 + \transpose{\vek{c}_2}
      & \transpose{\vek{c}_3} A_4 + \transpose{\vek{c}_4} \\
      0 & 1 & 0 & 0\\
      0 & \vek{b}_1 + A_1\vek{b}_2 & A_1A_2 & 0 \\
      0 & \vek{b}_3 + A_3\vek{b}_4 & 0 & A_3A_4
    \end{bmatrix}
    = \displaybreak[0]\\ =
    \begin{bmatrix}
      1 & d_1+d_3 & \transpose{\vek{c}_1} & \transpose{\vek{c}_3} \\
      0 & 1 & 0 & 0 \\
      0 & \vek{b}_1 & A_1 & 0 \\
      0 & \vek{b}_3 & 0 & A_3
    \end{bmatrix}
    \circ
    \begin{bmatrix}
      1 & d_2+d_4 & \transpose{\vek{c}_2} & \transpose{\vek{c}_4} \\
      0 & 1 & 0 & 0 \\
      0 & \vek{b}_2 & A_2 & 0 \\
      0 & \vek{b}_4 & 0 & A_4
    \end{bmatrix}
    = \\ =
    \begin{bmatrix}
      1 & d_1 & \transpose{\vek{c}_1} \\
      0 & 1 & 0 \\
      0 & \vek{b}_1 & A_1
    \end{bmatrix}
    \otimes
    \begin{bmatrix}
      1 & d_3 & \transpose{\vek{c}_3} \\
      0 & 1 & 0 \\
      0 & \vek{b}_3 & A_3
    \end{bmatrix}
    \circ
    \begin{bmatrix}
      1 & d_2 & \transpose{\vek{c}_2} \\
      0 & 1 & 0 \\
      0 & \vek{b}_2 & A_2
    \end{bmatrix}
    \otimes
    \begin{bmatrix}
      1 & d_4 & \transpose{\vek{c}_4} \\
      0 & 1 & 0 \\
      0 & \vek{b}_4 & A_4
    \end{bmatrix}
    \text{.}
  \end{multline*}
  
  An observation about the biaffine \PROP\ which is useful when 
  evaluating complicated expressions is that it can be thought of as 
  counting paths in a network. The elements of the matrix part keep 
  track of paths passing through it, entering through the input 
  corresponding to the column and exiting through the output 
  corresponding to the row. The elements of the vector part keep 
  track of paths which \textbf{b}egin within the network but leave 
  it, with a separate row for each output, and conversely the 
  elements of the \textbf{c}ovector part keep track of paths which 
  enter the network but terminate within it, with a separate column 
  for each input. Finally the scalar part keeps track of paths which 
  both begin and end within the network, never leaving it. It is not 
  hard to see that the composition law corresponds precisely to 
  composing two networks such that the inputs of the left factor are 
  identified with the outputs of the right factor, and the tensor 
  product law corresponds to putting two networks next to each other.
\end{example}

\section{Relations}
\label{Sec:Ordning}

Rewriting theory makes a greater use of relations as mathematical 
objects than many other branches of mathematics; this is in part due 
to that it operates in a setting where very little is given in terms 
of operations and their properties, so one has to make do with more 
fundamental and generic concepts. Two types of relation that will be 
of particular interest here are congruence relations and partial 
orders. Congruence relations, as the name suggests, can express the 
property that two expressions are congruent modulo some set of given 
identities. Partial orders are used to encode the fact that one 
expression is ``simpler'' than another. Frequently these types are 
unified into a \emph{rewriting relation} $\rightarrow$, where `\(a 
\rightarrow b\)' might be interpreted as `$a$ is congruent to $b$, but 
the latter is simpler' (on account of being the result of applying a 
single step reduction to $a$). However, in that particular role 
the rewriting formalism of~\cite{rpaper} employs maps rather than 
relations, so that is not a unification that will be made here. On 
the other hand, it turns out that both congruence relations and the 
partial orders of interest may be defined as special cases of \PROP\ 
quasi-orders, which saves some work below.

As usual, a binary relation $P$ on some set $S$ is considered to be a 
subset of $S \times S$, but \((x,y) \in P\) is often an impractical 
(and opaque) piece of notation for the common case of order 
relations, so instead I'll typically write `\(x \leqslant y \pin{P}\)'
\index{in!leqslant@$x \leqslant y \pin{P}$} 
to mean the same thing. This notation has the advantage of clarifying 
for the reader that it is $x$ that is on the ``small side'', and it 
also allows the variations
\begin{alignat*}{3}
  x \geqslant{}& y \pin{P} &\quad&\Epil& \quad
    y \leqslant{}& x \pin{P} \text{,}
    \index{in!geqslant@$x \geqslant y \pin{P}$}\\
  x <{}& y \pin{P}  &&\Epil& 
    x \leqslant{}& y \pin{P} \text{ and } 
    y \not\leqslant x \pin{P} \text{,}
    \index{in!lt@$x < y \pin{P}$}
    \displaybreak[0]\\
  x >{}& y \pin{P} &&\Epil& 
    y <{}& x \pin{P} \text{,}
    \index{in!gt@$x > y \pin{P}$}\\
  x \sim{}& y \pin{P}  &&\Epil& 
    x \leqslant{}& y \pin{P} \text{ and } 
    y \leqslant x \pin{P} 
    \index{in!sim@$x \sim y \pin{P}$}
\end{alignat*}
without introducing any new per-relation symbols. If several such 
relations appear in sequence, e.g.~\(x \leqslant y < z \pin{P}\), 
then this is primarily to be read as a shorthand for the conjunction 
of the individual relations, i.e.,~`\(x \leqslant y \pin{P}\) and 
\(y < z \pin{P}\)', but transitivity may of course permit one to also 
draw conclusions about the relation between nonadjacent steps. 
If $P$ is a partial order then \(x \sim y \pin{P}\) is the same 
thing as \(x=y\), but if $P$ is a more general quasi-order then 
this need not be the case. Quasi-orders are of interest in that 
context because they often occur as intermediate steps in the 
construction of specific partial orders.

A quasi-order $P$ is an equivalence relation if and only if it is 
symmetric. The $P$-equivalence class $[x]_P$ of \(x \in S\) is the 
set of all \(y \in S\) such that \(x \sim y \pin{P}\). The 
\DefOrd{quotient} of $S$ by an equivalence relation $P$ is denoted 
$S/P$\index{/@$/$!quotient}, and is as usual the set 
$\setOf[\big]{ [x]_P }{ x \in S }$ of all $P$-equivalence classes.

Let \(T \subseteq S\). An element \(x \in T\) is said to be 
\DefOrd[{minimal}*]{$P$-minimal} in $T$ if there is no \(y \in T\) 
such that \(y < x \pin{P}\). The quasi-order $P$ is said to be 
\DefOrd{well-founded} 
if every nonempty \(T \subseteq S\) contains an element which is 
$P$-minimal in $T$. Well-founded quasi-orders support induction 
arguments of the form
\begin{quote}
  if \(R \subseteq S\) has the property that
  \begin{quote}
    \(x \in R\) whenever all \(y < x \pin{P}\) satisfy \(y \in R\), 
  \end{quote}
  then \(R = S\);
\end{quote}
the proof is to consider \(T = S \setminus R\), since if that had 
been nonempty then it would contain a $P$-minimal element $x$, which 
would contradict the hypothesis. Well-foundedness is also known as 
satisfying the \DefOrd{descending chain condition}; a descending 
chain is then an infinite sequence \(\{x_i\}_{i=1}^\infty \subseteq 
S\) such that \(x_i \geqslant x_{i+1} \pin{P}\) for all \(i \in 
\Zp\), and the condition is that there must be some $n$ such that 
\(x_n \sim x_i \pin{P}\) for all \(i \geqslant n\). Alternatively one 
may state it as: there are no infinite strictly $P$-descending 
chains\Dash any sequence \(x_1 > x_2 > \dotsb \pin{P}\) must 
terminate after a finite number of steps. This form is convenient 
when proving that something is an algorithm, since any loop where at 
each iteration some quantity strictly $P$-decreases must be finite 
when $P$ is well-founded.

\begin{definition}
  An \DefOrd[*{N2-graded@$\N^2$-graded!quasi-order}]{$\N^2$-graded 
  quasi-order} $Q$ on an $\N^2$-graded set 
  $\mc{P}$ is a quasi-order such that if \(x \leqslant y \pin{Q}\) 
  then \(\alpha(x)=\alpha(y)\) and \(\omega(x)=\omega(y)\). It is 
  sometimes useful to let 
  \index{(m,n)@$\cdot(m,n)$!restriction of relation}$Q(m,n)$ denote 
  the restriction of $Q$ to $\mc{P}(m,n)$.
  
  A \DefOrd[*{PROP@\PROP!quasi-order}]{\PROP\ quasi-order} $Q$ is an 
  $\N^2$-graded quasi-order on a \PROP\ $\mc{P}$ which is compatible 
  with the \PROP\ operations, i.e., firstly
  \begin{subequations}
    \begin{equation} \label{Eq:PROP-order,comp}
      a \leqslant b \pin{Q(l,m)} \Ipil
      c \circ a \circ d \leqslant c \circ b \circ d \pin{Q(k,n)}
    \end{equation}
    for all \(a,b \in \mc{P}(l,m)\), \(c \in \mc{P}(k,l)\), and 
    \(d \in \mc{P}(m,n)\), and secondly
    \begin{equation}
      a \leqslant b \pin{Q(k,l)} \Ipil
      c \otimes a \otimes d \leqslant c \otimes b \otimes d 
      \pin{Q(i+k+m,j+l+n)}
    \end{equation}
  \end{subequations}
  for all \(a,b \in \mc{P}(k,l)\), \(c \in \mc{P}(i,j)\), and \(d \in 
  \mc{P}(m,n)\), for all \(i,j,k,l,m,n \in \N\). A \PROP\ quasi-order 
  is called a 
  \DefOrd[*{PROP@\PROP!congruence relation}]{\PROP\ congruence relation} 
  if it is symmetric, i.e., if \(a \leqslant b\) implies \(a \geqslant 
  b\). It is an \DefOrd[*{linear@$\mc{R}$-linear!PROP congruence 
  relation@\PROP\ congruence relation}]{$\mc{R}$-linear} \PROP\ 
  congruence relation if $\mc{P}$ is $\mc{R}$-linear and additionally 
  for any \(m,n \in \N\), \(r \in \mc{R}\), and \(a,b,c \in 
  \mc{P}(m,n)\) such that \(a \sim b \pin{Q(m,n)}\) it holds that 
  \(ra \sim rb \pin{Q(m,n)}\) and \(a+c \sim b+c \pin{Q(m,n)}\).
  
  A \PROP\ quasi-order $Q$ is said to be \DefOrd{strict} if the 
  \PROP\ operations preserve strict inequalities, i.e., if firstly
  \begin{subequations}
    \begin{equation} \label{Eq:PROP-order,comp,str}
      a < b \pin{Q(l,m)} \Ipil 
      c \circ a \circ d < c \circ b \circ d \pin{Q(k,n)}
    \end{equation}
    for all \(a,b \in \mc{P}(l,m)\), \(c \in \mc{P}(k,l)\), and \(d \in 
    \mc{P}(m,n)\), and secondly
    \begin{equation}
      a < b \pin{Q(k,l)} \Ipil 
      c \otimes a \otimes d < c \otimes b \otimes d \pin{Q(i+k+m,j+l+n)}
    \end{equation}
  \end{subequations}
  for all \(a,b \in \mc{P}(k,l)\), \(c \in \mc{P}(i,j)\), and \(d \in 
  \mc{P}(m,n)\), for all \(i,j,k,l,m,n \in \N\).
  
  $\N^2$-graded partial orders, \PROP\ partial orders, and strict 
  \PROP\ partial orders are ditto quasi-orders which additionally are 
  partial orders on the underlying set.
\end{definition}

The strict \PROP\ partial order concept is a generalisation of the 
monoid partial order concept, and will play a similar role in the 
theory. Unlike the case with monoids however, there are no strict 
\PROP\ total orders, and with the exception of certain degenerate 
\PROPs~(such as that in Example~\ref{Ex:Algebra-PROP}) there cannot 
even be \PROP\ partial orders that are total within each component. 
The reason for this is that any two elements that differ only by 
action of permutations have to be incomparable or equivalent, since 
if it could happen that \(a < a \circ \phi(\tau)\) then it would follow 
from \eqref{Eq:PROP-order,comp,str} that \(a < a \circ \phi(\tau) < 
a \circ \phi(\tau) \circ \phi(\tau)\), and thus by extension that 
\(a < a \circ \phi(\tau^n)\) for every \(n>0\), but because every 
permutation has finite order there exists some \(n>0\) for which 
$\tau^n$ is the identity, and from that one would get the contradiction 
\(a < a\). Still, examples of strict \PROP\ orders are not entirely 
trivial to construct, so for a first example of strictness one may 
be best off with the following, even though it is ``degenerate'' in 
the sense of mapping all permutations to the same thing.

\begin{example}
  Let $M$ be a commutative monoid. Build from it an $\N^2$-graded set 
  $\mc{P}$ by making \(\mc{P}(n,n) = M\) for all \(n \in \N\) and 
  \(\mc{P}(m,n) = \varnothing\) if \(m \neq n\). Define a \PROP\ 
  structure on $\mc{P}$ by making \(\phi_n(\sigma) = 1 \in 
  \mc{P}(n,n)\), \(a \circ b := ab \in \mc{P}(n,n)\), and \(a \otimes 
  c := ac \in \mc{P}(m +\nobreak n, m +\nobreak n)\) for all \(\sigma 
  \in \Sigma_n\), \(a,b \in \mc{P}(n,n)\) and \(c \in \mc{P}(m,m)\). 
  The axioms are trivial to verify.
  
  Let \(f\colon \mc{P} \Fpil M\) be the map which forgets arity and 
  coarity of elements. If $Q$ is a monoid quasi-order on $M$ then the 
  binary relation $P$ on $\mc{P}$ defined by
  \begin{equation}
    a \leqslant b \pin{P} \Epil
    \text{\(\alpha(a)=\alpha(b)\), \(\omega(a) = \omega(b)\), and 
    \(f(a) \leqslant f(b) \pin{Q}\)}
  \end{equation}
  is a \PROP\ quasi-order. If \(a < b \pin{Q}\) implies \(ac < bc 
  \pin{Q}\) for all \(c \in M\) then $P$ is a strict \PROP\ 
  quasi-order.
\end{example}

A \PROP\ which is not ``degenerate'' in the above sense of treating 
all permutations like the identity, and which comes with a strict 
order, is the ``connectivity \PROP'' below. The 
reader may however prefer to skip over it on a first read-through, 
since even though it can be defined in quite elementary terms, the 
explanation of \emph{what} is being compared and \emph{why} the 
definition yields a \PROP\ relies rather heavily on the network 
notation of Section~\ref{Sec:Natverk} and its close relationship with 
the free \PROP\ (Section~\ref{Sec:FriPROP}).

\begin{example}[Connectivity \PROP] \label{Ex:Sammanhangande}
  For any \(m,n \in \N\), let $\mc{C}(m,n)$ be the set of all pairs 
  $(B,c)$, where \(c \in \N\) and $B$ is a partition of \(L(m,n) := 
  \bigl( \{0\} \times [m] \bigr) \cup \bigl( \{1\} \times [n] \bigr)\), 
  i.e., $B$ is a set of 
  pairwise disjoint subsets of $L(m,n)$ whose union is the whole of 
  $L(m,n)$. Let \((B_1,c_1) \leqslant (B_2,c_2) \pin{P(m,n)}\) iff 
  \(c_1 \leqslant c_2\) and there for every \(A_1 \in B_1\) exists 
  some \(A_2 \in B_2\) such that \(A_1 \subseteq A_2\); as a poset, 
  this implies $\mc{C}(m,n)$ is isomorphic to the cartesian product 
  by $\N$ of the partition lattice for an $(m +\nobreak n)$-set. An 
  element of $\mc{C}(m,n)$ is more connected the more elements of 
  $L(m,n)$ are in the same block of the partition.
  
  As it happens, this natural number part makes it possible to turn 
  $\mc{C}(m,n)$ into a \PROP\ on which $P$ is a strict partial order. 
  The idea of this \PROP\ is to keep track of the connectivity of 
  the graph underlying an expression, but it can be defined entirely 
  in the terms of Definition~\ref{Def:PROP}:
  \begin{itemize}
    \item
      For any \(\sigma \in \Sigma_n\),
      \begin{equation*}
        \phi(\sigma) = \Biggl(
          \setOf[\bigg]{ 
            \Bigl\{ \bigl( 0, \sigma(j) \bigr) , (1,j) \bigr) \Bigr\}
          }{ j \in [n] }
        , 0 \Biggr) \text{,}
      \end{equation*}
      i.e., the partition consists of $n$ blocks of two elements 
      each\Ldash item $(1,j)$ goes with item $\bigl( 0, 
      \sigma(j)\bigr)$\Rdash and the $c$ part is $0$.
      
    \item
      The tensor product
      \begin{equation*}
        (B_1,c_1) \otimes (B_2,c_2) := 
        (B_1 \cup\nobreak B_2', c_1 +\nobreak c_2)
        \text{,}
      \end{equation*}
      where $B_2'$ is $B_2$ shifted so that it becomes disjoint 
      from $B_1$; if \((B_1,c_1) \in \mc{C}(k,l)\), then
      \begin{equation*}
        B_2' := \Bigl\{ 
        \setOf[\big]{ (0,k+i) }{ (0,i) \in A } \cup
        \setOf[\big]{ (1,l+j) }{ (1,j) \in A }
      \Bigr\}_{A \in B_2} \text{.}
      \end{equation*}
      
    \item
      Finally, if \((B_1,c_1) \in \mc{C}(l,m)\) and \((B_2,c_2) \in 
      \mc{C}(m,n)\), then
      \begin{equation} \label{Eq:KonnektivitetSammansattning}
        (B_1,c_1) \circ (B_2,c_2) := \bigl( B',
          c_1 + m + \card{B} - \card{B_1} - \card{B_2} + c_2 
        \bigr) \text{,}
      \end{equation}
      where $B$ and $B'$ are obtained by taking a transitive closure 
      of $B_1$ and $B_2$. In more detail, let \(S = \bigl( 
      \{0\} \times [l] \bigr) \cup \bigl( \{\frac{1}{2}\} \times [m] 
      \bigr) \cup \bigl( \{1\} \times [n] \bigr)\). Let $E_1$ 
      be the equivalence relation on $S$ for which \((x,i) \sim 
      (y,j)\) iff (i)~\((x,i)=(y,j)\) or (ii)~$(2x,i)$ and $(2y,j)$ 
      are in the same block of $B_1$; this leaves every \((1,j) \in 
      S\) in a singleton equivalence class. Similarly let $E_2$ be 
      the equivalence relation on $S$ for which \((x,i) \sim (y,j)\) 
      iff (i)~\((x,i)=(y,j)\) or (ii)~$(2x -\nobreak 1,i)$ and $(2y 
      -\nobreak 1,j)$ are in the same block of $B_2$; here it is 
      instead the \((0,j) \in S\) which end up as singletons. Let $E$ 
      be the transitive closure of $E_1 \cup E_2$, and let $B$ be the 
      set of equivalence classes of $E$. Let $B'$ be the set of 
      equivalence classes of the restriction of $E$ to $L(l,n)$.
  \end{itemize}
\end{example}
\begin{proof}[Explanation]
  In order to understand what is going on here, it is convenient to 
  introduce an auxiliary $\N^2$-graded set $\mc{G}$ of ``graphs 
  with terminals''; an element of $\mc{G}(m,n)$ consists of one 
  pseudograph $G$ (i.e., it may have loops and multiple edges, 
  although one could make do without them) and two lists of $m$ and 
  $n$ respectively vertices (where one does not need to require the 
  vertices to be distinct, although one could do so). Apart from 
  these lists, the graph is unlabelled; one may freely rename 
  vertices and edges, as long as the lists are updated accordingly. 
  Denote by $u_i(G)$ the $i$'th vertex ($i\in[m]$) in the first list 
  of $G$ and by $v_j(G)$ the $j$'th vertex ($j \in [n]$) in the 
  second list of $G$.
  
  There is a natural definition of the \PROP\ operations on $\mc{G}$. 
  Both when forming the composition \(H_\circ = G_1 \circ G_2$ and the 
  tensor product \(H_\otimes = G_1 \otimes G_2\) of \(G_1 \in 
  \mc{G}(k,l)\) and \(G_2 \in \mc{G}(m,n)\), one starts by forming the 
  disjoint union \(G = G_1 \mathbin{\dot{\cup}} G_2\) of $G_1$ and $G_2$ 
  (this is where relabelling vertices and edges tends to become 
  necessary, but we'll assume the labellings started out disjoint). As 
  a pseudograph, \(H_\otimes = G\), and the lists of vertices are 
  given by
  \begin{equation*}
    u_i(H_\otimes) = \begin{cases}
      u_i(G_1) & \text{if \(i \leqslant k\),}\\
      u_{i-k}(G_2) & \text{otherwise,}
    \end{cases}
    \qquad
    v_j(H_\otimes) = \begin{cases}
      v_j(G_1) & \text{if \(j \leqslant l\),}\\
      v_{j-l}(G_2) & \text{otherwise}
    \end{cases}
  \end{equation*}
  for all \(i \in [ k +\nobreak m ]\) and \(j \in [l +\nobreak n]\). 
  The composition $H_\circ$ is instead $G + \bigl\{ v_j(G_1)u_j(G_2) 
  \bigr\}_{j=1}^m$\Ldash i.e., $G$ with \(m=l\) extra edges connecting 
  the $G_1$ and $G_2$ parts\Rdash then taking the first list from $G_1$ 
  and the second from $G_2$, so that \(u_i(H_\circ) = u_i(G_1)\) and 
  \(v_j(H_\circ) = v_j(G_2)\). Finally, $\phi_n(\sigma)$ is a graph 
  with vertex set $\{0,1\} \times [n]$, lists given by \(u_i = 
  (0,i)\) and \(v_j = (1,j)$, and $n$ edges connecting $(1,j)$ with 
  $\bigl( 0, \sigma(j) \bigr)$ (i.e., the graph $\phi_n(\sigma)$ is a 
  matching). $\mc{G}$ is not quite a \PROP\ under these 
  operations\Ldash for example the composition identity axiom 
  fails\Rdash but it is pretty close, and would be a \PROP\ modulo 
  graph homeomorphism.
  
  One graph invariant that is relevant for this discussion is the 
  \emph{cyclomatic number} $\mathrm{cyc}(G)$, i.e., the dimension of the 
  cycle space (or for the more topologically oriented, the rank of 
  the homotopy group). In a graph $G$, the easiest way to compute this 
  is to first fix a maximal spanning forest, and then simply count the 
  number of edges not in the forest; every such edge must connect two 
  vertices in the same component of $G$, so there is a unique path 
  between them in the forest, and together with the edge this path 
  forms a cycle. The set of such cycles constitutes a basis for the 
  cycle space, the dimension of which is therefore equal to the 
  number of edges outside the forest. It is well known that the 
  number of edges in a tree is one less than the number of vertices. 
  The corresponding result for a forest is that \(\card{E} = 
  \card{V} - k\), where $k$ denotes the number of components. Hence 
  \(\mathrm{cyc}(G) = \card[\big]{E(G)} - \card[\big]{V(G)} + k(G)\). 
  Applying this to the $\circ$ operation, one finds that
  \begin{multline*}
    \mathrm{cyc}(G_1 \circ G_2) =
    \card[\big]{E(G_1 \circ G_2)} - \card[\big]{V(G_1 \circ G_2)} + 
      k(G_1 \circ G_2) = \\ =
    \card[\big]{E(G_1)} + m + \card[\big]{E(G_2)} - 
      \card[\big]{V(G_1)} - \card[\big]{V(G_2)} + k(G_1 \circ G_2) 
      = \\ =
    \mathrm{cyc}(G_1) - k(G_1) + m + \mathrm{cyc}(G_2) - k(G_2) + 
      k(G_1 \circ G_2)
    \text{.}
  \end{multline*}
  If one makes the interpretation of the connectivity \PROP\ $\mc{C}$ 
  that the $B$ part of an element is its set of components, then this 
  formula is exactly the $c$ part of 
  \eqref{Eq:KonnektivitetSammansattning}! What the tensor product and 
  permutations do with the $c$ part also coincides with what happens 
  to the cyclomatic number, so the connectivity \PROP\ may in fact be 
  viewed as a destillation of $\mc{G}$ down to the minimal amount of 
  information needed to keep track of what happens to the cyclomatic 
  number under \PROP\ operations. It is because both the $B$ and $c$ 
  parts can be considered properties of elements of the free \PROP\ 
  that $\mc{C}$ must be a \PROP.
  
  From the \PROP\ order construction point of view, it is however 
  rather the component partition $B$ that is primary, whereas the 
  integer $c$ is something extra that has to be put in to make the 
  order strict. It is very natural to order expressions by how much 
  different parts of them are interconnected, as an expression with 
  fewer connections has fewer restrictions on how its parts can be 
  rearranged, and is therefore simpler. A natural measure of 
  connectivity is the partition of inputs and outputs into 
  components, but an order by this alone will not be strict: 
  if \(\alpha(\mOp)=\omega(\Delta)=2\), \(\omega(\mOp) = \omega(u) = 
  \alpha(\Delta) = 1\), and \(\alpha(u)=0\) then $\mOp \circ 
  \phi(\same{1}) \otimes u$ and $\mOp \circ \Delta$ both have the input 
  connected to the output, but $\phi(\same{1}) \otimes u$ is less 
  connected than $\Delta$, so a strict order should make \(\mOp \circ 
  \phi(\same{1}) \otimes u < \mOp \circ \Delta\). Such a 
  distinguishing detail can be hidden away deep within a large 
  expression, but remarkably enough the cyclomatic number suffices as 
  a way of remembering the difference.
\end{proof}

Quite a lot of useful orders can be constructed by comparing the 
images of elements under some map to a different set (usually a 
``simpler'') set. Hence a definition may be in order.

\begin{definition}
  Let sets $A$ and $B$, a map \(f\colon A \Fpil B\), and a binary 
  relation $P$ on $B$ be given. The \DefOrd{pullback} of $P$ over $f$ 
  is then the binary relation $f^*P$
  \index{* sup@${}^*$!pullback@pullback ($f^*P$)} on $A$ defined by
  \begin{equation}
    x \leqslant y \pin{f^*P} \Epil f(x) \leqslant f(y) \pin{P}
    \qquad\text{for all \(x,y \in A\).}
  \end{equation}
  The \DefOrd[*{kernel}]{(set-theoretic) kernel} 
  \index{Ker@$\Ker$}$\Ker f$ of $f$ is the relation 
  $f^*E$, where $E$ denotes the equality relation on $B$.
\end{definition}

The more traditional algebraic kernel $\ker f$ of a homomorphism $f$, 
which is generally the $\Ker f$ equivalence class of $0$, is not as 
immediately applicable for \PROPs\ since the most common way of 
developing its theory relies on having a group structure, and not even 
linear \PROPs\ are groups in that way. One can however define the 
algebraic kernel as the $\N^2$-graded set of all equivalence classes 
of $0$ (noticing that there is a separate $0$ in each component) and 
then go on to prove that it is a sub-\PROP, after which the theory 
develops as one would expect. Working instead with relations has the 
merit that this observation is automatic (and in addition it works also 
for non-linear \PROPs).

\begin{lemma} \label{L:Pullback}
  Let $\N^2$-graded sets $\mc{P}$ and $\mc{Q}$, a map \(f\colon 
  \mc{P} \Fpil \mc{Q}\), and a binary relation $Q$ on $\mc{Q}$ be 
  given. Consider the pullback $f^*Q$ of $Q$ to $\mc{P}$.
  \begin{enumerate}
    \item
      If $Q$ is reflexive then $f^*Q$ is reflexive.
    \item
      If $Q$ is symmetric then $f^*Q$ is symmetric.
    \item
      If $Q$ is transitive then $f^*Q$ is transitive.
    \item
      If $Q$ is an $\N^2$-graded quasi-order and $f$ is an 
      $\N^2$-graded set morphism then $f^*Q$ is an $\N^2$-graded 
      quasi-order.
    \item
      If $\mc{Q}$ is a \PROP, $Q$ is a \PROP\ quasi-order, $\mc{P}$ 
      is a \PROP, and $f$ is a \PROP\ homomorphism then $f^*Q$ is a 
      \PROP\ quasi-order. If in addition $Q$ is strict then $f^*Q$ is 
      strict.
    \item
      If $Q$ is well-founded then $f^*Q$ is well-founded.
  \end{enumerate}
\end{lemma}
\begin{proof}
  Reflexivity, symmetry, and transitivity are trvial; for example if 
  \(a \leqslant b \pin{f^*Q}\) and \(b \leqslant c \pin{f^*Q}\) then 
  by definition \(f(a) \leqslant f(b) \pin{Q}\) and \(f(b) \leqslant 
  f(c) \pin{Q}\), from which follows \(f(a) \leqslant f(c) \pin{Q}\) 
  by transitivity, and hence \(a \leqslant c \pin{f^*Q}\) as required 
  for $f^*Q$ to be transitive. Similarly for $\N^2$-grading: if \(a 
  \leqslant b \pin{f^*Q}\) then \(f(a) \leqslant f(b) \pin{Q}\) and 
  hence \(\alpha(a) = \alpha\bigl( f(a) \bigr) = \alpha\bigl( f(b) 
  \bigr) = \alpha(b)\) and \(\omega(a) = \omega\bigl( f(a) \bigr) = 
  \omega\bigl( f(b) \bigr) = \omega(b)\) as claimed.
  
  Proving that $f^*$ is a \PROP\ order requires a bit more work, but 
  if \(a \leqslant b \pin{f^*Q}\) then \(f(c \otimes\nobreak a 
  \otimes\nobreak d) = f(c) \otimes f(a) \otimes f(d) \leqslant 
  f(c) \otimes f(b) \otimes f(d) = f(c \otimes\nobreak b 
  \otimes\nobreak d) \pin{Q}\) since \(f(a) \leqslant f(b) \pin{Q}\), 
  and hence \(c \otimes a \otimes d \leqslant c \otimes b \otimes d 
  \pin{f^*Q}\). The same argument works for composition.
  
  Then there is the matter of strictness. If $f^*Q$ is not 
  strict, i.e., if \(a < b \pin{f^*Q}\) but \(c \otimes a \otimes d 
  \not< c \otimes b \otimes d \pin{f^*Q}\) for some \(a,b,c,d \in 
  \mc{P}\) then \(c \otimes a \otimes d \sim c \otimes b \otimes d 
  \pin{f^*Q}\) and hence \(f(c) \otimes f(a) \otimes f(d) = 
  f(c \otimes\nobreak a \otimes\nobreak d) \sim f(c \otimes\nobreak b 
  \otimes\nobreak d) = f(c) \otimes f(b) \otimes f(d) \pin{Q}\). 
  Since \(f(a) < f(b) \pin{Q}\), this implies $Q$ isn't strict 
  either. Again the exact same argument applies for composition.
  
  Finally for well-foundedness, let \(T \subseteq \mc{P}\) be an 
  arbitrary nonempty set. Then $f(T)$ is nonempty, and hence it 
  contains some $Q$-minimal element $f(x)$ for \(x \in T\). This $x$ 
  is then $f^*Q$-minimal in $T$, since had some \(y \in T\) satisfied 
  \(y < x \pin{f^*Q}\) then it would follow that \(f(y) < f(x) 
  \pin{Q}\), contradicting the minimality of $x$.
\end{proof}


\begin{lemma}[First isomorphism theorem]
  Let $\mc{P}$ be a \PROP\ and $C$ a \PROP\ congruence relation on 
  $\mc{P}$. Then $\mc{P}/C$ is a \PROP\ with
  \begin{align*}
    \alpha_{\mc{P}\!/C}\bigl( [a]_C \bigr) ={}&
      \alpha_\mc{P}(a) \text{,}&
    \omega_{\mc{P}\!/C}\bigl( [a]_C \bigr) ={}&
      \omega_\mc{P}(a) \text{,}\\
    [a]_C \otimes_{\mc{P}\!/C} [b]_C ={}&
      [a \otimes_{\mc{P}} b]_C \text{,} &
    [a]_C \circ_{\mc{P}\!/C} [b]_C ={}&
      [a \circ_{\mc{P}} b]_C \text{,}\\
    \phi_{\mc{P}\!/C}(\sigma) ={}& \bigl[ \phi_\mc{P}(\sigma) \bigr]_C
  \end{align*}
  for all \(a,b \in \mc{P}\) and permutations $\sigma$. Moreover 
  \(f\colon \mc{P} \Fpil \mc{P}/C\) defined by \(f(a) = [a]_C\) is a 
  \PROP\ homomorphism.
  
  Conversely let \(g\colon \mc{P} \Fpil \mc{Q}\) be a \PROP\ 
  homomorphism. Then $\Ker g$ is a \PROP\ congruence relation. Let 
  \(f\colon \mc{P} \Fpil \mc{P}/\Ker g\) be defined by \(f(a) = 
  [a]_{\Ker g}\) for all \(a \in \mc{P}\). Then there exists a unique 
  map \(h\colon \mc{P}/\Ker g \Fpil \mc{Q}\) such that \(g = h \circ 
  f\), and that $h$ is a \PROP\ monomorphism. If $g$ is surjective, 
  then $h$ is an isomorphism.
\end{lemma}
\begin{proof}
  That the $\alpha$, $\omega$, $\circ$, and $\otimes$ of $\mc{P}/C$ 
  are well-defined follows from the definition of \PROP\ congruence 
  relation. That they satisfy the axioms of a \PROP\ is shown through 
  the usual exchange of operations and equivalence class brackets; 
  for example the permutation composition law follows from
  \begin{multline*}
    \phi_{\mc{P}\!/C}(\sigma) \circ_{\mc{P}\!/C} 
      \phi_{\mc{P}\!/C}(\tau) =
    \bigl[ \phi_\mc{P}(\sigma) \bigr]_C \circ_{\mc{P}\!/C} 
      \bigl[ \phi_\mc{P}(\tau) \bigr]_C = \\ =
    \bigl[ \phi_\mc{P}(\sigma) \circ_\mc{P} \phi_\mc{P}(\tau) \bigr]_C =
    \bigl[ \phi_\mc{P}(\sigma\tau) \bigr]_C =
    \phi_{\mc{P}\!/C}(\sigma\tau)
    \text{.}
  \end{multline*}
  That \(f\colon \mc{P} \Fpil \mc{P}/C\) is a \PROP\ homomorphism is 
  literally the definition of~$\mc{P}/C$.
  
  That $\Ker g$ is a \PROP\ congruence relation follows from 
  Lemma~\ref{L:Pullback}. The defining property of \(h\colon 
  \mc{P}/\Ker g \Fpil \mc{Q}\) is that \(h\bigl( f(a) \bigr) = g(a)\) 
  for all \(a \in \mc{P}\), and $h$ is well-defined and injective since 
  \(f(a) = f(b)\) iff \(a \sim b \pin{\Ker g}\), which by definition 
  is equivalent to \(g(a) = g(b)\). That $f$ is surjective means $h$ 
  is completely determined by this. That $h$ is a \PROP\ homomorphism 
  is another exchange exercise; for $\otimes$ is amounts to
  \begin{equation*}
    h\bigl( f(a) \otimes f(b) \bigr) =
    h\bigl( f(a \otimes b) \bigr) =
    g( a \otimes b ) =
    g(a) \otimes g(b) =
    h\bigl( f(a) \bigr) \otimes h\bigl( f(b) \bigr)
  \end{equation*}
  for all \(a,b \in \mc{P}\).
\end{proof}

This lemma thus establishes for \PROPs, through elementary means,  
the basic principle of universal algebra: that an arbitrary \PROP\ 
$\mc{Q}$ may be viewed as (is isomorphic to) a quotient of a preferred 
\PROP\ $\mc{P}$, if only one can construct a surjective homomorphism 
from $\mc{P}$ to $\mc{Q}$. The elementary construction of a free 
\PROP, which as usual comes with homomorphisms to arbitrary \PROPs, is 
the subject of Section~\ref{Sec:FriPROP}. 


\begin{construction} \label{Con:Lex.sammans.}
  Let $\mc{S}$ be a set. Define a binary operation 
  $\diamond$\index{$\diamond$} 
  on binary relations \(P,Q \subseteq \mc{S} \times \mc{S}\) by
  \begin{equation}
    P \diamond Q = \setOf[\big]{ (x,y) \in \mc{S} \times \mc{S} }{
      \text{\( x < y \pin{P}\), or \(x \sim y \pin{P}\) and \(x 
        \leqslant y \pin{Q}\)} }
    \text{.}
  \end{equation}
  This operation is associative and it is called the 
  \DefOrd{lexicographic composition}.
  The lexicographic composition of two quasi-orders is itself a 
  quasi-order. The lexicographic composition of a quasi-order 
  and a partial order is a partial order.
  
  If $\mc{S}$ is a \PROP, and $P$ and $Q$ are strict \PROP\ 
  quasi-orders, then $P \diamond Q$ is a strict \PROP\ quasi-order 
  too. If $P$ and $Q$ are well-founded then $P \diamond Q$ is 
  well-founded as well.
\end{construction}
\begin{remark}
  Some strictness condition for composing \PROP\ quasi-orders is 
  necessary, since if \(x < y \pin{P}\) but for example \(x \circ z 
  \sim y \circ z \pin{P}\) then \(x \leqslant y \pin{P \diamond Q}\) 
  regardless of $Q$, but whether \(x \circ z \leqslant y \circ z 
  \pin{P \diamond Q}\) depends entirely on $Q$.
  
  In~\cite{Saito}, the lexicographic composition is instead called the 
  $*$-product.
  
  The lexicographic composition is similar to the `ordinal product' of 
  two posets~\cite[p.~101]{Stanley}, which is a generalisation of 
  multiplication of ordinal numbers; the main difference is that the 
  ordinal product also forms a cartesian product of the underlying 
  sets, whereas the this lexicographic product produces a new 
  relation on the same underlying set.
\end{remark}
\begin{proof}
  First assume $P$ and $Q$ are quasi-orders on $\mc{S}$. To see that 
  $P \diamond Q$ will then be reflexive, it suffices to notice that 
  since $P$ and $Q$ are reflexive, \(x \sim x \pin{P}\) and \(x 
  \leqslant x \pin{Q}\) for all \(x \in \mc{S}\).
  To see that $P \diamond Q$ will be transitive, note that
  $$
    x \leqslant y \pin{P \diamond Q} \mkern 9mu \Epil \mkern 9mu
    (x \leqslant y \pin{P}) \wedge \bigl( (x \geqslant y \pin{P}) 
      \Longrightarrow (x \leqslant y \pin{Q}) \bigr)
  $$
  for all \(x,y \in \mc{S}\). Hence if \(x \leqslant y \pin{P \diamond 
  Q}\) and \(y \leqslant z \pin{P \diamond Q}\) for some \(x,y,z \in 
  \mc{S}\), then \(x \leqslant y \pin{P}\) and \(y \leqslant z 
  \pin{P}\), and thus \(x \leqslant z \pin{P}\). Furthermore the 
  assumption that \(x \geqslant z \pin{P}\) leads to \(x \geqslant y 
  \pin{P}\) and \(y \geqslant z \pin{P}\), hence \(x \leqslant y 
  \pin{Q}\) and \(y \leqslant z \pin{Q}\), and thus \(x \leqslant z 
  \pin{Q}\). The above equivalence now implies that \(x \leqslant z 
  \pin{P \diamond Q}\), as required.
  
  If $P$ is a partial order and $Q$ is reflexive then it is easy to 
  see that \(P \diamond Q = P\) and thus $P \diamond Q$ is trivially 
  a partial order. To see that it is a partial order if $P$ is a 
  quasi-order and $Q$ is a partial order, note that
  $$
    x \sim y \pin{P \diamond Q} \quad\Epil\quad 
    \text{\(x \sim y \pin{P}\) and \(x \sim y \pin{Q}\)}
  $$
  for all \(x,y \in \mc{S}\). Since \(x \sim y \pin{Q}\) by 
  assumption implies \(x=y\), it follows that \(x \sim y 
  \pin{P \diamond Q}\) implies that as well.
  
  Let \(x,y \in \mc{S}\) and \(P,Q,R \subseteq \mc{S}\times\mc{S}\) 
  be arbitrary. Then it follows from the equivalence
  \begin{align*}
    x \leqslant y \pin{P \diamond (Q \diamond R)} 
       \Epil{} \mkern-70mu & \\
    \Epil{}&
    (x < y \pin{P}) \vee \bigl((x \sim y \pin{P}) \wedge 
      (x \leqslant y \pin{Q \diamond R}) \bigr) 
      \displaybreak[0]\\
    \Epil{}& (x < y \pin{P}) \vee \bigl((x \sim y \pin{P}) \wedge 
      (x < y \pin{Q}) \bigr) \vee \\ &\quad{}\vee
      \bigl((x \sim y \pin{P}) \wedge 
      (x \sim y \pin{Q}) \wedge (x \leqslant y \pin{R}) \bigr) 
      \displaybreak[0]\\
    \Epil{}& (x < y \pin{P \diamond Q}) \vee \bigl( (x \sim y 
      \pin{P \diamond Q}) \wedge (x \leqslant y \pin{R}) \bigr) \\
    \Epil{}& x \leqslant y \pin{(P \diamond Q) \diamond R}
  \end{align*}
  that $\diamond$ is associative.
  
  Finally there is the matter of whether $P \diamond Q$ is a strict 
  \PROP\ quasi-order when $P$ and $Q$ are. First, $P \diamond Q$ is 
  clearly $\N^2$-graded since it can only relate pairs of elements 
  already related by $P$.
  
  In order to establish compatibility with composition 
  (\eqref{Eq:PROP-order,comp} and \eqref{Eq:PROP-order,comp,str}), one 
  must consider three cases for \(a,b \in \mc{S}(l,m)\), \(c \in 
  \mc{S}(k,l)\), and \(d \in \mc{S}(m,n)\). The first is that \(a < b 
  \pin{P \diamond Q}\) because \(a < b \pin{P}\), in which \(c \circ a 
  \circ d < c \circ b \circ d \pin{P}\) and hence \(c \circ a \circ d < 
  c \circ b \circ d \pin{P \diamond Q}\) as required. The second is 
  that \(a < b \pin{P \diamond Q}\) despite \(a \sim b \pin{P}\), in 
  which follows that \(a < b \pin{Q}\) and hence \(c \circ a \circ d < 
  c \circ b \circ d \pin{Q}\). By using \eqref{Eq:PROP-order,comp} 
  twice, since `$\sim$' is `$\leqslant$ and $\geqslant$', it also 
  follows that \(c \circ a \circ d \sim c \circ b \circ d \pin{P}\), so 
  again \(c \circ a \circ d < c \circ b \circ d \pin{P \diamond Q}\) as 
  required. The third case is merely that \(a \sim b \pin{P \diamond 
  Q}\), in which it similarly follows from \(a \sim b \pin{P}\) and 
  \(a \sim b \pin{Q}\) that \(c \circ a \circ d \sim c \circ b \circ d 
  \pin{P}\) and \(c \circ a \circ d \sim c \circ b \circ d \pin{Q}\), 
  so that \(c \circ a \circ d \sim c \circ b \circ d \pin{P \diamond 
  Q}\), which is the required conclusion in that case. Compatibility 
  with the tensor product is proved in exactly the same way.
  
  Finally, there is the matter of well-foundedness. Let \(T \subseteq 
  \mc{S}\) be an arbitrary nonempty set. Let $T'$ be the set of 
  $P$-minimal elements in $T$; this is nonempty by the 
  well-foundedness of $P$. Hence there exists a $Q$-minimal element 
  $x$ in $T'$. There cannot be a \(y \in T\) such that \(y < x 
  \pin{P \diamond Q}\), because \(y < x \pin{P}\) would contradict 
  the $P$-minimality of $x$ in $T$, and \(y \sim x \pin{P}\) would 
  imply that $y$ is $P$-minimal in $T$, meaning \(y \in T'\), and 
  then \(y < x \pin{Q}\) would contradict the $Q$-minimality of $x$ 
  in $T'$. Hence $x$ is $(P \diamond\nobreak Q)$-minimal in $T$, and 
  thus $P \diamond Q$ is well-founded.
\end{proof}

\section{Abstract index notation}
\label{Sec:AIN}


Recall the Einstein notation for tensors, under which
\begin{equation*}
  C^i_k = A^i_j B^j_k \qquad\text{is a shorthand for}\qquad
  C^i_k = \sum_j A^i_j B^j_k
\end{equation*}
and superscripts are not exponents, but like subscripts used as 
indices into the base symbol.
The abstract index notation is a generalisation of this, which follows 
the same syntactic principles but is given a more abstract 
interpretation as a shorthand for some combination of compositions, 
tensor products, and permutations. Nonetheless the reader may find 
it instructive to make direct Einstein notation interpretations of 
expressions, and in the case of $\HomPROP_V$ \PROPs\ both 
interpretations are equally valid, especially when $V$ is 
finite-dimensional (e.g.~a finite-dimensional Hopf algebra).

A point of view that is convenient when understanding the abstract 
index notation (and even more the network notation, below) is that a 
\PROP\ element $\gamma$ of arity $n$ and coarity $m$ may be thought 
of as a black box which takes $n$ inputs and produces $m$ outputs. 
Composite \PROP\ elements are formed by connecting outputs from one 
factor to the inputs of another. In order to keep track of these 
connections, every input and output is given a separate label (the 
`abstract index' symbol), and when the same label occurs for an output 
as for an input, then these are connected. Those inputs of a factor 
which are not connected to any output of a factor become inputs of the 
expression as a whole, whereas those outputs that are not connected to 
any input of a factor become outputs of the expression as a whole.

By convention in the case of the Einstein notation, output 
labels~(row indices, signalling a vector aspect of the tensor) are 
written as superscripts of factors and input labels~(column indices, 
signalling a covector aspect of the tensor) are written as subscripts 
of factors,\footnote{
  This convention has been used by Penrose~\cite{Penrose} and in much 
  physics literature, whereas for example 
  Cvitanovi\'{c}~\cite{Cvitanovic} uses the opposite convention. It 
  doesn't matter terribly much which convention is used as long as one 
  stays within one notation system, but it becomes very noticable if 
  one e.g.\@ mixes abstract index notation with categorical notation. 
} so if $\gamma$ has arity $3$ and coarity $2$ then it might appear in 
abstract index notation as the labelled factor $\gamma^{ab}_{cde}$. 
A \emDefOrd{naked expression} in abstract index notation is simply a 
product of labelled factors (satisfying some syntactic conditions), 
but from a formal point of view it is more convenient to 
start with \emph{closed} expressions, and introduce naked expressions 
as a relaxation of these. A \emDefOrd{closed expression} in abstract 
index notation has the form
\begin{equation}
  \fuse{\text{outputs}}{
    \beta^{\text{labels}}_{\text{labels}} \dotsb
    \psi^{\text{labels}}_{\text{labels}}
  }{\text{inputs}}
\end{equation}
where $\beta$ through $\psi$ are \PROP\ elements and the various 
`labels', `inputs', and `outputs' are finite lists of (variable) 
symbols satisfying the arity, matching, and acyclicity conditions 
detailed below. In this notation, the antipode axiom can be written 
as
\begin{equation}
  \fuse{b}{ \mOp^b_{cd} S^c_e \Delta^{ed}_a }{a} =
  \fuse{b}{ \eta^b \ve_a }{a} =
  \fuse{b}{ \mOp^b_{cd} S^d_e \Delta^{ce}_a }{a}
  \text{.}
\end{equation}
The complete bracket construction denotes a \PROP\ element, 
whose arity is the number of inputs and whose coarity is the number 
of outputs. The corresponding naked expression is
\begin{equation}
  \mOp^b_{cd} S^c_e \Delta^{ed}_a =
  \eta^b \ve_a =
  \mOp^b_{cd} S^d_e \Delta^{ce}_a
  \text{,}
\end{equation}
i.e., just the product parts (between a bar and the following right 
bracket) of the closed expressions. An equality of two closed 
expressions can be abbreviated as an equality of the corresponding 
naked expressions \emph{if} the input and output lists are the same; 
conversely any naked expression equality can be turned into an 
equivalent closed expression equality by wrapping the terms up in 
brackets with equal input and output lists. (It follows from 
Lemmas~\ref{L:fperm-right} and~\ref{L:fperm-left} that when closing a 
naked expression equality, every syntactically correct choice of input 
and output lists produces an equivalent equality of closed 
expressions.) A complication when going from closed expressions to 
naked expressions is however that one side of an equality might have 
the same label both as input and output, whereas the other side does 
not:
\[
  \fuse{a}{ \mOp^a_{bc} \eta^b }{c} =
  \fuse{a}{1}{a} \text{.}
\]
In these cases, it may be necessary to insert extra identity factors 
in the naked expression form, just to switch label. In traditional 
Einstein notation, such identities would be 
\emDefOrd[*{Kronecker delta}]{Kronecker deltas}, so with \(\delta := 
\phi(\same{1})\) as this kind of identity, the naked counterpart of 
the above is
\[
  \mOp^a_{bc} \eta^b = \delta^a_c \text{.}
\]
From Section~\ref{Sec:Deluttryck} and on, the symbol $\natural$ 
starts being used in the context of the network notation very much as 
the Kronecker delta is used in the abstract index notation, so 
arguably it might have been better to pick either one and stick with 
it, but offering familiar notation also has its merits.

When transcribing an expression between network and abstract index 
notation, one may at times feel confusion as to how up and down 
correspond to in and out; this is mostly due to the abstract index 
notation being basically one-dimensional (a line of text), whereas 
the network notation is two-dimensional (a diagram). A good principle 
to keep in mind is that \emph{connections go down}, from a 
superscript to a subscript. Conversely ``data'' go ``back up'' to the 
top of the line when a factor ``processes'' it. An image can be that 
an abstract index expression is somewhat like a network notation 
expression that has been rolled up on a transparent rod the size of 
a line of text, in such a way that all the factor are on the far side, 
whereas connections between them are on the near side\Dash that's why 
the in\slash out\slash up\slash down relationships are different.

For \PROPs\ that support operations beyond tensor product and 
composition, particularly addition and multiplication by scalar, 
the naked notation is (in analogy with the Einstein notation) often 
extended to support also these. Hence one might write
\[
  \Delta^{ab}_c \xi^c = \xi^a \eta^b + 2 \zeta^a \zeta^b + \eta^a \xi^b
\]
as a shorthand for
\[
  \fuse{ab}{\Delta^{ab}_c \xi^c}{} = 
  \fuse{ab}{\xi^a \eta^b}{} + 2\fuse{ab}{\zeta^a \zeta^b}{} +
  \fuse{ab}{\eta^a \xi^b}{}
  \text{,}
\]
i.e., in each term the labelled product part becomes a separate closed 
expression. Tensor product and composition both appear as 
``multiplication'' in the naked notation, but one can tell them apart 
by looking at the labels; Theorems~\ref{S:AIN,snitt} 
and~\ref{S:AIN,klyva} formally establishes this correspondence.

One feature commonly introduced as part of the abstract index notation, 
but which \textsc{shall not} be assumed here, is that of ``raising and 
lowering indices'': that one from any $u^a$ can manufacture a 
corresponding $u_a$ and vice versa. 
(A consequence of this is that there is no sense in 
which a superscript comes before or after a subscript; only the 
relative order of superscripts among themselves, and of subscripts 
among themselves, have significance.) Even in contexts where such 
raising and lowering of indices is in common use, it tends to merely 
be a shorthand for composition with two specific elements\Ldash 
usually the metric (i.e., the inner product) and its inverse\Rdash 
from the $(0,2)$ component (for lowering) and $(2,0)$ component 
(for raising) respectively. Hence there is no need for that feature 
in the generic abstract index notation.

A feature which is left for future extensions is that of allowing 
different indices to ``range'' over different sets\Dash e.g.~a 
typical formulation of the Lagrangian for the Standard Model of 
elementary particle physics would typically sport not only Lorentzian 
(space-time) indices, but also spinor indices, flavour indices, 
colour indices, and quite possibly some more types still (depending 
on how much the author seeks to compress the expressions); these must 
all be kept apart. Distinguishing several different types of index 
would correspond to a generalisation from \PROPs\ to (strict) 
symmetric monoidal categories (with free monoid of objects). While 
this would certainly be worth while if one's primary aim were to 
establish the notation system as such, it would also be a technical 
complication, and the issue is largely orthogonal to what is relevant 
for rewriting. Hence this text is better off without it.

When expressing rewrite rules in abstract index notation, it 
sometimes also becomes necessary to go beyond what can be formalised 
for abstract \PROPs, to make a statement to the effect that ``this holds 
even if this abstract index expression is extended with extra 
labelled factors that make input $e$ depend on output $a$'', which is 
the meaning of the \(a \canmake e\) in
\[
  \mOp^b_{cd} \, \mOp^c_{ef} \, S^f_h \, \Delta^{ah}_i \, \Delta^{id}_j 
    \longmapsto
  \delta^a_j \, \delta^b_e
  \qquad \text{where \(a \canmake e\).}
\]
The formal background for this practice here can be found in 
Section~\ref{Sec:Feedbacks} and page~\pageref{Sid:AIN-regel}, but 
informally it may be viewed as a statement that certain ``trace maps'' 
exist or are admissible in the situation at hand.

Finally, it should be observed that one sometimes has to put abstract 
index labels on a \PROP\ element for which the notation already 
contains an ordinary index; one may for example have elements 
$\zeta_1$ and $\zeta_2$. In such cases, it is advisable to put the 
abstract indices on a parenthesis around the symbol for the \PROP\ 
element, like so: $(\zeta_1)^a (\zeta_2)^b$.

%

\subsection{Formal definition and basic properties}

In the formal claims about the abstract index notation below, lists of 
abstract index labels are denoted by boldface letters, and notation for 
such lists is mostly multiplicative; a list may be viewed as a word on 
an alphabet of abstract index labels. Hence $\vek{a}\vek{b}$ is the list 
consisting of first the elements of $\vek{a}$ and then the elements of 
$\vek{b}$. The notation $\norm{\vek{a}}$\index{\norm@$\norm{\vek{a}}$}
denotes the multiset of elements 
in $\vek{a}$, and $\Norm{\vek{a}}$\index{\Norm@$\Norm{\vek{a}}$} 
is the length of $\vek{a}$ (thus 
cardinality of $\norm{\vek{a}}$). Removal of elements is denoted as 
division: $\vek{a}/\vek{b}$\index{/@$/$!for lists} 
is the list of those elements of $\vek{a}$ 
that are not in $\vek{b}$, always in the same order as in $\vek{a}$, 
whereas their order in $\vek{b}$ is ignored. Because of the matching 
condition on labels, the lists occurring below will not have any 
repeated elements, so $/$ behaves like division in that for example 
\(\vek{abc}/\vek{b} = \vek{ac}\).
The formula \(\vek{b} \mid \vek{a}\)\index{\mid@$\mid$ (list divides)}
(`$\vek{b}$ divides\index{divides} $\vek{a}$') is 
a shorthand for \(\norm{\vek{b}} \subseteq \norm{\vek{a}}\). Elements 
in a list of labels, i.e., actual abstract index labels, are denoted 
by lower case latin letters. Upper case letters will sometimes be used 
for products of zero or more labelled factors.
Write \index{N k l m@$\vek{N}_k(l,m)$}\(\vek{N}_k(l,m)\) for the list 
\(n_1 \dotsb n_k\) where \(n_i = (i-1)m + l\), e.g. \(\vek{N}_3(1,2) = 
135\); this is sometimes useful for ``automatically manufacturing'' 
disjoint lists of known lengths.



\begin{definition} \label{D:fuse}
  Let $\vek{a}_0, \dotsc, \vek{a}_n$ and $\vek{b}_0, \dotsc, 
  \vek{b}_n$ be lists. The notation
  \[
    \fuse[\Big]{\vek{a}_0}{ \textstyle
      \prod_{i=1}^n (\gamma_i)^{\vek{b}_i}_{\vek{a}_i} 
    }{\vek{b}_0} =
    \fuse[\Big]{\vek{a}_0}{ \textstyle
      (\gamma_1)^{\vek{b}_1}_{\vek{a}_1} 
      (\gamma_2)^{\vek{b}_2}_{\vek{a}_2} \dotsb
      (\gamma_n)^{\vek{b}_n}_{\vek{a}_n} 
    }{\vek{b}_0}
    \index{[]@$\fuse{\ldots}{\dotsb}{\ldots}$}
  \]
  fulfilling the conditions of
  \begin{description}
    \item[arity:]
      \(\gamma_i \in \mc{P}\bigl( \Norm{\vek{b}_i}, \Norm{\vek{a}_i} 
      \bigr)\) for \(i=1,\dotsc,n\),
    \item[matching:]
      \(V := \norm[\big]{ \prod_{i=0}^n \vek{a}_i} = 
      \norm[\big]{ \prod_{i=0}^n \vek{b}_i}\) is a set, and
    \item[acyclicity:]
      there exists a partial order $Q$ on $V$ such that \(x < y 
      \pin{Q}\) whenever \(x \in \norm{\vek{b}_i}\) and \(y \in 
      \norm{\vek{a}_i}\) for some \(1 \leqslant i \leqslant n\)
  \end{description}
  denotes an element of $\mc{P}\bigl( \Norm{\vek{a}_0}, 
  \Norm{\vek{b}_0} \bigr)$, namely:
  \begin{enumerate}
    \item \label{Item1:D:fuse}
      \(\fuse{\vek{a}}{1}{\vek{b}} = \phi_m(\sigma)\), where \(m = 
      \Norm{\vek{a}} = \Norm{\vek{b}}\) and \(\sigma \in \Sigma_m\) 
      is the permutation for which \(\vek{b} = a_{\sigma(1)} \dotsb 
      a_{\sigma(m)}\) where \(a_1\dotsb a_m := \vek{a}\).
    \item \label{Item2:D:fuse}
      \(\fuse{\vek{a}}{ \gamma^{\vek{b}}_{\vek{c}} }{\vek{d}} =
      \fuse{\vek{a}}{1}{\vek{b}(\vek{a}/\vek{b})} \circ
      \gamma \otimes \fuse{\vek{a}/\vek{b}}{1}{\vek{a}/\vek{b}} \circ
      \fuse{\vek{c}(\vek{a}/\vek{b})}{1}{\vek{d}}\).
    \item \label{Item3:D:fuse}
      If \(n \geqslant 2\) then for the smallest \(k\geqslant 1\) such 
      that \(\vek{b}_k \mid \vek{a}_0\),
      \begin{multline} \label{Eq:D:fuse}
        \fuse[\Big]{\vek{a}_0}{ \textstyle
          \prod_{i=1}^n (\gamma_i)^{\vek{b}_i}_{\vek{a}_i} 
        }{\vek{b}_0} = \\ =
        \fuse[\Big]{\vek{a}_0}{ \textstyle
          (\gamma_k)^{\vek{b}_k}_{\vek{a}_k} 
        }{ \vek{a}_k (\vek{a}_0/\vek{b}_k) } \circ
        \fuse[\Big]{ \vek{a}_k (\vek{a}_0/\vek{b}_k) }{ \textstyle
          \prod_{i=1}^{k-1} (\gamma_i)^{\vek{b}_i}_{\vek{a}_i} 
          \prod_{i=k+1}^n (\gamma_i)^{\vek{b}_i}_{\vek{a}_i} 
        }{\vek{b}_0} 
        \text{.}
      \end{multline}
  \end{enumerate}
  The number $n$ of factors in $\fuse[\Big]{\vek{a}_0}{ 
  \prod_{i=1}^n (\gamma_i)^{\vek{b}_i}_{\vek{a}_i} }{\vek{b}_0}$ is 
  called the \DefOrd{order} of this \emDefOrd{abstract index 
  expression}.
  
  In claims, an equality with one or several abstract index 
  expressions in the left hand side is to be read as being implicitly 
  conditioned with the fulfillment of the corresponding arity, 
  matching, and acyclicity conditions for these. When these and all 
  explicit conditions are fulfilled, all abstract index expressions on 
  the right hand side must fulfill their arity, matching, and 
  acyclicity conditions.
\end{definition}

If comparing this definition of abstract index notation, particularly 
expressions of order $0$ such as $\fuse{2341}{1}{1234}$, to the 
relation notation $\binom{1234}{2341}$ for permutations, it may seem 
as though rule~\ref{Item1:D:fuse} runs the permutation the wrong way 
(bottom to top rather than top to bottom), but this is because of a 
more fundamental difference between the two notations. In relation 
notation for permutations, elements in the top and bottom lists are 
elements in the permutation domain and codomain respectively, and two 
elements are linked if they are in the same position. In the 
abstract index notation, it is the \emph{positions} in the lists that 
are elements of the permutation domain and codomain, and two elements 
are linked if they contain the same label. This duality of list 
elements and positions explains why the permutation definition is 
reversed; Lemma~\ref{L:sammans.perm.} shows that it still behaves as 
expected (top row matched to things on the right, bottom row to 
things on the left).

\begin{theorem}
  The notation introduced in Definition~\ref{D:fuse} is well-defined, 
  i.e., whenever the three conditions are fulfilled, there is a 
  unique rule which assigns a value to the expression.
\end{theorem}
\begin{proof}
  Rule~\ref{Item1:D:fuse} determines the value of order~$0$ 
  expressions, rule~\ref{Item2:D:fuse} determines the value of 
  order~$1$ expressions, and recursive application of 
  rule~\ref{Item3:D:fuse} determines the value of all expressions 
  with order greater than~$1$.
  
  More precisely, in the order~$0$ case there is a unique 
  permutation $\sigma$ such that \(\vek{a} = a_1\dotsb a_m\) and 
  \(\vek{b} = a_{\sigma(1)} \dotsb a_{\sigma(m)}\) because 
  \(\norm{\vek{a}} = \norm{\vek{b}}\) is a set. 
  
  In the order~$1$ case, one may first 
  observe that \(\norm{\vek{a}} \cup \norm{\vek{c}} = 
  \norm{\vek{a}\vek{c}} = \norm{\vek{d}\vek{b}} = \norm{\vek{d}} 
  \cup \norm{\vek{b}}\), and since \(\norm{\vek{b}} \cap 
  \norm{\vek{c}} = \varnothing\) by the partial order condition, it 
  follows that \(\vek{b} \mid \vek{a}\) and \(\norm{\vek{d}} = 
  \norm{\vek{c}} \cup \norm{\vek{a}} \setminus \norm{\vek{b}}\). 
  Hence
  \begin{align*}
    \norm[\big]{ \vek{b}(\vek{a}/\vek{b}) } ={}&
      \norm{\vek{b}} \cup \norm[\big]{\vek{a}/\vek{b}} =
      \norm{\vek{b}} \cup \bigl( \norm{\vek{a}} \setminus 
        \norm{\vek{b}} \bigr) = 
      \norm{\vek{a}} \text{,} \\
    \norm[\big]{ \vek{c}(\vek{a}/\vek{b}) } ={}& 
      \norm{\vek{c}} \cup \bigl( \norm{\vek{a}} \setminus 
        \norm{\vek{b}} \bigr) = 
      \norm{\vek{d}}
  \end{align*}
  and thus all the abstract index expressions in the right hand side 
  of this rule are well-defined. Furthermore
  \[
    \gamma \otimes \fuse{\vek{a}/\vek{b}}{1}{\vek{a}/\vek{b}} \in 
    \mc{P}\bigl( \Norm{\vek{b}},\Norm{\vek{c}} \bigr) \otimes 
      \mc{P}\bigl( \Norm{\vek{a}/\vek{b}}, \Norm{\vek{a}/\vek{b}} 
      \bigr) \subseteq
    \mc{P}\bigl( \Norm{\vek{b}(\vek{a}/\vek{b})}, 
      \Norm{\vek{c}(\vek{a}/\vek{b})} \bigr)
  \]
  and so also the compositions are all correct.
  
  Finally, in the case of order $2$ or more, one must first establish 
  the existence of some \(k \geqslant 1\) such that \(\vek{b}_k \mid 
  \vek{a}_0\). Let $Q$ be the required partial order on the set of all 
  labels $V$, and define a second partial order $P$ on 
  $\{1,\dotsc,n\}$ by
  \begin{equation*}
    i < j \pin{P} \Epil
    \text{\(x \leqslant y \pin{Q}\) for some \(x \in 
      \norm{\vek{a}_i}\) and \(y \in \norm{\vek{b}_j}\);}
  \end{equation*}
  $P$ becomes transitive because $Q$ is transitive and $P$ is 
  antisymmetric by the antisymmetry and transitivity of $Q$. Now 
  consider some $j$ which is $P$-minimal. No \(x \in 
  \norm{\vek{b}_j}\) can belong to $\norm{\vek{a}_i}$ for any \(i 
  \geqslant 1\), since that would make \(i < j \pin{P}\). Hence 
  \(\vek{b}_j \mid \vek{a}_0\).
  
  Next observe that \(\norm{ \vek{a}_0\vek{a}_k } = \norm{\vek{a}_0} 
  \cup \norm{\vek{a}_k} =  \norm{\vek{a}_k} \cup \bigl( \norm{\vek{a}_0} 
  \setminus \norm{\vek{b}_k} \bigr) \cup \norm{\vek{b}_k} =
  \norm{ \vek{a}_k (\vek{a}_0/\vek{b}_k) \vek{b}_k }\), and so the 
  first abstract index expression fulfills the conditions. For the 
  second it follows from \(\bigcup_{i=0}^n \norm{\vek{b}_i} = 
  \norm[\Big]{ \prod_{i=0}^n \vek{b}_i } = 
  \norm[\Big]{ \prod_{i=0}^n \vek{a}_i } = \bigcup_{i=0}^n 
  \norm{\vek{a}_i}\) and \(\vek{b}_k \mid \vek{a}_0\) that
  \begin{multline*}
    \norm[\Big]{ \textstyle \prod_{i=0}^{k-1} \vek{b}_i 
      \prod_{i=k+1}^n \vek{b}_i } =
    \bigcup_{\substack{0 \leqslant i \leqslant n\\ i \neq k}}
      \norm{\vek{b}_i} =
    \norm[\big]{\vek{a}_0/\vek{b}_k} \cup 
      \bigcup_{1 \leqslant i \leqslant n} \norm{\vek{a}_i} 
      = \\ =
    \norm[\big]{\vek{a}_k(\vek{a}_0/\vek{b}_k)} \cup 
      \bigcup_{\substack{1 \leqslant i \leqslant n\\ i \neq k}}
        \norm{\vek{a}_i} =
    \norm[\Big]{ \textstyle
      \vek{a}_k(\vek{a}_0/\vek{b}_k)
      \prod_{i=1}^{k-1} \vek{a}_i 
      \prod_{i=k+1}^n \vek{a}_i
    } \text{.}
  \end{multline*}
  Restricting the partial order $Q$ from the order $n$ abstract 
  index expression to the labels available in the order $1$ and 
  $n-1$ subexpressions produces partial orders which will serve 
  for these, so they satisfy all the conditions.
\end{proof}

Once existence has been established, the next step is to establish 
some elementary results on how the abstract index notation compares 
to ordinary notation. This is the subject of the next three theorems.

\begin{theorem} \label{S:fuse-just}
  For all lists $\vek{a}$ and $\vek{b}$ such that $\norm{\vek{a}}$ 
  and $\norm{\vek{b}}$ are disjoint sets, and all \(\gamma \in 
  \mc{P}\bigl( \Norm{\vek{a}}, \Norm{\vek{b}} \bigr)\),
  \begin{align}
    \phi\left( \same{\Norm{\vek{a}}} \right) ={}&
      \fuse{\vek{a}}{1}{\vek{a}} 
      \label{Eq1:fuse-just}\text{,}\\
    \gamma ={}& \fuse{\vek{a}}{ \gamma^{\vek{a}}_{\vek{b}} }{\vek{b}}
    \text{.}
  \end{align}
\end{theorem}
\begin{proof}
  The first identity is straight off from the definition. In the 
  second identity, any partial order which has every element of 
  $\norm{\vek{a}}$ less than every element of $\norm{\vek{b}}$ 
  fulfills the acyclicity condition, so the right hand side exists. 
  Computing its value using rule~\ref{Item2:D:fuse} yields
  \begin{multline*}
    \fuse{\vek{a}}{ \gamma^{\vek{a}}_{\vek{b}} }{\vek{b}} =
    \fuse{\vek{a}}{1}{\vek{a}(\vek{a}/\vek{a})} \circ
      \gamma \otimes \fuse{\vek{a}/\vek{a}}{1}{\vek{a}/\vek{a}} \circ
      \fuse{\vek{b}(\vek{a}/\vek{a})}{1}{\vek{b}} = \\ =
    \fuse{\vek{a}}{1}{\vek{a}} \circ
      \gamma \otimes \phi(\same{0}) \circ
      \fuse{\vek{b}}{1}{\vek{b}} =
    \phi\left( \same{\Norm{\vek{a}}} \right) \circ \gamma \circ
      \phi\left( \same{\Norm{\vek{b}}} \right) =
    \gamma
  \end{multline*}
  since $\vek{a}/\vek{a}$ is the empty list.
\end{proof}

\begin{lemma} \label{L:sammans.perm.}
  For all $\vek{a}$, $\vek{b}$, and $\vek{c}$ such that 
  \(\norm{\vek{a}} = \norm{\vek{b}} = \norm{\vek{c}}\),
  \begin{equation}
    \fuse{\vek{a}}{1}{\vek{b}} \circ \fuse{\vek{b}}{1}{\vek{c}} = 
    \fuse{\vek{a}}{1}{\vek{c}} \text{.}
  \end{equation}
\end{lemma}
\begin{proof}
  Let \(a_1\dotsb a_m = \vek{a}\), \(b_1\dotsb b_m = \vek{b}\), and 
  \(c_1\dotsb c_m = \vek{c}\). Let \(\sigma,\tau \in \Sigma_m\) be 
  such that \(b_i = a_{\sigma(i)}\) and \(c_i = b_{\tau(i)}\) for 
  \(i=1,\dotsc,m\). Then \(c_i = b_{\tau(i)} = a_{\sigma(\tau(i))}\) 
  and thus
  \[
    \fuse{\vek{a}}{1}{\vek{c}} = 
    \phi_m(\sigma\tau) =
    \phi_m(\sigma) \circ \phi_m(\tau) =
    \fuse{\vek{a}}{1}{\vek{b}} \circ \fuse{\vek{b}}{1}{\vek{c}}
    \text{.}
  \]
\end{proof}

\begin{lemma} \label{L:tensor.perm.}
  For all $\vek{a}$, $\vek{b}$, $\vek{c}$, and $\vek{d}$ such that 
  \(\norm{\vek{a}} = \norm{\vek{b}}\) and \(\norm{\vek{c}} = 
  \norm{\vek{d}}\),
  \begin{equation}
    \fuse{\vek{a}\vek{c}}{1}{\vek{b}\vek{d}} =
    \fuse{\vek{a}}{1}{\vek{b}} \otimes \fuse{\vek{c}}{1}{\vek{d}} 
    \text{.}
  \end{equation}
\end{lemma}
\begin{proof}
  Let \(a_1 \dotsb a_k = \vek{a}\) and \(c_1 \dotsb c_l = \vek{c}\). 
  Let \(\sigma \in \Sigma_k\) and \(\tau \in \Sigma_l\) be such that 
  \(\vek{b} = a_{\sigma(1)} \dotsb a_{\sigma(k)}\) and \(\vek{d} = 
  c_{\tau(1)} \dotsb c_{\tau(l)}\). Define
  \[
    e_i = \begin{cases} 
      a_i & \text{if \(i \leqslant k\),} \\
      c_{i-k} & \text{otherwise,}
    \end{cases}
    \quad\text{and}\quad
    f_i = \begin{cases}
      a_{\sigma(i)} & \text{if \(i \leqslant k\),} \\
      c_{\tau(i-k)} & \text{otherwise.}
    \end{cases}
  \]
  Then \(\vek{ac} = e_1 \dotsb e_{k+l}\), \(\vek{bd} = f_1 \dotsb 
  f_{k+l}\), and \(f_i = e_{(\sigma\star\tau)(i)}\) for all $i$. Thus
  \[
    \fuse{\vek{a}\vek{c}}{1}{\vek{b}\vek{d}} =
    \phi_{k+l}( \sigma\star\tau ) =
    \phi_k(\sigma) \otimes \phi_l(\tau) =
    \fuse{\vek{a}}{1}{\vek{b}} \otimes \fuse{\vek{c}}{1}{\vek{d}}
    \text{.}
  \]
\end{proof}

\begin{lemma} \label{L:fperm-right}
  Whenever \(\norm{\vek{c}} = \norm{\vek{b}_0}\),
  \begin{equation}
    \fuse[\Big]{\vek{a}_0}{ \textstyle
      \prod_{i=1}^n (\gamma_i)^{\vek{b}_i}_{\vek{a}_i} 
    }{\vek{b}_0} \circ 
    \fuse{\vek{b}_0}{1}{\vek{c}} = 
    \fuse[\Big]{\vek{a}_0}{ \textstyle
      \prod_{i=1}^n (\gamma_i)^{\vek{b}_i}_{\vek{a}_i} 
    }{\vek{c}}
    \text{.}
  \end{equation}
\end{lemma}
\begin{proof}
  By induction on the order $n$. If \(n=0\) then this is the claim of 
  Lemma~\ref{L:sammans.perm.}. If \(n=1\) then
  \begin{multline*}
    \fuse{\vek{a}_0}{ (\gamma_1)^{\vek{b}_1}_{\vek{a}_1} }{\vek{b}_0} 
      \circ \fuse{\vek{b}_0}{1}{\vek{c}} 
      = \\ =
    \fuse{\vek{a}_0}{1}{ \vek{b}_1(\vek{a}_0/\vek{b}_1) } \circ
      \gamma_1 \otimes 
        \fuse{ \vek{a}_0/\vek{b}_1 }{1}{ \vek{a}_0/\vek{b}_1 }
      \circ \fuse{ \vek{a}_1(\vek{a}_0/\vek{b}_1) }{1}{\vek{b}_0} 
      \circ \fuse{\vek{b}_0}{1}{\vek{c}} 
      = \\ =
    \fuse{\vek{a}_0}{1}{ \vek{b}_1(\vek{a}_0/\vek{b}_1) } \circ
      \gamma_1 \otimes 
        \fuse{ \vek{a}_0/\vek{b}_1 }{1}{ \vek{a}_0/\vek{b}_1 }
      \circ \fuse{ \vek{a}_1(\vek{a}_0/\vek{b}_1) }{1}{\vek{c}} 
      =
    \fuse{\vek{a}_0}{ (\gamma_1)^{\vek{b}_1}_{\vek{a}_1} }{\vek{c}} 
    \text{.}
  \end{multline*}
  Finally if \(n>1\) and assuming it holds for expressions of order 
  $n-1$, then
  \begin{multline*}
    \fuse[\Big]{\vek{a}_0}{ \textstyle
      \prod_{i=1}^n (\gamma_i)^{\vek{b}_i}_{\vek{a}_i} 
    }{\vek{b}_0} 
    \circ \fuse{\vek{b}_0}{1}{\vek{c}} 
    = \\ =
    \fuse[\Big]{\vek{a}_0}{ 
      (\gamma_k)^{\vek{b}_k}_{\vek{a}_k} 
    }{ \vek{a}_k (\vek{a}_0/\vek{b}_k) } \circ
    \fuse[\Big]{ \vek{a}_k (\vek{a}_0/\vek{b}_k) }{ \textstyle
      \prod_{i=1}^{k-1} (\gamma_i)^{\vek{b}_i}_{\vek{a}_i} 
      \prod_{i=k+1}^n (\gamma_i)^{\vek{b}_i}_{\vek{a}_i} 
    }{\vek{b}_0} \circ
    \fuse{\vek{b}_0}{1}{\vek{c}} 
    = \\ =
    \fuse[\Big]{\vek{a}_0}{ 
      (\gamma_k)^{\vek{b}_k}_{\vek{a}_k} 
    }{ \vek{a}_k (\vek{a}_0/\vek{b}_k) } \circ
    \fuse[\Big]{ \vek{a}_k (\vek{a}_0/\vek{b}_k) }{ \textstyle
      \prod_{i=1}^{k-1} (\gamma_i)^{\vek{b}_i}_{\vek{a}_i} 
      \prod_{i=k+1}^n (\gamma_i)^{\vek{b}_i}_{\vek{a}_i} 
    }{\vek{c}}
    = \\ =
    \fuse[\Big]{\vek{a}_0}{ \textstyle
      \prod_{i=1}^n (\gamma_i)^{\vek{b}_i}_{\vek{a}_i} 
    }{\vek{c}}
  \end{multline*}
  where \(k \geqslant 1\) is minimal such that \(\vek{b}_k \mid 
  \vek{a}_0\).
\end{proof}

The analogous identity for left action of permutations also holds, 
but is a bit more technical as there formally is a greater dependency 
on the order of output labels.

\begin{lemma} \label{L:fperm-left}
  Whenever \(\norm{\vek{c}} = \norm{\vek{a}_0}\),
  \begin{equation}
    \fuse{\vek{c}}{1}{\vek{a}_0} \circ 
    \fuse[\Big]{\vek{a}_0}{ \textstyle
      \prod_{i=1}^n (\gamma_i)^{\vek{b}_i}_{\vek{a}_i} 
    }{\vek{b}_0} = 
    \fuse[\Big]{\vek{c}}{ \textstyle
      \prod_{i=1}^n (\gamma_i)^{\vek{b}_i}_{\vek{a}_i} 
    }{\vek{b}_0}
    \text{.}
  \end{equation}
\end{lemma}
\begin{proof}
  Again by induction on the order $n$. If \(n=0\) then this is the 
  claim of Lemma~\ref{L:sammans.perm.}. If \(n=1\) then by 
  Lemmas~\ref{L:sammans.perm.} and~\ref{L:tensor.perm.},
  \begin{align*}
    \fuse{\vek{c}}{1}{\vek{a}_0} \circ 
    \fuse{\vek{a}_0}{ (\gamma_1)^{\vek{b}_1}_{\vek{a}_1} 
    }{\vek{b}_0} 
    = {} \hspace{-8em}& \\ ={}&
    \fuse{\vek{c}}{1}{\vek{a}_0} \circ 
    \fuse{\vek{a}_0}{1}{ \vek{b}_1 (\vek{a}_0/\vek{b}_1) } \circ
    \gamma_1 \otimes 
      \fuse{ \vek{a}_0/\vek{b}_1 }{1}{ \vek{a}_0/\vek{b}_1 } \circ
    \fuse{ \vek{a}_1 (\vek{a}_0/\vek{b}_1) }{1}{ \vek{b}_0 }
    = \displaybreak[0]\\ ={}&
    \fuse{\vek{c}}{1}{ \vek{b}_1 (\vek{a}_0/\vek{b}_1) } \circ
    \gamma_1 \otimes 
      \fuse{ \vek{a}_0/\vek{b}_1 }{1}{ \vek{a}_0/\vek{b}_1 } \circ
    \fuse{ \vek{a}_1 (\vek{a}_0/\vek{b}_1) }{1}{ \vek{b}_0 }
    = \displaybreak[0]\\ ={}&
    \fuse{\vek{c}}{1}{ \vek{b}_1 (\vek{c}/\vek{b}_1) } \circ
    \fuse{ \vek{b}_1 (\vek{c}/\vek{b}_1) }{1}{ 
      \vek{b}_1 (\vek{a}_0/\vek{b}_1) } \circ
    \gamma_1 \otimes 
      \fuse{ \vek{a}_0/\vek{b}_1 }{1}{ \vek{a}_0/\vek{b}_1 } \circ
    \fuse{ \vek{a}_1 (\vek{a}_0/\vek{b}_1) }{1}{ \vek{b}_0 }
    = \displaybreak[0]\\ ={}&
    \fuse{\vek{c}}{1}{ \vek{b}_1 (\vek{c}/\vek{b}_1) } \circ
    \fuse{ \vek{b}_1 }{1}{ \vek{b}_1 } \otimes
      \fuse{ \vek{c}/\vek{b}_1 }{1}{ \vek{a}_0/\vek{b}_1 } \circ
    \gamma_1 \otimes 
      \fuse{ \vek{a}_0/\vek{b}_1 }{1}{ \vek{a}_0/\vek{b}_1 } \circ
    \fuse{ \vek{a}_1 (\vek{a}_0/\vek{b}_1) }{1}{ \vek{b}_0 }
    = \displaybreak[0]\\ ={}&
    \fuse{\vek{c}}{1}{ \vek{b}_1 (\vek{c}/\vek{b}_1) } \circ
    \Bigl( \fuse{ \vek{b}_1 }{1}{ \vek{b}_1 } \circ \gamma_1 \Bigr)
      \otimes \Bigl(
      \fuse{ \vek{c}/\vek{b}_1 }{1}{ \vek{a}_0/\vek{b}_1 } \circ
      \fuse{ \vek{a}_0/\vek{b}_1 }{1}{ \vek{a}_0/\vek{b}_1 } 
      \Bigr) \circ
    \fuse{ \vek{a}_1 (\vek{a}_0/\vek{b}_1) }{1}{ \vek{b}_0 }
    = \displaybreak[0]\\ ={}&
    \fuse{\vek{c}}{1}{ \vek{b}_1 (\vek{c}/\vek{b}_1) } \circ
    \Bigl( \gamma_1 \circ \fuse{ \vek{a}_1 }{1}{ \vek{a}_1 } \Bigr)
      \otimes \Bigl(
      \fuse{ \vek{c}/\vek{b}_1 }{1}{ \vek{c}/\vek{b}_1 } \circ
      \fuse{ \vek{c}/\vek{b}_1 }{1}{ \vek{a}_0/\vek{b}_1 } 
      \Bigr) \circ
    \fuse{ \vek{a}_1 (\vek{a}_0/\vek{b}_1) }{1}{ \vek{b}_0 }
    = \displaybreak[0]\\ ={}&
    \fuse{\vek{c}}{1}{ \vek{b}_1 (\vek{c}/\vek{b}_1) } \circ
    \gamma_1 \otimes 
      \fuse{ \vek{c}/\vek{b}_1 }{1}{ \vek{c}/\vek{b}_1 } \circ 
    \fuse{ \vek{a}_1 }{1}{ \vek{a}_1 } \otimes
      \fuse{ \vek{c}/\vek{b}_1 }{1}{ \vek{a}_0/\vek{b}_1 } \circ
    \fuse{ \vek{a}_1 (\vek{a}_0/\vek{b}_1) }{1}{ \vek{b}_0 }
    = \displaybreak[0]\\ ={}&
    \fuse{\vek{c}}{1}{ \vek{b}_1 (\vek{c}/\vek{b}_1) } \circ
    \gamma_1 \otimes 
      \fuse{ \vek{c}/\vek{b}_1 }{1}{ \vek{c}/\vek{b}_1 } \circ 
    \fuse{ \vek{a}_1(\vek{c}/\vek{b}_1) }{1}{ 
      \vek{a}_1(\vek{a}_0/\vek{b}_1) } \circ
    \fuse{ \vek{a}_1 (\vek{a}_0/\vek{b}_1) }{1}{ \vek{b}_0 }
    = \\ ={}&
    \fuse{\vek{c}}{1}{ \vek{b}_1 (\vek{c}/\vek{b}_1) } \circ
    \gamma_1 \otimes 
      \fuse{ \vek{c}/\vek{b}_1 }{1}{ \vek{c}/\vek{b}_1 } \circ 
    \fuse{ \vek{a}_1(\vek{c}/\vek{b}_1) }{1}{ \vek{b}_0 }
    =
    \fuse{\vek{c}}{ (\gamma_1)^{\vek{b}_1}_{\vek{a}_1} }{\vek{b}_0}
    \text{.}
  \end{align*}
  
  If \(n \geqslant 2\) then assume the lemma holds for all 
  expressions of order less than $n$. Let $k$ be minimal such that 
  \(\vek{b}_k \mid \vek{a}_0\), let \(A = \prod_{i=1}^{k-1} 
  (\gamma_i)^{\vek{b}_i}_{\vek{a}_i}\) and let \(B = \prod_{i=k+1}^n 
  (\gamma_i)^{\vek{b}_i}_{\vek{a}_i}\). Then
  \begin{multline*}
    \fuse{\vek{c}}{1}{\vek{a}_0} \circ 
    \fuse{\vek{a}_0}{ A (\gamma_k)^{\vek{b}_k}_{\vek{a}_k} B }
      {\vek{b}_0} 
    = 
    \fuse{\vek{c}}{1}{\vek{a}_0} \circ 
    \fuse{\vek{a}_0}{ (\gamma_k)^{\vek{b}_k}_{\vek{a}_k} }
      { \vek{a}_k(\vek{a}_0/\vek{b}_0) } \circ
    \fuse{ \vek{a}_k(\vek{a}_0/\vek{b}_0) }{ A B }{\vek{b}_0} 
    = \\ =
    \fuse{\vek{c}}{ (\gamma_k)^{\vek{b}_k}_{\vek{a}_k} }
      { \vek{a}_k(\vek{a}_0/\vek{b}_0) } \circ
    \fuse{ \vek{a}_k(\vek{a}_0/\vek{b}_0) }{ A B }{\vek{b}_0} 
    = \\ =
    \fuse{\vek{c}}{ (\gamma_k)^{\vek{b}_k}_{\vek{a}_k} }
      { \vek{a}_k(\vek{c}/\vek{b}_0) } \circ
    \fuse{ \vek{a}_k(\vek{c}/\vek{b}_0) }{1}
      { \vek{a}_k(\vek{a}_0/\vek{b}_0) } \circ
    \fuse{ \vek{a}_k(\vek{a}_0/\vek{b}_0) }{ A B }{\vek{b}_0} 
    = \\ =
    \fuse{\vek{c}}{ (\gamma_k)^{\vek{b}_k}_{\vek{a}_k} }
      { \vek{a}_k(\vek{c}/\vek{b}_0) } \circ
    \fuse{ \vek{a}_k(\vek{c}/\vek{b}_0) }{ A B }{\vek{b}_0} 
    = 
    \fuse{\vek{c}}{ A (\gamma_k)^{\vek{b}_k}_{\vek{a}_k} B }{\vek{b}_0}
  \end{multline*}
  where the second last step uses the induction hypothesis and the 
  last step is by definition.
\end{proof}

That composition of two abstract index expressions is a kind of 
concatenation happens not only when one of the subexpressions is of 
order~$0$, but also in general.

\begin{theorem} \label{S:AIN,snitt}
  If \(\vek{b}_0 = \vek{c}_0\) and 
  \(\norm[\Big]{ \prod_{i=0}^m \vek{b}_i } 
  \cap \norm[\Big]{ \prod_{i=0}^n \vek{c}_i } = \norm{\vek{b}_0}\) 
  then
  \begin{equation} \label{Eq:AIN,snitt}
    \fuse[\Big]{ \vek{a}_0 }{ \textstyle
      \prod_{i=1}^m (\beta_i)^{\vek{b}_i}_{\vek{a}_i}
    }{ \vek{b}_0 } \circ
    \fuse[\Big]{ \vek{c}_0 }{ \textstyle
      \prod_{i=1}^n (\gamma_i)^{\vek{d}_i}_{\vek{c}_i}
    }{ \vek{d}_0 } =
    \fuse[\Big]{ \vek{a}_0 }{ \textstyle
      \prod_{i=1}^m (\beta_i)^{\vek{b}_i}_{\vek{a}_i}
      \prod_{i=1}^n (\gamma_i)^{\vek{d}_i}_{\vek{c}_i}
    }{ \vek{d}_0 }
    \text{.}
  \end{equation}
\end{theorem}
\begin{proof}
  If \(m=0\) or \(n=0\) then this is just Lemma~\ref{L:fperm-left} 
  or~\ref{L:fperm-right} respectively. Hence it can be assumed that 
  \(m,n \geqslant 1\).
  
  Let \(\vek{e}_0 = \vek{a}_0\), and recursively define
  \begin{align*}
    k_i ={}& \min \biggl( 
      \setOf[\Big]{ k \in \{1,\dotsc,m\} }{ 
        (\vek{b}_k \mid \vek{e}_{i-1})
      } \setminus \{k_1,\dotsc,k_{i-1}\} 
    \biggr) \text{,}\\
    \vek{e}_i ={}& \vek{a}_{k_i}(\vek{e}_{i-1}/\vek{b}_{k_i})
  \end{align*}
  for \(i=1,\dotsc,m\). Then by definition
  \begin{equation*}
    \fuse[\Big]{ \vek{a}_0 }{ \textstyle
      \prod_{i=1}^m (\beta_i)^{\vek{b}_i}_{\vek{a}_i}
    }{ \vek{e}_m } =
    \fuse{ \vek{e}_0 }{ 
      (\beta_{k_1})^{\vek{b}_{k_1}}_{\vek{a}_{k_1}} 
    }{ \vek{e}_1 } \circ
    \dotsb \circ
    \fuse{ \vek{e}_{m-1} }{ 
      (\beta_{k_m})^{\vek{b}_{k_m}}_{\vek{a}_{k_m}} 
    }{ \vek{e}_m }
    \text{.}
  \end{equation*}
  
  Let $Q_1$ and $Q_2$ be the partial orders on labels in the two 
  abstract index expressions in the left hand side of 
  \eqref{Eq:AIN,snitt}. Let \(V_1 = \bigcup_{i=1}^m \norm{\vek{b}_i}\), 
  \(V_2 = \bigcup_{i=1}^n \norm{\vek{c}_i}\), and \(V_0 = 
  \norm{\vek{b}_0}\). Define a partial order $Q$ on $V_0 \cup V_1 
  \cup V_2$ by having \(x \leqslant y \pin{Q}\) iff
  \begin{enumerate}
    \item
      \(x,y \in V_1\) and \(x \leqslant y \pin{Q_1}\), or
    \item
      \(x,y \in V_0\) and \(x=y\), or
    \item
      \(x,y \in V_2\) and \(x \leqslant y \pin{Q_2}\), or
    \item
      \(x \in V_1\) and \(y \in V_0 \cup V_2\), or
    \item
      \(x \in V_0\) and \(y \in V_2\).
  \end{enumerate}
  This partial order fulfills the acyclicity condition for the right 
  hand side of \eqref{Eq:AIN,snitt}, and so it is clear that that 
  exists. Since \(m+n \geqslant 2\), repeated use of 
  rule~\ref{Item3:D:fuse} of Definition~\ref{D:fuse} on this right 
  hand side yields
  \begin{multline*}
    \fuse[\Big]{ \vek{a}_0 }{ \textstyle
      \prod_{i=1}^m (\beta_i)^{\vek{b}_i}_{\vek{a}_i}
      \prod_{i=1}^n (\gamma_i)^{\vek{d}_i}_{\vek{c}_i}
    }{ \vek{d}_0 }
    = \\ =
    \fuse{ \vek{e}_0 }{ 
      (\beta_{k_1})^{\vek{b}_{k_1}}_{\vek{a}_{k_1}} 
    }{ \vek{e}_1 } \circ
    \dotsb \circ
    \fuse{ \vek{e}_{m-1} }{ 
      (\beta_{k_m})^{\vek{b}_{k_m}}_{\vek{a}_{k_m}} 
    }{ \vek{e}_m } \circ
    \fuse[\Big]{ \vek{e}_m }{ \textstyle
      \prod_{i=1}^n (\gamma_i)^{\vek{d}_i}_{\vek{c}_i}
    }{ \vek{d}_0 } 
    = \\ =
    \fuse[\Big]{ \vek{a}_0 }{ \textstyle
      \prod_{i=1}^m (\beta_i)^{\vek{b}_i}_{\vek{a}_i}
    }{ \vek{e}_m } \circ
    \fuse[\Big]{ \vek{e}_m }{ \textstyle
      \prod_{i=1}^n (\gamma_i)^{\vek{d}_i}_{\vek{c}_i}
    }{ \vek{d}_0 } 
    = \\ =
    \fuse[\Big]{ \vek{a}_0 }{ \textstyle
      \prod_{i=1}^m (\beta_i)^{\vek{b}_i}_{\vek{a}_i}
    }{ \vek{b}_0 } \circ
    \fuse{ \vek{b}_0 }{1}{ \vek{e}_m } \circ
    \fuse[\Big]{ \vek{e}_m }{ \textstyle
      \prod_{i=1}^n (\gamma_i)^{\vek{d}_i}_{\vek{c}_i}
    }{ \vek{d}_0 } 
    = \\ =
    \fuse[\Big]{ \vek{a}_0 }{ \textstyle
      \prod_{i=1}^m (\beta_i)^{\vek{b}_i}_{\vek{a}_i}
    }{ \vek{b}_0 } \circ
    \fuse[\Big]{ \vek{b}_0 }{ \textstyle
      \prod_{i=1}^n (\gamma_i)^{\vek{d}_i}_{\vek{c}_i}
    }{ \vek{d}_0 } \text{,}
  \end{multline*}
  where Lemmas~\ref{L:fperm-right} and~\ref{L:fperm-left} are used to 
  adjust the order of labels at the composition to match \(\vek{b}_0 
  = \vek{c}_0\).
\end{proof}

\begin{lemma} \label{L:fuse-pad}
  If \(\norm[\Big]{ \prod_{i=0}^m \vek{b}_i }\) and 
  \(\norm{ \vek{c} }\) are disjoint then
  \begin{equation} \label{Eq:fuse-pad}
    \fuse[\Big]{ \vek{a}_0 }{ \textstyle
      \prod_{i=1}^n (\gamma_i)^{\vek{b}_i}_{\vek{a}_i}
    }{ \vek{b}_0 } \otimes
    \fuse{ \vek{c} }{1}{ \vek{c} } 
    =
    \fuse[\Big]{ \vek{a}_0 \vek{c} }{ \textstyle
      \prod_{i=1}^n (\gamma_i)^{\vek{b}_i}_{\vek{a}_i}
    }{ \vek{b}_0 \vek{c} }
    \text{.}
  \end{equation}
\end{lemma}
\begin{proof}
  This is proved through another induction on order. If \(n=0\) then 
  this is a special case of Lemma~\ref{L:tensor.perm.}. If \(n=1\) 
  then by direct calculation,
  \begin{multline*}
    \fuse{ \vek{a}_0 }{ (\gamma_1)^{\vek{b}_1}_{\vek{a}_1} }
      { \vek{b}_0 } \otimes
    \fuse{ \vek{c} }{1}{ \vek{c} } 
    = \\ =
    \Bigl(
    \fuse{ \vek{a}_0 }{1}{ \vek{b}_1(\vek{a}_0/\vek{b}_1) } \circ
    \gamma_1 \otimes 
      \fuse{ \vek{a}_0/\vek{b}_1 }{1}{ \vek{a}_0/\vek{b}_1 } \circ
    \fuse{ \vek{a}_1(\vek{a}_0/\vek{b}_1) }{1}{ \vek{b}_0 } 
    \Bigr) \otimes \Bigl(
    \fuse{ \vek{c} }{1}{ \vek{c} } \circ
    \fuse{ \vek{c} }{1}{ \vek{c} } \circ
    \fuse{ \vek{c} }{1}{ \vek{c} }
    \Bigr)
    = \displaybreak[0]\\ =
    \fuse{ \vek{a}_0 }{1}{ \vek{b}_1(\vek{a}_0/\vek{b}_1) } 
      \otimes \fuse{ \vek{c} }{1}{ \vek{c} } \circ
    \gamma_1 \otimes 
      \fuse{ \vek{a}_0/\vek{b}_1 }{1}{ \vek{a}_0/\vek{b}_1 } 
      \otimes \fuse{ \vek{c} }{1}{ \vek{c} } \circ
    \fuse{ \vek{a}_1(\vek{a}_0/\vek{b}_1) }{1}{ \vek{b}_0 } \otimes
      \fuse{ \vek{c} }{1}{ \vek{c} }
    = \displaybreak[0]\\ =
    \fuse{ \vek{a}_0\vek{c} }{1}
      { \vek{b}_1(\vek{a}_0/\vek{b}_1)\vek{c} } \circ
    \gamma_1 \otimes 
      \fuse{ (\vek{a}_0/\vek{b}_1)\vek{c} }{1}
        { (\vek{a}_0/\vek{b}_1)\vek{c} } \circ
    \fuse{ \vek{a}_1(\vek{a}_0/\vek{b}_1)\vek{c} }{1}
      { \vek{b}_0\vek{c} }
    = \\ =
    \fuse{ \vek{a}_0\vek{c} }{1}
      { \vek{b}_1(\vek{a}_0\vek{c}/\vek{b}_1) } \circ
    \gamma_1 \otimes 
      \fuse{ \vek{a}_0\vek{c}/\vek{b}_1 }{1}
        { \vek{a}_0\vek{c}/\vek{b}_1 } \circ
    \fuse{ \vek{a}_1(\vek{a}_0\vek{c}/\vek{b}_1) }{1}
      { \vek{b}_0\vek{c} }
    =
    \fuse{ \vek{a}_0\vek{c} }{ (\gamma_1)^{\vek{b}_1}_{\vek{a}_1} }
      { \vek{b}_0\vek{c} }
    \text{.}
  \end{multline*}
  Finally, if \(n \geqslant 2\), \(k \geqslant 1\) is minimal such 
  that \(\vek{b}_k \mid \vek{a}_0\), and assuming \eqref{Eq:fuse-pad} 
  holds for expressions of order $n-1$, let \(A = \prod_{i=1}^{k-1} 
  (\gamma_i)^{\vek{b}_i}_{\vek{a}_i}\) and \(B = \prod_{i=k+1}^n 
  (\gamma_i)^{\vek{b}_i}_{\vek{a}_i}\). Then
  \begin{multline*}
    \fuse{ \vek{a}_0 }{ A (\gamma_k)^{\vek{b}_k}_{\vek{a}_k} B }
      { \vek{b}_0 } \otimes
    \fuse{ \vek{c} }{1}{ \vek{c} }
    = \\ =
    \Bigl(
    \fuse{ \vek{a}_0 }{ (\gamma_k)^{\vek{b}_k}_{\vek{a}_k} }
      { \vek{a}_k(\vek{a}_0/\vek{b}_k) } \circ
    \fuse{ \vek{a}_k(\vek{a}_0/\vek{b}_k) }{ A B }{ \vek{b}_0 }
    \Bigr) \otimes \Bigl(
    \fuse{ \vek{c} }{1}{ \vek{c} } \circ \fuse{ \vek{c} }{1}{ \vek{c} }
    \Bigr)
    = \\ =
    \fuse{ \vek{a}_0 }{ (\gamma_k)^{\vek{b}_k}_{\vek{a}_k} }
      { \vek{a}_k(\vek{a}_0/\vek{b}_k) } \otimes
      \fuse{ \vek{c} }{1}{ \vek{c} } \circ
    \fuse{ \vek{a}_k(\vek{a}_0/\vek{b}_k) }{ A B }{ \vek{b}_0 }
      \otimes \fuse{ \vek{c} }{1}{ \vek{c} }
    = \\ =
    \fuse{ \vek{a}_0\vek{c} }{ (\gamma_k)^{\vek{b}_k}_{\vek{a}_k} }
      { \vek{a}_k(\vek{a}_0/\vek{b}_k)\vek{c} } \circ
    \fuse{ \vek{a}_k(\vek{a}_0/\vek{b}_k) \vek{c}}{ A B }
      { \vek{b}_0 \vek{c} }
    =
    \fuse{ \vek{a}_0\vek{c} }{ A (\gamma_k)^{\vek{b}_k}_{\vek{a}_k} B }
      { \vek{b}_0\vek{c} } \text{.}
  \end{multline*}
\end{proof}

\begin{theorem} \label{S:AIN,klyva}
  If \(\norm[\Big]{ \prod_{i=0}^m \vek{b}_i }\) and 
  \(\norm[\Big]{ \prod_{i=0}^n \vek{c}_i }\) are disjoint then
  \begin{equation} \label{Eq:AIN,klyva}
    \fuse[\Big]{ \vek{a}_0 }{ \textstyle
      \prod_{i=1}^m (\beta_i)^{\vek{b}_i}_{\vek{a}_i}
    }{ \vek{b}_0 } \otimes
    \fuse[\Big]{ \vek{c}_0 }{ \textstyle
      \prod_{i=1}^n (\gamma_i)^{\vek{d}_i}_{\vek{c}_i}
    }{ \vek{d}_0 } =
    \fuse[\Big]{ \vek{a}_0 \vek{c}_0 }{ \textstyle
      \prod_{i=1}^m (\beta_i)^{\vek{b}_i}_{\vek{a}_i}
      \prod_{i=1}^n (\gamma_i)^{\vek{d}_i}_{\vek{c}_i}
    }{ \vek{b}_0 \vek{d}_0 }
    \text{.}
  \end{equation}
\end{theorem}
\begin{proof}
  First, the special case that \(m=0\) and \(\vek{a}_0 = \vek{b}_0\); 
  this is complementary to Lemma~\ref{L:fuse-pad}, and uses it to 
  eliminate the tensor product:
  \begin{multline*}
    \fuse{ \vek{b}_0 }{1}{ \vek{b}_0 } \otimes
    \fuse[\Big]{ \vek{c}_0 }{ \textstyle
      \prod_{i=1}^n (\gamma_i)^{\vek{d}_i}_{\vek{c}_i}
    }{ \vek{d}_0 } 
    = \\ =
    \fuse{ \vek{b}_0\vek{c}_0 }{1}{ \vek{c}_0\vek{b}_0 } \circ
    \fuse{ \vek{c}_0\vek{b}_0 }{1}{ \vek{b}_0\vek{c}_0 } \circ
    \fuse{ \vek{b}_0 }{1}{ \vek{b}_0 } \otimes
    \fuse[\Big]{ \vek{c}_0 }{ \textstyle
      \prod_{i=1}^n (\gamma_i)^{\vek{d}_i}_{\vek{c}_i}
    }{ \vek{d}_0 }
    = \\ =
    \fuse{ \vek{b}_0\vek{c}_0 }{1}{ \vek{c}_0\vek{b}_0 } \circ
    \fuse[\Big]{ \vek{c}_0 }{ \textstyle
      \prod_{i=1}^n (\gamma_i)^{\vek{d}_i}_{\vek{c}_i}
    }{ \vek{d}_0 } \otimes
      \fuse{ \vek{b}_0 }{1}{ \vek{b}_0 } \circ
    \fuse{ \vek{d}_0\vek{b}_0 }{1}{ \vek{b}_0\vek{d}_0 } 
    = \\ =
    \fuse{ \vek{b}_0\vek{c}_0 }{1}{ \vek{c}_0\vek{b}_0 } \circ
    \fuse[\Big]{ \vek{c}_0\vek{b}_0 }{ \textstyle
      \prod_{i=1}^n (\gamma_i)^{\vek{d}_i}_{\vek{c}_i}
    }{ \vek{d}_0\vek{b}_0 } \circ
    \fuse{ \vek{d}_0\vek{b}_0 }{1}{ \vek{b}_0\vek{d}_0 } 
    =
    \fuse[\Big]{ \vek{b}_0\vek{c}_0 }{ \textstyle
      \prod_{i=1}^n (\gamma_i)^{\vek{d}_i}_{\vek{c}_i}
    }{ \vek{b}_0\vek{d}_0 }
    \text{.}
  \end{multline*}
  The general case is proved using Theorem~\ref{S:AIN,snitt} and the 
  two special cases where one factor is an identity.
  \begin{multline*}
    \fuse[\Big]{ \vek{a}_0 }{ \textstyle
      \prod_{i=1}^m (\beta_i)^{\vek{b}_i}_{\vek{a}_i}
    }{ \vek{b}_0 } \otimes
    \fuse[\Big]{ \vek{c}_0 }{ \textstyle
      \prod_{i=1}^n (\gamma_i)^{\vek{d}_i}_{\vek{c}_i}
    }{ \vek{d}_0 }
    = \\ =
    \biggl(
    \fuse[\Big]{ \vek{a}_0 }{ \textstyle
      \prod_{i=1}^m (\beta_i)^{\vek{b}_i}_{\vek{a}_i}
    }{ \vek{b}_0 } \circ
    \fuse{ \vek{b}_0 }{1}{ \vek{b}_0 }
    \biggr) \otimes \biggl(
    \fuse{ \vek{c}_0 }{1}{ \vek{c}_0 } \circ
    \fuse[\Big]{ \vek{c}_0 }{ \textstyle
      \prod_{i=1}^n (\gamma_i)^{\vek{d}_i}_{\vek{c}_i}
    }{ \vek{d}_0 }
    \biggr)
    = \displaybreak[0]\\ =
    \fuse[\Big]{ \vek{a}_0 }{ \textstyle
      \prod_{i=1}^m (\beta_i)^{\vek{b}_i}_{\vek{a}_i}
    }{ \vek{b}_0 } \otimes \fuse{ \vek{c}_0 }{1}{ \vek{c}_0 } \circ
    \fuse{ \vek{b}_0 }{1}{ \vek{b}_0 } \otimes 
    \fuse[\Big]{ \vek{c}_0 }{ \textstyle
      \prod_{i=1}^n (\gamma_i)^{\vek{d}_i}_{\vek{c}_i}
    }{ \vek{d}_0 }
    = \displaybreak[0]\\ =
    \fuse[\Big]{ \vek{a}_0\vek{c}_0 }{ \textstyle
      \prod_{i=1}^m (\beta_i)^{\vek{b}_i}_{\vek{a}_i}
    }{ \vek{b}_0\vek{c}_0 } \circ
    \fuse[\Big]{ \vek{b}_0\vek{c}_0 }{ \textstyle
      \prod_{i=1}^n (\gamma_i)^{\vek{d}_i}_{\vek{c}_i}
    }{ \vek{b}_0\vek{d}_0 }
    = \\ =
    \fuse[\Big]{ \vek{a}_0\vek{c}_0 }{ \textstyle
      \prod_{i=1}^m (\beta_i)^{\vek{b}_i}_{\vek{a}_i}
      \prod_{i=1}^n (\gamma_i)^{\vek{d}_i}_{\vek{c}_i}
    }{ \vek{b}_0\vek{d}_0 }
    \text{.}
  \end{multline*}
\end{proof}

The final thing about the abstract index notation that needs to 
be uncovered is that some details in it are irrelevant, namely 
the choice of labels and the order of factors. The first of these is 
mostly an observation.

\begin{theorem} \label{S:fuse-relabel}
  For every injection $f$ on the set of labels,
  \begin{equation}
    \fuse[\Big]{\vek{a}_0}{ \textstyle
      \prod_{i=1}^n (\gamma_i)^{\vek{b}_i}_{\vek{a}_i} 
    }{\vek{b}_0} =
    \fuse[\Big]{f(\vek{a}_0)}{ \textstyle
      \prod_{i=1}^n (\gamma_i)^{f(\vek{b}_i)}_{f(\vek{a}_i)} 
    }{f(\vek{b}_0)}
    \text{,}
  \end{equation}
  where $f$ is extended to lists of labels by acting on each element 
  separately, i.e., \(f(a_1 \dotsb a_m) := f(a_1) \dotsb f(a_m)\).
\end{theorem}
\begin{proof}
  No part of Definition~\ref{D:fuse} depends on the value of 
  individual labels, only on in which positions there are equal 
  labels. This is preserved by injections.
\end{proof}

The next lemma generalises rule~\ref{Item3:D:fuse} of 
Definition~\ref{D:fuse} in several ways: it covers also expressions 
of order~$1$, it allows an arbitrary order for the list $\vek{c}$ of 
intermediate labels, and most importantly it ignores the order of 
the labelled factors.

\begin{lemma} \label{L:fuse-anyfactor}
  If \(n \geqslant 1\), \(l\geqslant 1\) is such that \(\vek{b}_l 
  \mid \vek{a}_0\), and \(\norm{\vek{c}} = 
  \norm[\big]{ \vek{a}_l (\vek{a}_0/\vek{b}_l) }\) then
  \begin{equation} \label{Eq:fuse-anyfactor}
    \fuse[\Big]{\vek{a}_0}{ \textstyle
      \prod_{i=1}^n (\gamma_i)^{\vek{b}_i}_{\vek{a}_i} 
    }{\vek{b}_0} = 
    \fuse{\vek{a}_0}{ \textstyle
      (\gamma_l)^{\vek{b}_l}_{\vek{a}_l} 
    }{ \vek{c} } \circ
    \fuse[\Big]{ \vek{c} }{ \textstyle
      \prod_{i=1}^{l-1} (\gamma_i)^{\vek{b}_i}_{\vek{a}_i} 
      \prod_{i=l+1}^n (\gamma_i)^{\vek{b}_i}_{\vek{a}_i} 
    }{\vek{b}_0} 
    \text{.}
  \end{equation}
\end{lemma}
\begin{proof}
  If \(n=1\) then this is just the claim of 
  Lemma~\ref{L:fperm-right}. For \(n=2\) and \(l=1\) this is a 
  special case of Theorem~\ref{S:AIN,snitt}. For \(n=2\) and \(l=2\) 
  there are two possibilities. Either \(\vek{b}_1 \nmid \vek{a}_0\) 
  and then by definition, Lemma~\ref{L:fperm-right}, and 
  Lemma~\ref{L:fperm-left},
  \begin{multline*}
    \fuse{ \vek{a}_0 }{ 
      (\gamma_1)^{\vek{b}_1}_{\vek{a}_1}
      (\gamma_2)^{\vek{b}_2}_{\vek{a}_2}
    }{ \vek{b}_0 }
    =
    \fuse{ \vek{a}_0 }{ 
      (\gamma_2)^{\vek{b}_2}_{\vek{a}_2}
    }{ \vek{a}_2 (\vek{a}_0/\vek{b}_2) } \circ
    \fuse{ \vek{a}_2 (\vek{a}_0/\vek{b}_2) }{
      (\gamma_1)^{\vek{b}_1}_{\vek{a}_1}
    }{ \vek{b}_0 }
    = \\ =
    \fuse{ \vek{a}_0 }{ 
      (\gamma_2)^{\vek{b}_2}_{\vek{a}_2}
    }{ \vek{a}_2 (\vek{a}_0/\vek{b}_2) } \circ
    \fuse{ \vek{a}_2 (\vek{a}_0/\vek{b}_2) }{1}{ \vek{c} } \circ
    \fuse{ \vek{c} }{1}{ { \vek{a}_2 (\vek{a}_0/\vek{b}_2) } } \circ
    \fuse{ \vek{a}_2 (\vek{a}_0/\vek{b}_2) }{
      (\gamma_1)^{\vek{b}_1}_{\vek{a}_1}
    }{ \vek{b}_0 }
    = \\ =
    \fuse{ \vek{a}_0 }{ 
      (\gamma_2)^{\vek{b}_2}_{\vek{a}_2}
    }{ \vek{c} } \circ
    \fuse{ \vek{c} }{
      (\gamma_1)^{\vek{b}_1}_{\vek{a}_1}
    }{ \vek{b}_0 }
    \text{.}
  \end{multline*}
  Otherwise \(\vek{b}_1 \mid \vek{a}_0\), where the definition 
  instead puts $\gamma_1$ in the left operand. Let \(\vek{d} = 
  \vek{a}_0/\vek{b}_2\vek{b}_1\). By Theorems~\ref{S:AIN,snitt} 
  and~\ref{S:AIN,klyva},
  \begin{multline*}
    \fuse{ \vek{a}_0 }{ 
      (\gamma_1)^{\vek{b}_1}_{\vek{a}_1}
      (\gamma_2)^{\vek{b}_2}_{\vek{a}_2}
    }{ \vek{b}_0 }
    =
    \fuse{ \vek{a}_0 }{1}{ \vek{b}_2\vek{b}_1\vek{d} } \circ
    \fuse{ \vek{b}_2\vek{b}_1\vek{d} }{ 
      (\gamma_1)^{\vek{b}_1}_{\vek{a}_1}
    }{ \vek{b}_2\vek{a}_1\vek{d} } \circ
    \fuse{ \vek{b}_2\vek{a}_1\vek{d} }{
      (\gamma_2)^{\vek{b}_2}_{\vek{a}_2}
    }{ \vek{a}_2\vek{a}_1\vek{d} } \circ
    \fuse{ \vek{a}_2\vek{a}_1\vek{d} }{1}{ \vek{b}_0 }
    = \\ =
    \fuse{ \vek{a}_0 }{1}{ \vek{b}_2\vek{b}_1\vek{d} } \circ
    \fuse{ \vek{b}_2 }{1}{ \vek{b}_2 } \otimes
    \fuse{ \vek{b}_1\vek{d} }{ (\gamma_1)^{\vek{b}_1}_{\vek{a}_1}
    }{ \vek{a}_1\vek{d} } \circ
    \fuse{ \vek{b}_2 }{ (\gamma_2)^{\vek{b}_2}_{\vek{a}_2} }
      { \vek{a}_2 } \otimes
      \fuse{ \vek{a}_1\vek{d} }{1}{ \vek{a}_1\vek{d} } \circ
    \fuse{ \vek{a}_2\vek{a}_1\vek{d} }{1}{ \vek{b}_0 }
    = \displaybreak[0]\\ =
    \fuse{ \vek{a}_0 }{1}{ \vek{b}_2\vek{b}_1\vek{d} } \circ
    \Bigl( 
    \fuse{ \vek{b}_2 }{1}{ \vek{b}_2 } \circ
    \fuse{ \vek{b}_2 }{ (\gamma_2)^{\vek{b}_2}_{\vek{a}_2} }
      { \vek{a}_2 }
    \Bigr) \otimes \Bigl(
    \fuse{ \vek{b}_1\vek{d} }{ (\gamma_1)^{\vek{b}_1}_{\vek{a}_1}
    }{ \vek{a}_1\vek{d} } \circ
    \fuse{ \vek{a}_1\vek{d} }{1}{ \vek{a}_1\vek{d} } 
    \Bigr) \circ
    \fuse{ \vek{a}_2\vek{a}_1\vek{d} }{1}{ \vek{b}_0 }
    = \displaybreak[0]\\ =
    \fuse{ \vek{a}_0 }{1}{ \vek{b}_2\vek{b}_1\vek{d} } \circ
    \Bigl( 
    \fuse{ \vek{b}_2 }{ (\gamma_2)^{\vek{b}_2}_{\vek{a}_2} }
      { \vek{a}_2 } \circ
    \fuse{ \vek{a}_2 }{1}{ \vek{a}_2 } 
    \Bigr) \otimes \Bigl(
    \fuse{ \vek{b}_1\vek{d} }{1}{ \vek{b}_1\vek{d} } \circ
    \fuse{ \vek{b}_1\vek{d} }{ (\gamma_1)^{\vek{b}_1}_{\vek{a}_1}
    }{ \vek{a}_1\vek{d} } 
    \Bigr) \circ
    \fuse{ \vek{a}_2\vek{a}_1\vek{d} }{1}{ \vek{b}_0 }
    = \displaybreak[0]\\ =
    \fuse{ \vek{a}_0 }{1}{ \vek{b}_2\vek{b}_1\vek{d} } \circ
    \fuse{ \vek{b}_2 }{ (\gamma_2)^{\vek{b}_2}_{\vek{a}_2} }
      { \vek{a}_2 } \otimes
      \fuse{ \vek{b}_1\vek{d} }{1}{ \vek{b}_1\vek{d} } \circ
    \fuse{ \vek{a}_2 }{1}{ \vek{a}_2 } \otimes 
      \fuse{ \vek{b}_1\vek{d} }{ (\gamma_1)^{\vek{b}_1}_{\vek{a}_1}
      }{ \vek{a}_1\vek{d} } 
      \circ
    \fuse{ \vek{a}_2\vek{a}_1\vek{d} }{1}{ \vek{b}_0 }
    = \displaybreak[0]\\ =
    \fuse{ \vek{a}_0 }{1}{ \vek{b}_2\vek{b}_1\vek{d} } \circ
    \fuse{ \vek{b}_2\vek{b}_1\vek{d} }
      { (\gamma_2)^{\vek{b}_2}_{\vek{a}_2} }
      { \vek{a}_2\vek{b}_1\vek{d} } \circ
    \fuse{ \vek{a}_2\vek{b}_1\vek{d} }
      { (\gamma_1)^{\vek{b}_1}_{\vek{a}_1}}
      { \vek{a}_2\vek{a}_1\vek{d} } \circ
    \fuse{ \vek{a}_2\vek{a}_1\vek{d} }{1}{ \vek{b}_0 }
    = \\ =
    \fuse{ \vek{a}_0 }{ 
      (\gamma_2)^{\vek{b}_2}_{\vek{a}_2}
      (\gamma_1)^{\vek{b}_1}_{\vek{a}_1}
    }{ \vek{b}_0 }
    =
    \fuse{ \vek{a}_0 }{ (\gamma_2)^{\vek{b}_2}_{\vek{a}_2} }
      { \vek{c} } \circ
    \fuse{ \vek{c} }{ (\gamma_1)^{\vek{b}_1}_{\vek{a}_1} }
      { \vek{b}_0 }
    \text{.}
  \end{multline*}
  
  For \(n > 2\) this is again by induction, so assume 
  \eqref{Eq:fuse-anyfactor} holds for all expressions of order less 
  than $n$. If $l$ is minimal with respect to the condition 
  \(\vek{b}_l \mid \vek{a}_0\) then proving \eqref{Eq:fuse-anyfactor} 
  is just a matter of using the definition and then permuting the 
  intermediate labels, like above. Otherwise let \(k \geqslant 1\) 
  be minimal such that \(\vek{b}_k \mid \vek{a}_0\), 
  let \(A = \prod_{i=1}^{k-1} (\gamma_i)^{\vek{b}_i}_{\vek{a}_i}\), 
  let \(B = \prod_{i=k+1}^{l-1} (\gamma_i)^{\vek{b}_i}_{\vek{a}_i}\), 
  let \(C = \prod_{i=l+1}^n (\gamma_i)^{\vek{b}_i}_{\vek{a}_i}\), and 
  let \(\vek{d} = \vek{a}_k \vek{a}_l (\vek{a}_0/\vek{b}_k\vek{b}_l)\). 
  Then \(\prod_{i=1}^n (\gamma_i)^{\vek{b}_i}_{\vek{a}_i} = A 
  (\gamma_k)^{\vek{b}_k}_{\vek{a}_k} B 
  (\gamma_l)^{\vek{b}_l}_{\vek{a}_l} C\) and
  \begin{align*}
    \fuse[\Big]{\vek{a}_0}{ \textstyle
      \prod_{i=1}^n (\gamma_i)^{\vek{b}_i}_{\vek{a}_i} 
    }{\vek{b}_0} 
    ={}& 
    \fuse{ \vek{a}_0 }{ (\gamma_k)^{\vek{b}_k}_{\vek{a}_k} }
      { \vek{a}_k (\vek{a}_0/\vek{b}_k) } \circ
    \fuse{ \vek{a}_k (\vek{a}_0/\vek{b}_k) }
      { A B (\gamma_l)^{\vek{b}_l}_{\vek{a}_l} C }{ \vek{b}_0 }
    = \\ ={}&
    \fuse{ \vek{a}_0 }{ (\gamma_k)^{\vek{b}_k}_{\vek{a}_k} }
      { \vek{a}_k (\vek{a}_0/\vek{b}_k) } \circ
    \fuse{ \vek{a}_k (\vek{a}_0/\vek{b}_k) }
      { (\gamma_l)^{\vek{b}_l}_{\vek{a}_l} }
      { \vek{d} }
      \circ
    \fuse{ \vek{d} }{ A B C }{ \vek{b}_0 }
    = \displaybreak[0]\\ ={}&
    \fuse{ \vek{a}_0 }{ 
      (\gamma_k)^{\vek{b}_k}_{\vek{a}_k} 
      (\gamma_l)^{\vek{b}_l}_{\vek{a}_l}
    }{ \vek{d} } \circ
    \fuse{ \vek{d} }{ A B C }{ \vek{b}_0 }
    = \displaybreak[0]\\ ={}&
    \fuse{ \vek{a}_0 }{ 
      (\gamma_l)^{\vek{b}_l}_{\vek{a}_l}
    }{ \vek{c} } \circ
    \fuse{ \vek{c} }{ 
      (\gamma_k)^{\vek{b}_k}_{\vek{a}_k} 
    }{ \vek{d} } \circ
    \fuse{ \vek{d} }{ A B C }{ \vek{b}_0 }
    = \\ ={}&
    \fuse{ \vek{a}_0 }{ 
      (\gamma_l)^{\vek{b}_l}_{\vek{a}_l}
    }{ \vek{c} } \circ
    \fuse{ \vek{c} }{ 
      A (\gamma_k)^{\vek{b}_k}_{\vek{a}_k} B C 
    }{ \vek{b}_0 }
  \end{align*}
  where the first step is by definition and all others are by 
  the induction hypothesis.
\end{proof}

\begin{theorem} \label{S:fuse-factor-reorder}
  For every permutation \(\sigma \in \Sigma_n\),
  \begin{equation}
    \fuse[\Big]{\vek{a}_0}{ \textstyle
      \prod_{i=1}^n (\gamma_i)^{\vek{b}_i}_{\vek{a}_i} 
    }{\vek{b}_0} = 
    \fuse[\Big]{\vek{a}_0}{ \textstyle
      \prod_{i=1}^n 
      (\gamma_{\sigma(i)})^{\vek{b}_{\sigma(i)}}_{\vek{a}_{\sigma(i)}} 
    }{\vek{b}_0}
    \text{.}
  \end{equation}
  In other words, the abstract index notation does not care about 
  the order of factors in the product part.
\end{theorem}
\begin{proof}
  The only part of Definition~\ref{D:fuse} which depends on the 
  order of the labelled factors is the minimality of $k$ in 
  rule~\ref{Item3:D:fuse}, but Lemma~\ref{L:fuse-anyfactor} states 
  that \eqref{Eq:D:fuse} holds regardless of the minimality of $k$, 
  hence the order of the factors is irrelevant.
\end{proof}

\subsection{Derivatives}

One thing that becomes a lot simpler to state using abstract index 
notation is the definition of a derivative on a \PROP.

\begin{definition}
  Let $\mc{P}$ be an $\mc{R}$-linear \PROP. An 
  \DefOrd[*{derivative}]{(outer) derivative} $D$ on $\mc{P}$ and a 
  \DefOrd{gradient} $\nabla$ on $\mc{P}$ are degree $(0,1)$ and $(1,0)$ 
  respectively $\mc{R}$-linear maps \(\mc{P} \Fpil \mc{P}\) which 
  satisfy the respective 
  \emDefOrd[{product rule}{product rule (of derivative)}]{product rules}
  \begin{subequations} \label{Eq:Produktregeln}
  \begin{align}
    D\Bigl( \fuse{ \vek{a}_0 }{ 
      \beta_{\vek{a}_1}^{\vek{b}_1} \gamma_{\vek{a}_2}^{\vek{b}_2} 
    }{ \vek{b}_0 } \Bigr)
    ={}&
    \fuse{ \vek{a}_0 }{ 
      D(\beta)_{\vek{a}_1 c}^{\vek{b}_1} \gamma_{\vek{a}_2}^{\vek{b}_2} 
    }{ \vek{b}_0 c } 
    +
    \fuse{ \vek{a}_0 }{ 
      \beta_{\vek{a}_1}^{\vek{b}_1} D(\gamma)_{\vek{a}_2c}^{\vek{b}_2} 
    }{ \vek{b}_0 c } 
    \text{,}
    \\
    \nabla \Bigl( \fuse{ \vek{a}_0 }{ 
      \beta_{\vek{a}_1}^{\vek{b}_1} \gamma_{\vek{a}_2}^{\vek{b}_2} 
    }{ \vek{b}_0 } \Bigr)
    ={}&
    \fuse{ \vek{a}_0 c }{ 
      \nabla(\beta)_{\vek{a}_1}^{\vek{b}_1 c} 
      \gamma_{\vek{a}_2}^{\vek{b}_2} 
    }{ \vek{b}_0 } 
    +
    \fuse{ \vek{a}_0 c }{ 
      \beta_{\vek{a}_1}^{\vek{b}_1} 
      \nabla(\gamma)_{\vek{a}_2}^{\vek{b}_2 c} 
    }{ \vek{b}_0 } 
  \end{align}
  \end{subequations}
  where \(c \notin \norm{\vek{a}_0\vek{a}_1\vek{a}_2} = 
  \norm{\vek{b}_0\vek{b}_1\vek{b}_2}\).
\end{definition}

If $\mc{P}$ is some \PROP\ of fields on a manifold (for example that 
of Example~\ref{Ex:Tensorfalt}), then the $(0,0)$ component is the 
scalar fields (ordinary functions), the $(1,0)$ 
component is the vector fields, and the $(0,1)$ component is 
covectors ($1$-forms). Hence it is natural to call a degree $(0,1)$ 
derivation an outer derivative and a degree $(1,0)$ derivation a 
gradient.

For a pure composition or tensor product, the product rule amounts to
\begin{align*}
  D(\beta \circ \gamma) ={}& 
    D(\beta) \circ \gamma \otimes \phi(\same{1}) + 
    \beta \circ D(\gamma)
    \text{,}\\
  D(\beta \otimes \gamma) ={}& D(\beta) \otimes \gamma \circ 
    \phi( \same{\alpha(\beta)} \star \cross{\alpha(\gamma)}{1} ) +
    \beta \otimes D(\gamma)
  \text{.}
\end{align*}
From the former it follows that
\begin{align*}
  D\bigl( \phi(\same{n}) \bigr)
  =
  D\bigl( \phi(\same{n}) \circ \phi(\same{n}) \bigr)
  ={}&
  D\bigl( \phi(\same{n}) \bigr) \circ 
    \phi(\same{n}) \otimes \phi(\same{1}) +
  \phi(\same{n}) \circ D\bigl( \phi(\same{n}) \bigr)
  = \\ ={}&
  D\bigl( \phi(\same{n}) \bigr) \circ \phi(\same{n+1}) +
  D\bigl( \phi(\same{n}) \bigr)
  =
  2 D\bigl( \phi(\same{n}) \bigr)
\end{align*}
and thus \(D\bigl( \phi(\same{n}) \bigr) = 0\). Moreover,
\begin{align*}
  D\bigl( \phi(\cross{1}{1}) \bigr)
  ={}&
  D\Bigl( \fuse{b_2b_1}{ 
    \phi(\same{1})^{b_1}_{a_1} \phi(\same{1})^{b_2}_{a_2} 
  }{a_1a_2} \Bigr)
  = \\ ={}&
  \fuse[\Big]{b_2b_1}{ 
    D\bigl( \phi(\same{1}) \bigr)^{b_1}_{a_1c} 
    \phi(\same{1})^{b_2}_{a_2} 
  }{a_1a_2c} + \fuse[\Big]{b_2b_1}{ 
    \phi(\same{1})^{b_1}_{a_1} 
    D\bigl( \phi(\same{1}) \bigr)^{b_2}_{a_2c} 
  }{a_1a_2c}
  = \displaybreak[0]\\ ={}&
  \phi(\same{1}) \otimes D\bigl( \phi(\same{1}) \bigr) \circ
    \phi(\cross{1}{1} \star \same{1})
  + \phi(\cross{1}{1}) \circ 
    \phi(\same{1}) \otimes D\bigl( \phi(\same{1}) \bigr)
  = \\ ={}&
  \phi(\same{1}) \otimes 0 \circ \phi(\cross{1}{1} \star \same{1})
  + \phi(\cross{1}{1}) \circ \phi(\same{1}) \otimes 0
  = \\ ={}&
  0 \circ \phi(\cross{1}{1} \star \same{1})
  + \phi(\cross{1}{1}) \circ 0
  =
  0 + 0
  =
  0
  \text{.}
\end{align*}
As all permutations can be built up from $\same{1}$ and 
$\cross{1}{1}$, it follows that all permutations have $0$ as 
derivative.

An even nicer way of denoting derivatives, particularly in naked 
abstract index expressions, is to treat the derivative itself as 
a ``tensorial operator'' which takes one abstract index. Then the 
equivalents of \eqref{Eq:Produktregeln} becomes
\begin{subequations} \label{Eq2:Produktregeln}
  \begin{align}
    D_c( 
      \beta_{\vek{a}_1}^{\vek{b}_1} \gamma_{\vek{a}_2}^{\vek{b}_2} 
    )
    ={}&
    D_c(\beta_{\vek{a}_1}^{\vek{b}_1}) \gamma_{\vek{a}_2}^{\vek{b}_2} 
    +
    \beta_{\vek{a}_1}^{\vek{b}_1} D_c(\gamma_{\vek{a}_2}^{\vek{b}_2}) 
    \text{,}
    \\
    \nabla^c (
      \beta_{\vek{a}_1}^{\vek{b}_1} \gamma_{\vek{a}_2}^{\vek{b}_2} 
    )
    ={}&
    \nabla^c(\beta_{\vek{a}_1}^{\vek{b}_1}) 
    \gamma_{\vek{a}_2}^{\vek{b}_2} 
    +
    \beta_{\vek{a}_1}^{\vek{b}_1} 
    \nabla^c(\gamma_{\vek{a}_2}^{\vek{b}_2}) 
    \text{.}
  \end{align}
\end{subequations}
Note, however, that such ``abstract index notation with operators'' 
is not covered by Definition~\ref{D:fuse}. The natural extension that 
could be made to cover these would be that a subexpression
\[
  P_\vek{c}^\vek{d}\Bigl(
    \prod\nolimits_{i=1}^n (\gamma_i)_{\vek{a}_i}^{\vek{b}_i}
  \Bigr)
\]
where $P$ is some map of degree $\bigl( \Norm{\vek{d}}, 
\Norm{\vek{c}} \bigr)$ is to be interpreted as
\[
  P\left(
    \fuse{ \vek{a}_0 }{
    \textstyle \prod_{i=1}^n (\gamma_i)_{\vek{a}_i}^{\vek{b}_i}
    }{ \vek{b}_0 }
  \right)_{\vek{b}_0\vek{c}}^{\vek{a}_0\vek{d}}
\]
where \(\vek{a}_0 := \bigl( \prod_{i=1}^n \vek{b}_i \bigr) \big/ 
\bigl( \prod_{i=1}^n \vek{a}_i \bigr)\) and similarly \(\vek{b}_0 := 
\bigl( \prod_{i=1}^n \vek{a}_i \bigr) \big/ 
\bigl( \prod_{i=1}^n \vek{b}_i \bigr)\). 

Where there is a standard coordinate system, $\partial^c$ and 
$\partial_c$ are often used for the gradient and outer derivative 
respectively with respect to these coordinates. Leibniz-style partial 
derivative operators like $\frac{\partial}{\partial x_i}$ are also 
common, but note that sub- and superscripts by convention switch 
position in the denumerator: \(\frac{\partial}{\partial x_i} = 
\partial^i\) has degree $(1,0)$ and \(\frac{\partial}{\partial x^i} 
= \partial_i\) has degree $(0,1)$! This is apparently so that one can 
have \(\frac{\partial x^a}{\partial x^b} = \delta^a_b = 
\frac{\partial x_b}{\partial x_a}\), which is desirable in the 
coordinate interpretation of the notation. (Due to general rules 
about raising and lowering indices, $x^a$ and $x_a$ can be quite 
different numerically, even if they sort-of encode the same 
information, so the partial derivative of $x_a$ with respect to $x^b$ 
need not be $1$ for \(a=b\) and $0$ for \(a \neq b\).)

Yet another commonly occurring notation for derivation is that a 
punctuation mark in an abstract index sub- or superscript, 
like for example in `$g_{ab,c}$', signals that subsequent indices are 
ones that were added to the root symbol by some derivative. Different 
punctuation marks correspond to different derivatives, e.g.~a comma 
may signal the standard derivative, whereas a semicolon may signal 
some alternative 
derivative. In Riemannian geometry, it is common that the comma means 
``partial derivative'' (i.e., partial derivative with respect to the 
local coordinate system obtained from the current chart), whereas 
semicolon means ``covariant derivative'' (i.e., the globally defined 
derivative that sees the inner product as being constant).

\begin{lemma} \label{L:AIN-multilinear}
  Let $\mc{R}$ be an associateive and commutative unital ring. Let 
  $\mc{P}$ be an $\mc{R}$-linear \PROP. Let
  $$
    \fuse[\Big]{ \vek{a}_0 }{ 
      \textstyle\prod_{i=1}^n (\gamma_i)^{\vek{b}_i}_{\vek{a}_i}
    }{ \vek{b}_0 }
  $$
  be a well-formed abstract index expression. Then
  \begin{multline*}
    (\gamma_1,\dotsc,\gamma_n) \mapsto 
    \fuse[\bigg]{ \vek{a}_0 }{ 
      \prod_{i=1}^n (\gamma_i)^{\vek{b}_i}_{\vek{a}_i}
    }{ \vek{b}_0 }
    : \\ :
    \mc{P}\bigl( \Norm{\vek{b}_1}, \Norm{\vek{a}_1} \bigr) \times
    \dotsb \times
    \mc{P}\bigl( \Norm{\vek{b}_n}, \Norm{\vek{a}_n} \bigr) \Fpil
    \mc{P}\bigl( \Norm{\vek{a}_0}, \Norm{\vek{b}_0} \bigr) 
  \end{multline*}
  is an $\mc{R}$-multilinear map.
\end{lemma}
\begin{proof}
  In an $\mc{R}$-linear \PROP, the $\circ$ and $\otimes$ operations 
  are $\mc{R}$-bilinear. Hence in the case \(n=1\), the 
  $\mc{R}$-linearity of the abstract index expression follows 
  immediately from the definition (item~\ref{Item2:D:fuse} of 
  Definition~\ref{D:fuse}). For \(n>1\), item~\ref{Item3:D:fuse} of 
  Definition~\ref{D:fuse} defines the value of the abstract index 
  expression as a composition of $n$ order~$1$ abstract index 
  expressions, each of which by the above depends $\mc{R}$-linearly 
  on one factor. Hence the map as a whole is $\mc{R}$-multilinear, as 
  claimed.
\end{proof}

%

\section{Network notation}
\label{Sec:Natverk}

The abstract index notation improves upon the categorical notation in 
that it imposes less irrelevant structure upon the expression, but it 
is still obscure in that the very relevant structure of how separate 
factors fit together requires an effort to discern. The network 
notation is more spacious, but shows clearly the relations between 
factors and also makes it easy to spot larger subexpressions within a 
given expression. The downside of it is that it is graphical, 
spreading out in the plane rather than basically sticking to a 
baseline as text and mathematical formulae are mostly supposed to do. 
That \PROPs\ should be allowed such notational liberties when most of 
mathematics is perfectly well served by horizontal notation with only 
the occasional index, exponent, numerator, or denominator escaping 
from the baseline is perhaps not obvious, but a partial explanation 
is that \PROPs\ have both composition and tensor product; it is 
natural to orient these along separate axes. Nor is the issue without 
precedence; just like the notational inventions of abstract index 
notation were present already in the more specialised Einstein 
notation, so are the basic elements of the network notation present 
in e.g.~Penrose graphical notation~\cite{Penrose} and 
Majid's `shorthand diagrams'~\cite{Majid}. 
\c{S}tef\u{a}nescu~\cite[p.~8]{Stefanescu} attributes the term 
`network algebra' to Jan Bergstra~(1994).

\subsection{Formal definition and basic properties}

Concretely, a network notation expression looks a lot like an 
electrical curcuit diagram, with ``wires'' connecting various 
``gates'',\footnote{
  Or instead of `gates' more commonly `components', but that word 
  would risk confusion with the `\PROP\ component' concept.
} 
and additional wires going to external connectors. 
Figure~\ref{Fig:Network1} shows some styles in which the network 
expression equivalent of \(\fuse{ab}{ \mOp^a_{cd} \mOp^b_{ef} 
\Delta^{ce}_g \Delta^{df}_h }{gh}\) might be drawn, but for those who 
have not seen it already, it's probably just as informative to get on 
with the formalia.

\begin{definition} \label{Def:Network}
  A \DefOrd{network} of an $\N^2$-graded set $\mc{P}$ is a tuple
  \begin{equation}
    G = (V,E,h,g,t,s,D)
  \end{equation}
  where
  \begin{itemize}
    \item
      $V$ is a finite set, called the set of \emDefOrd[*{vertex}]{vertices}. 
      \(V \owns 0,1\), where $0$ is called the \emDefOrd{output vertex} 
      and $1$ is called the \emDefOrd{input vertex}. Elements of 
      \(V \setminus \{0,1\}\) are called 
      \emDefOrd[*{inner vertex}]{inner vertices}.
    \item
      $E$ is a finite set, called the set of \emDefOrd[*{edge}]{edges}.
    \item
      \(h\colon E \Fpil V \setminus \{1\}\) is called the 
      \emDefOrd{head} map. \(t\colon E \Fpil V \setminus \{0\}\) is 
      called the \emDefOrd{tail} map. Define the input 
      \index{d minus@$d^-(v)$}$d^-(v)$ and output 
      \index{d plus@$d^+(v)$}$d^+(v)$ valencies of a vertex $v$ by
      \begin{align*}
        d^-(v) ={}& 
          \norm[\Big]{ \setOf[\big]{ e \in E }{ h(e) = v } }
          \text{,}\\
        d^+(v) ={}& 
          \norm[\Big]{ \setOf[\big]{ e \in E }{ t(e) = v } }
          \text{.}
      \end{align*}
      An \DefOrd{output leg} of $G$ is an \(e \in E\) with 
      \(h(e)=0\), and an \DefOrd{input leg} of $G$ is an \(e \in E\) 
      with \(t(e)=1\). A \emDefOrd{leg} in general is an edge which 
      is an input or output leg. A \emDefOrd{stray edge} of a network 
      is an edge which is both an input and an output leg.
    \item
      \(g,s\colon E \Fpil \Zp\) are 
      called the \emDefOrd{head index} and \emDefOrd{tail index} 
      maps.
    \item
      \(D\colon V \setminus \{0,1\} \Fpil \mc{P}\) is called the 
      \emDefOrd{annotation}.
  \end{itemize}
  and these satisfy
  \begin{enumerate}
    \item \label{A1:Network}
      The tuple $(V,E,h,t)$ is a directed acyclic graph, i.e., there 
      is no sequence \(e_1,\dotsc,e_n \in E\) such that 
      \(h(e_i)=t(e_{i+1})\) for \(i=1,\dotsc,n-1\) and 
      \(h(e_n)=t(e_1)\).
    \item \label{A2:Network}
      Head and head index uniquely identifies an edge, as does the 
      combination of tail and tail index; if \(e_1,e_2 \in E\) satisfy 
      \(h(e_1)=h(e_2)\) and \(g(e_1) = g(e_2)\) then \(e_1=e_2\); 
      similarly if \(e_1,e_2 \in E\) satisfy \(t(e_1)=t(e_2)\) and 
      \(s(e_1) = s(e_2)\) then \(e_1=e_2\).
    \item \label{A3:Network}
      Head indices are assigned from $1$ and up; if \(e_1 \in E\) is 
      such that \(g(e_1)>1\) then there exists some \(e_2 \in E\) such 
      that \(h(e_2)=h(e_1)\) and \(g(e_2) = g(e_1)-1\). Similarly 
      tail indices are assigned from $1$ and up; if \(e_1 \in E\) is 
      such that \(s(e_1)>1\) then there exists some \(e_2 \in E\) such 
      that \(t(e_2)=t(e_1)\) and \(s(e_2) = s(e_1)-1\).
    \item \label{A4:Network}
      Arities and coarities of the annotations agree with the 
      in-valencies and out-valencies respectively of the inner vertices; 
      \(d^-(v) = \alpha\bigl( D(v) \bigr)\) and 
      \(d^+(v) = \omega\bigl( D(v) \bigr)\) for every \(v \in V 
      \setminus \{0,1\}\).
  \end{enumerate}
  The arity $\alpha(G)$ is defined to be the valency $d^+(1)$ of the 
  input vertex $1$ and coarity $\omega(G)$ is defined to be the 
  valency \(d^-(0)\) of the output vertex $0$.
  Denote by \index{Nw@$\Nw$}$\Nw(\mc{P})$ the $\N^2$-graded set of all 
  networks of $\mc{P}$ which have \(V,E \subset \N\).
\end{definition}

\begin{figure}[tp]
  \small
  \begin{tabular*}{\linewidth}{@{\extracolsep{\fill}}ccc}
    \begin{mpgraphics*}{3}
      beginfig(3);
        path P[];
        interim IC_legs_bb:=0.5;
        IC_equations1(2,1)();
        IC_equations2(2,1)(); y1=y2;
        IC_equations4(1,2)(); x4=x1;
        IC_equations5(1,2)(); x5=x2; y4=y5;
        IC_equations0(2,0)();
        IC_equations6(0,2)(); x0 = x6 = 0.5[x1,x2];
        y0-y1 = y1-y4 = y4-y6;
        
        x4ll = y6ul = 1in; 
        y1lr = y4ur + 1.3cm;
        x2ll-x1lr = 10pt;
        
        picture one,two;
        one := btex $\scriptstyle 1$ etex;
        two := btex $\scriptstyle 2$ etex;
        draw_IC1(one, "", btex $4\colon\Delta$ etex, one, two);
        draw_IC2(one, "", btex $5\colon\Delta$ etex, one, two);
        draw_IC4(one, two, btex $2\colon\mOp$ etex, one, "");
        draw_IC5(one, two, btex $3\colon\mOp$ etex, one, "");
        label.top(one,z0b1);
        label.top(two,z0b2);
        label.bot(one,z6t1);
        label.bot(two,z6t2);

        P1 := z0b1{down} .. {down}z1t1;
        drawarrow P1; label.ulft(btex $g$ etex, point 0.5 of P1);
        P2 := z0b2{down} .. {down}z2t1;
        drawarrow P2; label.urt (btex $h$ etex, point 0.5 of P2);
        drawarrow z1b1 -- z4t1 ;
        label.lft               (btex $c$ etex, 0.5[z1b1,z4t1]);
        P3 := z1b2{down} .. {down}z5t1;
        drawarrow P3; label.urt (btex $e$ etex, point 0.3 of P3);
        P4 := z2b1{down} .. {down}z4t2;
        drawarrow P4; label.ulft(btex $d$ etex, point 0.3 of P4);
        drawarrow z2b2 -- z5t2 ;
        label.rt                (btex $f$ etex, 0.5[z2b2,z5t2]);
        P5 := z4b1{down} .. {down}z6t1;
        drawarrow P5; label.llft(btex $a$ etex, point 0.5 of P5);
        P6 := z5b1{down} .. {down}z6t2;
        drawarrow P6; label.lrt (btex $b$ etex, point 0.5 of P6);
      endfig;
    \end{mpgraphics*}
    &
    \begin{mpgraphics*}{2}
      beginfig(2);
        interim IC_legs_bb:=0.5;
        IC_equations1(2,1)();
        IC_equations2(2,1)(); y1=y2;
        IC_equations4(1,2)(); x4=x1;
        IC_equations5(1,2)(); x5=x2; y4=y5;
        IC_equations0(2,0)();
        IC_equations6(0,2)(); x0 = x6 = 0.5[x1,x2];
        y0-y1 = y1-y4 = y4-y6;
        
        x4ll = y6ul = 1in; 
        y1lr = y4ur + 1.3cm;
        x2ll-x1lr = 10pt;
        
        draw_IC1("", "", btex $\Delta$ etex, "", "");
        draw_IC2("", "", btex $\Delta$ etex, "", "");
        draw_IC4("", "", btex $\mOp$ etex, "", "");
        draw_IC5("", "", btex $\mOp$ etex, "", "");
        draw z0b1{down} .. {down}z1t1;
        draw z0b2{down} .. {down}z2t1;
        draw z1b1 -- z4t1 ;
        draw z1b2{down} .. {down}z5t1;
        draw z2b1{down} .. {down}z4t2;
        draw z2b2 -- z5t2 ;
        draw z4b1{down} .. {down}z6t1;
        draw z5b1{down} .. {down}z6t2;
      endfig;
    \end{mpgraphics*}
    &
    \begin{mpgraphics*}{1}
      beginfig(1);
        u:=6pt;
        x1=x3; x5=x7; x6=x8; x2=x4; 0.5[x1,x2] = 0.5[x7,x8];
        x2-x1 = 30pt; x6-x5 = 8pt; 
        y7=y8; y1=y2; y3=y4; y5=y6;
        y7-y1 = y1-y3 = y3-y5 = x2-x1; 
        x1=y5=1in;
        
        forsuffixes $=1,2,3,4:
          draw fullcircle scaled 2u shifted z$;
        endfor
        draw z7{down} .. {down}(z1 + (0,u));
        draw z8{down} .. {down}(z2 + (0,u));
        draw (z1 + u*dir-135){dir-135} .. tension 1.5 .. {dir-45}(z3 + u*dir135);
        draw (z1 + u*dir-45) -- (z4 + u*dir135);
        draw (z2 + u*dir-135) -- (z3 + u*dir45);
        draw (z2 + u*dir-45){dir-45} .. tension 1.5 .. {dir-135}(z4 + u*dir45);
        draw (z3 + (0,-u)){down} .. {down}z5;
        draw (z4 + (0,-u)){down} .. {down}z6;
      endfig;
    \end{mpgraphics*}
    \\
    (a) Elaborate style&
    (b) Pragmatic style&
    (c) Bialgebra style
  \end{tabular*}
  
  \medskip
  \setlength{\parindent}{1em}
  
  The above are three different styles for the network expression 
  equivalent of \(\fuse{ab}{ \mOp^a_{cd} \mOp^b_{ef} \Delta^{ce}_g 
  \Delta^{df}_h }{gh}\). The \emDefOrd{elaborate style} explicitly 
  puts all information in the picture, but it is too cumbersome for 
  pretty much everything besides explaining the notation. The vertex 
  (`$2\colon$', `$3\colon$', etc.\@) and edge ($a$, $b$, $c$, etc.\@) 
  labels are typically omitted when one depicts the isomorphism class 
  of a network, but may help to show the correspondence with the 
  abstract index expression.
  
  The \emDefOrd{pragmatic style} drops the edge orientations 
  (arrowheads) and indices, since these can anyway be determined from 
  the picture if one sticks to certain conventions for how it is 
  drawn. First, edges are always oriented downwards, with heads 
  against the top side of a vertex and tails against the bottom side 
  of a vertex. Second, at each vertex the head\slash tail index 
  increases from left to right, so the leftmost input and output are 
  always those which have index~$1$, and the rightmost are always 
  those whose index equals the (co)arity of the vertex annotation. 
  Thus only the annotation remains in the vertex.
  
  The \emDefOrd{bialgebra style} is, like the Penrose graphical 
  notation, an example of a further streamlining of the notation to 
  suite a particular problem domain by eliding also the vertex 
  annotation, while leaving visual clues to this in the shape 
  of the vertex. In this style, $\Delta$, $\mOp$, $\ve$, and $\eta$ 
  vertices are all drawn as circles (which are quicker to draw by 
  hand than squares), and one can tell from the number of inputs and 
  outputs which of the four types is at hand; inputs are on the top 
  half of the circle, outputs are on the bottom half.
  
  \caption{Three styles for network drawing}
  \label{Fig:Network1}
\end{figure}

In short, networks are graphs with some extra structure,\footnote{
  And in case anyone wonders, this is \emph{not} the same extra 
  structure as one adds when one studies flows in graphs, even 
  though those graphs are called `networks' too.
} so it is only natural that one might prefer to draw them. The thing 
about them that is most uncommon for a graph is probably the head and 
tail index maps\Ldash that it matters \emph{which} point on a vertex 
that an edge is attached to\Dash but this is essential for their role 
as alternative presentations of the information in an abstract index 
expression.

One thing that one must watch out for when reading expressions in 
network notation is that different authors follow quite different 
conventions regarding what sides represents `in' and `out'. 
\c{S}tef\u{a}nescu~\cite{Stefanescu}, Majid~\cite{Majid}, 
Lafont~\cite{LafontBoolean}, Baez--Stay~\cite{BaezStay}, and I appear to 
prefer diagrams drawn with edges oriented from top to bottom, whereas 
for example Penrose~\cite{Penrose} and Markl--Voronov~\cite{MarklVoronov} 
rather prefer them oriented from bottom to top. Since 
many important theories have a large degree of symmetry between input 
and output, one can at times get rather far into an argument before 
realising that the notation is the opposite of what one expected it 
to be! There is of course also the possibility of making edges 
horizontal rather than vertical, but then there is equally much the 
ambiguity of whether to go 
left-to-right~\cite{qgate,GraphoidAutomata} (as in the common reading 
direction of Western languages) or right-to-left~\cite{Cvitanovic} 
(as in the common direction for composition of mathematical 
functions). There does not seem to be a convention that is singularly 
more natural than any other, so a certain amount of mental 
gymnastics will probably always be required when translating from one 
notation to another.

\begin{definition}
  Let a network \(G = (V,E,h,g,t,s,D)\) be given. For every \(v \in 
  V\), define $\vek{e}_G^+(v)$ and $\vek{e}_G^-(v)$ by
  \begin{align*}
    \vek{e}_G^+(v) ={}& e_1 \dots e_{d^+(v)} &&
      \index{e G +@$\vek{e}_G^+(v)$}
      \text{where \(t(e_i) = v\) and \(s(e_i)=i\) for all 
        \(i=1,\dotsc,d^+(v)\),}\\
    \vek{e}_G^-(v) ={}& e_1 \dots e_{d^-(v)} &&
      \index{e G -@$\vek{e}_G^-(v)$}
      \text{where \(h(e_i) = v\) and \(g(e_i)=i\) for all 
        \(i=1,\dotsc,d^-(v)\).}
  \end{align*}
  If $G$ is a network of a \PROP\ $\mc{P}$ then the 
  \DefOrd[*{value of network}]{value} of $G$ is
  \begin{equation} \label{Eq1:Network-value}
    \index{eval@$\eval(G)$}\eval(G) :=
    \fuse[\bigg]{ \vek{e}_G^-(0) }{
      \prod_{v \in V \setminus\{0,1\}} 
        D(v)^{\vek{e}_G^+(v)}_{\vek{e}_G^-(v)}
    }{ \vek{e}_G^+(1) }
    \text{.}
  \end{equation}
  If $\mc{P}$ is a \PROP, $G$ is a network of an $\N^2$-graded set 
  $\Omega$, and \(f\colon\Omega \Fpil \mc{P}\) is an $\N^2$-graded 
  set morphism then similarly define the \DefOrd[*{value}]{$f$-value} 
  of $G$ as
  \begin{equation} \label{Eq2:Network-value}
    \index{eval f@$\eval_f(G)$}\eval_f(G) :=
    \fuse[\bigg]{ \vek{e}_G^-(0) }{
      \prod_{v \in V \setminus\{0,1\}} 
        f\bigl( D(v) \bigr)^{\vek{e}_G^+(v)}_{\vek{e}_G^-(v)}
    }{ \vek{e}_G^+(1) }
    \text{.}
  \end{equation}
\end{definition}

The next step is of course to prove that $\eval$ is well-defined, but 
there is another concept which should be defined first.

\begin{definition}
  Two networks
  \[
    G = (V,E,h,g,t,s,D) 
    \qquad\text{and}\qquad
    H = (V',E',h',g',t',s',D')
  \]
  are said to be \DefOrd{isomorphic}, 
  symbolically \index{\simeq@$\simeq$}\(G \simeq H\), if there exists 
  bijections \(\chi\colon V \Fpil V'\) and \(\psi\colon E \Fpil E'\) 
  such that
  \begin{align*}
    h' \circ \psi ={}& \chi \circ h \text{,}&
    g' \circ \psi ={}& g \text{,}&
    \chi(0) ={}& 0\text{,} \\
    t' \circ \psi ={}& \chi \circ t \text{,}&
    s' \circ \psi ={}& s \text{,}&
    \chi(1) ={}& 1\text{,}\\
    &&
    D' \circ \chi ={}& D \text{.}
  \end{align*}
  Denote by \index{Nw tilde@$\Nwt$}$\Nwt(\Omega)$ the set of 
  isomorphism classes of $\Nw(\Omega)$.
\end{definition}

It may be observed that the pragmatic and bialgebra styles of 
Figure~\ref{Fig:Network1} depicts network isomorphism classes rather 
than specific networks in those classes. This is completely analogous 
to graph drawing, where one frequently depicts graphs without 
explicitly labelling vertices or edges.

\begin{theorem} \label{S:network-eval}
  Let $\mc{P}$ be a \PROP\ and $G$ a network of $\mc{P}$. Then 
  $\eval(G)$ is well-defined. If $H$ is a network of $\Omega$ and 
  \(f\colon \Omega \Fpil \mc{P}\) is an $\N^2$-graded set morphism, 
  then $\eval_f(H)$ is well-defined as well. Moreover any network 
  \(G' \simeq G\) has \(\eval(G') = \eval(G)\) and any network \(H' 
  \simeq H\) has \(\eval_f(H') = \eval_f(H)\). Hence $\eval$ is a map 
  \(\Nwt(\mc{P}) \Fpil \mc{P}\) and $\eval_f$ is a map \(\Nwt(\Omega) 
  \Fpil \mc{P}\).
\end{theorem}
\begin{proof}
  First observe that the $\vek{e}^+_G(v)$ and $\vek{e}^-_G(v)$ 
  notations are well-defined for all networks; axiom~\ref{A3:Network} 
  implies that the necessary $e_i$ exist and axiom~\ref{A2:Network} 
  implies that they are unique. Axiom~\ref{A4:Network} then ensures 
  that the arity condition of the abstract index expressions 
  \eqref{Eq1:Network-value} and \eqref{Eq2:Network-value} are 
  fulfilled, and the matching condition holds because $h$, $t$, $g$, 
  and $s$ are all maps. Finally, axiom~\ref{A1:Network} ensures the 
  acyclicity condition is fulfilled. A further point of ambiguity 
  could be that the order of factors in \eqref{Eq1:Network-value} and 
  \eqref{Eq2:Network-value} is unspecified, but 
  Theorem~\ref{S:fuse-factor-reorder} has already established that 
  this is irrelevant for the value of an abstract index expression.
  
  In order to see that $\eval$ and $\eval_f$ does not distinguish 
  between isomorphic networks, one may first observe that the vertex 
  labels do not occur directly in the abstract index expressions 
  defining these; for inner vertices only the annotations matter, and 
  these are preserved by isomorphisms, as are the vertex labels $0$ 
  and~$1$. Second, changing edge labels has by 
  Theorem~\ref{S:fuse-relabel} no effect on the value of an abstract 
  index expression. This covers everything a network isomorphism can 
  change.
\end{proof}

This correspondence between networks and abstract index expressions 
can also be carried in the opposite direction, to establish a 
text-like notation for networks. This is primarily useful when 
reasoning about families of networks on a particular form, as will be 
done in some proofs below.

\begin{definition}
  Let $\mc{P}$ be an $\N^2$-graded set. Let 
  $\vek{a}_0,\dotsc,\vek{a}_n$ and $\vek{b}_0,\dotsc,\vek{b}_n$ be 
  lists such that \(E := \norm[\Big]{\prod_{i=0}^n \vek{a}_i}\) is a 
  set and \(\norm[\Big]{\prod_{i=0}^n \vek{b}_i} = E\). Let 
  \(\gamma_i \in \mc{P}\bigl( \Norm{\vek{b}_i}, \Norm{\vek{a}_i} 
  \bigr)\) for \(i=1,\dotsc,n\) and let $v_1,\dotsc,v_n$ be distinct 
  labels such that \(0,1 \notin \{v_1,\dotsc,v_n\}\). If there exists 
  a partial order $P$ on $E$ such that \(x<y \pin{P}\) for every \(x 
  \in \norm{\vek{b}_i}\), \(y \in \norm{\vek{a}_i}\), and 
  \(i=1,\dotsc,n\), then what is denoted by
  \begin{equation}
    \Nwfuse[\Big]{ \vek{a}_0 }{
      \textstyle
      \prod_{i=1}^n (v_i\colon \gamma_i)^{\vek{b}_i}_{\vek{a}_i}
    }{ \vek{b}_0 }
  \end{equation}
  is the network \(G = \bigl( \{0,1,v_1,\dotsc,v_n\}, E, h,g,t,s,D 
  \bigr)\) of $\mc{P}$ such that 
  \begin{align*}
    \vek{e}_G^-(0) ={}& \vek{a}_0 \text{,}&
    \vek{e}_G^+(1) ={}& \vek{b}_0 \text{,}\\
    \vek{e}_G^-(v_i) ={}& \vek{a}_i \text{,}&
    \vek{e}_G^+(v_i) ={}& \vek{b}_i 
      && \text{for all \(i \in [n]\),}\\
    && D(v_i) ={}& \gamma_i 
      && \text{for all \(i \in [n]\).}
  \end{align*}
  In the special case that \(v_i = i+1\) for all \(i \in [n]\), this 
  notation may be simplified to
  \begin{equation}
    \Nwfuse[\Big]{ \vek{a}_0 }{
      \textstyle
      \prod_{i=1}^n (\gamma_i)^{\vek{b}_i}_{\vek{a}_i}
    }{ \vek{b}_0 }
  \end{equation}
  (that is: when no explicit vertex labelling is given, then assign 
  labels left to right in the product).
\end{definition}

Clearly, the equations for $\vek{e}_G^-$ define $h$ and $g$ on all of 
$E$, whereas the equations for $\vek{e}_G^+$ similarly define $t$ and 
$s$; these are well-defined since each edge label occurs once in 
$\prod_{i=0}^n \vek{a}_i$ and once in $\prod_{i=0}^n \vek{b}_i$. 
Network axiom~\ref{A1:Network} follows from the existence of the 
partial order~$P$, axioms~\ref{A2:Network} and~\ref{A3:Network} 
follow from the definitions of $\vek{e}_G^-$ and $\vek{e}_G^+$, and 
the $\gamma_i$ were explicitly chosen so that axiom~\ref{A4:Network} 
would be fulfilled.

\begin{corollary} \label{Kor:EnklaNatverk}
  Let $\mc{P}$ be a \PROP. Let $\vek{m}$ and $\vek{n}$ be lists of 
  labels such that $\norm{\vek{m}}$ and $\norm{\vek{n}}$ are disjoint 
  sets. Let \(\gamma \in \mc{P}\bigl( \Norm{\vek{m}}, \Norm{\vek{n}} 
  \bigr)\) and \(v \notin \{0,1\}\) be arbitrary. Then
  \begin{align}
    \phi(\same{\Norm{\vek{n}}}) ={}& \eval\Bigl( 
      \Nwfuse{ \vek{n} }{1}{ \vek{n} }
    \Bigr) \text{,}\\
    \gamma ={}& \eval\Bigl( 
      \Nwfuse{ \vek{m} }{ (v\colon\gamma)^{\vek{m}}_{\vek{n}} }
        { \vek{n} } 
    \Bigr) \text{.}
  \end{align}
\end{corollary}
\begin{proof}
  This is just as much a corollary of 
  Theorem~\ref{S:network-eval}\Ldash on the point that $\eval$ is 
  well-defined\Dash as of Theorem~\ref{S:fuse-just}, on the point 
  that it gives the values for exactly those abstract index 
  expressions as which the right hand sides are defined.
\end{proof}

More examples of this abstract index notation for networks and their 
corresponding diagrams can be found in Figure~\ref{Fig:PROP-axiom-Nw}. 
In those cases the abstract index notation is convenient because it 
formally allows one to leave the number of edges connecting two 
vertices unspecified; with an explicit diagram, it is more a matter 
of drawing some illustrative number of edges and using the head and 
tail index labels to suggest that they may vary.

Next up, the network counterparts of  
Theorems~\ref{S:AIN,snitt} and~\ref{S:AIN,klyva}. This requires two 
new concepts.

\begin{definition}
  A \DefOrd{cut} in a network \(G = (V,E,h,g,t,s,D)\) is a bipartition 
  \(W_0 \cup W_1\) of $V \setminus \{0,1\}$ such that no \(e \in E\) 
  satisfies \(h(e) \in W_1\) and \(t(e) \in W_0\). The cut is said to 
  be \DefOrd[*{cut!non-trivial}]{non-trivial} if $W_0$ and $W_1$ are 
  both nonempty.
  An edge \(e \in E\) is said to be a \emDefOrd[*{cut!edge}]{cut edge} 
  if \(h(e) \in W_0 \cup \{0\}\) and \(t(e) \in W_1 \cup \{1\}\). 
  Let $E_c$ be the set of all cut edges. The cut is said to be 
  \emDefOrd[*{cut!above}]{above} \(X \subseteq W_0\) if \(X = 
  \setmap{h}(E_c)\), and conversely \emDefOrd[*{cut!below}]{below} 
  \(X \subseteq W_1\) if \(X = \setmap{t}(E_c)\).
  
  A cut is said to be \DefOrd[*{cut!ordered}]{ordered} if it is 
  given as $(W_0,W_1,p)$ where \(p\colon E_c \Fpil \bigl[ \card{E_c} 
  \bigr]\) is a bijection. 
  The \emDefOrd[*{network decomposition}]{decomposition} 
  of $G$ induced by the ordered cut $(W_0,W_1,p)$ is $(G',G'')$, 
  where 
  \begin{align*}
    G' ={}& (V',E',h',g',t',s',D') \quad\text{and} &
    G'' ={}& (V'',E'',h'',g'',t'',s'',D'')
    \\ \intertext{are defined by}
    V' ={}& W_0 \cup \{0,1\} &   V'' ={}& W_1 \cup \{0,1\}\\
    E' ={}& \setOf[\big]{e \in E}{h(e) \in W_0 \cup \{0\}} &
    E'' ={}& \setOf[\big]{e \in E}{t(e) \in W_1 \cup \{1\}} 
      \displaybreak[0]\\
    h' ={}& \restr{h}{E'} & t'' ={}& \restr{t}{E''}\\
    t'(e) ={}& \begin{cases}
      1& \text{if \(e \in E_c\),}\\
      t(e)& \text{otherwise,}
    \end{cases} &
    h''(e) ={}& \begin{cases}
      0& \text{if \(e \in E_c\),}\\
      h(e)& \text{otherwise,}
    \end{cases} \displaybreak[0]\\
    g' ={}& \restr{g}{E'} & s'' ={}& \restr{s}{E''} 
      \displaybreak[0]\\
    s'(e) ={}& \begin{cases}
      p(e)& \text{if \(e \in E_c\),}\\
      s(e)& \text{otherwise,}
    \end{cases} &
    g''(e) ={}& \begin{cases}
      p(e)& \text{if \(e \in E_c\),}\\
      g(e)& \text{otherwise,}
    \end{cases} \\
    D' ={}& \restr{D}{W_0} & 
    D'' ={}& \restr{D}{W_1} \text{.}
  \end{align*}
  An ordered cut $(W_0,W_1,p)$ is said to be 
  \DefOrd[*{cut!obvious}]{obvious} 
  if \(p = \restr{g}{E_c} = \restr{s}{E_c}\). It is said to be 
  \DefOrd[*{cut!obvious above}]{obvious above} $(v_1,\dotsc,v_k)$ if 
  (i)~it is above $\{v_1,\dotsc,v_k\}$, (ii)~\(h(e_1)=v_i\), 
  \(h(e_2)=v_j\), and \(i<j\) for some \(e_1,e_2 \in E_c\) implies 
  \(p(e_1)<p(e_2)\), and (iii)~\(h(e_1)=h(e_2)\) and \(g(e_1)<g(e_2)\) 
  for some \(e_1,e_2 \in E_c\) implies \(p(e_1)<p(e_2)\). Analogously, 
  $(W_0,W_1,p)$ is \DefOrd[*{cut!obvious below}]{obvious below} 
  $(v_1,\dotsc,v_k)$ if (i)~it is below $\{v_1,\dotsc,v_k\}$, 
  (ii)~\(t(e_1)=v_i\), \(t(e_2)=v_j\), and \(i<j\) for some \(e_1,e_2 
  \in E_c\) implies \(p(e_1)<p(e_2)\), and 
  (iii)~\(t(e_1)=t(e_2)\) and \(s(e_1)<s(e_2)\) for some \(e_1,e_2 
  \in E_c\) implies \(p(e_1)<p(e_2)\).
\end{definition}

\begin{lemma} \label{L:evalCut}
  Let $\mc{P}$ be a \PROP\ and $G$ be a network of $\mc{P}$. If 
  $(G',G'')$ is a decomposition of $G$ induced by some ordered cut 
  $(W_0,W_1,p)$ then \(\eval(G) = \eval(G') \circ \eval(G'')\).
  
  More generally, if $\mc{P}$ is a \PROP, $\Omega$ is some 
  $\N^2$-graded set, \(f\colon \Omega \Fpil \mc{P}\) is an 
  $\N^2$-graded set morphism, $G$ a network of $\Omega$, and 
  $(G',G'')$ is a decomposition of $G$ induced by some ordered cut 
  $(W_0,W_1,p)$ then \(\eval_f(G) = \eval_f(G') \circ \eval_f(G'')\).
\end{lemma}
\begin{proof}
  Consider first the claim about $\eval_f$.
  Let $k$ be the number of cut edges, and let \(\vek{p} = 
  p^{-1}(1) \dotsb p^{-1}(k)\). Then \(\vek{e}_{G'}^-(0) = 
  \vek{e}_{G}^-(0)\), \(\vek{e}_{G'}^+(1) = \vek{p}\), 
  \(\vek{e}_{G''}^-(0) = \vek{p}\), and \(\vek{e}_{G''}^+(1) = 
  \vek{e}_G^+(1)\). Hence by Theorem~\ref{S:AIN,snitt},
  \begin{multline*}
    \eval_f(G') \circ \eval_f(G'') 
    = \\ =
    \fuse[\bigg]{ \vek{e}_G^-(0) }{
      \prod_{v \in W_0} 
        f\bigl( D(v) \bigr)^{\vek{e}_{G'}^+(v)}_{\vek{e}_{G'}^-(v)}
    }{ \vek{p} } \circ
    \fuse[\bigg]{ \vek{p} }{
      \prod_{v \in W_1} 
        f\bigl( D(v) \bigr)^{\vek{e}_{G''}^+(v)}_{\vek{e}_{G''}^-(v)}
    }{ \vek{e}_G^+(1) }
    = \\ =
    \fuse[\bigg]{ \vek{e}_G^-(0) }{
      \prod_{v \in W_0} 
        f\bigl( D(v) \bigr)^{\vek{e}_G^+(v)}_{\vek{e}_G^-(v)}
      \prod_{v \in W_1} 
        f\bigl( D(v) \bigr)^{\vek{e}_G^+(v)}_{\vek{e}_G^-(v)}
    }{ \vek{e}_G^+(1) }
    =
    \eval_f(G) \text{.}
  \end{multline*}
  The claim about $\eval$ is obtained by taking \(\Omega = \mc{P}\) 
  and $f$ being the identity map.
\end{proof}

A special, but frequently useful, kind of cut to make in a network is 
to cut below $1$ or above $0$, but in a non-obvious order to factor 
out a particular permutation. Conversely, it is natural to define 
actions of permutations on networks as permuting the legs.

\begin{definition}
  Let \(G = (V,E,h,g,t,s,D)\) be a network with \(m = \omega(G)\) and 
  \(n = \alpha(G)\). Then the left and right actions of \(\sigma \in 
  \Sigma_m\) and \(\tau \in \Sigma_n\) respectively on $G$ are 
  defined by
  \begin{align*}
    \sigma \cdot G :={}& (V,E,h,g',t,s,D) \text{,}\\
    G \cdot \tau :={}& (V,E,h,g,t,s',D) \text{,}
  \end{align*}
  where
  \begin{align*}
    g'(e) ={}& \begin{cases}
      \sigma\bigl( g(e) \bigr) & \text{if \(h(e)=0\),}\\
      g(e) & \text{if \(h(e) \neq 0\),}
    \end{cases}
    \\
    s'(e) ={}& \begin{cases}
      \tau^{-1}\bigl( s(e) \bigr) & \text{if \(t(e)=1\),}\\
      s(e) & \text{if \(t(e) \neq 1\)}
    \end{cases}
  \end{align*}
  for all \(e \in E\).
\end{definition}

It is easily checked that $\Nw(\Omega)(m,n)$ becomes a 
$(\Sigma_m,\Sigma_n)$-bimodule with these actions. In particular 
\((\sigma \cdot G) \cdot \tau = (V,E,h,g',t,s',D) = \sigma \cdot 
(G \cdot \tau)\), so parentheses can be dropped in such expressions.

\begin{lemma}
  Let $\mc{P}$ be a \PROP\ and $G$ be a network of $\mc{P}$. For any 
  permutations \(\sigma \in \Sigma_{\omega(G)}\) and \(\tau \in 
  \Sigma_{\alpha(G)}\) it holds that
  \begin{equation}
    \eval( \sigma \cdot G \cdot \tau ) =
    \phi(\sigma) \circ \eval(G) \circ \phi(\tau) \text{.}
  \end{equation}
  More generally, if $\mc{P}$ is a \PROP, $\Omega$ is some 
  $\N^2$-graded set, \(f\colon \Omega \Fpil \mc{P}\) is an 
  $\N^2$-graded set morphism, and $G$ a network of $\Omega$, then for 
  any permutations \(\sigma \in \Sigma_{\omega(G)}\) and \(\tau \in 
  \Sigma_{\alpha(G)}\) it holds that
  \begin{equation}
    \eval_f( \sigma \cdot G \cdot \tau ) =
    \phi(\sigma) \circ \eval_f(G) \circ \phi(\tau) \text{.}
  \end{equation}
\end{lemma}
\begin{proof}
  Let \((V,E,h,g,t,s,D) = G\) and \((V,E,h,g',t,s',D) = \sigma 
  \cdot G \cdot \tau\). Let \(W = V \setminus \{0,1\}\), \(E_0 = 
  \setinv{h}\bigl( \{0\} \bigr)\), and \(E_1 = \setinv{t}\bigl( \{1\} 
  \bigr)\). Then $(\varnothing,W,\restr{g}{E_0})$ is an ordered cut 
  in $\sigma \cdot G \cdot \tau$ that induces the decomposition 
  $(G_0,G_1)$, where \(G_1 = (V,E,h,g,t,s',D) = G \cdot \tau\) and 
  \(G_0 = \bigl( \{0,1\}, E_0, 0, g', 1, g, \varnothing \bigr)\) 
  (where the latter $0$ and $1$ are the constant maps \(E_0 \Fpil 
  \{0\}\) and \(E_0 \Fpil \{1\}\) respectively). 
  $(W,\varnothing,\restr{s}{E_1})$ is an ordered cut in $G_1$ and 
  similarly induces a decomposition $(G,G_{11})$ where \(G_{11} = 
  \bigl( \{0,1\}, E_1, 0, s, 1, s', \varnothing \bigr)\). Hence by 
  Lemma~\ref{L:evalCut},
  \begin{multline*}
    \eval_f( \sigma \cdot G \cdot \tau ) =
    \eval_f(G_0) \circ \eval_f(G) \circ \eval_f(G_{11}) 
    = \\ =
    \phi\bigl( g' \circ (\restr{g}{E_0})^{-1} \bigr) \circ 
      \eval_f(G) \circ 
      \phi\bigl( s \circ (\restr{s'}{E_1})^{-1} \bigr) =
    \phi(\sigma) \circ \eval_f(G) \circ \phi(\tau) \text{.}
  \end{multline*}
\end{proof}

\begin{definition}
  A \DefOrd{split} in a network \(G = (V,E,h,g,t,s,D)\) is a tuple 
  $(F_\mathrm{l},F_\mathrm{r},W_\mathrm{l},W_\mathrm{r})$, where 
  \(F_\mathrm{l} \cup F_\mathrm{r}\) is a bipartition of 
  $E$ and \(W_\mathrm{l} \cup W_\mathrm{r}\) is a bipartition 
  of $V \setminus \{0,1\}$, such that:
  \begin{enumerate}
    \item
      if \(e \in F_\mathrm{l}\) then \(h(e) \in W_\mathrm{l} \cup 
      \{0\}\) and \(t(e) \in W_\mathrm{l} \cup\{1\}\),
    \item
      if \(e \in F_\mathrm{r}\) then \(h(e) \in W_\mathrm{r} \cup 
      \{0\}\) and \(t(e) \in W_\mathrm{r} \cup\{1\}\),
    \item
      if \(e_\mathrm{l} \in F_\mathrm{l}\) and \(e_\mathrm{r} \in 
      F_\mathrm{r}\) are such that \(h(e_\mathrm{l}) = 
      h(e_\mathrm{r}) = 0\) then \(g(e_\mathrm{l}) < g(e_\mathrm{r})\), 
      and
    \item
      if \(e_\mathrm{l} \in F_\mathrm{l}\) and \(e_\mathrm{r} \in 
      F_\mathrm{r}\) are such that \(t(e_\mathrm{l}) = t(e_\mathrm{r}) 
      = 1\) then \(s(e_\mathrm{l}) < s(e_\mathrm{r})\).
  \end{enumerate}
  Parts in the bipartitions may be empty.
  
  Let \(k = \card[\big]{\setOf{e \in F_\mathrm{l}}{h(e)=0}}\) and 
  \(l = \card[\big]{\setOf{e \in F_\mathrm{l}}{t(e)=1}}\). The 
  \emDefOrd[*{network decomposition}]{decomposition} 
  of $G$ induced by the split $(F_\mathrm{l},F_\mathrm{r},W_\mathrm{l},
  W_\mathrm{r})$ is $(G',G'')$, 
  where \(G' = (V',F_\mathrm{l},h',g',t',s',D')\) and 
  \(G'' = (V'',F_\mathrm{r},h'',g'',t'',s'',D'')\) are defined by
  \begin{align*}
    V' ={}& W_0 \cup \{0,1\} &   V'' ={}& W_1 \cup \{0,1\}\\
    h' ={}& \restr{h}{F_\mathrm{l}} & h'' ={}& \restr{h}{F_\mathrm{r}}
      \displaybreak[0]\\
    t' ={}& \restr{t}{F_\mathrm{l}} & t'' ={}& \restr{t}{F_\mathrm{r}}
      \displaybreak[0]\\
    g' ={}& \restr{g}{F_\mathrm{l}} & 
      g''(e) ={}& \begin{cases}
        g(e)-k& \text{if \(h''(e)=0\),}\\
        g(e)& \text{otherwise,}
      \end{cases}
      \displaybreak[0]\\
    s' ={}& \restr{s}{F_\mathrm{l}} & 
      s''(e) ={}& \begin{cases}
        s(e)-l& \text{if \(t''(e)=1\),}\\
        s(e)& \text{otherwise,}
      \end{cases}\\
    D' ={}& \restr{D}{W_\mathrm{l}} & 
    D'' ={}& \restr{D}{W_\mathrm{r}} \text{.}
  \end{align*}
\end{definition}

\begin{lemma} \label{L:evalSplit}
  Let $\mc{P}$ be a \PROP\ and $G$ be a network of $\mc{P}$. If 
  $(G',G'')$ is a decomposition of $G$ induced by some split 
  $(F_\mathrm{l},F_\mathrm{r},W_\mathrm{l},W_\mathrm{r})$ then 
  \(\eval(G) = \eval(G') \otimes \eval(G'')\).
  
  More generally, let $\mc{P}$ be a \PROP, $\Omega$ be an 
  $\N^2$-graded set, $G$ be a network of $\Omega$, and \(f\colon 
  \Omega \Fpil \mc{P}\) be an $\N^2$-graded set morphism. If 
  $(G',G'')$ is a decomposition of $G$ induced by some split 
  $(F_\mathrm{l},F_\mathrm{r},W_\mathrm{l},W_\mathrm{r})$ then 
  \(\eval_f(G) = \eval_f(G') \otimes \eval_f(G'')\).
\end{lemma}
\begin{proof}
  For every \(v \in W_\mathrm{l}\), \(\vek{e}_{G'}^-(v) = 
  \vek{e}_G^-(v)\) and \(\vek{e}_{G'}^+(v) = \vek{e}_G^+(v)\); 
  similarly \(\vek{e}_{G''}^-(v) = \vek{e}_G^-(v)\) and 
  \(\vek{e}_{G''}^+(v) = \vek{e}_G^+(v)\) for all \(v \in W_\mathrm{r}\). 
  As for the vertices $0$ and $1$, \(\vek{e}_G^-(0) = 
  \vek{e}_{G'}^-(0) \vek{e}_{G''}^-(0)\) and \(\vek{e}_G^+(1) = 
  \vek{e}_{G'}^+(1) \vek{e}_{G''}^+(1)\). Hence by 
  Theorem~\ref{S:AIN,klyva},
  \begin{multline*}
    \eval_f(G') \otimes \eval_f(G'') 
    = \\ =
    \fuse[\bigg]{ \vek{e}_{G'}^-(0) }{
      \prod_{v \in W_\mathrm{l}} 
        f\bigl( D(v) \bigr)^{\vek{e}_{G'}^+(v)}_{\vek{e}_{G'}^-(v)}
    }{ \vek{e}_{G'}^+(1) } \otimes
    \fuse[\bigg]{ \vek{e}_{G''}^-(0) }{
      \prod_{v \in W_\mathrm{r}} 
        f\bigl( D(v) \bigr)^{\vek{e}_{G''}^+(v)}_{\vek{e}_{G''}^-(v)}
    }{ \vek{e}_{G''}^+(1) }
    = \\ =
    \fuse[\bigg]{ \vek{e}_{G'}^-(0)\vek{e}_{G''}^-(0) }{
      \prod_{v \in W_\mathrm{l}} 
        f\bigl( D(v) \bigr)^{\vek{e}_{G'}^+(v)}_{\vek{e}_{G'}^-(v)}
      \prod_{v \in W_\mathrm{r}} 
        f\bigl( D(v) \bigr)^{\vek{e}_{G''}^+(v)}_{\vek{e}_{G''}^-(v)}
    }{ \vek{e}_{G'}^+(1)\vek{e}_{G''}^+(1) }
    =
    \eval(G) \text{.}
  \end{multline*}
  Again the claim about $\eval$ is obtained by taking \(\Omega = 
  \mc{P}\) and $f$ being the identity map.
\end{proof}

\subsection{Networks as alternative foundation for \PROPs}

\begin{figure}
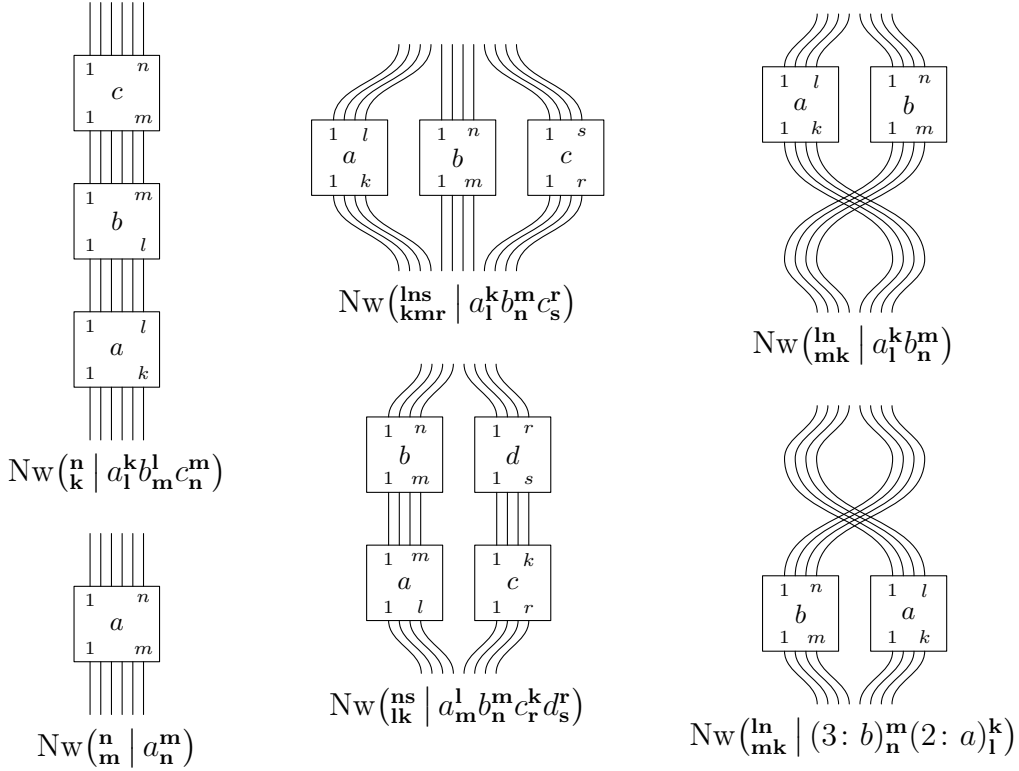

  \begin{tabular}{c}
    \begin{mpgraphics}{4}
      beginfig(4);
        wires:=6;
        IC_equations0(wires,0)();
        IC_equations1(wires,wires)();
        IC_equations2(wires,wires)();
        IC_equations3(wires,wires)();
        IC_equations4(0,wires)();
        x0 = x1 = x2 = x3 = x4;
        y0ll-y1ul = y1ll-y2ul = y2ll-y3ul = y3ll-y4ul = 0.7cm;
        y4t1 = x3ll = 1in;
        
        picture one, pk, pl, pm, pn;
        one := btex $\scriptstyle 1$ etex;
        pk  := btex $\scriptstyle k$ etex;
        pl  := btex $\scriptstyle l$ etex;
        pm  := btex $\scriptstyle m$ etex;
        pn  := btex $\scriptstyle n$ etex;
        draw_IC1( one, pn, btex $c$ etex, one, pm);
        draw_IC2( one, pm, btex $b$ etex, one, pl);
        draw_IC3( one, pl, btex $a$ etex, one, pk);
        
        for i := 1 upto 4:
           for j := 1 upto wires:
              draw z[i-1]b[j] -- z[i]t[j] ;
           endfor
        endfor
      endfig;
    \end{mpgraphics}
    \\
    $\Nwfuse{ \vek{k} }{ 
      a^{\vek{k}}_{\vek{l}} b^{\vek{l}}_{\vek{m}} 
      c^{\vek{m}}_{\vek{n}}
    }{ \vek{n} }$
    \\
    \\
    \begin{mpgraphics}{5}
      beginfig(5);
        wires:=6;
        IC_equations0(wires,0)();
        IC_equations1(wires,wires)();
        IC_equations2(0,wires)();
        x0 = x1 = x2;
        y0ll-y1ul = y1ll-y2ul = 0.7cm;
        y2t1 = x1ll = 1in;
        
        picture one, pm, pn;
        one := btex $\scriptstyle 1$ etex;
        pm  := btex $\scriptstyle m$ etex;
        pn  := btex $\scriptstyle n$ etex;
        draw_IC1( one, pn, btex $a$ etex, one, pm);
        
        for i := 1 upto 2:
           for j := 1 upto wires:
              draw z[i-1]b[j] -- z[i]t[j] ;
           endfor
        endfor
      endfig;
    \end{mpgraphics}
    \\
    $\Nwfuse{ \vek{m} }{ a^{\vek{m}}_{\vek{n}} }{ \vek{n} }$
  \end{tabular}
  \quad\hfill
  \begin{tabular}{c}
    \begin{mpgraphics}{6}
      beginfig(6);
        wires:=4;
        IC_equations0(3*wires,0)();
        IC_equations1(wires,wires)();
        IC_equations2(wires,wires)();
        IC_equations3(wires,wires)();
        IC_equations4(0,3*wires)();
        x0 = x2 = x4 = 0.5[x1,x3];  
        y1 = y2 = y3;
        x2ll-x1lr = 12pt;
        y0ll-y2ul = y2ll-y4ul = 1cm;
        y4t1 = x1ll = 1in;
        
        picture one, pk, pl, pm, pn, pr, ps;
        one := btex $\scriptstyle 1$ etex;
        pk  := btex $\scriptstyle k$ etex;
        pl  := btex $\scriptstyle l$ etex;
        pm  := btex $\scriptstyle m$ etex;
        pn  := btex $\scriptstyle n$ etex;
        pr  := btex $\scriptstyle r$ etex;
        ps  := btex $\scriptstyle s$ etex;
        draw_IC1( one, pl, btex $a$ etex, one, pk);
        draw_IC2( one, pn, btex $b$ etex, one, pm);
        draw_IC3( one, ps, btex $c$ etex, one, pr);
        
        for i := 0 upto 2:
           for j := 1 upto wires:
              draw z0b[i*wires+j]{down} .. {down}z[i+1]t[j] ;
              draw z[i+1]b[j]{down} .. {down}z4t[i*wires+j] ;
           endfor
        endfor
      endfig;
    \end{mpgraphics}
    \\
    $\Nwfuse{ \vek{kmr} }{ 
      a^{\vek{k}}_{\vek{l}} b^{\vek{m}}_{\vek{n}} 
      c^{\vek{r}}_{\vek{s}}
    }{ \vek{lns} }$
    \\
    \\
    \begin{mpgraphics}{7}
      beginfig(7);
        wires:=4;
        IC_equations1(2*wires,0)();
        IC_equations2(wires,wires)();
        IC_equations3(wires,wires)();
        IC_equations4(wires,wires)();
        IC_equations5(wires,wires)();
        IC_equations0(0,2*wires)();
        x0 = x1 = 0.5[x2,x4]; x3=x2; x5=x4;
        x4ll-x2lr = 12pt;
        y2=y4; y3=y5;
        y1b1-y3ur = y3lr-y2ur = y2lr-y0t1 = 0.7cm;
        y0t1 = x2ll = 1in;
        
        picture one, pk, pl, pm, pn, pr, ps;
        one := btex $\scriptstyle 1$ etex;
        pk  := btex $\scriptstyle k$ etex;
        pl  := btex $\scriptstyle l$ etex;
        pm  := btex $\scriptstyle m$ etex;
        pn  := btex $\scriptstyle n$ etex;
        pr  := btex $\scriptstyle r$ etex;
        ps  := btex $\scriptstyle s$ etex;
        draw_IC2( one, pm, btex $a$ etex, one, pl);
        draw_IC3( one, pn, btex $b$ etex, one, pm);
        draw_IC4( one, pk, btex $c$ etex, one, pr);
        draw_IC5( one, pr, btex $d$ etex, one, ps);
        
        for j := 1 upto wires:
           draw z1b[j]{down} .. {down}z3t[j] ;
           draw z1b[wires+j]{down} .. {down}z5t[j] ;
           draw z3b[j] -- z2t[j] ;
           draw z5b[j] -- z4t[j] ;
           draw z2b[j]{down} .. {down}z0t[j] ;
           draw z4b[j]{down} .. {down}z0t[wires+j] ;
        endfor
      endfig;
    \end{mpgraphics}
    \\
    $\Nwfuse{ \vek{lk} }{ 
      a^{\vek{l}}_{\vek{m}} b^{\vek{m}}_{\vek{n}} 
      c^{\vek{k}}_{\vek{r}} d^{\vek{r}}_{\vek{s}}
    }{ \vek{ns} }$
  \end{tabular}
  \quad\hfill
  \begin{tabular}{c}
    \begin{mpgraphics}{8}
      beginfig(8);
        wires:=4;
        IC_equations1(2*wires,0)();
        IC_equations3(wires,wires)();
        IC_equations2(wires,wires)();
        IC_equations0(0,2*wires)();
        for j := 1 upto wires:
           x2bb[j] = x2b[j];  x3bb[j] = x3b[j];
        endfor
        y2bb1 = y3bb1 for j:=2 upto wires: = y2bb[j] = y3bb[j] endfor;
        x0 = x1 = 0.5[x2,x3];  y2 = y3;
        x3ll-x2lr = 12pt;
        y1ll-y2ul = 0.45*(y2ll-y2bb1) = y2bb1-y0ul = 0.7cm;
        y0t1 = x2ll = 1in;
        
        picture one, pk, pl, pm, pn;
        one := btex $\scriptstyle 1$ etex;
        pk  := btex $\scriptstyle k$ etex;
        pl  := btex $\scriptstyle l$ etex;
        pm  := btex $\scriptstyle m$ etex;
        pn  := btex $\scriptstyle n$ etex;
        draw_IC2( one, pl, btex $a$ etex, one, pk);
        draw_IC3( one, pn, btex $b$ etex, one, pm);
        
        for j := 1 upto wires:
           draw z1b[j]{down} .. {down}z2t[j] ;
           draw z2b[j]{down} .. z3bb[j]{down} .. {down}z0t[wires+j] ;
           draw z1b[wires+j]{down} .. {down}z3t[j] ;
           draw z3b[j]{down} .. z2bb[j]{down} .. {down}z0t[j] ;
        endfor
      endfig;
    \end{mpgraphics}
    \\
    $\Nwfuse{ \vek{mk} }{ 
      a^{\vek{k}}_{\vek{l}} b^{\vek{m}}_{\vek{n}} 
    }{ \vek{ln} }$
    \\
    \\
    \begin{mpgraphics}{9}
      beginfig(9);
        wires:=4;
        IC_equations1(2*wires,0)();
        IC_equations3(wires,wires)();
        IC_equations2(wires,wires)();
        IC_equations0(0,2*wires)();
        for j := 1 upto wires:
           x2tt[j] = x2t[j];  x3tt[j] = x3t[j];
        endfor
        y2tt1 = y3tt1 for j:=2 upto wires: = y2tt[j] = y3tt[j] endfor;
        x0 = x1 = 0.5[x2,x3];  y2 = y3;
        x3ll-x2lr = 12pt;
        y1ll-y2tt1 = 0.45*(y2tt1-y2ul) = y2ll-y0ul = 0.7cm;
        y0t1 = x2ll = 1in;
        
        picture one, pk, pl, pm, pn;
        one := btex $\scriptstyle 1$ etex;
        pk  := btex $\scriptstyle k$ etex;
        pl  := btex $\scriptstyle l$ etex;
        pm  := btex $\scriptstyle m$ etex;
        pn  := btex $\scriptstyle n$ etex;
        draw_IC2( one, pn, btex $b$ etex, one, pm);
        draw_IC3( one, pl, btex $a$ etex, one, pk);
        
        for j := 1 upto wires:
           draw z1b[j]{down} .. z2tt[j]{down} .. {down}z3t[j] ;
           draw z3b[j]{down} .. {down}z0t[wires+j] ;
           draw z1b[wires+j]{down} .. z3tt[j]{down} .. {down}z2t[j] ;
           draw z2b[j]{down} .. {down}z0t[j] ;
        endfor
      endfig;
    \end{mpgraphics}
    \\
    $\Nwfuse{ \vek{mk} }{ 
      (3\colon b)^{\vek{m}}_{\vek{n}} (2\colon a)^{\vek{k}}_{\vek{l}} 
    }{ \vek{ln} }$
  \end{tabular}

  \caption{Networks for the proof of Theorem~\ref{S:PROP2-definition}}
  \label{Fig:PROP-axiom-Nw}
\end{figure}

\begin{theorem} \label{S:PROP2-definition}
  Let $\mc{P}$ be an $\N^2$-graded set equipped with
  \begin{itemize}
    \item
      an $\N^2$-graded set morphism \(L\colon \Nw(\mc{P}) \Fpil 
      \mc{P}\),
    \item
      a partial operation $\circ$ which for all \(l,m,n\in\N\) maps 
      all of \(\mc{P}(l,m) \times \mc{P}(m,n)\) into $\mc{P}(l,n)$, 
      and
    \item
      an operation $\otimes$ which for all \(k,l,m,n\in\N\) maps 
      all of \(\mc{P}(k,l) \times \mc{P}(m,n)\) into 
      $\mc{P}(k +\nobreak m, l +\nobreak n)$.
  \end{itemize}
  If these are such that
  \begin{enumerate}
    \item
      \(L(G) = L(H)\) whenever \(G \simeq H\),
    \item
      \(L(G) = L(G') \circ L(G'')\) for every cut decomposition 
      $(G',G'')$ of $G$,
    \item
      \(L(G) = L(G') \otimes L(G'')\) for every split decomposition 
      $(G',G'')$ of $G$,
    \item
      \(L\Bigl( \Nwfuse{ \vek{m} }{ \gamma^{\vek{m}}_{\vek{n}} }{ 
      \vek{n} } \Bigr) = \gamma\) for all \(\gamma \in \mc{P}\bigl( 
      \Norm{\vek{m}}, \Norm{\vek{n}} \bigr)\)
  \end{enumerate}
  then $\mc{P}$ is a \PROP\ with composition $\circ$, tensor product 
  $\otimes$, and permutation map given by
  \begin{equation}
    \phi_n(\sigma) = 
    L\Bigl( \bigl( \{0,1\}, [n], 0, \sigma, 1, \same{n}, \varnothing 
    \bigr) \Bigr)
  \end{equation}
  where the latter $0$ and $1$ denote the constant maps \([n] \Fpil 
  \{0\}\) and \([n] \Fpil \{1\}\) respectively.
\end{theorem}
\begin{proof}
  Not surprisingly, this is shown by verifying that the axioms are 
  fulfilled. The ease by which they derive from the network notation 
  is however rather striking. For each axiom, the calculations below 
  use the same variables $a$, $b$, $c$, $d$, $k$, $l$, $m$, $n$, $r$, 
  $s$, $\sigma$, and $\tau$ as in Defintion~\ref{Def:PROP}. In 
  addition, boldface letters $\vek{k}$, $\vek{l}$, $\vek{m}$, 
  $\vek{n}$, $\vek{r}$, and $\vek{s}$ are used to denote arbitrary 
  but pairwise disjoint lists of edge labels (thus technically lists 
  of natural numbers in this case, since that is how $\Nw(\mc{P})$ is 
  defined) whose lengths are equal to the corresponding normal 
  letter: \(\Norm{\vek{k}} = k\), \(\Norm{\vek{l}} = l\), etc.
  
  The composition associativity axiom \((a \circ\nobreak b) \circ c = 
  a \circ (b \circ\nobreak c)\) follows from considering the 
  two obvious nontrivial cuts in $\Nwfuse{ \vek{k} }{ 
  a^{\vek{k}}_{\vek{l}} b^{\vek{l}}_{\vek{m}} c^{\vek{m}}_{\vek{n}} }{ 
  \vek{n} }$;
  \begin{align*}
    (a \circ b) \circ c 
    ={}&
    \biggl( 
      L\Bigl( \Nwfuse{\vek{k}}{ a^{\vek{k}}_{\vek{l}} }{\vek{l}} 
        \Bigr) \circ
      L\Bigl( \Nwfuse{\vek{l}}{ (3\colon b)^{\vek{l}}_{\vek{m}} }{\vek{m}} 
        \Bigr)
    \biggr) \circ c
    = \\ ={}&
    L\Bigl( \Nwfuse{\vek{k}}{ a^{\vek{k}}_{\vek{l}} 
      b^{\vek{l}}_{\vek{m}} }{\vek{m}} \Bigr)
    \circ 
    L\Bigl( \Nwfuse{\vek{m}}{ (4\colon c)^{\vek{m}}_{\vek{n}} }{\vek{n}} 
        \Bigr)
    = \displaybreak[0]\\ ={}&
    L\Bigl( \Nwfuse{\vek{k}}{ a^{\vek{k}}_{\vek{l}} 
      b^{\vek{l}}_{\vek{m}} c^{\vek{m}}_{\vek{n}} }{\vek{n}} 
    \Bigr)
    = \displaybreak[0]\\ ={}&
    L\Bigl( 
      \Nwfuse{\vek{k}}{ a^{\vek{k}}_{\vek{l}} }{\vek{l}} 
    \Bigr) \circ L\Bigl(
      \Nwfuse{\vek{l}}{ (3\colon b)^{\vek{l}}_{\vek{m}} 
        (4\colon c)^{\vek{m}}_{\vek{n}} }{\vek{n}} 
    \Bigr)
    = \\ ={}&
    a \circ \biggl( L\Bigl(
      \Nwfuse{\vek{l}}{ (3\colon b)^{\vek{l}}_{\vek{m}} }{\vek{m}}
    \Bigr) \circ L\Bigl(
      \Nwfuse{\vek{m}}{ (4\colon c)^{\vek{m}}_{\vek{n}} }{\vek{n}} 
    \Bigr) \biggr)
    =
    a \circ (b \circ c)
    \text{.}
  \end{align*}
  
  The composition identity axiom follows from considering the two 
  obvious trivial cuts in $\Nwfuse{ \vek{m} }{ a^{\vek{m}}_{\vek{n}} }
  {\vek{n}}$. Let \(\vek{m}' = \vek{N}_m(1,1)\) and \(\vek{n}' = 
  \vek{N}_n(1,1)\); these specific lists are needed in this argument 
  because $\phi_n$ is defined as the value of $L$ for equally specific 
  networks. Then
  \begin{multline*}
    a
    =
    L\Bigl(
    \Nwfuse{\vek{m}}{a^{\vek{m}}_{\vek{n}}}{\vek{n}}
    \Bigr)
    =
    L\Bigl( 
      \Nwfuse{\vek{m}}{1}{\vek{m}}
    \Bigr) \circ L\Bigl(
      \Nwfuse{\vek{m}}{a^{\vek{m}}_{\vek{n}}}{\vek{n}}
    \Bigr)
    = \\ =
    L\Bigl( \Nwfuse{\vek{m}'}{1}{\vek{m}'} \Bigr) \circ a
    =
    \phi_m(\same{m}) \circ a
  \end{multline*}
  and
  \begin{multline*}
    a
    =
    L\Bigl(
    \Nwfuse{\vek{m}}{a^{\vek{m}}_{\vek{n}}}{\vek{n}}
    \Bigr)
    =
    L\Bigl(
      \Nwfuse{\vek{m}}{a^{\vek{m}}_{\vek{n}}}{\vek{n}}
    \Bigr) \circ L\Bigl( 
      \Nwfuse{\vek{n}}{1}{\vek{n}}
    \Bigr)
    = \\ =
    a \circ L\Bigl( \Nwfuse{\vek{n}'}{1}{\vek{n}'} \Bigr)
    =
    a \circ \phi_n(\same{n}) \text{.}
  \end{multline*}
  The tensor identity axiom similarly follows from considering the 
  two trivial splits of the same network. If $\vek{o}$ is the empty 
  list then
  \begin{gather*}
    a
    =
    L\Bigl(
    \Nwfuse{\vek{om}}{a^{\vek{m}}_{\vek{n}}}{\vek{on}}
    \Bigr)
    =
    L\Bigl( 
      \Nwfuse{\vek{o}}{1}{\vek{o}}
    \Bigr) \otimes L\Bigl(
      \Nwfuse{\vek{m}}{a^{\vek{m}}_{\vek{n}}}{\vek{n}}
    \Bigr)
    =
    \phi_0(\same{0}) \otimes a
    \text{,}\\
    a
    =
    L\Bigl(
    \Nwfuse{\vek{mo}}{a^{\vek{m}}_{\vek{n}}}{\vek{no}}
    \Bigr)
    =
    L\Bigl(
      \Nwfuse{\vek{m}}{a^{\vek{m}}_{\vek{n}}}{\vek{n}}
    \Bigr) \otimes L\Bigl( 
      \Nwfuse{\vek{o}}{1}{\vek{o}}
    \Bigr)
    =
    a \otimes \phi_0(\same{0}) 
    \text{.}
  \end{gather*}
  
  By now, it should not be a surprise that the tensor associativity 
  axiom follows from considering the two nontrivial splits in 
  $\Nwfuse{ \vek{kmr} }{ a^{\vek{k}}_{\vek{l}} b^{\vek{m}}_{\vek{n}} 
  c^{\vek{r}}_{\vek{s}} }{ \vek{lns} }$, like so:
  \begin{align*}
    (a \otimes b) \otimes c
    ={}&
    \biggl( L\Bigl(
      \Nwfuse{ \vek{k} }{ a^{\vek{k}}_{\vek{l}} }{ \vek{l} }
    \Bigr) \otimes L\Bigl(
      \Nwfuse{ \vek{m} }{ (3\colon b)^{\vek{m}}_{\vek{n}} }{ \vek{n} }
    \Bigr) \biggr) \otimes c
    = \\ ={}&
    L\Bigl(
      \Nwfuse{ \vek{km} }{ 
        a^{\vek{k}}_{\vek{l}} b^{\vek{m}}_{\vek{n}} 
      }{ \vek{ln} }
    \Bigr) \otimes L\Bigl(
      \Nwfuse{ \vek{r} }{ (4\colon c)^{\vek{r}}_{\vek{s}} }{ \vek{s} }
    \Bigr)
    = \displaybreak[0]\\ ={}&
    L\Bigl(
      \Nwfuse{ \vek{kmr} }{ 
        a^{\vek{k}}_{\vek{l}} b^{\vek{m}}_{\vek{n}} 
        c^{\vek{r}}_{\vek{s}}
      }{ \vek{lns} }
    \Bigr) 
    = \displaybreak[0]\\ ={}&
    L\Bigl(
      \Nwfuse{ \vek{k} }{ a^{\vek{k}}_{\vek{l}} }{ \vek{l} }
    \Bigr) \otimes L\Bigl(
      \Nwfuse{ \vek{mr} }{ 
        (3\colon b)^{\vek{m}}_{\vek{n}} 
        (4\colon c)^{\vek{r}}_{\vek{s}}
      }{ \vek{ns} }
    \Bigr) 
    = \displaybreak[0]\\ ={}&
    a \otimes \biggl( L\Bigl(
      \Nwfuse{ \vek{m} }{ 
        (3\colon b)^{\vek{m}}_{\vek{n}} 
      }{ \vek{n} }
    \Bigr) \otimes L\Bigl(
      \Nwfuse{ \vek{r} }{ 
        (4\colon c)^{\vek{r}}_{\vek{s}}
      }{ \vek{s} }
    \Bigr) \biggr)
    = \\ ={}&
    a \otimes (b \otimes c) 
    \text{.}
  \end{align*}
  Similarly the composition--tensor compatibility axiom follows from 
  considering the nontrivial split and nontrivial cut in 
  $\Nwfuse{ \vek{lk} }{ a^{\vek{l}}_{\vek{m}} b^{\vek{m}}_{\vek{n}} 
  c^{\vek{k}}_{\vek{r}} d^{\vek{r}}_{\vek{s}} }{ \vek{ns} }$, the 
  latter of which should be obvious above $(2,4)$ (or obvious below 
  $(3,5)$, which here is the same thing) to yield
  \begin{align*}
    (a \circ b) \otimes (c \circ d) 
    ={}&
    \biggl( L\Bigl(
      \Nwfuse{ \vek{l} }{ a^{\vek{l}}_{\vek{m}} }{ \vek{m} }
    \Bigr) \circ L\Bigl(
      \Nwfuse{ \vek{m} }{ (3\colon b)^{\vek{m}}_{\vek{n}} }{ \vek{n} }
    \Bigr) \biggr) \\
    &\quad {}\otimes \biggl( L\Bigl(
      \Nwfuse{ \vek{k} }{ (4\colon c)^{\vek{k}}_{\vek{r}} }{ \vek{r} }
    \Bigr) \circ L\Bigl(
      \Nwfuse{ \vek{r} }{ (5\colon d)^{\vek{r}}_{\vek{s}} }{ \vek{s} }
    \Bigr) \biggr) 
    = \displaybreak[0]\\ ={}&
    L\Bigl(
      \Nwfuse{ \vek{l} }{ 
        a^{\vek{l}}_{\vek{m}} b ^{\vek{m}}_{\vek{n}} 
      }{ \vek{n} }
    \Bigr) \otimes L\Bigl(
      \Nwfuse{ \vek{k} }{ 
        (4\colon c)^{\vek{k}}_{\vek{r}}
        (5\colon d)^{\vek{r}}_{\vek{s}}
      }{ \vek{s} }
    \Bigr)
    = \displaybreak[0]\\ ={}&
    L\Bigl(
      \Nwfuse{ \vek{lk} }{ 
        a^{\vek{l}}_{\vek{m}} b ^{\vek{m}}_{\vek{n}} 
        c^{\vek{k}}_{\vek{r}} d^{\vek{r}}_{\vek{s}}
      }{ \vek{ns} }
    \Bigr)
    = \displaybreak[0]\\ ={}&
    L\Bigl(
      \Nwfuse{ \vek{lk} }{ 
        (2\colon a)^{\vek{l}}_{\vek{m}} 
        (4\colon c)^{\vek{k}}_{\vek{r}}
      }{ \vek{mr} }
    \Bigr) \\ &\quad{}\circ L\Bigl(
      \Nwfuse{ \vek{mr} }{ 
        (3\colon b)^{\vek{m}}_{\vek{n}} 
        (5\colon d)^{\vek{r}}_{\vek{s}}
      }{ \vek{ns} }
    \Bigr)
    = \displaybreak[0]\\ ={}&
    \biggl( L\Bigl(
      \Nwfuse{ \vek{l} }{ a^{\vek{l}}_{\vek{m}} }{ \vek{m} }
    \Bigr) \otimes L\Bigl(
      \Nwfuse{ \vek{k} }{ 
        (4\colon c)^{\vek{k}}_{\vek{r}}
      }{ \vek{r} }
    \Bigr) \biggr) \\ 
    &\quad{}\circ \biggl( L\Bigl(
      \Nwfuse{ \vek{m} }{ 
        (3\colon b)^{\vek{m}}_{\vek{n}} 
      }{ \vek{n} }
    \Bigr) \otimes L\Bigl(
      \Nwfuse{ \vek{r} }{ 
        (5\colon d)^{\vek{r}}_{\vek{s}}
      }{ \vek{s} }
    \Bigr) \biggr)
    = \\ ={}&
    (a \otimes c) \circ (b \otimes d)
    \text{.}
  \end{align*}
  
  The permutation composition axiom is again a matter of making the 
  right cut in the right network, but here the networks can be given 
  directly in tuple form. Letting
  \begin{align*}
    G ={}& \bigl( 
      \{0,1\}, [n], 0, \sigma\tau, 1, \same{n}, \varnothing
    \bigr) \text{,}&
    G' ={}& \bigl( 
      \{0,1\}, [n], 0, \sigma\tau, 1, \tau, \varnothing
    \bigr) \text{,}\\
    G'' ={}& \bigl( 
      \{0,1\}, [n], 0, \tau, 1, \same{n}, \varnothing
    \bigr) \text{,}&
    G''' ={}& \bigl( 
      \{0,1\}, [n], 0, \sigma, 1, \same{n}, \varnothing
    \bigr) 
  \end{align*}
  one finds that \(G''' \simeq G'\) and $(G',G'')$ is a cut 
  decomposition of $G$, which implies
  \begin{equation*}
    \phi_n(\sigma\tau) =
    L(G) =
    L(G') \circ L(G'') =
    L(G''') \circ \phi_n(\tau) =
    \phi_n(\sigma) \circ \phi_n(\tau) \text{.}
  \end{equation*}
  Analogously, the permutation juxtaposition axiom is a matter of 
  making the right split of the right network. Defining \(E = 
  [m +\nobreak n] \setminus [m]\), \(g(e) = 
  \tau(e -\nobreak m)\) and \(s(e) = e-m\) for all \(e \in E\), the 
  networks
  \begin{align*}
    G ={}& \bigl( 
      \{0,1\}, [m+n], 0, \sigma\star\tau, 1, \same{m+n}, \varnothing
    \bigr) \text{,}&
    G'' ={}& \bigl( 
      \{0,1\}, E, 0, g, 1, s, \varnothing
    \bigr) \text{,}\\
    G' ={}& \bigl( 
      \{0,1\}, [m], 0, \sigma, 1, \same{m}, \varnothing
    \bigr) \text{,}&
    G''' ={}& \bigl( 
      \{0,1\}, [n], 0, \tau, 1, \same{n}, \varnothing
    \bigr) 
  \end{align*}
  are such that $(G',G'')$ is a split decomposition of $G$ and \(G'' 
  \simeq G'''\), from which follows
  \begin{equation*}
    \phi_{m+n}(\sigma\star\tau) =
    L(G) =
    L(G') \otimes L(G'') =
    \phi_m(\sigma) \otimes L(G''') =
    \phi_m(\sigma) \otimes \phi_n(\tau) \text{.}
  \end{equation*}
  
  Finally, the tensor permutation axiom follows primarily from 
  considering two trivial cuts in $\Nwfuse{\vek{mk}}{ 
  a^{\vek{k}}_{\vek{l}} b^{\vek{m}}_{\vek{n}} }{\vek{ln}}$, namely 
  that which is obvious below $(2,3)$ and that which is obvious above 
  $(3,2)$. In both these cuts, the nonempty part has a nontrivial 
  split, but what is the left part in one split is the right part in 
  the other, and that is what the axiom is all about. Let
  \begin{align*}
    G' ={}& \bigl( 
      \{0,1\}, [k+m], 0, \cross{k}{m}, 1, \same{k+m}, \varnothing 
    \bigr) \text{,}\\
    G'' ={}& \bigl( 
      \{0,1\}, [l+n], 0, \cross{l}{n}, 1, \same{l+n}, \varnothing 
    \bigr) 
  \end{align*}
  so that \(G' \simeq \Nwfuse{\vek{mk}}{1}{\vek{km}}\) and \(G'' 
  \simeq \Nwfuse{\vek{nl}}{1}{\vek{ln}}\). Then
  \begin{align*}
    \phi(\cross{k}{m}) \circ (a \otimes b)
    ={} \hspace{-3em}& \\ ={}&
    L(G') \circ
    \biggl( L\Bigl( 
     \Nwfuse{\vek{k}}{ (2\colon a)^{\vek{k}}_{\vek{l}} }{\vek{l}}
    \Bigr) \otimes L\Bigl( 
      \Nwfuse{\vek{m}}{ (3\colon b)^{\vek{m}}_{\vek{n}} }{\vek{n}}
    \Bigr) \biggr)
    = \displaybreak[0]\\ ={}&
    L\Bigl( \Nwfuse{\vek{mk}}{1}{\vek{km}} \Bigr) \circ
    L\Bigl( \Nwfuse{\vek{km}}{ 
      (2\colon a)^{\vek{k}}_{\vek{l}} (3\colon b)^{\vek{m}}_{\vek{n}} 
    }{\vek{ln}} \Bigr)
    = \displaybreak[0]\\ ={}&
    L\Bigl( \Nwfuse{\vek{mk}}{ 
      (2\colon a)^{\vek{k}}_{\vek{l}} (3\colon b)^{\vek{m}}_{\vek{n}} 
    }{\vek{ln}} \Bigr)
    = \displaybreak[0]\\ ={}&
    L\Bigl( \Nwfuse{\vek{mk}}{ 
      (3\colon b)^{\vek{m}}_{\vek{n}} (2\colon a)^{\vek{k}}_{\vek{l}} 
    }{\vek{nl}} \Bigr) \circ 
    L\Bigl( \Nwfuse{\vek{nl}}{1}{\vek{ln}} \Bigr)
    = \displaybreak[0]\\ ={}&
    \biggl( L\Bigl( 
      \Nwfuse{\vek{m}}{ (3\colon b)^{\vek{m}}_{\vek{n}} }{\vek{n}} 
    \Bigr) \otimes L\Bigl( 
      \Nwfuse{\vek{k}}{ (2\colon a)^{\vek{k}}_{\vek{l}} }{\vek{l}} 
    \Bigr) \biggr) \circ 
    L(G'')
    = \\ ={}&
    (b \otimes a) \circ \phi(\cross{l}{n})
    \text{,}
  \end{align*}
  which completes the proof.
\end{proof}

Theorem~\ref{S:PROP2-definition} can be taken as an alternative axiom 
system for \PROPs, but it is somewhat awkward in that it refers to 
the rather technical concepts of `cut' and `split'. A simpler concept 
that one can use instead is that of ``boxing up'' a convex 
subnetwork as a new vertex.

\begin{definition} \label{Def:KonvextDeluttryck}
  A network \(H = (V_H, E_H, h_H, g_H, t_H, s_H, D_H)\) is said to be 
  a \DefOrd{convex subnetwork} of the network \(G = 
  (V_G, E_G, h_G, g_G, t_G, s_G, D_G)\), written 
  \index{\sqsubseteq@$\sqsubseteq$}\(H \sqsubseteq G\), if:
  \begin{itemize}
    \item
      \(V_H \subseteq V_G\), \(E_H \subseteq E_G\), and \(D_H = 
      \restr{D_G}{V_H \setminus \{0,1\}}\).
    \item
      For any \(e \in E_H\), \(h_H(e)=0\) or \((h_H,g_H)(e) = 
      (h_G,g_G)(e)\), and similarly \(t_H(e)=1\) or \((t_H,s_H)(e) = 
      (t_G,s_G)(e)\).
    \item
      There is no sequence \(e_1,\dotsc,e_n \in E_G\) with \(n>1\) such 
      that \(e_1,e_n \in E_H\), \(h_H(e_1) = 0\), \(t_H(e_n) = 1\), 
      and \(h_G(e_i) = t_G(e_{i+1})\) for \(i \in [n -\nobreak 1]\).
  \end{itemize}
  The sense in which the third condition is convexity is that it says 
  no ``ray'' (path) leaving the subnetwork will visit it again.
  
  Given \(G \in \Nw(\mc{P})\) and \(H \sqsubseteq G\), define the 
  \DefOrd{splice map} \index{\div@$\div$}$G \div H$ by 
  \((G \div\nobreak H)(\gamma) = (V',E',h',g',t',s',D')\) for any 
  \(\gamma \in \mc{P}\bigl( \omega(H), \alpha(H) \bigr)\), where
  \begin{align*}
    u_0 ={}& 1 + \max V_G \text{,}\\
    V' ={}& \{0,1\} \cup (V_G \setminus V_H) \cup \{u_0\}
      \text{,}\\
    E' ={}& \{ 3e \}_{e \in E_G \setminus E_H}
      \begin{aligned}[t]
        &{}\cup
        \setOf{ 3e+1 }{ \text{\(e \in E_H\) and \(h_H(e) = 0\)} }
        \\ &{}\cup
        \setOf{ 3e+2 }{ \text{\(e \in E_H\) and \(t_H(e) = 1\)} }
        \text{,} 
      \end{aligned}
      \displaybreak[0]\\
    D'(v) ={}& \begin{cases}
      D_G(v)& \text{if \(v \in V_G \setminus V_H\),}\\
      \gamma& \text{if \(v=u_0\)}
    \end{cases} 
    \qquad\text{for \(v \in V' \setminus \{0,1\}\),}
    \displaybreak[0]\\
    (h',g')(3e) ={}& (h_G,g_G)(e)
      \qquad\text{for \(3e \in E'\),}
      \\
    (h',g')(3e+1) ={}& (h_G,g_G)(e)
      \qquad\text{for \(3e+1 \in E'\),}
      \\
    (h',g')(3e+2) ={}& \bigl( u_0, s_H(e) \bigr)
      \qquad\text{for \(3e+2 \in E'\),}
      \displaybreak[0]\\
    (t',s')(3e) ={}& (t_G,s_G)(e)
      \qquad\text{for \(3e \in E'\),}
      \\
    (t',s')(3e+1) ={}& \bigl( u_0, g_H(e) \bigr)
      \qquad\text{for \(3e+1 \in E'\),}
      \\
    (t',s')(3e+2) ={}& (t_G,s_G)(e)
      \qquad\text{for \(3e+2 \in E'\).}
  \end{align*}
\end{definition}

Informally, $G \div H$ may be thought of as drawing a frame around 
the $H$ part of $G$, erasing the interior of that frame, and instead 
putting the argument of $G \div H$ there as decoration of this new 
vertex. The frame needs to be convex for it to be valid as the 
``shape'' of a new vertex.

\begin{lemma}
  Let $\mc{P}$ be an $\N^2$-graded set. For any convex 
  subnetwork $H$ of a network \(G \in \Nw(\mc{P})\),
  \[
     G \div H  
     \quad\text{maps}\quad
     \mc{P}\bigl( \omega(H), \alpha(H) \bigr) 
     \quad\text{into}\quad
     \Nw(\mc{P})\bigl( \omega(G), \alpha(G) \bigr)
     \text{.}
  \]
\end{lemma}
\begin{proof}
  What needs to be proved is mainly that $(G \div\nobreak 
  H)(\gamma)$, for every \(\gamma \in \mc{P}\bigl( \omega(H), 
  \alpha(H) \bigr)\), fits the definition of a network. Beginning 
  with the annotation, this fits for a vertex \(v \neq u_0\) because 
  it is the same as in $G$. It fits for vertex $u_0$ because 
  \(d^-(u_0) = \alpha(H) = \alpha(\gamma)\) and \(d^+(u_0) = 
  \omega(H) = \omega(\gamma)\). Head and tail indices in 
  $(G \div\nobreak H)(\gamma)$ are assigned from $1$ and up because 
  at every vertex of $(G \div\nobreak H)(\gamma)$ the head (and tail, 
  respectively) indices are the indices at some vertex of $G$ or $H$ 
  (in the latter case those vertices are $0$ and $1$). 
  Head and head index uniquely identifies an edge for the same 
  reason, and the same is true for tail and tail index.
  
  What requires $H$ to be a \emph{convex} subnetwork is the 
  acyclicity condition. Seeking a contradition, one may assume 
  \(e_1,\dotsc,e_n \in E'\) are such that \(h'(e_i) = t'(e_{i+1})\) 
  for \(i \in [n -\nobreak 1]\), that \(h'(e_n) = t'(e_1)\), and that 
  $n$ is the least possible for such a sequence of edges. Then 
  these edges cannot all be of the $3e$ kind, because that would 
  violate the acyclicity of $G$. Hence there must be edges incident 
  with $u_0$, and it can without loss of generality be assumed that 
  \(h'(e_n) = u_0 = t'(e_1)\). There can be no further such edges, 
  because if \(h'(e_k) = u_0\) for some \(k < n\) then 
  \(e_1,\dotsc,e_k\) is a shorter sequence satisfying the same 
  conditions, violating the minimality of $n$. On the other hand, 
  since \(h'(e_i) = u_0\) implies \(e_i \equiv 2 \pmod{3}\) and 
  \(t'(e_i) = u_0\) implies \(e_i \equiv 1 \pmod{3}\), it follows 
  that \(n \geqslant 2\). Hence \(\lfloor e_1/3 \rfloor, \dotsc, 
  \lfloor e_n/3 \rfloor \in E_G\) is a sequence of length greater 
  than one such that \(\lfloor e_1/3 \rfloor, \lfloor e_n/3 \rfloor 
  \in E_H\), \(h_H\bigl(\lfloor e_1/3 \rfloor\bigr) = 0\), 
  \(t_H\bigl(\lfloor e_n/3 \rfloor\bigr) = 1\), and
  \(h_G\bigl(\lfloor e_i/3 \rfloor\bigr) = 
  t_G\bigl(\lfloor e_{i+1}/3 \rfloor\bigr)\) for \(i \in 
  [n -\nobreak 1]\), but this violates the last condition for $H$ to 
  be a convex subnetwork.
\end{proof}

\begin{definition}
  Let $\mc{P}$ be an $\N^2$-graded set. A \DefOrd{network evaluation 
  map} for $\mc{P}$ is an $\N^2$-graded set morphism \(L\colon 
  \Nw(\mc{P}) \Fpil \mc{P}\) such that
  \begin{enumerate}
    \item
      if \(G \simeq H\) then \(L(G) = L(H)\),
    \item
      if \(H \sqsubseteq G\) then
      \begin{math}
        L\Bigl( (G \div H)\bigl( L(H) \bigr) \Bigr) = L(G)
      \end{math},
    \item
      \(L\Bigl( \Nwfuse{ \vek{m} }{ \gamma^{\vek{m}}_{\vek{n}} }{ 
      \vek{n} } \Bigr) = \gamma\) for all \(\gamma \in \mc{P}\bigl( 
      \Norm{\vek{m}}, \Norm{\vek{n}} \bigr)\).
  \end{enumerate}
\end{definition}

\begin{theorem} \label{S:PROP3-definition}
  An $\N^2$-graded set $\mc{P}$ equipped with a network evaluation 
  map $L$ is a \PROP, with operations defined as
  \begin{align}
    a \circ b :={}& 
      L\Bigl( \Nwfuse{\vek{l}}{ 
        a^{\vek{l}}_{\vek{m}} b^{\vek{m}}_{\vek{n}} 
      }{\vek{n}} \Bigr) &
      \text{for \(a \in \mc{P}\bigl( \Norm{\vek{l}}, \Norm{\vek{m}} 
      \bigr)\) }&\text{and \(b \in \mc{P}\bigl( \Norm{\vek{m}}, 
      \Norm{\vek{n}} \bigr)\),}
    \label{Eq:EvalDefCirc}
    \\
    a \otimes b :={}& 
      L\Bigl( \Nwfuse{\vek{km}}{ 
        a^{\vek{k}}_{\vek{l}} b^{\vek{m}}_{\vek{n}} 
      }{\vek{ln}} \Bigr)  &
      \text{for \(a \in \mc{P}\bigl( \Norm{\vek{k}}, \Norm{\vek{l}} 
      \bigr)\) }&\text{and \(b \in \mc{P}\bigl( \Norm{\vek{m}}, 
      \Norm{\vek{n}} \bigr)\),}
    \label{Eq:EvalDefOtimes}
    \\
    \phi(\sigma) :={}&
    L\Bigl( \bigl( \{0,1\}, [n], 0, \sigma, 1, \same{n}, 
      \varnothing 
    \bigr) \Bigr)
      \hspace{-4em}&\hspace{4em}
      &\text{for \(\sigma \in \Sigma_n\).}
  \end{align}
  Conversely, if $\mc{P}$ is a \PROP\ then \(\eval\colon \Nw(\mc{P}) 
  \Fpil \mc{P}\) is a network evaluation map.
\end{theorem}
\begin{remark}
  Since many $\N^2$-graded sets support several different \PROP\ 
  structures, there may also be several different network evaluation 
  maps.
\end{remark}
\begin{proof}
  That the \PROP\ axioms are, for the most part, implicit in the 
  network notation was the main point already of 
  Theorem~\ref{S:PROP2-definition}, so what is needed here is 
  primarily a verification that the composition and tensor product 
  defined above exhibits the behaviour with respect to cuts and 
  splits that that theorem requires.
  
  Suppose $G$ is a network and $(G_0,G_1)$ is a decomposition of $G$ 
  induced by some cut. Then $G_0$ and $G_1$ are both convex 
  subnetworks of $G$. Let \(\vek{l} = \vek{e}^-_{G_0}(0)\), \(\vek{m} 
  = \vek{e}^+_{G_0}(1) = \vek{e}^-_{G_1}(0)\), and \(\vek{n} = 
  \vek{e}^+_{G_1}(1)\). Let \(H = (G \div G_0)\bigl( L(G_0) \bigr)\) 
  and \(u_0 = \max V_H\). Let $G_1'$ be the network isomorphic to 
  $G_1$ such that the isomorphism $(\chi,\psi)\colon G_1 \Fpil G_1'\) 
  has \(\chi(v)=v\) for all \(v \in V_{G_1}\), \(\psi(e) = 3e\) for 
  \(e \in E_{G_1} \setminus E_{G_0}\), and \(\psi(e) = 3e+2\) for 
  \(e \in E_{G_1} \cap E_{G_0}\). Then \(G_1' \sqsubseteq H\), and 
  hence
  \begin{multline*}
    L(G_0) \circ L(G_1) 
    =
    L\Bigl( \Nwfuse{\vek{l}}{ 
      \bigl(2\colon L(G_0)\bigr)^{\vek{l}}_{\vek{m}} 
      \bigl(3\colon L(G_1)\bigr)^{\vek{m}}_{\vek{n}} 
    }{\vek{n}} \Bigr) 
    = \\ =
    L\Bigl( \Nwfuse{\vek{l}}{ 
      \bigl(u_0\colon L(G_0)\bigr)^{\vek{l}}_{\vek{m}} 
      \bigl(u_0+1\colon L(G_1)\bigr)^{\vek{m}}_{\vek{n}} 
    }{\vek{n}} \Bigr) 
    = 
    L\Bigl( (H \div G_1')\bigl( L(G_1) \bigr) \Bigr) 
    = \\ =
    L\Bigl( (H \div G_1')\bigl( L(G_1') \bigr) \Bigr) 
    =
    L(H)
    =
    L\Bigl( (G \div G_0)\bigl( L(G_0) \bigr) \Bigr)
    =
    L(G)\text{.}
  \end{multline*}
  
  Next suppose $G$ is a network and $(G_0,G_1)$ is a decomposition of 
  $G$ induced by some split. Then $G_0$ and $G_1$ are both convex 
  subnetworks of $G$. Let \(\vek{k} = \vek{e}^-_{G_0}(0)\), \(\vek{l} 
  = \vek{e}^+_{G_0}(1)\), \(\vek{m} = \vek{e}^-_{G_1}(0)\), and 
  \(\vek{n} = \vek{e}^+_{G_1}(1)\). Let \(H = (G \div G_0)\bigl( 
  L(G_0) \bigr)\) and \(u_0 = \max V_H\). Let $G_1'$ be the network 
  isomorphic to $G_1$ such that the isomorphism $(\chi,\psi)\colon 
  G_1 \Fpil G_1'\) has \(\chi(v)=v\) for all \(v \in V_{G_1}\) and 
  \(\psi(e) = 3e\) for \(e \in E_{G_1}\). Then \(G_1' \sqsubseteq H\), 
  and hence
  \begin{multline*}
    L(G_0) \otimes L(G_1) 
    =
    L\Bigl( \Nwfuse{\vek{km}}{ 
      L(G_0)^{\vek{k}}_{\vek{l}} L(G_1)^{\vek{m}}_{\vek{n}} 
    }{\vek{ln}} \Bigr) 
    =
    L\Bigl( (H \div G_1')\bigl( L(G_1) \bigr) \Bigr) 
    = \\ =
    L\Bigl( (H \div G_1')\bigl( L(G_1') \bigr) \Bigr) 
    =
    L(H)
    =
    L\Bigl( (G \div G_0)\bigl( L(G_0) \bigr) \Bigr)
    =
    L(G)\text{.}
  \end{multline*}
  This has shown a network evaluation map gives rise to a \PROP\ 
  structure.
  
  For the converse, it must be shown that
  \[
    \eval\Bigl( (G \div H)\bigl( \eval(H) \bigr) \Bigr) = \eval(G)
  \]
  whenever \(H \sqsubseteq G\). Let $W_0'$ be set set of all vertices 
  in $V_G \setminus \{0,1\}\) that can be written as $h_G(e_n)$ for 
  some \(e_1,\dotsc,e_n \in E_G\) where \(h_G(e_i) = t_G(e_{i+1})\) 
  for \(i \in [n -\nobreak 1]\) and \(e_1 \in E_H\). Let \(W_0 = W_0' 
  \setminus V_H\), \(W_1 = V_G \setminus \{0,1\} \setminus W_0'\), and 
  \(W_2 = V_H \setminus \{0,1\}\). Then $(W_0 \cup\nobreak W_2, W_1)$ 
  and $(W_0, W_2 \cup\nobreak W_1)$ are both cuts in $G$; a back-edge 
  to $W_1$ would contradict the defining property of $W_0'$, and a 
  back-edge to $W_2$ would contradict the convexity of $H$. Let
  \begin{align*}
    F_0 ={}& \setOf[\big]{ e \in E_H }{ h_H(e) = 0 } 
      \text{,}\\
    F_1 ={}& \setOf[\big]{ e \in E_H }{ t_H(e) = 1 } 
      \text{,}\\
    F_2 ={}& \setOf[\big]{ e \in E_G \setminus E_H }{ 
      h_G(e) \in W_0 \cup \{0\}, t_G(e) \in W_1 \cup \{1\} }
      \text{;}
  \end{align*}
  then $F_0 \cup F_2$ are the cut edges of $(W_0, W_2 \cup\nobreak 
  W_1)$ and $F_1 \cup F_2$ are the cut edges of $(W_0 \cup\nobreak W_2, 
  W_1)$. Let $p_0$ be an ordering of the former cut such that 
  \(p_0(e) \leqslant \card{F_2}\) for \(e \in F_2\) and \(p_0(e) = 
  \card{F_2} + g_H(e)\) for \(e \in F_0\). Let \(p_1(e) = p_0(e)\) 
  for \(e \in F_2\) and \(p_1(e) = \card{F_2} + s_H(e)\) for \(e \in 
  F_1\). Then $(W_0 \cup\nobreak W_2, W_1, p_1)$ is an ordered cut in 
  $G$, so let $(G_0,G_1)$ be the corresponding decomposition. 
  Furthermore $(W_0, W_2, p_0)$ is an ordered cut in $G_0$, so let 
  $(G_{00},G_{01})$ be the corresponding decomposition of $G_0$. Then
  \begin{multline*}
    \eval(G) =
    \eval(G_{00}) \circ \eval(G_{01}) \circ \eval(G_1) = \\ =
    \eval(G_{00}) \circ \phi(\same{\card{F_2}}) \otimes \eval(H) 
      \circ \eval(G_1)
      \text{.}
  \end{multline*}
  
  Next let \(G' = (G \div H)\bigl( \eval(H) \bigr)\) and \(u = \max 
  V_{G'}\). Then $\bigl( W_0 \cup\nobreak \{u\}, W_1 \bigr)$ and 
  $\bigl( W_0, \{u\} \cup\nobreak W_1 \bigr)$ are cuts in $G'$; two 
  orderings of these are
  \begin{align*}
    p_0' \colon{}&
    \left\{ \begin{aligned}
      p_0'(3e) ={}& p_0(e) && \text{for \(e \in F_2\),}\\
      p_0'(3e+1) ={}& p_0(e) && \text{for \(e \in F_0\),}
    \end{aligned}\right.
    \\
    p_1' \colon{}&
    \left\{ \begin{aligned}
      p_1'(3e) ={}& p_1(e) && \text{for \(e \in F_2\),}\\
      p_1'(3e+2) ={}& p_1(e) && \text{for \(e \in F_1\).}
    \end{aligned}\right.
  \end{align*}
  Let $(G_0',G_1')$ be the decomposition of $G'$ induced by 
  $\bigl( W_0 \cup\nobreak \{u\}, W_1, p_1' \bigr)$ and let 
  $(G_{00}',G_{01}')$ be the decomposition of $G_0'$ induced by 
  $\bigl( W_0, \{u\}, p_0'\bigr)$. Then \(G_{00}' \simeq G_{00}\) and 
  \(G_1' \simeq G_1\). Furthermore
  \begin{math}
    G_{01}' = \Nwfuse[\Big]{ \vek{km} }{ 
      \bigl( u\colon \eval(H) \bigr)^\vek{m}_\vek{n}
    }{ \vek{kn} }
  \end{math}
  for some $\vek{k}$, $\vek{m}$, and $\vek{n}$ such that 
  \(\Norm{\vek{k}} = \card{F_2}\), \(\Norm{\vek{m}} = \card{F_0}\), 
  and \(\Norm{\vek{n}} = \card{F_1}\). Hence \(\eval(G_{01}' = 
  \phi(\same{\card{F_2}}) \otimes \eval(H)\) and thus
  \begin{multline*}
    \eval\Bigl( (G \div H)\bigl( \eval(H) \bigr) \Bigr) =
    \eval(G_{00}') \circ \eval(G_{01}') \circ \eval(G_1') 
    = \\ =
    \eval(G_{00}) \circ \phi(\same{\card{F_2}}) \otimes \eval(H) 
      \circ \eval(G_1)
    =
    \eval(G)
  \end{multline*}
  as claimed.
\end{proof}

\begin{example}[Evaluation map in $\mc{R}^{\bullet\times\bullet}$]
  \label{Ex:Matris-eval}
  A practical implementation of the evaluation map in some matrix 
  \PROP\ $\mc{R}^{\bullet\times\bullet}$ can be to view the edges as 
  conduits for messages in the form of row vectors.
  
  Concretely, to define the evaluation map \(\eval\colon 
  \Nw(\mc{R}^{\bullet\times\bullet}) \Fpil 
  \mc{R}^{\bullet\times\bullet}\) for a network \(G = 
  (V,E,h,g,t,s,D)\) of $\mc{R}^{\bullet\times\bullet}$, one first 
  constructs an assignment \(\lambda\colon E \Fpil \mc{R}^{\alpha(G)}\) 
  of vectors to edges, according to the following rules:
  \begin{itemize}
    \item
      If \(t(e) = 1\) then $\lambda(e)$ is the vector which has a $1$ in 
      the $s(e)$'th position and $0$ in all other positions, i.e., it 
      is the $s(e)$'th row in the identity matrix with side 
      $\alpha(G)$.
    \item
      If \(t(e) = v \neq 1\) then $\lambda(e)$ is the $s(e)$'th row of the 
      matrix
      $$
        D(v)
        \cdot
        \begin{bmatrix}
          \lambda(e_1) \\ \lambda(e_2) \\ \vdots \\ \lambda(e_{d^-(v)})
        \end{bmatrix}
      $$
      where each $e_i$ is the unique edge with \(h(e_i)=v\) and 
      \(g(e_i)=i\).
  \end{itemize}
  Then
  \[
    \eval(G) =
    \begin{bmatrix} 
      \lambda(e_1) \\ \lambda(e_2) \\ \vdots \\ \lambda(e_{\omega(G)}) 
    \end{bmatrix} \text{,}
  \]
  where each $e_i$ is the unique edge with \(h(e_i)=0\) and \(g(e_i) = i\).
  
  For example,
  \[
    \eval \left( \begin{mpgraphics*}{-93}
      PROPdiagram(0,0)
        same(3) circ
        elabels(btex $\scriptstyle 6$ etex, btex $\scriptstyle 7$ etex,
          btex $\scriptstyle 8$ etex) circ
        same(1) otimes box(2,2, btex $C$\strut etex) \circ
        cross(1,1) \otimes elabel(btex $\scriptstyle 5$ etex) circ
        lelabel(btex $\scriptstyle 4$ etex) otimes
          box(2,2, btex $B$\strut etex) \circ
        elabels(false, btex $\scriptstyle 3$ etex, false) circ
        box(2,2,btex $A$\strut etex) \otimes same(1) circ
        elabels(btex $\scriptstyle 0$ etex, btex $\scriptstyle 1$ etex, 
          btex $\scriptstyle 2$ etex)
      ;
    \end{mpgraphics*} \right)
    =
    \begin{pmatrix}
      b_{11}a_{21} & b_{11}a_{22} & b_{12} \\
      c_{11}a_{11} + c_{12}b_{21}a_{21} & 
        c_{11}a_{12} + c_{12}b_{21}a_{22} &
        c_{12}b_{22} \\
      c_{21}a_{11} + c_{22}b_{21}a_{21} & 
        c_{21}a_{12} + c_{22}b_{21}a_{22} &
        c_{22}b_{22}
    \end{pmatrix}
  \]
  for \(A = \begin{pmatrix} a_{11} & a_{12} \\ a_{21} & a_{22} 
  \end{pmatrix}\), \(B = \begin{pmatrix} b_{11} & b_{12} \\ b_{21} & 
  b_{22} \end{pmatrix}\), and \(C = \begin{pmatrix} c_{11} & c_{12} \\ 
  c_{21} & c_{22} \end{pmatrix}\), because
  \begin{align*}
    \begin{bmatrix}
      \lambda(0) \\ \lambda(1) \\ \lambda(2)
    \end{bmatrix}
    ={}& I =
    \begin{pmatrix}
      1 & 0 & 0 \\
      0 & 1 & 0 \\
      0 & 0 & 1
    \end{pmatrix}
    \text{,}\displaybreak[0]\\
    \begin{bmatrix}
      \lambda(4) \\ \lambda(3)
    \end{bmatrix}
    ={}&
    A \begin{bmatrix}
      \lambda(0) \\ \lambda(1)
    \end{bmatrix}
    =
    \begin{pmatrix}
      a_{11} & a_{12} \\ a_{21} & a_{22} 
    \end{pmatrix}
    \begin{pmatrix}
      1 & 0 & 0 \\
      0 & 1 & 0
    \end{pmatrix}
    =
    \begin{pmatrix}
      a_{11} & a_{12} & 0 \\ 
      a_{21} & a_{22} & 0
    \end{pmatrix}
    \text{,}\displaybreak[0]\\
    \begin{bmatrix}
      \lambda(6) \\ \lambda(5)
    \end{bmatrix}
    ={}&
    B \begin{bmatrix}
      \lambda(3) \\ \lambda(2)
    \end{bmatrix}
    =
    \begin{pmatrix} 
      b_{11} & b_{12} \\ b_{21} & b_{22}
    \end{pmatrix}
    \begin{pmatrix} 
      a_{21} & a_{22} & 0 \\
      0 & 0 & 1
    \end{pmatrix}
    =
    \begin{pmatrix} 
      b_{11}a_{21} & b_{11}a_{22} & b_{12} \\
      b_{21}a_{21} & b_{21}a_{22} & b_{22}
    \end{pmatrix}
    \text{,}\displaybreak[0]\\
    \begin{bmatrix}
      \lambda(7) \\ \lambda(8)
    \end{bmatrix}
    ={}&
    C \begin{bmatrix}
      \lambda(4) \\ \lambda(5)
    \end{bmatrix}
    \begin{aligned}[t]
      ={}&
      \begin{pmatrix} 
        c_{11} & c_{12} \\ c_{21} & c_{22}
      \end{pmatrix}
      \begin{pmatrix} 
        a_{11} & a_{12} & 0 \\
        b_{21}a_{21} & b_{21}a_{22} & b_{22}
      \end{pmatrix}
      = \\ ={}&
      \begin{pmatrix} 
        c_{11}a_{11} + c_{12}b_{21}a_{21} & 
          c_{11}a_{12} + c_{12}b_{21}a_{22} &
          c_{12}b_{22} \\
        c_{21}a_{11} + c_{22}b_{21}a_{21} & 
          c_{21}a_{12} + c_{22}b_{21}a_{22} &
          c_{22}b_{22}
      \end{pmatrix}
      \text{.}
    \end{aligned}
  \end{align*}
  
  Alternatively (and equivalently), one can start from the output side 
  and assign column vectors to edges, multiplying each $D(v)$ on the left 
  by the matrix built from the outgoing edge messages, and cutting 
  these intermediate matrix products up into columns rather than rows.
\end{example}

Thus having established networks as a notation for \PROP\ 
expressions\Ldash arguably, in view of 
Theorem~\ref{S:PROP3-definition}, even as the \emph{canonical} 
notation for \PROP\ expressions\Rdash the next step towards rewriting 
would be to sort out what constitutes a \emph{sub}expression. That 
will however be the subject of Section~\ref{Sec:Deluttryck}, since 
some settings for the answer bring forth technical considerations 
regarding boolean matrices that will turn out to be very important 
later. As the relevant material is not all that well-known, it can be 
found in the next section.

\section{Ordering matrices}
\label{Sec:Matrisordning}

Within optimisation, it is common to regard matrices (and even more 
so vectors) as ordered by element-wise comparisons: \(A \leqslant B\) 
iff \(A_{ij} \leqslant B_{ij}\) for all rows $i$ and columns $j$. 
Although this is less commonly seen in other branches of mathematics, 
its use within optimisation alone is important enough to earn the 
title of \emph{standard ordering} of matrices. Pullbacks of this 
standard ordering is also going to be our main source of nontrivial 
\PROP\ quasi-orders.

\subsection{Ordered semirings}

\begin{definition}
  A \DefOrd{partially ordered semiring} $(\mc{R},P)$ is a semiring 
  $\mc{R}$ on which is defined a partial order $P$ such that:
  \begin{enumerate}
    \item
      if \(x \leqslant y \pin{P}\) then \(x+z \leqslant y+z \pin{P}\) 
      for all \(z \in \mc{R}\);
    \item
      if \(x \leqslant y \pin{P}\) then there exists some \(z 
      \geqslant 0 \pin{P}\) such that \(y=x+z\);
    \item
      if \(x,y \geqslant 0 \pin{P}\) then \(xy \geqslant 0 \pin{P}\).
  \end{enumerate}
  The \DefOrd{semipositive cone} (sometimes \emDefOrd{nonnegative 
  cone}) in $\mc{R}$ with respect to $P$ is the set
  \begin{equation*}
    \mc{R}_{\geqslant 0} = \setOf{x \in \mc{R}}{x \geqslant 0 \pin{P}}
    \text{.}
  \end{equation*}
  An \(x \in \mc{R}\) is said to be \DefOrd{positive} if \(x 
  \geqslant 0 \pin{P}\) and multiplication by it preserves all strict 
  inequalities, i.e.,
  \begin{equation*}
    y < z \pin{P} \Ipil 
    \text{\(xy < xz \pin{P}\) and \(yx < zx \pin{P}\).}
  \end{equation*}
  The \DefOrd{positive cone} \index{+ sub@${}_+$ (positive cone)}$\mc{R}_+$ 
  in $\mc{R}$ with respect to $P$ is the set of all positive elements 
  of $\mc{R}$.
\end{definition}

Obvious examples of partially (totally, even) ordered semirings are 
the subsemirings of $\R$ under the standard order, e.g.~$\N$, $\Z$, 
and $\Q$. Another elementary semiring that will be of some interest 
is the \DefOrd{Boolean semiring} \index{B@$\B$}\(\B = \{0,1\}\), which 
has `\textsc{or}' as addition, `\textsc{and}' as multiplication, and 
\(0 < 1\). Composite examples of partially ordered semirings are 
provided by matrix rings and semirings.

\begin{construction} \label{Konstr:Matrisordning}
  Let an associative unital partially ordered semiring $(\mc{R},P)$ be 
  given. The \DefOrd{standard order} on $\mc{R}^{\bullet \times \bullet}$ 
  is the $Q$ defined for \(A,B \in \mc{R}^{m \times n}\) by
  \begin{equation}
    A \leqslant B \pin{Q} \Epil
    A_{ij} \leqslant B_{ij} \pin{P}
    \text{ for all \(i \in [m]\), \(j \in [n]\);}
  \end{equation}
  matrices of different shapes are unrelated. This is an 
  $\N^2$-graded partial order which satisfies:
  \begin{enumerate}
    \item \label{Item:Semipositiv}
      \(0 \leqslant A \pin{Q(m,n)}\) if and only if \(A \in 
      (\mc{R}_{\geqslant 0})^{m \times n}\).
    \item \label{Item:Translation}
      If \(A,B,C \in \mc{R}^{m\times n}\) and \(A \leqslant B 
      \pin{Q}\) then \(A+C \leqslant B+C \pin{Q}\).
    \item \label{Item:Reduktion}
      If \(A \leqslant B \pin{Q(m,n)}\) then \(B=A+C\) for some \(C 
      \in (\mc{R}_{\geqslant 0})^{m \times n}\).
    \item \label{Item:Multiplikation}
      If \(A \in (\mc{R}_{\geqslant 0})^{l \times m}\) and
      \(B \in (\mc{R}_{\geqslant 0})^{m \times n}\) then 
      \(AB \in (\mc{R}_{\geqslant 0})^{l \times n}\).
    \item \label{Item:R*,*-tensor}
      The following are equivalent for all \(A,B,C \in 
      \mc{R}^{\bullet\times\bullet}\):
      \begin{enumerate}
        \item \(A \leqslant B \pin{Q}\),
        \item \(A \otimes C \leqslant B \otimes C \pin{Q}\),
        \item \(C \otimes A \leqslant C \otimes B \pin{Q}\).
      \end{enumerate}
    \item \label{Item:R*,*-grundad}
      If $P$ is well-founded then $Q$ is well-founded.
  \end{enumerate}
  Hence $\bigl( \mc{R}^{n \times n}, Q(n,n) \bigr)$ will be 
  a partially ordered semiring for every \(n \geqslant 1\). 
  Moreover the restriction of $Q$ to $(\mc{R}_{\geqslant 0})^
  {\bullet\times\bullet}$ is a \PROP\ partial order.
\end{construction}
\begin{proof}
  It is clear that $Q(m,n)$ is reflexive. That it is transitive follows 
  from
  \begin{multline*}
    A \leqslant B \leqslant C \pin{Q(m,n)} \Epil \\ \Epil
    A_{ij} \leqslant B_{ij} \leqslant C_{ij} \pin{P}
      \text{ for all \((i,j) \in [m] \times [n]\)} \Ipil \\ \Ipil
    A_{ij} \leqslant C_{ij} \pin{P} 
      \text{ for all \((i,j) \in [m] \times [n]\)} \Epil
    A \leqslant C \pin{Q(m,n)}
  \end{multline*}
  and that it is antisymmetric from
  \begin{multline*}
    A \leqslant B \leqslant A \pin{Q(m,n)} \Epil \\ \Epil
    A_{ij} \leqslant B_{ij} \leqslant A_{ij} \pin{P}
      \text{ for all \((i,j) \in [m] \times [n]\)} \Ipil \\ \Ipil
    A_{ij} = B_{ij} \text{ for all \((i,j) \in [m] \times [n]\)} 
      \Epil
    A = B \text{.}
  \end{multline*}
  This has shown that $Q$ is an $\N^2$-graded partial order.
  
  Claim~\ref{Item:Semipositiv} is immediate from the definitions of 
  $Q$ and $\mc{R}_{\geqslant 0}$. Claim~\ref{Item:Translation} 
  (translation) is similarly the observation that
  \begin{multline*}
    A \leqslant B \pin{Q(m,n)} \Epil
    A_{ij} \leqslant B_{ij} \pin{P} 
      \text{ for all \((i,j) \in [m] \times [n]\)} 
    \Ipil \\ \Ipil
    A_{ij} + C_{ij} \leqslant B_{ij} + C_{ij} \pin{P} 
      \text{ for all \((i,j) \in [m] \times [n]\)} 
      \Epil \\ \Epil
    A+C \leqslant B+C \pin{Q(m,n)}
  \end{multline*}
  and if \(A \leqslant B \pin{Q(m,n)}\) then choosing \(C_{ij} \in 
  \mc{R}_{\geqslant 0}\) such that \(B_{ij} = A_{ij}+C_{ij}\) for all 
  \((i,j) \in [m] \times [n]\) will produce \(C \geqslant 0 \pin{Q}\) 
  such that \(B = A+C\), thus verifying claim~\ref{Item:Reduktion}. 
  Claim~\ref{Item:R*,*-tensor} is equally trivial, as those component 
  comparisons in \(A \otimes C \leqslant B \otimes C \pin{Q}\) and 
  \(C \otimes A \leqslant C \otimes B \pin{Q}\) that are not those of 
  \(A \leqslant B \pin{Q}\) are all of the form \(0 \leqslant 0 
  \pin{P}\) or \(C_{i,j} \leqslant C_{i,j} \pin{P}\).
  
  For claim~\ref{Item:Multiplikation}, let \(A \in 
  (\mc{R}_{\geqslant 0})^{l \times m}\) and \(B \in 
  (\mc{R}_{\geqslant 0})^{m \times n}\) be given. About \(C=AB\) one 
  finds for all \(i \in [l]\) and \(j \in [n]\) that
  \begin{equation*}
    C_{ij} = \sum_{k=1}^m A_{ik} B_{kj} \geqslant
    \sum_{k=1}^{m-1} A_{ik} B_{kj} \geqslant \dotsb \geqslant
    A_{i1} B_{1j} \geqslant 0 \pin{P}
  \end{equation*}
  since \(A_{ik} B_{kj} \geqslant 0 \pin{P}\) and hence 
  \(\sum_{r=1}^k A_{ir} B_{rj} \geqslant \sum_{r=1}^{k-1} A_{ir} 
  B_{rj} \pin{P}\) for all \(k \in [m]\).
  
  Claim~\ref{Item:R*,*-grundad} can be shown by considering an 
  infinite sequence \(A_1 \geqslant A_2 \geqslant \dotsb \pin{Q}\). 
  If \(A_k > A_{k+1} \pin{Q}\), then there must be some position 
  $(i,j)$ such that \((A_k)_{ij} > (A_{k+1})_{ij} \pin{P}\). Because 
  $P$ is well-founded, there can only be a finite number of indices 
  $k$ in which this happens in position $(i,j)$. Since all matrices 
  in this sequence have $\omega(A_1)$ rows and $\alpha(A_1)$ columns, 
  there is a bounded number of matrix positions at which such 
  descents can happen. Hence the number of strict inequalities in the 
  sequence must be finite, and thus $Q$ is well-founded.
  
  Finally, the conclusion that $Q$ is a \PROP\ partial order on 
  $(\mc{R}_{\geqslant 0})^{\bullet\times\bullet}$ may need some 
  elaboration. If \(0 \leqslant A \leqslant B \pin{Q(l,m)}\) then 
  there is some \(B' \geqslant 0 \pin{Q(l,m)}\) such that \(B = 
  A+B'\) and hence \(CAD \leqslant CAD + CB'D = CBD \pin{Q(k,n)}\) 
  for all \(C \in (\mc{R}_{\geqslant 0})^{k \times l}\) and 
  \(D \in (\mc{R}_{\geqslant 0})^{m \times n}\).
\end{proof}

The next step will be to exhibit examples of sub-\PROPs\ of 
$\mc{R}^{\bullet\times\bullet}$ where the standard order is strict.

By the first two axioms, \(x \leqslant y\) if and only if there exists 
some \(z \in \mc{R}_{\geqslant 0}\) such that \(y=x+z\); hence a 
semiring partial order is fully specified by its semipositive cone. 
An alternative set of axioms for a subset $C$ of $\mc{R}$ to be the 
semipositive cone of a semiring partial order are:
\begin{enumerate}
  \item
    \(0 \in C\).
  \item
    $C$ is closed under addition.
  \item
    $C$ is closed under multiplication.
  \item
    If \(x = y+u\) and \(y = x+v\) for some \(u,v \in C\) then 
    \(x=y\).
\end{enumerate}
The last of these is the claim that $P$ is antisymmetric. For a 
partially ordered \emph{ring} $\mc{R}$ it can be simplified to the 
claim that $\{u,-u\} \cap C$ has at most one element for any \(u \in 
\mc{R}\), but semirings can have \(x+u+v=x\) even if \(u+v \neq 0\). 
That the semipositive cone is closed under addition follows from 
transitivity and translation of the partial order, since 
\(x,y \geqslant 0 \pin{P}\) implies \(x+y \geqslant 0+y = y 
\geqslant 0 \pin{P}\). That \(0 \geqslant 0\) implies that $0$ is 
always semipositive.


The term `positive cone' suggests that positive elements also form a 
cone, and in many case this is indeed so, but it depends on 
properties of the semiring.

\begin{lemma} \label{L:Positivitet}
  Let a partially ordered semiring $(\mc{R},P)$ be given.
  \begin{enumerate}
    \item
      If addition is cancellative then strictness of inequalities 
      is preserved under translation and \(x+y \in \mc{R}_+\) for 
      any \(x \in \mc{R}_+\) and \(y \in \mc{R}_{\geqslant 0}\). 
    \item
      If multiplication is associative then \(xy \in \mc{R}_+\) for 
      all \(x,y \in \mc{R}_+\).
  \end{enumerate}
\end{lemma}
\begin{proof}
  First assume addition in $\mc{R}$ is cancellative, i.e., 
  \(x+z=y+z\) implies \(x=y\). If \(x < y \pin{P}\) then \(x 
  \leqslant y \pin{P}\) and hence \(x+z \leqslant y+z \pin{P}\) by 
  translation invariance of $P$. Furthermore \(x \not\geqslant y 
  \pin{P}\), hence \(x \neq y\), and thus \(x+z \neq y+z\), from 
  which follows that \(x+z \not\geqslant y+z \pin{P}\). Therefore 
  \(x+z < y+z \pin{P}\), i.e., translation by some arbitrary \(z \in 
  \mc{R}\) preserves strict inequalities. To see that \(x+y \in 
  \mc{R}_+\) when \(x \in \mc{R}_+\) and \(y \in \mc{R}_{\geqslant 
  0}\), let \(w,z \in \mc{R}\) such that \(w<z \pin{P}\) be 
  arbitrary. Then
  \begin{gather*}
    w (x+y) = wx + wy < zx + wy \leqslant zx + zy = z(x+y) \pin{P}
      \text{\rlap{,}}\\
    (x+y)w = xw + yw < xz + yw \leqslant xz + yz = (x+y)z \pin{P}
  \end{gather*}
  where invariance of strict inequalities is needed e.g.~to conclude 
  from \(wx < zx\) (by the \(x \in \mc{R}_+\) hyphothesis) that 
  \(wx + wy < zx + wy\).
  
  Second assume multiplication in $\mc{R}$ is associative. Let \(x,y 
  \in \mc{R}_+\) and \(w,z \in \mc{R}\) such that \(w < z \pin{P}\) 
  be arbitrary. Then \(wx < zx \pin{P}\) and hence \(w(xy) = (wx)y < 
  (zx)y = z(xy) \pin{P}\), similarly \(yw < yz \pin{P}\) and hence 
  \((xy)w = x(yw) < x(yz) = (xy)z \pin{P}\), which has verified that 
  \(xy \in \mc{R}_+\).
\end{proof}

An example of $\mc{R}_+$ failing to be closed under addition when 
addition is not cancellative can be found in \(\mc{R} = 
\B^{2 \times 2}\); the matrices \(\left( \begin{smallmatrix} 1&0 \\ 
0&1 \end{smallmatrix} \right)\) and \(\left( \begin{smallmatrix} 0&1 
\\ 1&0 \end{smallmatrix} \right)\) are both positive (on account of 
being invertible) but their sum \(\left( \begin{smallmatrix} 1&1 
\\ 1&1 \end{smallmatrix} \right)\) is not.

\begin{lemma} \label{L:Positiv}
  Let $(\mc{R},P)$ be an associative unital partially ordered semiring 
  with cancellative addition. Let $Q$ be the standard order on 
  $\mc{R}^{\bullet\times\bullet}$. Let $\mc{Q}$ be the set of those 
  elements of $(\mc{R}_{\geqslant 0})^{\bullet\times\bullet}$ which 
  have at least one element on $\mc{R}_+$ in each row and column. 
  Then the following holds:
  \begin{enumerate}
    \item \label{Item1:Positiv}
      If \(A < B \pin{Q(l,m)}\), and \(C \in 
      (\mc{R}_{\geqslant 0})^{m \times n}\) is such that every row 
      contains at least one element of $\mc{R}_+$, then \(AC < BC 
      \pin{Q(l,n)}\).
    \item \label{Item2:Positiv}
      If \(A < B \pin{Q(m,l)}\), and \(C \in 
      (\mc{R}_{\geqslant 0})^{n \times m}\) is such that every column 
      contains at least one element of $\mc{R}_+$, then \(CA < CB 
      \pin{Q(n,l)}\).
    \item \label{Item3:Positiv}
      If \(0 < 1 \pin{P}\) then $\mc{Q}$ is a sub-\PROP\ of 
      $(\mc{R}_{\geqslant 0})^{\bullet\times\bullet}$, and the 
      restriction of $Q$ to $\mc{Q}$ is strict.
    \item \label{Item4:Positiv}
      If \(0 < 1 \pin{P}\) then all elements of $\mc{Q}(n,n)$ are in 
      the positive cone of the partially ordered semiring 
      $\bigl( \mc{R}^{n \times n}, Q(n,n) \bigr)$.
  \end{enumerate}
\end{lemma}
\begin{remark}
  Since unitality is generally considered to imply \(1 \neq 0\), the 
  \(0 < 1 \pin{P}\) condition is equivalent to \(0 \leqslant 1 
  \pin{P}\). The case that primarily needs to be ruled out is that 
  $0$ and $1$ are unrelated in $P$, since that would put the image of 
  $\phi$ outside $\mc{Q}$. 
\end{remark}
\begin{proof}
  For claim~\ref{Item1:Positiv}, one may first observe that \(B = A + 
  A'\) for some \(A' \geqslant 0 \pin{Q(l,m)}\), and hence \(AC 
  \leqslant AC + A'C = BC \pin{Q(l,n)}\). Second, there must be some 
  position \((i,j) \in [l] \times [m]\) such that \(A_{ij} < B_{ij} 
  \pin{P}\). Let \(k \in [n]\) be such that \(C_{jk} \in \mc{R}_+\). 
  Then 
  \begin{math}
    (A'C)_{ik} = 
    \sum_{r=1}^m A'_{ir} C_{rk} \geqslant 
    A'_{ij} C_{jk} > 0 \pin{P}
  \end{math},
  and hence \((AC)_{ik} < (BC)_{ik} \pin{P}\) by cancellativity of 
  addition. 
  The proof of claim~\ref{Item2:Positiv} is completely analogous.
  
  These two claims, together with those in 
  Construction~\ref{Konstr:Matrisordning}, demonstrate that $Q$ on 
  $\mc{Q}$ is strict with respect to composition and tensor product. 
  Since every row of $A \otimes B$ for \(A, B \in \mc{Q}\) contains a 
  row of $A$ or $B$ in its entirety, and similarly every column of 
  $A \otimes B$ contains a column of $A$ or $B$ in its entirety, it 
  follows that $\mc{Q}$ is closed under the tensor product. To see 
  that it is closed under composition, one may consider some 
  arbitrary \(A \in \mc{Q}(l,m)\) and \(B \in \mc{Q}(m,n)\). For any 
  row index \(i \in [l]\), there is some \(j \in [m]\) such that 
  \(A_{ij} \in \mc{R}_+\) and some \(k \in [n]\) such that \(B_{jk} 
  \in \mc{R}_+\). Hence \((AB)_{ik} = \sum_{r=1}^m A_{ir} B_{rk} 
  \geqslant A_{ij} B_{jk} > 0 \pin{P}\), and thus row $i$ of $AB$ 
  contains a positive element in column $k$. Similarly to show that 
  column $k$ of $AB$ contains a positive element, one first chooses 
  \(j \in [m]\) such that \(B_{jk} \in \mc{R}_+\) and then \(i \in 
  [l]\) such that \(A_{ij} \in \mc{R}_+\). Together they demonstrate 
  that \(A \circ B \in \mc{Q}\).
  
  What remains for $\mc{Q}$ to be a \PROP\ is that it contains  
  $\phi(\sigma)$ for all permutations $\sigma$. These matrices have a 
  $1$ in every row and column, which by \(0 < 1 \pin{P}\) is 
  positive; the special case of \(0 = \phi(\same{0})\) has neither 
  row nor columns, so it gets by without any $1$ elements. This 
  completes the proof of claim~\ref{Item3:Positiv}. 
  Claim~\ref{Item4:Positiv}, finally, is merely the special case 
  \(l=m=n\) of the first two claims.
\end{proof}

\begin{corollary} \label{Kor:BiaffinOrdning}
  Let $(\mc{R},P)$ be an associative unital semiring with 
  cancellative addition, partially ordered so that \(0 < 1 \pin{P}\). 
  Note that the restriction $Q$ of the standard order on 
  $\mc{R}^{\bullet\times\bullet}$ to $\Baff(\mc{R}_{\geqslant 0})$ is 
  a \PROP\ partial order. Let $\mc{B}$ be the set of those elements of 
  $\Baff(\mc{R}_{\geqslant 0})$ which have at least one element of 
  $\mc{R}_+$ in each row and column. Then $\mc{B}$ is a \PROP, and 
  the restriction of $Q$ to $\mc{B}$ is strict.
\end{corollary}
\begin{remark}
  Since an element of $\Baff(\mc{R}_{\geqslant 0})$ has the form
  \begin{equation*}
    \begin{bmatrix}
      1& d& \vek{c}^\mathrm{T}\\
      0& 1& 0\\
      0& \vek{b}& A
    \end{bmatrix}
    \text{,}
  \end{equation*}
  the first two columns and first two rows always contain some \(1 \in 
  \mc{R}_+\). If the $A$ part is chosen as an element of the $\mc{Q}$ 
  of Lemma~\ref{L:Positiv} then the whole matrix becomes an element of 
  $\mc{B}$, but one may also meet the condition about a positive 
  element in each row and column by putting them in the $\vek{b}$ and 
  $\vek{c}$ parts.
\end{remark}

\subsection{Nilpotence}

A second application of the standard ordering of matrices exists 
in the formalisation of the ``may depend on'' arrow $\canmake$ seen 
in \eqref{Eq:Regelfamiljer}. The final details of this will have to 
wait until Section~\ref{Sec:Omskrivning}, but closely related 
technicalities will appear repeatedly in the next couple of sections. 
What they are all about is ultimately the prevention of cycles in the 
networks being considered, and keeping track of things to that end 
requires some additional operations on matrices.

\begin{definition}
  The \DefOrd{Kleene star} $A^*$\index{* sup@${}^*$!Kleene star} 
  of a square matrix $A$ is defined to 
  be $\sum_{k=0}^\infty A^k$ if that series converges. Similarly 
  the \DefOrd{Kleene plus} $A^+$\index{+ sup@${}^+$ (Kleene plus)} 
  of $A$ is defined to be 
  $\sum_{k=1}^\infty A^k$ if that series converges.
\end{definition}

The Kleene star and plus are defined for all \(A \in 
\B^{n \times n}\), since all series in $\B$ converges; either all 
terms are $0$, and then the sum is $0$, or at least one term is $1$, 
and then the sum is $1$. Over other semirings, existence of $A^*$ in 
general depends on the topology, but one case in which it does not is 
that $A$ is nilpotent.

\begin{lemma} \label{L:Nilpotens}
  For any \(A \in \B^{n \times n}\), the following are equivalent:
  \begin{enumerate}
    \item \(A^n = 0\).
    \item $A$ is nilpotent.
    \item The main diagonal of $A^+$ is zero.
  \end{enumerate}
\end{lemma}
\begin{proof}
  It is actually easier to reason about the negations of the given 
  claims. Clearly, if $A$ is not nilpotent then \(A^n \neq 0\).
  
  Next assume \(A^n \neq 0\). Then there is some nonzero position 
  $(i_0,i_n)$ in $A^n$, and hence there must be some sequence 
  \(i_0, \dotsc, i_n \in [n]\) such that \(A_{i_{k-1},i_k} = 1\) for 
  all \(k=1,\dotsc,n\). By the pidgeonhole principle, there must be 
  \(0\leqslant k < l \leqslant n\) such that \(i_k = i_l\). Hence the 
  $(i_k,i_k)$ position of $A^{l-k}$ is $1$ and it follows that the 
  same is true for $A^+$; not all of its diagonal elements are $0$.
  
  If the $(i,i)$ position of $A^+$ is $1$ then that position is $1$ 
  also in $(A^+)^2$, and by extension in any $(A^+)^m$ for \(m 
  \geqslant 1\). Since \(\left( \sum_{k=1}^\infty A^k \right)^m = 
  \sum_{k=m}^\infty A^k = A^m \sum_{k=0} A^k\), it follows that \(A^m 
  \neq 0\) if the $(i,i)$ position of $A^+$ is nonzero. Hence $A$ is 
  not nilpotent.
\end{proof}

The relevance of boolean matrices here is that many matrices that are 
nilpotent are so because of their shape: the positions of their 
nonzero elements are such that one can predict based on that alone 
that a certain power is $0$. Via the encoding that $1$ means 
`possibly nonzero' and $0$ means `definitely zero', boolean matrices 
provide a convenient means of keeping tack of such matrix shapes, as 
an addition or multiplication in $\mc{R}^{n \times n}$ under that 
encoding corresponds to the same operation in $\B^{n \times n}$. 
Slightly more formally, one can consider the sets of matrices that 
have the shape of a certain pattern of zeroes and ones.

\begin{definition}
  Let $\mc{R}$ be an associative unital semiring. The \DefOrd{matrix 
  dependency filtration} of $\mc{R}^{\bullet\times\bullet}$ is the 
  family $\{F_q\}_{q \in \B^{\bullet\times\bullet}}$ of subsets of 
  $\mc{R}^{\bullet\times\bullet}$ defined by
  \begin{equation}
    F_q = \setOf{ A \in \mc{R}^{\omega(q) \times \alpha(q)} }{ 
      \text{\(A_{ij}=0\) for all $(i,j)$ such that \(q_{ij}=0\)}
    }
  \end{equation}
  for all \(q \in \B^{\bullet\times\bullet}\).
\end{definition}

The matrix dependency filtration has the important property that if 
\(A \in F_q\) and \(B \in F_r\), then \(AB \in F_{qr}\), since if 
(say) position $(i,k)$ in $AB$ is nonzero then there must be some $j$ 
such that \(A_{i,j} \neq 0\) and \(B_{j,k} \neq 0\), meaning 
\(q_{i,j}=1\) and \(r_{j,k}=1\), and thus position $(i,k)$ in $qr$ is 
$1$ too.

\begin{lemma} \label{L:Semiringstjarna}
  Let $\mc{R}$ be an associative and unital semiring. Let 
  $\{F_q\}_{q \in \B^{\bullet\times\bullet}}$ be the matrix 
  dependency filtration of $\mc{R}^{\bullet\times\bullet}$. If \(q 
  \in \B^{n \times n}\) is nilpotent then every \(A \in F_q\) is 
  nilpotent and $A^*$ exists. If $\mc{R}$ is a ring, then $A^*$ is 
  the multiplicative inverse of $I-A$.
  
  If \(q \in \B^{m \times n}\) and \(r \in \B^{n \times m}\) are such 
  that $qr$ is nilpotent then $rq$ is nilpotent, and \((AB)^* A = A 
  (BA)^*\) for all \(A \in F_q\) and \(B \in F_r\).
  
  For any nilpotent \(q \in \B^{n \times n}\) and \(A,B \in F_q\),
  \begin{equation}
    (A +\nobreak B)^* = (A^* B)^* A^* = A^* (B A^*)^* \text{.}
  \end{equation}
\end{lemma}
\begin{proof}
  Suppose \(q^m = 0\). Then \(A^m \in F_{q^m} = \{0\}\), 
  demonstrating that $A$ is nilpotent. \(A^* = \sum_{k=0}^{m-1} A^k\) 
  which clearly exists, and \((I -\nobreak A) \sum_{k=0}^{m-1} A^k = 
  I - A^m = I\).
  
  If $qr$ is nilpotent for \(q \in \B^{m \times n}\) and \(r \in 
  \B^{n \times m}\) then \((qr)^m = 0\) by Lemma~\ref{L:Nilpotens} 
  and hence \((rq)^{m+1} = r (qr)^m q = 0\), demonstrating the 
  nilpotency of $rq$. In $\mc{R}^{\bullet\times\bullet}$, one finds 
  \((AB)^* A = \sum_{k=0}^\infty (AB)^k A = \sum_{k=0}^\infty A 
  (BA)^k = A (BA)^*\).
  
  Considering instead $(A +\nobreak B)^*$ for \(A,B \in F_q\) where 
  \(q \in \B^{n \times n}\) is nilpotent, one has of course
  \begin{equation*}
    (A + B)^* = 
    \sum_{k=0}^\infty (A + B)^k = 
    \sum_{k=0}^\infty \, \sum_{W: [k] \Fpil \{A,B\}} \,
      \prod_{i=1}^k W(i)
    \text{.}
  \end{equation*}
  Collecting sequences of adjacent $A$ factors, taking $l$ as the 
  number of $B$ factors and $i_j$ as the number of $A$s between the 
  $j$'th and $j+1$'th $B$ factor, this last sum can be rewritten as
  \begin{multline*}
    \sum_{k=0}^\infty \sum_{l=0}^k 
      \sum_{\substack{i_0,\dotsc,i_l \geqslant 0 \\ 
        i_0+\dotsb+i_l = k-l}} A^{i_0} 
      \prod_{j=1}^l B A^{i_j}
    =
    \sum_{0 \leqslant l \leqslant k}
      \sum_{\substack{i_0,\dotsc,i_l \geqslant 0 \\ 
        i_0+\dotsb+i_l = k-l}} A^{i_0} 
      \prod_{j=1}^l B A^{i_j}
    = \\ =
    \sum_{0 \leqslant l} \sum_{i_0,\dotsc,i_l \geqslant 0} 
      A^{i_0} \prod_{j=1}^l B A^{i_j}
    =
    \sum_{l=0}^\infty A^* \prod_{j=1}^l B A^*
    =
    A^* (B A^*)^* \text{,}
  \end{multline*}
  with the other variant following from the previous claim in the 
  lemma.
\end{proof}

\begin{lemma} \label{L:Matrisnilpotens}
  For any \(A \in \B^{m \times m}\), \(B \in \B^{m \times n}\), \(C 
  \in \B^{n \times m}\), and \(D \in \B^{n \times n}\),
  the following are equivalent:
  \begin{enumerate}
    \item \label{Item1:Matrisnilpotens}
      \(\left[ \begin{smallmatrix} A & B \\ C & D \end{smallmatrix} 
      \right]\) is nilpotent.
    \item \label{Item2:Matrisnilpotens}
      $A$, $D$, and $A^* B D^* C$ are nilpotent.
    \item \label{Item3:Matrisnilpotens}
      $A$, $D$, and $D^* C A^* B$ are nilpotent.
    \item \label{Item4:Matrisnilpotens}
      $A$ and $D + C A^* B$ are nilpotent.
    \item \label{Item5:Matrisnilpotens}
      $D$ and $A + B D^* C$ are nilpotent.
  \end{enumerate}
\end{lemma}
\begin{proof}
  Equivalence of \ref{Item2:Matrisnilpotens} and 
  \ref{Item3:Matrisnilpotens} is immediate from 
  Lemma~\ref{L:Semiringstjarna}.
  
  Now assume~\ref{Item1:Matrisnilpotens} and let \(q = \left[ 
  \begin{smallmatrix} A & B \\ C & D \end{smallmatrix} \right]\).
  For the nilpotency of $A$ and $D$, one may observe that
  \begin{equation*}
    0 = 
    \begin{bmatrix} 
      A & B \\ C & D 
    \end{bmatrix}^{m+n}
    \geqslant
    \begin{bmatrix} 
      A & 0 \\ 0 & D 
    \end{bmatrix}^{m+n}
    =
    \begin{bmatrix} 
      A^{m+n} & 0 \\ 0 & D^{m+n}
    \end{bmatrix}
    \text{.}
  \end{equation*}
  Let \(N = \bigl\lceil (m +\nobreak n)/2 \bigr\rceil\). By 
  Lemma~\ref{L:Nilpotens}, \(q^{2N} = 0\) and hence the nilpotency of 
  $A^* B D^* C$ and $D^* C A^* B$ 
  follows from
  \begin{multline*}
    0 = 
    (q^*)^{2N} \cdot 0 =
    (q^*)^{2N} q^{2N} =
    (q^* q)^{2N} \geqslant
    \left( 
      \begin{bmatrix} A & 0 \\ 0 & D \end{bmatrix}^*
      \begin{bmatrix} 0 & B \\ C & 0 \end{bmatrix}
    \right)^{2N} = \\ =
    \begin{bmatrix} 
      0 & A^* B \\ D^* C & 0 
    \end{bmatrix}^{2N} = 
    \begin{bmatrix} 
      A^* B D^* C & 0 \\
      0 & D^* C A^* B
    \end{bmatrix}^N 
    \text{.}
  \end{multline*}
  It has thus been shown that \ref{Item1:Matrisnilpotens} implies 
  \ref{Item2:Matrisnilpotens} and \ref{Item3:Matrisnilpotens}.
  
  Conversely assuming that $A$, $D$, $A^* B D^* C$, and $D^* C A^* B$ 
  are nilpotent, one may let \(M = \max\{m,n\}\) and observe that the 
  $M$th power of any of these four expressions is $0$. Next consider
  \begin{multline*}
    q^{2M^2} =
    \left( 
      \begin{bmatrix} A & 0 \\ 0 & D \end{bmatrix} +
      \begin{bmatrix} 0 & B \\ C & 0 \end{bmatrix}
    \right)^{2M^2} \leqslant
    \left( 
      \begin{bmatrix} A & 0 \\ 0 & D \end{bmatrix}^*
      \begin{bmatrix} 0 & B \\ C & 0 \end{bmatrix}
    \right)^{2M} q^* = \\ =
    \begin{bmatrix}
      0 & A^* B \\ D^* C & 0
    \end{bmatrix}^{2M} q^* =
    \begin{bmatrix}
      A^* B D^* C & 0 \\ 
      0 & D^* C A^* B
    \end{bmatrix}^M q^* = \\ =
    \begin{bmatrix}
      (A^* B D^* C)^M & 0 \\ 
      0 & (D^* C A^* B)^M
    \end{bmatrix} q^* =
    0
  \end{multline*}
  where the first line inequality holds because any term in the 
  expansion of 
  \begin{math}
    \left( 
      \left[ \begin{smallmatrix} 
        A & 0 \\ 0 & D 
      \end{smallmatrix} \right] + 
      \left[ \begin{smallmatrix} 
        0 & B \\ C & 0 
      \end{smallmatrix} \right]
    \right)^{2M^2}
  \end{math}
  which has more than $M$ consequtive \( \left[ \begin{smallmatrix} 
  A & 0 \\ 0 & D \end{smallmatrix} \right] \) factors is 
  $0$ since \(A^M = D^M = 0\); the remaining terms have at 
  least $2M$ factors \( \left[ \begin{smallmatrix} 0 & B \\ 
  C & 0 \end{smallmatrix} \right] \), and the star factors 
  before and after cover all possibilities for what else there may be 
  in one of these terms.
  This has shown that \ref{Item2:Matrisnilpotens} and 
  \ref{Item3:Matrisnilpotens} imply \ref{Item1:Matrisnilpotens}.
  
  Next consider the matter of whether \ref{Item2:Matrisnilpotens} is 
  implied by \ref{Item5:Matrisnilpotens}. The nilpotence of $D$ is 
  common to these. For $A$, it is clear that \(A^m \leqslant 
  (A +\nobreak B D^* C)^m = 0\). For $A^* B D^* C$, it follows from 
  Lemma~\ref{L:Semiringstjarna} that 
  \begin{multline*}
    (A^* B D^* C)^+ =
    A^* B D^* C 
      (A^* B D^* C)^* =
    A^* (B D^* C A^*)^* 
      B D^* C = \\ =
    (A + B D^* C)^* B D^* C
      \leqslant
    (A + B D^* C)^+
  \end{multline*}
  and thus the left hand side cannot have any nonzero diagonal 
  elements since the right hand side does not have any. It is 
  similarly shown that \ref{Item4:Matrisnilpotens} 
  implies~\ref{Item3:Matrisnilpotens}.
  
  For the converse, the nilpotence of $D$ is again not an issue, so 
  the task is to show that $A+B D^* C$ is nilpotent. Here it is 
  convenient to assume~\ref{Item1:Matrisnilpotens}, since
  \begin{multline*}
    \begin{bmatrix}
      (A+B D^* C)^+ & 0 \\ 0 & 0
    \end{bmatrix} =
    \begin{bmatrix}
      A+B D^* C & 0 \\ 0 & 0
    \end{bmatrix}^+ = \\ =
    \left(
      \begin{bmatrix} A & 0 \\ 0 & 0 \end{bmatrix} +
      \begin{bmatrix} 0 & B \\ 0 & 0 \end{bmatrix} 
      \begin{bmatrix} 0 & 0 \\ 0 & D \end{bmatrix}^*
      \begin{bmatrix} 0 & 0 \\ C & 0 \end{bmatrix} 
    \right)^+ \leqslant
    (q + q q^* q)^+ = q^+
  \end{multline*}
  and the all zero diagonal of $q^+$ thus requires $(A +\nobreak 
  BD^*C)^+$ to have an all zero diagonal as well. This has thus 
  shown~\ref{Item5:Matrisnilpotens}, and again the argument 
  for~\ref{Item4:Matrisnilpotens} is completely analogous.
\end{proof}

\begin{lemma} \label{L:Matrisstjarna}
  Let $\mc{R}$ be an associative and unital semiring. Let 
  $\{F_q\}_{q \in \B^{\bullet\times\bullet}}$ be the matrix 
  dependency filtration of $\mc{R}^{\bullet\times\bullet}$. 
  Let \(q = \left[ \begin{smallmatrix} q_{11} & q_{12} \\ q_{21} & 
  q_{22} \end{smallmatrix} \right] \in \B^{(m+n)\times(m+n)}\) where 
  \(q_{11} \in \B^{m \times m}\) and \(q_{22} \in \B^{n \times n}\). 
  If $q$ is nilpotent then
  \begin{equation}
    \begin{bmatrix}
      A_{11} & A_{12} \\ A_{21} & A_{22}
    \end{bmatrix}^* =
    \begin{bmatrix}
      (A_{11}^* A_{12} A_{22}^* A_{21})^* A_{11}^* & 
        A_{11}^* A_{12} (A_{22}^* A_{21} A_{11}^* A_{12})^* A_{22}^* \\
      A_{22}^* A_{21} (A_{11}^* A_{12} A_{22}^* A_{21})^* A_{11}^* &
        (A_{22}^* A_{21} A_{11}^* A_{12})^* A_{22}^*
    \end{bmatrix}
  \end{equation}
  for all \(A_{11} \in F_{q_{11}}\), \(A_{12} \in F_{q_{12}}\), 
  \(A_{21} \in F_{q_{21}}\), and \(A_{22} \in F_{q_{22}}\).
\end{lemma}
\begin{proof}
  That $q$ is nilpotent ensures that the left hand side exists, 
  whereas the conditions for Lemma~\ref{L:Semiringstjarna} to imply 
  thay all stars in the right hand side exist are that $q_{11}$, 
  $q_{22}$, $q_{11}^* q_{12} q_{22}^* q_{21}$, and $q_{22}^* q_{21} 
  q_{11}^* q_{12}$ are nilpotent, but the combination of those is 
  equivalent to the nilpotency of $q$ by 
  Lemma~\ref{L:Matrisnilpotens}.
  Under these conditions, it follows from other claims in 
  Lemma~\ref{L:Semiringstjarna} that
  \begin{align*}
    \begin{bmatrix}
      A_{11} & A_{12} \\ A_{21} & A_{22}
    \end{bmatrix}^* ={}&
    \left(
      \begin{bmatrix} A_{11} & 0 \\ 0 & A_{22} \end{bmatrix} +
      \begin{bmatrix} 0 & A_{12} \\ A_{21} & 0 \end{bmatrix}
    \right)^* 
    = \\ ={}&
    \left(
      \begin{bmatrix} A_{11} & 0 \\ 0 & A_{22} \end{bmatrix}^*
      \begin{bmatrix} 0 & A_{12} \\ A_{21} & 0 \end{bmatrix}
    \right)^* 
    \begin{bmatrix} A_{11} & 0 \\ 0 & A_{22} \end{bmatrix}^* 
    = \displaybreak[0]\\ ={}&
    \left(
      \begin{bmatrix} A_{11}^* & 0 \\ 0 & A_{22}^* \end{bmatrix}
      \begin{bmatrix} 0 & A_{12} \\ A_{21} & 0 \end{bmatrix}
    \right)^*
    \begin{bmatrix} A_{11}^* & 0 \\ 0 & A_{22}^* \end{bmatrix} 
    = \displaybreak[0] \\ ={}&
    \begin{bmatrix} 
      0 & A_{11}^* A_{12} \\ A_{22}^* A_{21} & 0 
    \end{bmatrix}^*
    \begin{bmatrix} A_{11}^* & 0 \\ 0 & A_{22}^* \end{bmatrix}
    = \displaybreak[0] \\ ={}&
    \Biggl( \sum_{k=0}^\infty 
      \begin{bmatrix} 
        0 & A_{11}^* A_{12} \\ A_{22}^* A_{21} & 0 
      \end{bmatrix}^k
    \Biggr)
    \begin{bmatrix} A_{11}^* & 0 \\ 0 & A_{22}^* \end{bmatrix}
    = \displaybreak[0]\\ ={}&
    \Biggl( 
      \begin{bmatrix} 
        I & A_{11}^* A_{12} \\ A_{22}^* A_{21} & I
      \end{bmatrix}
      \sum_{k=0}^\infty 
      \begin{bmatrix} 
        0 & A_{11}^* A_{12} \\ A_{22}^* A_{21} & 0 
      \end{bmatrix}^{2k}
    \Biggr)
    \begin{bmatrix} A_{11}^* & 0 \\ 0 & A_{22}^* \end{bmatrix}
    = \displaybreak[0]\\ ={}&
    \begin{bmatrix} 
      I & A_{11}^* A_{12} \\ A_{22}^* A_{21} & I
    \end{bmatrix}
      \times \\ &\qquad
    \begin{bmatrix}
      A_{11}^* A_{12} A_{22}^* A_{21} & 0 \\
      0 & A_{22}^* A_{21} A_{11}^* A_{12}
    \end{bmatrix}^*
    \begin{bmatrix} A_{11}^* & 0 \\ 0 & A_{22}^* \end{bmatrix}
    = \\ ={}&
    \begin{bmatrix} 
      I & A_{11}^* A_{12} \\ A_{22}^* A_{21} & I
    \end{bmatrix}
      \times \\ &\qquad
    \begin{bmatrix}
      (A_{11}^* A_{12} A_{22}^* A_{21})^* A_{11}^* & 0 \\
      0 & (A_{22}^* A_{21} A_{11}^* A_{12})^* A_{22}^*
    \end{bmatrix}
    = \\ ={}&
    \begin{bmatrix}
      (A_{11}^* A_{12} A_{22}^* A_{21})^* A_{11}^* & 
        A_{11}^* A_{12} (A_{22}^* A_{21} A_{11}^* A_{12})^* A_{22}^* \\
      A_{22}^* A_{21} (A_{11}^* A_{12} A_{22}^* A_{21})^* A_{11}^* & 
        (A_{22}^* A_{21} A_{11}^* A_{12})^* A_{22}^*
    \end{bmatrix}
    \text{.}
  \end{align*}
\end{proof}

\subsection{Filtration and transference}

The matrix dependency filtration is not an isolated construction, but 
rather the basic example of a very general concept.

\begin{definition} \label{Def:PROP-filtrering}
  Let $\mc{P}$ and $\mc{Q}$ be \PROPs, and let $Q$ be \PROP\ 
  quasi-order on $\mc{Q}$. A 
  \DefOrd[*{PROP@\PROP!filtration in}]{$(\mc{Q},Q)$-filtration} 
  in $\mc{P}$ is a family \(F = \{ F_q \}_{q \in \mc{Q}}\) of subsets 
  of $\mc{P}$ such that:
  \begin{enumerate}
    \item
      For all \(m,n \in \N\), and every \(q \in \mc{Q}(m,n)\), \(F_q 
      \subseteq \mc{P}(m,n)\).
    \item
      If \(q \leqslant r \pin{Q}\) then \(F_q \subseteq F_r\).
    \item
      \(a \otimes b \in F_{q \otimes r}\)
      for all \(q,r \in \mc{Q}\), \(a \in F_q\), and \(b \in F_r\). 
    \item
      \(a \circ b \in F_{q \circ r}\) 
      for all \(q,r \in \mc{Q}\) such that \(\alpha(q) = \omega(r)\), 
      all \(a \in F_q\), and all \(b \in F_r\).
    \item
      \(\phi_\mc{P}(\sigma) \in F_{\phi_\mc{Q}(\sigma)}\)
      for all \(n \in \N\) and \(\sigma \in \Sigma_n\).
  \end{enumerate}
  $F$ is a $(\mc{Q},Q)$-filtration 
  \DefOrd[*{PROP@\PROP!filtration of}]{of} $\mc{P}$ if 
  \(\mc{P}(m,n) = \bigcup_{q \in \mc{Q}(m,n)} F_q\) for 
  all \(m,n \in \N\). A filtration $F$ is 
  \DefOrd[*{linear@$\mc{R}$-linear!filtration}]{$\mc{R}$-linear} if 
  every component $F_q$ is an $\mc{R}$-module.
\end{definition}

For \(q,r \in \B^{m \times n}\), \(q \leqslant r\) means \(r_{ij}=0 
\Longrightarrow q_{ij}=0\), and hence \(F_q \subseteq F_r\) in the matrix 
dependency filtration. Verifying the $\otimes$ property is a simple 
matter of keeping track of which matrix positions in the factors 
correspond to which matrix positions in the product, and the 
permutation property is immediate from the definition of $\phi$ in 
a matrix \PROP. The matrix dependency filtration is moreover 
$\mc{R}$-linear.

%

As with filtrations in rings, one can use the form of an expression 
for a specific element to argue that this element must be in a 
specific component of a filtration. 

\begin{lemma} \label{L:eval-filtrering}
  Let $\Omega$ be an $\N^2$-graded set, $\mc{P}$ and $\mc{Q}$ be \PROPs, 
  $Q$ a \PROP\ quasi-order on $\mc{Q}$, and $\{F_q\}_{q \in \mc{Q}}$ be 
  a $(\mc{Q},Q)$-filtration in $\mc{P}$. If \(f\colon \Omega \Fpil 
  \mc{P}\) and \(J\colon \Omega \Fpil \mc{Q}\) are $\N^2$-graded set 
  morphisms such that \(f(x) \in F_{J(x)}\) for all \(x \in \Omega\), 
  then
  \begin{equation}
    \eval_f(G) \in F_{\eval_J(G)} 
    \quad\text{for all \(G \in \Nw(\Omega)\).}
  \end{equation}
\end{lemma}
\begin{proof}
  Consider first the case that \(G = (V,E,h,g,t,s,D)\) has no inner 
  vertices. Then \(\sigma := g \circ s^{-1}\) is a permutation and 
  \(\eval_f(G) = \phi_{\mc{P}}(\sigma) \in F_{\phi_{\mc{Q}}(\sigma)} = 
  F_{\eval_J(G)}\).
  
  Consider next the case that $G$ has one inner vertex $v$ with 
  decoration \(x = D(v)\). Then
  \(
    G = \Nwfuse{\vek{m}}{ (v\colon x)^\vek{k}_\vek{l} }{\vek{n}}
  \)
  for some \(\vek{k}, \vek{l}, \vek{m}, \vek{n}\) such that 
  \(\norm{\vek{k}} \subseteq \norm{\vek{m}}\) and \(\norm{\vek{l}} 
  \subseteq \norm{\vek{n}}\) but \(\norm{\vek{k}} \cap \norm{\vek{l}} 
  = \varnothing\). Hence
  $$
    \eval_f(G) =
    \fuse{ \vek{m} }{ f(x)^\vek{k}_\vek{l} }{ \vek{n} }_\mc{P} =
    \fuse{ \vek{m} }{ 1 }{ \vek{k(m/k)} }_\mc{P} \circ 
      f(x) \otimes \fuse{ \vek{m/k} }{1}{ \vek{m/k} }_\mc{P} \circ
      \fuse{ \vek{l(m/k)} }{1}{ \vek{n} }_\mc{P} 
    \text{.}
  $$
  Now \(f(x) \in F_{J(x)}\) by assumption, and
  $$
    \fuse{ \vek{m/k} }{1}{ \vek{m/k} }_\mc{P} \in
    F_{\fuse{ \vek{m/k} }{1}{ \vek{m/k} }_\mc{Q}}
    \text{,\quad so }
    f(x) \otimes \fuse{ \vek{m/k} }{1}{ \vek{m/k} }_\mc{P} \in
    F_{J(x) \otimes \fuse{ \vek{m/k} }{1}{ \vek{m/k} }_\mc{Q}}
    \text{,}
  $$
  and similarly
  $$
    \fuse{ \vek{m} }{1}{ \vek{k(m/k)} }_\mc{P} \in
    F_{\fuse{ \vek{m} }{1}{ \vek{k(m/k)} }_\mc{Q}}
    \text{,}\qquad
    \fuse{ \vek{l(m/k)} }{1}{ \vek{n} }_\mc{P} \in
    F_{\fuse{ \vek{l(m/k)} }{1}{ \vek{n} }_\mc{Q}}
    \text{.}
  $$
  Thus \(\eval_f(G) \in F_q\) for
  $$
    q =
    \fuse{ \vek{m} }{ 1 }{ \vek{k(m/k)} }_\mc{Q} \circ 
      J(x) \otimes \fuse{ \vek{m/k} }{1}{ \vek{m/k} }_\mc{Q} \circ
      \fuse{ \vek{l(m/k)} }{1}{ \vek{n} }_\mc{Q} =
    \fuse{ \vek{m} }{ J(x)^\vek{k}_\vek{l} }{ \vek{n} }_\mc{Q} =
    \eval_J(G) \text{.}
  $$
  
  Finally, induction over the number of inner vertices is used to 
  establish the lemma in the remaining cases. When there are at least 
  two inner vertices, there is some cut $(W_0,W_1,p)$ in $G$ such 
  that neither $W_0$ nor $W_1$ is empty. Hence the obvious induction 
  hypothesis applies to both parts of the corresponding decomposition 
  $(G_0,G_1)$ of $G$, and thus
  $$
    \eval_f(G) = \eval_f(G_0) \circ \eval_f(G_1) \in 
    F_{\eval_J(G_0) \circ \eval_J(G_1)} = F_{\eval_J(G)} \text{.}
  $$
\end{proof}

It is not at this point in the exposition possible to apply the 
filtration concept to the networks themselves, since it has not yet been 
established that these carry a \PROP\ structure for a filtration to be 
related to, but one can define something which keeps track of pretty 
much the same information as a dependency filtration would.

\begin{definition} \label{Def:Transferens}
  The \DefOrd{transference} of a network $G$ is the matrix \(\Tr(G) \in 
  \B^{\omega(G) \times \alpha(G)}\) which has $1$ in position $(i,j)$ 
  if and only if there is a directed path from $1$ to $0$ in $G$ such 
  that the initial edge has tail index $j$ and the final edge has 
  head index $i$.
\end{definition}

As defined, transference is a purely syntactic concept. Its main 
application will be to predict whether a network has some output that 
could be fed back to an input without creating a cycle; this is a key 
part of the formalisation of the ``may depend on'' arrow $\canmake$ 
seen in \eqref{Eq:Regelfamiljer}. Obviously, if position $(i,j)$ in 
$\Tr(G)$ is $1$ then joining the $i$th output of $G$ with its $j$th 
input will create a cycle. The purpose of making the transference 
a matrix is to handle simultaneous joining of several inputs and 
outputs, as it turns out nilpotence and the Kleene star are exactly 
the concepts needed to keep track of things in that situation.


If viewing networks as circuits, a more conceptual description is 
that the transference keeps track of which inputs contributes to 
which outputs. The classical information interpretation of the 
transference of a network is thus that a $0$ in some position $(i,j)$ 
signals that ``output $i$ cannot depend on input $j$, since there is no 
path along which such information could be carried.'' Whether this 
interpretation is valid will however depend on the \PROP\ in which a 
network is evaluated; there are sometimes mechanisms which can be used 
as loopholes to escape that conclusion about how information may be 
transferred, especially in ``quantum'' (or more generally linear) 
\PROPs. There are for example \PROPs\ with a ``cup'' element $U$ in the 
$(0,2)$ component and a ``cap'' element $\Lambda$ in the $(2,0)$ 
component which satisfy the equation 
\(\fuse{a}{ \Lambda^{ab} U_{bc} }{c} = \phi(\same{1}) = 
\fuse{a}{ U_{cb} \Lambda^{ba} }{c}\); this will happen if $U$ is an inner 
product on a finite dimensional vector space and $\Lambda$ is the 
corresponding inner product on the dual vector space. 
The transference of the networks $\Nwfuse{a}{ \Lambda^{ab} U_{bc} }{c}$ 
and $\Nwfuse{a}{ U_{cb} \Lambda^{ba} }{c}$ is $0$, but they may anyway 
evaluate to the identity (which puts ``exactly the same information'' 
on the output as it receives from the input) in a particular \PROP. 
Attempting an information transfer interpretation of 
transference may thus be somewhat tenuous, but the same is true for 
network expressions in general: in $\fuse{a}{ U_{cb} \Lambda^{ba} }{c}$, 
one interpretation is that cups and caps can ``reverse the direction of 
time'' (causality), whereas another is that the cap $\Lambda$ outputs 
a state $\sum_x x \otimes x$ that is a superposition of all pairs of 
identical basis states and conversely the cup $U$ destroys (maps to 
$0$) all histories where its two inputs are not equal. Neither 
interpretation is particularly physical, although the second bears 
some resemblance to quantum teleportation (but noticeably missing the 
steps involving classical information).


Yet another view of transference is that it is an evaluation of a 
network. 

\begin{lemma} \label{L:TrHomomorfi}
  Let $\Omega$ be an $\N^2$-graded set, and let \(J\colon \Omega 
  \Fpil \B^{\bullet\times\bullet}\) be the $\N^2$-graded set morphism 
  which maps every \(x \in \Omega\) to the $\omega(x) \times 
  \alpha(x)$ matrix of ones. Then \(\Tr(G) = \eval_J(G)\) for every 
  network \(G \in \Nw(\Omega)\).
\end{lemma}
\begin{proof}
  Beginning with networks without inner vertices, one may observe 
  that any path in that case has the form $0 \: e \: 1$ (written going 
  right-to-left) for some edge $e$. Hence the $(i,j)$ entry of 
  $\Tr(G)$ is $1$ if and only if \(g^{-1}(i) = s^{-1}(j)\), which 
  is precisely the defining condition for the permutation matrix 
  $\phi_{\B^{\bullet\times\bullet}}(\sigma)$ for \(\sigma = g \circ 
  s^{-1}\). This permutation matrix is also what $\eval_J(G)$ is 
  defined to be.
  
  Next considering networks \(G = \Nwfuse{ \vek{a} }{ 
  (v\colon x)^\vek{a}_\vek{b} }{ \vek{b} }\), one immediately finds 
  that \(\eval_J(G) = J(x)\) is a matrix of ones. $\Tr(G)$ also 
  consists only of ones, since \(0 \: e \: v \: f \: 1\) is a path 
  from $1$ to $0$ for any \(e \in \norm{\vek{a}}\) and \(f \in 
  \norm{\vek{b}}\). The equality has thus been verified for both base 
  cases.
  
  If the network $G$ has a split 
  $(F_\mathrm{l},F_\mathrm{r},W_\mathrm{l},W_\mathrm{r})$ then 
  Lemma~\ref{L:evalSplit} implies that
  \begin{equation*}
    \eval_J(G) = 
    \eval_J(G_\mathrm{l}) \otimes \eval_J(G_\mathrm{r}) =
    \begin{bmatrix}
      \eval_J(G_\mathrm{l})& 0 \\
      0 & \eval_J(G_\mathrm{r})
    \end{bmatrix}
  \end{equation*}
  for the corresponding decomposition $(G_\mathrm{l},G_\mathrm{r})$ 
  of $G$. On the transference side of things, any path from $1$ to $0$ 
  in $G$ is either a path in $G_\mathrm{l}$ or a path in $G_\mathrm{r}$. 
  From the definition of split (in particular the modification of 
  indices in the right component), it thus follows that
  \begin{equation*}
    \Tr(G) = \begin{bmatrix}
      \Tr(G_\mathrm{l})& 0 \\
      0 & \Tr(G_\mathrm{r})
    \end{bmatrix}
    \text{.}
  \end{equation*}
  Hence if \(\Tr(G_\mathrm{l}) = \eval_J(G_\mathrm{l})\) and 
  \(\Tr(G_\mathrm{r}) = \eval_J(G_\mathrm{r})\) then \(\Tr(G) = 
  \eval_J(G)\). In particular, the lemma claim holds for \(G = 
  \Nwfuse{ \vek{ac} }{ x^\vek{a}_\vek{b} }{ \vek{bc} }\).
  
  Similarly, if the network $G$ has a cut $(W_0,W_1,q)$ inducing a 
  decomposition $(G_0,G_1)$ then Lemma~\ref{L:evalCut} implies that 
  \(\eval_J(G) = \eval_J(G_0) \eval_J(G_1)\). 
  On the transference side of things, if \(\Tr(G)_{i,k} = 1\) then a 
  corresponding path $P$ from $1$ to $0$ in $G$ contains exactly one cut 
  edge $e$. The edges up to and including $e$ induce a path $P_1$ in $G_1$ 
  from $1$ to $0$, and the edges from $e$ and on induce a path $P_0$ in 
  $G_0$ from $1$ to $0$, meaning \(\Tr(G_0)_{i,q(e)} = 
  \Tr(G_1)_{q(e),k} = 1\). Hence \(\Tr(G) \leqslant \Tr(G_0) 
  \Tr(G_1)\). Conversely, if \(\Tr(G_0)_{i,j} = \Tr(G_1)_{j,k} = 1\) 
  then there is a path $P_0$ from $1$ to $0$ in $G_0$ whose first 
  edge is $q^{-1}(j)$ and whose last edge has head index $i$, and 
  also a path $P_1$ from $1$ to $0$ in $G_1$ whose first edge has 
  tail index $k$ and whose last edge is $q^{-1}(j)$. Combining these 
  in $G$ yields a path $P$ from $1$ to $0$ whose first edge has tail 
  index $k$ and whose last edge has head index $i$. It follows that 
  \(\Tr(G_0) \Tr(G_1) \leqslant \Tr(G)\). Hence if \(\Tr(G_0) = 
  \eval_J(G_0)\) and \(\Tr(G_1) = \eval_J(G_1)\) then \(\Tr(G) = 
  \eval_J(G)\).
  
  From the recursive definition of the abstract index expression 
  defining $\eval_J(G)$, one gets a multicut decomposition of $G$ 
  into parts which are each either on the form $\Nwfuse{ \vek{a} }{ 1 }{ 
  \vek{b} }$ or on the form $\Nwfuse{ \vek{ac} }{ x^\vek{a}_\vek{b} 
  }{ \vek{bc} }$. Since $\Tr$ and $\eval_J$ coincide for all of 
  these, it now follows that \(\Tr(G) = \eval_J(G)\) in general.
\end{proof}

An alternative way of arriving to the same conclusion would be to use 
Example~\ref{Ex:Matris-eval}, since every $1$ in the result there can 
be traced back along a path to the input.



\begin{corollary} \label{Kor:eval-filtrering}
  Let $\mc{P}$ be a \PROP, $\Omega$ be an $\N^2$-graded set, and 
  \(f\colon \Omega \Fpil \mc{P}\) be an $\N^2$-graded set morphism.
  If $\{F_q\}_{q \in \B^{\bullet\times\bullet}}$ is a filtration of 
  $\mc{P}$ then for any network $G$ of $\Omega$,
  \begin{equation}
    \eval_f(G) \in F_{\Tr(G)} \text{.}
  \end{equation}
\end{corollary}
\begin{proof}
  In Lemma~\ref{L:eval-filtrering}, take \(\mc{Q} = 
  \B^{\bullet\times\bullet}\) and let $J(x)$ be the $\omega(x) \times 
  \alpha(x)$ matrix of ones. Then $J(x)$ is the maximum element in 
  $\B^{\omega(x) \times \alpha(x)}$ and hence 
  \(F_{J(x)} = \mc{P}\bigl( \omega(x),\alpha(x) \bigr)\) because 
  $\{F_q\}_{q \in \B^{\bullet\times\bullet}}$ is a filtration 
  \emph{of} $\mc{P}$. The claim now follows from \(\Tr(G) = 
  \eval_J(G)\).
\end{proof}

\section{Subexpressions}
\label{Sec:Deluttryck}

With the view of networks as ``\PROP\ expressions'' firmly 
established, and some technical lemmas taken care of, it is now time to 
consider what it would mean for one network to be a subexpression of 
another, as that is an initial step of doing rewriting.

The categorically inclined might suggest as definition that $A$ is a 
factor\slash subexpression of $B$ iff \(B = C_1 \circ C_2 \otimes A 
\otimes C_3 \circ C_4\) for some $C_1$ through $C_4$, or perhaps 
even streamline this to the equivalent condition that \(B = C_1' \circ 
A \otimes \phi(\same{k}) \circ C_4'\) (combine $C_2$ and $C_3$ with 
$C_1$, then insert some permutations to get rid of one of the 
$\otimes$ factors). This is equivalent to being a convex subnetwork, 
but that is in fact unnecessarily restrictive. 
A better starting point is the abstract index notation, 
which rather suggests the concept that a naked expression $A$ is a 
subexpression of any naked expression $B=AC$, i.e., $A$ is a 
subexpression of $B$ iff padding $A$ with some extra factors $C$ 
produces $B$ (after some minor cleanup, such as renaming labels). The 
smallest example that demonstrates the difference is probably 
\eqref{Eq:MinstaFeedback-exempel} since
\begin{multline*}
  \fuse{ ab }{ 
    \mOp^b_{cd} \, \mOp^c_{ef} \, S^d_g \, \Delta^{af}_i \, 
    \Delta^{ig}_j 
  }{ ej } = \\ =
  \phi(\same{1}) \otimes \mOp \circ
    \phi(\same{1}) \otimes \mOp \otimes \phi(\same{1}) \circ
    \phi(\cross{1}{1} \star \same{1}) \otimes S \circ
    \phi(\same{1}) \otimes \Delta \otimes \phi(\same{1}) \circ
    \phi(\same{1}) \otimes \Delta
\end{multline*}
is a subexpression of
\begin{multline*}
  \fuse{ hb }{ 
    \mOp^b_{cd} \, \mOp^c_{ef} \, S^d_g \, \Delta^{af}_i \, 
    \Delta^{ig}_j \, \Delta^{he}_a
  }{ j } = \\ =
  \phi(\same{1}) \otimes \mOp \circ
    \phi(\same{1}) \otimes \mOp \otimes \phi(\same{1}) \circ
    \Delta \otimes \phi(\same{1}) \otimes S \circ
    \Delta \otimes \phi(\same{1}) \circ
    \Delta
\end{multline*}
in the latter sense, but not in the former; in the more transparent 
network notation, the two expressions are
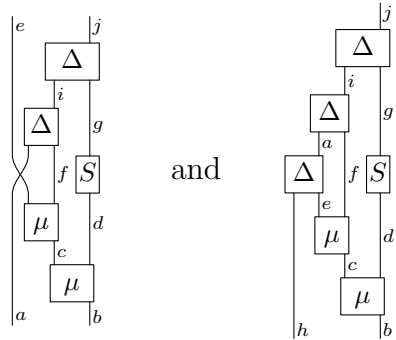
\[
  \begin{mpgraphics*}{13}
    beginfig(13);
      PROPdiagramnoarrow(0,0)
        elabels(btex $\scriptstyle a$ etex, btex $\scriptstyle b$ etex) 
        circ
        same(1) otimes box(1,2)(btex \strut$\mOp$ etex) circ
        elabels(0, btex $\scriptstyle c$ etex, 0) circ
        same(1) otimes box(1,2)(btex \strut$\mOp$ etex) 
          otimes elabel(btex $\scriptstyle d$ etex) circ
        cross(1,1) otimes 
          elabel(btex $\scriptstyle f$ etex) otimes 
          box(1,1)(btex \strut$S$ etex) circ
        same(1) otimes box(2,1)(btex \strut$\Delta$ etex) otimes
          elabel(btex $\scriptstyle g$ etex) circ
        elabels(false, btex $\scriptstyle i$ etex , false) circ
        same(1) otimes box(2,1)(btex \strut$\Delta$ etex) circ
        elabels(btex $\scriptstyle e$ etex,btex $\scriptstyle j$ etex)
    endfig;
  \end{mpgraphics*}
  \qquad\text{and}\qquad
  \begin{mpgraphics*}{14}
    beginfig(14);
      PROPdiagramnoarrow(0,0)
        elabels(btex $\scriptstyle h$ etex, btex $\scriptstyle b$ etex) 
        circ
        same(1) otimes box(1,2)(btex \strut$\mOp$ etex) circ
        elabels(0, btex $\scriptstyle c$ etex, 0) circ
        same(1) otimes box(1,2)(btex \strut$\mOp$ etex) 
          otimes elabel(btex $\scriptstyle d$ etex) circ
        same(1) otimes elabel(btex $\scriptstyle e$ etex) otimes
          same(2) circ
        box(2,1)(btex \strut$\Delta$ etex) otimes 
          elabel(btex $\scriptstyle f$ etex) otimes 
          box(1,1)(btex \strut$S$ etex) circ
        elabels(btex $\scriptstyle a$ etex , false, false) circ
        box(2,1)(btex \strut$\Delta$ etex) otimes
          elabel(btex $\scriptstyle g$ etex) circ
        elabels(btex $\scriptstyle i$ etex , false) circ
        box(2,1)(btex \strut$\Delta$ etex) circ
        elabel(btex $\scriptstyle j$ etex)
    endfig;
  \end{mpgraphics*}
\]
The obstruction is that 
the extra $\Delta$ is connected to the $A$ subexpression both with an 
incoming edge ($a$) and with an outgoing edge ($e$), so it cannot be 
in either of $C_1$ or $C_4$, but it can be in a generic product $C$ 
of abstract index factors.

At first sight, such non-convex subexpressions may seem strange, 
and downright implausible\Ldash why would one want to look for 
\emph{that} as a subexpression?\Rdash but experience 
seems to indicate that rules with this kind of left hand side arise 
naturally as soon as the rules become complicated enough for nonconvexity 
to occur. When a nonconvex rule is a derived rule, it can always be 
understood as a combination of more elementary rules (some of which 
may be used backwards), and these may of course encircle vertices 
outside the subexpression that is being acted upon, because the mere 
fact that two vertices in an expression are affected by rules does 
not necessarily imply that all vertices on every path between those 
vertices are affected by some rule as well. (Such a conclusion could 
however be drawn in cases where all expressions are treelike, since 
in that case there would only be one path to consider. Nonconvexity 
of rules can therefore be viewed as a complication that arises when 
generalising from treelike to general DAG expressions.)

\subsection{The symmetric join}

What is then the network counterpart of the abstract index ``multiply 
by some naked expression''? As the example indicates, it would be a 
network where paths can go back and forth between the two factors, 
so it follows that a general construction may end up looking somewhat 
like Figure~\ref{Fig:Korskoppling}. This raises the problem of knowing 
whether the end result will obey the acyclicity condition for a network, 
but luckily all aspects of the parts that are relevant for this are 
recorded in their transferences.

\begin{figure}
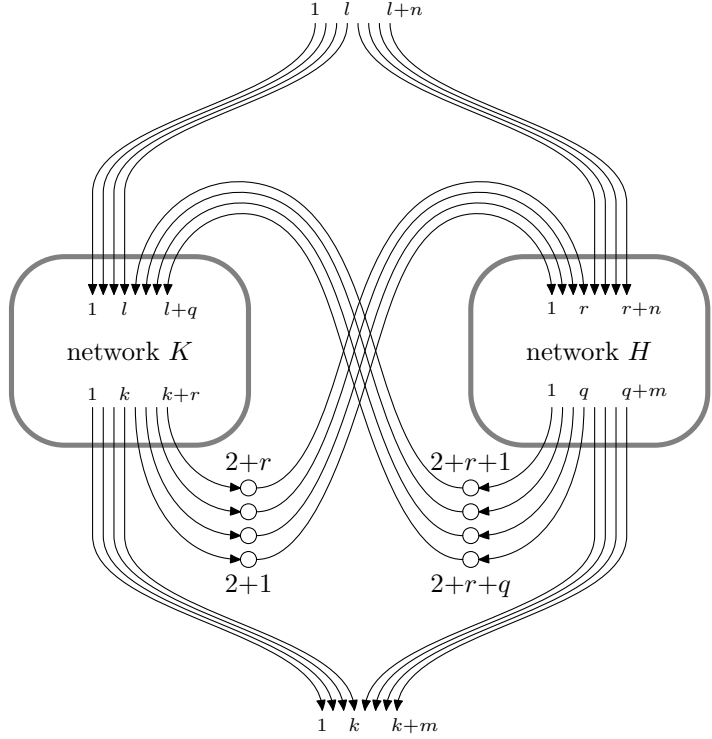

  \begin{center}
%
%
%
%
%
%
%
    \begin{mpgraphics}{30}
      beginfig(30);
        wires:=4;
        intrusion:=0.5cm;
        slant:=2;
        squash:=1;
        u:=3pt;
        rounding:=0.7cm;
        picture two_one, two_r, two_r_one, two_r_q;
        two_one := btex $2{+}1$ etex;
        two_r   := btex $2{+}r$ etex;
        two_r_one := btex $2{+}r{+}1$ etex;
        two_r_q   := btex $2{+}r{+}q$ etex;
        width_two_r_one := 
          xpart lrcorner two_r_one - xpart llcorner two_r_one;
        IC_equations1(2*wires,2*wires)(
           interim IC_height := 2.5cm;
           interim IC_width := 3cm;
           interim IC_legs_a := width_two_r_one;
           interim IC_legs_c := width_two_r_one;
        );
        IC_equations2(2*wires,2*wires)(
           interim IC_height := 2.5cm;
           interim IC_width := 3cm;
           interim IC_legs_a := width_two_r_one;
           interim IC_legs_c := width_two_r_one;
        );
        0.5[x3[1],x3[wires]] = 0.5[x1,x2]; 
        y1 = y2 for j:=1 upto wires: = y3[j] endfor ;
        x1ur for j:=1 upto wires: = x4[j] = x6[j] endfor;
        x2ul for j:=1 upto wires: = x5[j] = x7[j] endfor;
        for j := 1 upto 2*wires:
           x1tt[j] = x1t[j]; y1tt[j] = y1t[j] - intrusion;
           x2tt[j] = x2t[j]; y2tt[j] = y2t[j] - intrusion;
           x1bb[j] = x1b[j]; y1bb[j] = y1b[j] + intrusion;
           x2bb[j] = x2b[j]; y2bb[j] = y2b[j] + intrusion;
        endfor
        for j := 2 upto wires: 
           y7[j-1] - y7[j] = 3u;
           y5[j] - y5[1] = x2t[j] - x2t[1] = x3[1] - x3[j];
           y4[1] - y4[j] = x1t[j+wires] - x1t[1+wires];
           y6[j] - y6[j-1] = 3u;
        endfor
        y6[wires] = y7[1] = y2bb[1] - squash*(x6[1]-x1b[2*wires]);
        y4[wires] = y5[1] = y2tt[1] + squash*(x6[1]-x1b[2*wires]);
        y6[2] for j:=1 upto wires: = y8[j] endfor;
        y7[wires-1] for j:=1 upto wires: = y9[j] endfor;
        for j:=1 upto wires:
          x8[j] = x1b[j];
          x9[j] = x2b[wires+j];
        endfor
        for j:=2 upto wires:
          x10[j] - x8[j] = x10[1] - x8[1];
          y10[j] = y10[1];
          x10[wires+j] - x9[j] = x10[wires+1] - x9[1];
          y10[wires+j] = y10[wires+1];
        endfor
        x10[wires+1]-x10[wires] = x10[2]-x10[1];
        y10[wires+1] = y10[wires];
        z10[1] = z8[1] + whatever*(3,-2.3);
        z10[wires+1] = z9[1] + whatever*(-3,-2.3);
        y4[2] for j:=1 upto wires: = y11[j] endfor;
        y5[wires-2] for j:=1 upto wires: = y12[j] endfor;
        for j:=1 upto wires:
          x11[j] = x1t[j];
          x12[j] = x2t[wires+j];
        endfor
        for j:=2 upto wires:
          x13[j] - x11[j] = x13[1] - x11[1];
          y13[j] = y13[1];
          x13[wires+j] - x12[j] = x13[wires+1] - x12[1];
          y13[wires+j] = y13[wires+1];
        endfor
        x13[wires+1]-x13[wires] = x13[2]-x13[1];
        y13[wires+1] = y13[wires];
        z13[1] = z11[1] + whatever*(3,2.3);
        z13[wires+1] = z12[1] + whatever*(-3,2.3);
        y5[wires] - y7[wires] = slant*(x7[wires] - x6[1] - 4u);
        y6[1] = 0;  x1ll = 0;
        
        draw z1ll+(rounding,0) --- z1lr-(rounding,0) ..
          z1lr+(0,rounding) --- z1ur-(0,rounding) ..
          z1ur-(rounding,0) --- z1ul+(rounding,0) ..
          z1ul-(0,rounding) --- z1ll+(0,rounding) .. cycle
          withpen pencircle scaled 2pt withcolor 0.5[black,white];
        label (btex network $K$ etex, z1);
        draw z2ll+(rounding,0) --- z2lr-(rounding,0) ..
          z2lr+(0,rounding) --- z2ur-(0,rounding) ..
          z2ur-(rounding,0) --- z2ul+(rounding,0) ..
          z2ul-(0,rounding) --- z2ll+(0,rounding) .. cycle
          withpen pencircle scaled 2pt withcolor 0.5[black,white];
        label (btex network $H$ etex, z2);
        
        path P;
        for j := 1 upto wires:
           drawarrow z1bb[wires+j]{down} .. {right}(z6[j]-(u,0));
           draw fullcircle scaled 2u shifted z6[j];
           drawarrow (z6[j]+(u,0)){right} ... z3[j]{0.6,slant} ... 
             z5[j] .. {down}z2tt[j];
           drawarrow z2bb[j]{down} .. {left}(z7[j]+(u,0));
           draw fullcircle scaled 2u shifted z7[j];
           drawarrow (z7[j]-(u,0)){left} ... z3[j]{-0.6,slant} ... 
             z4[j] .. {down}z1tt[wires+j];
           drawarrow z1bb[j] --- z8[j] .. tension 1.5 .. {down}z10[j];
           drawarrow z2bb[wires+j] --- z9[j] .. tension 1.5 .. 
             {down}z10[wires+j];
           drawarrow z13[j]{down} .. tension 1.5 .. z11[j] --- z1tt[j];
           drawarrow z13[wires+j]{down} .. tension 1.5 .. 
             z12[j] --- z2tt[wires+j];
        endfor
        
        label.bot (two_one,  z6[1]-(0,u));
        label.top (two_r,    z6[wires]+(0,u));
        label.top (two_r_one,z7[1]+(0,u));
        label.bot (two_r_q,  z7[wires]-(0,u));
        
        label.bot (btex $\scriptstyle 1\vphantom{k}$ etex, z10[1]);
        label.bot (btex $\scriptstyle k$ etex, z10[wires]);
        label.bot (btex $\scriptstyle k$\rlap{$\scriptstyle+m$} etex, 
          z10[2*wires]);
        label.top (btex $\scriptstyle 1$ etex, z13[1]);
        label.top (btex $\scriptstyle l$ etex, z13[wires]);
        label.top (btex $\scriptstyle l$\rlap{$\scriptstyle\smash{+}n$} etex, 
          z13[2*wires]);
        label.top (btex $\scriptstyle 1$ etex, z1bb[1]);
        label.top (btex $\scriptstyle k$ etex, z1bb[wires]);
        label.top (btex $\scriptstyle k$\rlap{$\scriptstyle\smash{+}r$} etex, 
          z1bb[2*wires]);
        label.top (btex $\scriptstyle 1\vphantom{q}$ etex, z2bb[1]);
        label.top (btex $\scriptstyle q$ etex, z2bb[wires]);
        label.top (btex $\scriptstyle q$\rlap{$\scriptstyle\smash{+}m$} etex, 
          z2bb[2*wires]);
          label.bot (btex $\scriptstyle 1\vphantom{l}$ etex, z1tt[1]);
        label.bot (btex $\scriptstyle l$ etex, z1tt[wires]);
        label.bot (btex $\scriptstyle l$\rlap{$\scriptstyle+q$} etex, 
          z1tt[2*wires]);
        label.bot (btex $\scriptstyle 1$ etex, z2tt[1]);
        label.bot (btex $\scriptstyle r\vphantom{1}$ etex, z2tt[wires]);
        label.bot (btex $\scriptstyle r$%
            \rlap{$\scriptstyle\vphantom{1}\smash{+}n$} etex, 
          z2tt[2*wires]);
          
      endfig;
    \end{mpgraphics}
  \end{center}
  \caption{Symmetric join of two networks}
  \label{Fig:Korskoppling}
\end{figure}

\begin{construction}[Symmetric join] \label{Kons:symjoin}
  Let $\Omega$ be an $\N^2$-graded set with \(y \in \Omega\) such 
  that \(\alpha(y) = \omega(y) = 1\). Let \(k,l,m,n,q,r \in \N\),
  \begingroup
    \addtolength{\medmuskip}{-2mu plus 2mu}%
  \begin{align*}
    K = (V_K,E_K,h_K,g_K,t_K,s_K,D_K) \in{}& 
      \Nw(\Omega)(k +\nobreak r, l +\nobreak s)
      \text{, and}\\
    H = (V_H,E_H,h_H,g_H,t_H,s_H,D_H) \in{}&
      \Nw(\Omega)(s +\nobreak m, r +\nobreak n)
  \end{align*}
  be given. Consider Figure~\ref{Fig:Korskoppling}. Define
  \begin{align*}
    K \join[y]^r_q H :={}& (V,E,h,g,t,s,D)
  \intertext{where}
    E ={}& \{2e\}_{e \in E_K} \mathbin{\dot{\cup}} 
      \{2e+1\}_{e \in E_H}
      \text{,}
    \\
    V ={}& \{0,1\} \mathbin{\dot{\cup}}  
      \setOf{v \in \N}{ 1 \leqslant v-2 \leqslant r+q} 
      \mathbin{\dot{\cup}} \\ &\qquad {}\mathbin{\dot{\cup}}
      \{r+q+2v\}_{v \in V_K \setminus \{0,1\}} \mathbin{\dot{\cup}} 
      \{r+q+2v+1\}_{v \in V_H \setminus \{0,1\}} 
      \text{,} 
    \displaybreak[0]\\
    D(v) ={}& \begin{cases}
      y& \text{if \(1 \leqslant v-2 \leqslant r+q\),}\\
      D_K\bigl( \bigl\lfloor (v-r-q)/2 \bigr\rfloor \bigr)&
        \text{if \(v > 2+r+q\) and \(2 \mid v-r-q\),}\\
      D_H\bigl( \bigl\lfloor (v-r-q)/2 \bigr\rfloor \bigr)&
        \text{if \(v > 2+r+q\) and \(2 \nmid v-r-q\),}
    \end{cases} 
    \displaybreak[0]\\
    (h,g)(2e) ={}& \begin{cases}
      \bigl( r+q+2h_K(e), g_K(e) \bigr)& 
        \text{if \(h_K(e) \neq 0\),}\\
      \bigl( 0, g_K(e) \bigr)& 
        \text{if \(h_K(e) = 0\) and \(g_K(e) \leqslant k\),}\\
      \bigl( 2+g_K(e)-k, 1\bigr)& 
        \text{if \(h_K(e) = 0\) and \(g_K(e) > k\),}
    \end{cases}\\
    (h,g)(2e+1) ={}& \begin{cases}
      \bigl( r+q+2h_H(e)+1, g_H(e) \bigr)& 
        \text{if \(h_H(e) \neq 0\),}\\
      \bigl( 2+r+g_H(e), 1 \bigr)& 
        \text{if \(h_H(e) = 0\) and \(g_H(e) \leqslant q\),}\\
      \bigl( 0, k+ h_H(e)-q \bigr)& 
        \text{if \(h_H(e) = 0\) and \(g_H(e) > q\),}
    \end{cases}
    \displaybreak[0]\\
    (t,s)(2e) ={}& \begin{cases}
      \bigl( r+q + 2t_K(e), s_K(e) \bigr)&
        \text{if \(t_K(e) \neq 1\),}\\
      \bigl( 1, s_K(e) \bigr)&
        \text{if \(t_K(e)=1\) and \(s_K(e) \leqslant l\),}\\
      \bigl( 2+r + s_K(e)-l , 1\bigr)&
        \text{if \(t_K(e)=1\) and \(s_K(e) > l\),}
    \end{cases}\\
    (t,s)(2e+1) ={}& \begin{cases}
      \bigl( r+q + 2t_H(e)+1, s_H(e) \bigr)&
        \text{if \(t_H(e) \neq 1\),}\\
      \bigl( 2 + s_H(e), 1 \bigr)&
        \text{if \(t_H(e) = 1\) and \(s_H(e) \leqslant r\),}\\
      \bigl( 1, l + s_H(e)-r \bigr)&
        \text{if \(t_H(e) = 1\) and \(s_H(e) > r\).}
    \end{cases}
  \end{align*}
  \endgroup 
  Let
  \begin{align*}
    \begin{bmatrix}
      a_{11} & a_{12} \\ a_{21} & a_{22}
    \end{bmatrix} :={}&
      \Tr(K) \text{,} &
    \begin{bmatrix}
      b_{22} & b_{23} \\ b_{32} & b_{33}
    \end{bmatrix} :={}&
      \Tr(H)
  \end{align*}
  where \(a_{22} \in \B^{r \times q}\) and \(b_{22} \in \B^{q \times 
  r}\). 
  
  If $a_{22} b_{22}$ is nilpotent then \((V,E,h,g,t,s,D) \in 
  \Nw(\Omega)(k +\nobreak m, l +\nobreak n)\) and
  \begin{equation} \label{Eq:TrSymJoin}
    \Tr( K \join[y]^r_q H ) = \begin{bmatrix}
      a_{11} + a_{12} b_{22} (a_{22} b_{22})^* a_{21} &
        a_{12} (b_{22} a_{22})^* b_{23} \\
      b_{32} (a_{22} b_{22})^* a_{21} &
        b_{33} + b_{32} a_{22} (b_{22} a_{22})^* b_{23}
    \end{bmatrix}
    \text{.}
  \end{equation}
\end{construction}
\begin{remark}
  In Section~\ref{Sec:Feedbacks} it will be shown that 
  \eqref{Eq:TrSymJoin} is a special case of the formula 
  \eqref{Eq:Feedback-join} for $\eval_f(K \join^r_q\nobreak H)$.
  
  In practice, the join-vertex annotation $y$ is typically going to 
  be a ``neutral'' symbol $\natural$ that is not present in the 
  $\Omega$ primarily under consideration, but rather in some extended 
  \(\Omega' = \Omega \mathbin{\dot{\cup}} \{\natural\}\). Such 
  $\natural$ vertices are mostly there for technical purposes and 
  should not contribute any mathematical content, but since it for 
  some results makes no difference either way, these will be stated for 
  a general \(y \in \Omega(1,1)\). Notationally, the $y$ may be omitted 
  from $\join[y]$ when \(y = \natural\).
  
  Another notational simplification is that $K 
  \join^{\alpha(H)}_{\omega(H)} H$, i.e., when the right part $H$ is 
  completely engulfed by $K$, is written $K \rtimes H$. This 
  annexation of one network ($H$) into another ($K$) is the original 
  application for $\join$, but it turned out the more general 
  symmetric join was useful when proving things about it.
\end{remark}
\begin{proof}
  The first claim to verify is that the unions in the formulae for 
  $E$ and $V$ really are distinct. For $E$ this is immediately clear 
  (any $2e$ is even, any $2e+1$ is odd), and the same idea is used 
  for the last two parts of $V$. When it comes to separation of these 
  two from the previous parts, one must realise that \(v \geqslant 2\) 
  for \(v \in V_K \setminus\{0,1\}\) or \(v \in V_H 
  \setminus\{0,1\}\), so the least element that can be found in these 
  latter parts is $r+q+4$, whereas the largest join vertex label is 
  $r+q+2$. That there is never any \(r+q+3 \in V\), nor \(2 \in V\), is 
  deliberate\Dash the gaps simplify the formulae, and an offset of 
  $2$ for the join vertex labels is more visible than an offset 
  of~$1$.
  
  Once it is realised that the definitions of $E$ and $V$ are mostly 
  about relabelling old vertices so that $K$ and $H$ don't collide, 
  it should be clear that the top cases of the definitions of $h$, 
  $g$, $t$, and $s$ are merely reproducing the inner parts of $K$ and 
  $H$. The remaining two cases deal with external edges, and have to 
  distinguish between external edges that get joined up and external 
  edges that remain external. The deciding factor here is the head 
  and tail indices, and it is easily checked that everything works 
  out as suggested in Figure~\ref{Fig:Korskoppling}. At the same 
  time, one may verify that the conditions for $L = (V,E,h,g,t,s,D)$ 
  to be a network, save the acyclicity condition, are automatically 
  fulfilled.
  
  Now consider a walk (in practice, a path or cycle) \(P = v_N e_N 
  v_{N-1} \dotsb e_1 v_0\) in $L$. This has a decomposition at 
  vertices \(v \leqslant 2+r+q\) into subpaths with all of their 
  edges from either $K$ or $H$. If in particular \(v_N,v_0 \leqslant 
  2+r+q\) then each such subpath gives rise to a $1$ in some 
  $a_{ij}$ or $b_{ij}$ matrix. Say a walk is \emph{odd} if all edges 
  in it are odd (i.e., come from $H$) and \emph{even} if all edges in 
  it are even (i.e., come from $K$). Then there is an odd path from 
  $2+j$ to $2+r+i$ if and only if the $(i,j)$ position in $b_{22}$ is 
  $1$, and there is an even path from $2+r+j$ to $2+i$ if and only if 
  the $(i,j)$ position in $a_{22}$ is $1$.
  
  Any cycle $P$ in $L$ must contain some edge from $H$ and some edge 
  from $K$, since neither $H$ nor $K$ contain a cycle within 
  themselves. Hence one can assume \(v_0 = v_N \leqslant 2+r\) for 
  $P$ that is a cycle. The matrix $a_{22}b_{22}$ has a $1$ in position 
  $(i,j)$ if and only if 
  there is an odd--even path combination that goes from vertex $2+j$ 
  to vertex $2+i$. For a cycle $P$ consisting of $2p$ subpaths 
  (alternatingly odd and even), it follows that the $(v_0 -\nobreak 
  2, v_0 -\nobreak 2)$ position of $(a_{22}b_{22})^p$ must be $1$. 
  Therefore any cycle in $L$ will give rise to a $1$ in the main 
  diagonal of $(a_{22}b_{22})^+$, but if $a_{22}b_{22}$ is nilpotent 
  then the main diagonal of $(a_{22}b_{22})^+$ is zero by 
  Lemma~\ref{L:Nilpotens}. Thus $L$ will be a network whenever 
  $a_{22}b_{22}$ is nilpotent.
  
  Finally consider \eqref{Eq:TrSymJoin}. Let $P$ be a path from $1$ to 
  $0$ with \(s(e_1) > l\) and \(g(e_N) > k\). Then $e_1$ and $e_N$ are 
  both odd. If all edges in $P$ are odd then it corresponds to a path 
  in $H$, and hence the $\bigl( g(e_N) -\nobreak k +\nobreak q, 
  s(e_1) -\nobreak l +\nobreak r \bigr)$ element in $\Tr(H)$, and 
  equivalently the $\bigl( g(e_N) -\nobreak k, s(e_1) -\nobreak l 
  \bigr)$ element in $b_{33}$, is $1$. Otherwise $P$ consists of some 
  $2p+1 \geqslant 3$ segments which are alternatingly odd and even; let 
  \(v_0 = u_0, u_1, \dotsc, u_{2p+1} = v_N\) be the segment boundary 
  vertices. 
  As stated above, this implies that the $(u_{2i} -\nobreak 2, u_{2i-1}
  -\nobreak 2 -\nobreak r)$ position of $a_{22}$ is $1$ for \(i \in 
  [p]\), and also that the $(u_{2i+1} -\nobreak 2 -\nobreak r, u_{2i} 
  -\nobreak 2)$ position of $b_{22}$ is $1$ for \(i \in [p -\nobreak 
  1]\). Similarly the $\bigl( u_1 -\nobreak 2 -\nobreak r, s(e_1) 
  -\nobreak l \bigr)$ position of $b_{23}$ and the 
  $\bigl( g(e_N) -\nobreak k, u_{2p} -\nobreak 2 \bigr)$ position of 
  $b_{32}$ must be $1$. Hence the $\bigl( g(e_N) -\nobreak k, s(e_1) 
  -\nobreak l \bigr)$ element in $b_{32} a_{22} (b_{22}a_{22})^{p-1} 
  b_{23}$ must be $1$. Since this holds for an arbitrary $P$, it 
  follows that $b_{33} + b_{32} a_{22} (b_{22}a_{22})^* b_{23}$ is an 
  upper bound for the lower right $m \times n$ submatrix of 
  $\Tr(L)$. From similar considerations in 
  the other three quadrants, it follows that the left hand side of 
  \eqref{Eq:TrSymJoin} is $\leqslant$ the right hand side.
  
  Conversely, any $1$ in the right hand side of \eqref{Eq:TrSymJoin} 
  corresponds to a $1$ in some position in one of the component 
  terms. Each such term is a product of alternatingly $a_{ij}$ and 
  $b_{ji}$ matrices, so that $1$ corresponds in turn to a sequence 
  $(u_p,u_{p-1})$, $(u_{p-1},u_{p-2})$, \dots, $(u_1,u_0)$ of 
  positions such that $(u_j,u_{j-1})$ is $1$ in the $j$'th factor 
  from the end. For each such position that is $1$, there is a 
  corresponding path in $K$ or $H$ respectively, and any path 
  corresponding to $(u_j,u_{j-1})$ joins up in $L$ with any path 
  corresponding to $(u_{j+1},u_j)$, at the common vertex $2+u_j$ (if 
  the factors are on the form $b_{i2}a_{2i'}$) or $2+r+u_j$ (if the 
  factors are on the form $a_{i2}b_{2i'}$). Joining up all segments 
  produces a path from $1$ to $0$ in $L$, so the left hand side of 
  \eqref{Eq:TrSymJoin} is $\geqslant$ the right hand side.
\end{proof}

The symmetric join has several nice algebraic properties.

\begin{lemma} \label{L:join-associativitet}
  Let an $\N^2$-graded set $\Omega$, not necessarily distinct \(y,z 
  \in \Omega(1,1)\), and \(G,H,K \in \Nw(\Omega)\) be given. Let
  \begin{equation*}
    \begin{bmatrix} A_{11} & A_{12} \\ A_{21} & A_{22} 
      \end{bmatrix} := \Tr(G) \text{,}\ 
    \begin{bmatrix} B_{22} & B_{23} & B_{24} \\ 
      B_{32} & B_{33} & B_{34} \\ B_{42} & B_{43} & B_{44}
      \end{bmatrix} := \Tr(H) \text{,}\ 
    \begin{bmatrix} C_{44} & C_{45} \\ C_{54} & C_{55} 
      \end{bmatrix} := \Tr(K) 
  \end{equation*}
  where the block sizes are such that $A_{11}$ is $i \times j$, 
  $A_{22}$ is $p \times q$, $B_{22}$ is $q \times p$, $B_{33}$ is 
  $k \times l$, $B_{44}$ is $r \times s$, $C_{44}$ is $s \times r$, and 
  $C_{55}$ is $m \times n$ for some \(i,j,k,l,m,n,p,q,r,s \in \N\).
  
  If $A_{22}B_{22}$, $B_{44}C_{44}$, and $(A_{22}B_{22})^* 
  A_{22}B_{24}C_{44} (B_{44}C_{44})^* B_{42}$ are nilpotent then both 
  sides of \eqref{Eq:join-associativitet} are networks of $\Omega$ and
  \begin{equation} \label{Eq:join-associativitet}
    (G \join[y]^p_q H) \join[z]^r_s K \simeq
    G \join[y]^p_q (H \join[z]^r_s K) \text{.}
  \end{equation}
\end{lemma}
\begin{proof}
  That $A_{22}B_{22}$ and $B_{44}C_{44}$ are nilpotent ensure that 
  $G \join[y]^p_q H$ and $H \join[z]^r_s K$ respectively are networks. 
  That also $(A_{22}B_{22})^* A_{22}B_{24}C_{44} (B_{44}C_{44})^* 
  B_{42}$ is nilpotent implies by Lemma~\ref{L:Matrisnilpotens} that 
  the matrices
  \begin{gather}
    \begin{bmatrix}
      A_{22}B_{22} & A_{22}B_{24}C_{44} \\
      B_{42} & B_{44}C_{44}
    \end{bmatrix}
    \\
    B_{44}C_{44} + B_{42}(A_{22}B_{22})^*A_{22}B_{24}C_{44}
    \label{Eq2:join-associativitet}
    \\
    A_{22}B_{22} + A_{22}B_{24}C_{44}(B_{44}C_{44})^*B_{42}
    \label{Eq3:join-associativitet}
  \end{gather}
  are nilpotent too. That \eqref{Eq2:join-associativitet} is 
  nilpotent is the condition for $(G \join[y]^p_q\nobreak H) 
  \join[z]^r_s K$ to be a network, whereas the condition for 
  $G \join[y]^p_q (H \join[z]^r_s\nobreak K)$ to be a network is that 
  \eqref{Eq3:join-associativitet} is nilpotent.
  
  The isomorphism $(\chi,\psi)$ from left hand side to right hand side 
  is immediate from the interpretation of $\join$ as mainly making a 
  combined copy of its factors:
  \begingroup
    \addtolength{\medmuskip}{-2mu plus 2mu minus -2mu}%
  \begin{align*}
    \psi(4e) ={}& 2e & e \in{}& E_G
      \text{,}\\
    \psi(4e+2) ={}& 4e+1 & e \in{}& E_H
      \text{,}\\
    \psi(2e+1) ={}& 4e+3 & e \in{}& E_K
      \text{,} \displaybreak[0]\\
    \chi\bigl( r+s + 2(p+q + 2v) \bigr) ={}&
      p+q + 2v & v \in{}& V_G \setminus\{0,1\} 
      \text{,} \displaybreak[0]\\
    \chi\bigl( r+s + 2(2 + i) \bigr) ={}& 2+i &
      i \in{}& [p+q]
      \text{,} \displaybreak[0]\\
    \chi\bigl( r+s + 2(p+q + 2v + 1) \bigr) ={}&
      p+q + 2(r+s + 2v) + 1 & v \in{}& V_H \setminus\{0,1\} 
      \text{,} \displaybreak[0]\\
    \chi( 2 + i ) ={}& p+q + 2(2+i) &
      i \in{}& [r+s]
      \text{,} \displaybreak[0]\\
    \chi\bigl( r+s + 2v + 1 \bigr) ={}&
      p+q + 2(r+s + 2v + 1) + 1 & v \in{}& V_K \setminus\{0,1\} 
      \text{,}\\
    \chi(i) ={}& i & i \in{}& \{0,1\}
      \text{.}
  \end{align*}
  \endgroup
  To verify that this really is an isomorphism, one would on one hand 
  check that it preserves incidences within each part, and on the 
  other check that the parts are joined in the same way in the left 
  and right hand sides. As an example of the latter, one may consider 
  some input $e$ of $H$, which corresponds to edge $2e+1$ in $G 
  \join[y]^p_q H$, edge $4e+2$ in $(G \join[y]^p_q\nobreak H) 
  \join[z]^r_s K$, edge $2e$ in $H \join[z]^r_s K$, and edge \(4e+1 = 
  \psi( 4e +\nobreak 2 )\) in $G \join[y]^p_q (H \join[z]^r_s\nobreak 
  K)$. If \(s_H(e) \leqslant p\), then the head of that edge is the 
  input vertex $1$ in $H$ and $H \join[z]^r_s K$, but a join-vertex 
  decorated by $y$ in the other three, carrying the label $2 + 
  s_H(e)$ in $G \join[y]^p_q H$, $r+s + 2\bigl( 2 +\nobreak s_H(e) 
  \bigr)$ in $(G \join[y]^p_q\nobreak H) \join[z]^r_s K$, and again 
  \( 2 + s_H(e) = \chi\bigl( r+s + 2( 2 +\nobreak s_H(e) ) \bigr) \) 
  in $G \join[y]^p_q (H \join[z]^r_s\nobreak K)$. 
  If \(p < s_H(e) \leqslant p+k\) then the tail is the input vertex 
  $1$ in all five networks, but the tail index varies: it is $s_H(e)$ 
  in $H$ and $H \join[z]^r_s K$, but $i + s_K(e) - k$ in 
  $G \join[y]^p_q H$, $(G \join[y]^p_q\nobreak H) \join[z]^r_s K$, and 
  $G \join[y]^p_q (H \join[z]^r_s\nobreak K)$. 
  Finally if if \(s_H(e) > p+k\) then it is in $H$ and $G 
  \join[y]^p_q H$ that the tail is the input vertex $1$, and in the 
  other cases it is a join-vertex decorated by $z$, namely that 
  labelled $2 + r + s_H(e)-p-k$ in $H \join[z]^r_s K$ and 
  $(G \join[y]^p_q\nobreak H) \join[z]^r_s K$, but \(p+q + 2\bigl( 
  2 +\nobreak r +\nobreak s_H(e) -\nobreak p -\nobreak k \bigr) = 
  \chi\bigl( 2 +\nobreak r +\nobreak s_H(e) -\nobreak p -\nobreak k 
  \bigr)\) in $G \join[y]^p_q (H \join[z]^r_s\nobreak K)$.
\end{proof}

\begin{lemma} \label{L:Symjoin-transposition}
  For every $\N^2$-graded set $\Omega$, numbers \(k,l,m,n,q,r \in 
  \N\), networks \(K \in \Nw(\Omega)(k +\nobreak r, l +\nobreak q)\), 
  \(H \in \Nw(\Omega)( q +\nobreak m, r +\nobreak n )\), and \(y \in 
  \Omega(1,1)\), it holds that
  \begin{equation} \label{Eq1:Symjoin-transposition}
    \cross{k}{m} \cdot ( K \join[y]^r_q H ) \cdot \cross{n}{l}
    \simeq
    (\cross{q}{m} \cdot H \cdot \cross{n}{r}) \join[y]^q_r
      (\cross{k}{r} \cdot K \cdot \cross{q}{l})
  \end{equation}
  whenever either symmetric join is defined.
\end{lemma}
\begin{remark}
  On one hand, this lemma states a symmetric join can be moved out of 
  (some) permutations, so it is a kind of distributivity. On the other, 
  it states that
  $$
    \begin{mpgraphics*}{58}
      beginfig(58);
        PROPdiagram(0,0) 
          tightcross(0.3,2,1,1) circ
          frame (
          box(2,2)(btex \strut$a$ etex) symjoinIi 
            box(2,2)(btex \strut$b$ etex)
          )
          circ tightcross(0.3,2,1,1)
        ;
      endfig;
    \end{mpgraphics*}
    \quad\text{is isomorphic to}\quad
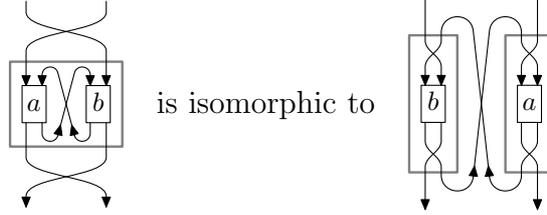
    \begin{mpgraphics*}{59}
      beginfig(59);
        PROPdiagram(0,0) 
          same(2) circ frame( 
            cross(1,1) circ box(2,2)(btex \strut$b$ etex) circ cross(1,1)
          ) symjoinIi frame(
            cross(1,1) circ box(2,2)(btex \strut$a$ etex) circ cross(1,1)
          ) circ same(2)
        ;
      endfig;
    \end{mpgraphics*}
  $$
  and these diagrams are schematic in the sense that any ``wire'' 
  above can be replaced by any number of parallel ``wires'' (as long 
  as these do not cross among themselves).
\end{remark}
\begin{proof}
  Writing
  \begin{align*}
    \Tr(K) ={}& 
    \begin{bmatrix} A_{11} & A_{12} \\ A_{21} & A_{22} \end{bmatrix}
    \text{,}& 
    \Tr(H) ={}& 
    \begin{bmatrix} B_{22} & B_{23} \\ B_{32} & B_{33} \end{bmatrix}
  \end{align*}
  one finds
  \begin{align*}
    \Tr(\cross{k}{r} \cdot K \cdot \cross{q}{l}) ={}& 
    \begin{bmatrix} A_{22} & A_{21} \\ A_{12} & A_{11} \end{bmatrix}
    \text{,}& 
    \Tr(\cross{q}{m} \cdot H \cdot \cross{n}{r}) ={}& 
    \begin{bmatrix} B_{33} & B_{32} \\ B_{23} & B_{22} \end{bmatrix}
  \end{align*}
  meaning the condition for either side to be defined is that 
  $A_{22}B_{22}$ is nilpotent.
  
  The isomorphism $(\chi,\psi)$ from left to right hand side of 
  \eqref{Eq1:Symjoin-transposition} is a straightforward relabelling; 
  in the left hand side the $K$ edges are even and the $H$ edges are 
  odd, whereas in the right hand side it is the other way around, so 
  \(\psi(e) = e + (-1)^e\) (add $1$ if even, subtract $1$ if odd). 
  For the vertices it is given by
  \begin{align*}
    \chi(v) ={}& v & \text{for } v \in{}& \{0,1\} \text{,}\\
    \chi(2+i) ={}& 2+q+i & \text{for } i \in{}& [r] \text{,}\\
    \chi(2+r+i) ={}& 2+i & \text{for } i \in{}& [q] \text{,}\\
    \chi(r+q+2v) ={}& q+r+2v+1 & 
      \text{for } v \in{}& V_K \setminus\{0,1\} \text{,}\\
    \chi(r+q+2v+1) ={}& q+r+2v & 
      \text{for } v \in{}& V_H \setminus\{0,1\} \text{.}
  \end{align*}
  The thing that needs to be checked is primarily that both sides 
  attach $K$- and $H$-legs in the same way.
  
  Suppose \(e \in E_K\) has \(h_K(e) = 0\). In the left hand side, 
  $2e$ will go with head index $1$ to the join-vertex $2 + g_K(e)-k$ 
  if \(g_K(e) > k\), and to the output vertex $0$ with head index 
  \(\cross{k}{m}\bigl( g_K(e) \bigr) = m + g_K(e)\) if \(g_K(e) 
  \leqslant k\). In the right hand side, if \(\cross{k}{r}\bigl( 
  g_K(e) \bigr) > r\) (i.e., if \(g_K(e) \leqslant k\)) then \(2e+1 
  = \psi(e)\) will go to the output vertex $0$ with head index 
  \(m + \cross{k}{r}\bigl( g_K(e) \bigr) - k = m + g_K(e)\), and if 
  \(\cross{k}{r}\bigl( g_K(e) \bigr) \leqslant r\) (i.e., if \(g_K(e) 
  > k\)) then \(2e+1 = \psi(e)\) will go with head index $1$ to the 
  join-vertex \(2+q + \cross{k}{r}\bigl( g_K(e) \bigr) = 2+q + 
  g_K(e)-k = \chi\bigl( 2 +\nobreak g_K(e) -\nobreak k \bigr)\). The 
  other three cases (out-legs of $H$, in-legs of $K$, and in-legs of 
  $H$) are checked in the same way. For the in-legs, it may be useful 
  to observe that e.g.~\((\cross{q}{l})^{-1} = \cross{l}{q}\).
\end{proof}

\subsection{Homeomorphisms}

The beauty of algebraic identities aside, it is time to get back to 
the matter of network subexpressions, because the symmetric join by 
itself does not completely define such a concept. When \(B = CA\) in 
naked abstract index notation, there are typically some indices 
that are common to $C$ and $A$, but in \(L = K \join^r_q H\) there 
is no corresponding commonality of edges. Instead an edge from $K$ is 
joined with an edge from $H$ by means of a new neutral vertex, which 
in abstract index notation would correspond to doing all joining by 
means of Kronecker delta factors. What more is needed for a network 
subexpression concept is thus a way of getting at what the combined 
network would be if such Kronecker factors were eliminated.

\begin{definition}
  A \DefOrd[*{homeomorphism}]{network homeomorphism} from
  \begin{equation*}
    L = (V_L,E_L,h_L,g_L,t_L,s_L,D_L)
    \quad\text{to}\quad 
    G = (V_G,E_G,h_G,g_G,t_G,s_G,D_G)
  \end{equation*}
  is a pair $(\beta,\gamma)$, where \(\gamma\colon E_L \Fpil E_G\) and 
  \(\beta\colon V_G \Fpil V_L\) (note reversal) are such that:
  \begin{enumerate}
    \item \(\beta(0) = 0\), \(\beta(1)=1\), and $\beta$ is injective. 
    \item $\gamma$ is surjective.
    \item 
      \(D_L\bigl( \beta(v) \bigr) = D_G(v)\) for all \(v \in V_G 
      \setminus \{0,1\}\).
    \item
      If \(e \in E_L\) satisfies \(h_L(e) \in \setim\beta\) then 
      \(h_L(e) = (\beta \circ\nobreak h_G \circ\nobreak \gamma)(e)\) and 
      \(g_L(e) = (g_G \circ\nobreak \gamma)(e)\).
      If \(e \in E_L\) satisfies \(t_L(e) \in \setim\beta\) then 
      \(t_L(e) = (\beta \circ\nobreak t_G \circ\nobreak \gamma)(e)\) and 
      \(s_L(e) = (s_G \circ\nobreak \gamma)(e)\).
    \item
      If \(v \in V_L \setminus \setim\beta\) then 
      \(\alpha\bigl( D_L(v) \bigr) = \omega\bigl( D_L(v) \bigr) = 1\) 
      and \(\gamma(e)=\gamma(f)\) for the \(e,f \in E_L\) such that 
      \(h_L(e) = v = t_L(f)\).
  \end{enumerate}
  A \DefOrd[*{homeomorphism!$y$-homeomorphism}]{$y$-homeomorphism} 
  is a homeomorphism such that \(D_L(v) = y\) for all 
  \(v \in V_L \setminus \setim\beta\). A 
  \DefOrd[*{subdivision}]{$y$-subdivision} of a network $G$ is 
  some network $L$ from which there exists a $y$-homeomorphism to $G$.
\end{definition}

\begin{lemma} \label{L:UnderdeladKant}
  Let $(\beta,\gamma)$ be a network homeomorphism from
  \begin{equation*}
    L = (V_L,E_L,h_L,g_L,t_L,s_L,D_L)
    \quad\text{to}\quad 
    G = (V_G,E_G,h_G,g_G,t_G,s_G,D_G)
  \end{equation*}
  and let \(X = \setim\beta \subseteq V_L\). 
  For every \(e \in E_G\), the set $\setinv{\gamma}\bigl( \{e\} \bigr)$ 
  has the form $\{e_1,\dotsc,e_n\}$ where \(t_L(e_1) \in X\), 
  \(h_L(e_i) = t_L(e_{i+1}) \notin X\) for \(i \in [n -\nobreak 1]\), 
  and \(h_L(e_n) \in X\).
\end{lemma}
\begin{proof}
  Let \(E = \setinv{\gamma}\bigl( \{e\} \bigr)\). If \(v \in X \cap 
  \setmap{t_L}(E)\) then \(v = \beta\bigl( t_H(e) \bigr)\) by 
  definition of homeomorphism. Furthermore if \(f \in E\) is such 
  that \(t_L(f) \in X\) then \(s_L(f) = s_G(e)\) by definition of 
  homeomorphism. Since \(t_L(f) = \beta\bigl( t_G(e) \bigr)\) and 
  \(s_L(f) = s_G(e)\) uniquely identify $f$ as an edge of $L$, 
  there can be at most one \(f \in E\) with \(t_L(f) \in X\); this 
  will be the edge \(e_1 \in E\).
  
  In order to exhibit this $e_1$, one may pick an arbitrary \(f_1 \in 
  E\) and proceed iteratively as follows: If \(t_L(f_i) \in X\) then 
  \(e_1 = f_i\), and otherwise \((\alpha \circ\nobreak D_L 
  \circ\nobreak t_L)(f_i) = 1\), meaning there is a unique \(f_{i+1} 
  \in E_L\) such that \(h_L(f_{i+1}) = t_L(f_i)\). By acyclicity of 
  the network $L$, \(f_{i+1} \notin \{f_1,\dotsc,f_i\}\), and by 
  definition of homeomorphism \(\gamma(f_{i+1})=e\), so \(f_{i+1} \in 
  E\). Since $E$ is finite, this will produce $e_1$ after a finite 
  number of steps. Then one may work in the opposite direction and 
  define \(e_{i+1} \in E_L\) to be the unique edge with 
  \(t_L(e_{i+1}) = h_L(e_i)\), for as long as \(h_L(e_i) \notin X\); 
  again uniqueness is because \((\omega \circ\nobreak D_L 
  \circ\nobreak h_L)(e_i) = 1\), \(e_{i+1} \in E\) by definition 
  of homeomorphism since \(h_L(e_i) \notin X\), and eventually there 
  is some $e_n$ for which \(h_L(e_n) \in X\) since $E$ is finite and 
  $L$ contains no cycles.
  
  The above has exhibited a sequence \(e_1,\dotsc,e_n \in E\) 
  satisfying the stated conditions for heads and tails, so what 
  remains is only to establish that these are all the edges in $E$. 
  However, had there been some \(f \in E \setminus 
  \{e_1,\dotsc,e_n\}\) then even when taking $f$ as $f_1$ one would 
  still have to arrive at the same $e_1$, so not being able to 
  conversely reach $f$ from $e_1$ would contradict the network nature 
  of $L$.
\end{proof}

It may be observed that homeomorphisms can be composed to form new 
homeomorphisms; the composition of \((\beta_1,\gamma_1)\colon L \Fpil 
G\) and \((\beta_2,\gamma_2)\colon G \Fpil H\) is \( (\beta_1 
\circ\nobreak \beta_2, \gamma_2 \circ\nobreak \gamma_1)\colon L \Fpil 
H\). It follows that the networks constitute a category with 
homeomorphisms as morphisms. It may also be observed that every 
isomorphism gives rise to an invertible homeomorphism: if 
$(\chi,\psi)$ is an isomorphism from $G$ to $H$ then $(\chi^{-1},\psi)$ 
is a homeomorphism from $G$ to $H$. Conversely, a pair of 
homeomorphisms with common target will give rise to an isomorphism if 
they subdivide each edge the same number of times.

\begin{corollary} \label{Kor:Homeomorfilyft}
  Let \(G = (V_G,E_G,h_G,g_G,t_G,s_G,D_G)\) be a network, and let 
  $L_1$ and $L_2$ be $y$-subdivisions of $G$ for some $y$. If the 
  homeomorphisms $(\beta_1,\gamma_1)$ from $L_1$ to $G$ and 
  $(\beta_2,\gamma_2)$ from $L_2$ to $G$ satisfy 
  \(\card[\big]{\setinv{\gamma_1}(\{e\})} = 
  \card[\big]{\setinv{\gamma_2}(\{e\})}\) for all \(e \in E_G\), then 
  there is a unique isomorphism $(\chi,\psi)$ from $L_1$ to $L_2$ 
  such that \(\chi \circ \beta_1 = \beta_2\).
\end{corollary}
\begin{proof}
  The condition \(\chi \circ \beta_1 = \beta_2\) defines $\chi$ on 
  $\setim\beta_1$, i.e., those vertices that are endpoints of those 
  paths that by Lemma~\ref{L:UnderdeladKant} constitute the inverse 
  image of an edge in $G$. Since these paths have the same length in 
  $L_2$ as in $L_1$, there is a consistent extension of $\chi$ to all 
  vertices and a related definition of $\psi$ which is to map the 
  $i$th edge of $\setinv{\gamma_1}(\{e\})$ to the $i$th edge of 
  $\setinv{\gamma_2}(\{e\})$, for each \(e \in E_G\). The annotations 
  match because all vertices not in $\setim\beta_1$ or 
  $\setim\beta_2$ respectively have annotation $y$.
\end{proof}

Continuing to explore the relation between isomorphisms and 
homeomorphisms, one may observe that having a common subdivision can 
be a way of showing that two networks are isomorphic.

\begin{lemma} \label{L:Dropping}
  Let $\Omega'$ be an $\N^2$-graded set with some \(y \in 
  \Omega'(1,1)\), and define \(\Omega = \Omega' \setminus\{y\}\). 
  Suppose
  \begin{align*}
    L ={}& (V,E,h,g,t,s,D) 
      && \text{is a network of $\Omega'$,}\\
    G_1 ={}& (V_1,E_1,h_1,g_1,t_1,s_1,D_1)
      && \text{is a network of $\Omega$, and}\\
    G_2 ={}& (V_2,E_2,h_2,g_2,t_2,s_2,D_2)
      && \text{is a network of $\Omega$.}
  \end{align*}
  If there are $y$-homeomorphisms \((\beta_1,\gamma_1)\colon L \Fpil 
  G_1\) and \((\beta_2,\gamma_2)\colon L \Fpil G_2\), then \(G_1 
  \simeq G_2\).
\end{lemma}
\begin{proof}
  Since $G_i$ is a network of \(\Omega \not\owns y\) and 
  $(\beta_i,\gamma_i)$ is a $y$-homeomorphism, it follows that \(D(v) 
  = y\) for some \(v \in V\) if and only if \(v \notin 
  \setim\beta_i\). Since $\beta_i$ furthermore is injective, it 
  follows that \(\chi = \beta_2^{-1} \circ \beta_1\) is a bijection 
  from $V_1$ to $V_2$, which satisfies \(\chi(0)=0\), \(\chi(1)=1\), 
  and \(D_1 = D_2 \circ \chi\). For the corresponding edge map 
  \(\psi\colon E_1 \Fpil E_2\) it is natural to use the definition 
  \(\psi\bigl( \gamma_1(e) \bigr) = \gamma_2(e)\) for all \(e \in E\), 
  as the fact that $(\beta_1,\gamma_1)$ and $(\beta_2,\gamma_2)$ are 
  homeomorphisms immediately fulfils all the conditions for 
  $(\chi,\psi)$ to be an isomorphism, provided that this $\psi$ is 
  well-defined and a bijection.
  
  To that end, let \(e,e' \in E\) such that \(\gamma_1(e) = 
  \gamma_1(e')\) be given. By Lemma~\ref{L:UnderdeladKant} there are 
  then \(e_1,\dotsc,e_n \in E\) such that \(h(e_k)=t(e_{k+1}) \notin 
  \setim\beta_1\) for all \(k \in [n -\nobreak 1]\) and \(i,j \in 
  [n]\) such that \(e_i = e\) and \(e_j = e'\). Hence there is 
  between $e$ and $e'$ a sequence of vertices not in $\setim\beta_2$, 
  and this implies \(\gamma_2(e) = \gamma_2(e')\). Thus $\psi$ is 
  indeed well-defined. Exchanging the roles of $\gamma_1$ and 
  $\gamma_2$ in this argument demonstrates that $\psi$ is invertible, 
  so it is also a bijection.
\end{proof}

But let's get back to the subexpressions.
As it should be, removing ``Kronecker delta factors'' from a network 
will not change its value.

\begin{lemma} \label{L:evalNatural}
  Let a network $G$ of $\Omega$ and a $y$-subdivision $L$ of $G$ be 
  given, where \(y \in \Omega(1,1)\). If $\mc{P}$ is a \PROP\ and 
  \(f\colon \Omega \Fpil \mc{P}\) is an $\N^2$-graded set morphism 
  such that \(f(y) = \phi_{\mc{P}}(\same{1})\), then \(\eval_f(L) = 
  \eval_f(G)\).
\end{lemma}
\begin{proof}
  Let \((\beta,\gamma)\colon L \Fpil G\) be the homeomorphism. Without 
  loss of generality, it may be assumed that \(\beta(v) = v\) for all 
  \(v \in V(G) \setminus\{0,1\}\), since any $L$ would be isomorphic 
  to a network $L'$ satisfying this condition and network 
  isomorphisms preserve the value. Furthermore it is sufficient to 
  consider the case that $L$ has one vertex more than $G$, since the 
  general equality is obtained by adding the vertices one by one.
  
  Let \(v \in V(L) \setminus V(G)\) be the new vertex that had been 
  added. Let \(W_0 \subset V(L)\) be the set of inner vertices of $L$ 
  reachable from $v$ via a path of length at least $1$ (meaning \(v 
  \notin W_0\)) and let \(W_1 = V(G) \setminus W_0 \setminus 
  \{0,1\}\); then $(W_0,W_1)$ is a cut in $G$ and $\bigl( W_0, \{v\} 
  \cup\nobreak W_1 \bigr)$ is a cut in $L$. Let \(e_0,e_1 \in E(L)\) 
  be the unique edges such that \(t(e_0) = v = h(e_1)\). Pick some 
  ordering $p$ of the cut $(W_0,W_1)$ such that \(p\bigl( \gamma(e_0) 
  \bigr) = 1\). Then $\bigl( W_0, \{v\} \cup\nobreak W_1, 
  p \circ\nobreak \gamma \bigr)$ is an ordered cut in $L$ which induces 
  a decomposition $(L_0,L_1)$ of $L$. Moreover, $\bigl( \{v\}, W_1, 
  p \circ\nobreak \gamma \bigr)$ is an ordered cut in $L_1$ which 
  induces a decomposition $(L_{10},L_{11})$ where \(L_{10} = 
  \Nwfuse{e_0 \vek{f} }{ (v\colon y)^{e_0}_{e_1} }{ e_1 \vek{f} }\) for 
  some list of edges $\vek{f}$. Additionally \(L_0 \simeq G_0\) and 
  \(L_{11} \simeq G_1\) where $(G_0,G_1)$ is the decomposition of $G$ 
  induced by $(W_0,W_1,p)$. Hence
  \begin{multline*}
    \eval_f(L) = 
    \eval_f(L_0) \circ \eval_f(L_1) =
    \eval_f(L_0) \circ \eval_f(L_{10}) \circ \eval_f(L_{11}) = \\ =
    \eval_f(G_0) \circ f(y) \otimes \phi(\same{\Norm{\vek{f}}}) \circ 
      \eval_f(G_1) =
    \eval_f(G_0) \circ \phi(\same{1+\Norm{\vek{f}}}) \circ 
      \eval_f(G_1) = \\ =
    \eval_f(G_0) \circ \eval_f(G_1) =
    \eval_f(G)
  \end{multline*}
  as claimed.
\end{proof}

\subsection{Embeddings}

A direct translation of the abstract index subexpression concept to 
networks is thus: \emph{$H$ is a subexpression of $G$ if there is 
some $K$ such that $K \rtimes H$ is a $\natural$-subdivision of $G$}. 
This is practical as a method of enumerating the networks $G$ which 
has a given $H$ as a subexpression, but it is not so practical for 
testing whether a specific $G$ contains $H$. Luckily, there is an 
alternative characterisation in terms of networks that makes the 
latter straightforward.

\begin{definition}
  An \DefOrd{embedding} of a network \(H = (V',E',h',g',t',s',D')\) 
  into a network \(G = (V,E,h,g,t,s,D)\) is a pair $(\chi,\psi)$ of 
  maps \(\chi\colon V'\setminus\{0,1\} \Fpil V\setminus\{0,1\}\) and 
  \(\psi\colon E' \Fpil E\) such that:
  \begin{enumerate}
    \item
      $\chi$ is injective.
    \item
      \(D' = D \circ \chi\).
    \item
      If \(e \in E'\) satisfies \(h'(e) \neq 0\) then 
      \(h\bigl( \psi(e) \bigr) = \chi\bigl( h'(e) \bigr)\) and \(g(e) 
      = g'(e)\).
    \item
      If \(e \in E'\) satisfies \(t'(e) \neq 1\) then 
      \(t\bigl( \psi(e) \bigr) = \chi\bigl( t'(e) \bigr)\) and 
      \(s(e)=s'(e)\).
    \item
      If \(\psi(e) = \psi(f)\) for \(e,f \in E'\) then (i)~\(e=f\), 
      (ii)~\(h'(e)=0\) and \(t'(f)=1\), or (iii)~\(t'(e)=1\) and 
      \(h'(f)=0\). In cases (ii) and~(iii), the legs $e$ and $f$ are 
      said to be \emDefOrd{joined} by the embedding.
  \end{enumerate}
\end{definition}

An embedding is basically that each inner vertex of $H$ should be mapped 
injectively to an equally decorated inner vertex of $G$, while 
preserving all incidences from $H$. A convex subnetwork $H$ of $G$ has 
a trivial embedding into $G$, where $\chi$ and $\psi$ both map every 
element in their domains to themselves. For testing whether $H$ has an 
embedding into $G$, one may practically start by picking some vertex 
$v$ of $H$ and merely go through all possibilities for $\chi(v)$, 
i.e., those vertices of $G$ which has the same decoration as $v$. 
Fixing any $\chi(v)$ immediately determines $\psi(e)$ for all edges 
$e$ incident with $v$, and conversely fixing $\psi(e)$ determines 
$\chi$ for both endpoints of $e$. Hence the way an entire component 
of $H$ is embedded gets fixed after choosing just how to embed a 
single vertex, so the search tree that has to be examined is usually 
very small\Dash much smaller than in the corresponding problem for 
ordinary graphs.

%

\begin{lemma} \label{L:Inbaddning}
  Let an $\N^2$-graded set $\Omega'$ and some \(y \in \Omega(1,1)\) 
  be given. If \(G,H,K \in \Nw\bigl( \Omega' \setminus\nobreak \{y\} 
  \bigr)\) are such that $K \overset{y}\rtimes H$ is a $y$-subdivision of 
  $G$ then there is an embedding of $H$ into $G$.
\end{lemma}
\begin{proof}
  To fix notation, let
  \begin{align*}
    (V_G,E_G,h_G,g_G,t_G,s_G,D_G) :={}& G \text{,}\\
    (V_H,E_H,h_H,g_H,t_H,s_H,D_H) :={}& H \text{,}\displaybreak[0]\\
    (V_K,E_K,h_K,g_K,t_K,s_K,D_K) :={}& K \text{,}\\
    L = (V_L,E_L,h_L,g_L,t_L,s_L,D_L) :={}& K \rtimes H \text{,}
  \end{align*}
  and let $(\beta,\gamma)$ be the homeomorphism from $L$ to $G$. Let 
  \(q=\omega(H)\) and \(r=\alpha(H)\). The embedding $(\chi,\psi)$ 
  of $H$ in $G$ is then given by
  \begin{align*}
    \chi(v) ={}& \beta^{-1}( r+q + 2v + 1 ) &&
      \text{for \(v \in V_H \setminus\{0,1\}\),}\\
    \psi(e) ={}& \gamma(2e+1) &&
      \text{for \(e \in E_H\),}
  \end{align*}
  where $r+q + 2v + 1$ is known to be in the image of $\beta$ since 
  \(D_L( r +\nobreak q +\nobreak 2v +\nobreak 1) = D_H(v)\) by the 
  definition of $L$ and \(D_H(v) \neq y\). That $\chi$ is injective 
  is clear from the construction, and its compatibility with 
  annotations follows from the similar properties for homeomorphisms 
  and symmetric joins. The incidence conditions for  \(e \in E_H\) 
  satisfying \(h_H(e) \neq 0\) and \(t_H(e) \neq 1\) respectively are 
  established by chasing the commutative diagram
  \begin{equation*}
    \begin{CD}
      E_H @>{e \mapsto 2e+1}>> E_L @>{\gamma}>> E_G \\
      @V{h_H}V{t_H}V  @V{h_L}V{t_L}V    @V{h_G}V{t_G}V  \\
      V_H @>>{v \mapsto r+q+2v+1}> V_L @<<{\beta}< V_G
    \end{CD}
  \end{equation*}
  Finally, if \(\psi(e) = \psi(f)\) for \(e \neq f\) then this must 
  be because \(\gamma(2e +\nobreak 1) = \gamma(2f +\nobreak 1)\). By 
  Lemma~\ref{L:UnderdeladKant}, \(\setinv{\gamma}\bigl( \bigl\{ \gamma(2e 
  +\nobreak 1) \bigr\} \bigr) = \{e_1,\dotsc,e_n\}\) where \(2 < 
  h_L(e_i) = t_L(e_{i+1}) \leqslant 2+r+q\) for \(i \in [n -\nobreak 
  1]\); since \(2e+1 \neq 2f+1\), it follows that \(n \geqslant 2\). 
  Having \(h_H(e) \neq 0\) would make \(h_L(2e +\nobreak 1) > 2+r+q\) 
  and thus imply \(2e+1 = e_n\), in which case \(2 < t_L(2e +\nobreak 
  1) \leqslant 2+r+q\) making \(t_H(e)=1\) and moreover \(2f+1 = 
  e_i\) for some \(i<n\) so that \(2 < h_L(2f +\nobreak 1) \leqslant 
  2+r+q\) and thus \(h_H(f) = 0\). Similarly having \(t_H(f) \neq 1\) 
  implies \(2f+1 = e_1\), from which follows \(t_H(e) = 1\) and 
  \(h_H(f) = 0\). Hence \(h_H(e)=0\) and \(t_H(f)=1\), or \(t_H(e) = 1\) 
  and \(h_H(f) = 0\), which is exactly the final condition for an 
  embedding.
\end{proof}

A downside with the straightforward embedding concept is that it does 
not always uniquely identify the padding $K$\Dash not even up to 
isomorphism, even though it gets close. The case where it throws away 
information is that $H$ has two or more edges directly connecting 
input to output, and two of these are mapped to the same edge in $G$; 
in this case the relative order of those two isolated $H$-edges is 
not recorded in the embedding, whereas it would be apparent in $K 
\rtimes H$. To remedy this, it is appropriate to introduce a slightly 
stronger kind of embedding.

\begin{definition}
  Given two networks \(H = (V_H,E_H,h_H,g_H,t_H,s_H,D_H)\) and
  \(G = (V_G,E_G,h_G,g_G,t_G,s_G,D_G)\) where \(E_G \subseteq \N\), a 
  \DefOrd{strong embedding} of a $H$ into $G$ is a triplet 
  $(\chi,\psi,m)$ where \(\chi\colon V_H\setminus\{0,1\} \Fpil 
  V_G\setminus\{0,1\}\), \(\psi\colon E_H \Fpil \N\), and \(m \in 
  \Zp\) are such that:
  \begin{enumerate}
    \item
      $(\chi,R_m \circ \psi)$ is an embedding of $H$ into $G$, 
      where \(R_m\colon \N \Fpil \N\) denotes the remainder function for 
      integer division by $m$; \(R_m(k) = k \bmod m < m\).
    \item
      $\psi$ is injective.
    \item
      \(e < m\) for all \(e \in E_G\).
    \item
      If \(e,f \in E_H\) are such that \(\psi(e) \equiv \psi(f) 
      \pmod{m}\) and \(t_H(e) \neq 1\) then \(\psi(e) \leqslant 
      \psi(f)\).
    \item
      If \(e,f \in E_H\) are such that \(\psi(e) \equiv \psi(f) 
      \pmod{m}\) and \(h_H(e) \neq 0\) then \(\psi(e) \geqslant 
      \psi(f)\).
  \end{enumerate}
  The values of $\psi$ are called \emDefOrd{segment labels}, or 
  sometimes \emph{$H$-segments}.
\end{definition}

The idea is to (i)~make $\psi$ injective by adding some multiple of 
$m$ to the basic $G$-edge labels that it has as values and (ii)~encode 
the relative order of the $H$-segments within a $G$-edge into these 
multiples, so that the multiple increases when going from tail to 
head. The only conditions that have to be imposed are to ensure that 
segments that have a $H$-vertex at one end fall constitute an 
endpoint for the range of multiples for that $G$-edge; for other 
edges, every assignment of multiples actually means something.

\begin{lemma} \label{L:StarkInbaddning}
  Let \(G = (V_G,E_G,h_G,g_G,t_G,s_G,D_G) \in \Nw(\Omega)\) and \(H = 
  (V_H,E_H,h_H,g_H,t_H,s_H,D_H) \in \Nw(\Omega)\) together with an 
  embedding $(\chi,\psi')$ of $H$ into $G$ be given. Then for any \(m 
  \in \Zp\) with \(m > e\) for all \(e \in E_G\), there exists a 
  strong embedding $(\chi,\psi,m)$ of $H$ into $G$ such that 
  \(\psi(e) \equiv \psi'(e) \pmod{m}\) for all \(e \in E_H\).
\end{lemma}
\begin{proof}
  Merely take \(\psi(e) = \psi'(e) + m\theta(e)\) for some suitable 
  \(\theta\colon E_H \Fpil \N\). Defining \(E_e = \setinv{\psi}\bigl( 
  \{ \psi(e) \} \bigr)\) for \(e \in E_H\), the conditions 
  are automatically fulfilled provided for example that the restriction 
  of $\theta$ to any $E_e$ is injective, \(\theta(e) < \card{E_e}\) 
  for all \(e \in E_H\), \(\theta(e) = 0\) if \(t_H(e) \neq 1\), and 
  \(\theta(e) = \card{E_e} - 1\) if \(h_H(e) \neq 0\). This is 
  trivial to arrange.
\end{proof}

The final result in this section is that the embedding 
characterisation of a subexpression is in fact completely equivalent 
to that using the symmetric join.

\begin{theorem} \label{S:InbaddningHomeomorfi}
  Let an $\N^2$-graded set $\Omega'$ and \(y \in \Omega'(1,1)\) be 
  given. Let \(\Omega = \Omega' \setminus \{y\}\). For \(G,H \in 
  \Nw(\Omega)\), the following are equivalent:
  \begin{enumerate}
    \item \label{Item1:InbaddningHomeomorfi}
      There exists an embedding of $H$ into $G$.
    \item \label{Item5:InbaddningHomeomorfi}
      There exists a strong embedding of $H$ into $G$.
    \item \label{Item2:InbaddningHomeomorfi}
      There exists some \(K \in \Nw(\Omega)\) for which there is a 
      $y$-homeomorphism from $K \overset{y}\rtimes H$ to $G$.
  \end{enumerate}
\end{theorem}
\begin{proof}
  That \ref{Item1:InbaddningHomeomorfi} implies 
  \ref{Item5:InbaddningHomeomorfi} is the claim of 
  Lemma~\ref{L:StarkInbaddning}, and that 
  \ref{Item2:InbaddningHomeomorfi} implies 
  \ref{Item1:InbaddningHomeomorfi} is the claim of 
  Lemma~\ref{L:Inbaddning}. Hence what remains is to, from a given 
  strong embedding $(\chi,\psi,m)$ of $H$ into $G$, construct the 
  symmetric join \(L = K \overset{y}\rtimes H\) that is a 
  $y$-subdivision of $G$.
  Not surprisingly, the first step in that will be to construct the 
  ``padding'' network \(K = (V_K,E_K,h_K,g_K,t_K,s_K,D_K)\) from the 
  given data \(G = (V_G,E_G,h_G,g_G,t_G,s_G,D_G)\) and \(H = 
  (V_H,E_H,h_H,g_H,t_H,s_H,D_H)\), but when digesting this 
  construction it is more instructive to think about $K$ as already a 
  part of the symmetric join $L$.
  
  The vertices of $K$ are simply going to be the vertices of $G$ that 
  are not in the image of the embedding, so
  \begin{equation*}
    V_K = V_G \setminus \setim\chi
    \qquad\text{and}\qquad
    D_K = \restr{D_G}{V_K \setminus \{0,1\}}
    \text{.}
  \end{equation*}
  The edges of $K$ fall into four different categories, according to 
  whether their head and tail in $L$ would be vertices of $K$ or 
  join-vertices; an index $0$ will denote ``even'' (in $K$) and an 
  index $1$ will denote ``odd'' (joining to $H$). Edges with an odd 
  end will have their labels derived from $H$-segment labels, whereas 
  edges with two even ends will be labelled as in $G$. Let
  \begin{align*}
    S ={}& \setOf[\big]{ \psi(e) }{ 
      \text{\(e \in E_H\), with \(h_H(e)=0\) or \(t_H(e)=1\)} 
    } \text{,}\\
    S_e ={}& \setOf[\big]{ f \in S }{ f \equiv e \pmod{m} }
      \qquad \text{for all \(e \in E_G\)}
  \intertext{\Dash then the set of even--even edges can be set up as}
    B_{00} ={}& \setOf[\big]{ 
      e \in \setinv{h_G}(V_K) \cap \setinv{t_G}(V_K)
    }{ S_e = \varnothing }
    \text{.}
  \end{align*}
  The rule for other edges is that an $H$-segment donates its label 
  to the $K$-segment preceding it (if any), whereas a $K$-segment not 
  followed by another $H$-segment gets as label $m$ more than the 
  label of the last $H$-segment on that $G$-edge. Hence
  \begin{align*}
    B_{01} ={}& \setOf{ m + \max S_e }{ 
      \text{\(e \in \setinv{h_G}(V_K)\) and \(S_e \neq \varnothing\)}
    } \text{,}\\
    B_{10} ={}& \setOf[\big]{ \min S_e }{ 
      \text{\(e \in \setinv{t_G}(V_K)\) and \(S_e \neq \varnothing\)}
    } \text{,}\\
    B_{11} ={}& \bigcup_{\substack{e \in E_G \\ S_e \neq \varnothing}}
      \bigl( S_e \setminus \{\min S_e\} \bigr)
      \text{,} \qquad\text{and}\\
    E_K ={}& B_{00} \cup B_{01} \cup B_{10} \cup B_{11} \text{.}
  \end{align*}
  The label of an $H$-segment whose tail is attached to a vertex in $G$ 
  will not become the label of an edge in $K$, so $B_{10} \cup B_{11}$ 
  can be a proper subset of $S$.
  
  When defining head and head index, there are two different cases 
  simply corresponding to whether the head endpoint is even or odd:
  \begin{equation*}
    (h_K,g_K)(e) = \begin{cases}
      (h_G,g_G)(e \bmod m) & \text{if \(e \in B_{00} \cup B_{01}\),}\\
      \bigl( 0, \omega(G) + (s_H \circ \psi^{-1})(e) \bigr)&
        \text{if \(e \in B_{10} \cup B_{11}\);}
    \end{cases}
  \end{equation*}
  the head index in the latter case is to make it join up with the 
  appropriate input leg in $H$. 
  The tail and tail index are completely analogous, but there it is 
  convenient to first define a helper function for determining the 
  $H$-label of the $H$-segment preceding a $K$-segment:
  \begin{gather*}
    \theta(e) = \psi^{-1}\Bigl( \max \setOf[\big]{ f \in S }{ 
      \text{\(f < e\) and \(f \equiv e \pmod{m}\)}
    } \Bigr) \quad\text{for \(e \in B_{01} \cup B_{11}\),}\\
    (t_K,s_K)(e) = \begin{cases}
      (t_G,s_G)(e \bmod m) & \text{if \(e \in B_{00} \cup B_{10}\),}\\
      \bigl( 1, \alpha(G) + (g_H \circ \theta)(e) \bigr) &
        \text{if \(e \in B_{01} \cup B_{11}\).}
    \end{cases}
  \end{gather*}
  When showing this $K$ is a network, the acyclicity condition is for 
  once trivial: any edge endpoint that is not the same as in $G$ is 
  taken to be $0$ or $1$ as appropriate, and from there a walk cannot 
  be extended any further. It should on the other hand be observed 
  that it is because of the allowed ranges for $e$ in the definitions 
  of $B_{00}$, $B_{01}$, and $B_{10}$ that one knows $h_K$ and $t_K$ 
  only assume values in $V_K$. That a combination of head 
  and head index uniquely identifies an edge requires a slight case 
  analysis. First, the combination is either the same as for some 
  edge of $G$, or it is $(0,i)$ for some \(i > \omega(G)\); these two 
  cannot overlap. In the former case it follows from unique 
  identification of edges in $G$ that a head and head index 
  combination determines $e \bmod m$ for \(e \in B_{00} \cup 
  B_{01}\), and that is unique by condition on $m$ and construction of 
  $B_{01}$. In the latter case of \(e \in B_{10} \cup B_{11}\), it holds 
  that \(t_H\bigl( \psi^{-1}(e) \bigr) = 1\) since if \(t_H(f) \neq 1\) 
  for some \(f \in E_H\) with \(\psi(f) \equiv e \pmod{m}\) then the 
  definition of embedding requires $\psi(f)$ to be minimal within 
  $S_{e \bmod m}$ and thus not an element of $B_{11}$, and it also 
  places $h_G(e \bmod m)$ in the image of $\chi$ and thus outside 
  $V_K$, meaning \(e \notin B_{10}\). Then the value of $s_H(f)$ 
  uniquely identifies an \(f \in E_H\) with \(t_H(f) = 1\). The 
  argument for the combination of tail and tail index is similar. 
  That for each \(v \in V_K \setminus\{0,1\}\) conversely any head 
  index up to \(\alpha\bigl( D_K(v) \bigr) = \alpha\bigl( D_G(v) 
  \bigr) = d_G^-(v)\) is assumed by some \(e \in B_{00} \cup B_{01}\) 
  with \(h_K(e)=v\) follows from considering the corresponding edge 
  in $G$; it has a counterpart in $B_{00} \cup B_{01}$. For the 
  output vertex of $K$ it follows by combining considerations for the 
  output vertex of $G$ and the input vertex of $H$, and again the 
  corresponding argument for tail indices is similar.
  
  Letting \(L = (V_L,E_L,h_L,g_L,t_L,s_L,D_L) := K 
  \overset{y}\rtimes H\), the claimed homeomorphism $(\beta,\gamma)$ 
  is easy to define:
  \begin{align*}
    \gamma(2e) ={}& 
       \rlap{$e \bmod m$} \hphantom{\psi(e) \bmod m}
       \qquad \text{for \(e \in E_K\),}\\
    \gamma(2e+1) ={}& 
       \psi(e) \bmod m \qquad \text{for \(e \in E_H\),}\\
    \beta(v) ={}& \begin{cases}
      \alpha(H) + \omega(H) + 2\chi^{-1}(v) + 1& 
        \text{if \(v \in \setim \chi\),}\\
      \alpha(H) + \omega(H) + 2v& 
        \text{if \(v \in V_K \setminus \{0,1\}\),}\\
      v& \text{if \(v \in \{0,1\}\).}
    \end{cases}
  \end{align*}
  Proving $L$ to be acyclic (and thus at all a network) is another 
  matter, but for this the above $\gamma$ provides a useful support. 
  Since $G$ is a network, there is a partial order $P_G$ on $E_G$ 
  satisfying \(e_1 > e_2 \pin{P_G}\) for all \(e_1,e_2 \in E_G\) such 
  that \(t_G(e_1) = h_G(e_2)\). Define \(\varphi\colon E_L \Fpil \N\) 
  by \(\varphi(e)=e\) for even $e$ and \(\varphi(e) = 
  2\psi(\frac{e-1}{2}) + 1\) for odd $e$; this so to say gives the 
  label an edge would have been given in $L$ had each edge of $H$ 
  been labelled $\psi(e)$ rather than $e$. Denote the standard order 
  on $\N$ by $T$ and define \(P_L = \gamma^*P_G \diamond 
  \varphi^*T\) (a lexicographic composition of two pullbacks of 
  quasi-orders). From the claim that this $P_L$ is a partial order on 
  $E_L$ satisfying \(e_1 > e_2 \pin{P_L}\) for all \(e_1,e_2 \in E_L\) 
  such that \(t_L(e_1) = h_L(e_2)\), it now follows that $L$ is 
  acyclic and thus a network.
  
  The verification of this claim splits into a number of cases 
  corresponding to the different kinds of vertices there are in $L$. 
  At a vertex \(v = \alpha(H) + \omega(H) + 2u + 1\), any \(e_1 \in 
  \setinv{t_L}\bigl(\{v\}\bigr)\) and \(e_2 \in 
  \setinv{h_L}\bigl(\{v\}\bigr)\) are of the form \(e_i = 2f_i+1\) 
  for some \(f_i \in E_H\), where \(t_H(f_1) = u = h_H(f_2)\). Then 
  \(t_G\bigl( \gamma(e_1) \bigr) = t_G\bigl( \psi(f_1) \bmod m \bigr) 
  = \chi\bigl( t_H(f_1) \bigr) = \chi(u) = \chi\bigl( h_H(f_2) \bigr) = 
  h_G\bigl( \psi(f_2) \bmod m \bigr) = h_G\bigl( \gamma(e_2) \bigr)\) 
  and hence \(e_1 > e_2 \pin{\gamma^*P_G}\). Moreover \(v = 
  \beta\bigl( \chi(u) \bigr)\), so the conditions for $(\beta,\gamma)$ 
  to be a homeomorphism are fulfilled at this kind of vertex.
  
  At a vertex \(v = \alpha(H) + \omega(H) + 2u\), any \(e_1 \in 
  \setinv{t_L}\bigl(\{v\}\bigr)\) and \(e_2 \in 
  \setinv{h_L}\bigl(\{v\}\bigr)\) satisfy \(e_1 = 2f_1\) for some 
  \(f_1 \in B_{00} \cup B_{10}\) and \(e_2 = 2f_2\) for some \(f_2 
  \in B_{00} \cup B_{01}\) respectively. Here, \(t_G\bigl( 
  \gamma(e_1) \bigr) = t_G(f_1 \bmod m) = t_K(f_1) = u = h_K(f_2) = 
  h_G(f_2 \bmod m) = h_G\bigl( \gamma(e_2) \bigr)\) and thus \(e_1 > 
  e_2 \pin{\gamma^*P_G}\). Moreover \(\beta(u) = v\), so the 
  homeomorphism conditions are fulfilled at this kind of vertex too. 
  This latter argument also applies for \(v \in \{0,1\}\), where the 
  $P_L$ claim is void as $0$ is never the tail of an edge and $1$ is 
  never the head of an edge.
  
  At a vertex \(v = 2 + i\) for \(1 \leqslant i \leqslant 
  \alpha(H)\), any \(e_1 \in \setinv{t_L}\bigl(\{v\}\bigr)\) must 
  have the form \(e_1 = 2f_1 + 1\) for some \(f_1 \in E_H\) with 
  \(t_H(f_1)=1\) and any \(e_2 \in \setinv{h_L}\bigl(\{v\}\bigr)\) 
  must have the form \(e_2 = 2f_2\) for some \(f_2 \in B_{10} \cup 
  B_{11}\). For these two have been joined, it must be the case that 
  \(g_K(f_2) - \omega(G) = i = s_H(f_1)\), i.e., \(f_1 = 
  \psi^{-1}(f_2)\). Hence \(\gamma(e_1) = \gamma(e_2)\), which means 
  the \(v \notin \setim \beta\) homeomorphism condition is fulfilled 
  here, and \(\varphi(e_1) = 2\psi(f_1) + 1 > 2f_2 = e_2 = 
  \varphi(e_2)\). Thus \(e_1 > e_2 \pin{P_L}\). At a vertex \(v = 2 + 
  \alpha(H) + i\) for \(1 \leqslant i \leqslant \omega(H)\), any 
  \(e_1 \in \setinv{t_L}\bigl(\{v\}\bigr)\) must have the form \(e_1 = 
  2f_1\) for some \(f_1 \in B_{01} \cup B_{11}\) and any \(e_2 \in 
  \setinv{h_L}\bigl(\{v\}\bigr)\) must have the form \(e_2 = 2f_2+1\) 
  for some \(f_2 \in E_H\) with \(h_H(f_1)=0\). For these two have 
  been joined, it must be the case that \(s_K(f_1) - \alpha(G) = i = 
  g_H(f_2)\), i.e., \(\theta(f_1) = f_2\). From \(f_1 \equiv 
  \psi(f_2) \pmod{m}\) then follows \(\gamma(e_1) = f_1 \bmod m = 
  \psi(f_2) \bmod m = \gamma(e_2)\), fulfilling the homeomorphism 
  condition at \(v \notin \setim \beta\), and from \(\psi(f_2) < 
  f_1\) follows \(\varphi(e_1) = 2f_1 > 2\psi(f_2)+1 = 
  \varphi(e_2)\). Hence \(e_1 > e_2 \pin{P_L}\).
  
  Finally, the \(v \in V_L \setminus \setim\beta\) are those with \(2 
  < v \leqslant 2 + \alpha(H) + \omega(H)\), and for these 
  \(D_L(v))=y\). Hence $L$ is a $y$-subdivision of $G$.
\end{proof}

\section{The free \PROP}
\label{Sec:FriPROP}

For many readers, it was no doubt becoming obvious already in 
Section~\ref{Sec:Natverk} that $\Nwt(\Omega)$ is the (set underlying 
the) free \PROP\ generated by $\Omega$; sets of expressions are 
generally good candidates for the construction of free objects. Is it 
necessary to go through the trouble of making an explicit construction, 
though; isn't the existence and properties of this free \PROP\ 
guaranteed by general principles of universal algebra? To some extent, 
they certainly are, but the catch here is that $\Nwt(\Omega)$ has more 
structure than what follows immediately from the universal property, 
and the rewriting mechanisms will make explicit use of this. 
Theorem~\ref{S:InbaddningHomeomorfi} 
links the subexpression concept to the symmetric join of networks, and 
since the elements of the free \PROP\ are (isomorphism classes of) 
networks, there is a way to define symmetric joins also of elements 
of $\Nwt(\Omega)$, even though this operation will have to be subject 
to some restrictions of a mostly syntactical nature. Therefore,
it turns out that the family of modules on which one wants to apply 
the diamond lemma is not that of the components of the free \PROP, but 
a $\B^{\bullet\times\bullet}$-filtration of the free \PROP!


A great help when putting forth $\Nwt(\Omega)$ as the free \PROP\ 
generated by $\Omega$ is that the extension of an arbitrary 
$\N^2$-graded set morphism \(f\colon \Omega \Fpil \mc{P}\) to a 
homomorphism \(\Nwt(\Omega) \Fpil \mc{P}\) required by the universal 
property is already known: it is the map $\eval_f$. It has however 
not been shown yet that this map is a \PROP\ homomorphism, nor for 
that matter that $\Nwt(\Omega)$ is a \PROP\ or even what its 
operations are (even though Theorems~\ref{S:AIN,snitt},  
\ref{S:AIN,klyva}, and~\ref{S:PROP3-definition} provide strong hints)! 
For the latter points, 
one may preferably turn to Theorem~\ref{S:PROP2-definition}, as there 
is a natural choice of evaluation map \(L\colon 
\Nw\bigl( \Nwt(\Omega) \bigr) \Fpil \Nwt(\Omega)\). A practical 
construction of it is the bottom row in the commutative diagram
\begin{equation} \label{Eq:FlattenSmoothen}
  \begin{CD}
    \Nw\bigl( \Nw(\Omega) \bigr) 
    @>{\text{flatten}}>>
    \Nw(\Omega')
    @>\text{smoothen}>>
    \Nw(\Omega) 
    \\
    @V{\Nw(c)}VV  @VV{c'}V  @VV{c}V
    \\
    \Nw\left( \Nwt(\Omega) \right) 
    @>>F>
    \Nwt(\Omega')
    @>>M>
    \Nwt(\Omega)
  \end{CD}
\end{equation}
where \(\Omega' = \Omega \cup \{\natural\}\), the vertical $c$ and $c'$ 
arrows map an element of $\Nw(\Omega)$ and $\Nw(\Omega')$ respectively 
to its equivalence class in the corresponding $\Nwt$, and $\Nw(c)$ 
applies $c$ to each annotation of a network. The flattening map takes the 
inner networks found in the outer network annotations, splits the $0$ 
and $1$ vertices of each annotation into separate vertices for each 
input and output (these new vertices get the distinct annotation 
\(\natural \in \Omega' \setminus \Omega\), which has arity and 
coarity~$1$), and joins it all up with the edges from the outer 
network. Then the smoothening map removes these $\natural$ vertices, 
joining the incident edges with each other. Thus one goes
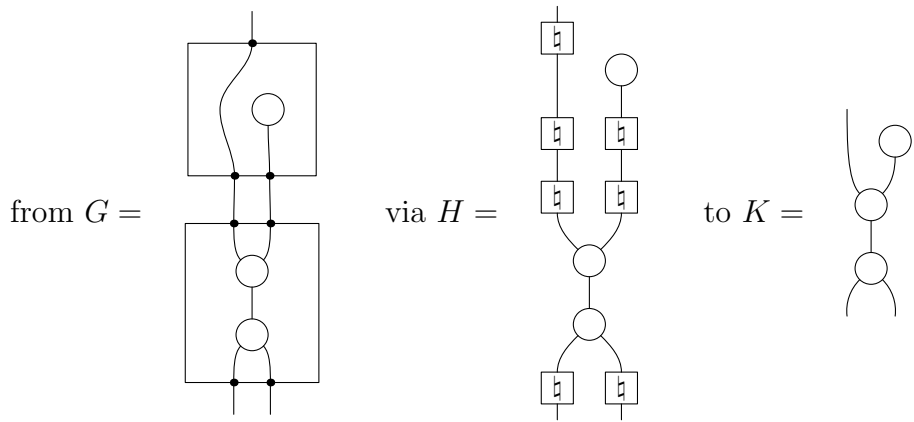
\[
  \text{from $G={}$}\quad
  \begin{mpgraphics*}{10}
    beginfig(10);
      u:=6pt;
      IC_equations3(2,1)(
         interim IC_width:=8u;  interim IC_height:=50pt;
         interim IC_legs_bb:=0.8;
      );
      y1 = y2 = 0.5[y3t1,y3b1];
      x1 = x3ll+2u; x2 = x3lr-3u;
      IC_equations6(2,2)(
         interim IC_width:=50pt;  interim IC_height:=10u;
         interim IC_legs_bb:=0.8;
      );
      x4 = x5 = 0.5[x6t1,x6t2];
      y4 = y6t1 - 3u;  y5 = y6b1 + 3u;
      x3 = x6;
      y3b1 - y6t1 = 3u;
      z6ll = (in,in);
      
      forsuffixes $=2,4,5:
        draw fullcircle scaled 2u shifted z$;
      endfor
      draw_IC3("", "", "", "", "");
      draw_IC6("", "", "", "", "");
      draw (z3t1 + (0,2u)) --- z3t1 ... z1{down} ... z3b1{down} ...
         z6t1{down} ... {dir-45}(z4 + u*dir135);
      draw (z2 + (0,-u)){down} ... z3b2{down} ... z6t2{down} ...
        {dir-135}(z4 + u*dir45);
      draw (z4 + (0,-u)) -- (z5 + (0,u));
      draw (z5 + u*dir-135){dir-135} ... z6b1 --- (z6b1 + (0,-2u));
      draw (z5 + u*dir-45){dir-45} ... z6b2 --- (z6b2 + (0,-2u));
      for $ := z3t1,z3b1,z3b2,z6t1,z6t2,z6b1,z6b2:
         draw $ withpen pencircle scaled dotlabeldiam;
      endfor
    endfig;
  \end{mpgraphics*}
  \qquad\text{via $H={}$}\quad
  \begin{mpgraphics*}{11}
    beginfig(11);
      u:=6pt;
      interim IC_width:=IC_height:=2u;
      forsuffixes $=1,3,4,5,6,9,10: IC_equations$(1,1)(); endfor
      x1 = x3 = x5 = x9;  x2 = x4 = x6 = x10;
      x7 = x8 = 0.5[x5,x6];  x5lr + 2u = x6ll;
      y3=y4; y5=y6; y9=y10;
      y2-y3 = y3-y5 = y5-y7 = y7-y8 = y8-y9;
      y3ll = y5ul + 2u;  y1lr = y2 + u; 
      z9ll = (in,in);
      
      picture natural;
      natural := btex $\natural$ etex;
      forsuffixes $=1,3,4,5,6,9,10: 
         draw_IC$("", "", natural, "", "");
      endfor
      forsuffixes $=2,7,8:
        draw fullcircle scaled 2u shifted z$;
      endfor
      draw (z1t1 + (0,u)) -- z1t1;
      draw z1b1{down} ... {down}z3t1;
      draw (z2 + (0,-u)){down} ... {down}z4t1;
      draw z3b1{down} ... {down}z5t1;
      draw z4b1{down} ... {down}z6t1;
      draw z5b1{down} ... {dir-45}(z7 + u*dir135);
      draw z6b1{down} ... {dir-135}(z7 + u*dir45);
      draw (z7 + (0,-u)){down} ... {down}(z8 + (0,u));
      draw (z8 + u*dir-135){dir-135} ... {down}z9t1;
      draw (z8 + u*dir-45){dir-45} ... {down}z10t1;
      draw z9b1 -- (z9b1 + (0,-u));
      draw z10b1 -- (z10b1 + (0,-u));
    endfig;
  \end{mpgraphics*}
  \qquad\text{to $K={}$}\quad
  \begin{mpgraphics*}{12}
    beginfig(12);
      u:=6pt;
      x4 = x5 = 0.5[x1,x2]; x1=x6; x2=x7;
      x2-x1 = 3u;
      y2-y4 = y4-y5 = 4u; y5-y6 = 3u; y7=y6; y1 = y2 + 2u;
      z6 = (in,in);
      
      forsuffixes $=2,4,5:
        draw fullcircle scaled 2u shifted z$;
      endfor
      draw z1{down} ... {dir-45}(z4 + u*dir135);
      draw (z2 + (0,-u)){down} ... {dir-135}(z4 + u*dir45);
      draw (z4 + (0,-u)) -- (z5 + (0,u));
      draw (z5 + u*dir-135){dir-135} ... z6;
      draw (z5 + u*dir-45){dir-45} ... z7;
    endfig;
  \end{mpgraphics*}
\]
when evaluating a network of networks. Smoothening can also be handy 
in defining the symmetric join, but more on that after the \PROP\ 
structure of $\Nwt(\Omega)$ has been established.

\subsection{Formal construction of \PROP\ structure}

From now on, since many networks will be considered 
simultaneously, the various data defining a network will be denoted 
by their symbols in the definition ($V$, $E$, $h$, $g$, $t$, $s$, 
and~$D$) subscripted with the symbol of the network (typically $G$, 
$H$, $K$, or $D_G(u)$ for some vertex $u$ of $G$).
Beginning with the smoothening map, one has to consider how some 
\(H \in \Nw(\Omega')\) is mapped to the corresponding \(K = 
\mathrm{smoothen}(H) \in \Nw(\Omega)\). Clearly,
\begin{align*}
  V_K ={}& V_H \setminus \setOf[\big]{ v \in V_H }{ D(v) = \natural }
    \text{,}\\
  D_K ={}& \restr{D_H}{V_K \setminus \{0,1\}}
\end{align*}
and one may choose
\begin{equation*}
  E_K = \setOf[\big]{ e \in E_H }{ t(e) \in V_K }
\end{equation*}
after which it is natural to take
\begin{equation*}
  t_K = \restr{t_H}{E_K}
  \qquad\text{and}\qquad
  s_K = \restr{s_H}{E_K}
  \text{.}
\end{equation*}
For the head ends, it is convenient to first define
\begin{equation*}
  S(e) = \begin{cases}
    e& \text{if \(h_H(e) = 0\),}\\
    e& \text{if \(h_H(e) > 1\) and \(D_H\bigl( h_H(e) \bigr) \neq \natural\),}\\
    S(f)\text{, }\parbox[t]{0.3\linewidth}{\raggedright
       where \(f \in E_H\) is such that \(t_H(f) = h_H(e)\)
    }& \parbox[t]{0.3\linewidth}{\strut\\otherwise,}
  \end{cases}
\end{equation*}
for all \(e \in E_H\), since then $S(e)$ is the headmost edge on the 
path in $H$ beginning with $e$ and only having $\natural$ inner 
vertices; this is well-defined since $H$ is acyclic. Given $S$, the 
remaining maps are trivially defined as
\begin{equation*}
  h_K(e) = h_H\bigl( S(e) \bigr)
  \quad\text{and}\quad
  g_K(e) = g_H\bigl( S(e) \bigr)
  \qquad\text{for all \(e \in E_K\).}
\end{equation*}
The completes the definition of the smoothening map for labelled 
networks. In order to see that the unlabelled counterpart $M$ exists, 
it suffices to observe that $S$ can be pushed forward across 
isomorphisms (the property of being the headmost edge does not depend 
on the labelling). It should also be observed that $(\mathrm{id},S)$ 
is a homeomorphism from $H$ to \(\mathrm{smoothen}(H) = K\). 

When concretely constructing the flattening map from 
\(\Nw\bigl( \Nw(\Omega) \bigr)\) to \(\Nw(\Omega')\), 
one should first observe that there is 
both an inner and an outer level of networks in the domain of this map; 
the inner networks (which are networks of $\Omega$) appear as 
annotations in the outer network (which is a network of $\Nw(\Omega)$) 
that the map is defined for. Since the idea is to combine all these 
networks into one, the following construction employs an injection 
\(p\colon \N^2 \Fpil \N\), namely \(p(i,j) = (i +\nobreak j)^2 + i\), 
to avoid collisions when assigning labels in the flat result. 
An edge $e$ in the outer network will become edge $p(1,e)$ in the 
result, whereas an edge $e$ in the inner network decorating vertex 
$u$ will become edge $p(u,e)$. 
An inner vertex $v$ in the network 
decorating vertex $u$ will become vertex $p(u,v)$, whereas the 
$\natural$ vertices will have labels on the forms 
$p\bigl( i, p(u,j) \bigr)$, where \(i \in \{0,1\}\) depends on 
whether they were made from an output or input, and $j$ is the 
corresponding head or tail index. Finally, the outer output and input 
vertices retain the labels $0$ and $1$; these do not collide with the 
aforementioned since \(p(i,j) \in \{0,1\}\) implies \(i=0\) and \(j 
\in \{0,1\}\), while \(u \geqslant 2\) in all cases where these 
formulae are applied.

Given \(G \in \Nw\bigl( \Nw(\Omega) \bigr)\), the flattened form of 
$G$ is thus the network $H$, where
\begin{align*}
  V_H ={}& \{0,1\} \cup \bigcup_{u \in V_G} W_u 
    \text{,}\\
  W_u ={}& \setOf[\big]{ p(u,v) }{ 
      v \in V_{D_G(u)} \setminus \{0,1\}
    } 
    \text{ for \(u \in V_G \setminus \{0,1\}\),}\\
  W_0 ={}& \setOf[\Big]{ p\bigl( 0, p(u,j) \bigr) }{ 
      u \in V_G \setminus \{0,1\}, 
      j \in \bigl[ d^+_G(u) \bigr]
    }
    \text{,}\\
  W_1 ={}& \setOf[\Big]{ p\bigl( 1, p(u,j) \bigr) }{ 
      u \in V_G \setminus \{0,1\}, 
      j \in \bigl[ d^-_G(u) \bigr]
    }\text{,}
    \displaybreak[0]\\
  E_H ={}& \bigcup_{u \in V_G \setminus \{0\}} F_u \text{,}\\
  F_1 ={}& \setOf[\big]{ p(1,e) }{ e \in E_G }\text{,}\\
  F_u ={}& \setOf[\big]{ p(u,e) }{ e \in E_{D_G(u)} }
    \quad \text{for \(u \in V_G \setminus \{0,1\}\).} 
\end{align*}
For any \(u \in V_G \setminus \{0,1\}\) and \(e \in E_{D_G(u)}\) 
(i.e., for any inner edge),
\begin{align*}
  (h_H,g_H)\bigl( p(u,e) \bigr) ={}& \begin{cases}
    \Bigl( p\bigl( u, h_{D_G(u)}(e) \bigr), g_{D_G(u)}(e) \Bigr)&
      \text{if \(h_{D_G(u)}(e) \neq 0\),}\\
    \Bigl( p\bigl( 0, p\bigl( u, g_{D_G(u)}(e) \bigr) \bigr), 1 \Bigr)& 
     \text{otherwise,}
  \end{cases}
  \displaybreak[0]\\
  (t_H,s_H)\bigl( p(u,e) \bigr) ={}& \begin{cases}
    \Bigl( p\bigl( u, t_{D_G(u)}(e) \bigr), s_{D_G(u)}(e) \Bigr)& 
      \text{if \(t_{D_G(u)}(e) \neq 1\),}\\
    \Bigl( p\bigl( 1, p\bigl( u, s_{D_G(u)}(e) \bigr) \bigr), 1 \Bigr)& 
      \text{otherwise.}
  \end{cases}
\end{align*}
For any \(e \in E_G\) (i.e., any outer edge),
\begin{align*}
  (h_H,g_H)\bigl( p(1,e) \bigr) ={}& \begin{cases}
    \bigl( 0, g_G(e) \bigr)& \text{if \(h_G(e)=0\),}\\
    \Bigl( p\bigl( 1, p\bigl( h_G(e), g_G(e) \bigr) \bigr),
      1 \Bigr)& \text{otherwise,}
  \end{cases}
  \displaybreak[0]\\
  (t_H,s_H)\bigl( p(1,e) \bigr) ={}& \begin{cases}
    \bigl( 1, s_G(e) \bigr)& \text{if \(t_G(e)=1\),}\\
    \Bigl( p\bigl( 0, p\bigl( t_G(e), s_G(e) \bigr) \bigr),
      1 \Bigr)& \text{otherwise.}
  \end{cases}
\end{align*}
Lastly
\begin{equation*}
  D_H\bigl( p(u,v) \bigr) = \begin{cases}
    D_{D_G(u)}(v)& \text{if \(u \notin \{0,1\}\),}\\
    \natural& \text{otherwise}
  \end{cases}
\end{equation*}
for all \(u \in V_G\) and \(v\in\N\) such that \(p(u,v) \in V_H\). 
That $H$ so defined fulfils axioms~\ref{A2:Network}, 
\ref{A3:Network}, and~\ref{A4:Network} of 
Definition~\ref{Def:Network} is a direct consequence of $G$ and the 
inner networks meeting these conditions, since they are all about 
local conditions which for non-$\natural$ vertices are preserved by 
the flattening map and for $\natural$ vertices are trivial. The 
acyclicity axiom~\ref{A1:Network} is not local, but still easy to 
verify; an oriented cycle in $H$ either contains some edge $p(1,e)$ 
for \(e \in E_G\), and than the set of such edges would constitute an 
oriented cycle in $G$ thus contradicting its acyclicity, or it only 
contains edges on the form $p(u,e)$ for some fixed $u$, and then it 
contradicts the acyclicity of $D_G(u)$.

Finally, there is the matter of whether there exists some map $F$ 
which makes the left square in the commutative diagram commute. Since 
the effect of the vertical maps is to replace networks with equivalence 
classes of same, what needs to be shown is that $\mathrm{flatten}$ 
maps isomorphic elements in the domain to isomorphic elements in the 
codomain, or more precisely that there are isomorphisms of the 
flattened network that correspond to isomorphisms of the inner 
networks; only edges and vertices corresponding to edges and vertices 
of some inner network may be affected by the class of isomorphism 
considered. By definition, 
an isomorphism is just a relabelling, and the deterministic 
construction of the flattening map is such that one can read off the 
pre-flattening labels from the flattened labels. Hence the 
construction of a corresponding isomorphism of the flattened target 
merely consists of: compute pre-flattening inner label of the edge or 
vertex, map it using the isomorphism corresponding to the outer 
vertex it is part of, and then compute the new post-flattening label. 
Thus the map $F$ is well-defined.


\begin{lemma} \label{L:NwL-uppdelning}
  Let \(L = F \circ M : \Nw\bigl( \Nwt(\Omega) \bigr) \Fpil 
  \Nwt(\Omega)\) as described by \eqref{Eq:FlattenSmoothen}.
  If \(G \in \Nw\bigl( \Nwt(\Omega) \bigr)\) has a split or cut 
  decomposition $(G_0,G_1)$, then there is in each \(K \in L(G)\) a 
  split or cut respective decomposition $(K_0,K_1)$ such that \(K_0 \in 
  L(G_0)\) and \(K_1 \in L(G_1)\).
\end{lemma}
\begin{proof}
  Let $H$ be a counterpart of $G$ in $\Nw\bigl( \Nw(\Omega) \bigr)$, 
  i.e., choose \(D_H(u) \in D_G(u)\) for every \(u \in V_G \setminus 
  \{0,1\}\) and define \(H = (V_G,E_G,h_G,g_G,t_G,s_G,D_H)\). 
  Let \(H' \in F(G)\) be the 
  flattening of $H$ and let \(H'' \in L(G)\) be the smoothening of 
  $H'$. Observe that split and cut decompositions are graphic 
  properties independent of the labelling, so it is sufficient to 
  exhibit a decomposition of one element of $L(G)$ (namely $H''$) 
  whose parts satisfy the wanted conditions.
  
  Suppose $(F_\mathrm{l},F_\mathrm{r},W_\mathrm{l},W_\mathrm{r})$ is 
  a split in $G$ inducing the decomposition 
  $(G_\mathrm{l},G_\mathrm{r})$. A split in $G$ is also a split in 
  $H$; let $(H_\mathrm{l},H_\mathrm{r})$ be the induced 
  decomposition of $H$. Splits survive flattening; there is a 
  corresponding split $(F_\mathrm{l}',F_\mathrm{r}',
  W_\mathrm{l}',W_\mathrm{r}')$ in $H'$, namely
  \begin{align*}
    F_\mathrm{l}' ={}& 
      \setOf[\big]{ p(1,e) \in E_{H'} }{ e \in F_\mathrm{l} } \cup
      \setOf[\Big]{ p(u,e) \in E_{H'} }{ u \in W_\mathrm{l} }
      \text{,}\\
    F_\mathrm{r}' ={}& 
      \setOf[\big]{ p(1,e) \in E_{H'} }{ e \in F_\mathrm{r} } \cup
      \setOf[\Big]{ p(u,e) \in E_{H'} }{ u \in W_\mathrm{r} }
      \text{,}\\
    W_\mathrm{l}' ={}&
      \setOf[\big]{ p(u,v) \in V_{H'} }{ u \in W_\mathrm{l} } \cup
      \setOf[\Big]{ 
        p\bigl( i, p(u,j) \bigr) \in V_{H'} 
      }{ u \in W_\mathrm{l} }
      \text{,}\\
    W_\mathrm{r}' ={}& 
      \setOf[\big]{ p(u,v) \in V_{H'} }{ u \in W_\mathrm{r} } \cup
      \setOf[\Big]{ 
        p\bigl( i, p(u,j) \bigr) \in V_{H'} 
      }{ u \in W_\mathrm{r} }
      \text{.}
  \end{align*}
  Let $(H'_\mathrm{l},H'_\mathrm{r})$ be the induced decomposition of 
  $H'$. Since flattening only considers local properties of networks, 
  it follows that $H_\mathrm{l}$ flattens to $H'_\mathrm{l}$ and 
  $H_\mathrm{r}$ flattens to $H'_\mathrm{r}$.
  
  Splits also survive smoothening, and the corresponding split in $H''$ 
  is $(F_\mathrm{l}'',F_\mathrm{r}'',W_\mathrm{l}'',W_\mathrm{r}'')$ 
  where \(F_\mathrm{l}'' = F_\mathrm{l}' \cap E_{H''}\), 
  \(F_\mathrm{r}'' = F_\mathrm{r}' \cap E_{H''}\), \(W_\mathrm{l}'' = 
  W_\mathrm{l}' \cap V_{H''}\), and \(W_\mathrm{r}'' = W_\mathrm{r}' 
  \cap V_{H''}\). The locality of the smoothening map definition 
  again implies that the smoothenings $H''_\mathrm{l}$ and 
  $H''_\mathrm{r}$ of $H'_\mathrm{l}$ and $H'_\mathrm{r}$ 
  respectively are precisely the parts of the decomposition 
  $(H''_\mathrm{l},H''_\mathrm{r})$ induced by 
  $(F_\mathrm{l}'',F_\mathrm{r}'',W_\mathrm{l}'',W_\mathrm{r}'')$. 
  Since \(\Nw(c)(H_\mathrm{l}) = G_\mathrm{l}\) and 
  \(\Nw(c)(H_\mathrm{r}) = G_\mathrm{r}\), it follows that 
  \(H''_\mathrm{l} \in L(G_\mathrm{l})\) and \(H''_\mathrm{r} \in 
  L(G_\mathrm{r})\).
  
  For a cut decomposition $(G_0,G_1)$ of $G$, one similarly 
  observes that the ordered cut $(W_0,W_1,q)$ inducing it is also an 
  ordered cut in $H$. The flattening $H'$ has the ordered cut 
  $(W_2,W_3,q_{23})$ where
  \begin{align*}
    W_2 ={}& \setOf[\big]{ p(u,v) \in V_{H'} }{ u \in W_0 } \cup
      \setOf[\Big]{ p\bigl( i, p(u,j) \bigr) \in V_{H'} }{ u \in W_0 }
      \text{,}\\
    W_3 ={}& \setOf[\big]{ p(u,v) \in V_{H'} }{ u \in W_1 } \cup
      \setOf[\Big]{ p\bigl( i, p(u,j) \bigr) \in V_{H'} }{ u \in W_1 }
      \text{,}\\
    q_{23}\bigl( p(1,e) \bigr) ={}& q(e)
      \quad\text{for cut edges \(e \in E_H\).}
  \end{align*}
  The corresponding ordered cut in $H''$ is $(W_4,W_5,q_{45})$, where 
  \(W_4 = W_2 \cap V_{H''}\), \(W_5 = W_3 \cap V_{H''}\), and 
  \(q_{45}^{-1} = T \circ q_{23}^{-1}\), where \(T\colon E_{H'} \Fpil 
  E_{H''}\) is the map which satisfies \(T(e) = e\) for \(e \in E_{H'}\) 
  such that \(t_{H'}(e) \in V_{H''}\) and \(T(e) = T\bigl( f(e) \bigr)\) 
  for other \(e \in E_{H'}\), where \(f(e) \in E_{H'}\) is the unique 
  edge such that \(h_{H'}\bigl( f(e) \bigr) = t_{H'}(e)\).
  
  Let $(H_0,H_1)$ be the cut decomposition of $H$ induced by 
  $(W_0,W_1,q)$, let $(H'_0,H'_1)$ be the cut decomposition of $H'$ 
  induced by $(W_2,W_3,q_{23})$, and let $(H''_0,H''_1)$ be the cut 
  decomposition of $H''$ induced by $(W_4,W_5,q_{45})$. As in the 
  split case, $H_0$ and $H_1$ flatten to $H'_0$ and $H'_1$ 
  respectively, and $H'_1$ smoothens to $H''_1$, but $H'_0$ smoothens 
  to some $H''_{00}$ which is usually not the same as $H''_0$; the cut 
  edges of $H'$ survive smoothening in $H'_0$ as they there have the 
  non-$\natural$ tail $1$, whereas a cut edge in $H'$ would only have 
  a non-$\natural$ tail if its counterpart in $G$ had $1$ as tail. 
  $H''_{00}$ and $H''_0$ are however isomorphic, with the isomorphism 
  \(H''_{00} \Fpil H''_0\) being the identity map on vertices and the 
  map $T$ on edges. Hence \(H''_0 \in L(G_0)\) and \(H''_1 \in L(G_1)\).
\end{proof}

\begin{theorem} \label{S:NwtPROP}
  For any $\N^2$-graded set $\Omega$, the set $\Nwt(\Omega)$ is a 
  \PROP\ with evaluation $L$ as described above, composition as 
  defined by \eqref{Eq:EvalDefCirc}, and tensor product as defined by 
  \eqref{Eq:EvalDefOtimes}.
\end{theorem}
\begin{proof}
  The idea of the proof is simply to verify the conditions of 
  Theorem~\ref{S:PROP2-definition}. The first condition is that $L$ 
  must be invariant under isomorphisms of the outer network 
  (isomorphisms of inner networks figured already in the definition 
  of $L$). 
  Let \(G_1,G_2 \in \Nw\left( \Nwt(\Omega) \right)\) such that \(G_1 
  \simeq G_2\) be arbitrary; let \((\chi,\psi)\colon G_1 \Fpil 
  G_2\) be the isomorphism. Define \(D'\colon V_{G_1} \setminus 
  \{0,1\} \Fpil \Nw(\Omega)\) by picking as $D'(v)$ some element of 
  $D_{G_1}(v)$. Define
  \begin{align*}
    G_1' ={}& (V_{G_1},E_{G_1},h_{G_1},g_{G_1},t_{G_1},s_{G_1},D')
      \in \Nw\bigl(\Nw(\Omega)\bigr)
      \text{,}\\
    G_2' ={}& (V_{G_2},E_{G_2},h_{G_2},g_{G_2},t_{G_2},s_{G_2},
      D' \circ \chi^{-1})
      \in \Nw\bigl(\Nw(\Omega)\bigr)
      \text{.}
  \end{align*}
  Since $G_1$ and $G_2$ are the images of $G_1'$ and $G_2'$ 
  respectively under the $\Nw(c)$ map of the commutative diagram 
  \eqref{Eq:FlattenSmoothen}, it will follow that \(F(G_1) = F(G_2)\) 
  (and hence \(L(G_1)=L(G_2)\)) once it has been shown that the 
  flattenings of $G_1'$ and $G_2'$ are isomorphic.
  
  Let $H_1'$ and $H_2'$ be the flattenings of $G_1'$ and $G_2'$ 
  respectively. Then the isomorphism \((\chi',\psi')\colon H_1' 
  \Fpil H_2'\) is given by
  \begin{align*}
    \chi'\bigl( p(u,v) \bigr) ={}&
      p\bigl( \chi(u), v \bigr) 
      \\
    \chi'\Bigl( p\bigl(0, p(u,j)\bigr) \Bigr) ={}&
      p\Bigl(0, p\bigl( \chi(u), j \bigr) \Bigr)
      \\
    \chi'\Bigl( p\bigl(1, p(u,k)\bigr) \Bigr) ={}&
      p\Bigl(1, p\bigl( \chi(u), k \bigr) \Bigr)
  \end{align*}
  for all \(v \in V_{D'(u)} \setminus \{0,1\}\), \(j \in 
  \bigl[ d_{G_1}^+(u) \bigr]\), \(k \in \bigl[ d_{G_1}^-(u) \bigr]\), 
  and \(u \in V_{G_1} \setminus \{0,1\}\) on one hand, and by
  \begin{align*}
    \psi'\bigl( p(1,e) \bigr) ={}& p\bigl( 1, \psi(e) \bigr) \text{,}\\
    \psi'\bigl( p(u,f) \bigr) ={}& p\bigl( \chi(u), f \bigr)
  \end{align*}
  for all \(e \in E_{G_1}\), \(f \in E_{D_{G_1}(u)}\), and \(u \in 
  V_{G_1} \setminus \{0,1\}\) on the other.
  
  Next consider the composition condition: that \(L(G) = L(G_0) \circ 
  L(G_1)\) for every cut decomposition $(G_0,G_1)$ of $G$. By 
  Lemma~\ref{L:NwL-uppdelning}, every \(K \in L(G)\) has a cut 
  decomposition $(K_0,K_1)$ such that \(K_0 \in L(G_0)\) and \(K_1 
  \in L(G_1)\). Hence to exhibit a representative of $L(G_0) \circ 
  L(G_1)$, one may first let 
  \begin{math}
    H = \Nwfuse[\big]{\vek{l}}{ (2\colon K_0)^\vek{l}_\vek{m} 
    (3\colon K_1)^\vek{m}_\vek{n} }{\vek{n}}
  \end{math}
  where \(\vek{l} = \vek{N}_{\omega(K)}(0,3)\), \(\vek{m} = 
  \vek{N}_{\omega(K_1)}(1,3)\), and \(\vek{n} = 
  \vek{N}_{\alpha(K)}(2,3)\). Taking $H'$ to be the flattening 
  of $H$ and $H''$ to be the smoothening of $H'$, the problem reduces 
  to that of showing \(K \simeq H''\). Let $(W_0,W_1)$ be the cut in 
  $K$ that induces $(K_0,K_1)$. Then the sought isomorphism 
  \((\chi,\psi)\colon K \Fpil H''\) is
  \begin{align*}
    \chi(v) ={}& \begin{cases}
      v& \text{if \(v \in \{0,1\}\),}\\
      p(2,v)& \text{if \(v \in W_0\),}\\
      p(3,v)& \text{if \(v \in W_1\)}
    \end{cases}
    &&\text{for \(v \in V_K\),}\\
    \psi(e) ={}& \begin{cases}
      p(2,e)& \text{if \(t_K(e) \in W_0\),}\\
      p(3,e)& \text{if \(t_K(e) \in W_1\),}\\
      p\bigl( 1, 3s_K(e)-1 \bigr) & \text{if \(t_K(e) = 1\)}
    \end{cases}
    &&\text{for \(e \in E_K\);}
  \end{align*}
  an edge of $K$ may have up to five counterparts in $H'$, but the 
  one that survives smoothening is that whose tail is at a 
  non-$\natural$ vertex.
  
  Third consider the tensor product condition: that \(L(G) = L(G_2) 
  \otimes  L(G_3)\) for every split decomposition $(G_2,G_3)$ of $G$. 
  By Lemma~\ref{L:NwL-uppdelning}, every \(K \in L(G)\) has a split 
  decomposition $(K_2,K_3)$ such that \(K_2 \in L(G_2)\) and \(K_3 
  \in L(G_3)\). Hence to exhibit a representative of $L(G_2) \otimes 
  L(G_3)$, one may first let 
  \begin{math}
    H = \Nwfuse[\big]{\vek{km}}{ 
      (2\colon K_2)^\vek{k}_\vek{l} (3\colon K_3)^\vek{m}_\vek{n} 
    }{\vek{ln}}
  \end{math}
  where \(\vek{k} = \vek{N}_{\omega(K_2)}(0,2)$, \(\vek{l} = 
  \vek{N}_{\alpha(K_2)}(1,2)\), \(\vek{m} = 
  \vek{N}_{\omega(K_3)}\bigl( 2\omega(K_2), 2 \bigr)\), and \(\vek{n} = 
  \vek{N}_{\alpha(K_3)}\bigl( 2\alpha(K_2) +\nobreak 1, 2 \bigr)\).
  Taking $H'$ to be the flattening of $H$ and $H''$ to be the 
  smoothening of $H'$, the problem reduces to that of showing 
  \(K \simeq H''\). Let $(F_2,F_3,W_2,W_3)$ be the split in $K$ that 
  induces $(K_2,K_3)$. Then the sought isomorphism \((\chi,\psi)\colon 
  K \Fpil H''\) is
  \begin{align*}
    \chi(v) ={}& \begin{cases}
      v& \text{if \(v \in \{0,1\}\),}\\
      p(2,v)& \text{if \(v \in W_2\),}\\
      p(3,v)& \text{if \(v \in W_3\)}
    \end{cases}
    &&\text{for \(v \in V_K\),}\\
    \psi(e) ={}& \begin{cases}
      p(2,e)& \text{if \(t_K(e) \in W_2\),}\\
      p(3,e)& \text{if \(t_K(e) \in W_3\),}\\
      p\bigl( 1, 2s_K(e) - 1 \bigr) & \text{if \(t_K(e) = 1\)}
    \end{cases}
    &&\text{for \(e \in E_K\).}
  \end{align*}
  
  Finally consider the elementary network condition: that given any 
  \(\gamma \in \Nwt(\Omega)\), it holds that 
  \(L\Bigl( \Nwfuse{ \vek{m} }{ \gamma^{\vek{m}}_{\vek{n}} }{ \vek{n} } 
  \Bigr) = \gamma\) for \(\vek{m} = 
  \vek{N}_{\omega(\gamma)}(0,2)\) and \(\vek{n} = 
  \vek{N}_{\alpha(\gamma)}(1,2)\). 
  Pick some \(G \in \gamma\) and let \(K \in \Nw(\Omega)\) be the result 
  of flattening and smoothening $\Nwfuse{ \vek{m} }{ 
  (2\colon G)^{\vek{m}}_{\vek{n}} }{ \vek{n} }$. 
  The claim is then that \(K \in \gamma\), or equivalently that \(G 
  \simeq K\). Since all inner vertices of $K$ are images under 
  flattening of inner vertices in $G$, and similarly all edges in $K$ 
  except those with tail $1$ are images under flattening of edges in 
  $G$, it turns out that this isomorphism \((\chi,\psi)\colon G 
  \Fpil K\) can be expressed as
  \begin{align*}
    \chi(v) ={}& \begin{cases}
      v& \text{if \(v \in \{0,1\}\),}\\
      p(2,v)& \text{otherwise}
    \end{cases}
    &&\text{for all \(v \in V_G\),}\\
    \psi(e) ={}& \begin{cases}
      p\bigl( 1, 2s_G(e)-1 \bigr)& \text{if \(t_G(e)=1\),}\\
      p(2,e)& \text{otherwise}
    \end{cases}
    &&\text{for all \(e \in E_G\).}
  \end{align*}
\end{proof}

The next order of business is the universal property.

\begin{lemma} \label{L:evalHomomorfi}
  Let an $\N^2$-graded set $\Omega$, a \PROP\ $\mc{P}$, and an 
  $\N^2$-graded set morphism \(f\colon \Omega \Fpil \mc{P}\) be 
  given. If \(i\colon \Omega \Fpil \Nwt(\Omega)\) is the $\N^2$-graded 
  set morphism satisfying
  \begin{equation} \label{Eq:Nwt-injektion}
    i(x) \owns \Nwfuse[\Big]{ \vek{N}_{\omega(x)}(0,2) }{ 
       x^{\vek{N}_{\omega(x)}(0,2)}_{\vek{N}_{\alpha(x)}(1,2)} 
    }{ \vek{N}_{\alpha(x)}(1,2) } 
    \qquad\text{for all \(x \in \Omega\)}
  \end{equation}
  then \(\eval_f\colon\Nwt(\Omega) \Fpil \mc{P}\) is a \PROP\ 
  homomorphism satisfying \(\eval_f \circ i = f\).
\end{lemma}
\begin{proof}
  To begin with, a \PROP\ homomorphism must preserve composition. 
  To that end, let \(l,m,n \in \N\), \(\beta \in \Nwt(\Omega)(l,m)\), 
  and \(\gamma \in \Nwt(\Omega)(m,n)\) be 
  arbitrary. Pick some \(G_0 \in \beta\) and \(G_1 \in \gamma\). Let 
  \[
    G = \Nwfuse[\Big]{ \vek{N}_l(0,3) }{ 
    (G_0)^{\vek{N}_l(0,3)}_{\vek{N}_m(1,3)} 
    (G_1)^{\vek{N}_m(1,3)}_{\vek{N}_n(2,3)} }{ \vek{N}_n(2,3) }
    \text{.} 
  \]
  Let $H$ be the flattening of $G$ and $K$ be the smoothening of $H$. 
  Then \(K \in \beta \circ \gamma\). Let \(f'(x)=f(x)\) for \(x \in 
  \Omega\) and \(f'(\natural) = \phi_{\mc{P}}(\same{1})\); then 
  \[
    \eval_f( \beta \circ \gamma ) =
    \eval_f(K) =
    \eval_{f'}(H)
  \]
  by Lemma~\ref{L:evalNatural}, since $H$ is a $\natural$-subdivision 
  of $K$. $H$ has an ordered cut $(W_0,W_1,q)$ where
  \begin{align*}
    W_0 ={}& \Bigl\{ p\bigl( 0, p(2,j) \bigr) \Bigr\}_{j \in [l]} \cup
      \Bigl\{ p\bigl( 1, p(2,j) \bigr) \Bigr\}_{j \in [m]} \cup
      \bigl\{ p(2,v) \bigr\}_{v \in V_{G_0}} 
      \text{,}\\
    W_1 ={}& \Bigl\{ p\bigl( 0, p(3,j) \bigr) \Bigr\}_{j \in [m]} \cup
      \Bigl\{ p\bigl( 1, p(3,j) \bigr) \Bigr\}_{j \in [n]} \cup
      \bigl\{ p(3,v) \bigr\}_{v \in V_{G_1}} 
      \text{,}\\
    q^{-1}(j) ={}& p( 1, 3j-2 ) \quad\text{for \(j \in [m]\).}
  \end{align*}
  The corresponding decomposition $(H_0,H_1)$ is such that $H_0$ is a 
  $\natural$-subdivision of $G_0$ and $H_1$ is a 
  $\natural$-subdivision of $G_1$. Hence
  \begin{multline*}
    \eval_{f'}(H) = 
    \eval_{f'}(H_0) \circ \eval_{f'}(H_1) = \\ =
    \eval_f(G_0) \circ \eval_f(G_1) =
    \eval_f(\beta) \circ \eval_f(\gamma)
  \end{multline*}
  and thus \(\eval_f(\beta \circ\nobreak \gamma) = \eval_f\beta \circ 
  \eval_f\gamma\).
  
  Next consider the tensor product: let \(k,l,m,n \in \N\), 
  \(\beta \in \Nwt(\Omega)(k,l)\), and \(\gamma \in \Nwt(\Omega)(m,n)\) 
  be arbitrary. Pick some \(G_\mathrm{l} \in \beta\) and \(G_\mathrm{r} 
  \in \gamma\). Let \(G = \Nwfuse[\Big]{ \vek{N}_k(0,4)\vek{N}_m(2,4) }{ 
  (G_\mathrm{l})^{\vek{N}_k(0,4)}_{\vek{N}_l(1,4)} 
  (G_\mathrm{r})^{\vek{N}_m(2,4)}_{\vek{N}_n(3,4)} }{ 
  \vek{N}_l(1,4)\vek{N}_n(3,4) }\). Let $H$ be the flattening of $G$ and 
  $K$ be the smoothening of $H$. Then \(K \in \beta \otimes \gamma\). 
  Let \(f'(x)=f(x)\) for \(x \in \Omega\) and \(f'(\natural) = 
  \phi_{\mc{P}}(\same{1})\); then by Lemma~\ref{L:evalNatural},
  \[
    \eval_f( \beta \otimes \gamma ) =
    \eval_f(K) =
    \eval_{f'}(H) \text{.}
  \]
  $H$ has a split $(F_\mathrm{l},F_\mathrm{r},W_\mathrm{l},W_\mathrm{r})$ 
  where
  \begin{align*}
    F_\mathrm{l} ={}& \bigl\{ p(1,4j-4) \bigr\}_{j \in [k]} \cup 
      \bigl\{ p(1,4j-3) \bigr\}_{j \in [l]} \cup 
      \bigl\{ p(2,e) \bigr\}_{e \in E_{G_\mathrm{l}}} \text{,}\\
    F_\mathrm{r} ={}& \bigl\{ p(1,4j-2) \bigr\}_{j \in [m]} \cup 
      \bigl\{ p(1,4j-1) \bigr\}_{j \in [n]} \cup 
      \bigl\{ p(3,e) \bigr\}_{e \in E_{G_\mathrm{r}}} \text{,}\\
    W_\mathrm{l} ={}& 
      \Bigl\{ p\bigl( 0, p(2,j) \bigr) \Bigr\}_{j \in [k]} \cup
      \Bigl\{ p\bigl( 1, p(2,j) \bigr) \Bigr\}_{j \in [l]} \cup
      \bigl\{ p(2,v) \bigr\}_{v \in V_{G_0}} 
      \text{,}\\
    W_\mathrm{r} ={}& 
      \Bigl\{ p\bigl( 0, p(3,j) \bigr) \Bigr\}_{j \in [m]} \cup
      \Bigl\{ p\bigl( 1, p(3,j) \bigr) \Bigr\}_{j \in [n]} \cup
      \bigl\{ p(3,v) \bigr\}_{v \in V_{G_1}} 
      \text{.}
  \end{align*}
  The corresponding decomposition $(H_\mathrm{l},H_\mathrm{r})$ is such 
  that $H_\mathrm{l}$ is a $\natural$-subdivision of $G_\mathrm{l}$ and 
  $H_\mathrm{r}$ is a $\natural$-subdivision of $G_\mathrm{r}$. Hence
  \begin{multline*}
    \eval_{f'}(H) = 
    \eval_{f'}(H_\mathrm{l}) \otimes \eval_{f'}(H_\mathrm{r}) = \\ =
    \eval_f(G_\mathrm{l}) \otimes \eval_f(G_\mathrm{r}) =
    \eval_f(\beta) \otimes \eval_f(\gamma)
  \end{multline*}
  and thus \(\eval_f(\beta \otimes\nobreak \gamma) = \eval_f\beta 
  \otimes \eval_f\gamma\).
  
  Third consider permutations: let \(n \in \N\) and \(\sigma \in 
  \Sigma_n\) be arbitrary. Then
  \begin{multline*}
    \eval_f\bigl( \phi_{\Nwt(\Omega)}(\sigma) \bigr) =
    \eval_f\Bigl( \bigl( \{0,1\}, [n], 0, \sigma, 1, \same{n}, 
      \varnothing \bigr) \Bigr) = \\ =
    \fuse[\Big]{ \sigma^{-1}(1)\dotsb\sigma^{-1}(n) }{ 1 }{ 
      \vek{N}_n(1,1) }_\mc{P} =
    \phi_\mc{P}(\sigma)
  \end{multline*}
  since \(\vek{N}_n(1,1) = a_{\sigma(1)}\dotsb a_{\sigma(n)}\) if 
  \(a_i = \sigma^{-1}(j)\) for all \(j \in [n]\).
  
  Finally consider the map $i$: let \(x \in \Omega\) be arbitrary. 
  Let \(\vek{m} = \vek{N}_{\omega(x)}(0,2)\) and \(\vek{n} = 
  \vek{N}_{\alpha(x)}(1,2)\). Then
  \begin{equation*}
    \eval_f\bigl( i(x) \bigr) =
    \eval_f\Bigl( \Nwfuse{\vek{m}}{ x^\vek{m}_\vek{n} }{\vek{n}} 
      \Bigr) =
    \fuse{\vek{m}}{ f(x)^\vek{m}_\vek{n} }{\vek{n}}_\mc{P} =
    f(x) \text{,}
  \end{equation*}
  as claimed.
\end{proof}

\begin{theorem} \label{S:NwtFriPROP}
  $\Nwt(\Omega)$ is the free \PROP\ on $\Omega$.
\end{theorem}
\begin{proof}
  Lemma~\ref{L:evalHomomorfi} already states that any $\N^2$-graded 
  set map \(f\colon \Omega \Fpil \mc{P}\) for a \PROP\ $\mc{P}$ 
  extends to a \PROP\ homomorphism \(\eval_f\colon \Nwt(\Omega) \Fpil 
  \mc{P}\), so what remains to show is that this homomorphism is 
  unique; with $i$ as in the lemma, it must be shown that \(\varphi_1 
  \circ i = f = \varphi_2 \circ i\) for \PROP\ homomorphisms 
  \(\varphi_1,\varphi_2\colon \Nwt(\Omega) \Fpil \mc{P}\) implies 
  \(\varphi_1 = \varphi_2\).
  
  Let \(\gamma \in \Nwt(\Omega)\) and \(G \in \gamma\) be arbitrary. 
  The proof is by induction on the number of vertices in $G$. If $G$ 
  has two vertices (i.e., no inner vertices) then \(\gamma = 
  \phi_{\Nwt(\Omega)}(\sigma)\) for some permutation $\sigma$ and 
  \(\varphi_1(\gamma) = \varphi_1\bigl( \phi_{\Nwt(\Omega)}(\sigma) 
  \bigr) = \phi_\mc{P}(\sigma) = \varphi_2\bigl( 
  \phi_{\Nwt(\Omega)}(\sigma) \bigr) = \varphi_2(\gamma)\). If $G$ 
  has three vertices then \(V_G = \{0,1,v\}\), meaning $\bigl( 
  \varnothing, \{v\} \bigr)$ and $\bigl( \{v\}, \varnothing \bigr)$ 
  are cuts in $G$. Defining \(\vek{a} = \vek{e}_G^-(0)\), \(\vek{b} 
  = \vek{e}_G^+(v)\), \(\vek{c} = \vek{e}_G^-(v)\), and \(\vek{d} = 
  \vek{e}_G^+(1)\), one may utilise these cuts to produce the 
  decomposition
  \begin{equation*}
    \gamma = 
    \fuse{ \vek{a} }{ 1 }{ \vek{b}(\vek{a}/\vek{b}) } \circ
    i\bigl( D_G(v) \bigr) \otimes 
      \fuse{ \vek{a}/\vek{b} }{ 1 }{ \vek{a}/\vek{b} } \circ
    \fuse{ \vek{c} (\vek{a}/\vek{b}) }{ 1 }{ \vek{d} }
    \text{.}
  \end{equation*}
  Since \(\varphi_1\bigl( i(x) \bigr) = \varphi_2\bigl( i(x) \bigr)\) 
  for all \(x \in \Omega\), it again follows that \(\varphi_1(\gamma) 
  = \varphi_2(\gamma)\).
  
  Finally consider the case that \(n := \card{V_G} > 3\), and assume 
  $\varphi_1$ and $\varphi_2$ coincide for all networks with less 
  than $n$ vertices. Let $v$ be an inner vertex of $G$ whose 
  out-neighbours are all in $\{0\}$. Then $\bigl( \{v\}, V_G 
  \setminus\nobreak \{0,1,v\} \bigr)$ is a cut in $G$; let 
  $(G_0,G_1)$ be a corresponding decomposition. Then \(\gamma = 
  \eval_i(G_0) \circ \eval_i(G_1)\) and hence
  \begin{multline*}
    \varphi_1(\gamma) = 
    \varphi_1\bigl( \eval_i(G_0) \circ \eval_i(G_1) \bigr) =
    \varphi_1\bigl( \eval_i(G_0) \bigr) \circ 
      \varphi_1\bigl( \eval_i(G_1) \bigr) = \\ =
    \varphi_2\bigl( \eval_i(G_0) \bigr) \circ 
      \varphi_2\bigl( \eval_i(G_1) \bigr) =
    \varphi_2\bigl( \eval_i(G_0) \circ \eval_i(G_1) \bigr) =
    \varphi_2(\gamma) \text{.}
  \end{multline*}
  It follows that \(\varphi_1(\gamma) = \varphi_2(\gamma)\) for all 
  \(\gamma \in \Nwt(\Omega)\).
\end{proof}

With the free \PROP\ properly established, it is next time to sort 
out the free $\mc{R}$-linear \PROP. It seems appropriate to follow 
the suite of $\mc{R}(X)$, $\mc{R}[X]$, and $\mc{R}\langle X\rangle$ 
in a basic notation for this with $\mc{R}\{X\}$, as braces are the 
simplest delimiters not yet appropriated, and furthermore the 
``spurs'' on the braces suggest the possibility (as in networks) of 
external connections. 

\begin{definition}
  Let $\mc{R}$ be an associative and commutative ring with unit. Let 
  $\Omega$ be an $\N^2$-graded set. Then the 
  \DefOrd[{free linear PROP}*]{free $\mc{R}$-linear \PROP\ on $\Omega$} 
  is denoted $\mc{R}\{\Omega\}$\index{R{\Omega}@$\mc{R}\{\Omega\}$}.
\end{definition}

\begin{corollary}
  $\mc{R}\{\Omega\}$ is the same as $\mc{R}\Nwt(\Omega)$. In 
  particular, $\Nwt(\Omega)$ may be regarded as a subset of 
  $\mc{R}\{\Omega\}$ and $\Nwt(\Omega)(m,n)$ is a basis of the 
  $\mc{R}$-module $\mc{R}\{\Omega\}(m,n)$, for all \(m,n \in \N\). 
  Similarly, any \(x \in \Omega\) may be identified with the 
  isomorphism class of $\Nwfuse{ \vek{N}_{\omega(x)}(0,2) }{ 
  x^{\vek{N}_{\omega(x)}(0,2)}_{\vek{N}_{\alpha(x)}(1,2)} }{ 
  \vek{N}_{\alpha(x)}(1,2) }$ in $\Nwt(\Omega)$.
\end{corollary}
\begin{proof}
  Merely combine Theorem~\ref{S:NwtFriPROP} with 
  Construction~\ref{Kons:R-linjar}.
\end{proof}

\begin{remark}
  When using network notation\index{network notation} for elements in 
  $\mc{R}\{\Omega\}$, it can be a good idea to surround each network 
  with a bracket, as in
  \[
    \left[ \begin{mpgraphics*}{-105}
      PROPdiagramnoarrow(0,0)
        same(2) circ
        box(2,1)(btex $\Delta$\strut etex) circ
        box(1,2)(btex $\mOp$\strut etex)
        circ same(2)
      ;
    \end{mpgraphics*} \right]
    -
    \left[ \begin{mpgraphics*}{-106}
      PROPdiagramnoarrow(0,0)
        same(2) circ
        box(1,2)(btex $\mOp$\strut etex) otimes same(1) circ
        same(1) otimes box(2,1)(btex $\Delta$\strut etex) 
        circ same(2)
      ;
    \end{mpgraphics*} \right]
    -
    \left[ \begin{mpgraphics*}{-107}
      PROPdiagramnoarrow(0,0)
        same(2) circ
        same(1) otimes box(1,2)(btex $\mOp$\strut etex) circ
        box(2,1)(btex $\Delta$\strut etex) otimes same(1) 
        circ same(2)
      ;
    \end{mpgraphics*} \right]
    +
    \left[ \begin{mpgraphics*}{-108}
      draw (default_sep,0) -- (default_sep,38);
      draw (3default_sep,0) -- (3default_sep,38);
    \end{mpgraphics*} \right]
    \text{.}
  \]
  There are two reasons for this. One, being more theoretical, is 
  that the elements of the \PROP\ $\Nwt(\Omega)$ are isomorphism 
  classes of networks whereas the diagrams themselves are concrete 
  networks, so the brackets are ``needed'' to make it a notation for 
  an element of the \PROP. The other reason, of a more practical 
  nature, is that they improve legibility; when networks become part 
  of a greater formula in more traditional notation, it can be hard 
  to tell exactly where the boundaries of one network are. Delimiters 
  that are placed around a network solve this problem, by visually 
  becoming a frame.
\end{remark}

It may also be observed that $\mc{R}[X]$ and $\mc{R}\langle X\rangle$ 
both arise as special cases of $\mc{R}\{X\}$, for suitable choices of 
arity and coarity for the generators: if \(\alpha(x)=\omega(x)=0\) 
for all \(x \in X\) then \(\mc{R}\{X\}(0,0) \cong \mc{R}[X]\), and 
if \(\alpha(x)=\omega(x)=1\) for all \(x \in X\) then 
\(\mc{R}\{X\}(1,1) \cong \mc{R}\langle X\rangle\).

\subsection{The symmetric join}

So far, this section has delt with establishing that this concrete 
construction of the free \PROP\ from networks satisfy the basic 
category-theoretical properties required of a free object. That focus 
will now shift towards properties that are not so basic from a 
categorical perspective\Ldash to categorically define for example a 
\PROP\ filtration seems to require climbing a step or two on the 
abstraction ladder\Rdash but remain perfectly basic from the more naive 
perspective of classical abstract algebra. The first order of 
business is to define the symmetric join in the free \PROP, so that 
one can get a free \PROP\ characterisation of the subexpression 
concept of Section~\ref{Sec:Deluttryck}.

\begin{lemma} \label{L:Join-homeomorfi}
  Let $\Omega$ be an $\N^2$-graded set and \(z \in \Omega(1,1)\) be 
  arbitrary. Let \(H,H',K,K' \in \Nw(\Omega')\) and \(r,q \in \N\) be 
  such that $K \join[z]^r_q H$ and $K' \join[z]^r_q H'$ are networks. 
  If $(\beta_0,\gamma_0)$ is a homeomorphism from $K'$ to $K$ and 
  $(\beta_1,\gamma_1)$ is a homeomorphism from $H'$ to $H$, then a 
  homeomorphism from $K' \join[z]^r_q H'$ to $K \join[z]^r_q H$ is
  $(\beta,\gamma)$, where
  \begin{align*}
    \beta(v) ={}& v && \text{if \(v \leqslant 2+r+q\),}\\
    \beta(r+q+2v) ={}& r+q+2\beta_0(v) 
      && \text{if \(v \in V_K \setminus \{0,1\}\),}\\
    \beta(r+q+2v+1) ={}& r+q+2\beta_1(v) +1 
      && \text{if \(v \in V_H \setminus \{0,1\}\),}
      \displaybreak[0]\\
    \gamma(2e) ={}& 2\gamma_0(e) && \text{if \(e \in E_{K'}\),}\\
    \gamma(2e+1) ={}& 2\gamma_1(e)+1 && \text{if \(e \in E_{H'}\).}
  \end{align*}
  Conversely, if $(\beta,\gamma)$ is a homeomorphism from 
  $K' \join[z]^r_q H'$ to $K \join[z]^r_q H$ such that \(\beta(v) = v\) for 
  \(v-2 \in [r +\nobreak q]\) and \(\beta(v) \equiv v \pmod{2}\) for 
  all \(v \in V_{K \join[z]^r_q H}\), then there is a homeomorphism 
  $(\beta_0,\gamma_0)$ from $K'$ to $K$ and a homeomorphism 
  $(\beta_1,\gamma_1)$ from $H'$ to $H$ satisfying
  \begin{align*}
    \gamma_0(e) ={}& \tfrac{1}{2}\gamma(2e)
      && \text{for \(e \in E_{K'}\),}\\
    \gamma_1(e) ={}& \tfrac{1}{2}\bigl( \gamma(2e+1) - 1 \bigr)
      && \text{for \(e \in E_{H'}\),}
      \displaybreak[0]\\
    \beta_0(v) ={}& \tfrac{1}{2}\bigl( \beta(r+q+2v) -r -q \bigr)
      && \text{for \(v \in V_K \setminus \{0,1\}\),}
      \\
    \beta_1(v) ={}& \tfrac{1}{2}\bigl( \beta(r+q+2v+1) -r -q -1 \bigr)
      && \text{for \(v \in V_H \setminus \{0,1\}\).}
  \end{align*}
\end{lemma}
\begin{proof}
  Merely verify that the constructed pairs of maps satisfy the 
  conditions for being homeomorphisms.
\end{proof}

A special case of the above lemma is that if \(K \simeq K'\) and \(H 
\simeq H'\) then \(K \join^r_q H \simeq K' \join^r_q H'\), because an 
isomorphism is an invertible homeomorphism. Hence the symmetric join 
can be defined as an operation on $\Nwt(\Omega)$ too.

\begin{definition}
  Let $\Omega$ be an $\N^2$-graded set with \(\natural \notin \Omega\) 
  and $\mc{R}$ be an associative and commutative ring with unit. 
  The \DefOrd{symmetric join} $\join^r_q$ on $\Nwt(\Omega)$ is 
  defined by
  \begin{equation}
    [K]_\simeq \join^r_q [H]_\simeq := 
    \bigl[ \mathrm{smoothen}(K \join^r_q H) \bigr]_\simeq
  \end{equation}
  for all \(K,H \in \Nw(\Omega)\) such that the right hand side 
  exists. The symmetric join operation is extended to 
  $\mc{R}\{\Omega\}$ by bilinearity. In both cases, the 
  \emDefOrd[*{annexation}]{right annexation} $a \rtimes b$ is defined 
  as $a \join^{\alpha(b)}_{\omega(b)} b$.
\end{definition}

The next lemma extends the subexpression characterisations of 
Theorem~\ref{S:InbaddningHomeomorfi} with a claim purely in terms of 
elements of the free \PROP.


\begin{lemma}
  Let $\Omega$ be an $\N^2$-graded set with \(\natural \notin 
  \Omega\). Let networks \(G,H,K \in \Nw(\Omega)\) and \(r,q \in \N\) 
  such that $K \join^r_q H$ is a network be given. Then \([G]_\simeq = 
  [K]_\simeq \join^r_q [H]_\simeq\) if and only if $K \join^r_q H$ is 
  a $\natural$-subdivision of $G$.
\end{lemma}
\begin{proof}
  There is always a $\natural$-homeomorphism 
  $(\mathrm{id},\mathrm{smoothen})$ from $K \join^r_q H$ to 
  \(G' := \mathrm{smoothen}(K \join^r_q\nobreak H)\). 
  If \([G]_\simeq = [K]_\simeq \join^r_q [H]_\simeq\), then there is 
  by definition of $\join$ an isomorphism from $G'$ to $G$, and this 
  can be composed with $(\mathrm{id},\mathrm{smoothen})$ to yield 
  the necessary homeomorphism from $K \join^r_q H$ to $G$. 
  If conversely there is a $\natural$-homeomorphism from 
  $K \join^r_q H$ to $G$, then \(G \simeq G'\) by 
  Lemma~\ref{L:Dropping}.
\end{proof}

The definition makes it clear how to compute the symmetric join of 
two explicitly given networks, but it is less clear how one might 
compute it for more abstract expressions. It turns out, however, that 
the abstract index notation makes a good match with the symmetric 
join.

\begin{theorem} \label{S:AIN-sammanbindning}
  Let $\Omega$ be an $\N^2$-graded set with \(\natural \notin 
  \Omega\). Let $\mc{R}$ be an associative and commutative unital 
  ring. Let $\{a_i\}_{i=1}^I,\{b_j\}_{j=1}^J \subset 
  \mc{R}\{\Omega\}$ be arbitrary. Let $\{\vek{k}_i\}_{i=0}^I$, 
  $\{\vek{l}_i\}_{i=0}^I$, $\{\vek{m}_j\}_{j=0}^J$, and 
  $\{\vek{n}_j\}_{j=0}^J$ be lists of labels. Let $\vek{q}$ and 
  $\vek{r}$ be lists such that \(\norm{\vek{r}} = 
  \bigl( \bigcup_{i=0}^I \norm{\vek{l}_i} \bigr) \setminus \bigl( 
  \bigcup_{i=0}^I \norm{\vek{k}_i} \bigr)\) and \(\norm{\vek{q}} = 
  \bigl( \bigcup_{j=0}^J \norm{\vek{n}_j} \bigr) \setminus \bigl( 
  \bigcup_{j=0}^J \norm{\vek{m}_j} \bigr) \). Then
  \begin{equation} \label{Eq:AIN-sammanbindning}
    \fuse[\bigg]{ \vek{k}_0 \vek{m}_0 }{
      \prod_{i=1}^I (a_i)^{\vek{l}_i}_{\vek{k}_i}
      \prod_{j=1}^J (b_j)^{\vek{n}_j}_{\vek{m}_j}
    }{ \vek{l}_0 \vek{n}_0 }
    =
    \fuse[\bigg]{ \vek{k}_0 \vek{r} }{
      \prod_{i=1}^I (a_i)^{\vek{l}_i}_{\vek{k}_i}
    }{ \vek{l}_0 \vek{q} }
    \join^{\Norm{\vek{r}}}_{\Norm{\vek{q}}}
    \fuse[\bigg]{ \vek{q} \vek{m}_0 }{
      \prod_{j=1}^J (b_j)^{\vek{n}_j}_{\vek{m}_j}
    }{ \vek{r} \vek{n}_0 }
  \end{equation}
  \Dash that is to say: if the label lists are such that the abstract 
  index expression in the left hand side is defined, then the right 
  hand side is defined too, and the two expressions are equal.
\end{theorem}
\begin{proof}
  By Lemma~\ref{L:AIN-multilinear}, abstract index expressions are 
  multilinear, and the symmetric join is bilinear by definition. 
  Hence it is sufficient to verify \eqref{Eq:AIN-sammanbindning} for 
  \(\{a_i\}_{i=1}^I, \{b_j\}_{j=1}^J \subset \Nwt(\Omega)\). By 
  Theorem~\ref{S:fuse-relabel} any bijective relabelling of an 
  abstract index expression is equivalent to the original one, so it 
  can without loss of generality be assumed that all the labels are 
  natural numbers, so that
  \begin{align*}
    G ={}& \Nwfuse[\bigg]{ \vek{k}_0 \vek{m}_0 }{
      \prod_{i=1}^I (a_i)^{\vek{l}_i}_{\vek{k}_i}
      \prod_{j=1}^J (b_j)^{\vek{n}_j}_{\vek{m}_j}
    }{ \vek{l}_0 \vek{n}_0 }
    \text{,}\\
    K ={}& \Nwfuse[\bigg]{ \vek{k}_0 \vek{r} }{
      \prod_{i=1}^I (a_i)^{\vek{l}_i}_{\vek{k}_i}
    }{ \vek{l}_0 \vek{q} }
    \text{,}&
    H ={}& \Nwfuse[\bigg]{ \vek{r} \vek{m}_0 }{
      \prod_{j=1}^J (b_j)^{\vek{n}_j}_{\vek{m}_j}
    }{ \vek{q} \vek{n}_0 }
  \end{align*}
  are elements of $\Nw\bigl( \Nwt(\Omega) \bigr)$. That the 
  expressions for $K$ and $H$, as well as their counterparts in 
  \eqref{Eq:AIN-sammanbindning}, fulfill the arity and acyclicity 
  conditions is immediate from the fact that they are fulfilled for 
  $G$, but the matching condition depends on the way $\vek{q}$ and 
  $\vek{r}$ were chosen. 
  Since the expression for $G$ fulfills the matching condition, 
  \(\bigcup_{i=0}^I \norm{\vek{k}_i} \cup \bigcup_{j=0}^J 
  \norm{\vek{m}_j} = \bigcup_{i=0}^I \norm{\vek{l}_i} \cup 
  \bigcup_{j=0}^J \norm{\vek{n}_j}\), and these unions are disjoint. 
  Hence
  \begin{multline*}
    \norm{\vek{r}}
    = 
    \biggl( \bigcup_{i=0}^I \norm{\vek{l}_i} \biggr) \setminus 
      \biggl( \bigcup_{i=0}^I \norm{\vek{k}_i} \biggr)
    =
    \biggl( 
      \bigcup_{i=0}^I \norm{\vek{l}_i} \cup
      \bigcup_{j=0}^J \norm{\vek{n}_j}
    \biggr) \setminus \biggl( 
      \bigcup_{i=0}^I \norm{\vek{k}_i} \cup
      \bigcup_{j=0}^J \norm{\vek{n}_j}
    \biggr)
    = \\ =
    \biggl( 
      \bigcup_{i=0}^I \norm{\vek{k}_i} \cup 
      \bigcup_{j=0}^J \norm{\vek{m}_j}
    \biggr) \setminus \biggl( 
      \bigcup_{i=0}^I \norm{\vek{k}_i} \cup
      \bigcup_{j=0}^J \norm{\vek{n}_j}
    \biggr)
    =
    \biggl( 
      \bigcup_{j=0}^J \norm{\vek{m}_j}
    \biggr) \setminus \biggl( 
      \bigcup_{j=0}^J \norm{\vek{n}_j}
    \biggr)
  \end{multline*}
  meaning $\norm{\vek{r}}$ is the half of the symmetric difference 
  between $\bigcup_{j=0}^J \norm{\vek{m}_j}$ and $\bigcup_{j=0}^J 
  \norm{\vek{n}_j}$ which is not in $\vek{q}$, and thus the 
  expression for $H$ fulfills the matching condition. Similarly one 
  sees that $\vek{q}$ is the part of the symmetric difference between 
  $\bigcup_{i=0}^I \norm{\vek{k}_i}$ and $\bigcup_{i=0}^I 
  \norm{\vek{l}_i}$ that is not in $\vek{r}$.
  
  The left hand side of \eqref{Eq:AIN-sammanbindning} is by 
  \eqref{Eq1:Network-value} equal to $\eval(G)$, and likewise 
  $\eval(K)$ and $\eval(H)$ are equal to the two abstract index 
  expressions in the right hand side of 
  \eqref{Eq:AIN-sammanbindning}. Hence what needs to be shown is that  
  \(\eval(G) = \eval(K) \join^{\Norm{\vek{r}}}_{\Norm{\vek{q}}} 
  \eval(H)\). Since the symmetric join on $\Nwt(\Omega)$ is defined in 
  terms of the symmetric join on $\Nw(\Omega)$, a way to establish 
  that equality is to exhibit \(G' \in \eval(G)\), \(K' \in 
  \eval(K)\), and \(H' \in \eval(H)\) such that $G'$ and $K' 
  \join^{\Norm{\vek{r}}}_{\Norm{\vek{q}}} H'$ have a common
  $\natural$-subdivision.
  Pick some \(K_i \in a_i\) for \(i=1,\dotsc,I\) and \(H_j \in b_j\) 
  for \(j=1,\dotsc,J\). Let
  \begin{align*}
    G'' ={}& \Nwfuse[\bigg]{ \vek{k}_0 \vek{m}_0 }{
      \prod_{i=1}^I (K_i)^{\vek{l}_i}_{\vek{k}_i}
      \prod_{j=1}^J (H_j)^{\vek{n}_j}_{\vek{m}_j}
    }{ \vek{l}_0 \vek{n}_0 }
    \text{,}&
    K'' ={}& \Nwfuse[\bigg]{ \vek{k}_0 \vek{r} }{
      \prod_{i=1}^I (K_i)^{\vek{l}_i}_{\vek{k}_i}
    }{ \vek{l}_0 \vek{q} }
    \text{,}\\ &&
    H'' ={}& \Nwfuse[\bigg]{ \vek{q} \vek{m}_0 }{
      \prod_{j=1}^J (H_j)^{\vek{n}_j}_{\vek{m}_j}
    }{ \vek{r} \vek{n}_0 }
    \text{.}
  \end{align*}
  Clearly 
  \(G'', K'', H'' \in \Nw\bigl( \Nw(\Omega) \bigr)\). Keeping in mind 
  that $\eval$ on $\Nw\bigl( \Nwt(\Omega) \bigr)$ is also defined by 
  \eqref{Eq:FlattenSmoothen}, it follows that the smoothenings of the 
  flattenings of $G''$, $K''$, and $H''$ are elements of $\eval(G)$, 
  $\eval(K)$, and $\eval(H)$ respectively, whence these smoothenings 
  of flattenings are candidates for being the $G'$, $K'$, and $H'$ 
  described above.
  
  As for subdivisions, it is easy to see that defining
  \begin{align*}
    \beta(v) ={}& v && \text{if \(v \in \{0,1\}\),}\\
    \beta(v) ={}& \Norm{\vek{r}} + \Norm{\vek{q}} + 2v
      && \text{if \(v-1 \in [I]\),}\\
    \beta(v) ={}& \Norm{\vek{r}} + \Norm{\vek{q}} + 2(v-I) + 1
      && \text{if \(v-I-1 \in [J]\),} \displaybreak[0]\\
    \gamma(2e) ={}& e && \text{if \(e \in E_{K''}\),}\\
    \gamma(2e+1) ={}& e && \text{if \(e \in E_{H''}\)}
  \end{align*}
  makes $(\beta,\gamma)$ a $\natural$-homeomorphism from $K'' 
  \join^{\Norm{\vek{r}}}_{\Norm{\vek{q}}} H''$ to $G''$. Hence there 
  is also a ditto homeomorphism from the flattening of the former to 
  the flattening of the latter, and thus the flattening of $K'' 
  \join^{\Norm{\vek{r}}}_{\Norm{\vek{q}}} H''$ is the sought common 
  $\natural$-subdivision.
\end{proof}

Having the ability to put a generic abstract index expression back 
togeter from an arbitrary bipartition of it makes the symmetric join 
a very versatile operation; one recipe for expressing something as a 
symmetric join is to first express it as an abstract index expression 
and then collect the individual factors into separate operands of the 
symmetric join. For example, one can use this to show that \(a \circ b 
\circ c = \fuse{ \vek{k} }{ a^\vek{k}_\vek{l} b^\vek{l}_\vek{m} 
c^\vek{m}_\vek{n} }{ \vek{n} } = \fuse { \vek{km} }{ 
a^\vek{k}_\vek{l} c^\vek{m}_\vek{n} }{ \vek{nl} } 
\join^{\Norm{\vek{m}}}_{\Norm{\vek{l}}} \fuse{ \vek{l} }{ 
b^\vek{l}_\vek{m} }{ \vek{m} } = \bigl( a \otimes\nobreak 
c \circ\nobreak \phi(\cross{{\Norm{\vek{n}}}}{{\Norm{\vek{l}}}}) \bigr) 
\rtimes b\) for appropriate label lists $\vek{k}$, $\vek{l}$, 
$\vek{m}$, and $\vek{n}$.

\begin{corollary} \label{Kor:Sammanbindingsformler}
  Let $\Omega$ be an $\N^2$-graded set with \(\natural \notin 
  \Omega\). Let $\mc{R}$ be an associative and commutative unital 
  ring. The following hold for any \(k,l,m,n \in \N\):
  \begin{enumerate}
    \item
      \(a \otimes b = a \join^0_0 b\) for all \(a \in 
      \mc{R}\{\Omega\}(k,l)\) and \(b \in \mc{R}\{\Omega\}(m,n)\).
    \item
      \(a \join^0_m b = a \circ b = b \join^m_0 a\) for all 
      \(a \in \mc{R}\{\Omega\}(l,m)\) and \(b \in 
      \mc{R}\{\Omega\}(m,n)\).
    \item \label{Item3:Sammanbindingsformler}
      \(a \join^m_n \phi(\cross{m}{n}) = a = \phi(\cross{l}{k}) 
      \join^l_k a\) for all \(a \in \mc{R}\{\Omega\}(k +\nobreak m, l 
      +\nobreak n)\).
  \end{enumerate}
\end{corollary}
\begin{proof}
  Letting \(\vek{k} = \vek{N}_k(0,4)\), \(\vek{l} = \vek{N}_l(1,4)\), 
  \(\vek{m} = \vek{N}_m(2,4)\), and \(\vek{n} = \vek{N}_n(3,4)\), 
  all these identities follow from Theorem~\ref{S:AIN-sammanbindning} 
  for suitable choices of $\vek{q}$ and $\vek{r}$. Taking both empty 
  and employing Theorem~\ref{S:AIN,klyva}, one gets
  \begin{equation*}
    a \otimes b =
    \fuse{ \vek{k} }{ a^\vek{k}_\vek{l} }{ \vek{l} }
    \otimes 
    \fuse{ \vek{m} }{ b^\vek{m}_\vek{n} }{ \vek{n} }
    =
    \fuse{ \vek{km} }{ 
      a^\vek{k}_\vek{l} b^\vek{m}_\vek{n} 
    }{ \vek{ln} }
    =
    \fuse{ \vek{k} }{ a^\vek{k}_\vek{l} }{ \vek{l} }
    \join^0_0
    \fuse{ \vek{m} }{ b^\vek{m}_\vek{n} }{ \vek{n} }
    =
    a \join^0_0 b
    \text{.}
  \end{equation*}
  Taking \(\vek{q} = \vek{m}\) and $\vek{r}$ empty gives
  \begin{equation*}
    a \circ b
    =
    \fuse{ \vek{l} }{
      a^\vek{l}_\vek{m} b^\vek{m}_\vek{n}
    }{ \vek{n} }
    =
    \fuse{ \vek{l} }{ a^\vek{l}_\vek{m} }{ \vek{m} }
    \join^0_m
    \fuse{ \vek{m} }{ b^\vek{m}_\vek{n} }{ \vek{n} }
    =
    a \join^0_m b \text{,}
  \end{equation*}
  whereas taking \(\vek{r} = \vek{m}\) and $\vek{q}$ empty gives
  \begin{equation*}
    a \circ b
    =
    \fuse{ \vek{l} }{
      a^\vek{l}_\vek{m} b^\vek{m}_\vek{n}
    }{ \vek{n} }
    =
    \fuse{ \vek{m} }{ b^\vek{m}_\vek{n} }{ \vek{n} }
    \join^m_0
    \fuse{ \vek{l} }{ a^\vek{l}_\vek{m} }{ \vek{m} }
    =
    b \join^m_0 a \text{.}
  \end{equation*}
  Taking \(\vek{q} = \vek{n}\) and \(\vek{r} = \vek{m}\) gives
  \begin{equation*}
    a = 
    \fuse{ \vek{km} }{ a^\vek{km}_\vek{ln} }{ \vek{ln} }
    =
    \fuse{ \vek{km} }{ a^\vek{km}_\vek{ln} }{ \vek{ln} }
    \join^m_n
    \fuse{ \vek{nm} }{ 1 }{ \vek{mn} }
    =
    a \join^m_n \phi(\cross{m}{n})
  \end{equation*}
  whereas taking \(\vek{q} = \vek{k}\) and \(\vek{r} = \vek{l}\) gives
  \begin{equation*}
    a = 
    \fuse{ \vek{km} }{ a^\vek{km}_\vek{ln} }{ \vek{ln} }
    =
    \fuse{ \vek{kl} }{ 1 }{ \vek{lk} }
    \join^l_k
    \fuse{ \vek{km} }{ a^\vek{km}_\vek{ln} }{ \vek{ln} }
    =
    \phi(\cross{l}{k}) \join^l_k a \text{.}
  \end{equation*}
\end{proof}

The symmetric join formulae for $\otimes$ and $\circ$ above suggest 
that an alternative approach to introducing the \PROP\ structure on 
$\Nwt(\Omega)$ would be to start with the symmetric join and then use 
these formulae to define the $\otimes$ and $\circ$ operations. The 
kind of associativity that Lemma~\ref{L:join-associativitet} 
establishes for the symmetric join gives both
\[
  (a \circ b) \circ c =
  (a \join^0_l b) \join^0_m c =
  a \join^0_l (b \join^0_m c) =
  a \circ (b \circ c)
\]
and
\[
  (a \otimes b) \otimes c =
  (a \join^0_0 b) \join^0_0 c =
  a \join^0_0 (b \join^0_0 c) =
  a \otimes (b \otimes c)
  \text{.}
\]
Similarly Lemma~\ref{L:Symjoin-transposition} contains the bulk of 
the tensor permutation axiom for \PROPs\ and 
item~\ref{Item3:Sammanbindingsformler} of the corollary has the tensor 
and permutation identity axioms as special cases.

A disadvantage of the symmetric join as defined so far, which would 
be more pronounced if seeking to use it as the primitive operation in 
the free \PROP, is that it is a \emph{partial} operation not 
necessarily defined for every pair of operands. It is however possible 
to distinguish fairly large subsets of $\Nwt(\Omega)$ on which one can 
predict that a certain 
symmetric join is always defined; knowing this should help reduce any 
worries one might have about strange corner cases in the extension to 
$\mc{R}\{\Omega\}$ of the symmetric join. As it turns out, the 
subsets where $\join$ is predictably defined constitute a filtration.

\begin{lemma} \label{L:Pullback-filtrering}
  Let $\mc{P}$ and $\mc{Q}$ be \PROPs, let \(f\colon \mc{P} \Fpil 
  \mc{Q}\) be a \PROP\ homomorphism, and let $Q$ be a \PROP\ 
  quasi-order on $\mc{Q}$. Then the family $\{F_q\}_{q \in \mc{Q}}$Ê
  of subsets of $\mc{P}$ defined by
  \begin{equation}
    F_q = \setOf[\big]{ a \in \mc{P} }{ f(a) \leqslant q \pin{Q} }
    \quad\text{for all \(q \in \mc{Q}\)}
  \end{equation}
  if a $(\mc{Q},Q)$-filtration of $\mc{P}$.
\end{lemma}
\begin{proof}
  What must be done is to verify the conditions from 
  Definition~\ref{Def:PROP-filtrering}. That \(F_q \subseteq 
  \mc{P}\bigl( \omega(q), \alpha(q) \bigr)\) is because \(f(a) 
  \leqslant q \pin{Q}\) implies \(\omega(a) = \omega\bigl( f(a) 
  \bigr) = \omega(q)\) and \(\alpha(a) = \alpha\bigl( f(a) 
  \bigr) = \alpha(q)\) by definition of $\N^2$-graded set morphism 
  $f$ and quasi-order $Q$. That \(F_q \subseteq F_r\) if \(q 
  \leqslant r \pin{Q}\) follows from the transitivity of $Q$ via
  \[
    a \in F_q  \Epil  f(a) \leqslant q \pin{Q}
    \Ipil
    f(a) \leqslant r \pin{Q}   \Epil  a \in F_r
    \text{.}
  \]
  For \(a \in F_q\) and \(b \in F_r\) one finds that
  \begin{align*}
    f(a \otimes b) ={}& f(a) \otimes f(b) \leqslant q \otimes r 
    \text{,}\\
    f(a \circ b) ={}& f(a) \circ f(b) \leqslant q \circ r
    \text{.}
  \end{align*}
  For any permutation $\sigma$, \(f\bigl( \phi_\mc{P}(\sigma) 
  \bigr) = \phi_\mc{Q}(\sigma) \leqslant \phi_\mc{Q}(\sigma) 
  \pin{Q}\) and thus \(\phi_\mc{P}(\sigma) \in 
  F_{\phi_\mc{Q}(\sigma)}\). Finally it is a filtration \emph{of} 
  $\mc{P}$ since any \(a \in \mc{P}\) belongs to $F_{f(a)}$.
\end{proof}

\begin{definition} \label{Def:Y&M}
  Let an $\N^2$-graded set $\Omega$ and an associative and 
  commutative ring $\mc{R}$ with unit be given. For every \(A \in 
  \B^{\bullet\times\bullet}\), let
  \begin{align}
    \mc{Y}_\Omega(A) :={}& \setOf[\Big]{ 
      a \in \Nwt(\Omega)\bigl( \omega(A), \alpha(A) \bigr)
    }{ \Tr(a) \leqslant A }
      \text{,}\\
    \mc{M}_\Omega(A) :={}& \Span\bigl( \mc{Y}_\Omega(A) \bigr)
      \subseteq \mc{R}\{\Omega\}\bigl( \omega(A), \alpha(A) \bigr)
    \text{.}
  \end{align}
  In contexts where the signature $\Omega$ is fixed, the subscript may 
  be omitted from $\mc{Y}_\Omega$ and $\mc{M}_\Omega$.
\end{definition}

\begin{corollary}
  $\bigl\{ \mc{Y}_\Omega(A) \bigr\}_{A \in \B^{\bullet\times\bullet}}$ 
  is a filtration of $\Nwt(\Omega)$. 
  $\bigl\{ \mc{M}_\Omega(A) \bigr\}_{A \in \B^{\bullet\times\bullet}}$ 
  is an $\mc{R}$-linear filtration of $\mc{R}\{\Omega\}$. 
\end{corollary}
\begin{proof}
  For $\mc{Y}$, this follows immediately from 
  Lemma~\ref{L:Pullback-filtrering} by taking the transference 
  \(\Tr\colon \Nwt(\Omega) \Fpil \B^{\bullet\times\bullet}\) as the 
  homomorphism; that $\Tr$ is a homomorphism was the subject of 
  Lemma~\ref{L:TrHomomorfi}. For $\mc{M}$, this follows from the 
  linearity of $\circ$ and $\otimes$ on $\mc{R}\{\Omega\}$.
\end{proof}


\begin{remark}
  If \(k,l,m,n,r,q \in \N\), \(A_{11} \in \B^{k \times l}\), 
  \(A_{12} \in \B^{k \times q}\), 
  \(A_{21} \in \B^{r \times l}\), \(A_{22} \in \B^{r \times q}\), 
  \(B_{22} \in \B^{q \times r}\), \(B_{23} \in \B^{q \times n}\), 
  \(B_{32} \in \B^{m \times r}\), and \(B_{33} \in \B^{m \times n}\) 
  are such that $A_{22}B_{22}$ is nilpotent, then the symmetric join 
  gives rise to maps
  \begin{align*}
    \join^r_q \colon \mc{Y}_\Omega\left(
      \begin{bmatrix} A_{11} & A_{12} \\ A_{21} & A_{22} \end{bmatrix}
    \right) \times \mc{Y}_\Omega\left(
      \begin{bmatrix} B_{22} & B_{23} \\ B_{32} & B_{33} \end{bmatrix}
    \right) \Fpil {}& \mc{Y}_\Omega\left( 
      \begin{bmatrix} C_{11} & C_{13} \\ C_{31} & C_{33} \end{bmatrix}
    \right)\\
    \join^r_q \colon \mc{M}_\Omega\left(
      \begin{bmatrix} A_{11} & A_{12} \\ A_{21} & A_{22} \end{bmatrix}
    \right) \times \mc{M}_\Omega\left(
      \begin{bmatrix} B_{22} & B_{23} \\ B_{32} & B_{33} \end{bmatrix}
    \right) \Fpil {}& \mc{M}_\Omega\left( 
      \begin{bmatrix} C_{11} & C_{13} \\ C_{31} & C_{33} \end{bmatrix}
    \right)
  \end{align*}
  where
  \begin{align*}
    C_{11} :={}&
      A_{11} + A_{12} B_{22} (A_{22} B_{22})^* A_{21} \text{,}&
    C_{13} :={}&
      A_{12} (B_{22} A_{22})^* B_{23} \text{,}\\
    C_{31} :={}&
      B_{32} (A_{22} B_{22})^* A_{21} \text{,}&
    C_{33} :={}&
      B_{33} + B_{32} (A_{22} B_{22})^* A_{22} B_{23} \text{.}
  \end{align*}
  A variation on this observation, where one keeps the left factor 
  fixed, is the core of Definition~\ref{D:KategorinV}.
\end{remark}


The symmetric join moreover has some factorisation properties that 
are almost reminiscent of what one has in a commutative polynomial 
algebra, but there are also some easily believed hypotheses about 
$\join$-factorisation that turn out to not be true. One of these is 
that \(a_1 \join^r_q b = a_2 \join^r_q b\) should imply \(a_1=a_2\), 
but for any \(c \in \Nwt(\Omega)(1,0)\) it turns out that 
\(\phi(\same{2}) \join^0_2 (c \otimes\nobreak c) = \phi(\same{2}) 
\circ (c \otimes\nobreak c) = c \otimes c = \phi(\cross{1}{1}) 
\circ (c \otimes\nobreak c) = \phi(\cross{1}{1}) \join^0_2 (c 
\otimes\nobreak c)\). Another example is that \(\fuse{ikj}{1}{ijk} 
\join^1_1 \phi(\same{1}) = \phi(\same{2}) = \fuse{kji}{1}{ijk} 
\join^1_1 \phi(\same{1})\). On the other hand, if one can nail down 
the join-vertices and vertex parity then isomorphism of symmetric 
joins implies isomorphism of respective factors by the second part of 
Lemma~\ref{L:Join-homeomorfi}

\begin{definition}
  Let a homeomorphism $(\beta,\gamma)$ from a network $L$ to a 
  network $G$ be given. For any \(e \in E_L\) such that \(h_L(e) \in 
  V_L \setminus \setim\beta\), one says that \(\gamma(e) \in E_G\) is 
  \DefOrd[*{edge subdivided}]{the edge subdivided by} the vertex 
  $h_L(e)$. The \DefOrd{subdivision index} of the edge \(e \in E_L\) 
  is the index $k$ such that \(e=e_k\) if \(\{e_1,\dotsc,e_n\} = 
  \setinv{\gamma}\bigl( \bigl\{ \gamma(e) \bigr\} \bigr)\) are 
  indexed such that \(t_L(e_1) \in \setim\beta\), 
  \(h_L(e_i)=t_L(e_{i+1})\) for \(i \in [n -\nobreak 1]\), and 
  \(t_L(e_n) \in \setim\beta\). The \DefOrd[*{subdivision index}]
  {subdivision index of the vertex} \(h_L(e) \in V_L \setminus 
  \setim\beta\) is the subdivision index of $e$. All of these 
  concepts are with respect to the homeomorphism $(\beta,\gamma)$.
\end{definition}

That the subdivision index is well-defined is a consequence of 
Lemma~\ref{L:UnderdeladKant}; defining $\zeta(v)$ to be the edge 
subdivided by $v$ and $\theta(v)$ to be the subdivision index, one 
even gets that \((\zeta,\theta)\colon V_L \setminus \setim\beta \Fpil 
E_G \times \Zp\) is an injective map. One point of the next lemma is 
that these data determine the subdivision up to isomorphism, but also 
that between subdivisions made using a subset of the vertices there 
are homeomorphisms in the expected manner.

\begin{lemma} \label{L:UnderdelningNatverk}
  Let $\Omega$ be an $\N^2$-graded set and \(y \in \Omega(1,1)\). 
  Let \(G \in \Nw(\Omega)\) and \(X \subseteq \N \setminus V_G\) be 
  given, together with functions \(\zeta\colon X \Fpil E_G\) and 
  \(\theta\colon X \Fpil \Zp\) such that \((\zeta,\theta)\colon 
  X \Fpil E_G \times \Zp\) is injective. Then for every 
  \(A \subseteq X\) there exists a $y$-subdivision $G(A)$ of $G$ with 
  \(V_{G(A)} = V_G \cup A\) and a homeomorphism from $G(A)$ to $G$ of 
  the form $(\mathrm{id},\gamma_A)$ such that
  \begin{itemize}
    \item 
      If \(h_{G(A)}(e) = u\) or \(t_{G(A)}(e) = u\) for some \(u \in 
      A\) and \(e \in E_{G(A)}\) then \(\gamma_A(e) = \zeta(u)\).
    \item
      If \(e \in E_{G(A)}\) is such that \(h_{G(A)}(e), t_{G(A)}(e) 
      \in A\) then \(\theta\bigl( h_{G(A)}(e) \bigr) > 
      \theta\bigl( t_{G(A)}(e) \bigr)\).
  \end{itemize}
  Moreover there exist \(\Psi_{A,B}\colon E_{G(B)} \Fpil E_{G(A)}\) 
  for all \(A \subseteq B \subseteq X\) such that $(\mathrm{id}, 
  \Psi_{A,B})$ is a homeomorphism from $G(B)$ to $G(A)$ and 
  \(\gamma_B = \gamma_A \circ \Psi_{A,B}\).
\end{lemma}
\begin{proof}
  One way to show this is to give an explicit construction of $G(A)$. 
  The vertex labels are prescribed in the statement, and a way of 
  inventing edge labels would be to make them $p$-pairs of an edge 
  label from $G$ and the position along that edge as reflected in the 
  value of $\theta$ for the tail end (not quite the subdivision 
  index, but close enough). This amounts to
  \begin{align*}
    V_{G(A)} ={}& V_G \cup A \text{,}\\
    D_{G(A)}(v) ={}& \begin{cases}
      D_G(v)& \text{if \(v \in V_G \setminus \{0,1\}\),}\\
      y& \text{if \(v \in A\),}
    \end{cases}\\
    E_{G(A)} ={}& \setOf[\big]{ p(e,0) }{ e \in E_G } \cup
      \setOf[\Big]{ p\bigl( \zeta(u), \theta(u) \bigr) }{ 
        u \in A }
      \text{,}\\
    \gamma_A\bigl( p(e,i) \bigr) ={}& e \text{,}\\
  \intertext{and of course}
    (t_{G(A)},s_{G(A)})\bigl( p(e,i) \bigr) ={}& \begin{cases}
      (t_G,s_G)(e)& \text{if \(i=0\),}\\
      \bigl( (\zeta,\theta)^{-1}(e,i), 1 \bigr) &
        \text{otherwise.}
    \end{cases}
  \end{align*}
  The head map is a ``next vertex on edge'' map, so defining
  \begin{align*}
    N(e,i,A) ={}& \setOf[\big]{ \theta(v) }{ 
      \text{\(v \in A\), \(\zeta(v) = e\), and \(\theta(v) > i\)}
    } \text{,}\\
    M(e,i,A) ={}& \setOf[\big]{ \theta(v) }{ 
      \text{\(v \in A\), \(\zeta(v) = e\), and \(\theta(v) \leqslant i\)}
    } \cup \{0\}
  \end{align*}
  as helper functions, one can state it as
  \begin{equation*}
    (h_{G(A)},g_{G(A)})\bigl( p(e,i) \bigr) = \begin{cases}
      (h_G,g_G)(e)& \text{if \(N(e,i,A) = \varnothing\),}\\
      \Bigl( (\zeta,\theta)^{-1}\bigl( e, \min N(e,i,A) \bigr), 1
      \Bigr)& \text{otherwise.}
    \end{cases}
  \end{equation*}
  The first three conditions for $(\mathrm{id},\gamma_A)$ to be a 
  homeomorphism are trivially fulfilled. If \(p(e,i) \in E_{G(A)}\) 
  is such that \(h_{G(A)}\bigl( p(e,i) \bigr) \in V_G\) then this 
  head is by definition \(h_G(e) = (\mathrm{id} \circ\nobreak h_G 
  \circ\nobreak \gamma_A)\bigl( p(e,i) \bigr)\) and similarly at the 
  tail. Any \(v \in A\) has $p\bigl( \zeta(v), \theta(v) \bigr)$ 
  as only edge with $v$ as tail and $p\bigl( \zeta(v), i \bigr)$ 
  where \(i = \max M\bigl( \zeta(v), \theta(v) -\nobreak 1, A \bigr)\) 
  as only edge with $v$ as head, both of which are mapped to 
  $\zeta(v)$ by $\gamma_A$. Hence it has been shown that 
  $(\mathrm{id},\gamma_A)$ is a homeomorphism from $G(A)$ to $G$.
  
  The definition of $\Psi_{A,B}$ is \(\Psi_{A,B}\bigl( p(e,i) \bigr) 
  = p\bigl( e, \max M(e,i,A) \bigr)\), which trivially 
  satisfies \(\gamma_A \circ \Psi_{A,B} = \gamma_B\). To see that it is 
  surjective, one may observe that if \(p(e,i) \in E_{G(A)}\) then 
  \(p(e,i) \in E_{G(B)}\) too since \(A \subseteq B\), and clearly 
  \(\max M(e,i,A) = i\), meaning \(\Psi_{A,B}\bigl( p(e,i) \bigr) = 
  p(e,i)\). These edges are precisely those which 
  have \(t_{G(B)}\bigl( p(e,i) \bigr) \in V_{G(A)}\), and the 
  restriction of $(t_{G(B)},s_{G(B)})$ to the set of these edges is 
  exactly $(t_{G(A)},s_{G(A)})$. On the head end, let \(j = 
  \max M(e,i,A)\). \(h_{G(B)}\bigl( p(e,i) \bigr) \in 
  V_G\) iff \(N(e,i,B) = \varnothing\), which implies \(N(e,j,A) = 
  \varnothing\), and thus \((h_{G(B)},g_{G(B)})\bigl( p(e,i) \bigr) = 
  (h_G,g_G)(e) = (h_{G(A)},g_{G(A)})\bigl( p(e,j) \bigr)\) as 
  required for $(\mathrm{id},\Psi_{A,B})$ to be a homeomorphism. For 
  edges $p(e,i)$ with \(N(e,i,B) \neq \varnothing\) there are two 
  cases for \(v = h_{G(B)}\bigl( p(e,i) \bigr) = 
  (\zeta,\theta)^{-1}\bigl( e, \min N(e,i,B) \bigr)\). If \(v \in 
  A\) then since \(\theta(v) > i \geqslant j\) it follows that 
  \(\theta(v) \in N(e,j,A)\) and thus \(\min N(e,j,A) = \theta(v) 
  = \min N(e,i,B)\), meaning \((h_{G(B)},g_{G(B)})\bigl( p(e,i) \bigr) 
  = (v,1) = (h_{G(A)},g_{G(A)})\bigl( p(e,j) \bigr)\). If instead \(v 
  \in B \setminus A\) then \(M\bigl( e, \theta(v), A) = 
  M(e,i,A)\) and hence \(\Psi_{A,B}\bigl( p(e,\theta(v)) \bigr) = 
  p(e,j) = \Psi_{A,B}\bigl( p(e,i) \bigr)\), as required at a vertex 
  that is smoothed over by a homeomorphism. The remaining conditions 
  for  $(\mathrm{id},\Psi_{A,B})$ to be a homeomorphism are trivially 
  fulfilled.
\end{proof}

The following theorem states that two $\rtimes$-factorisations of a 
network can be combined, in the sense of combined factorisation that 
any factor of a component factorisation appears as a join of factors 
in the combined factorisation. For practical reasons, the result is 
stated in the particular case that one right factor has an embedding 
into the other right factor, but multiple applications of it would 
yield a combined factorisations also in the more general situation 
where two factors partially overlap; the trick is to observe that 
their ``intersection'' has embeddings into both.

\begin{theorem} \label{S:Inbaddningsuppdelning}
  Let \(G,H_1,H_2,K_1,K_2 \in \Nw(\Omega)\) be networks such that 
  \([G]_\simeq = [K_1]_\simeq \rtimes [H_1]_\simeq = [K_2]_\simeq 
  \rtimes [H_2]_\simeq\). Let \((\chi_1,\psi_1)\colon H_1 \Fpil G\) and 
  \((\chi_2,\psi_2)\colon H_2 \Fpil G\) be the corresponding embeddings. 
  If there exists an intermediate embedding \((\chi_0,\psi_0)\colon 
  H_1 \Fpil H_2\) such that \(\chi_1 = \chi_2 \circ \chi_0\) and 
  \(\psi_1 = \psi_2 \circ \psi_0\) then there also exists some 
  \(K_0 \in \Nw(\Omega)\) such that \([K_1]_\simeq = [K_2]_\simeq 
  \join^{\alpha(H_2)}_{\omega(H_2)} [K_0]_\simeq\) and \([H_2]_\simeq 
  = [K_0]_\simeq \rtimes [H_1]_\simeq\).
\end{theorem}
\begin{remark}
  It may appear as though this theorem should follow immediately from 
  Theorem~\ref{S:InbaddningHomeomorfi}\Ldash since $H_1$ has an 
  embedding into $H_2$ it follows that there is some $K_0$ for which 
  $H_2$ is a subdivision of $K_0 \rtimes H_1$, and hence
  \begin{multline} \label{Eq:Inbaddningsuppdelning}
    [K_1]_\simeq \rtimes [H_1]_\simeq = 
    [G]_\simeq = 
    [K_2]_\simeq \rtimes [H_2]_\simeq = 
    [K_2]_\simeq \rtimes \bigl( [K_0]_\simeq \rtimes\nobreak 
      [H_1]_\simeq \bigr) = \\ =
    [K_2]_\simeq \join^{\alpha(H_2)}_{\omega(H_2)} 
      \bigl( [K_0]_\simeq \rtimes\nobreak [H_1]_\simeq \bigr) = 
    \bigl( [K_2]_\simeq \join^{\alpha(H_2)}_{\omega(H_2)}\nobreak 
      [K_0]_\simeq \bigr) \rtimes [H_1]_\simeq
    \text{.}
  \end{multline}
  The catch is however that a mere embedding of $H_1$ into $H_2$ need 
  not uniquely identify the padding $[K_0]_\simeq$\Ldash one needs a 
  strong embedding for that\Rdash and thus it could happen that $K_2 
  \join^{\alpha(H_2)}_{\omega(H_2)} K_0$ as constructed above is not 
  a subdivision of the given $K_1$, but a subdivision of some other 
  network $K_1'$ which also happens to satisfy \([K_1']_\simeq 
  \rtimes [H_1]_\simeq = [G]_\simeq\). That $[K_2]_\simeq 
  \join^{\alpha(H_2)}_{\omega(H_2)} [K_0]_\simeq \rtimes [H_1]_\simeq$ 
  is a factorisation combining the \emph{exact} decompositions from 
  $[K_1]_\simeq \rtimes [H_1]_\simeq$ and $[K_2]_\simeq \rtimes 
  [H_2]_\simeq$ is what the theorem is all about.
\end{remark}
\begin{proof}
  Even if \eqref{Eq:Inbaddningsuppdelning} alone does not produce the 
  wanted result, it captures the basic idea; what is missing is a 
  detailed record of the subdivisions that are involved and the 
  homeomorphisms between them. The first step for producing this will 
  be to use Lemma~\ref{L:UnderdelningNatverk} to construct a 
  subdivision $G(A)$ of $G$ which contains all join-vertices 
  introduced by the symmetric joins in 
  \eqref{Eq:Inbaddningsuppdelning}.
  
  Let \(n_1 = \max V_G\), let \(A_1 = \{n_1+j\}_{j \in 
  [\alpha(H_1)+\omega(H_1)]}\), let \(n_2 = n_1 + 
  \alpha(H_1)+\omega(H_1)\), let \(A_2 = \{n_2 + j\}_{j 
  \in [\alpha(H_2)+\omega(H_2)]}\), and let \(A = A_1 \cup A_2\). For 
  \(i=1,2\), let $(\beta_i,\gamma_i)$ be the homeomorphism from 
  $K_i \rtimes H_i$ to $G$. Define \(\zeta\colon A \Fpil E_G\) by 
  making $\zeta(n_i +\nobreak j)$ be the edge that under 
  $(\beta_i,\gamma_i)$ is subdivided by \(2+j \in V_{K_i \rtimes 
  H_i}\), and let $\theta_i(n_i +\nobreak j)$ be the corresponding 
  subdivision index of the join vertex $2+j$. Also let $\theta_2'(e)$ 
  be the subdivision index of \(e \in E_{K_2 \rtimes H_2}\) with 
  respect to $(\beta_2,\gamma_2)$. Let \(m = \alpha(H_1) + 
  \omega(H_1) + 1\), let $f(v)$ for \(v \in A_1\) be the edge \(e \in 
  E_{H_1}\) such that $2e+1$ is incident with $v -n_1 +2$ in $K_1 
  \rtimes H_1$ (i.e., \(h_{K_1 \rtimes H_1}\bigl( 2f(v) +\nobreak 1 
  \bigr) = v -n_1 +2\) for \(v > n_1 + \alpha(H_1)\), and 
  \(t_{K_1 \rtimes H_1}\bigl( 2f(v) +\nobreak 1 \bigr) = v -n_1 +2\) 
  for \(v \leqslant n_1 + \alpha(H_1)\)), and define 
  \(\theta\colon A \Fpil \Zp\) by
  \begin{equation*}
    \theta(v) = \begin{cases}
      m \theta_2(v) & \text{if \(v \in A_2\),}\\
      m \theta_2'\bigl( 2(\psi_0 \circ f)(v) + 1 \bigr) - m + 
        \theta_1(v)
        & \text{if \(v \in A_1\).}
    \end{cases}
  \end{equation*}
  Since \(0 < \theta_1(v) < m\) for all \(v \in A_1\), it follows from 
  the injectivity of $(\zeta,\theta_1)$ and $(\zeta,\theta_2)$ that 
  $(\zeta,\theta)$ is injective as well. Hence the conditions in 
  Lemma~\ref{L:UnderdelningNatverk} are fulfilled and the three 
  subdivisions (in the notation of that lemma) $G(A)$, $G(A_1)$, and 
  $G(A_2)$ of $G$ exist.
  
  First consider $G(A_2)$, which has a homeomorphism 
  $(\mathrm{id},\gamma_3)$ to $G$. Since for any \(e \in E_G\),
  \begin{equation*}
    \card[\big]{ \setinv{\gamma_3}(\{e\}) } =
    1 + \card[\big]{ A_2 \cap \setinv{\zeta}(\{e\}) } =
    \card[\big]{ \setinv{\gamma_2}(\{e\}) }
    \text{,}
  \end{equation*}
  it follows from Corollary~\ref{Kor:Homeomorfilyft} that \(G(A_2) 
  \simeq K_2 \rtimes H_2\), with an isomorphism 
  \((\chi_3,\psi_3)\colon G(A_2) \Fpil K_2 \rtimes H_2\) satisfying 
  \(\chi(v) = \beta_3(v)\) for \(v \in V_G\). Since $\theta$ and 
  $\theta_2$ assign the same relative order to subdivision vertices, 
  it also follows that the relation between \(v \in A_2\) and $\chi_3(v)$ 
  is the same as between $\theta_2$ and the $(\beta_2,\gamma_2)$ 
  subdivision index: \(\chi_3(v) = 2 + v - n_2\) for \(v \in A_2\).
  
  Next observe that $(\chi_3,\psi_3)$ composes with the canonical 
  homeomorphism from $G(A)$ to $G(A_2)$ to yield a homeomorphism 
  $(\beta_4,\gamma_4)$ from $G(A)$ to $K_2 \rtimes H_2$; this 
  satisfies \(\beta_4(v) = \beta_2^{-1}(v)\) for \(v \in \setim\beta_2
  \) and \(\beta_4(2 +\nobreak j) = n_2 + j\) for \(j \in 
  \bigl[ \alpha(H_2) +\nobreak \omega(H_2) \bigr]\). The vertices not 
  in $\setim\beta_4$ are those in $A_1$, 
  so one may conversely wonder which edges these subdivide. Any 
  \(n_1+j \in A_1\) for \(j \leqslant \alpha(H_1)\) 
  corresponds to the edge \(e \in E_{H_1}\) with \(t_{H_1}(e)=1\) and 
  \(s_{H_1}(e) = j\), so \(\theta(n_1 +\nobreak j) = 
  m\theta_2'\bigl( 2\psi_0(e) +\nobreak 1 \bigr) - m + 
  \theta_1(n_1 +\nobreak j)\). This means it subdivides $2\psi_0(e) + 
  1$, because \(\zeta(2 +\nobreak j) = \gamma_2\bigl( 2\psi_0(e) 
  +\nobreak 1 \bigr)\), if \(u := h_{K_2 \rtimes H_2}\bigl( 2\psi_0(e) 
  +\nobreak 1 \bigr) \notin \setim\beta_2\) then \(\theta\bigl( 
  \beta_4(u) \bigr) = m\theta_2\bigl( \beta_4(u) \bigr) = 
  m\theta_2'\bigl( 2\psi_0(e) +\nobreak 1 \bigr) > 
  \theta( n_1 +\nobreak j )\), and if 
  \(v := t_{K_2 \rtimes H_2}\bigl( 2\psi_0(e) +\nobreak 1 \bigr) \notin 
  \setim\beta_2\) then \(\theta\bigl( \beta_4(v) \bigr) = 
  m\theta_2\bigl( \beta_4(u) \bigr) \leqslant m\theta_2'\bigl( 
  2\psi_0(e) +\nobreak 1 \bigr) - m < \theta( n_1 +\nobreak j )\). 
  Similarly provided \(t_{H_1}(e)=1\), any \(n_1 + \alpha(H_1) + 
  g_{H_1}(e) \in A_1\) subdivides the edge \(2\psi_0(e) + 1 \in 
  E_{K_2 \rtimes H_2}\) under $(\beta_4,\gamma_4)$.
  
  Being a homeomorphism to a symmetric join, $(\beta_4,\gamma_4)$ 
  transports the symmetric join structure back to $G(A)$; a mere  
  relabelling of the edges and vertices will make this explicit. 
  Concretely there is a bipartition $E_0 \cup E_1$ of $E_{G(A)}$ into 
  the ``even'' edges coming from the left factor and the ``odd'' edges 
  coming from the right factor:
  \begin{align*}
    E_0 ={}& \setOf[\big]{ e \in E_{G(A)} }{ 2 \mid \gamma_4(e) } 
      \text{,}&
    E_1 ={}& \setOf[\big]{ e \in E_{G(A)} }{ 2 \nmid \gamma_4(e) } 
      \text{.}
  \end{align*}
  Similarly there is a tripartition $V_0 \cup V_\natural \cup V_1$ of 
  $V_{G(A)}$ into ``even'', ``neutral'', and ``odd'' vertices, where 
  one preferably defines
  \begin{align*}
    V_\natural ={}& \{0,1\} \cup A_2 = \setmap{\beta_4}\Bigl( 
      \{0,1\} \cup \{ 2 + j\}_{j \in [\alpha(H_2)+\omega(H_2)]} 
    \Bigr) \text{,}\\
    V_0 ={}& \setmap{h_{G(A)}}(E_0) \cup \setmap{t_{G(A)}}(E_0) 
      \setminus V_\natural
      \text{,}\\
    V_1 ={}& \setmap{h_{G(A)}}(E_1) \cup \setmap{t_{G(A)}}(E_1) 
      \setminus V_\natural
      \text{.}
  \end{align*}
  These $V_0$ and $V_1$ come out as disjoint, since if for example 
  \(e_0 \in E_0\) and \(e_1 \in E_1\) are such that \(h_{G(A)}(e_0) = 
  t_{G(A)}(e_1)\) then this vertex must be in $\setim\beta_4$ since 
  \(\gamma_4(e_0) \neq \gamma_4(e_1)\), and thus \(h_{K_2 \rtimes 
  H_2}\bigl( \gamma_4(e_0) \bigr) = t_{K_2 \rtimes H_2}\bigl( 
  \gamma_4(e_1) \bigr)\), which implies that \(2 < h_{K_2 \rtimes 
  H_2}\bigl( \gamma_4(e_0) \bigr) \leqslant 2 + \alpha(H_2)\) and 
  thus \(h_{G(A)}(e_0) \in V_\natural\); the other combinations of 
  head and tail work out similarly.
  
  Since $\gamma_4$ is injective on $E_0$ (because the $A_1$ vertices 
  all subdivide edges in the odd $H_2$ part of $K_2 \rtimes H_2$), it 
  turns out that the left factor can be taken to be $K_2$. A 
  relabelling $(\chi_4,\psi_4)$ of $G(A)$ which makes it so is given 
  by
  \begin{align*}
    \psi_4(e) ={}& \begin{cases}
      2e+1 & \text{if \(e \in E_1\),}\\
      \gamma_4(e) & \text{if \(e \in E_0\),}
    \end{cases}\\
    \chi_4(v) ={}& \begin{cases}
      \beta_2(v) & \text{if \(v \in V_0\),}\\
      \beta_4^{-1}(v)& \text{if \(v \in V_\natural\),}\\
      \alpha(H_2) + \omega(H_2) + 2v + 1& \text{if \(v \in V_1\);}
    \end{cases}
  \end{align*}
  this is an isomorphism from $G(A)$ to a network of the form 
  $K_2 \rtimes H_2'$, where \(E_{H_2'} = E_1\) and \(V_{H_2'} = 
  \{0,1\} \cup V_1\). By Lemma~\ref{L:Join-homeomorfi}, the composite 
  homeomorphism $(\chi_4 \circ\nobreak \beta_4, \gamma_4 
  \circ\nobreak \psi_4^{-1})$ splits into separate homeomorphisms for 
  each factor, and so there is a homeomorphism $(\beta_5,\gamma_5)$ 
  from $H_2'$ to $H_2$.
  
  On the $A_1$ side it similarly follows that \(G(A_1) \simeq K_1 
  \rtimes H_1\) and therefore there is a homeomorphism 
  $(\beta_6,\gamma_6)$ from $G(A)$ to $K_1 \rtimes H_1$ which has 
  \(\beta_6(2 +\nobreak j) = n_1 + j\) for all \(n_1 + j \in A_1\) 
  and \(\beta_6\bigl( \beta_1(v) \bigr) = v\) for all \(v \in V_G\). 
  For no \(e \in E_{H_1}\) is there any \(u \in A_2\) that subdivides 
  \(2e+1 \in E_{K_1 \rtimes H_1}\) under $(\beta_6,\gamma_6)$, 
  because $\setinv{\gamma_6}\bigl( \{ 2e +\nobreak 1\} \bigr)$ is 
  contained within $\setinv{\gamma_4}\bigl( \{ 2\psi_0(e) +\nobreak 
  1\} \bigr)$. On the tail end, \(u \in A_2\) with \(\theta(u) 
  \leqslant m\theta_2'\bigl( 2\psi_0(e) +\nobreak 1 \bigr) - m\) cannot 
  subdivide 
  $2e+1$, because the existence of that $u$ forces \(\theta_2'\bigl( 
  2\psi_0(e) +\nobreak 1 \bigr) \geqslant 2\) and thus \(t_{H_2}\bigl( 
  \psi_0(e) \bigr) = 1\), whence \(t_{H_1}(e) = 1\) renders 
  \(\beta_6\bigl( t_{K_1 \rtimes H_1}( 2e +\nobreak 1) \bigr) = n_1 + 
  s_{H_1}(e) \in A_1\), and \(\theta\bigl( n_1 +\nobreak s_{H_1}(e) 
  \bigr) > m\theta'\bigl( 2\psi_0(e) +\nobreak 1 \bigr) - m \geqslant 
  \theta(u)\); such $u$ must be subdividing some other edge of $K_1 
  \rtimes H_1$. Similarly on the head end, some \(u \in A_2\) with 
  \(\theta(u) \geqslant m\theta_2'\bigl( 2\psi_0(e) +\nobreak 1 \bigr)\) 
  cannot subdivide $2e+1$, because the existence of such a $u$ requires 
  there to be a segment of $\psi_2\bigl( \psi_0(e) \bigr)$ with 
  subdivision index higher than that of $\psi_0(e)$, whence 
  \(h_{H_2}\bigl( \psi_0(e) \bigr) = 0\) and thus \(h_{H_1}(e) = 0\), 
  from which follows \(\beta_6\bigl( h_{K_1 \rtimes H_1}( 2e +\nobreak 
  1) \bigr) = n_1 + \alpha(H_1) + g_{H_1}(e) \in A_1\), and 
  \(\theta\bigl( n_1 +\nobreak \alpha(H_1) +\nobreak g_{H_1}(e) 
  \bigr) < m\theta'\bigl( 2\psi_0(e) +\nobreak 1 \bigr) \leqslant 
  \theta(u)\). Therefore there is for every \(e \in E_{H_1}\) a 
  unique \(e' \in E_{G(A)}\) such that \(\gamma_6(e') = 2e+1\), and 
  moreover every \(e' \in E_{G(A)}\) with \(\gamma_6(e') \equiv 1 
  \pmod{2}\) has \(\gamma_4(e') \equiv 1 \pmod{2}\), \(h_{G(A)}(e') = 
  (\beta_6 \circ\nobreak h_{K_1 \rtimes H_1} \circ\nobreak 
  \gamma_6)(e')\) and \(t_{G(A)}(e') = (\beta_6 \circ\nobreak 
  t_{K_1 \rtimes H_1} \circ\nobreak \gamma_6)(e')\).
  
  As before, the symmetric join structure of $K_1 \rtimes H_1$ pulls 
  back to $G(A)$ over the homeomorphism $(\beta_6,\gamma_6)$, and 
  since all of $E_0$ is in the even part, it can be seen as a 
  symmetric join structure on $H_2'$. The bipartition of $E_1$ is
  \[
    E_{10} = \setOf[\big]{ e \in E_1 }{ 2 \mid \gamma_6(e) }
    \qquad\text{and}\qquad
    E_{11} = \setOf[\big]{ e \in E_1 }{ 2 \nmid \gamma_6(e) }
    \text{,}
  \]
  and since furthermore $\gamma_6$ is injective on $E_{11}$, one can take 
  the right factor in the symmetric join producing $H_2'$ to be $H_1$.  
  A relabelling which does this is $(\chi_5,\psi_5)$, where
  \begin{align*}
    \chi_5(v) ={}& \begin{cases}
      v - n_1 + 2& \text{if \(v \in A_1\),}\\
      \beta_1(v)& \text{if \(v \in V_{H_2'} \setminus A_1\),}
    \end{cases} \displaybreak[0]\\
    \psi_5(e) ={}& \begin{cases}
      2e& \text{if \(e \in E_{10}\),}\\
      \gamma_6(e)& \text{if \(e \in E_{11}\).}
    \end{cases}
  \end{align*}
  From this relabelling of $H_2'$ one gets $K_0$ such that \(H_2' 
  \simeq K_0 \rtimes H_1\), and thus \([H_2]_\simeq = [K_0]_\simeq 
  \rtimes [H_1]_\simeq\), which is half the claim in the theorem.
  
  For the other half, one may observe that the isomorphism sequence 
  \(G(A) \simeq K_2 \rtimes H_2' \simeq K_2 \rtimes (K_0 \rtimes\nobreak 
  H_1) \simeq \bigl( K_2 \join^{\alpha(H_2)}_{\omega(H_2)}\nobreak 
  K_0 \bigr) \rtimes H_1\) composes with $(\beta_6,\gamma_6)$ to yield 
  a homeomorphism $(\beta_7,\gamma_7)$ from 
  $\bigl( K_2 \join^{\alpha(H_2)}_{\omega(H_2)}\nobreak K_0 \bigr) 
  \rtimes H_1$ to $K_1 \rtimes H_1$. Working through the steps one 
  finds that \(\beta_7(2 +\nobreak j) = 2 + j\) for \(j \in \bigl[ 
  \alpha(H_1) +\nobreak \omega(H_1) \bigr]\) and \(\beta_7(v) \equiv 
  1 \pmod{2}\) if and only if \(v \equiv 1 \pmod{2}\) for \(v \in 
  V_{K_1 \rtimes H_1}\). Hence $(\beta_7,\gamma_7)$ splits by 
  Lemma~\ref{L:Join-homeomorfi} and yields in the left factor a 
  homeomorphism $(\beta_8,\gamma_8)$ from $K_2 
  \join^{\alpha(H_2)}_{\omega(H_2)} K_0$ to $K_1$. Thus \([K_1]_\simeq = 
  [K_2]_\simeq \join^{\alpha(H_2)}_{\omega(H_2)} [K_0]_\simeq\).
\end{proof}

\section{Feedbacks}
\label{Sec:Feedbacks}

Symmetric joins of networks will in the rewriting formalism be used for 
placing rules in contexts. Since rules are expected to be compatible 
with an ordering of the networks, this implies that this ordering must 
be preserved by the operation of placing something in a context. 
Therefore there is a need to have a name for orderings that are 
preserved under symmetric join.

\begin{definition}
  Let $\Omega$ be an $\N^2$-graded set. 
  An $\N^2$-graded quasi-order $P$ on the free \PROP\ $\Nwt(\Omega)$ 
  is said to be \DefOrd[*{N2-graded@$\N^2$-graded!quasi-order!preserved 
  under symmetric join}]{preserved under symmetric join} if
  \begin{equation} \label{Eq:BeveradUnderSammanbindning}
    a \leqslant a' \pin{P} \text{ and }
    b \leqslant b' \pin{P} \qquad\text{implies}\qquad
    a \join^r_q b \leqslant a' \join^r_q b' \pin{P}
  \end{equation}
  for all \(r,q \in \N\), \(A,B \in \B^{\bullet\times\bullet}\), 
  \(a,a' \in \mc{Y}_\Omega(A)\), and \(b,b' \in \mc{Y}_\Omega(B)\) such 
  that \(A = \left[ \begin{smallmatrix} A_{11} & A_{12} \\ A_{21} & A_{22} 
  \end{smallmatrix} \right]\) where \(A_{22} \in \B^{r \times q}\), 
  \(B = \left[ \begin{smallmatrix} B_{22} & B_{23} \\ B_{32} & 
  B_{33} \end{smallmatrix} \right]\) where \(B_{22} \in \B^{q \times 
  r}\), and $A_{22}B_{22}$ is nilpotent. The quasi-order $P$ is said 
  to be \DefOrd[*{N2-graded@$\N^2$-graded!quasi-order!preserved under 
  symmetric join}]{strictly preserved} if 
  a strict inequality for either factor in 
  \eqref{Eq:BeveradUnderSammanbindning} implies a strict inequality 
  between the respective symmetric joins.
\end{definition}

\begin{lemma} \label{L:OrdningSammanbindingPROP}
  An $\N^2$-graded quasi-order $P$ on the free \PROP\ $\Nwt(\Omega)$ 
  which is preserved under symmetric join is a \PROP\ quasi-order. If 
  $P$ is strictly preserved then it is a strict \PROP\ quasi-order.
\end{lemma}
\begin{proof}
  In order to relate a $P(l,m)$ comparison of $b$ with $b'$ to 
  a $P(k,n)$ comparison of $c \circ b \circ d$ with $c \circ b' \circ 
  d$, where \(c \in \Nwt(\Omega)(k,l)\) and \(d \in \Nwt(\Omega)(m,n)\), 
  let \(a = c \otimes \phi(\same{m}) \circ 
  \phi(\cross{m}{l}) \circ d \otimes \phi(\same{l})\), as this makes 
  \(c \circ b \circ d = a \rtimes b\) and \(c \circ b' \circ d = a 
  \rtimes b'\).
  
  A similar argument for $c \otimes b \otimes d$ where \(c \in 
  \Nwt(\Omega)(i,j)\), \(b,b' \in \Nwt(\Omega)(k,l)\), and \(d \in 
  \Nwt(\Omega)(m,n)\) follows from taking \(a = \phi(\same{i} 
  \star\nobreak \cross{m}{k} \star\nobreak \same{l}) \circ c \otimes 
  d \otimes \phi(\cross{l}{k}) \circ \phi(\same{j} \star\nobreak 
  \cross{l}{n} \star\nobreak \same{k})\) and comparing \(a \rtimes b 
  = c \otimes b \otimes d\) with \(a \rtimes b' = c \otimes b \otimes 
  d\).
\end{proof}

Previous work on \PROP\ quasi-orders is therefore not lost, as all 
quasi-orders preserved under symmetric join are to be found among the 
\PROP\ quasi-orders, but it reintroduces the question of how they 
might be constructed. As stated before, the main approach to the 
construction of orders on $\Nwt(\Omega)$ is to pull back an order on 
some other other \PROP\ $\mc{P}$ over an evaluation homomorphism 
$\eval_f$, but this works well because the \PROP\ operations are 
defined on both the domain and the codomain. The symmetric join is 
only defined on $\Nw(\Omega)$ and $\Nwt(\Omega)$, so in seeking to 
establish preservation under symmetric join of an order pulled back, 
one faces the problem of showing things about 
$\eval_f(a \join^r_q\nobreak b)$ given only information about 
$\eval_f(a)$ and $\eval_f(b)$\Dash it is not clear that one can even 
compute the former from just the two latter! As it turns out, the 
\PROP\ operations alone need not suffice, but if the \PROP\ structure 
of $\mc{P}$ is augmented with a \emph{feedback} then there is a formula 
(which is the subject of Theorem~\ref{S:Feedback-join}) in $\mc{P}$ for 
$\eval_f(a \join^r_q\nobreak b)$.

\subsection{\PROPs\ with unrestricted feedback}

Informally, a feedback is to take an output and connect it back to an 
input; in some cases this makes sense, and in other cases it doesn't. 
Consequently, some \PROPs\ support an extra `feedback' operation and 
others do not. We will primarily consider feedbacks that are partial 
operations\Ldash defined only on such elements in the \PROP\ that 
they make sense\Rdash but it is useful to first define what they 
behave like when defined on all elements.

\begin{definition} \label{Def:Feedback}
  A \DefOrd[*{PROP@\PROP!with feedback}]{\PROP\ with feedback} is a 
  \PROP\ $\mc{P}$ together with a family 
  \(\{\feedback{n}\}_{n\in\N}\)\index{\feedback@$\feedback{n}$} of 
  partial unary operations\index{feedback} such that \(a\feedback{n} 
  \in \mc{P}(k,l)\) for all \(a \in \mc{P}(k +\nobreak n, 
  l +\nobreak n)\). These must satisfy the following axioms:
  \begin{description}
    \item[tightening]
      \(\bigl( a \otimes \phi(\same{n}) \circ b \circ c \otimes 
      \phi(\same{n}) \bigr)\feedback{n} = a \circ b\feedback{n} \circ 
      c\) for all \(a \in \mc{P}(i,j)\), \(b \in \mc{P}(j +\nobreak 
      n, k +\nobreak n)\), and \(c \in \mc{P}(k,l)\).
    \item[superposing]
      \(a \otimes b\feedback{n} = (a \otimes b)\feedback{n}\) for all 
      \(a \in \mc{P}(i,j)\), \(b \in \mc{P}(k +\nobreak n, l 
      +\nobreak n)\).
    \item[sliding]
      \(\bigl( \phi(\same{k}) \otimes a \circ b \bigr)\feedback{m} = 
      \bigl( b \circ \phi(\same{l}) \otimes a \bigr)\feedback{n}\) 
      for all \(a \in \mc{P}(m,n)\), \(b \in \mc{P}(k +\nobreak 
      n, l +\nobreak m)\).
    \item[vanishing]
      \(a\feedback{n}\feedback{m} = a\feedback{m+n}\) if 
      \(\alpha(a),\omega(a) \geqslant m+n\) and 
      \(a\feedback{0} = a\) for all \(a \in \mc{P}\).
    \item[yanking]
      \(\phi(\cross{n}{n}) \feedback{n} = \phi(\same{n})\).
  \end{description}
  If in addition they satisfy the axiom of
  \begin{description}
    \item[normalisation]
      \(\phi(\same{n})\feedback{n} = \phi(\same{0})\),
  \end{description}
  then the feedback is \DefOrd[*{normalised feedback}]{normalised}.
\end{definition}

\begin{figure}
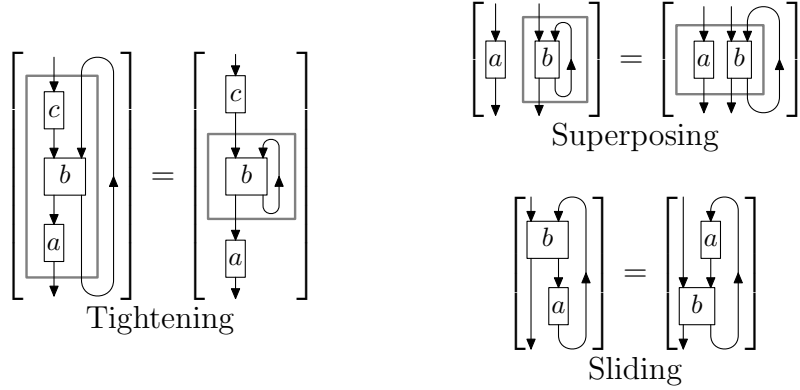

  \begin{center}
    \begin{tabular}{c}
      \(
      \left[ \begin{mpgraphics*}{60}
        beginfig(60);
          PROPdiagram(0,0)
            frame (
              box(1,1)(btex \strut$a$ etex) otimes same(1) circ
              box(2,2)(btex \strut$b$ etex) circ
              box(1,1)(btex \strut$c$ etex) otimes same(1) 
            ) feedback 1
          ;
        endfig;
      \end{mpgraphics*}\right]
      =
      \left[ \begin{mpgraphics*}{61}
        beginfig(61);
          PROPdiagram(0,0)
            box(1,1)(btex \strut$a$ etex) circ
            frame (
              box(2,2)(btex \strut$b$ etex) feedback 1
            ) circ
            box(1,1)(btex \strut$c$ etex) 
          ;
        endfig;
      \end{mpgraphics*}\right]
      \)\\
      Tightening
    \end{tabular}
    \hfil
    \begin{tabular}{c}
      \(
      \left[ \begin{mpgraphics*}{63}
      beginfig(63);
        PROPdiagram(0,0)
          box(1,1)(btex \strut$a$ etex) otimes frame(
             box(2,2)(btex \strut$b$ etex) feedback 1
          ) 
        ;
      endfig;
      \end{mpgraphics*} \right]
      =
      \left[ \begin{mpgraphics*}{62}
      beginfig(62);
        PROPdiagram(0,0)
          frame(
            box(1,1)(btex \strut$a$ etex) otimes 
            box(2,2)(btex \strut$b$ etex)
          ) feedback 1
        ;
      endfig;
      \end{mpgraphics*} \right]
      \)\\
      Superposing
      \\
      \\
      \(
      \left[ \begin{mpgraphics*}{64}
        beginfig(64);
          PROPdiagram(0,0)
            (
              same(1) otimes box(1,1)(btex \strut$a$ etex) circ
              box(2,2)(btex \strut$b$ etex) 
            ) feedback 1
          ;
        endfig;
      \end{mpgraphics*}\right]
      =
      \left[ \begin{mpgraphics*}{65}
        beginfig(65);
          PROPdiagram(0,0)
            (
              box(2,2)(btex \strut$b$ etex) circ
              same(1) otimes box(1,1)(btex \strut$a$ etex)
            ) feedback 1
          ;
        endfig;
      \end{mpgraphics*}\right]
      \)\\
      Sliding
    \end{tabular}
  \end{center}
  
  \caption{Schematic representation of some feedback axioms}
\end{figure}

See \c{S}tef\u{a}nescu~\cite[pp.~22--24]{Stefanescu} for references 
and notes on this particular axiom system. He includes normalisation 
in the basic axiom system.

In the context of abstract index notation, taking a feedback is 
generally known as \emDefOrd{contracting}: using the same index as 
sub- and superscript of a single factor. When that factor is a 
matrix, one ends up with an expression of the form $A^i_i$, which for 
\(A \in \HomPROP_{\mc{R}^n}(1,1)\) and in more classical notation 
becomes \(\sum_{i=1}^n A_{i,i} = \tr(A)\); hence another term used for 
feedback operations is \emph{trace map}. Note however that the 
normalisation axiom would require $\feedback{1}$ of the identity 
matrix \(I = \fuse{i}{1}{i}\) to be $1$ rather than the dimension $n$, 
so a normalised $\feedback{1}$ is more like $\frac{1}{n}\tr$ than $\tr$. 
Such normalisation can actually be useful, because operators on 
infinite-dimensional spaces need not have a finite trace (those that do 
are said to be of \emDefOrd{trace-class}, which is comparatively small), 
whereas $\frac{1}{n}\tr$\Ldash an ``average diagonal element map''\Rdash 
generalises to a larger class of operators. On the other hand, the 
normalisation by $\frac{1}{n}$ in $\frac{1}{n}\tr$ will break the 
yanking axiom for $\HomPROP_{\mc{R}^n}$, so one cannot achieve 
normalisation simply by slapping on an extra factor. Furthermore it 
can be useful to go unnormalised; when Penrose~\cite{Penrose} wrote 
about ``applications of negative dimensional tensors'', he was 
referring specifically to theories where \(\fuse{i}{1}{i}\feedback{1}\) 
is negative.

\begin{theorem} \label{S:CupCapFeedback}
  Let $\mc{P}$ be a \PROP\ with elements \(U \in \mc{P}(0,2)\) (the 
  \emDefOrd{cup}) and \(\Lambda \in \mc{P}(2,0)\) (the \emDefOrd{cap}) 
  such that
  \begin{equation} \label{Eq1:CupCapFeedback}
    \phi(\same{1}) \otimes U \circ
      \phi(\cross{1}{1}) \otimes \phi(\same{1}) \circ
      \phi(\same{1}) \otimes \Lambda
    = \phi(\same{1}) \text{.}
  \end{equation}
  Then $\mc{P}$ is a \PROP\ with feedback, where
  \begin{align}
    a \feedback{0} :={}& a \text{,}\\
    a \feedback{1} :={}&
    \phi(\same{\omega(a)-1}) \otimes U \circ
      a \otimes \phi(\same{1}) \circ
      \phi(\same{\alpha(a)-1}) \otimes \Lambda
    \text{,}\\
    a \feedback{n} :={}& (a \feedback{1})\feedback{n-1} \text{,}
  \end{align}
  for all \(a \in \mc{P}\) and \(1 < n \leqslant \max\bigl\{ 
  \alpha(a), \omega(a) \bigr\}\). Defining
  \begin{align*}
    \Lambda_0 :={}& \phi(\same{0}) \text{,}&
    \Lambda_{n+1} :={}& 
      \phi(\same{n}) \otimes \Lambda \otimes \phi(\same{n}) \circ
      \Lambda_n
      \text{,}\\
    U_0 :={}& \phi(\same{0}) \text{,}&
    U_{n+1} :={}& 
      U_n \circ \phi(\same{n}) \otimes U \otimes \phi(\same{n})
  \end{align*}
  for all \(n \in \N\), it also holds that
  \begin{equation} \label{Eq4:CupCapFeedback}
    a \feedback{n} = 
    \phi(\same{\omega(a)-n}) \otimes U_n \circ
      a \otimes \phi(\same{n}) \circ
      \phi(\same{\alpha(a)-n}) \otimes \Lambda_n
  \end{equation}
  for all \(n \in \N\) and \(a \in \mc{P}\) with 
  \(\alpha(a),\omega(a) \geqslant n\).
\end{theorem}
\begin{remark}
  Instances of \eqref{Eq1:CupCapFeedback} are not all that common in 
  the literature; more common are instances of the ``zig--zag 
  identities''
  \begin{align*}
    U \otimes \phi(\same{1}) \circ \phi(\same{1}) \otimes \Pi
      ={}& \phi(\same{1}) \text{,}\\
    \phi(\same{1}) \otimes u \circ \Lambda \otimes \phi(\same{1})
      ={}& \phi(\same{1}) \text{,}
  \end{align*}
  which however are easily converted to \eqref{Eq1:CupCapFeedback} by 
  defining \(\Lambda = \phi(\cross{1}{1}) \circ \Pi\) or \(U = u 
  \circ \phi(\cross{1}{1})\) respectively. Furthermore, it is not 
  uncommon that cups and caps are symmetric (\(\Lambda=\Pi\) and 
  \(U=u\)). 
  One case where the cap and cup satisfying the zig--zag 
  identities are \emph{not} symmetric is that of ribbon 
  categories,\footnote{
    Which are natively braided monoidal categories rather than 
    symmetric ditto, so the issue of what $\phi(\cross{1}{1})$ even 
    means in that case is not entirely trivial. However if one does 
    not require the ribbons to be concrete topological objects then 
    it is mathematically no problem to consider ribbons that can pass 
    through each other, and then the braiding reduces to full 
    symmetry.
  } and in that case the value of $\phi(\same{1}) \otimes u \circ
  \phi(\cross{1}{1}) \otimes \phi(\same{1}) \circ \phi(\same{1}) 
  \otimes \Pi$ would be a twist of a ribbon rather than 
  $\phi(\same{1})$.
\end{remark}
\begin{proof}
  Beginning with \eqref{Eq4:CupCapFeedback}, it is clear that
  \begin{equation*}
    \phi(\same{\omega(a)}) \otimes U_0 \circ
      a \otimes \phi(\same{0}) \circ
      \phi(\same{\alpha(a)}) \otimes \Lambda_0
    =
    \phi(\same{\omega(a)}) \circ a \circ \phi(\same{\alpha(a)}) 
    =
    a
    =
    a\feedback{0} \text{,}
  \end{equation*}
  providing the basis for an induction on $n$. Assuming that 
  \eqref{Eq4:CupCapFeedback} holds for \(n=m\), it follows from
  \begin{align*}
    a \feedback{m+1}
    ={}&
    (a \feedback{1}) \feedback{m}
    = \\ ={}&
    \phi(\same{\omega(a)-1-m}) \otimes U_m 
      \\ & \quad{} \circ
    \bigl( 
      \phi(\same{\omega(a)-1}) \otimes U \circ
        a \otimes \phi(\same{1}) \circ
        \phi(\same{\alpha(a)-1}) \otimes \Lambda
    \bigr) \otimes \phi(\same{m}) 
      \\ & \quad{} \circ
    \phi(\same{\alpha(a)-1-m}) \otimes \Lambda_m
    = \displaybreak[0]\\ ={}&
    \phi(\same{\omega(a)-1-m}) \otimes U_m \circ
    \phi(\same{\omega(a)-1-m}) \otimes \phi(\same{m}) \otimes U 
      \otimes \phi(\same{m}) 
      \\ & \quad{} \circ
    a \otimes \phi(\same{1}) \otimes \phi(\same{m}) 
      \\ & \quad{} \circ
    \phi(\same{\alpha(a)-1-m}) \otimes \phi(\same{m}) \otimes \Lambda
      \otimes \phi(\same{m}) \circ
    \phi(\same{\alpha(a)-1-m}) \otimes \Lambda_m
    = \\ ={}&
    \phi(\same{\omega(a)-(m+1)}) \otimes U_{m+1} \circ
    a \otimes \phi(\same{m+1}) \circ
    \phi(\same{\alpha(a)-(m+1)}) \otimes \Lambda_{m+1}
  \end{align*}
  that it also holds for \(n=m+1\); hence it holds for general $n$.
  
  Concerning the axioms, one may observe that
  \begin{align*}
    U_n ={}& \fuse{}{ \textstyle \prod_{i=1}^n (U)_{r_i t_i} }{ 
        r_1 \dotsb r_n t_n \dotsb t_1
      }
      \text{,}\displaybreak[0]\\
    \Lambda_n ={}& \fuse{s_1 \dotsb s_n t_n \dotsb t_1}{ 
      \textstyle \prod_{i=1}^n (\Lambda)^{s_i t_i} }{}
      \text{,}\displaybreak[0]\\
    \phi(\cross{n}{n}) ={}& \fuse{ s_1 \dotsb s_n r_1 \dotsb r_n }{
      1
    }{ r_1 \dotsb r_n s_1 \dotsb s_n }
    \text{.}
  \end{align*}
  Hence defining \(\vek{r} = r_1 \dotsb r_n\), \(\vek{s} = s_1 \dotsb 
  s_n\), and \(\vek{t} = t_n \dotsb t_1\) (note order of $t_i$!), the 
  yanking axiom is verified through
  \begin{multline*}
    \phi(\cross{n}{n}) \feedback{n}
    = \\ =
    \phi(\same{n}) \otimes \fuse{}{
        \textstyle \prod_{i=1}^n (U)_{r_i t_i}
      }{\vek{r} \vek{t}} \circ
    \fuse{ \vek{s} \vek{r} }{
      1
    }{ \vek{r} \vek{s} } \otimes \phi(\same{n}) \circ
    \phi(\same{n}) \otimes 
      \fuse{\vek{s} \vek{t}}{
        \textstyle \prod_{i=1}^n (\Lambda)^{s_i t_i}
      }{}
    = \displaybreak[0]\\ =
    \fuse{\vek{s}}{
        \textstyle \prod_{i=1}^n (U)_{r_i t_i}
      }{\vek{s} \vek{r} \vek{t}} \circ
    \fuse{ \vek{s} \vek{r} \vek{t} }{
      1
    }{ \vek{r} \vek{s} \vek{t} } \circ
    \fuse{\vek{r} \vek{s} \vek{t}}{
        \textstyle \prod_{i=1}^n (\Lambda)^{s_i t_i}
      }{\vek{r}}
    = \displaybreak[0]\\ =
    \fuse{ \vek{s} }{ \textstyle
      \prod_{i=1}^n (\Lambda)^{s_i t_i}
      \prod_{i=1}^n (U)_{r_i t_i}
    }{ \vek{r} }
    = \\ =
    \bigotimes_{i=1}^n \fuse{ s_i }{ 
      (\Lambda)^{s_i t_i} (U)_{r_i t_i}
    }{ r_i }
    = 
    \bigotimes_{i=1}^n \phi(\same{1})
    =
    \phi(\same{n}) \text{.}
  \end{multline*}
  Vanishing is satisfied by definition. For sliding of \(a \in 
  \mc{P}(m,n)\) past \(b \in \mc{P}(k +\nobreak n, l +\nobreak m)\), 
  one may define \(\vek{r} = r_1 \dotsb r_l\), 
  \(\vek{s} = s_1 \dotsb s_m\), \(\vek{t} = t_1 \dotsb t_m\), 
  \(\vek{u} = u_1 \dotsb u_n\), \(\vek{v} = v_1 \dotsb v_n\), 
  \(\vek{w} = w_1 \dotsb w_k\), \(\vek{x} = x_1 \dotsb x_n\), and 
  \(\vek{z} = z_1 \dotsb z_m\). Then
  \begin{align*}
    \bigl( \phi(\same{k}) \otimes a \circ b \bigr) \feedback{m} 
    ={} \kern-6em& \\ ={}&
    \phi(\same{k}) \otimes U_m \circ
      \phi(\same{k}) \otimes a \otimes \phi(\same{m}) \circ 
      b \otimes \phi(\same{m}) \circ 
      \phi(\same{l}) \otimes \Lambda_m
    = \displaybreak[0]\\ ={}&
    \fuse{\vek{w}}{ (U_m)_{\vek{z}\vek{t}} }{\vek{wzt}} \circ
      \fuse{\vek{wzt}}{ a^\vek{z}_\vek{u} }{\vek{wut}} \circ
      \phi( \same{k} \star \same{n} \star \same{m}) \circ
      \fuse{\vek{wxt}}{ b^\vek{wx}_\vek{rs} }{\vek{rst}} \circ
      \fuse{\vek{rst}}{ (\Lambda_m)^{\vek{st}} }{\vek{r}}
    = \displaybreak[0]\\ ={}&
    \fuse{\vek{w}}{ 
        (U_m)_\vek{zt} a^\vek{z}_\vek{u} 
      }{\vek{wut}} \circ
      \fuse{\vek{wut}}{
        (\Lambda_n)^\vek{uv} (U_n)_\vek{xv}
      }{\vek{wxt}} \circ
      \fuse{\vek{wxt}}{ 
        b^\vek{wx}_\vek{rs} (\Lambda_m)^{\vek{st}} 
      }{\vek{r}}
    = \displaybreak[0]\\ ={}&
    \fuse{\vek{w}}{ 
        (U_m)_\vek{zt} a^\vek{z}_\vek{u} 
        (\Lambda_n)^\vek{uv} (U_n)_\vek{xv}
        b^\vek{wx}_\vek{rs} (\Lambda_m)^{\vek{st}} 
      }{\vek{r}}
    = \displaybreak[0]\\ ={}&
    \fuse{\vek{w}}{ 
        (U_n)_\vek{xv} b^\vek{wx}_\vek{rs} 
        (U_m)_\vek{zt} (\Lambda_m)^{\vek{st}}  
        a^\vek{z}_\vek{u} (\Lambda_n)^\vek{uv} 
      }{\vek{r}}
    = \displaybreak[0]\\ ={}&
    \fuse{\vek{w}}{ 
        (U_n)_\vek{xv} b^\vek{wx}_\vek{rs} 
      }{\vek{rsv}} \circ
      \fuse{\vek{rsv}}{ 
        (U_m)_\vek{zt} (\Lambda_m)^{\vek{st}}  
      }{\vek{rzv}} \circ
      \fuse{\vek{rzv}}{ 
        a^\vek{z}_\vek{u} (\Lambda_n)^\vek{uv} 
      }{\vek{r}}
    = \displaybreak[0]\\ ={}&
    \fuse{\vek{w}}{ (U_n)_\vek{xv} }{\vek{wxv}} \circ
      \fuse{\vek{wxv}}{ b^\vek{wx}_\vek{rs} }{\vek{rsv}} \circ
      \phi( \same{k} \star \same{m} \star \same{n}) \circ
      \fuse{\vek{rzv}}{ a^\vek{z}_\vek{u} }{\vek{ruv}} \circ
      \fuse{\vek{ruv}}{ (\Lambda_n)^\vek{uv} }{\vek{r}}
    = \displaybreak[0]\\ ={}&
    \phi(\same{k}) \otimes U_n \circ
      b \otimes \phi(\same{n}) \circ
      \phi(\same{l}) \otimes a \otimes \phi(\same{n}) \circ
      \phi(\same{l}) \otimes \Lambda_n
    = \\ ={}&
    \bigl( b \circ \phi(\same{l}) \otimes a \bigr) \feedback{n}
    \text{.}
  \end{align*}
  Superposing is merely that
  \begin{multline*}
    a \otimes b\feedback{n}
    =
    \bigl( 
      \phi(\same{i}) \circ a \circ \phi(\same{j})
    \bigr) \otimes \bigl(
      \phi(\same{k}) \otimes U_n \circ
      b \otimes \phi(\same{n}) \circ 
      \phi(\same{l}) \otimes \Lambda_n
    \bigr)
    = \\ =
    \phi(\same{i} \star \same{k}) \otimes U_n \circ
    a \otimes b \otimes \phi(\same{n}) \circ 
    \phi(\same{j} \star \same{l}) \otimes \Lambda_n
    =
    (a \otimes b) \feedback{n}
    \text{.}
  \end{multline*}
  Similarly tightening is that
  \begin{multline*}
    \bigl( 
      a \otimes \phi(\same{n}) \circ b \circ c \otimes \phi(\same{n}) 
    \bigr)\feedback{n} 
    = \\ =
    \phi(\same{i}) \otimes U_n \circ
    a \otimes \phi(\same{n}) \otimes \phi(\same{n}) \circ
    b \otimes \phi(\same{n}) \circ
    c \otimes \phi(\same{n}) \otimes \phi(\same{n}) \circ
    \phi(\same{l}) \otimes \Lambda_n
    = \\ =
    a \circ
    \phi(\same{j}) \otimes U_n \circ
    b \otimes \phi(\same{n}) \circ
    \phi(\same{k}) \otimes \Lambda_n \circ
    c 
    =
    a \circ b\feedback{n} \circ c
    \text{.}
  \end{multline*}
\end{proof}

\begin{example}
  The connectivity \PROP\ of Example~\ref{Ex:Sammanhangande} has a cap 
  and cup pair, namely
  \begin{align*}
    \Lambda ={}& \Bigl( \bigl\{ \{1,2\} \times \{0\} \bigr\}, 0 \Bigr)
      \text{,}\\
    U ={}& \Bigl( \bigl\{ \{1,2\} \times \{1\} \bigr\}, 0 \Bigr)
    \text{.}
  \end{align*}
  With respect to \eqref{Eq1:CupCapFeedback}, one finds that
  \begin{align*}
    \phi(\same{1}) \otimes \Lambda ={}&
      \Bigl( \Bigl\{ \bigl\{ (1,0), (1,1) \bigr\},
        \bigl\{ (2,0), (3,0) \bigr\} \Bigr\}, 
      0 \Bigr) \text{,}\\
    \phi(\cross{1}{1} \star \same{1}) ={}&
      \Bigl( \Bigl\{ \bigl\{ (1,0), (2,1) \bigr\},
        \bigl\{ (2,0), (1,1) \bigr\}, \bigl\{ (3,0), (3,1) \bigr\}
      \Bigr\}, 0 \Bigr) 
      \text{,}\\
    \phi(\same{1}) \otimes U ={}&
      \Bigl( \Bigl\{ \bigl\{ (1,0), (1,1) \bigr\},
        \bigl\{ (2,1), (3,1) \bigr\} \Bigr\}, 
      0 \Bigr) \text{,}
  \end{align*}
  meaning
  \begin{equation*}
    \phi(\same{1}) \otimes U \circ \phi(\cross{1}{1} \star \same{1})
    =
    \Bigl( \Bigl\{ \bigl\{ (1,0), (2,1) \bigr\},
      \bigl\{ (1,1), (3,1) \bigr\} \Bigr\}, 
    0 \Bigr) \text{,}
  \end{equation*}
  so that \(\phi(\same{1}) \otimes U \circ \phi(\cross{1}{1} 
  \star\nobreak \same{1}) \circ \phi(\same{1}) \otimes \Lambda = 
  \phi(\same{1})\).
  
  This \PROP\ is not normalised, since
  \begin{equation}
    U \circ \Lambda =
    ( \varnothing, 1 ) \neq
    ( \varnothing, 0 ) = \phi(\same{0})
    \text{.}
  \end{equation}
\end{example}

As mentioned previously, caps and cups may also turn up in 
$\HomPROP_V$ style \PROPs, but here it should be observed that it 
makes a significant difference whether $V$ is finite- or 
infinite-dimensional. When $V$ is finite-dimensional, 
$\{e_i\}_{i=1}^n$ is some basis of \(\HomPROP_V(1,0) \cong V\), and 
$\{f_i\}_{i=1}^n$ is the corresponding dual basis of 
$\HomPROP_V(0,1)\), then one can simply let \(\Lambda = 
\sum_{i=1}^n e_i \otimes e_i\) and \(U = \sum_{j=1}^n f_j \otimes 
f_j\), since
\begin{multline*}
  \phi(\same{1}) 
  =
  \sum_{i=1}^n e_i \circ f_i
  =
  \sum_{i=1}^n \sum_{j=1}^n e_i \circ (f_j \circ e_i) \circ f_j
  = \\ =
  \sum_{i=1}^n \sum_{j=1}^n \phi(\same{1}) \otimes f_j \otimes f_j
    \circ \phi(\cross{1}{1}) \otimes \phi(\same{1}) \circ
    \phi(\same{1}) \otimes e_i \otimes e_i
  = \\ =
  \phi(\same{1}) \otimes U \circ
    \phi(\cross{1}{1}) \otimes \phi(\same{1}) \circ
    \phi(\same{1}) \otimes \Lambda
  \text{.}
\end{multline*}
A taste of the infinite-dimensional case can be had by 
considering as $V$ some suitable class of functions \(\R \Fpil \C\). 
Zig--zags can be found for such things\Dash one starting point could 
be the Fourier and inverse Fourier transformation combination
\begin{equation} \label{Eq:Fouriertransform}
  f(x) = \frac{1}{2\pi} \int_{-\infty}^\infty e^{ipx}
    \int_{-\infty}^\infty e^{-ipy} f(y) \, \mathit{dy} \, \mathit{dp}
  \qquad\text{for all \(x \in \R\) and \(f \in V\).}
\end{equation}
The inner integral here is an inner product $\int_{-\infty}^\infty 
f(y) \overline{g(y)} \, \mathit{dy}$ for \(g(y) = e^{ipy}\), so this 
would constitute an obvious cup candidate\Dash were it not for the 
fact that this is sesquilinear rather than bilinear, as an element of 
$\HomPROP_V(0,2)$ would have to be. Instead taking that integral 
without conjugation as cup, and the rest of 
\eqref{Eq:Fouriertransform} as the corresponding cap, one gets
\begin{align*}
  \Lambda(x,y) ={}& 
    \frac{1}{2\pi} \int_{-\infty}^\infty e^{ip(x-y)} \, \mathit{dp}
    \quad\text{for \(x,y \in \R\),}\\
  U(f,g) ={}& \int_{-\infty}^\infty f(y) g(y) \, \mathit{dy}
    \quad\text{for \(f,g \in V\).}
\end{align*}
Because these were obtained from \eqref{Eq:Fouriertransform}, they 
satisfy the basic instance \eqref{Eq1:CupCapFeedback} of the yanking 
axiom, but regarding normalisation one finds that
$$
  \phi(\same{1})\feedback{1} = U \circ \Lambda =
  \frac{1}{2\pi} \int_{-\infty}^\infty \int_{-\infty}^\infty 1
   \, \mathit{dp} \, \mathit{dy}
$$
which is rather divergent; clearly the composition of two \PROP\ 
elements should rather be something finite.

Of course, one could argue that the divergence problem is present 
already in the definition of $\Lambda$. Certainly this $\Lambda$ is 
not an element of \(\HomPROP_V(2,0) = V \otimes V\), since \(\Lambda 
= \frac{1}{2\pi} \int_{-\infty}^\infty f_p \otimes g_p 
\,\mathit{dp}\) for \(f_p(x) = e^{ipx}\) and \(g_p(y) = e^{-ipy}\) is 
at best an integral over elements $f \otimes g$ for \(f,g \in V\), 
rather than the finite sum of such elements that an element of $V 
\otimes V$ must be; that the tensor product only does finite sums is 
also why the dual of an algebra need not be a coalgebra. There 
\emph{could} be some other \PROP\ which is similar to $\HomPROP_V$ but 
not so restricted (for example, it might be defined using a tensor 
product analogue that also takes a topological closure) in which 
$\Lambda$ finds a home. While it may seem to a mathematician that the 
integral should still be hopelessly divergent since the integrand does 
not tend to $0$ as \(p \rightarrow \pm\infty\), a physicist would 
probably have recognised $\frac{1}{2\pi} \int_{-\infty}^\infty 
e^{ip(x-y)} \,\mathit{dp}$ as a standard formula for the Dirac 
$\delta(x -\nobreak y)$, which is a very natural cap candidate when 
dealing with spaces of functions. Having tamed $\Lambda$ with the 
formal framework of Dirac delta functions and found a name for what it 
is does however not get rid of the infinity: \(U \circ \Lambda = 
\int_{-\infty}^\infty \delta(x -\nobreak x) \,\mathit{dx} = 
\int_{-\infty}^\infty \delta(0) \,\mathit{dx}\) is every bit as 
divergent as the previous expression for this composition.

\subsection{\PROPs\ with formal feedback}

The problem above is not so much constructing a feedback from separate 
cap and cup elements, but rather that the feedback itself becomes a 
troublesome concept when it must apply to arbitrary elements. (Several 
technicalities in Section~\ref{Sec:AIN} are all about making sense of 
the abstract index notation \emph{without} resorting to feedbacks; if 
the target is from the start known to be a \PROP\ with feedbacks then 
it would be straightforward to define an ``abstract index notation with 
contraction'' that drops the acyclicity constraint.)
An interesting problem when given a \PROP\ is therefore to distinguish 
sets of well-behaved elements on which feedback maps can be defined.

In the free \PROP\ $\Nwt(\Omega)$, a natural approach to defining 
such a restricted feedback is to make use of the symmetric join and 
define
\begin{equation} \label{Eq:DefFriFeedback}
  a \feedback{n} = a \join^n_n \phi(\same{n}) \text{.}
\end{equation}
The necessary and sufficient condition for this to be defined is that 
\(\Tr(a) = \left[ \begin{smallmatrix} q_{11} & q_{12} \\ 
q_{21} & q_{22} \end{smallmatrix} \right]\) has \(q_{22} \in \B^{n 
\times n}\) nilpotent, so it appears the domain of that 
$\feedback{n}$ is
\[
  \bigcup_{\substack{
    \left[ \begin{smallmatrix} q_{11} & q_{12} \\ 
    q_{21} & q_{22} \end{smallmatrix} \right] \in 
    \B^{\bullet\times\bullet}
    \\
    q_{22} \in \B^{n \times n} \text{ is nilpotent}
  }} \mkern-36mu
  \mc{Y}_\Omega\left( \left[ \begin{smallmatrix} q_{11} & q_{12} \\ 
    q_{21} & q_{22} \end{smallmatrix} \right] \right)
  \text{.}
\]
From \eqref{Eq:DefFriFeedback} and Construction~\ref{Kons:symjoin}, one 
may also conclude that \(\Tr(a\feedback{n}) = q_{11} + 
q_{12}q_{22}^*q_{21}\), which is equivalent to \(a\feedback{n} \in 
\mc{Y}_\Omega(q_{11} +\nobreak q_{12}q_{22}^*q_{21})\). Replacing 
$\mc{Y}_\Omega$ by a general $\B^{\bullet\times\bullet}$-filtration, 
we arrive at a definition that makes sense for a general \PROP.

\begin{definition}
  Let a \PROP\ $\mc{P}$ and $\B^{\bullet\times\bullet}$-filtration 
  $\{F_q\}_{q \in \B^{\bullet\times\bullet}}$ in $\mc{P}$ be given. A 
  \DefOrd{formal feedback} in $\mc{P}$ with respect to $\{F_q\}_{q 
  \in \B^{\bullet\times\bullet}}$ is a family \(\{\feedback{n}\}_{n 
  \in \N}\) of partial maps from $\mc{P}$ into itself satisfying:
  \begin{enumerate}
    \item
      $a\feedback{n}$ is defined whenever \(a \in F_q\) such that 
      \(q = \left[ \begin{smallmatrix} q_{11} & q_{12} \\ q_{21} & q_{22} 
      \end{smallmatrix} \right]\) has \(q_{22} \in \B^{n \times n}\) 
      nilpotent, and moreover \(a\feedback{n} \in F_r\) where \(r = 
      q_{11} + q_{12}q_{22}^*q_{21}\).
    \item
      The tightening, superposing, sliding, vanishing, and 
      yanking axioms of Definition~\ref{Def:Feedback} are satisfied 
      whenever the expressions involved are defined.
  \end{enumerate}
  A \DefOrd[*{PROP@\PROP!with formal feedback}]{\PROP\ with formal 
  feedback} is a triplet $\bigl( \mc{P}, \{F_q\}_{q \in 
  \B^{\bullet\times\bullet}}, \{\feedback{n}\}_{n\in\N} \bigr)$, 
  where $\mc{P}$ is a \PROP, $\{F_q\}_{q \in 
  \B^{\bullet\times\bullet}}$ is a filtration of $\mc{P}$, and 
  $\{\feedback{n}\}_{n\in\N}$ is a formal feedback with respect to 
  these.
\end{definition}

There would not be much point in defining normalised formal 
feedbacks, because the normalisation axiom will typically not be 
applicable for formal feedbacks. If \(\phi(\same{n}) 
\in F_q\) and $q$ is nilpotent (as it would need to be for a formal 
$\feedback{n}$ to be defined), then \(\phi(\same{n}) = 
\phi(\same{n})^{\circ n} \in F_{q^n} = F_0\). Furthermore for \(i = 1 
\in \B^{1 \times 1}\), \(z = 0 \in \B^{n \times n}\), and \(Z = 0 \in 
\B^{(n+1)\times(n+1)}\), one finds \(\phi(\same{n+1}) = 
\phi(\same{1}) \otimes \phi(\same{n}) \circ \phi(\same{n}) \otimes 
\phi(\same{1}) \in F_{i \otimes z \circ z \otimes i} = F_Z\) and hence 
\(\phi(\same{1}) = \phi(\same{1}) \otimes \phi(\same{n})\feedback{n} 
= \phi(\same{n+1})\feedback{n} \in F_0\). From this follows that the 
$F_0$ sets are the whole \PROP, so the feedbacks must in fact be 
unrestricted; a \PROP\ with formal feedbacks where the normalisation 
axiom would apply \emph{somewhere} is in fact a \PROP\ with 
feedbacks, period.

An important example of \PROP\ with formal feedbacks is the 
$\mc{R}^{\bullet\times\bullet}$ matrix \PROP.

\begin{lemma} \label{L:Matrisfeedback}
  Let $\mc{R}$ be an associative unital semiring. Then 
  $\mc{R}^{\bullet\times\bullet}$, with the dependency filtration 
  $\{F_q\}_{q \in \B^{\bullet\times\bullet}}$, is a \PROP\ with formal 
  feedbacks given by
  \begin{equation}
    \left. \begin{bmatrix}
      A_{11} & A_{12} \\ A_{21} & A_{22} 
    \end{bmatrix} \rightfeedback{n} :=
    A_{11} + A_{12} A_{22}^* A_{21}
  \end{equation}
  for all \(A_{11} \in \mc{R}^{k \times l}\), \(A_{12} \in 
  \mc{R}^{k \times n}\), \(A_{21} \in \mc{R}^{n \times l}\), and 
  \(A_{22} \in F_q\) where \(q \in \B^{n \times n}\) is nilpotent.
\end{lemma}
\begin{proof}
  That $A_{22}^*$ is defined follows from 
  Lemma~\ref{L:Matrisstjarna}, so what remains is to check that the 
  feedback defined satisfies all the axioms.
  
  With notation as in Definition~\ref{Def:Feedback}, one finds for 
  the tightening axiom that
  \begin{multline*}
    \bigl( a \otimes \phi(\same{n}) \circ b \circ 
      c \otimes \phi(\same{n}) \bigr)\feedback{n} = \\ =
    \left. \left(
      \begin{bmatrix} a & 0 \\ 0 & \phi(\same{n}) \end{bmatrix}
      \begin{bmatrix} b_{11} & b_{12} \\ b_{21} & b_{22} \end{bmatrix}
      \begin{bmatrix} c & 0 \\ 0 & \phi(\same{n}) \end{bmatrix}
    \right) \rightfeedback{n} = 
    \left. \begin{bmatrix}
      a b_{11} c & a b_{12} \\ b_{21} c & b_{22}
    \end{bmatrix} \rightfeedback{n} = \\ =
    a b_{11} c + a b_{12} b_{22}^* b_{21} c = 
    a (b_{11} + b_{12} b_{22}^* b_{21}) c = 
    a \circ b\feedback{n} \circ c
  \end{multline*}
  for all \(a \in \mc{R}^{i \times j}\), \(b = \left[ 
  \begin{smallmatrix} b_{11} & b_{12} \\ b_{21} & b_{22} 
  \end{smallmatrix} \right] \in \mc{R}^{(j+n) \times (k+n)}\) such 
  that \(b_{22} \in F_q\) where \(q \in \B^{n \times n}\) is 
  nilpotent, and \(c \in \mc{R}^{k \times l}\).
  
  The superposing axiom is similarly just a matter of
  \begin{equation*}
    a \otimes b\feedback{n} =
    \begin{bmatrix}
      a & 0 \\ 0 & b\feedback{n}
    \end{bmatrix} =
    \begin{bmatrix}
      a & 0 \\ 0 & b_{11} + b_{12} b_{22}^* b_{21}
    \end{bmatrix} =
    \left. \begin{bmatrix}
      a & 0 & 0 \\
      0 & b_{11} & b_{12} \\
      0 & b_{21} & b_{22}
    \end{bmatrix} \rightfeedback{n} =
    (a \otimes b)\feedback{n}
  \end{equation*}
  for all \(a \in \mc{R}^{i \times j}\) and \(b = \left[ 
  \begin{smallmatrix} b_{11} & b_{12} \\ b_{21} & b_{22} 
  \end{smallmatrix} \right] \in \mc{R}^{(k+n) \times (l+n)}\) such 
  that \(b_{22} \in F_q\) where \(q \in \B^{n \times n}\) is 
  nilpotent.
  
  For sliding, consider \(a \in F_r\) and \(b = \left[ \begin{smallmatrix} 
  b_{11} & b_{12} \\ b_{21} & b_{22} \end{smallmatrix} \right] \in F_q\) 
  where \(r \in \B^{m \times n}\), \(q = \left[ \begin{smallmatrix} 
  q_{11} & q_{12} \\ q_{21} & q_{22} \end{smallmatrix} \right]\), 
  \(q_{11} \in \B^{k \times l}\), \(q_{12} \in \B^{k \times m}\),
  \(q_{21} \in \B^{n \times l}\), and \(q_{22} \in \B^{n \times m}\) 
  are such that $r q_{22}$ (and equivalently $q_{22} r$) is nilpotent. 
  Then \((a b_{22})^* a = a (b_{22} a)^*\) by 
  Lemma~\ref{L:Matrisstjarna}, and hence
  \begin{multline*}
    \bigl( \phi(\same{k}) \otimes a \circ b \bigr)\feedback{m} = 
    \left. \left( 
      \begin{bmatrix} \phi(\same{k}) & 0 \\ 0 & a \end{bmatrix}
      \begin{bmatrix} b_{11} & b_{12} \\ b_{21} & b_{22} \end{bmatrix}
    \right) \rightfeedback{m} =
    \left. \begin{bmatrix} 
       b_{11} & b_{12} \\
       a b_{21} & a b_{22}
    \end{bmatrix} \rightfeedback{m} = \\ =
    b_{11} + b_{12} (a b_{22})^* a b_{21} =
    b_{11} + b_{12} a (b_{22} a)^* b_{21} = \\ =
    \left. \begin{bmatrix} 
       b_{11} & b_{12} a \\
       b_{21} & b_{22} a
    \end{bmatrix} \rightfeedback{n} =
    \left. \left( 
      \begin{bmatrix} b_{11} & b_{12} \\ b_{21} & b_{22} \end{bmatrix}
      \begin{bmatrix} \phi(\same{l}) & 0 \\ 0 & a \end{bmatrix}
    \right) \rightfeedback{n} =
    \bigl( b \circ \phi(\same{l}) \otimes a \bigr)\feedback{n}
    \text{.}
  \end{multline*}
  
  Vanishing is a bit trickier. Here one must consider
  \begin{equation*}
    a = \begin{bmatrix}
      a_{11} & a_{12} & a_{13}\\
      a_{21} & a_{22} & a_{23}\\
      a_{31} & a_{32} & a_{33}
    \end{bmatrix} \in F_q
    \qquad\text{where}\qquad
    q = \begin{bmatrix}
      q_{11} & q_{12} & q_{13}\\
      q_{21} & q_{22} & q_{23}\\
      q_{31} & q_{32} & q_{33}
    \end{bmatrix}
  \end{equation*}
  has \(q_{11} \in \B^{k \times l}\), \(q_{22} \in \B^{m \times m}\), 
  and \(q_{33} \in \B^{n \times n}\). Moreover, the immediate 
  condition for $a\feedback{m+n}$Êto be defined is that $\left[ 
  \begin{smallmatrix} q_{22} & q_{23} \\ q_{32} & q_{33} 
  \end{smallmatrix} \right]$ is nilpotent, whereas the immediate 
  condition for $a \feedback{n} \feedback{m}$ to be defined is that 
  $q_{33}$ and $q_{22} + q_{23} q_{33}^* q_{32}$ are nilpotent; these 
  conditions are however equivalent by Lemma~\ref{L:Matrisnilpotens}. 
  Introducing the shorthand \(a_{232} = a_{23} a_{33}^* a_{32}\), and 
  using Lemma~\ref{L:Matrisstjarna}, one finds
  \begin{multline*}
    \left. \left. \begin{bmatrix}
      a_{11} & a_{12} & a_{13} \\
      a_{21} & a_{22} & a_{23} \\
      a_{31} & a_{32} & a_{33}
    \end{bmatrix} \rightfeedback{n} \rightfeedback{m}
    = 
    \left. \begin{bmatrix}
      a_{11} + a_{13} a_{33}^* a_{31} &
        a_{12} + a_{13} a_{33}^* a_{32} \\
      a_{21} + a_{23} a_{33}^* a_{31} &
        a_{22} + a_{232}
    \end{bmatrix} \rightfeedback{m}
    = \\ =
    a_{11} + a_{13} a_{33}^* a_{31} +
      (a_{12} + a_{13} a_{33}^* a_{32})
      (a_{22} + a_{232})^*
      (a_{21} + a_{23} a_{33}^* a_{31})
    = \\ =
    a_{11} + a_{13} a_{33}^* a_{31} +
      (a_{12} + a_{13} a_{33}^* a_{32})
      (a_{22}^* a_{232})^* a_{22}^*
      (a_{21} + a_{23} a_{33}^* a_{31})
    = \displaybreak[0]\\ = \begin{aligned}[t]
      & a_{11} + 
        (a_{12} + a_{13} a_{33}^* a_{32})
        (a_{22}^* a_{232})^* a_{22}^* a_{21} + \\
      & \quad {}+ a_{13} a_{33}^* a_{31} +
        (a_{12} + a_{13} a_{33}^* a_{32})
        (a_{22}^* a_{23} a_{33}^* a_{32})^* a_{22}^*
        a_{23} a_{33}^* a_{31} =
    \end{aligned}
    \displaybreak[0]\\ = \begin{aligned}[t]
      & a_{11} + 
        (a_{12} + a_{13} a_{33}^* a_{32})
        (a_{22}^* a_{232})^* a_{22}^* a_{21} + 
        a_{13} a_{33}^* a_{31} + \\
      & \quad {}+ 
        a_{12} a_{22}^* a_{23} 
          (a_{33}^* a_{32} a_{22}^* a_{23})^* a_{33}^* a_{31} +
        a_{13} (a_{33}^* a_{32} a_{22}^* a_{23})^+ a_{33}^* a_{31} =
    \end{aligned}
    \displaybreak[0]\\ = \begin{aligned}[t]
      & a_{11} + 
        (a_{12} + a_{13} a_{33}^* a_{32})
        (a_{22}^* a_{232})^* a_{22}^* a_{21} + \\
      & \qquad {}+ 
        (a_{12} a_{22}^* a_{23} + a_{13})
        (a_{33}^* a_{32} a_{22}^* a_{23})^* a_{33}^* a_{31} =
    \end{aligned}
    \displaybreak[0]\\ =
    a_{11} + 
      \begin{bmatrix} a_{12} & a_{13} \end{bmatrix}
      \begin{bmatrix} 
        (a_{22}^* a_{232})^* a_{22}^* &
          a_{22}^* a_{23} (a_{33}^* a_{32} a_{22}^* a_{23})^* a_{33}^*\\
        a_{33}^* a_{32} (a_{22}^* a_{232})^* a_{22}^* &
          (a_{33}^* a_{32} a_{22}^* a_{23})^* a_{33}^*
      \end{bmatrix}
      \begin{bmatrix} a_{21} \\ a_{31} \end{bmatrix}
    = \\ =
    a_{11} + 
      \begin{bmatrix} a_{12} & a_{13} \end{bmatrix}
      \begin{bmatrix} 
        a_{22} & a_{23} \\ a_{32} & a_{33}
      \end{bmatrix}^*
      \begin{bmatrix} a_{21} \\ a_{31} \end{bmatrix}
    = 
    \left. \begin{bmatrix}
      a_{11} & a_{12} & a_{13} \\
      a_{21} & a_{22} & a_{23} \\
      a_{31} & a_{32} & a_{33}
    \end{bmatrix} \rightfeedback{n+m}
    \text{.}
  \end{multline*}
  
  Finally, the verification of the yanking axiom is simply that
  $$
    \phi(\cross{n}{n}) \feedback{n} 
    =
    \left. \begin{bmatrix} 0 & I_n \\ I_n & 0 \end{bmatrix}
      \rightfeedback{n}
    =
    0 + I_n 0^* I_n 
    = 
    I_n
    =
    \phi(\same{n})
    \text{,}
  $$
  where $I_n$ denotes the $n \times n$ identity matrix.
\end{proof}

The same $\feedback{n}$ operations also constitute formal feedbacks in 
the biaffine \PROP\ $\Baff(\mc{R})$, but it turns out the end of 
establishing that a particular quasi-order is preserved under 
symmetric join does not require that each and every \PROP\ employed 
in defining this orders is equipped with a formal feedback; there is 
often an easier way of getting there. That way does however still 
rely upon having a general formula for $\eval_f(K \join^r_q\nobreak H)$ 
in terms of $\eval_f(K)$, $\eval_f(H)$, and the operations of a 
\PROP\ with formal feedbacks.

\begin{theorem} \label{S:Feedback-join}
  Let $\bigl( \mc{P}, \{F_q\}_{q \in \B^{\bullet\times\bullet}}, 
  \{\feedback{n}\}_{n\in\N} \bigr)$ be a \PROP\ with formal 
  feedbacks. Let $\Omega$ be an $\N^2$-graded set and 
  \(y \in \Omega(1,1)\) some element. Let \(f\colon \Omega \Fpil 
  \mc{P}\) be an $\N^2$-graded set morphism such that \(f(y) = 
  \phi_{\mc{P}}(\same{1})\). If \(H,K \in \Nw(\Omega)\) are such that
  \begin{align*}
    \Tr(K) ={}& 
    \begin{bmatrix}
      a_{11} & a_{12} \\ a_{21} & a_{22}
    \end{bmatrix}
    \text{,} 
    &
    \Tr(H) ={}&
    \begin{bmatrix}
      b_{22} & b_{23} \\ b_{32} & b_{33}
    \end{bmatrix} 
    \text{,}
  \end{align*}
  where \(a_{11} \in \B^{k \times l}\), \(a_{22} \in \B^{r \times 
  q}\), \(b_{22} \in \B^{q \times r}\), and \(b_{33} \in \B^{m \times 
  n}\) are such that $a_{22} b_{22}$ is nilpotent, then 
  \begin{equation} \label{Eq:Feedback-join}
    \eval_f( K \join[y]^r_q H ) =
    \bigl( 
      \phi(\same{k} \star \cross{r}{m}) \circ
      \eval_f(K) \otimes \phi(\same{m}) \circ
      \phi(\same{l}) \otimes \eval_f(H) \circ
      \phi(\same{l} \star \cross{n}{r}) 
    \bigr) \feedback{r}
    \text{.}
  \end{equation}
\end{theorem}

Schematically, what this theorem says is that
$$
  \eval_f\left( \begin{mpgraphics*}{38}
    beginfig(38);
      PROPdiagram(0,0)
        same(2) circ
        box(2,2)(btex \strut$K$ etex) symjoinIi
          box(2,2)(btex \strut$H$ etex)
        circ same(2)
      ;
    endfig;
  \end{mpgraphics*} \right)
  =
  \left[ \begin{mpgraphics*}{39}
    beginfig(39);
      save defaultsep; defaultsep:=8pt;
      PROPdiagram(0,0)
        same(2) circ (
          same(1) otimes tightcross(0.5,2,1,1) circ
          box(2,2)(btex \strut$\eval_f(K)$ etex) otimes same(1) circ
          same(1) otimes box(2,2)(btex \strut$\eval_f(H)$ etex) circ
          same(1) otimes tightcross(0.5,2,1,1)
        ) feedback 1
        circ same(2)
      ;
    endfig;
  \end{mpgraphics*} \right]
$$
which should seem plausible on account of the two diagrams being 
topologically equivalent. The right hand side is however an 
expression in a specific \PROP\ with feedback, whereas the left hand 
side is (the evaluation of) a network, so the topological equivalence 
must be taken more as a roadmap to understand what individual steps 
below are about than as an actual proof.

Moreover, the proof is somewhat long and will be interrupted by 
various lemmas that establish facts useful not only in the specific 
context of proving Theorem~\ref{S:Feedback-join}. As in 
Section~\ref{Sec:FriPROP}, the various data defining a network are 
referred to using their basic symbol ($V$, $E$, $h$, $g$, $t$, $s$, 
and $D$ respectively) subscripted by the network at hand.

\begin{lemma}
  For any \(K,H \in \Nw(\Omega)\), the network \(G = K 
  \join[y]^0_0 H\) has a split $(F_\mathrm{l}, F_\mathrm{r}, 
  W_\mathrm{l}, W_\mathrm{r})$ such that the induced decomposition 
  $(G_\mathrm{l},G_\mathrm{r})$ has \(G_\mathrm{l} \simeq K\) and 
  \(G_\mathrm{r} \simeq H\).
\end{lemma}
\begin{proof}
  The split is given by
  \begin{align*}
    F_\mathrm{l} ={}& \setOf{ e \in E_G }{ \text{$e$ even} }
      \text{,}&
    F_\mathrm{r} ={}& \setOf{ e \in E_G }{ \text{$e$ odd} }
      \text{,}\\
    W_\mathrm{l} ={}& 
      \setOf[\big]{ v \in V_G \setminus \{0,1\} }{ \text{$v$ even} }
      \text{,}&
    W_\mathrm{r} ={}& 
      \setOf[\big]{ v \in V_G \setminus \{0,1\} }{ \text{$v$ odd} }
  \end{align*}
  and the two isomorphisms mostly consist of doubling the labels (for 
  $K$), or doubling the labels and adding $1$ (for $H$).
\end{proof}

This lemma provides for the base case of an induction over $r+q$, 
since if \(r=q=0\) it follows that
\begin{multline*}
  \eval_f( K \join[y]^r_q H ) =
  \eval_f(K) \otimes \eval_f(H) = \\ =
  \eval_f(K) \otimes \phi(\same{m}) \circ
    \phi(\same{l}) \otimes \eval_f(H) 
  = \\ =
  \phi(\same{k}) \otimes \phi(\same{m}) \circ
  \eval_f(K) \otimes \phi(\same{m}) \circ
  \phi(\same{l}) \otimes \eval_f(H) \circ
  \phi(\same{l}) \otimes \phi(\same{n}) 
  = \\ =
  \phi(\same{k}) \otimes \phi(\cross{0}{m}) \circ
  \eval_f(K) \otimes \phi(\same{m}) \circ
  \phi(\same{l}) \otimes \eval_f(H) \circ
  \phi(\same{l}) \otimes \phi(\cross{n}{0}) 
  = \\ =
  \bigl(
  \phi(\same{k}) \otimes \phi(\cross{r}{m}) \circ
  \eval_f(K) \otimes \phi(\same{m}) \circ
  \phi(\same{l}) \otimes \eval_f(H) \circ
  \phi(\same{l}) \otimes \phi(\cross{n}{r}) 
  \bigr) \feedback{r}
  \text{.}
\end{multline*}
Unsurprisingly, the induction hypothesis is that 
\eqref{Eq:Feedback-join} holds whenever $r+q$ is one less than 
for the two networks $K$ and $H$ currently being considered.

What is perhaps more surprising is that the only case in which this 
hypothesis will be invoked directly is when $K$ has an edge $e_0$ 
such that \(h_K(e_0) = 0\), \(g_K(e_0) = k\), \(t_K(e_0) = 1\), and 
\(s_K(e_0) = l+q\). Schematically, what one wants to prove then looks 
like
\begin{equation} \label{Eq2:Feedback-join}
  \eval_f \left(
  \begin{mpgraphics*}{70}
    beginfig(70);
      PROPdiagram(0,0)
        same(3) circ
        symjoin_pq( frame(
          same(1) otimes cross(1,1) circ
          box(2,2)(btex \strut$K'$ etex) otimes same(1)
        ), 1, 2, 
          box(3,2)(btex \strut$H$ etex)
        )
        circ same(2)
      ;
    endfig;
  \end{mpgraphics*} \right)
  =
  \left[ \begin{mpgraphics*}{71}
    beginfig(71);
      PROPdiagram(0,0)
        same(3) circ (
          same(2) otimes tightcross(0.5,2,1,1) circ
          frame(
            same(1) otimes tightcross(0.7,1.3,1,1) circ
            box(2,2)(btex \strut$\eval_f(K')$ etex) otimes same(1)
          ) otimes same(1) circ
          same(1) otimes box(3,2)(btex \strut$\eval_f(H)$ etex) circ
          same(1) otimes tightcross(0.5,2,1,1)
        ) feedback 1
        circ same(2)
      ;
    endfig;
  \end{mpgraphics*} \right]
  \text{,}
\end{equation}
and this is done by reducing to
\begin{equation} \label{Eq3:Feedback-join}
  \eval_f \left(
  \begin{mpgraphics*}{72}
    beginfig(72);
      PROPdiagram(0,0)
        same(3) circ
        symjoin_pq( 
          box(2,2)(btex \strut$K'$ etex) 
        , 1, 1, 
          box(3,2)(btex \strut$H$ etex)
        )
        circ same(2)
      ;
    endfig;
  \end{mpgraphics*} \right)
  =
  \eval_f( K' \join[y]^r_{q-1} H )
  =
  \left[ \begin{mpgraphics*}{73}
    beginfig(73);
      PROPdiagram(0,0)
        same(3) circ (
          same(1) otimes tightcross(0.7,1.3,1,2) circ
          box(2,2)(btex \strut$\eval_f(K')$ etex) otimes same(2) circ
          same(1) otimes box(3,2)(btex \strut$\eval_f(H)$ etex) circ
          same(1) otimes tightcross(0.7,1.3,1,1)
        ) feedback 1
        circ same(2)
      ;
    endfig;
  \end{mpgraphics*} \right]
\end{equation}
where the latter equality holds by the induction hypothesis.

A proper proof will show that the left hand sides of 
\eqref{Eq2:Feedback-join} and \eqref{Eq3:Feedback-join} are equal, 
and ditto the right hand sides, but first it must be sorted out what 
$K'$ is for the given $K$. In $(\same{k-1} \star\nobreak \cross{1}{r}) 
\cdot K$, the head index of $e_0$ is the maximal $k+r$, just as the 
tail index is the maximal $l+q$. Hence $(\same{k-1} \star\nobreak 
\cross{1}{r}) \cdot K$ has a split with $e_0$ in the right part and 
everything else in the left part; let $K'$ be this left part. Then
\begin{multline*}
  \eval_f(K) 
  =
  \phi( \same{k-1} \star \cross{r}{1}) \circ
  \eval_f\bigl( (\same{k-1} \star \cross{1}{r}) \cdot K \bigr)
  = \\ =
  \phi( \same{k-1}) \otimes \phi(\cross{r}{1}) \circ
  \eval_f(K') \otimes \phi(\same{1})
  \text{,}
\end{multline*}
hence
\begin{align*}
  \hspace{5em} & \hspace{-5em}
  \phi(\same{k} \star \cross{r}{m}) \circ
  \eval_f(K) \otimes \phi(\same{m}) \circ
  \phi(\same{l}) \otimes \eval_f(H) \circ
  \phi(\same{l} \star \cross{n}{r}) 
  = \\ ={}&
  \phi(\same{k} \star \cross{r}{m}) \circ
  \bigl(
    \phi( \same{k-1}) \otimes \phi(\cross{r}{1}) \circ
    \eval_f(K') \otimes \phi(\same{1})
  \bigr) \otimes \phi(\same{m}) 
    \\ & \qquad\qquad {}\circ
  \phi(\same{l}) \otimes \eval_f(H) \circ
  \phi(\same{l} \star \cross{n}{r}) 
  = \\ ={}&
  \phi(\same{k} \star \cross{r}{m}) \circ
  \phi( \same{k-1}) \otimes \phi(\cross{r}{1}) \otimes 
    \phi(\same{m}) \circ
  \eval_f(K') \otimes \phi(\same{1+m}) 
    \\ & \qquad\qquad {}\circ
  \phi(\same{l}) \otimes \eval_f(H) \circ
  \phi(\same{l} \star \cross{n}{r}) 
  = \\ ={}&
  \phi( \same{k-1}) \otimes \phi(\cross{r}{1+m}) \circ
  \eval_f(K') \otimes \phi(\same{1+m}) 
    \\ & \qquad\qquad {}\circ
  \phi(\same{l}) \otimes \eval_f(H) \circ
  \phi(\same{l} \star \cross{n}{r}) 
\end{align*}
and thus \(\bigl( \phi(\same{k} \star \cross{r}{m}) \circ
\eval_f(K) \otimes \phi(\same{m}) \circ 
\phi(\same{l}) \otimes \eval_f(H) \circ
\phi(\same{l} \star \cross{n}{r}) \bigr) \feedback{r} = 
\eval_f(K' \join[y]^r_{q-1}\nobreak H)\) by the induction hypothesis.

The reason \(G = K \join[y]^r_q H\) and \(G' = K' \join[y]^r_{q-1} H\) 
evaluate to the same thing is that the former is a subdivision of the 
latter; the extra vertex in $G$ is that labelled $2+r+q$, which also 
is the tail of $e_0$'s counterpart $2e_0$. Cutting $G$ above that 
vertex, one obtains a decomposition $(G_0,G_1)$ where \(G_0 = 
\Nwfuse{ \vek{a} b \vek{c} }{ (2+r+q\colon y)^b_d }{ \vek{a} d 
\vek{c} }\) for \(\Norm{\vek{a}} = k-1\), \(\Norm{\vek{c}} = m\), \(b 
= 2e_0\) and \(d = 2e_1+1\) for that \(e_1 \in E_H\) which has 
\(h_H(e_1) = 0\) and \(g_H(e_1) = q\). Furthermore \(G_1 \simeq 
G'\)\Ldash the only difference in the labelling being a shift of 
vertex labels above $2+r+q$ by $1$ since $G'$ has one join vertex 
less\Rdash and thus
\begin{multline*}
  \eval_f(G) =
  \eval_f(G_0) \circ \eval_f(G_1) =
  \phi(\same{k-1}) \otimes f(y) \otimes \phi(\same{m}) \circ
    \eval_f(G') = \\ =
  \phi(\same{k-1}) \otimes \phi(\same{1}) \otimes \phi(\same{m})
    \circ \eval_f(G') =
  \phi(\same{k+m}) \circ \eval_f(G') =
  \eval_f(G')
\end{multline*}
as claimed.

The next case, which can be reduced to the previous one, is the mild 
generalisation that the edge $e_0$ has \(l < s_K(e_0) < l+q\) rather 
than \(s_K(e_0) = l+q\). The way to perform that reduction is 
effectively to permute the join vertices.

\begin{lemma}
  For any \(K \in \Nw(\Omega)(k +\nobreak r, l +\nobreak q)\), 
  \(H \in \Nw(\Omega)(q +\nobreak m, r +\nobreak n)\), 
  \(\sigma \in \Sigma_q\), and \(\tau \in \Sigma_r\), it holds 
  that
  \begin{equation}
    \bigl( K \cdot (\same{l} \star \sigma) \bigr) 
    \join[y]^r_q 
    \bigl( H \cdot (\tau \star \same{n}) \bigr)
    \simeq
    \bigl( (\same{k} \star \tau) \cdot K \bigr)
    \join[y]^r_q
    \bigl( (\sigma \star \same{m}) \cdot H \bigr)
  \end{equation}
  whenever either side of the equation is defined.
\end{lemma}
\begin{proof}
  Let \(K' = K \cdot (\same{l} \star\nobreak \sigma)\), \(H' = 
  H \cdot (\tau \star\nobreak \same{n})\), \(K'' = (\same{k} 
  \star\nobreak \tau) \cdot K\), and \(H'' = (\sigma \star\nobreak 
  \same{m}) \cdot H\). Then the isomorphism $(\chi,\psi)$ from $K' 
  \join[y]^r_q H'$ to $K'' \join[y]^r_q H''$ has $\psi$ being the 
  identity, and
  \[
    \chi(v) = \begin{cases}
      v& \text{if \(v \leqslant 1\) or \(v > 2+r+q\),}\\
      2 + \tau(v-2)& \text{if \(2 < v \leqslant 2+r\),}\\
      2+r + \sigma(v-2-r)& \text{if \(2+r < v \leqslant 2+r+q\),}
    \end{cases}
  \]
  because if \(e_1 \in E_H\) is an output leg with \(g_H(e_1) 
  \leqslant q\) then the corresponding edge $2e_1+1$ from the 
  symmetric joins has head $2+r+g_K(e_1)$ in $K' \join[y]^r_q H'$ and 
  head $2+r + \sigma\bigl( g_K(e_1) \bigr)$ in $K'' \join[y]^r_q 
  H''$. In the latter case, the edge leaving that vertex is $2e_2$ 
  for \(e_2 \in E_K\) such that \(s_K(e_2) - l = \sigma\bigl( 
  g_K(e_1) \bigr)\), but that is equivalent to \(\sigma^{-1}\bigl( 
  s_K(e_2) -\nobreak l \bigr) = g_K(e_1)\), which is the condition 
  that the head in $K' \join[y]^r_q H'$ of $2e_2$ is  $2+r+g_K(e_1)$. 
  Hence the joining of output legs of $H'$Êto input legs of $K'$ is 
  the same as that of output legs of $H''$ to input legs of $K''$, 
  which is what the isomorphism requires. The effect of $\tau$ is 
  checked in the same way.
\end{proof}

Thus, if \(\sigma \in \Sigma_q\) is such that \(l + \sigma(q) = 
(\same{l} \star\nobreak \sigma)( l +\nobreak q ) = s_K(e_0)\) then
\begin{align*}
  \hspace{3em} & \hspace{-3em}
  \eval_f( K \join[y]^r_q H )
  = \\ ={}&
  \eval_f\Bigl( \bigl(K \cdot (\same{l} \star \sigma) \cdot 
  (\same{l} \star \sigma^{-1}) \bigr) \join[y]^r_q H \Bigr)
  = \displaybreak[0]\\ ={}&
  \eval_f\Bigl( \bigl(K \cdot (\same{l} \star \sigma) \bigr) 
  \join[y]^r_q \bigl( (\sigma^{-1} \star \same{m}) \cdot H \bigr) \Bigr)
  = \displaybreak[0]\\ ={}&
  \Bigl(
    \phi(\same{k} \star \cross{r}{m}) \circ
    \eval_f\bigl( K \cdot (\same{l} \star \sigma) \bigr) \otimes
      \phi(\same{m}) 
    \\ & \qquad{}\circ
    \phi(\same{l}) \otimes 
      \eval_f\bigl( (\sigma^{-1} \star \same{m}) \cdot H \bigr) \circ
    \phi(\same{l} \star \cross{n}{r})
  \Bigr) \feedback{r}
  = \displaybreak[0]\\ ={}&
  \Bigl(
    \phi(\same{k} \star \cross{r}{m}) \circ
    \bigl( \eval_f(K) \circ \phi(\same{l} \star \sigma) \bigr) \otimes
      \phi(\same{m}) 
  \\ & \qquad{} \circ
    \phi(\same{l}) \otimes 
      \bigl( \phi(\sigma^{-1} \star \same{m}) \circ \eval_f(H) \bigr) \circ
    \phi(\same{l} \star \cross{n}{r})
  \Bigr) \feedback{r}
  = \displaybreak[0]\\ ={}&
  \Bigl(
    \phi(\same{k} \star \cross{r}{m}) \circ
    \eval_f(K) \otimes \phi(\same{m}) \circ 
    \phi(\same{l} \star \sigma \star \same{m}) \circ
    \phi(\same{l} \star\sigma^{-1} \star \same{m}) 
    \\ & \qquad{} \circ 
    \phi(\same{l}) \otimes \eval_f(H) \circ
    \phi(\same{l} \star \cross{n}{r})
  \Bigr) \feedback{r}
  = \\ ={}&
  \Bigl(
    \phi(\same{k} \star \cross{r}{m}) \circ
    \eval_f(K) \otimes \phi(\same{m}) \circ 
    \phi(\same{l}) \otimes \eval_f(H) \circ
    \phi(\same{l} \star \cross{n}{r})
  \Bigr) \feedback{r} 
  \text{.}
\end{align*}

Finding such edges through $K$ is however not a very common 
occurrence, so the third case deals with manufacturing them, by making 
a suitable cut in the network. Let \(G = K \join[y]^r_q H\). By 
acyclicity, there is a partial order $P$ on $V_G$ which has \(h_G(e) < 
t_G(e) \pin{P}\) for all \(e \in E_G\); pick the minimal such order. 
The third case covers those $G$ for which the restriction of $P$ to 
the join-vertices has a minimal element $v_0$ satisfying \(2+r < v_0 
\leqslant 2+r+q\) (and which have not been covered by the first or 
second case already).

Let \(W_0 = \setOf[\big]{ v \in V_G \setminus \{0,1\} }{ v < v_0 
\pin{P} }\) and \(W_1 = V_G \setminus \{0,1\} \setminus W_0\); then 
$(W_0,W_1)$ is a cut in $G$. Let $E_\mathrm{c}$ be the set of cut 
edges. The odd (i.e., corresponding to $H$-) edges in $E_\mathrm{c}$ 
are the odd edges with head $0$, since all join-vertices are in $W_1$, 
and by minimality of $P$ all inner vertices of $H$ therefore must be 
mapped into $W_1$ as well. There are $m$ such edges, so define \(k' = 
\card{E_\mathrm{c}} - m\). For an ordered cut $(W_0,W_1,p)$ it is 
convenient to define \(p(e) = g_G(e) - k + k'\) for these odd edges. 
Furthermore let \(e_0 \in E_K\) be the edge such that \(t_G(2e_0) = 
v_0\). This is a cut edge too, so define \(p(2e_0) = m'\). Remaining 
cut edges can be put in any order.

The main idea for this third case is that in the decomposition 
$(G_0,G_1)$ induced by the ordered cut $(W_0,W_1,p)$, the upper part 
$G_1$ is itself a symmetric join $K_1 \join[y]^r_q H$ for the $K_1$ 
obtained by making the corresponding cut in $K$, whereas the lower 
parts $G_0$ and $K_0$ differ only by some padding. Clearly, 
$(W_0',W_1')$ is a cut in $K$ for
\begin{align*}
  W_0' ={}& 
  \setOf[\big]{ v \in V_K \setminus\{0,1\} }{ r+q + 2v \in W_0 }
  \text{,}\\
  W_1' ={}& 
  \setOf[\big]{ v \in V_K \setminus\{0,1\} }{ r+q + 2v \in W_1 }
  \text{;}
\end{align*}
let $E_\mathrm{c}'$ be the corresponding set of cut edges. $k'$ of 
those are the \(e \in E_K\) for which \(2e \in E_\mathrm{c}\), but 
there are also $r$ edges with head $0$ and head index $>k$; these are 
the edges whose counterparts in $G$ have their heads at join 
vertices. Therefore define
$$
  p'(e) = \begin{cases}
    p(2e)& \text{if \(2e \in E_\mathrm{c}\),}\\
    k' + s_K(e)-k& \text{otherwise}
  \end{cases}
  \qquad\text{for \(e \in E'_\mathrm{c}\).}
$$
Then let $(K_0,K_1)$ be the decomposition of $K$ corresponding to the 
ordered cut $(W_0',W_1',p')$.

What happens is that \(G_1 = K_1 \join[y]^r_q H\)\Ldash not merely 
isomorphic (even though that would suffice), but actually equal\Rdash 
since $\join^r_q$ relabels vertices and edges in the exact same way 
regardless of what networks they are part of, and all vertices and 
edges in $K_1$ are present also in $K$; clearly $K_1 \join[y]^r_q H$ 
contains all the $H$-, join-, and non-inner vertices of $G$ and hence 
$G_1$, and the condition for a $K$-vertex to end up above the cut in 
$G$ is exactly the same as for it to end up in $K_1$. Similarly \(e 
\in E_K\) satisfies \(e \in E_{K_1}\) iff \(t_K(e) \in W_1' \cup 
\{1\}\), which is equivalent to the condition \(t_G(2e) \in W_1 
\cup \{1\}\) for its counterpart $2e$ to end up in $G_1$. The only 
area where something could happen is at the cut, but this choice of 
$p'$ is the one which ensures that any edge $2e$ for \(e \in 
E'_\mathrm{c}\) has the same head and head index in $K_1 \join[y]^r_q 
H$ as it does in $G_1$. Letting \(a = \phi(\same{l}) \otimes 
\eval_f(H) \circ \phi(\same{l} \star\nobreak \cross{n}{r})\) as a 
shorthand, one therefore has
\begin{multline*}
  \eval_f(G) =
  \eval_f(G_0) \circ \eval_f(G_1) =
  \eval_f(G_0) \circ \eval_f(K_1 \join[y] H)
  = \\ =
  \eval_f(G_0) \circ \bigl(
    \phi(\same{k'} \star \cross{r}{m}) \circ
    \eval_f(K_1) \otimes \phi(\same{m}) \circ
    a
  \bigr) \feedback{r}
\end{multline*}
on account of $K_1$ having the edge $e_0$ which makes $K_1 \join[y] 
H$ fit the first or second case.

Regarding the lower parts $G_0$ and $K_0$, one may observe that these 
have splits in which the right parts consist of $m$ and $r$ 
respectively parallel edges. Writing $(G_{00},G_{01})$ and 
$(K_{00},K_{01})$ for the corresponding decompositions, one may 
furthermore observe that \(G_{00} \simeq K_{00}\) with the 
isomorphism $(\chi,\psi)$ from $K_{00}$ to $G_{00}$ having 
\(\psi(e) = 2e\) and \(\chi(v) = r+q+2v\) for \(v > 1\). Therefore 
\(\eval_f(G_0) = \eval_f(G_{00}) \otimes \eval_f(G_{01}) = 
\eval_f(K_{00}) \otimes \phi(\same{m})\), and hence
\begin{multline*}
  \eval_f(G_0) \circ \bigl(
    \phi(\same{k'} \star \cross{r}{m}) \circ
    \eval_f(K_1) \otimes \phi(\same{m}) \circ
    a
  \bigr) \feedback{r}
  = \\ =
  \eval_f(K_{00}) \otimes \phi(\same{m}) \circ \bigl(
    \phi(\same{k'} \star \cross{r}{m}) \circ
    \eval_f(K_1) \otimes \phi(\same{m}) \circ
    a
  \bigr) \feedback{r}
  = \displaybreak[0]\\ =
  \bigl(
    \eval_f(K_{00}) \otimes \phi(\same{m}) \otimes \phi(\same{r}) \circ
    \phi(\same{k'} \star \cross{r}{m}) \circ
    \eval_f(K_1) \otimes \phi(\same{m}) \circ
    a
  \bigr) \feedback{r}
  = \displaybreak[0]\\ =
  \bigl(
    \phi(\same{k} \star \cross{r}{m}) \circ
    \eval_f(K_{00}) \otimes \phi(\same{r}) \otimes \phi(\same{m}) \circ
    \eval_f(K_1) \otimes \phi(\same{m}) \circ
    a
  \bigr) \feedback{r}
  = \displaybreak[0]\\ =
  \Bigl(
    \phi(\same{k} \star \cross{r}{m}) \circ
    \bigl(
      \eval_f(K_{00}) \otimes \phi(\same{r}) \circ
      \eval_f(K_1) 
    \bigr) \otimes \phi(\same{m}) \circ
    a
  \Bigr) \feedback{r}
  = \displaybreak[0]\\ =
  \Bigl(
    \phi(\same{k} \star \cross{r}{m}) \circ
    \bigl(
      \eval_f(K_0) \circ
      \eval_f(K_1) 
    \bigr) \otimes \phi(\same{m}) \circ
    a
  \Bigr) \feedback{r}
  = \\ =
  \bigl(
    \phi(\same{k} \star \cross{r}{m}) \circ
    \eval_f(K) \otimes \phi(\same{m}) \circ
    a
  \bigr) \feedback{r}
\end{multline*}
as claimed in \eqref{Eq:Feedback-join}; this sequence of equalities 
features the only use of tightening in this proof to move some right 
hand side material (other than mere permutations) out of the feedback.

The fourth and final case of the induction step concerns those $K 
\join[y]^r_q H$ where none of the join-vertices from $H$ to $K$ is 
minimal in the partial order $P$. Since the set of join-vertices is 
finite, it must contain at least one minimal element, and since it is 
not one of those from right to left factor it must be one from right 
to left factor and may be denoted by $2+i$ for some \(i \in [r]\). 
This implies that in
\[
  G' := 
  (\cross{q}{m} \cdot H \cdot \cross{n}{r}) \join[y]^q_r
    (\cross{k}{r} \cdot K \cdot \cross{q}{l})
  \simeq
  \cross{k}{m} \cdot G \cdot \cross{n}{l} 
\]
(isomorphism by Lemma~\ref{L:Symjoin-transposition}) the join-vertex 
$2+q+i$ is minimal under this partial order, and hence $G'$ falls 
into one of the first three cases. 
\begin{figure}
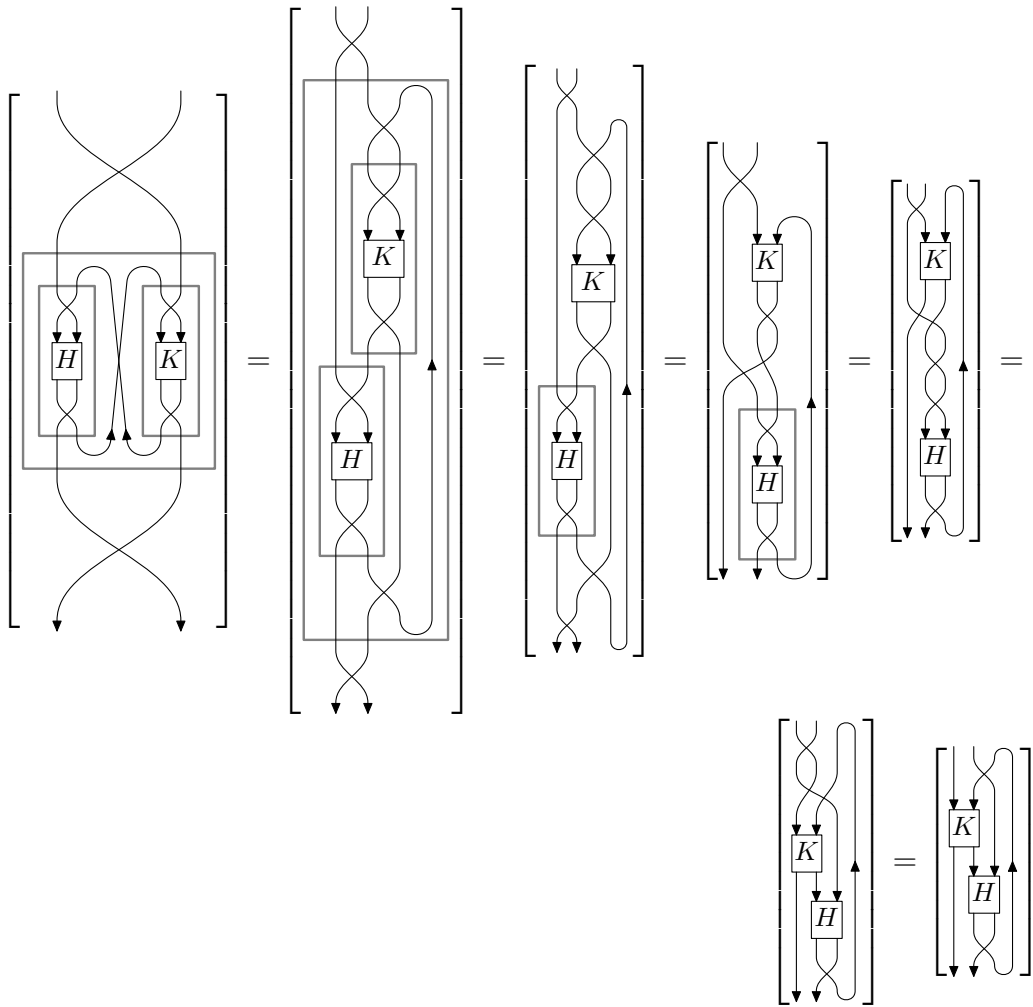

\begin{center}
  \hspace*{-\leftskip}\begin{math}
  \left[ \begin{mpgraphics*}{74}
    beginfig(74);
      PROPdiagram(0,0)
        tightcross(0.7,1.2,1,1) circ frame(
          same(2) circ
          frame(
            cross(1,1) circ 
            box(2,2)(btex \strut$H$ etex) circ
            cross(1,1)
          ) symjoinIi frame(
            cross(1,1) circ 
            box(2,2)(btex \strut$K$ etex) circ
            cross(1,1)
          )
          circ same(2)
        ) circ tightcross(0.7,1.2,1,1)
      ;
    endfig;
  \end{mpgraphics*} \right]
  =
  \left[ \begin{mpgraphics*}{75}
    beginfig(75);
      PROPdiagram(0,0)
        cross(1,1) circ frame(
          (
            same(1) otimes cross(1,1) circ
            frame(
              cross(1,1) circ 
              box(2,2)(btex \strut$H$ etex) circ
              cross(1,1)
            ) otimes same(1) circ
            same(1) otimes frame(
              cross(1,1) circ 
              box(2,2)(btex \strut$K$ etex) circ
              cross(1,1)
            ) circ
            same(1) otimes cross(1,1)
          ) feedback 1
        ) circ cross(1,1)
      ;
    endfig;
  \end{mpgraphics*} \right]
  =
  \left[ \begin{mpgraphics*}{76}
    beginfig(76);
      PROPdiagram(0,0)
        (
          cross(1,1) otimes same(1) circ 
          same(1) otimes cross(1,1) circ
          frame(
            cross(1,1) circ 
            box(2,2)(btex \strut$H$ etex) circ
            cross(1,1)
          ) otimes same(1) circ
          same(1) otimes (
            cross(1,1) circ 
            box(2,2)(btex \strut$K$ etex) circ
            cross(1,1)
          ) circ
          same(1) otimes cross(1,1)
        ) feedback 1
        circ cross(1,1)
      ;
    endfig;
  \end{mpgraphics*} \right]
  =
  \left[ \begin{mpgraphics*}{77}
    beginfig(77);
      PROPdiagram(0,0)
        (
          same(1) otimes frame(
            cross(1,1) circ 
            box(2,2)(btex \strut$H$ etex) circ
            cross(1,1)
          ) circ
          cross(2,1) circ
          same(1) otimes cross(1,1) circ 
          same(1) otimes box(2,2)(btex \strut$K$ etex)
        ) feedback 1
        circ cross(1,1)
      ;
    endfig;
  \end{mpgraphics*} \right]
  =
  \left[ \begin{mpgraphics*}{78}
    beginfig(78);
      PROPdiagram(0,0)
        (
          same(1) otimes (
            cross(1,1) circ 
            box(2,2)(btex \strut$H$ etex) circ
            cross(1,1)
          ) circ
          same(1) otimes cross(1,1) circ 
          cross(1,2) circ
          same(1) otimes box(2,2)(btex \strut$K$ etex)
          circ cross(1,1) otimes same(1)
        ) feedback 1
      ;
    endfig;
  \end{mpgraphics*} \right]
  =
  \left[ \begin{mpgraphics*}{79}
    beginfig(79);
      PROPdiagram(0,0)
        (
          same(1) otimes cross(1,1) circ 
          same(1) otimes box(2,2)(btex \strut$H$ etex) circ
          box(2,2)(btex \strut$K$ etex) otimes same(1) circ
          cross(1,2) circ
          cross(1,1) otimes same(1)
        ) feedback 1
      ;
    endfig;
  \end{mpgraphics*} \right]
  =
  \left[ \begin{mpgraphics*}{80}
    beginfig(80);
      PROPdiagram(0,0)
        (
          same(1) otimes cross(1,1) circ 
          same(1) otimes box(2,2)(btex \strut$H$ etex) circ
          box(2,2)(btex \strut$K$ etex) otimes same(1) circ
          same(1) otimes cross(1,1)
        ) feedback 1
      ;
    endfig;
  \end{mpgraphics*} \right]
  \end{math}\hspace*{-\rightskip}
\end{center}
  \caption{Final case of Theorem~\ref{S:Feedback-join}: unwind the 
  symmetric join by half a turn}
  \label{Fig:VevaUpp}
\end{figure}
Thus an evaluation of this symmetric join can schematically be reduced 
to an expression involving feedbacks as shown in 
Figure~\ref{Fig:VevaUpp}, the formal counterpart of which reads
\begin{align*}
  \eval_f(G) ={}&
  \phi(\cross{m}{k}) \circ \eval_f(G') \circ \phi(\cross{l}{n}) 
  = \\ ={}&
  \phi(\cross{m}{k}) \circ \Bigl(
    \phi(\same{m} \star \cross{q}{k}) \circ
    \eval_f(\cross{q}{m} \cdot H \cdot \cross{n}{r}) \otimes
      \phi(\same{k}) \circ
        {}\\ & \qquad{}\circ
    \phi(\same{n}) \otimes 
      \eval_f(\cross{k}{r} \cdot K \cdot \cross{q}{l}) \circ
    \phi(\same{n} \star \cross{l}{q})
  \Bigr) \feedback{q} \circ \phi(\cross{l}{n}) 
  = \displaybreak[0]\\ ={}&
  \Bigl(
    \phi(\cross{m}{k}) \otimes \phi(\same{q}) \circ 
    \phi(\same{m} \star \cross{q}{k}) \circ
    \eval_f(\cross{q}{m} \cdot H \cdot \cross{n}{r}) \otimes
      \phi(\same{k}) \circ
        {}\\ & \qquad{}\circ
    \phi(\same{n}) \otimes \bigl(
      \phi(\cross{k}{r}) \circ 
      \eval_f(K) \circ \phi(\cross{q}{l}) 
    \bigr) \circ
    \phi(\same{n} \star \cross{l}{q}) 
  \Bigr) \feedback{q} \circ \phi(\cross{l}{n})  
  = \displaybreak[0]\\ ={}&
  \Bigl(
    \phi(\cross{m+q}{k}) \circ
    \eval_f(\cross{q}{m} \cdot H \cdot \cross{n}{r}) \otimes
      \phi(\same{k}) \circ
        {}\\ & \qquad{}\circ
    \phi(\same{n}) \otimes \bigl(
      \phi(\cross{k}{r}) \circ \eval_f(K) \circ 
      \phi(\cross{q}{l}) \circ \phi(\cross{l}{q})
    \bigr) 
  \Bigr) \feedback{q}  \circ \phi(\cross{l}{n}) 
  = \displaybreak[0]\\ ={}&
  \Bigl(
    \phi(\same{k}) \otimes
      \eval_f(\cross{q}{m} \cdot H \cdot \cross{n}{r}) \circ
    \phi(\cross{n+r}{k}) \circ
        {}\\ & \qquad{}\circ
    \phi(\same{n} \star \cross{k}{r}) \circ 
    \phi(\same{n}) \otimes \eval_f(K) 
  \Bigr) \feedback{q}  \circ \phi(\cross{l}{n}) 
  = \displaybreak[0]\\ ={}&
  \Bigl(
    \phi(\same{k}) \otimes \bigl(
      \phi(\cross{q}{m}) \circ \eval_f(H) \circ \phi(\cross{n}{r}) 
    \bigr) \circ
    \phi(\same{k} \star \cross{r}{n}) \circ
        {}\\ & \qquad{}\circ
    \phi(\cross{n}{k+r}) \circ 
    \phi(\same{n}) \otimes \eval_f(K) \circ
    \phi(\cross{l}{n}) \otimes \phi(\same{q})
  \Bigr) \feedback{q} 
  = \displaybreak[0]\\ ={}&
  \bigl(
    \phi(\same{k} \star \cross{q}{m}) \circ 
    \phi(\same{k}) \otimes \eval_f(H) \circ 
        {}\\ & \qquad{}\circ
    \eval_f(K) \otimes \phi(\same{n}) \circ
    \phi(\cross{l+q}{n}) \circ 
    \phi(\cross{l}{n} \star \same{q})
  \bigr) \feedback{q} 
  = \\ ={}&
  \bigl(
    \phi(\same{k} \star \cross{q}{m}) \circ 
    \phi(\same{k}) \otimes \eval_f(H) \circ 
    \eval_f(K) \otimes \phi(\same{n}) \circ
    \phi(\same{l} \star \cross{n}{q})
  \bigr) \feedback{q} 
\end{align*}
but this expression differs from the right hand side of 
\eqref{Eq:Feedback-join} in that $\eval_f(H)$ and $\eval_f(K)$ are in 
different orders (also the feedback width is $q$ rather than $r$). 
That the two expressions are still equal is the subject of another 
lemma.

\begin{lemma}
  Let $\bigl( \mc{P}, \{F_A\}_{A \in \B^{\bullet\times\bullet}}, 
  \{\feedback{n}\}_{n\in\N} \bigr)$ be a \PROP\ with formal feedbacks. 
  Let
  \begin{align*}
    A = \begin{bmatrix}
      A_{11} & A_{12} \\ A_{21} & A_{22}
    \end{bmatrix} \in{}& \B^{(k+r)\times(l+q)} 
    \text{,}&
    B = \begin{bmatrix}
      B_{22} & B_{23} \\ B_{32} & B_{33}
    \end{bmatrix} \in{}& \B^{(q+m)\times(r+n)} 
    \text{,}
  \end{align*}
  where \(A_{22} \in \B^{r \times q}\) and \(B_{22} \in \B^{q \times 
  r}\) are such that $A_{22}B_{22}$ is nilpotent. Then for each \(a 
  \in F_A\) and \(b \in F_B\),
  \begin{multline}
    \bigl(
      \phi(\same{k} \star \cross{r}{m}) \circ
      a \otimes \phi(\same{m}) \circ
      \phi(\same{l}) \otimes b \circ
      \phi(\same{l} \star \cross{n}{r})
    \bigr) \feedback{r}
    = \\ =
    \bigl(
      \phi(\same{k} \star \cross{q}{m}) \circ
      \phi(\same{k}) \otimes b \circ
      a \otimes \phi(\same{n}) \circ
      \phi(\same{l} \star \cross{n}{q})
    \bigr) \feedback{q}
    \text{.}
  \end{multline}
\end{lemma}
\begin{proof}
  This is a straightforward application of the \PROP\ with formal 
  feedback axioms. Conceptually, the steps are
  \begin{center}
    \(\hspace{-\leftskip}
    \left[ \begin{mpgraphics*}{40}
      beginfig(40);
        PROPdiagram(0,0)
          (
            same(1) otimes cross(1,1) circ
            box(2,2)(btex \strut$a$ etex) otimes same(1) circ
            same(1) otimes box(2,2)(btex \strut$b$ etex) circ
            same(1) otimes cross(1,1)
          ) feedback 1
        ;
      endfig;
    \end{mpgraphics*}\right]
    =
    \left[ \begin{mpgraphics*}{41}
      beginfig(41);
        PROPdiagram(0,0)
          (
            same(1) otimes cross(1,1) circ
            box(2,2)(btex \strut$a$ etex) otimes same(1) circ
            same(1) otimes cross(1,1) circ
            same(1) otimes cross(1,1) circ
            same(1) otimes box(2,2)(btex \strut$b$ etex) circ
            same(1) otimes cross(1,1)
          ) feedback 1
        ;
      endfig;
    \end{mpgraphics*}\right]
    =
    \left[ \begin{mpgraphics*}{42}
      beginfig(42);
        PROPdiagram(0,0)
          (
            same(1) otimes cross(1,1) circ
            box(2,2)(btex \strut$a$ etex) otimes same(1) circ
            same(1) otimes cross(1,1) circ
            same(2) otimes cross(1,1) feedback 1 circ
            same(1) otimes cross(1,1) circ
            same(1) otimes box(2,2)(btex \strut$b$ etex) circ
            same(1) otimes cross(1,1)
          ) feedback 1
        ;
      endfig;
    \end{mpgraphics*}\right]
    =
    \left[ \begin{mpgraphics*}{43}
      beginfig(43);
        PROPdiagram(0,0)
          (
            same(1) otimes cross(1,1) otimes same(1) circ
            box(2,2)(btex \strut$a$ etex) otimes same(2) circ
            same(1) otimes cross(1,1) otimes same(1) circ
            same(2) otimes cross(1,1) circ
            same(1) otimes cross(1,1) otimes same(1) circ
            same(1) otimes box(2,2)(btex \strut$b$ etex) otimes same(1) circ
            same(1) otimes cross(1,1) otimes same(1) 
          ) feedback 2
        ;
      endfig;
    \end{mpgraphics*}\right]
    =
    \left[ \begin{mpgraphics*}{44}
      beginfig(44);
        PROPdiagram(0,0)
          (
            same(1) otimes cross(1,1) otimes same(1) circ
            box(2,2)(btex \strut$a$ etex) otimes same(2) circ
            same(2) otimes cross(1,1) circ
            same(2) otimes box(2,2)(btex \strut$b$ etex) circ
            same(2) otimes cross(1,1) circ
            same(1) otimes cross(2,1)
          ) feedback 2
        ;
      endfig;
    \end{mpgraphics*}\right]
    = \penalty 10
    \left[ \begin{mpgraphics*}{45}
      beginfig(45);
        PROPdiagram(0,0)
          (
            same(1) otimes cross(2,1) circ
            box(2,2)(btex \strut$a$ etex) otimes 
              box(2,2)(btex \strut$b$ etex) circ
            same(1) otimes cross(1,2) circ
            same(2) otimes cross(1,1) 
          ) feedback 2
        ;
      endfig;
    \end{mpgraphics*}\right]
    = \penalty-10
    \left[ \begin{mpgraphics*}{46}
      beginfig(46);
        PROPdiagram(0,0)
          (
            same(2) otimes cross(1,1) circ
            same(1) otimes cross(2,1) circ
            box(2,2)(btex \strut$a$ etex) otimes 
              box(2,2)(btex \strut$b$ etex) circ
            same(1) otimes cross(1,2) 
          ) feedback 2
        ;
      endfig;
    \end{mpgraphics*}\right]
    = \penalty 10
    \left[ \begin{mpgraphics*}{47}
      beginfig(47);
        PROPdiagram(0,0)
          (
            same(1) otimes cross(1,2) circ
            same(2) otimes cross(1,1) circ
            same(2) otimes box(2,2)(btex \strut$b$ etex) circ
            same(2) otimes cross(1,1) circ
            box(2,2)(btex \strut$a$ etex) otimes same(2) circ
            same(1) otimes cross(1,1) otimes same(1)
          ) feedback 2
        ;
      endfig;
    \end{mpgraphics*}\right]
    =
    \left[ \begin{mpgraphics*}{48}
      beginfig(48);
        PROPdiagram(0,0)
          (
            same(1) otimes cross(1,1) otimes same(1) circ
            same(1) otimes box(2,2)(btex \strut$b$ etex) otimes same(1) circ
            same(1) otimes cross(1,1) otimes same(1) circ
            same(2) otimes cross(1,1) circ
            same(1) otimes cross(1,1) otimes same(1) circ
            box(2,2)(btex \strut$a$ etex) otimes same(2) circ
            same(1) otimes cross(1,1) otimes same(1)
          ) feedback 2
        ;
      endfig;
    \end{mpgraphics*}\right]
    =
    \left[ \begin{mpgraphics*}{49}
      beginfig(49);
        PROPdiagram(0,0)
          (
            same(1) otimes cross(1,1) circ
            same(1) otimes box(2,2)(btex \strut$b$ etex) circ
            same(1) otimes cross(1,1) circ
            same(2) otimes cross(1,1) feedback 1 circ
            same(1) otimes cross(1,1) circ
            box(2,2)(btex \strut$a$ etex) otimes same(1) circ
            same(1) otimes cross(1,1)
          ) feedback 1
        ;
      endfig;
    \end{mpgraphics*}\right]
    =
    \left[ \begin{mpgraphics*}{50}
      beginfig(50);
        PROPdiagram(0,0)
          (
            same(1) otimes cross(1,1) circ
            same(1) otimes box(2,2)(btex \strut$b$ etex) circ
            same(1) otimes cross(1,1) circ
            same(1) otimes cross(1,1) circ
            box(2,2)(btex \strut$a$ etex) otimes same(1) circ
            same(1) otimes cross(1,1)
          ) feedback 1
        ;
      endfig;
    \end{mpgraphics*}\right]
    =
    \left[ \begin{mpgraphics*}{51}
      beginfig(51);
        PROPdiagram(0,0)
          (
            same(1) otimes cross(1,1) circ
            same(1) otimes box(2,2)(btex \strut$b$ etex) circ
            box(2,2)(btex \strut$a$ etex) otimes same(1) circ
            same(1) otimes cross(1,1)
          ) feedback 1
        ;
      endfig;
    \end{mpgraphics*}\right]
    \hspace{-\rightskip}
    \),
  \end{center}
  and in formulae (with a separate step for each axiom application) 
  that reads
  \begin{multline*}
    \bigl(
      \phi(\same{k} \star \cross{r}{m}) \circ
      a \otimes \phi(\same{m}) \circ
      \phi(\same{l}) \otimes b \circ
      \phi(\same{l} \star \cross{n}{r})
    \bigr) \feedback{r}
    = \\ =
    \bigl(
      \phi(\same{k} \star \cross{r}{m}) \circ
      a \otimes \phi(\same{m}) \circ
      \phi(\same{l} \star \cross{m}{q}) \circ
      \phi(\same{l} \star \cross{q}{m}) \circ
      \phi(\same{l}) \otimes b \circ
      \phi(\same{l} \star \cross{n}{r})
    \bigr) \feedback{r}
    = \displaybreak[0]\\ =
    \begin{aligned}[t] \bigl( &
      \phi(\same{k} \star \cross{r}{m}) \circ
      a \otimes \phi(\same{m}) \circ
      \phi(\same{l} \star \cross{m}{q}) \circ
      \phi(\same{l} \star \same{m}) \otimes 
        \phi(\cross{q}{q})\feedback{q} \circ 
      \\ & \quad
      \phi(\same{l} \star \cross{q}{m}) \circ
      \phi(\same{l}) \otimes b \circ
      \phi(\same{l} \star \cross{n}{r})
    \bigr) \feedback{r} = \end{aligned}
    \displaybreak[0]\\ =
    \begin{aligned}[t] \bigl( &
      \phi(\same{k} \star \cross{r}{m}) \circ
      a \otimes \phi(\same{m}) \circ
      \phi(\same{l} \star \cross{m}{q}) \circ
      \phi(\same{l} \star \same{m} \star \cross{q}{q})\feedback{q} \circ 
      \\ & \quad
      \phi(\same{l} \star \cross{q}{m}) \circ
      \phi(\same{l}) \otimes b \circ
      \phi(\same{l} \star \cross{n}{r})
    \bigr) \feedback{r} = \end{aligned}
    \displaybreak[0]\\ =
    \begin{aligned}[t] \bigl( &
      \phi(\same{k} \star \cross{r}{m} \star \same{q}) \circ
      a \otimes \phi(\same{m} \star \same{q}) \circ
      \phi(\same{l} \star \cross{m}{q} \star \same{q}) \circ
      \phi(\same{l} \star \same{m} \star \cross{q}{q}) \circ
      \\ & \quad
      \phi(\same{l} \star \cross{q}{m} \star \same{q}) \circ
      \phi(\same{l}) \otimes b \otimes \phi(\same{q}) \circ
      \phi(\same{l} \star \cross{n}{r} \star \same{q})
    \bigr) \feedback{q}\feedback{r} = \end{aligned}
    \displaybreak[0]\\ =
    \bigl(
      \phi(\same{k} \star \cross{r+q}{m}) \circ
      a \otimes b \circ
      \phi(\same{l} \star \cross{n}{q+r}) \circ
      \phi(\same{l} \star \same{n} \star \cross{r}{q})
    \bigr) \feedback{q}\feedback{r}
    = \displaybreak[0]\\ =
    \bigl(
      \phi(\same{k} \star \same{m} \star \cross{r}{q}) \circ
      \phi(\same{k} \star \cross{r+q}{m}) \circ
      a \otimes b \circ
      \phi(\same{l} \star \cross{n}{q+r})
    \bigr) \feedback{r}\feedback{q}
    = \displaybreak[0]\\ =
    \begin{aligned}[t] \bigl( &
      \phi(\same{k} \star \cross{q}{m} \star \same{r}) \circ
      \phi(\same{k}) \otimes b \otimes \phi(\same{r}) \circ
      \phi(\same{k} \star \cross{n}{r} \star \same{r}) \circ
      \phi(\same{k} \star \same{n} \star \cross{r}{r}) \circ
      \\ & \quad
      \phi(\same{k} \star \cross{r}{n} \star \same{r}) \circ
      a \otimes \phi(\same{n} \star \same{r}) \circ
      \phi(\same{l} \star \cross{n}{q} \star \same{r})
    \bigr) \feedback{r}\feedback{q} = \end{aligned}
    \displaybreak[0]\\ =
    \begin{aligned}[t] \bigl( &
      \phi(\same{k} \star \cross{q}{m}) \circ
      \phi(\same{k}) \otimes b \circ
      \phi(\same{k} \star \cross{n}{r}) \circ
      \phi(\same{k} \star \same{n} \star \cross{r}{r})\feedback{r} \circ
      \\ & \quad
      \phi(\same{k} \star \cross{r}{n}) \circ
      a \otimes \phi(\same{n}) \circ
      \phi(\same{l} \star \cross{n}{q})
    \bigr) \feedback{q} = \end{aligned}
    \displaybreak[0]\\ =
    \begin{aligned}[t] \bigl( &
      \phi(\same{k} \star \cross{q}{m}) \circ
      \phi(\same{k}) \otimes b \circ
      \phi(\same{k} \star \cross{n}{r}) \circ
      \phi(\same{k} \star \same{n}) \otimes 
        \phi(\cross{r}{r})\feedback{r} \circ
      \\ & \quad
      \phi(\same{k} \star \cross{r}{n}) \circ
      a \otimes \phi(\same{n}) \circ
      \phi(\same{l} \star \cross{n}{q})
    \bigr) \feedback{q} = \end{aligned}
    \displaybreak[0]\\ =
    \bigl(
      \phi(\same{k} \star \cross{q}{m}) \circ
      \phi(\same{k}) \otimes b \circ
      \phi(\same{k} \star \cross{n}{r}) \circ
      \phi(\same{k} \star \cross{r}{n}) \circ
      a \otimes \phi(\same{n}) \circ
      \phi(\same{l} \star \cross{n}{q})
    \bigr) \feedback{q}
    = \\ =
    \bigl(
      \phi(\same{k} \star \cross{q}{m}) \circ
      \phi(\same{k}) \otimes b \circ
      a \otimes \phi(\same{n}) \circ
      \phi(\same{l} \star \cross{n}{q})
    \bigr) \feedback{q}
    \text{.}
  \end{multline*}
\end{proof}

It may be observed that although this is in one sense a stronger form 
of sliding (the subexpression being slided can have external 
connections), in proving it the sliding axiom was only applied for 
$\phi(\cross{r}{q})$. Hence, in the presence of the other feedback 
axioms, the sliding axiom could (as is well known) be weakened to 
sliding of permutations.

\subsection{The free \PROP}

What holds for networks typically has a counterpart for the free 
\PROP.

\begin{corollary}[to Theorem~\ref{S:Feedback-join}]
  Let $\bigl( \mc{P}, \{F_q\}_{q \in \B^{\bullet\times\bullet}}, 
  \{\feedback{n}\}_{n\in\N} \bigr)$ be a \PROP\ with formal 
  feedbacks. Let $\Omega$ be an $\N^2$-graded set and \(f\colon 
  \Nwt(\Omega) \Fpil \mc{P}\) be a \PROP\ homomorphism. Let \(A = 
  \left[ \begin{smallmatrix} A_{11} & A_{12} \\ A_{21} & A_{22} 
  \end{smallmatrix} \right] \in \B^{\bullet\times\bullet}\) and \(B = 
  \left[ \begin{smallmatrix} B_{11} & B_{12} \\ B_{21} & B_{22} 
  \end{smallmatrix} \right] \in \B^{\bullet\times\bullet}\) be such 
  that \(A_{11} \in \B^{k \times l}\), \(A_{22} \in \B^{r \times 
  q}\), \(B_{22} \in \B^{q \times r}\), and \(B_{33} \in \B^{m \times 
  n}\) for some \(k,l,m,n,q,r \in \N\).
  
  If $A_{22} B_{22}$ is nilpotent, then 
  \begin{equation} \label{Eq:Kor:Feedback-join}
    f( a \join^r_q b ) =
    \bigl( 
      \phi(\same{k} \star \cross{r}{m}) \circ
      f(a) \otimes \phi(\same{m}) \circ
      \phi(\same{l}) \otimes f(b) \circ
      \phi(\same{l} \star \cross{n}{r}) 
    \bigr) \feedback{r}
  \end{equation}
  for all \(a \in \mc{Y}_\Omega(A)\) and \(b \in \mc{Y}_\Omega(B)\).
\end{corollary}
\begin{proof}
  By Theorem~\ref{S:NwtFriPROP}, every \PROP\ homomorphism $f$ from 
  $\Nwt(\Omega)$ is of the form $\eval_g$ for \(g = 
  \restr{f}{\Omega}\). Hence \eqref{Eq:Kor:Feedback-join} follows 
  from \eqref{Eq:Feedback-join} for some \(K \in a\) and \(H \in b\).
\end{proof}

Somewhat notable here is that it suffices for $f$ to be a \PROP\ 
homomorphism; one need not assume that it respects the feedbacks, as 
those are sufficiently determined by the axioms anyway. This holds 
even though the free \PROP\ is also a \PROP\ with formal feedbacks.

\begin{lemma}
  For any $\N^2$-graded set, $\bigl( \Nwt(\Omega), \mc{Y}_\Omega, 
  \{ \rtimes \phi(\same{n})\}_{n \in \N} \bigr)$ is a \PROP\ with 
  formal feedbacks.
\end{lemma}
\begin{proof}
  Since \(a\feedback{n} = a \rtimes \phi(\same{n})\) was used as 
  template for the conditions about domain and range of a formal 
  feedback, this is a straightforward matter of verifying the various 
  axioms, to which end Theorem~\ref{S:AIN-sammanbindning} and 
  Lemma~\ref{L:join-associativitet} are of great help. 
  
  For tightening where \(a \in \Nwt(\Omega)(i,j)\), \(b \in 
  \Nwt(\Omega)(j +\nobreak n, k +\nobreak n)\), and \(c \in 
  \Nwt(\Omega)(k,l)\), one has after defining \(\vek{i} = 
  \vek{N}_i(0,6)\), \(\vek{j} = \vek{N}_j(1,6)\), \(\vek{k} = 
  \vek{N}_k(2,6)\), \(\vek{l} = \vek{N}_l(3,6)\), \(\vek{n} = 
  \vek{N}_n(4,6)\), and \(\vek{n'} = \vek{N}_n(5,6)\),
  \begin{multline*}
    \bigl( a \otimes \phi(\same{n}) \circ b \circ
      c \otimes \phi(\same{n}) \bigr) \feedback{n}
    =
    \fuse{ \vek{in} }{
      a^\vek{i}_\vek{j} b^\vek{jn}_\vek{kn'} c^\vek{k}_\vek{l}
    }{ \vek{ln'} } \feedback{n}
    = \\ =
    \Bigl(
      \fuse{ \vek{ik} }{
        a^\vek{i}_\vek{j} c^\vek{k}_\vek{l}
      }{ \vek{lj} } 
      \join^k_j
      \fuse{ \vek{jn} }{ b^\vek{jn}_\vek{kn'} }{ \vek{kn'} } 
    \Bigr) \rtimes \phi(\same{n})
    =
    \fuse{ \vek{ik} }{
      a^\vek{i}_\vek{j} c^\vek{k}_\vek{l}
    }{ \vek{lj} } 
    \join^k_j \bigl(
      b \rtimes \phi(\same{n})
    \bigr)
    = \\ =
    \fuse{ \vek{ik} }{
      a^\vek{i}_\vek{j} c^\vek{k}_\vek{l}
    }{ \vek{lj} } 
    \join^k_j 
    \fuse{ \vek{j} }{
      (b\feedback{n})^\vek{j}_\vek{k}
    }{ \vek{k} }
    =
    \fuse{ \vek{i} }{
      a^\vek{i}_\vek{j} 
      (b\feedback{n})^\vek{j}_\vek{k}
      c^\vek{k}_\vek{l}
    }{ \vek{l} } 
    =
    a \circ b\feedback{n} \circ c
    \text{.}
  \end{multline*}
  For superposing where \(a \in \Nwt(\Omega)(i,j)\) and \(b \in 
  \Nwt(\Omega)(k +\nobreak n, l +\nobreak n)\), one by 
  Corollary~\ref{Kor:Sammanbindingsformler} and 
  after defining \(\vek{i} = \vek{N}_i(0,6)\), \(\vek{j} = 
  \vek{N}_j(1,6)\), \(\vek{k} = \vek{N}_k(2,6)\), \(\vek{l} = 
  \vek{N}_l(3,6)\), \(\vek{n} = \vek{N}_n(4,6)\), and \(\vek{n'} = 
  \vek{N}_n(5,6)\) similarly has 
  \begin{multline*}
    a \otimes b\feedback{n}
    =
    a \join^0_0 \bigl( b \rtimes \phi(\same{n}) \bigr)
    = \\ =
    (a \join^0_0 b) \rtimes \phi(\same{n})
    =
    \fuse{ \vek{ikn} }{
      a^\vek{i}_\vek{j}
      b^\vek{kn}_\vek{ln'}
    }{ \vek{jln'} } \rtimes \phi(\same{n})
    =
    (a \otimes b) \feedback{n} \text{.}
  \end{multline*}
  
  For sliding \(a \in \Nwt(\Omega)(m,n)\) past \(b \in \Nwt(\Omega)(k 
  +\nobreak n, l +\nobreak m)\), the trick is to move $a$ to the 
  right operand of the symmetric join and back. After defining 
  \(\vek{k} = \vek{N}_k(0,6)\), \(\vek{l} = 
  \vek{N}_l(1,6)\), \(\vek{m} = \vek{N}_m(2,6)\), \(\vek{m'} = 
  \vek{N}_m(3,6)\), \(\vek{n} = \vek{N}_n(4,6)\), and \(\vek{n'} = 
  \vek{N}_n(5,6)\), one has
  \begin{align*}
    \bigl( \phi(\same{k}) \otimes a \circ b \bigr) \feedback{m}
    = {}&
    \fuse{ \vek{km} }{
      a^\vek{m}_\vek{n} b^\vek{kn}_\vek{lm'}
    }{ \vek{lm'} } \feedback{m}
    = \\ ={}&
    \Bigl(
      \fuse{ \vek{kn} }{
        b^\vek{kn}_\vek{lm'}
      }{ \vek{lm'} } 
      \join^n_m
      \fuse{ \vek{m'm} }{
        a^\vek{m}_\vek{n}
      }{ \vek{nm'} }
    \Bigr) \rtimes \phi(\same{m})
    = \displaybreak[0]\\ ={}&
    \fuse{ \vek{kn} }{
      b^\vek{kn}_\vek{lm'}
    }{ \vek{lm'} } 
    \join^n_m
    \Bigl(
      \fuse{ \vek{m'm} }{
        a^\vek{m}_\vek{n}
      }{ \vek{nm'} }
      \rtimes 
      \fuse{ \vek{m'} }{ \phi(\same{m})^\vek{m'}_\vek{m} }{ \vek{m} }
    \Bigr) 
    = \displaybreak[0]\\ ={}&
    \fuse{ \vek{kn} }{
      b^\vek{kn}_\vek{lm'}
    }{ \vek{lm'} } 
    \join^n_m
    \fuse{ \vek{m'} }{
      \phi(\same{m})^\vek{m'}_\vek{m}
      a^\vek{m}_\vek{n}
    }{ \vek{n} }
    = \displaybreak[0]\\ ={}&
    \fuse{ \vek{kn'} }{
      b^\vek{kn'}_\vek{lm}
    }{ \vek{lm} } 
    \join^n_m
    \fuse{ \vek{m} }{
      a^\vek{m}_\vek{n}
      \phi(\same{n})^\vek{n}_\vek{n'}
    }{ \vek{n'} }
    = \displaybreak[0]\\ ={}&
    \fuse{ \vek{kn'} }{
      b^\vek{kn'}_\vek{lm}
    }{ \vek{lm} } 
    \join^n_m
    \Bigl(
      \fuse{ \vek{mn'} }{
        a^\vek{m}_\vek{n}
      }{ \vek{n'n} }
      \rtimes
      \fuse{ \vek{n} }{
        \phi(\same{n})^\vek{n}_\vek{n'}
      }{ \vek{n'} }
    \Bigr)
    = \displaybreak[0]\\ ={}&
    \Bigl(
      \fuse{ \vek{kn'} }{
        b^\vek{kn'}_\vek{lm}
      }{ \vek{lm} } 
      \join^n_m
      \fuse{ \vek{mn'} }{
        a^\vek{m}_\vek{n}
      }{ \vek{n'n} }
    \Bigr)
    \rtimes \phi(\same{n})
    = \\ ={}&
    \fuse{ \vek{kn'} }{
      b^\vek{kn'}_\vek{lm}
      a^\vek{m}_\vek{n}
    }{ \vek{ln} }
    \feedback{n}
    =
    \bigl( b \circ \phi(\same{l}) \otimes a \bigr) \feedback{n}
    \text{.}
  \end{align*}
  For vanishing of feedbacks on \(a \in \Nwt(\Omega)(k +\nobreak m 
  +\nobreak n, l +\nobreak m +\nobreak n)\), one similarly might want 
  to define \(\vek{m} = \vek{N}_m(0,4)\), \(\vek{m'} = \vek{N}_m(1,4)\), 
  \(\vek{n} = \vek{N}_n(2,4)\), and \(\vek{n'} = \vek{N}_n(3,4)\). By 
  Corollary~\ref{Kor:Sammanbindingsformler},
  \begin{align*}
    a \feedback{n} \feedback{m}
    ={}& 
    \Bigl(
      \bigl( a \join^{m+n}_{m+n} \phi(\cross{m+n}{m+n}) \bigr)
      \rtimes \phi(\same{n})
    \Bigr) \feedback{m}
    = \\ ={}& 
    \Bigl(
      a \join^{m+n}_{m+n} 
      \bigl(
        \fuse{ \vek{m'n'mn} }{ 1 }{ \vek{mnm'n'} }
        \join^n_n
        \fuse{ \vek{n'} }{ \phi(\same{n})^\vek{n'}_\vek{n} }{ \vek{n} }
      \bigr)
    \Bigr) \feedback{m}
    = \displaybreak[0]\\ ={}&
    \Bigl(
      a \join^{m+n}_{m+n} 
      \fuse{ \vek{m'n'm} }{ \phi(\same{n})^\vek{n'}_\vek{n} }{ \vek{mnm'} }
    \Bigr) \rtimes \phi(\same{m})
    = \displaybreak[0]\\ ={}&
    a \join^{m+n}_{m+n} \Bigl(
      \fuse{ \vek{m'n'm} }{ \phi(\same{n})^\vek{n'}_\vek{n} }{ \vek{mnm'} }
      \join^m_m
      \fuse{ \vek{m'} }{ \phi(\same{m})^\vek{m'}_\vek{m} }{ \vek{m} }
    \Bigr)
    = \displaybreak[0]\\ ={}&
    a \join^{m+n}_{m+n}
    \fuse{ \vek{m'n'} }{ 
      \phi(\same{n})^\vek{n'}_\vek{n} 
      \phi(\same{m})^\vek{m'}_\vek{m}
    }{ \vek{mn} }
    = \\ ={}&
    a \join^{m+n}_{m+n} \bigl(
      \phi(\same{m}) \otimes \phi(\same{n})
    \bigr)
    =
    a \join^{m+n}_{m+n} \phi(\same{m+n})
    =
    a \feedback{m+n}
    \text{.}
  \end{align*}
  Vanishing also requires \(a\feedback{0} = a \join^0_0 \phi(\same{0}) 
  = a\), which like the yanking identity 
  \(\phi(\cross{n}{n})\feedback{n} = \phi(\cross{n}{n}) \join^n_n 
  \phi(\same{n}) = \phi(\same{n})\) is an immediate consequence of 
  that same corollary.
\end{proof}

In order to show that a pullback $f^*Q$ to $\Nwt(\Omega)$ of some 
order $Q$ on a helper \PROP\ $\mc{P}$ is preserved under symmetric 
join, the obvious approach which suggests itself is that one equips 
$\mc{P}$ with a (possibly formal) feedback and checks that $Q$ is 
preserved under that. This works extremely well for \PROPs\ with a 
cap--cup pair as in Theorem~\ref{S:CupCapFeedback}, for example the 
connectivity \PROP\ and $\HomPROP_V$ for finite-dimensional $V$, 
since the feedback is in that case just a \PROP\ expression and any 
\PROP\ quasi-order will therefore be preserved by the feedback. For 
other \PROPs\ it may however get a bit roundabout, since setting up a 
formal feedback can be quite a bit of work. An alternative 
approach is to verify the following property, which may be more 
complicated to state but often is easier to prove.

\begin{definition}
  Let $\mc{P}$ be a \PROP\ and $Q$ be an $\N^2$-graded quasi-order on 
  $\mc{P}$. The relation $Q$ is said to have the \DefOrd{uncut 
  property} if for all \(l,m,m',n \in \N\), all \(a \in 
  \mc{P}(l,m +\nobreak 1)\), \(b \in \mc{P}( m +\nobreak 1, n)\), 
  \(c \in \mc{P}(l,m' +\nobreak 1)\), and \(d \in \mc{P}(m',n 
  +\nobreak 1)\) it holds that
  \begin{subequations} \label{Eq:Okapad}
    \begin{align} \label{Eq:a:Okapad}
      a \otimes \phi(\same{1}) \circ 
      \phi(\same{m} \star \cross{1}{1}) \circ
      b \otimes \phi(\same{1})
      \leqslant{}&
      c \otimes \phi(\same{1}) \circ 
      \phi(\same{m'} \star \cross{1}{1}) \circ
      d \otimes \phi(\same{1})
      \pin{Q}\\
    \intertext{implies}
      a \circ b \leqslant{}& c \circ d \pin{Q} \text{.}
    \end{align}
  \end{subequations}
  (In other words, the `un-' in this `uncut' has the same meaning as 
  in `undo'; one can go from the case where an edge has been cut to 
  the case where is hasn't been cut.)
  $Q$ is said to have the \DefOrd{strict uncut property} if the 
  implication \eqref{Eq:Okapad} holds also if the $\leqslant$ 
  are replaced by $<$.
\end{definition}

It is for example straightforward to verify that the standard order 
on matrices has the uncut property.

\begin{lemma}
  Let $(\mc{R},P)$ be an associative and unital partially ordered 
  semiring, and let $Q$ be the standard order on 
  $(\mc{R}_{\geqslant 0})^{\bullet\times\bullet}$. Then $Q$ has the 
  uncut property.
  Furthermore, if $P$ is cancellative then $Q$ has the strict uncut 
  property on the sub-\PROP\ of matrices which have at least one 
  positive element in each row and column.
\end{lemma}
\begin{proof}
  The hypothesis \eqref{Eq:a:Okapad} can be written in block matrix 
  form as
  \begin{multline*}
    \begin{bmatrix}
      A_{11} & A_{12} & 0 \\
      A_{21} & A_{22} & 0 \\
      0 & 0 & 1
    \end{bmatrix}
    \begin{bmatrix}
      I & 0 & 0 \\
      0 & 0 & 1 \\
      0 & 1 & 0
    \end{bmatrix}
    \begin{bmatrix}
      B_{11} & B_{12} & 0 \\
      B_{21} & B_{22} & 0 \\
      0 & 0 & 1
    \end{bmatrix}
    \leqslant \\ \leqslant
    \begin{bmatrix}
      C_{11} & C_{12} & 0 \\
      C_{21} & C_{22} & 0 \\
      0 & 0 & 1
    \end{bmatrix}
    \begin{bmatrix}
      I & 0 & 0 \\
      0 & 0 & 1 \\
      0 & 1 & 0
    \end{bmatrix}
    \begin{bmatrix}
      D_{11} & D_{12} & 0 \\
      D_{21} & D_{22} & 0 \\
      0 & 0 & 1
    \end{bmatrix}
    \pin{Q}
  \end{multline*}
  where the last two rows and columns have side~$1$. Performing these 
  multiplications yields
  \begin{equation} \label{Eq:MatrisOkapad}
    \begin{bmatrix}
      A_{11}B_{11} & A_{11}B_{12} & A_{12} \\
      A_{21}B_{11} & A_{21}B_{12} & A_{22} \\
      B_{21} & B_{22} & 0
    \end{bmatrix}
    \leqslant
    \begin{bmatrix}
      C_{11}D_{11} & C_{11}B_{12} & C_{12} \\
      C_{21}D_{11} & C_{21}B_{12} & C_{22} \\
      D_{21} & D_{22} & 0
    \end{bmatrix}
    \pin{Q} \text{,}
  \end{equation}
  or equvalently
  \begin{align*}
    \begin{bmatrix}
      A_{11}B_{11} & A_{11}B_{12} \\
      A_{21}B_{11} & A_{21}B_{12}
    \end{bmatrix}
    \leqslant{}&
    \begin{bmatrix}
      C_{11}D_{11} & C_{11}D_{12} \\
      C_{21}D_{11} & C_{21}D_{12}
    \end{bmatrix}
    \pin{Q} \text{,}&
    \begin{bmatrix}
      A_{12} \\ A_{22}
    \end{bmatrix}
    \leqslant{}&
    \begin{bmatrix}
      C_{12} \\ C_{22}
    \end{bmatrix}
    \pin{Q} \text{,}\\
    \text{and}\qquad
    \begin{bmatrix}
      B_{21} & B_{22}
    \end{bmatrix}
    \leqslant{}&
    \begin{bmatrix}
      D_{21} & D_{22}
    \end{bmatrix}
    \pin{Q}
    \text{.}
  \end{align*}
  Hence
  \begin{align*}
    \begin{bmatrix}
      A_{11} & A_{12} \\ A_{21} & A_{22}
    \end{bmatrix}
    \begin{bmatrix}
      B_{11} & B_{12} \\ B_{21} & B_{22}
    \end{bmatrix}
    ={}&
    \begin{bmatrix}
      A_{11}B_{11} & A_{11}B_{12} \\
      A_{21}B_{11} & A_{21}B_{12}
    \end{bmatrix}
    +
    \begin{bmatrix}
      A_{12} \\ A_{22}
    \end{bmatrix}
    \begin{bmatrix}
      B_{21} & B_{22}
    \end{bmatrix}
    \leqslant \\ \leqslant{}&
    \begin{bmatrix}
      C_{11}D_{11} & C_{11}D_{12} \\
      C_{21}D_{11} & C_{21}D_{12}
    \end{bmatrix}
    +
    \begin{bmatrix}
      C_{12} \\ C_{22}
    \end{bmatrix}
    \begin{bmatrix}
      D_{21} & D_{22}
    \end{bmatrix}
    = \\ ={}&
    \begin{bmatrix}
      C_{11} & C_{12} \\ C_{21} & C_{22}
    \end{bmatrix}
    \begin{bmatrix}
      D_{11} & D_{12} \\ D_{21} & D_{22}
    \end{bmatrix}
    \pin{Q}
    \text{.}
  \end{align*}
  
  For preservation of strict inequalities, one may observe that this 
  is immediate from cancellativity if the strictness occurs within 
  the first two row and column blocks of \eqref{Eq:MatrisOkapad}. If 
  it occurs in the last column block ($A_{i2}$ versus $C_{i2}$) then
  \begin{equation*}
    \begin{bmatrix}
      A_{12} \\ A_{22}
    \end{bmatrix}
    \begin{bmatrix}
      B_{21} & B_{22}
    \end{bmatrix}
    \leqslant
    \begin{bmatrix}
      A_{12} \\ A_{22}
    \end{bmatrix}
    \begin{bmatrix}
      D_{21} & D_{22}
    \end{bmatrix}
    <
    \begin{bmatrix}
      C_{12} \\ C_{22}
    \end{bmatrix}
    \begin{bmatrix}
      D_{21} & D_{22}
    \end{bmatrix}
    \pin{Q}
  \end{equation*}
  because there is at least one positive element in the row 
  $\begin{bmatrix} D_{21} & D_{22} \end{bmatrix}$. Strictness in the 
  last row block ($B_{2j}$ versus $D_{2j}$) is similar.
\end{proof}

\begin{corollary} \label{Kor:BaffOkapad}
  The standard matrix order on $\Baff(\mc{R}_{\geqslant 0})$ also has 
  the uncut property, and if the partial order on $\mc{R}$ is 
  cancellative then then the sub-\PROP\ of matrices with at least one 
  positive element in each row and column has the strict uncut 
  property.
\end{corollary}
\begin{proof}
  Let \(\mc{P} = (\mc{R}_{\geqslant 0})^{\bullet\times\bullet}\) and 
  \(\mc{Q} = \Baff(\mc{R}_{\geqslant 0})\). Recall that $\mc{Q}$ is a 
  subset of $\mc{P}$. The tensor product and permutation operations 
  in $\mc{P}$ are not the same as in $\mc{Q}$, but \(a \otimes_\mc{Q} 
  \phi_\mc{Q}(\same{1}) = a \otimes_\mc{P} \phi_\mc{P}(\same{1})\) 
  for all \(a \in \mc{Q}\) and \(\phi_\mc{Q}(\sigma) = 
  \phi_\mc{P}(\same{2} \star\nobreak \sigma)\) for all permutations 
  $\sigma$, so any instance in $\mc{Q}$ of \eqref{Eq:Okapad} is also 
  an instance of if it in $\mc{P}$ (although with $l$, $m$, $m'$, and 
  $n$ all incremented by $2$).
\end{proof}

\begin{theorem} \label{S:UncutOrder}
  An $\N^2$-graded quasi-order $Q$ on the free \PROP\ $\Nwt(\Omega)$ 
  is preserved under symmetric join if and only if it is a \PROP\ 
  quasi-order which has the uncut property. Moreover it is strictly 
  preserved under symmetric join if and only if it is a strict \PROP\ 
  quasi-order which has the strict uncut property.
\end{theorem}
\begin{proof}
  That a quasi-order (strictly) preserved under symmetric join is a 
  (strict) \PROP\ quasi-order was shown in 
  Lemma~\ref{L:OrdningSammanbindingPROP}. The uncut property is a 
  immediate consequence of being preserved under the symmetric join 
  $\join^1_1 \phi(\same{1})$, as
  \begin{multline*}
    a \circ b
    =
    a \circ \phi(\same{m} \star \cross{1}{1})\feedback{1} \circ b
    =
    \bigl( 
      a \otimes \phi(\same{1}) \circ 
      \phi(\same{m} \star \cross{1}{1}) \circ
      b \otimes \phi(\same{1})
    \bigr) \feedback{1}
    = \\ =
    \bigl( 
      a \otimes \phi(\same{1}) \circ 
      \phi(\same{m} \star \cross{1}{1}) \circ
      b \otimes \phi(\same{1})
    \bigr) \join^1_1 \phi(\same{1})
    \leqslant \\ \leqslant
    \bigl( 
      c \otimes \phi(\same{1}) \circ 
      \phi(\same{m'} \star \cross{1}{1}) \circ
      d \otimes \phi(\same{1})
    \bigr) \join^1_1 \phi(\same{1})
    = \\ =
    \bigl( 
      c \otimes \phi(\same{1}) \circ 
      \phi(\same{m'} \star \cross{1}{1}) \circ
      d \otimes \phi(\same{1})
    \bigr) \feedback{1}
    = \\ =
    c \circ \phi(\same{m'} \star \cross{1}{1})\feedback{1} \circ d
    =
    c \circ d 
    \pin{Q}
  \end{multline*}
  for all \(a \in \Nwt(\Omega)(l,m)\), \(b \in \Nwt(\Omega)(m,n)\), 
  \(c \in \Nwt(\Omega)(l,m')\), and \(d \in \Nwt(\Omega)(m',n)\).
  
  That conversely any \PROP\ quasi-order $Q$ with the uncut property 
  will be preserved under symmetric join can be shown using the 
  formula \eqref{Eq:Kor:Feedback-join} for the symmetric join. Since 
  $\Nwt(\Omega)$ is a \PROP\ with formal feedback and 
  \(\mathrm{id}\colon \Nwt(\Omega) \Fpil \Nwt(\Omega)\) is a \PROP\ 
  homomorphism, it follows that
  \[
    a \join^r_q b =
    \bigl( 
      \phi(\same{k} \star \cross{r}{m}) \circ
      a \otimes \phi(\same{m}) \circ
      \phi(\same{l}) \otimes b \circ
      \phi(\same{l} \star \cross{n}{r}) 
    \bigr) \feedback{r}
  \]
  for \(a \in \Nwt(\Omega)(k +\nobreak r, l +\nobreak q)\) and \(b \in 
  \Nwt(\Omega)(q +\nobreak m, r +\nobreak n)\) such that the 
  symmetric join exists. Obviously, if \(a \leqslant a' \pin{Q}\) and 
  \(b \leqslant b' \pin{Q}\) then
  \begin{multline*}
    \phi(\same{k} \star \cross{r}{m}) \circ
    a \otimes \phi(\same{m}) \circ
    \phi(\same{l}) \otimes b \circ
    \phi(\same{l} \star \cross{n}{r}) 
    \leqslant \\ \leqslant
    \phi(\same{k} \star \cross{r}{m}) \circ
    a' \otimes \phi(\same{m}) \circ
    \phi(\same{l}) \otimes b' \circ
    \phi(\same{l} \star \cross{n}{r}) 
    \pin{Q}
  \end{multline*}
  since $Q$ is a \PROP\ quasi-order. Hence the wanted conclusion 
  follows once it has been established that \(c \leqslant c' 
  \pin{Q}\) implies \(c\feedback{r} \leqslant c'\feedback{r} 
  \pin{Q}\) for all \(c,c' \in \mc{Y}_\Omega(C)\), where \(C = \left[ 
  \begin{smallmatrix} C_{11} & C_{12} \\ C_{21} & C_{22} 
  \end{smallmatrix} \right] \in \B^{\bullet \times \bullet}\) has 
  \(C_{22} \in \B^{r \times r}\) nilpotent.
  
  By vanishing of feedbacks, this follows for all $r$ once it has 
  been shown for \(r=1\). Let \(G \in c\) and \(G' \in c'\) be 
  arbitrary. Let \(l = \omega(G)-1 = \omega(G')-1\) and \(n = 
  \alpha(G)-1 = \alpha(G') - 1\). Let $e_1$ be the edge in $G$ with tail 
  $1$ and tail index $n+1$, and let $e_0$ be the edge in $G$ with head 
  $0$ and head index $l+1$. Let $W_0$ be the set of all inner vertices 
  of $G$ which can be the headwards endpoint of a path whose first 
  edge is $e_1$, and let $W_1$ be the set of all other inner vertices 
  of $G$. 
  Then $(W_0,W_1)$ is a cut in $G$ and $e_1$ is by construction a cut 
  edge. $e_0$ is also a cut edge, because the head of $e_0$ is $0$, 
  and had the tail of $e_0$ been in $W_0$ then the $(l +\nobreak 1, n 
  +\nobreak 1)$ position of \(\Tr(G) = C\) would have been $1$, which 
  is impossible since $C_{22}$ is nilpotent. Let $m$ be the number of 
  cut edges other than $e_0$ and $e_1$. There is an ordering $p_1$ of 
  the cut such that \(p_1(e_0) = m+1\) and \(p_1(e_1) = m+2\); let 
  $(G_0,G_1)$ be the corresponding decomposition of $G$. For \(p_0 = 
  (\same{m} \star\nobreak \cross{1}{1}) \circ p_1\), one has that 
  $(W_0,\varnothing,p_0)$ is an ordered cut in $G_0$, and hence $G_0$ 
  has a cut decomposition $(G_{00},G_{01})$. From the choice of 
  ordering of these cuts, one sees
  \begin{align*}
    \eval(G_1) ={}& \eval(H_1) \otimes \phi(\same{1}) 
      && \text{for some \(H_1 \in \Nw(\Omega)(m+1,n)\),}\\
    \eval(G_{01}) ={}& \phi(\same{m} \star \cross{1}{1}) 
      \text{,}\\
    \eval(G_{00}) ={}& \eval(H_0) \otimes \phi(\same{1}) 
      && \text{for some \(H_1 \in \Nw(\Omega)(l,m+1)\).}
  \end{align*}
  By applying the same argument to $G'$, one gets some $m'$ (which 
  need not be equal to $m$), \(H_0' \in \Nw(\Omega)(l,m' +\nobreak 1)\) 
  and \(H_1' \in \Nw(\Omega)(m' +\nobreak 1, n)\) such that
  \[
    c' = 
    \eval(H_0') \otimes \phi(\same{1}) \circ
    \phi(\same{m} \star \cross{1}{1}) \circ
    \eval(H_1') \otimes \phi(\same{1})
    \text{.}
  \]
  Thus the uncut property applies to \(c \leqslant c' \pin{Q}\), and 
  it follows that \(c\feedback{1} = \eval(H_0) \circ \eval(H_1) 
  \leqslant \eval(H_0') \circ \eval(H_1') = c'\feedback{1} \pin{Q}\).
  
  Replacing some $\leqslant$ by $<$ yields the similar argument for $Q$ 
  also being strictly preserved under symmetric join.
\end{proof}

One could in $\mc{R}^{\bullet\times\bullet}$ similarly argue that the 
uncut property implies preservation under feedback, but this 
implication need not hold in general \PROPs; the tricky part is whether 
the fact that $c$ is in the domain of $\feedback{1}$ also gives rise 
to a decomposition of $c$ as $a \otimes \phi(\same{1}) \circ 
\phi(\same{m} \star\nobreak \cross{1}{1}) \circ b \otimes 
\phi(\same{1})$. For our purposes here, it is however sufficient that 
it holds in $\Nwt(\Omega)$, since that is what we primarily wish to 
order. For that end, a few results from Section~\ref{Sec:Ordning} 
need to be extended to the uncut property.

\begin{lemma}
  Let $\mc{P}$ and $\mc{Q}$ be \PROPs, and let \(f\colon \mc{P} \Fpil 
  \mc{Q}\) be a \PROP\ homomorphism. If $Q$ is a quasi-order on 
  $\mc{Q}$ which has the (strict) uncut property then the pullback 
  $f^*Q$ also has the (strict) uncut property.
  
  If $P$ and $Q$ are quasi-orders on $\mc{Q}$ which have the strict 
  uncut property, then $P \diamond Q$ has the strict uncut property 
  as well.
\end{lemma}
\begin{proof}
  Introduce the shorthand
  \[
    U(a,m,b) = 
    a \otimes \phi(\same{1}) \circ 
    \phi(\same{m} \star\nobreak \cross{1}{1}) \circ 
    b \otimes \phi(\same{1}) \text{.}
  \]
  For the first part, let \(l,m,m',n \in \N\), \(a \in \mc{P}(l, m 
  +\nobreak 1)\), \(b \in \mc{P}(m +\nobreak 1, n)\), \(c \in 
  \mc{P}(l, m' +\nobreak 1)\), and \(d \in \mc{P}(m' +\nobreak 1, n)\) 
  such that \(U(a,m,b) \leqslant U(c,m',d) \pin{f^*Q}\) be given. 
  Then by definition of pullback
  \[
    U\bigl( f(a), m, f(b) \bigr)
    =
    f\bigl( U(a,m,b) \bigr)
    \leqslant 
    f\bigl( U(c,m',d) \bigr)
    =
    U\bigl( f(c), m', f(d) \bigr)
    \pin{Q}
  \]
  and hence \(f(a \circ\nobreak b) = f(a) \circ f(b) \leqslant 
  f(c) \circ f(d) = f(c \circ\nobreak d) \pin{Q}\), meaning \(a \circ 
  b \leqslant c \circ d \pin{f^*Q}\). Replacing $\leqslant$ in this 
  argument by $<$ then yields the strict uncut property.
  
  For the lexicographic composition, one would instead let
  \(l,m,m',n \in \N\), \(a \in \mc{Q}(l, m +\nobreak 1)\), \(b \in 
  \mc{Q}(m +\nobreak 1, n)\), \(c \in \mc{Q}(l, m' +\nobreak 1)\), 
  and \(d \in \mc{Q}(m' +\nobreak 1, n)\) such that 
  \(U(a,m,b) < U(c,m',d) \pin{P \diamond Q}\) be given. If this is 
  because \(U(a,m,b) < U(c,m',d) \pin{P}\) then \(a \circ b < c \circ 
  d \pin{P}\) by the strict uncut property of $P$, and thus \(a \circ 
  b < c \circ d \pin{P \diamond Q}\). If this instead is because 
  \(U(a,m,b) \sim U(c,m',d) \pin{P}\) and \(U(a,m,b) < U(c,m',d) 
  \pin{Q}\), then \(a \circ b \leqslant c \circ d \leqslant a \circ b 
  \pin{P}\) by twice applying the uncut property of $P$ and \(a \circ 
  b < c \circ d \pin{Q}\) by the strict uncut property of $Q$, from 
  which similarly follows \(a \circ b < c \circ d \pin{P \diamond 
  Q}\). Finally if $a$, $b$, $c$, and $d$ instead satisfy \(U(a,m,b) 
  \sim U(c,m',d) \pin{P \diamond Q}\) then this holds separately in 
  both $P$ and $Q$, whence \(a \circ b \sim c \circ d\) in both $P$ 
  and $Q$ as well, and thus it also holds in $P \diamond Q$.
\end{proof}

\section{Rewriting}
\label{Sec:Omskrivning}

This section formally sets up a rewriting theory for free linear 
\PROPs\ $\mc{R}\{\Omega\}$, finishing off with some Diamond Lemmas. 
It is convenient to let the coefficient 
ring $\mc{R}$  and signature $\Omega$ of $\mc{R}\{\Omega\}$ be fixed 
throughout the section, so let $\mc{R}$ be an arbitrary associative 
and commutative ring with unit, and let $\Omega$ be an arbitrary 
$\N^2$-graded set.

The rewriting formalism used here will be that of~\cite{rpaper}. For 
the reader's convenience, elementary definitions of most concepts we 
shall employ are given below, whereas results are typically used 
without a precise statement being included here; instead there are 
detailed references. One thing to bear in mind is also that several 
of the more advanced ``features'' of the formalism in~\cite{rpaper}, 
in particular the topological structure used to handle rewriting of 
power series and the like, are ``deactivated'' here since they are 
fairly orthogonal to the issues at hand and would unnecessarily 
clutter the presentation if included; ``reactivating'' these features 
should present no great difficulty for the interested reader, since 
most intermediate results below apply as they are written also in 
that more general setting. 

One feature of the rewriting formalism that we will most definitly 
have to ``keep on'' is that it is many-sorted, meaning it works with 
expressions of several different sorts. An expression of one sort 
always rewrites to an expression of the same sort, but a rewrite step 
may employ a rule that was primarily stated for some other sort, since 
the expression it is applied to has a subexpression of that other 
sort. Since what matters for the basic \PROP\ operations are arity 
and coarity, one might expect that the sorts should be pairs $(m,n)$ 
of natural numbers, but in fact a much finer separation will be made: 
the (index) set for possible sorts is the \PROP\ of all boolean matrices 
$\B^{\bullet\times\bullet}$. The set of expressions of sort \(q \in 
\B^{\bullet\times\bullet}\) is then the module $\mc{M}(q)$ of 
Definition~\ref{Def:Y&M}. (It is not a problem that many elements of 
$\mc{R}\{\Omega\}$ belong to several of these modules, because when 
something is to be rewritten, it is formally being rewritten as an 
element of some specific $\mc{M}(q)$.)

\subsection{Reductions}

Formally, the idea that some $\mu$ of sort $r$ appears as a 
subexpression of some $\nu$ of sort $q$ is taken to mean that 
\(\nu = v(\mu)\) for some map $v$ in a specific family $\mc{V}(q,r)$.

\begin{definition} \label{D:KategorinV}
  For any \(q,r \in \B^{\bullet\times\bullet}\), let $\mc{V}(q,r)$ be the 
  set of all maps \(v\colon \mc{M}(r) \Fpil \mc{M}(q)\) of the form 
  \(v(b) = \mu \rtimes b\) for some \(\mu \in \mc{Y}(p)\), 
  where \(p = \left[ \begin{smallmatrix} p_{11} & p_{12} \\ p_{21} & 
  p_{22} \end{smallmatrix} \right]\) is such that $p_{22} r$ is 
  nilpotent and \(p_{11} + p_{12} r (p_{22} r)^* p_{21} \leqslant 
  q\). Putting it differently,
  \begin{equation}
    \mc{V}(q,r) = \bigcup_{\substack{ 
      \left[ \begin{smallmatrix} p_{11} & p_{12} \\ p_{21} & p_{22}
      \end{smallmatrix} \right] 
      \in \B^{(\omega(q)+\alpha(r)) \times (\alpha(q)+\omega(r))}
      \\
      p_{11} + p_{12} r (p_{22} r)^* p_{21} \leqslant q\\
      \text{$p_{22} r$ is nilpotent}
    }}
    \setOf[\Big]{ b \mapsto \mu \rtimes b }{ \mu \in \mc{Y}\left(
      \left[ \begin{smallmatrix} p_{11} & p_{12} \\ p_{21} & p_{22}
      \end{smallmatrix} \right]
    \right) }
    \text{.}
  \end{equation}
\end{definition}

\begin{lemma} \label{L:KategorinV}
  The $\mc{V}(q,r)$ families of maps have the following properties:
  \begin{enumerate}
    \item \label{Item:omtypning}
      If \(q \geqslant r\) then the identity map $\mathrm{id}$ on 
      $\mc{M}(r)$ is in $\mc{V}(q,r)$. This can be used to change the 
      sort of an expression to something less specific.
    \item \label{Item:kategori}
      If \(v_1 \in \mc{V}(q,r)\) and \(v_2 \in \mc{V}(p,q)\), then 
      \(v_2 \circ v_1 \in \mc{V}(p,r)\). Hence, $\mc{V}$ is a category.
    \item \label{Item:permutationer}
      For any \(\sigma \in \Sigma_m\), \(\tau \in \Sigma_n\), and \(q 
      \in \B^{m \times n}\), the map \(a \mapsto 
      \phi_{\mc{R}\{\Omega\}}(\sigma) \circ a \circ 
      \phi_{\mc{R}\{\Omega\}}(\tau)\) is in $\mc{V}\bigl( 
      \phi_{\B^{\bullet\times\bullet}}(\sigma) \circ\nobreak q 
      \circ\nobreak \phi_{\B^{\bullet\times\bullet}}(\tau), q 
      \bigr)\).
    \item \label{Item:Sammansattning}
      For any \(\lambda \in \mc{Y}(p)\) and \(\nu \in \mc{Y}(r)\), 
      the map \(a \mapsto \lambda \circ a \circ \nu\) is in 
      $\mc{V}(pqr,q)$.
    \item \label{Item:Tensorprodukt}
      For any \(\lambda \in \mc{Y}(p)\) and \(\nu \in \mc{Y}(r)\), 
      the map \(a \mapsto \lambda \otimes a \otimes \nu\) is in 
      $\mc{V}(p \otimes\nobreak q \otimes\nobreak r, q)$.
  \end{enumerate}
\end{lemma}
\begin{proof}
  For item~\ref{Item:omtypning}, let \(m = \omega(r)\) and 
  \(n=\alpha(r)\). Then \(a = \phi_{\mc{R}\{\Omega\}}(\cross{n}{m}) 
  \join^n_m a\) by Corollary~\ref{Kor:Sammanbindingsformler}, so the 
  wanted $\mu$ is $\phi_{\mc{R}\{\Omega\}}(\cross{n}{m})$ having 
  \(p_{11} = 0\), \(p_{12} = \phi(\same{m})\), \(p_{21} = 
  \phi(\same{n})\), and \(p_{22} = 0\).
  
  For item~\ref{Item:kategori}, let \(m = \omega(r)\), \(n = 
  \alpha(r)\), \(k = \omega(q)\), and \(l = \alpha(q)\). Then 
  \(v_1(c) = \mu_1 \join^n_m c\) for some \(\mu_1 \in \Nwt(\Omega)(k 
  +\nobreak m, l +\nobreak n)\) and \(v_2(c') = \mu_2 \join^l_k c'\) 
  for some \(\mu_2 \in \Nwt(\Omega)\bigl( \omega(p) +\nobreak k, 
  \alpha(p) +\nobreak l \bigr)\). More concretely,
  \[
    \mu_1 \in \mc{Y}\left( 
      \begin{bmatrix} b_{22} & b_{24} \\ b_{42} & b_{44} \end{bmatrix}
    \right)
    \qquad\text{and}\qquad
    \mu_2 \in \mc{Y}\left( 
      \begin{bmatrix} a_{11} & a_{12} \\ a_{21} & a_{22} \end{bmatrix}
    \right)
  \]
  for some boolean matrices satisfying that $a_{22}q$ and $b_{44}r$ are 
  nilpotent, \(a_{11} + a_{12} q (a_{22}q)^* a_{21} \leqslant p\), and 
  \(b_{22} + b_{24} r (b_{44}r)^* b_{42} \leqslant q\). Nilpotence of 
  $a_{22}q$ implies nilpotence of $a_{22}b_{22} + a_{22} b_{24} r 
  (b_{44}r)^* b_{42}$ and hence by Lemma~\ref{L:Matrisnilpotens} the 
  nilpotence of $a_{22}b_{22}$ and 
  $(a_{22}b_{22})^* a_{22}b_{24}r (b_{44}r)^* b_{42}$. This fulfils 
  the conditions for associativity of symmetric join 
  (cf.~Lemma~\ref{L:join-associativitet}) and hence
  \[
    v_2\bigl( v_1(c) \bigr) =
    \mu_2 \join^l_k (\mu_1 \join^n_m c) =
    (\mu_2 \join^l_k \mu_1) \join^n_m c =
    (\mu_2 \join^l_k \mu_1) \rtimes c
    \text{.}
  \]
  By Construction~\ref{Kons:symjoin},
  \begin{multline*}
    \mu_2 \join^l_k \mu_1 \in \mc{Y}\left( \begin{bmatrix}
      s_{11} & s_{14} \\ s_{41} & s_{44}
    \end{bmatrix} \right)
    \quad\text{where}\\
    \left\{\begin{aligned}
      s_{11} :={}& a_{11} + a_{12} b_{22} (a_{22}b_{22})^* a_{21} &
        s_{14} :={}& a_{12} (b_{22}a_{22})^* b_{24} \\
      s_{41} :={}& b_{42} (a_{22}b_{22})^* a_{21} &
        s_{44} :={}& b_{44} + b_{42} a_{22} (b_{22}a_{22})^* b_{24}
     \end{aligned} \right.
  \end{multline*}
  and it is convenient to also have the shorthand \(s_{44}' = 
  b_{42} a_{22} (b_{22}a_{22})^* b_{24}\). Then
  \begin{align*}
    p \geqslant{}&
    a_{11} + a_{12} q (a_{22}q)^* a_{21}
    \geqslant \\ \geqslant{}&
    a_{11} + a_{12} q \Bigl( a_{22}
      \bigl( b_{22} + b_{24} r (b_{44}r)^* b_{42} \bigr)
    \Bigr)^* a_{21}
    = \displaybreak[0]\\ ={}&
    a_{11} + a_{12} q \bigl( 
      a_{22} b_{22} + 
      a_{22} b_{24} r (b_{44}r)^* b_{42}
    \bigr)^* a_{21}
    = \displaybreak[0]\\ ={}&
    a_{11} + a_{12} q (a_{22} b_{22})^* \bigl( 
      a_{22} b_{24} r (b_{44}r)^* b_{42}
      (a_{22} b_{22})^* 
    \bigr)^* a_{21}
    = \displaybreak[0]\\ ={}&
    a_{11} + 
    a_{12} q (a_{22} b_{22})^* a_{21} + 
    a_{12} q (a_{22} b_{22})^* \bigl( 
      a_{22} b_{24} r (b_{44}r)^* b_{42}
      (a_{22} b_{22})^* 
    \bigr)^+ a_{21}
    \geqslant \displaybreak[0]\\ \geqslant{}&
    a_{11} + 
    a_{12} \bigl(
      b_{22} + b_{24} r (b_{44}r)^* b_{42}
    \bigr) (a_{22} b_{22})^* a_{21} + 
      {}\\ & \quad {}+
    a_{12} q (a_{22} b_{22})^* \bigl( 
      a_{22} b_{24} r (b_{44}r)^* b_{42}
      (a_{22} b_{22})^* 
    \bigr)^+ a_{21}
    = \displaybreak[0]\\ ={}&
    s_{11} + 
    a_{12} b_{24} r (b_{44}r)^* b_{42} (a_{22} b_{22})^* a_{21} + 
      {}\\ & \quad {}+
    a_{12} q (a_{22} b_{22})^* \bigl( 
      a_{22} b_{24} r (b_{44}r)^* b_{42}
      (a_{22} b_{22})^* 
    \bigr)^*
    a_{22} b_{24} r (b_{44}r)^* b_{42}
    (a_{22} b_{22})^* 
    a_{21}
    = \displaybreak[0]\\ ={}&
    s_{11} + 
    a_{12} b_{24} r (b_{44}r)^* s_{41} + 
      {}\\ & \quad {}+
    a_{12} q (a_{22} b_{22})^* a_{22} b_{24} r \bigl( 
      (b_{44}r)^* b_{42} (a_{22} b_{22})^* a_{22} b_{24} r 
    \bigr)^* (b_{44}r)^* s_{41}
    = \displaybreak[0]\\ ={}&
    s_{11} + 
    a_{12} b_{24} r (b_{44}r)^* s_{41} + 
    a_{12} q (a_{22} b_{22})^* a_{22} b_{24} r 
    ( b_{44}r + s_{44}' r )^* s_{41}
    = \displaybreak[0]\\ ={}&
    s_{11} + 
    a_{12} b_{24} r (b_{44}r)^* s_{41} + 
    a_{12} q (a_{22} b_{22})^* a_{22} b_{24} r (s_{44} r)^* s_{41}
    \geqslant \displaybreak[0]\\ \geqslant{}&
    s_{11} + 
    a_{12} b_{24} r (b_{44}r)^* s_{41} + 
      {}\\ & \quad {}+
    a_{12} b_{24} r (b_{44}r)^* b_{42} (a_{22} b_{22})^* a_{22} 
    b_{24} r (s_{44} r)^* s_{41} +
      {}\\ & \quad {}+
    a_{12} b_{22} (a_{22} b_{22})^* a_{22} b_{24} r (s_{44} r)^* s_{41}
    = \displaybreak[0]\\ ={}&
    s_{11} + 
    a_{12} b_{24} r \bigl( 
       (b_{44}r)^*  + 
       (b_{44}r)^* s_{44}' r (b_{44} r + s_{44}' r)^* 
    \bigr) s_{41} +
      {}\\ & \quad {}+
    a_{12} (b_{22} a_{22})^+ b_{24} r (s_{44} r)^* s_{41}
    = \displaybreak[0]\\ ={}&
    s_{11} + 
    a_{12} b_{24} r \bigl( 
       b_{44}r + s_{44}' r 
    \bigr)^* s_{41} +
    a_{12} (b_{22} a_{22})^+ b_{24} r (s_{44} r)^* s_{41}
    = \displaybreak[0]\\ ={}&
    s_{11} + 
    a_{12} (b_{22} a_{22})^* b_{24} r (s_{44} r)^* s_{41}
    = \\ ={}&
    s_{11} + s_{14} r (s_{44} r)^* s_{41}
  \end{align*}
  and hence \(c \mapsto (\mu_2 \join^l_k\nobreak \mu_1) \rtimes c\) 
  meets all conditions for being in $\mc{V}(p,r)$. The full family 
  $\bigl\{ \mc{V}(p,r) \bigr\}_{p,r \in \B^{\bullet\times\bullet}}$ then 
  fulfils the conditions for being a category: it contains 
  identities by the previous point, and it is closed under 
  composition (when domain and codomain matches).
  
  For item~\ref{Item:Sammansattning}, one may let \(\vek{k} = 
  \vek{N}_{\omega(p)}(0,4)\), \(\vek{l} = \vek{N}_{\omega(q)}(1,4)\), 
  \(\vek{m} = \vek{N}_{\alpha(q)}(2,4)\), and \(\vek{n} = 
  \vek{N}_{\alpha(r)}(3,4)\). Then
  \[
    \lambda \circ a \circ \nu 
    =
    \fuse{ \vek{k} }{ 
      \lambda^\vek{k}_\vek{l} a^\vek{l}_\vek{m} \nu^\vek{m}_\vek{n}
    }{ \vek{n} }
    =
    \fuse{ \vek{km} }{ 
      \lambda^\vek{k}_\vek{l} \nu^\vek{m}_\vek{n}
    }{ \vek{nl} }
    \rtimes a
    =
    \bigl( \lambda \otimes \nu \circ 
    \phi(\cross{\alpha(r)}{\alpha(p)}) \bigr) \rtimes a
  \]
  by Theorem~\ref{S:AIN-sammanbindning}, and \(\lambda \otimes 
  \nu \circ \phi(\cross{\alpha(r)}{\alpha(p)}) \in 
  \mc{Y}\left( \left[ \begin{smallmatrix} 0 & p \\ r & 0 
  \end{smallmatrix} \right] \right)\), which has \( 0 + p q (0q)^* r 
  = pqr \leqslant pqr\) as the definition of $\mc{V}$ requires. 
  Item~\ref{Item:permutationer} is a special case of 
  item~\ref{Item:Sammansattning}.
  
  For item~\ref{Item:Tensorprodukt}, one may similarly let \(\vek{i} = 
  \vek{N}_{\omega(p)}(0,6)\), \(\vek{j} = \vek{N}_{\alpha(p)}(1,6)\), 
  \(\vek{k} = \vek{N}_{\omega(q)}(2,6)\), \(\vek{l} = 
  \vek{N}_{\alpha(q)}(3,6)\), \(\vek{m} = \vek{N}_{\omega(r)}(4,6)\), 
  and \(\vek{n} = \vek{N}_{\alpha(r)}(5,6)\). Then
  \[
    \lambda \otimes a \otimes \nu
    =
    \fuse{ \vek{ikm} }{
      \lambda^\vek{i}_\vek{j} a^\vek{k}_\vek{l} \nu^\vek{m}_\vek{n}
    }{ \vek{jln} }
    =
    \fuse{ \vek{ikml} }{
      \lambda^\vek{i}_\vek{j} \nu^\vek{m}_\vek{n}
    }{ \vek{jlnk} } \rtimes a
  \]
  and
  \[
    \fuse{ \vek{ikml} }{
      \lambda^\vek{i}_\vek{j} \nu^\vek{m}_\vek{n}
    }{ \vek{jlnk} }
    \in
    \mc{Y}\Bigl( \fuse{ \vek{ikml} }{
      p^\vek{i}_\vek{j} r^\vek{m}_\vek{n}
    }{ \vek{jlnk} }_{\B^{\bullet\times\bullet}} \Bigr)
    =
    \mc{Y} \left( \begin{bmatrix}
      p & 0 & 0 & 0 \\
      0 & 0 & 0 & 
        \phi
        ( \same{\omega(q)} ) \\
      0 & 0 & r & 0 \\
      0 & \phi
        ( \same{\alpha(q)} ) 
        & 0 & 0 
    \end{bmatrix} \right)
  \]
  where again
  \[
    \begin{bmatrix}
      p & 0 & 0 \\
      0 & 0 & 0 \\
      0 & 0 & r
    \end{bmatrix}
    +
    \begin{bmatrix}
      0 \\
      \phi
        ( \same{\omega(q)} ) \\
      0 
    \end{bmatrix}
    q ( 0 q )^*
    \begin{bmatrix}
      0 & \phi
         ( \same{\alpha(q)} ) & 0
    \end{bmatrix}
    =
    \begin{bmatrix}
      p & 0 & 0 \\
      0 & q & 0 \\
      0 & 0 & r
    \end{bmatrix}
    \leqslant
    p \otimes q \otimes r
  \]
  as required.
\end{proof}

\begin{definition}
  A \DefOrd{rewrite rule} $s$ for $\mc{R}\{\Omega\}$ is a triplet 
  $(q_s, \mu_s, a_s)$, where \(q_s \in \B^{\bullet\times\bullet}\) is 
  the \DefOrd{transference type} of the rule, \(\mu_s \in \mc{Y}(q_s)\) 
  is the \DefOrd{left hand side} of the rule, and \(a_s \in \mc{M}(q_s)\) 
  is the \DefOrd{right hand side} of the rule. A \DefOrd{rewriting 
  system} $S$ for $\mc{R}\{\Omega\}$ is a set of rewrite rules.
  
  A rewrite rule $s$ is said to be \DefOrd{sharp} if \(q_s = 
  \Tr(\mu_s)\), and a rewriting system is sharp if all its rules are 
  sharp.
\end{definition}

A practical and convenient style of writing rewrite rules is that used 
in \eqref{Eq:SporadiskaRegler} and \eqref{Eq:Regelfamiljer}, but it may 
be best to explain in detail how it is meant to be interpreted. First, 
the left and right hand sides of a rule $s$ are given in naked 
abstract index notation, roughly as a formula on the form
\[ \label{Sid:AIN-regel}
  M \longmapsto \sum_{i=1}^k r_i A_i
\]
where \(\{r_i\}_{i=1}^k \subseteq \mc{R}\) whereas $M$ and the $A_i$ 
are notational products on the form $\prod_{j=1}^{l_i} 
(x_{i,j})^{\vek{u}_{i,j}}_{\vek{v}_{i,j}}$ and each \(x_{i,j} \in 
\Omega\bigl( \Norm{\vek{u}_{i,j}}, \Norm{\vek{v}_{i,j}} \bigr) \cup 
\{\delta\}\) (where the Kronecker $\delta$ is taken to have arity and 
coarity $1$; it is essentially the same as the $\natural$ used in 
Section~\ref{Sec:FriPROP}). For this formula to be well-formed, each 
$A_i$ must have the same unmatched indices as $M$; let \(\vek{b} = 
b_1 \dotsb b_m\) and \(\vek{c} = c_1 \dotsb c_n\) be lists such that 
$\norm{\vek{b}}$ is the set of unmatched superscripts in $M$ and 
similarly $\norm{\vek{c}}$ is the set of unmatched subscripts. Then 
in closed abstract index notation,
\begin{align*}
  \mu_s ={}& \fuse{ \vek{b} }{ M }{ \vek{c} } \text{,} &
  a_s ={}& \sum_{i=1}^k r_u \fuse{ \vek{b} }{ A_i }{ \vek{c} } \text{.}
\end{align*}
The meaning of a `where \(b_i \canmake c_j\)' clause is that $q_s$ 
has a $0$ in the $(i,j)$ position; in the typical case most outputs 
have a dependence on most of the inputs, so it tends to be more compact 
to list those entries that are $0$ than those that are $1$.

One thing that this style of writing a rule does not specify is the 
order in $\vek{b}$ and $\vek{c}$ of the unmatched indices, but that 
turns out to be irrelevant for the reduction steps that the rule can 
give rise to, mostly because of item~\ref{Item:permutationer} in 
Lemma~\ref{L:KategorinV}.

\begin{definition}
  A \DefOrd{simple reduction} of type \(p \in 
  \B^{\bullet\times\bullet}\), with respect to the system $S$, is an 
  $\mc{R}$-linear map \(t_{v,s}\colon \mc{M}(p) \Fpil \mc{M}(p)\) 
  ($\mc{R}$-module homomorphism) satisfying
  \begin{equation}
    t_{v,s}(\lambda) = \begin{cases}
      v(a_s)& \text{if \(\lambda = v(\mu_s)\),}\\
      \lambda& \text{otherwise}
    \end{cases}
    \qquad\text{for all \(\lambda \in \mc{Y}(p)\),}
  \end{equation}
  where \(s \in S\) is some rule and \(v \in \mc{V}(p,q_s)\). Denote by 
  $T_1(S)(p)$ the set of all simple reductions of type $p$ with 
  respect to $S$. A \DefOrd{reduction} (of type $p$, with respect to 
  $S$) in general is a finite composition of zero or more simple 
  reductions (of type $p$, with respect to $S$). Denote by $T(S)(p)$ 
  the set of all reductions of type $p$ with respect to $S$.
\end{definition}

By~\cite[Lemma~7.3]{rpaper}, all the $\mc{V}$-maps satisfy the 
technical condition of being \emDefOrd{advanceable} 
\cite[Def.~6.1]{rpaper} with respect to $T_1(S)$: for all \(p,q \in 
\B^{\bullet\times\bullet}\), \(t \in T_1(S)(q)\), \(\mu \in 
\mc{Y}(q)\), and \(v \in \mc{V}(p,q)\) there exists some \(t' \in 
T(S)(p)\) such that \(t'\bigl( v(\mu) \bigr) = v\bigl( t(\mu) 
\bigr)\). They are in general not absolutely advanceable however, 
since the symmetric join with one factor fixed need not be injective: 
if \(\eta \in \Omega(1,0)\) and \(\ve \in \Omega(0,1)\) then for 
\(\mu = \eta \otimes \eta \circ \ve \otimes \ve\) one finds \(\mu 
\rtimes \phi(\cross{1}{1}) = (\ve \circ\nobreak \eta) \otimes 
(\ve \circ\nobreak \eta) = \mu \rtimes \phi(\same{2})\).

\begin{definition}
  Let $S$ be a rewriting system for $\mc{R}\{\Omega\}$.
  A reduction \(t \in T(S)(q)\) is said to \DefOrd{act trivially} on 
  some \(a \in \mc{M}(q)\) if \(t(a)=a\). An element \(a \in \mc{M}(q)\) 
  is said to be \DefOrd{irreducible} (with respect to $S$ and $q$) if 
  all \(t \in T(S)(q)\) act trivially on it. The set of all irreducible 
  elements in $\mc{M}(q)$ is denoted \index{Irr@$\Irr$}\(\Irr(S)(q)\).
  Also let\index{I(S)@$\mc{I}(S)$}
  \begin{equation} \label{Eq:Def.I'(S)}
    \mc{I}(S)(q) = \sum_{t \in T(S)(q)} 
      \setOf[\big]{ a - t(a) }{ a \in \mc{M}(q) }
  \end{equation}
  and write \(a \equiv b \pmod{S}\)\index{= mod S@$\equiv\pmod{S}$} 
  for \(a - b \in \mc{I}(S)(J_{\omega(a)\times\alpha(a)})\), where 
  $J_{m \times n}$ denotes the $m \times n$ matrix of ones. An \(a \in 
  \Irr(S)(q)\) is said to be a \DefOrd[*{normal form}]{normal form of 
  \(b \in \mc{M}(q)\)} if \(a - b \in \mc{I}(S)(q)\).
\end{definition}

By Lemma~3.5 of~\cite{rpaper}, $\Irr(S)(q)$ is an $\mc{R}$-module and 
an alternative definition of it is
\[
  \Irr(S)(q) = \setOf[\big]{ b \in \mc{M}(q) }{ 
    \text{\(t(b)=b\) for all \(t \in T_1(S)(q)\)}
  }
  \text{,}
\]
i.e., it is sufficient to consider simple reductions. Under (a subset 
of) the conditions in the Diamond Lemma, it furthermore holds that
\begin{equation}
  \Irr(S)(q) = \Span\Bigl( \mc{Y}(q) \setminus 
    \bigl\{ v(\mu_s) \bigr\}_{v \in \mc{V}(q,q_s), s \in S}
  \Bigr)
\end{equation}
by Theorem~5.6 of~\cite{rpaper}. Therefore it in general becomes 
interesting to seek positive descriptions of the complement of the 
set of those networks which are right annexations of one of the rule 
left hand sides (in the terminology of formal language theory, one 
might wish for a generative description of the \emph{complement} of 
the language of networks which have one of the left hand sides as a 
subnetwork).

Note that it does not necessarily follow from \(a \equiv b \pmod{S}\) 
for some \(a,b \in \mc{M}(q)\) that \(a-b \in \mc{I}(S)(q)\) if $q$ is 
not a matrix of ones. The typical situation in which this happens is 
that there is a proof that $a$ is congruent to $b$ modulo $S$, but 
some intermediate step in this proof requires a transference strictly 
larger than $q$, even though $a$ and $b$ by themselves do not.

\begin{example} \label{Ex:ZigZag}
  Suppose \(\mathsf{u} \in \Omega(0,2)\) and \(\mathsf{n} \in 
  \Omega(2,0)\). The zig--zag identities for these may then be 
  expressed as the rules
  \begin{align*}
    s_1 ={}& \Bigl( J_{1 \times 1}, 
      \fuse{3}{ \mathsf{u}_{12} \mathsf{n}^{23} }{1}, 
      \phi(\same{1}) \Bigr)
      \text{,}\\
    s_2 ={}& \Bigl( J_{1 \times 1}, 
      \fuse{3}{ \mathsf{n}^{32} \mathsf{u}_{21} }{1},
      \phi(\same{1}) \Bigr)
      \text{.}
  \end{align*}
  These make \(\fuse{3}{ \mathsf{u}_{12} \mathsf{n}^{23} }{1} \equiv 
  \fuse{3}{ \mathsf{n}^{32} \mathsf{u}_{21} }{1} \pmod{\{s_1,s_2\}}\), 
  since
  \begin{multline*}
    \fuse{3}{ \mathsf{u}_{12} \mathsf{n}^{23} }{1} -
    \fuse{3}{ \mathsf{n}^{32} \mathsf{u}_{21} }{1}
    =
    \fuse{3}{ \mathsf{u}_{12} \mathsf{n}^{23} }{1} - \phi(\same{1}) 
    + \phi(\same{1}) - \fuse{3}{ \mathsf{n}^{32} \mathsf{u}_{21} }{1}
    = \\ =
    (\mathrm{id} - t_{\mathrm{id},s_1})\Bigl(
      \fuse{3}{ \mathsf{u}_{12} \mathsf{n}^{23} }{1}
    \Bigr) - (\mathrm{id} - t_{\mathrm{id},s_2})\Bigl(
      \fuse{3}{ \mathsf{n}^{32} \mathsf{u}_{21} }{1}
    \Bigr)
    \in
    \mc{I}\bigl( \{s_1,s_2\} \bigr)(J_{1 \times 1})
    \text{,}
  \end{multline*}
  but \(\fuse{3}{ \mathsf{u}_{12} \mathsf{n}^{23} }{1} -
  \fuse{3}{ \mathsf{n}^{32} \mathsf{u}_{21} }{1} \notin \mc{I}\bigl( 
  \{s_1,s_2\} \bigr)(0)\). Demonstrating this latter point is not 
  quite trivial\Ldash one would like to make use of the Diamond Lemma 
  for that\Rdash but it is straightforward at this point in the 
  exposition to show that \(\fuse{3}{ \mathsf{u}_{12} \mathsf{n}^{23} 
  }{1} - \fuse{3}{ \mathsf{n}^{32} \mathsf{u}_{21} }{1} \in \Irr\bigl( 
  \{s_1,s_2\} \bigr)(0)\), which together with (c)~of 
  Lemma~\ref{L:entyps-DL} yields the \(\notin \mc{I}\bigl( \{s_1,s_2\} 
  \bigr)(0)\).
  
  Considering first the problem of whether there exists some simple 
  reduction \(t_{v,s} \in T_1\bigl( \{s_1,s_2\} \bigr)(0)\) which acts 
  nontrivially on $\fuse{3}{ \mathsf{u}_{12} \mathsf{n}^{23} }{1}$, one 
  may look for an embedding of some \(H \in \mu_s\) into \(G = 
  \Nwfuse{3}{ \mathsf{u}_{12} \mathsf{n}^{23} }{1}\), and the vertex 
  part of such an embedding is fixed by the decorations. There is no 
  way to embed \(H_2 = \Nwfuse{3}{ \mathsf{n}^{32} \mathsf{u}_{21} }{1}\) 
  into $G$ however, as in $H_2$ the second out-edge of the $\mathsf{n}$ 
  vertex and the first in-edge of the $\mathsf{u}$ vertex are the same, 
  whereas in $G$ they are distinct. Hence the possibility that 
  remains is $s=s_1$, and since \(G \in \mu_{s_1}\) it is pretty 
  obvious that there is such an embedding\Ldash the identity 
  $(\mathrm{id},\mathrm{id})$\Rdash and since the mapping of the 
  vertices is fixed and there are no isolated edges in these 
  networks, that is the only embedding there is. The lack of isolated 
  edges also imply that there is only one strong embedding, and thus 
  the only way of writing $\fuse{3}{ \mathsf{u}_{12} \mathsf{n}^{23} 
  }{1}$ as a right annexation of some $\mu_{s_i}$ is \(\fuse{3}{ 
  \mathsf{u}_{12} \mathsf{n}^{23} }{1} = \phi(\cross{1}{1}) \rtimes 
  \mu_{s_1}\). There is however also the matter of the 
  transferences to consider. \(q_{s_1} = (1)\) and \(\Tr\bigl( 
  \phi(\cross{1}{1}) \bigr) = \left( \begin{smallmatrix} 0 & 1 \\ 1 & 
  0 \end{smallmatrix} \right)\), which means \(v \in 
  \mc{V}(p,q_{s_1})\) such that \(v(b) = \phi(\cross{1}{1}) \rtimes 
  b\) only exist for \(p \geqslant 1 \cdot q_{s_1} \cdot 1 = J_{1 
  \times 1}\). Hence that \(t_{v,s_1} \notin T_1\bigl( \{s_1,s_2\} 
  \bigr)(0)\), and \(\fuse{3}{ \mathsf{u}_{12} \mathsf{n}^{23} }{1} 
  \in \Irr\bigl( \{s_1,s_2\} \bigr)(0)\). The situation for the other 
  term $\fuse{3}{ \mathsf{n}^{32} \mathsf{u}_{21} }{1}$ is completely 
  analogous.
\end{example}

The following lemma corresponds closely to Lemma~3.7 and 
Corollary~3.9 of~\cite{rpaper}, but the differences in context warrant 
an explicit proof here.

\begin{lemma} \label{L:I(S)}
  Let $S$ be a rewriting system for $\mc{R}\{\Omega\}$.
  The $\equiv \!\! \pmod{S}$ relation is an $\mc{R}$-linear \PROP\ 
  congruence relation on $\mc{R}\{\Omega\}$.
  
  $\mc{I}(S)(q)$ is an $\mc{R}$-module for every \(q \in 
  \B^{\bullet\times\bullet}\), and \(\ker t \subseteq \mc{I}(S)(q)\) 
  for every \(t \in T(S)(q)\). For every \(v \in \mc{V}(q,p)\) it 
  holds that \(\setmap{v}\bigl( \mc{I}(S)(p) \bigr) \subseteq 
  \mc{I}(S)(q)\). Furthermore any \(b \in \mc{I}(S)(p)\) can be 
  written as
  \begin{equation} \label{Eq1:I(S)}
    b = \sum_{i=1}^n r_i v_i(\mu_{s_i} - a_{s_i})
  \end{equation}
  for some \(n \in \N\), \(\{r_i\}_{i=1}^n \subseteq \mc{R}\), 
  \(\{s_i\}_{i=1}^n \subseteq S\), and \(v_i \in \mc{V}(p,q_{s_i})\) for 
  \(i=1,\dotsc,n\).
\end{lemma}
\begin{proof}
  Let \(t \in T(S)(q)\) be arbitrary. For any \(b,c \in \mc{M}(q)\) 
  and \(r \in \mc{R}\) one finds that
  \begin{align*}
    \bigl( b - t(b) \bigr) - \bigl( c - t(c) \bigr) =
    b - c - t(b) + t(c) ={}& \\ =
    (b-c) - t(b-c) \in{}&
    \setOf[\big]{ a - t(a) }{ a \in \mc{M}(q) } 
    \text{,}\\
    r \bigl( b - t(b) \bigr) =
    rb - t(rb) \in{}&
    \setOf[\big]{ a - t(a) }{ a \in \mc{M}(q) }
  \end{align*}
  and hence \(N_t := \setOf[\big]{ a - t(a) }{ a \in \mc{M}(q) }\) is 
  an $\mc{R}$-module. Obviously the sum of a family of $\mc{R}$-module 
  is an $\mc{R}$-module, so $\mc{I}(S)(q)$ is an $\mc{R}$-module. From 
  the similar observation that any \(a \in \ker t\) satisfies \(a = 
  a - t(a) \in N_t \subseteq \mc{I}(S)(q)\), it follows that \(\ker t 
  \subseteq \mc{I}(S)(q)\) for \(t \in T(S)(q)\).
  
  Next, it will be shown that
  \begin{equation} \label{Eq2:I(S)}
    \mc{I}(S)(q) = 
    \sum_{t \in T_1(S)(q)} \setOf[\big]{ a - t(a) }{ a \in \mc{M}(q) }
    \text{,}
  \end{equation}
  i.e., it is sufficient to let $t$ range over the simple reductions. 
  Clearly $N_\mathrm{id}$ is just $\{0\}$, and thus does not contribute 
  anything unique to $\mc{I}(S)(q)$. Any other nonsimple reduction \(t 
  \in T(S)(q)\) is a finite composition \(t_n \circ \dotsb \circ t_1 
  = t\) of simple reductions \(t_1,\dotsc,t_n \in T_1(S)(q)\) and it 
  holds that \(N_t \subseteq \sum_{k=1}^n N_{t_k}\), because if \(u_k = 
  t_k \circ \dotsb \circ t_1\) for \(k=1,\dotsc,n\) then any \(a - t(a) 
  \in N_t\) can be written as \(a - t_1(a) + u_1(a) - 
  t_2\bigl(u_1(a)\bigr) + \dotsb + u_{n-1}(a) - t_n\bigl( u_{n-1}(a) 
  \bigr) \in N_{t_1} + N_{t_2} + \dotsb + N_{t_n}\).
  Hence \(\sum_{t \in T(S)(q)} N_t \subseteq \sum_{t \in T_1(S)(q)} N_t\); 
  the terms for simple reductions suffice for producing the total sum.
  
  By definition, the only element of $\mc{Y}(p)$ that a simple reduction 
  \(t_{w,s} \in T_1(S)(p)\) acts nontrivially on is $w(\mu_s)$, so 
  \begin{multline*}
    \setOf[\big]{ a - t_{w,s}(a) }{ a \in \mc{M}(p) }
    =
    \setOf[\Big]{ rw(\mu_s)  - t_{w,s}\bigl( rw(\mu_s) \bigr) }{ 
      r \in \mc{R} }
    = \\ =
    \setOf[\Big]{ rw(\mu_s - a_s) }{ r \in \mc{R} }
    \text{.}
  \end{multline*}
  Thus any \(b \in \mc{I}(S)(p)\), which by \eqref{Eq2:I(S)} can be 
  written as a sum of such $a - t_{w,s}(a)$, is therefore also as 
  claimed in \eqref{Eq1:I(S)}. Now let \(q \in 
  \B^{\bullet\times\bullet}\) and \(v \in \mc{V}(q,p)\) be given. \(v 
  \circ v_i \in \mc{V}(q,q_i)\) by Lemma~\ref{L:KategorinV}, so
  \begin{multline*}
    v(b) =
    v\biggl( \sum_{i=1}^n r_i v_i(\mu_{s_i} - a_{s_i}) \biggr)
    =
    \sum_{i=1}^n r_i (v \circ v_i)(\mu_{s_i} - a_{s_i}) 
    = \\ =
    \sum_{i=1}^n r_i \bigl( 
      (v \circ v_i)(\mu_{s_i}) - 
      t_{v \circ v_i,s_i}\bigl( (v \circ v_i)(\mu_{s_i}) 
    \bigr)
    \in \mc{I}(S)(q) \text{.}
  \end{multline*}
  
  Finally consider $\equiv \! \pmod{S}$. It is congruent under 
  addition by definition, and congruent under multiplication by a 
  scalar since $\mc{I}(S)(q)$ is an $\mc{R}$-module; the latter also 
  makes it reflexive, transitive, and symmetric. Being congruent 
  under composition and tensor product is by 
  items~\ref{Item:Sammansattning} and~\ref{Item:Tensorprodukt} of 
  Lemma~\ref{L:KategorinV} merely special cases of being congruent 
  under all \(v \in \mc{V}(J_{m \times n},J_{k \times l})\), and it 
  was shown above that any \(v \in \mc{V}(q,p)\) maps $\mc{I}(S)(p)$ 
  into $\mc{I}(S)(q)$.
\end{proof}

\begin{definition}
  An \(a \in \mc{M}(q)\) is said to be \DefOrd{uniquely reducible} 
  (with respect to $S$ and $q$) if for any \(t_1,t_2 \in T(S)(q)\) 
  such that \(t_1(a),t_2(a) \in \Irr(S)(q)\) it holds that \(t_1(a) = 
  t_2(a)\). An \(a \in \mc{M}(q)\) is said to be \DefOrd{persistently 
  reducible} (with respect to $S$ and $q$) if there for every \(t_1 
  \in T(S)(q)\) exists some \(t_2 \in T(S)(q)\) such that \(t_2\bigl( 
  t_1(a) \bigr) \in \Irr(S)(q)\). The set of all elements of 
  $\mc{M}(q)$ that are uniquely and persistently reducible is denoted 
  $\Red(S)(q)$. Denote by $t^S_q$\index{t S@$t^S$}, or just $t^S$ if 
  $q$ is known from context, the map \(\Red(S)(q) \Fpil \Irr(S)(q)\) 
  with the property that there for every \(b \in \Red(S)(q)\) is some 
  \(t_1 \in T(S)(q)\) such that \(t^S_q(b) = t_1(b)\), i.e., 
  $t^S_q(b)$ is the element of $\Irr(S)(q)$ to which $b$ can be 
  mapped by a reduction \(t_1 \in T(S)(q)\).
\end{definition}

By Lemma~4.7 of~\cite{rpaper}, $\Red(S)(q)$ is a submodule of 
$\mc{M}(q)$ and $t^S_q$ is an $\mc{R}$-linear projection of 
$\Red(S)(q)$ onto $\Irr(S)(q)$ with \(\ker t^S_q \subseteq 
\mc{I}(S)(q)\).

\subsection{Ambiguities}
\label{Ssec:Tvetydigheter}

Clearly, an obstruction to unique reducibility must take the form of 
something that can be reduced in two different ways. That reductions 
are defined to be maps of a specific form immediately allows the 
concept to be slightly specialised.

\begin{definition} \label{Def:Tvetydighet}
  Let \(q \in \B^{\bullet\times\bullet}\) and a rewriting system $S$ 
  for $\mc{R}\{\Omega\}$ be given. A type $q$ \DefOrd{ambiguity} of $S$ 
  is a triplet $(t_1,\mu,t_2)$, where \(t_1,t_2 \in T_1(S)(q)\) act 
  nontrivially on \(\mu \in \mc{Y}(q)\); the ambiguities $(t_1,\mu,t_2)$ 
  and $(t_2,\mu,t_1)$ are considered to be the same. The $\mu$ part 
  is called the \DefOrd{site} of the ambiguity. An ambiguity 
  $(t_1,\mu,t_2)$ is said to be 
  \DefOrd[*{ambiguity!resolvable}]{resolvable} if there exist reductions 
  \(t_3,t_4 \in T(S)(q)\) such that \(t_3\bigl( t_1(\mu) \bigr) = 
  t_4\bigl( t_2(\mu) \bigr)\).
\end{definition}

Every rule $s$ gives rise to a trivial ambiguity 
$(t_{\mathrm{id},s},\mu_s,t_{\mathrm{id},s})$ of type $q_s$. Since 
\(t_1 = t_2 = t_{\mathrm{id},s}\), these ambiguities are also trivially 
resolvable. Not all self-ambiguities are trivial, however.

\begin{example} \label{Ex:Associativitet}
  Suppose \(\mathsf{m} \in \Omega(1,2)\) and consider the rule
  \[
    s = \Bigl(
      J_{1 \times 3},
      \mathsf{m} \circ \phi(\same{1}) \otimes \mathsf{m},
      \mathsf{m} \circ \mathsf{m} \otimes \phi(\same{1})
    \Bigr)
    \text{;}
  \]
  this rule can be used to express associativity (if $\mathsf{m}$ is 
  the symbol for the multiplication operation, then it changes a 
  product $a(bc)$ to $(ab)c$). The rewriting system $\{s\}$ has the 
  type $J_{1 \times 4}$
  ambiguity
  \[
    \Bigl( t_{v_1,s}, 
    \fuse{ 7 }{ 
      \mathsf{m}^7_{16} \mathsf{m}^6_{25} \mathsf{m}^5_{34} 
    }{ 1234 },
    t_{v_2,s} \Bigr)
  \]
  where, keeping in mind that \(\fuse{ 7 }{ \mathsf{m}^7_{16} 
  \mathsf{m}^6_{25} }{ 125 } = \mu_s = \fuse{ 6 }{ \mathsf{m}^6_{25} 
  \mathsf{m}^5_{34} }{ 234 }\),
  \[
    v_1(b) = \fuse{ 7125 }{ \mathsf{m}^5_{34} }{ 12347 } \rtimes b
    \qquad\text{and}\qquad
    v_2(b) = \fuse{ 7234 }{ \mathsf{m}^7_{16} }{ 12346 } \rtimes b
    \text{.}
  \]
  This ambiguity is resolvable in $\{s\}$, but not trivially so; the 
  steps of the resolution make up the ``associativity pentagon''. 
  (Figuring them out is a good exercise; a tip is to employ network 
  notation for these calculations.)
\end{example}

That all ambiguities of a rewriting system are resolvable is sometimes 
called the \emDefOrd{diamond property} (hence `Diamond Lemma'), but 
more commonly \emDefOrd{local confluence}. Ordinary 
\emDefOrd{confluence} is rather that all elements of $\mc{R}\{\Omega\}$ 
are uniquely reducible. Not all ambiguities are resolvable, not even 
self-ambiguities.

\begin{example}
  Suppose \(\mathsf{x},\mathsf{y} \in \Omega(1,1)\) and consider the 
  rule
  \[
    s_1 = \bigl( J_{1 \times 1}, \mathsf{y}\circ\mathsf{y}, 
    \phi(\same{1}) - \mathsf{x}\circ\mathsf{x} \bigr)
    \text{;}
  \]
  this rule can be viewed as encoding the congruence \( 
  \mathsf{x}\circ\mathsf{x} + \mathsf{y}\circ\mathsf{y} \equiv 
  \phi(\same{1}) \pmod{\{s_1\}}\), which is the defining equation of 
  the basic noncommutative circle.\footnote{
    In noncommutative algebraic geometry, a `noncommutative circle' 
    is an associative algebra generated by two generators $x$ and $y$ 
    satisfying an equation whose commutative counterpart is 
    \(x^2 + y^2 = 1\). Clearly \(x^2 + y^2 = 1\) is such an equation, 
    but so is \(x^2 + rxy - ryx + y^2 = 1\) for any value of the 
    scalar $r$, and the resulting algebras are not the same.
  } The rewriting system $\{s_1\}$ has the type $J_{1 \times 1}$ 
  ambiguity
  \[
    (t_{v_1,s_1}, \mathsf{y}\circ\mathsf{y}\circ\mathsf{y}, 
    t_{v_2,s_1})
  \]
  where \(v_1(b) = b \circ \mathsf{y}\) and \(v_2(b) = \mathsf{y} 
  \circ b\). In this case, \(t_{v_1,s_1}( \mathsf{y} \circ\nobreak 
  \mathsf{y} \circ\nobreak \mathsf{y}) = \mathsf{y} - \mathsf{x} 
  \circ \mathsf{x} \circ \mathsf{y}\) and \(t_{v_2,s_1}( 
  \mathsf{y} \circ\nobreak \mathsf{y} \circ\nobreak \mathsf{y}) = 
  \mathsf{y} - \mathsf{y} \circ \mathsf{x} \circ \mathsf{x}\) are 
  both in $\Irr\bigl( \{s_1\} \bigr)(J_{1 \times 1})$, so the 
  ambiguity is not resolvable. The calculation does however 
  demonstrate that \(\mathsf{y} - \mathsf{x} \circ \mathsf{x} \circ 
  \mathsf{y} \equiv \mathsf{y} - \mathsf{y} \circ \mathsf{x} \circ 
  \mathsf{x} \pmod{\{s_1\}}\), or equivalently \(\mathsf{x} \circ 
  \mathsf{x} \circ \mathsf{y} \equiv \mathsf{y} \circ \mathsf{x} 
  \circ \mathsf{x} \pmod{\{s_1\}}\), which suggests it might be a 
  good idea to add the rule \(s_2 = (J_{1 \times 1}, \mathsf{y} 
  \circ\nobreak \mathsf{x} \circ\nobreak \mathsf{x}, \mathsf{x} 
  \circ\nobreak \mathsf{x} \circ\nobreak \mathsf{y})\) to one's 
  rewriting system, as doing so will render this ambiguity resolvable 
  without changing the congruence relation.
\end{example}

It is not too hard to see that every rule is formally involved in 
infinitely many ambiguities, but there are only a few that one 
actually need to consider.

\begin{definition} \label{D:Avskalad}
  Let $\bigl( t_{v_1,s_1}, [G]_\simeq, t_{v_2,s_2} \bigr)$ be an 
  ambiguity. Fix some \(H_1 \in [\mu_{s_1}]_\simeq\) and \(H_2 \in 
  [\mu_{s_2}]_\simeq\). For \(i=1,2\), let $(\chi_i,\psi_i)$ be the 
  embedding of $H_i$ in $G$ that corresponds to $v_i$. Write 
  \((V_G,E_G,h_G,g_G,t_G,s_G,D_G) = G\) and 
  \((V_i,E_i,h_i,g_i,t_i,s_i,D_i) = H_i\) for \(i=1,2\). The 
  ambiguity is said to be \DefOrd{terse} if all of the following 
  hold:
  \begin{enumerate}
    \item \label{I1:Kritisk}
      \(V_G = \{0,1\} \cup \setim\chi_1 \cup \setim\chi_2\).
    \item \label{I1.5:Kritisk}
      \(E_G = \setim\psi_1 \cup \setim\psi_2\).
    \item \label{I2:Kritisk}
      $\setmap{\psi_1}\bigl( \setinv{h_1}(\{0\}) \bigr)$ and 
      $\setmap{\psi_2}\bigl( \setinv{t_2}(\{1\}) \bigr)$ are 
      disjoint. (I.e., no edge of $G$ is simultaneously the image of 
      an output leg of $H_1$ and an input leg of $H_2$.)
    \item \label{I3:Kritisk}
      $\setmap{\psi_1}\bigl( \setinv{t_1}(\{1\}) \bigr)$ and 
      $\setmap{\psi_2}\bigl( \setinv{h_2}(\{0\}) \bigr)$ are 
      disjoint. (Ditto, but input leg of $H_1$ and output leg of 
      $H_2$.)
    \item \label{I4:Kritisk}
      For any \(e,e' \in E_1\) such that \(\psi_1(e) = \psi_1(e') 
      \notin \setim \psi_2\), \(h_1(e)=0\), and \(t_1(e')=1\) it 
      follows that \(e=e'\).
    \item \label{I5:Kritisk}
      For any \(e,e' \in E_2\) are such that \(\psi_2(e) = \psi_2(e') 
      \notin \setim \psi_1\), \(h_2(e)=0\), and \(t_2(e')=1\) it 
      follows that \(e=e'\).
  \end{enumerate}
  A type $q$ ambiguity $(t,\mu,u)$ is said to be a \DefOrd{shadow} of 
  the type $p$ ambiguity $(t',\mu',u')$ if there exists some \(v 
  \in \mc{V}(q,p)\) such that \(\mu = v(\mu')\), \(t(\mu) = v\bigl( 
  t'(\mu') \bigr)\), and \(u(\mu) = v\bigl( u'(\mu') \bigr)\).
\end{definition}

Conditions~\ref{I1:Kritisk} and~\ref{I1.5:Kritisk} alone imply that a 
given pair of rules are only involved in finitely many terse 
ambiguities, which is of great theoretical importance as it 
demonstrates various operations of importance for completion are 
effective. Conditions~\ref{I3:Kritisk}--\ref{I5:Kritisk} (and also 
conditions~\ref{Ia:Kritisk}--\ref{Ic:Kritisk} of Lemma~\ref{L:Kritisk}) 
are rather conditions that serve to eliminate branches from the 
search tree one would traverse when enumerating all relevant 
ambiguities; this is helpful for actual computations. What says 
non-terse ambiguities need not be considered 
is primarily the following lemma.

\begin{lemma} \label{L:Avskalad}
  Let $S$ be a sharp rewriting system for $\mc{R}\{\Omega\}$. Then 
  every ambiguity of $S$ is a shadow of a terse ambiguity of $S$.
\end{lemma}
\begin{proof}
  Let $\bigl( t_{v_1,s_1}, [G]_\simeq, t_{v_2,s_2} \bigr)$ be an 
  arbitrary ambiguity. Fix some \(H_1 \in [\mu_{s_1}]_\simeq\) and 
  \(H_2 \in [\mu_{s_2}]_\simeq\). For \(i=1,2\), let $(\chi_i,\psi_i)$ 
  be the embedding of $H_i$ in $G$ that corresponds to $v_i$. Write 
  \((V_G,E_G,h_G,g_G,t_G,s_G,D_G) = G\) and 
  \((V_i,E_i,h_i,g_i,t_i,s_i,D_i) = H_i\) for \(i=1,2\).
  The idea for this proof is to view each instance in which this 
  ambiguity violates a terseness condition as an opportunity to peel 
  away some part of the ambiguity\Dash explicitly 
  producing some suitable $G'$ between $G$ and the $H_i$, and then 
  using Theorem~\ref{S:Inbaddningsuppdelning} to factorise $v_1$ and 
  $v_2$ at $[G']_\simeq$. Through a finite number of such steps, each 
  producing a new ambiguity which has the previous as a shadow, one 
  arrives at an ambiguity which fulfils all of the conditions. Since 
  the rules are sharp, the type of the ambiguity at $[G']_\simeq$ can 
  always be taken as low as $\Tr(G')$.
  
  Beginning with condition~\ref{I1:Kritisk}, one may 
  notice that if \(u \in V_G \setminus \{0,1\} \setminus \setim\chi_1 
  \setminus \setim\chi_2\) then $(\chi_1,\psi_1)$ and 
  $(\chi_2,\psi_2)$ are embeddings also into
  \begin{align*}
    G' ={}& \bigl( V_G \setminus \{u\}, E_G, h', g', t', s', 
      \restr{D_G}{ V_G \setminus \{0,1,u\} } \bigr)\\
  \intertext{where}
    (h',g')(e) ={}& \begin{cases}
      (h_G,g_G)(e) & \text{if \(h_G(e) \neq u\),}\\
      \bigl( 0, \omega(G) + g_G(e) \bigr) & \text{if \(h_G(e) = u\),}
    \end{cases}\\
    (t',s')(e) ={}& \begin{cases}
      (t_G,s_G)(e) & \text{if \(t_G(e) \neq u\),}\\
      \bigl( 1, \alpha(G) + s_G(e) \bigr) & \text{if \(t_G(e) = u\).}
    \end{cases}
  \end{align*}
  Moreover \([G]_\simeq = \nu \rtimes [G']_\simeq\) for \(\nu = 
  \phi(\cross{\alpha(G)}{\omega(G)}) \otimes D_G(u)\) and the 
  corresponding embedding of $G'$ into $G$ is 
  $(\mathrm{id},\mathrm{id})$, so the conditions for two applications 
  of Theorem~\ref{S:Inbaddningsuppdelning} are fulfilled; it follows 
  that there are \(\nu_1,\nu_2 \in \Nwt(\Omega)\) such that on one 
  hand \([G']_\simeq = \nu_1 \rtimes \mu_{s_1} = \nu_2 \rtimes 
  \mu_{s_2}\), and on the other \(v_1(b) = \nu \rtimes (\nu_1 
  \rtimes\nobreak b)\) and \(v_2(b) = \nu \rtimes (\nu_2 
  \rtimes\nobreak b)\). Hence for \(v_1'(b) = \nu_1 \rtimes b\) and 
  \(v_2'(b) = \nu_2 \rtimes b\) one sees that the ambiguity $\bigl( 
  t_{v_1,s_1}, [G]_\simeq, t_{v_2,s_2} \bigr)$ is a shadow of $\bigl( 
  t_{v_1',s_1}, [G']_\simeq, t_{v_2',s_2} \bigr)$.
  Since $G$ only has finitely many vertices, this step can only be 
  repeated a finite number of times.
  
  Next, if condition~\ref{I1:Kritisk} holds but 
  condition~\ref{I1.5:Kritisk} is violated then any \(e_0 \in E_G 
  \setminus \setim\psi_1 \setminus \setim\psi_2\) must have 
  \(h_G(e_0)=0\) and \(t_G(e_0)=1\), since any other vertex is in 
  $\setim\chi_1$ or $\setim\chi_2$ and any edge incident to such a 
  vertex is therefore in $\setim\psi_1$ or $\setim\psi_2$ 
  respectively. Hence
  \begin{align*}
    G' ={}& \bigl( 
      V_G, E_G \setminus \{e_0\}, h_G, g', t_G, s', D_G \bigr)
  \intertext{is a network, where}
    g'(e) ={}& \begin{cases}
      g_G(e) & \text{if \(h_G(e) \neq 0\) or \(g_G(e) < g_G(e_0)\),}\\
      g_G(e)-1 & \text{if \(h_G(e) = 0\) and \(g_G(e) > g_G(e_0)\),}
    \end{cases}\\
    s'(e) ={}& \begin{cases}
      s_G(e) & \text{if \(t_G(e) \neq 1\) or \(s_G(e) < s_G(e_0)\),}\\
      s_G(e)-1 & \text{if \(t_G(e) = 1\) and \(s_G(e) > s_G(e_0)\)}
    \end{cases}
  \end{align*}
  for all \(e \in E_G \setminus \{e_0\}\). 
  $(\chi_i,\psi_i)$ is an embedding of $H_i$ into $G'$ for \(i=1,2\), 
  and $(\mathrm{id},\mathrm{id})$ is an embedding of $G'$ into $G$, 
  which goes with \([G]_\simeq = \phi(\sigma) \rtimes [G']_\simeq\), 
  where \(\sigma(j) = \omega(G) + j\) for \(j < s_G(e_0)\), 
  \(\sigma\bigl( s_G(e_0) \bigr) = g_G(e_0)\), \(\sigma(j) = 
  \omega(G) + j - 1\) for \(s_G(e_0) < j \leqslant \alpha(G)\), 
  \(\sigma(j) = j - \alpha(G)\) for \(\alpha(G) < j < \alpha(G) + 
  g_G(e_0)\), and \(\sigma(j) = j - \alpha(G) + 1\) for \(\alpha(G) + 
  g_G(e_0) \leqslant j\). Another two applications of 
  Theorem~\ref{S:Inbaddningsuppdelning} then yield \(\nu_1,\nu_2 \in 
  \Nwt(\Omega)\) such that on one hand \([G']_\simeq = \nu_1 \rtimes 
  \mu_{s_1} = \nu_2 \rtimes \mu_{s_2}\), and on the other \(v_1(b) = 
  \phi(\sigma) \rtimes (\nu_1 \rtimes\nobreak b)\) and \(v_2(b) = 
  \phi(\sigma) \rtimes (\nu_2 \rtimes\nobreak b)\). Hence for \(v_1'(b) 
  = \nu_1 \rtimes b\) and \(v_2'(b) = \nu_2 \rtimes b\) one sees that 
  the ambiguity $\bigl( t_{v_1,s_1}, [G]_\simeq, t_{v_2,s_2} \bigr)$ is 
  a shadow of $\bigl( t_{v_1',s_1}, [G']_\simeq, t_{v_2',s_2} 
  \bigr)$. Since $E_G \setminus \setim\psi_1 \setminus \setim\psi_2$ 
  is finite, this step can only be repeated a finite number of times, 
  and it preserves condition~\ref{I1:Kritisk}.
  
  In order to see that the remaining steps likewise can only be 
  repeated a finite number of times, one may observe that they 
  decrease the number of distinct quadruplets $(e,i,f,j)$ such that 
  \(i,j \in \{1,2\}\), \(e \in E_i\), \(f \in E_j\), and \(\psi_i(e) = 
  \psi_j(f)\), which obviously starts out finite. Moreover, these 
  steps preserve conditions~\ref{I1:Kritisk} and~\ref{I1.5:Kritisk}.
  
  When condition~\ref{I2:Kritisk} is violated by the ambiguity, 
  then a smaller ambiguity of which it is a shadow can be exhibited by 
  cutting an edge. Let \(e_1 \in \setmap{\psi_1}\bigl( 
  \setinv{h_1}(\{0\}) \bigr) \cap \setmap{\psi_2}\bigl( 
  \setinv{t_2}(\{1\}) \bigr)\) and let \(e_2 = 1 + \max E_G\); after 
  cutting, $e_1$ will be above the cut and $e_2$Êwill be below it. Let
  \begin{subequations} \label{Eq1:Kritisk}
    \begin{align}
      G' ={}& \bigl( V_G, E_G \cup \{e_2\}, h', g', t', s', D_G \bigr)
    \intertext{where}
      (t',s')(e) ={}& \begin{cases}
        (t_G,s_G)(e) & \text{if \(e \in E_G\),}\\
        \bigl( 1, \alpha(G)+1 \bigr) & \text{if \(e = e_2\),}
      \end{cases}\\
      (h',g')(e) ={}& \begin{cases}
        (h_G,g_G)(e) & \text{if \(e \in E_G \setminus \{e_1\}\),}\\
        \bigl( 0, \omega(G)+1 \bigr)& \text{if \(e=e_1\),}\\
        (h_G,g_G)(e_1) & \text{if \(e = e_2\).}
      \end{cases}
    \end{align}
  \end{subequations}
  Here \([G]_\simeq = \nu \rtimes [G']_\simeq\) for \(\nu = \phi( 
  \cross{\alpha(G)}{\omega(G)} \star\nobreak \same{1}) \), and the 
  corresponding embedding of $G'$ into $G$ maps $e_1$ and $e_2$ both to 
  $e_1$, whereas everything else is mapped injectively to itself. 
  $(\chi_1,\psi_1)$ and $(\chi_2,\psi_2)$ are not embeddings into 
  $G'$ in this case, but the slight modifications $(\chi_1,\psi_1')$ 
  and $(\chi_2,\psi_2')$ defined by
  \begin{align*}
    \psi_1'(e) ={}& \begin{cases}
      \psi_1(e) & \text{if \(\psi_1(e) \neq e_1\) or \(h_1(e) = 0\),}\\
      e_2& \text{otherwise}
    \end{cases}\\
    \psi_2'(e) ={}& \begin{cases}
      e_2& \text{if \(\psi_2(e) = e_1\) and \(t_2(e) = 1\),}\\
      \psi_2(e)& \text{otherwise,}
    \end{cases}
  \end{align*}
  are, and compose with the embedding of $G'$ into $G$ to yield their 
  unprimed counterparts; the idea is to map elements of 
  $\setinv{\psi_1}\bigl( \{e_1\} \bigr)$ to $e_1$ and elements of 
  $\setinv{\psi_2}\bigl( \{e_1\} \bigr)$ to $e_2$, unless forced to 
  do the opposite by an edge endpoint which is an inner vertex. Hence 
  there are by Theorem~\ref{S:Inbaddningsuppdelning} \(\nu_1,\nu_2 
  \in \Nwt(\Omega)\) such that on one hand \([G']_\simeq = \nu_1 
  \rtimes \mu_{s_1} = \nu_2 \rtimes \mu_{s_2}\), and on the other 
  \(v_1(b) = \nu \rtimes (\nu_1 \rtimes\nobreak b)\) and \(v_2(b) = 
  \nu \rtimes (\nu_2 \rtimes\nobreak b)\). Hence for \(v_1'(b) = 
  \nu_1 \rtimes b\) and \(v_2'(b) = \nu_2 \rtimes b\) one sees that 
  the ambiguity $\bigl( t_{v_1,s_1}, [G]_\simeq, t_{v_2,s_2} \bigr)$ 
  is a shadow of $\bigl( t_{v_1',s_1}, [G']_\simeq, t_{v_2',s_2} 
  \bigr)$. At least one output leg of $H_1$ and input leg of $H_2$ 
  form a quadruplet at $e_1$ in $G$ but do not do so in $G'$, 
  since $\psi_2'$ maps the input leg in question to $e_2$ whereas 
  $\psi_1'$ map the output leg in question to $e_1$. 
  For condition~\ref{I3:Kritisk} the argument is the same, 
  only with the roles of $H_1$ and $H_2$ reversed.
  
  If condition~\ref{I4:Kritisk} is violated then the situation is 
  basically the same\Ldash a cut can be made between two segments of 
  an edge\Rdash but one has to proceed a bit differently since the 
  embedding $\psi_1$ does not contain enough information to tell which 
  of $e$ and $e'$ is above the cut. Let $\theta_1$ be the subdivision 
  index map for edges associated with the homeomorphism to $G$ from 
  some $K_1 \rtimes H_1$ such that \(v_1(b) = [K_1]_\simeq \rtimes 
  b\). It may without loss of generality be assumed that 
  \(\theta_1(2e +\nobreak 1) < \theta_1(2e' +\nobreak 1)\), or in 
  words that $e$ is above $e'$, since if 
  it's the other way around then it must also be the case that 
  \(h_1(e') = 0\) and \(t_1(e) = 1\). Therefore one may let \(e_1 = 
  \psi_1(e)\), \(e_2 = 1 + \max E_G\), and again define $G'$ using 
  \eqref{Eq1:Kritisk}. No edge of $H_2$ is mapped to $e_1$, so 
  $(\chi_2,\psi_2)$ is also an embedding of $H_2$ into $G'$. An 
  embedding of $H_1$ into $G'$ is $(\chi_1,\psi_1')$, where
  \[
    \psi_1'(f) = \begin{cases}
      e_2 & \text{if \(\psi_1(f)=e_1\) and 
        \(\theta_1(2f + 1) > \theta_1(2e + 1)\),}\\
      \psi_1(f) & \text{otherwise.}
    \end{cases}
  \]
  Both compose with the usual embedding of $G'$ into $G$ to yield the 
  original embeddings of $H_2$ and $H_1$ respectively into $G$, so 
  by Theorem~\ref{S:Inbaddningsuppdelning} there are \(\nu,\nu_1,\nu_2 
  \in \Nwt(\Omega)\) such that on one hand \([G']_\simeq = \nu_1 
  \rtimes \mu_{s_1} = \nu_2 \rtimes \mu_{s_2}\), and on the other 
  \([G]_\simeq = \nu \rtimes [G']_\simeq\), \(v_1(b) = \nu \rtimes 
  (\nu_1 \rtimes\nobreak b)\) and \(v_2(b) = \nu \rtimes (\nu_2 
  \rtimes\nobreak b)\). Hence for \(v_1'(b) = \nu_1 \rtimes b\) and 
  \(v_2'(b) = \nu_2 \rtimes b\) one sees that the ambiguity $\bigl( 
  t_{v_1,s_1}, [G]_\simeq, t_{v_2,s_2} \bigr)$ is a shadow of $\bigl( 
  t_{v_1',s_1}, [G']_\simeq, t_{v_2',s_2} \bigr)$. $(e,1,e',1)$ was a 
  quadruplet with respect to $G$ but not with respect to $G'$. For 
  condition~\ref{I5:Kritisk} the argument is again the same, only with 
  the roles of $H_1$ and $H_2$ reversed.
\end{proof}

That condition~\ref{I4:Kritisk} above completely excludes edges in 
the image of $\psi_2$ is in part because not having to worry about 
$\psi_2$ simplified the argument. The other reason for stating it that 
way is that instead having \(\psi_1(e) = \psi_1(e') \in \setim\psi_2\) 
would lead to a situation being covered by at least one of 
conditions~\ref{I2:Kritisk}, \ref{I3:Kritisk}, and 
\ref{Ib:Kritisk}~(of Lemma~\ref{L:Kritisk}). The situation for 
condition~\ref{I5:Kritisk} is similar.

Terseness is a condition that is practical to look for, but also one 
whose definition relies entirely on the way that $\Nwt(\Omega)$ is 
constructed. The generic machinery of~\cite{rpaper} is rather stated in 
terms of the shadow relation alone, making use of the fact that this 
is a quasi-order.

\begin{definition} \label{Def:Kritisk}
  Let $A$ and $B$ be ambiguities, and write \(A \geqslant B \pin{Q}\) 
  if $A$ is a shadow of $B$. Then $A$ is a \DefOrd{proper shadow} of 
  $B$ if \(A > B \pin{Q}\), and $B$ is \DefOrd{shadow-minimal} if it 
  is minimal with respect to $Q$. An ambiguity $A$ is said to be 
  \DefOrd{shadow-critical} if there is no shadow-minimal ambiguity 
  $B$ of which $A$ is a proper shadow. Two ambiguities $A$ and $B$ are 
  said to be \DefOrd{shadow-equivalent} if \(A \sim B \pin{Q}\).
  
%
  
  A type $q$ ambiguity $(t_{v_1,s_1},\mu,t_{v_2,s_2})$ of $S$ is said 
  to be a \DefOrd[*{montage ambiguity}]{montage} of the rules 
  \(s_1,s_2 \in S\) if there is some \(v \in \mc{V}(q, q_{s_1} 
  \otimes\nobreak q_{s_2})\) 
  such that \(v_1(b) = v( b \otimes\nobreak \mu_{s_2})\) 
  for all \(b \in \mc{M}(q_{s_1})\) and \(v_2(b) = v( \mu_{s_1} 
  \otimes\nobreak b)\) for all \(b \in \mc{M}(q_{s_2})\).
  
  An ambiguity that is shadow-critical and not a montage is said to 
  be \DefOrd{critical}. An ambiguity that is terse and not a montage is 
  said to be \DefOrd{decisive}.
\end{definition}

The `critical' here is the same as in `critical pair', and `critical 
ambiguity' is pretty much the same as what in rewriting in general is 
captured by the concept of critical pair. Note, however, that even if 
the essence is generally the same, what is formally defined to be a 
`critical pair' can vary considerably between different branches of 
rewriting. The `decisive' concept is another approach to more 
concretely formalise the same essential idea, but there are some 
subtle differences; see the technical discussion in the next 
subsection.

\begin{lemma} \label{L:Kritisk}
  Let $\bigl( t_{v_1,s_1}, [G]_\simeq, t_{v_2,s_2} \bigr)$ be a 
  terse ambiguity. Fix some \(H_1 \in [\mu_{s_1}]_\simeq\) and 
  \(H_2 \in [\mu_{s_2}]_\simeq\). For \(i=1,2\), let $(\chi_i,\psi_i)$ 
  be the embedding of $H_i$ in $G$ that corresponds to $v_i$. Write 
  \((V_G,E_G,h_G,g_G,t_G,s_G,D_G) = G\) and 
  \((V_i,E_i,h_i,g_i,t_i,s_i,D_i) = H_i\) for \(i=1,2\). If the 
  ambiguity is decisive then at least one of the following must hold:
  \begin{enumerate}
    \renewcommand{\theenumi}{\alph{enumi}}
    \renewcommand{\labelenumi}{(\theenumi)}
    \item \label{Ia:Kritisk}
      There is some \(u \in \setim\chi_1 \cap \setim\chi_2\).
    \item \label{Ib:Kritisk}
      There is some \(e_1 \in E_1\) with \(h_1(e_1) = 0\), \(t_1(e_1) 
      = 1\) and some \(e_2 \in E_2\) with \(h_1(e_2) \neq 0\), 
      \(t_2(e_2) \neq 1\) such that \(\psi_1(e_1) = \psi_2(e_2)\).
    \item \label{Ic:Kritisk}
      There is some \(e_1 \in E_1\) with \(h_1(e_1) \neq 0\), 
      \(t_1(e_1) \neq 1\) and some \(e_2 \in E_2\) with \(h_1(e_2) = 
      0\), \(t_2(e_2) = 1\) such that \(\psi_1(e_1) = \psi_2(e_2)\).
  \end{enumerate}
\end{lemma}
\begin{proof}
  For the contrapositive, assume that none of 
  conditions~\ref{Ia:Kritisk}--\ref{Ic:Kritisk} are fulfilled, and 
  seek to show that the ambiguity is a montage. It then follows that 
  not only $\setim\chi_1$ and $\setim\chi_2$ are disjoint (directly by 
  the negation of~\ref{Ia:Kritisk}), but also that $\setim\psi_1$ and 
  $\setim\psi_2$ are disjoint. The main reason for this is that an 
  endpoint of an edge $\psi_1(e_1)$ in $G$ can be covered by at most 
  one of $(\chi_1,\psi_1)$ and $(\chi_2,\psi_2)$ since $\setim\chi_1$ 
  and $\setim\chi_2$ are disjoint. Suppose first \(h_1(e_1) \neq 0\)\Dash 
  then any \(e_2 \in E_2\) with \(\psi_2(e_2) = \psi_1(e_1)\) must have 
  \(h_2(e_2) = 0\) and thus violate terseness condition~\ref{I3:Kritisk} 
  unless \(t_1(e_1) \neq 1\), but that would force \(t_2(e_2) = 1\) and 
  thus fulfil condition~\ref{Ic:Kritisk}. Supposing instead that 
  \(h_1(e_1) = 0\), any \(e_2 \in E_2\) with \(\psi_2(e_2) = \psi_1(e_1)\)
  not violating terseness condition~\ref{I2:Kritisk} must have 
  \(t_2(e_2) \neq 1\) and hence \(t_1(e_1) = 1\)\Dash then \(h_2(e_2) = 
  0\) will violate terseness condition~\ref{I3:Kritisk} and \(h_2(e_2) 
  \neq 0\) will fulfil condition~\ref{Ib:Kritisk}. Thus $\setim\psi_1$ 
  and $\setim\psi_2$ are indeed disjoint. 
  
  This means \(h_G\bigl( \psi_1(e_1) \bigr) = 0\) for all \(e_1 \in 
  \setinv{h_1}\bigl(\{0\}\bigr)\) and \(h_G\bigl( \psi_2(e_2) \bigr) = 
  0\) for all \(e_2 \in \setinv{h_2}\bigl(\{0\}\bigr)\). Hence there is 
  a permutation \(\sigma \in \Sigma_{\omega(G)}\) 
  such that \((\sigma \circ\nobreak g_G \circ\nobreak \psi_1)(e_1) = 
  g_1(e_1)\) for all \(e_1 \in \setinv{h_1}\bigl(\{0\}\bigr)\) and 
  \((\sigma \circ\nobreak g_G \circ\nobreak \psi_2)(e_2) = 
  \omega(H_1) + g_2(e_2)\) for all \(e_2 \in 
  \setinv{h_2}\bigl(\{0\}\bigr)\). Similarly there is a permutation 
  \(\tau \in \Sigma_{\alpha(G)}\) such that \((\tau^{-1} \circ\nobreak 
  s_G \circ\nobreak \psi_1)(e_1) = s_1(e_1)\) for all \(e_1 \in 
  \setinv{t_1}\bigl(\{1\}\bigr)\) and \((\tau^{-1} \circ\nobreak s_G 
  \circ\nobreak \psi_2)(e_2) = \alpha(H_1) + s_2(e_2)\) for all \(e_2 \in 
  \setinv{t_2}\bigl(\{1\}\bigr)\). Hence $\sigma \cdot G \cdot \tau$ has 
  the split $(\setim\psi_1, \setim\psi_2, \setim\chi_1, 
  \setim\chi_2)$, and \([G]_\simeq = \phi(\sigma^{-1}) \circ 
  [H_1]_\simeq \otimes [H_2]_\simeq \circ \phi(\tau^{-1})\). Thus 
  \(v_1(b) = \phi(\sigma^{-1}) \circ b \otimes \mu_{s_2} \circ 
  \phi(\tau^{-1})\), \(v_2(b) = \phi(\sigma^{-1}) \circ \mu_{s_1} 
  \otimes b \circ \phi(\tau^{-1})\), and the ambiguity is a montage.
\end{proof}

The various conditions (\ref{I1:Kritisk}--\ref{I5:Kritisk} and 
\ref{Ia:Kritisk}--\ref{Ic:Kritisk}) introduced above make a good job 
of restricting the set of ambiguities that are considered, but it 
should be pointed out that there is one more source of multiplicity 
that are addressed by neither criticality nor decisiveness: 
permutation of inputs and outputs. It is immediately clear that two 
ambiguities which differ only in the order of their legs are 
shadow-equivalent (by item~\ref{Item:permutationer} of 
Lemma~\ref{L:KategorinV}), but it is very hard to provide a general 
condition that would select just one representative ambiguity for 
each equivalence class. On the other hand, it is in a concrete 
situation typically straightforward to select representatives, so 
that matter is no big deal.

\begin{example}
  Continuing with the reduction system $\{s\}$ considered in  
  Example~\ref{Ex:Associativitet}, one may conclude from 
  condition~\ref{I1:Kritisk} that a terse ambiguity can have at most 
  four inner vertices in its site, and in combination with 
  condition~\ref{Ia:Kritisk} (neither~\ref{Ib:Kritisk} 
  nor~\ref{Ic:Kritisk} can be fulfilled since $\mu_s$ has no stray 
  edges) that bound can be lowered to three inner vertices for sites 
  of decisive ambiguities.
  
  A practical way of finding all decisive ambiguities of $\{s\}$ is 
  then to start with \(H_1 = \Nwfuse{ 7 }{ (2\colon 
  \mathsf{m})^7_{16} (3\colon \mathsf{m})^6_{25} }{ 125 } \in \mu_s\) 
  and consider all the ways this $H_1$ could be extended to some $G$ 
  (i.e., the embedding $(\chi_1,\psi_1)$ has \(\chi_1 = 
  \mathrm{id}\)) such that $(t_{v_1,s}, [G]_\simeq, t_{v_2,s})$ is a 
  decisive ambiguity. Taking also \(H_2 = H_1\), an initial tripartition 
  of the search space is as follows: (i)~it may be that \(2 \notin 
  \setim \chi_2\) (the bottom vertex of $H_1$ is not identified with any 
  vertex of $H_2$), (ii)~it may be that \(2 = \chi_2(2)\) (the bottom 
  vertex of $H_1$ is identified with the bottom vertex of $H_2$), or 
  (iii)~it may be that \(2 = \chi(3)\) (the bottom vertex of $H_1$ is 
  identified with the top vertex of $H_2$). In case~(i) it follows 
  that \(3 \in \setim\chi_2\), but \(\chi_2(3)=3\) would force 
  \(\chi_2(2)=2\), so it must be the case that \(\chi_2(2)=3\) and 
  \(G \simeq \Nwfuse{ 7 }{ (2\colon \mathsf{m})^7_{16} (3\colon 
  \mathsf{m})^6_{25} (4\colon \mathsf{m})^5_{34} }{ \vek{a} }\) for 
  some list $\vek{a}$ with \(\norm{\vek{a}} = \{1,2,3,4\}\). Up to 
  the order of the input legs, this is the ambiguity that was 
  considered in Example~\ref{Ex:Associativitet}, so that has already 
  been dealt with.
  
  In case~(ii), it follows that \(\chi(3)=3\) and \(G \simeq H_1 = H_2\), so 
  this is the trivial ambiguity. In case~(iii), one arrives at \(G 
  \simeq \Nwfuse{ 8 }{ (2\colon \mathsf{m})^7_{16} (3\colon 
  \mathsf{m})^6_{25} (4\colon \mathsf{m})^8_{97} }{ \vek{a} }\) for 
  some list $\vek{a}$ with \(\norm{\vek{a}} = \{9,1,2,5\}\). This is 
  the same ambiguity as in case~(i), with the roles of $H_1$ and 
  $H_2$ exchanged. Hence the ambiguity of 
  Example~\ref{Ex:Associativitet} was (up to symmetry) the only 
  nontrivial decisive ambiguity of $\{s\}$, and thus the only one 
  that requires an explicit resolution.
\end{example}

\subsection{The Diamond Lemma}

The last order of business before the diamond lemma can be stated is 
to link rewrite rules and ambiguities to quasi-orders on 
$\Nwt(\Omega)$.

\begin{definition}
  Let $P$ be a quasi-order on $\Nwt(\Omega)$, \(q \in 
  \B^{\bullet\times\bullet}\), and \(\mu \in \mc{Y}(q)\). Then the 
  type $q$ \DefOrd{down-set module} of $\mu$ with respect to $P$ is 
  the set
  \[
    \DSM(\mu,q,P) = \Span\Bigl( 
    \setOf[\big]{ \nu \in \mc{Y}(q) }{ \nu < \mu \pin{P} }
    \Bigr) \text{.}
  \]
  A rule \(s = (q_s,\mu_s,a_s)\) is said to be \DefOrd{compatible} 
  with $P$ if \(a_s \in \DSM(\mu_s,q_s,P)\). A rewriting system is 
  compatible with $P$ if all rules in it are compatible with $P$.
\end{definition}

\begin{example}
  Suppose \(\Omega = \{\mathsf{m}\}\), where \(\omega(\mathsf{m})=1\) 
  and \(\alpha(\mathsf{m})=2\). Consider the $\N^2$-graded set 
  morphisms \(f_1,f_2\colon \Omega \Fpil \Baff(\N)\) which have 
  \[
    f_1(\mathsf{m}) = \begin{pmatrix}
      1 & 0 & 0 & 0 \\
      0 & 1 & 0 & 0 \\
      0 & 0 & 1 & 2
    \end{pmatrix}
    \quad\text{and}\quad
    f_2(\mathsf{m}) = \begin{pmatrix}
      1 & 0 & 0 & 1 \\
      0 & 1 & 0 & 0 \\
      0 & 0 & 1 & 1
    \end{pmatrix}
  \]
  respectively. Both of these have at least one positive element in 
  each row and column, so the images of $\eval_{f_1}$ and 
  $\eval_{f_2}$ resides in a sub-\PROP\ of $\Baff(\N)$ on which, by 
  Corollaries~\ref{Kor:BiaffinOrdning} and~\ref{Kor:BaffOkapad}, the 
  standard order $Q$ on these matrices is a strict \PROP\ partial 
  order which also has the strict uncut property. Hence the pullbacks 
  to $\Nwt(\Omega)$ of these standard orders are strict \PROP\ 
  quasi-orders with the strict uncut property. They are furthermore 
  well-founded, since $Q$ is well-founded.
  
  Considering the rewrite rule $s$ of 
  Example~\ref{Ex:Associativitet}, which has \(\mu_s = \mathsf{m} 
  \circ \phi(\same{1}) \otimes \mathsf{m}\) and \(a_s = \mathsf{m} 
  \circ \mathsf{m} \otimes \phi(\same{1})\), one finds
  \begin{align*}
    \eval_{f_1}(\mu_s) = \begin{pmatrix}
      1 & 0 & 0 & 0 & 0 \\
      0 & 1 & 0 & 0 & 0 \\
      0 & 0 & 1 & 2 & 4
    \end{pmatrix}
    >{}&
    \begin{pmatrix}
      1 & 0 & 0 & 0 & 0 \\
      0 & 1 & 0 & 0 & 0 \\
      0 & 0 & 1 & 2 & 2
    \end{pmatrix} = \eval_{f_1}(a_s)
    \text{,}\\
    \eval_{f_2}(\mu_s) = \begin{pmatrix}
      1 & 0 & 0 & 1 & 2 \\
      0 & 1 & 0 & 0 & 0 \\
      0 & 0 & 1 & 1 & 1
    \end{pmatrix}
    >{}&
    \begin{pmatrix}
      1 & 0 & 0 & 1 & 1 \\
      0 & 1 & 0 & 0 & 0 \\
      0 & 0 & 1 & 1 & 1
    \end{pmatrix} = \eval_{f_2}(a_s)
    \text{,}
  \end{align*}
  which means $s$ is compatible with both $(\eval_{f_1})^* Q$ and 
  $(\eval_{f_2})^* Q$.
\end{example}

\begin{definition}
   Let $P$ be a quasi-order on $\Nwt(\Omega)$, \(q \in 
  \B^{\bullet\times\bullet}\), and \(\mu \in \mc{Y}(q)\). 
  Also let $S$ be a rewriting system for $\mc{R}\{\Omega\}$. Then the 
  type $q$ \DefOrd{down-set $\mc{I}(S)$ section} of \(\mu \in 
  \mc{Y}(q)\) with respect to $P$ and $S$ is the set
  \begin{multline*}
    \DIS(\mu,q,P,S) = \\ =
    \Span\Bigl( \setOf[\big]{ v(\mu_s - a_s) }{ 
      \text{\(v \in \mc{V}(q,q_s)\), \(s \in S\), and 
      \(v(\mu_s) < \mu \pin{P}\)}
    } \Bigr)
    \text{.}
  \end{multline*}
  A type $q$ ambiguity $(t_1,\mu,t_2)$ is said to be 
  \DefOrd[*{ambiguity!resolvable relative to}]{resolvable relative to} 
  $P$ if \(t_1(\mu) - t_2(\mu) \in \DIS(\mu,q,P,S)\).
\end{definition}

Relative resolvability is a technical condition which simplifies 
various parts of the proof of the diamond lemma, but it is not so 
common that it is easier to verify than ordinary resolvability.

\begin{lemma} \label{L:Korrelation}
  Let $P$ be a well-founded strict \PROP\ quasi-order on $\Nwt(\Omega)$ 
  which has the strict uncut property. Then for all \(q,r \in 
  \B^{\bullet\times\bullet}\), \(v \in \mc{V}(q,r)\), and \(\mu \in 
  \mc{Y}(r)\),
  \begin{equation} \label{Eq:Korrelation}
    \setmap{v}\bigl( \DSM(\mu,r,P) \bigr) \subseteq 
    \DSM\bigl( v(\mu), q, P \bigr) \text{.}
  \end{equation}
\end{lemma}
\begin{proof}
  Any given $v$ must be of the form \(v(b) = \lambda \rtimes b\) for 
  some \(\lambda \in \Nwt(\Omega)\). By Theorem~\ref{S:UncutOrder}, 
  $P$ is strictly preserved under symmetric join. Hence \(\nu < \mu 
  \pin{P}\) implies \(v(\nu) = \lambda \rtimes \nu < \lambda \rtimes 
  \mu = v(\mu) \pin{P}\), so \(v(\nu) \in \DSM\bigl( v(\mu), q, P 
  \bigr)\). \eqref{Eq:Korrelation} now follows from the linearity of 
  $v$ and the fact that a down-set module is a module.
\end{proof}

The above is sufficient for stating the single-sort generic diamond 
lemma~\cite[Theorem~5.11]{rpaper} in the present setting.

\begin{lemma} \label{L:entyps-DL}
  Let $S$ be a rewriting system for $\mc{R}\{\Omega\}$ and $P$ be a 
  well-founded strict \PROP\ quasi-order on $\Nwt(\Omega)$ which has 
  the strict uncut property and with which $S$ is compatible. Then 
  for each \(q \in \B^{\bullet\times\bullet}\), the following are 
  equivalent:
  \begin{enumerate}
    \item[\parenthetic{a}] 
      Every type $q$ ambiguity of $S$ is resolvable.
    \item[\parenthetic{a\textprime}] 
      Every type $q$ ambiguity of $S$ is resolvable relative to $P$.
    \item[\parenthetic{b}]
      Every element of $\mc{M}(q)$ is persistently and uniquely 
      reducible, i.e., \(\Red(S)(q) = \mc{M}(q)\).
    \item[\parenthetic{c}]
      \(\mc{M}(q) = \Irr(S)(q) \oplus \mc{I}(S)(q)\), i.e., every 
      element of $\mc{M}(q)$ has a unique normal form.
  \end{enumerate}
\end{lemma}
\begin{proof}
  The most noticeable difference between this claim and Theorem~5.11 
  of~\cite{rpaper} is probably the topological conditions in the 
  latter, but those are trivially fulfilled in the present setting of 
  discrete topology. A difference of slight technical concern is that 
  this claim allows $P$ to be a quasi-order, whereas the theorem 
  requires it to be a partial order, but since compatibility is only 
  concerned with strict inequalities, one may simply apply that 
  theorem to the partial order $P \diamond E$, where $E$ is the 
  equality relation on $\Nwt(\Omega)$.
  
  The only nontrivial difference is that the theorem requires that
  \(t(\nu) \in \{\nu\} \cup \DSM(\nu,q,P)\) for all \(\nu \in 
  \Nwt(\Omega)(q)\) and \(t \in T_1(S)(q)\), i.e., that the \emph{simple 
  reductions} are compatible with the order $P$, rather than the 
  \emph{rules} as in the present claim. The former is however a direct 
  consequence of Lemma~\ref{L:Korrelation} and the way the simple 
  reductions were constructed: any \(t \in T_1(S)(q)\) is of the form 
  $t_{v,s}$ for some \(s \in S\) and \(v \in \mc{V}(q,q_s)\). Hence 
  if \(\nu \neq v(\mu_s)\) then \(t_{v,s}(\nu) = \nu\), whereas if 
  \(\nu = v(\mu_s)\) then \(t_{v,s}(\nu) = v(a_s) \in 
  \setmap{v}\bigl( \DSM(\mu_s,q_s,P) \bigr) \subseteq 
  \DSM(\nu,q,P)\); either way the condition is fulfilled.
\end{proof}

The reason for this switch to a statement in terms of quasi-orders is 
primarily a build-up of experience. In almost any application of the 
diamond lemma one constructs the order as a lexicographic composition 
of quasi-orders, which does usually not end up automatically being a 
partial order. Hence most applications of the diamond lemma 
of~\cite{rpaper} will contain one extra 
step of composing the constructed quasi-order with the 
equality relation to make it a partial order, and that step might as 
well be inlined into the theorem.

Lemma~\ref{L:entyps-DL} is however not entirely satisfying, for two 
reasons. One is that it only speaks about one transference type at a 
time, whereas arguments involving \PROP\ elements rarely restrict 
themselves to just one type. The second reason, which is the more 
important one, is that conditions \parenthetic{a} and 
\parenthetic{a\textprime} in the lemma ask for (explicit) resolutions 
of \emph{all} ambiguities, rather than just the ones of which all 
others are shadows. Combining the above with Theorem~6.9 
of~\cite{rpaper} does however suffice for producing the following 
theorem.

\begin{theorem}
  [Diamond Lemma for $\mc{R}$-linear \PROPs, criticality version]
  \label{S:PROP-DL}
  Let $S$ be a rewriting system for $\mc{R}\{\Omega\}$ and $P$ be a 
  well-founded strict \PROP\ quasi-order on $\Nwt(\Omega)$ which has 
  the strict uncut property and with which $S$ is compatible. Then the 
  following are equivalent:
  \begin{enumerate}
    \item[\parenthetic{a}]
      All critical ambiguities of $S$ are resolvable.
    \item[\parenthetic{a\textprime}]
      All critical ambiguities of $S$ are resolvable relative to $P$.
    \item[\parenthetic{b}]
      \(\Red(S)(q) = \mc{M}(q)\) for all \(q \in 
      \B^{\bullet\times\bullet}\), i.e., the rewriting system $S$ is 
      terminating and confluent.
    \item[\parenthetic{c}]
      \(\mc{M}(q) = \Irr(S)(q) \oplus \mc{I}(S)(q)\) for all \(q \in 
      \B^{\bullet\times\bullet}\). In particular, 
      \[\mc{R}\{\Omega\}(m,n) = \Irr(S)(J_{m \times n}) \oplus 
      \mc{I}(S)(J_{m \times n})\] for all \(m,n \in \N\).
  \end{enumerate}
\end{theorem}

Finishing at this point would however not be satisfactory either, 
because the typical application of that theorem requires the user to 
enumerate all critical ambiguities, and that turns out to be a matter 
which runs into some unexpected technical subtleties. These 
subtleties are ultimately the reasons that 
Subsection~\ref{Ssec:Tvetydigheter} also defines the analogous 
concepts of `terse' and `decisive', but the matter warrants a 
discussion.

In~\cite{rpaper}, a distinction had to be made between 
`shadow-minimal' and `shadow-critical', because the generality of the 
setting there (topological spaces, no explicit expression structure, 
and not even an explicit concept of rewrite rule) made possible all 
sorts of infinite descending 
sequences. A reasonable expectation is however that once one's set of 
``monomials'' $\mc{Y}$ has been given an explicit discrete structure, 
then any infinite descending chain $\bigl\{ (t_n,\mu_n,u_n) 
\bigr\}_{n=0}^\infty$ of ambiguities (i.e., each $(t_n,\mu_n,u_n)$ is a 
shadow of $(t_{n+1},\mu_{n+1},u_{n+1})$) will eventually stabilise 
(for some $N$, all $(t_n,\mu_n,u_n)$ for \(n \geqslant N\) are 
shadow-equivalent) simply because there cannot be infinitely many 
proper refinements in the corresponding factorisations of the site 
$\mu_n$; this happens for example in the case of a free associative 
algebra, even if~\cite[Example~6.10]{rpaper} takes a more direct 
approach. In such cases, shadow-critical becomes equivalent to 
shadow-minimal, and the shadow-minimal ambiguities are something 
that one may expect to be able to enumerate explicitly.

Unfortunately, that line of reasoning does not apply for \PROPs\ 
(when acted upon by the class of reductions considered here); it 
is perfectly possible to keep factorising forever, even if it is 
probably a bit counter-intuitive to describe what the elements of the 
sequence in the following example do as ``getting smaller''.

\begin{example} \label{Ex:Girth}
  For all \(n \in \N\), let \(\sigma_n \in \Sigma_n\) be the ``shift 
  one step right'' permutation, i.e., \(\sigma_n(i) = i+1\) if \(i < 
  n\) and \(\sigma_n(n) = 1\). For \(n \geqslant 2\), define \(q_n 
  \in \B^{n \times n}\) by \(q_n = \phi(\sigma_n) + \phi(\same{n})\) 
  (note: boolean matrix addition). Also define \(\tau_n \in 
  \Sigma_{2n+1}\) to be the permutation satisfying
  \[
    \phi_{\B^{\bullet\times\bullet}}(\tau_n) = \begin{bmatrix}
      0 & \phi(\same{n-1}) & 0 & 0 \\
      0 & 0 & 0 & \phi(\same{1}) \\
      \phi(\same{n}) & 0 & 0 & 0 \\
      0 & 0 & \phi(\same{1}) & 0
    \end{bmatrix}
  \]
  and define \(v_n(b) = \phi(\tau_n) \rtimes b\) for all \(b \in 
  \mc{M}(q_{n+1})\). 
  Then the following hold for all \(n \geqslant 2\), exhibiting what 
  could be the $\mu$, $q$, and $v$ for an infinite descending chain 
  of ambiguities:
  \begin{itemize}
    \item
      \(\phi(\same{n}) \in \mc{Y}(q_n)\).
    \item
      \(\phi(\same{n}) = \phi(\tau_n) \rtimes \phi(\same{n+1}) = 
      v_n\bigl( \phi(\same{n+1}) \bigr)\).
    \item
      \(v_n \in \mc{V}(q_n,q_{n+1})\).
  \end{itemize}
  However, if \(2 \leqslant n < m\) then there is no \(v \in 
  \mc{V}(q_m,q_n)\) such that \(v\bigl( \phi(\same{n}) \bigr) = 
  \phi(\same{m})\)! This may seem surprising, since \(\lambda = 
  \phi(\same{m-n} \star\nobreak \cross{n}{n})\) obviously satisfies 
  \(\phi(\same{m}) = \lambda \rtimes \phi(\same{n})\), but the catch 
  is that \(b \mapsto \lambda \rtimes b\) is not in 
  $\mc{V}(q_m,q_n)$, and in fact this family of maps turns out to be 
  empty.
  
  In order to see why, it is convenient to think of the boolean 
  matrices $q_n$ as biadjacency matrices of bipartite graphs (say, 
  with vertex sets $\{0,1\} \times [n]$) and view a map \(v \in 
  \mc{V}(q,r)\) as making the graph for $q$ from the graph from $r$. 
  Which are then the operations that such a map can perform? There 
  are four: it can permute vertices within a side of the graph (so 
  the graph is effectively unlabelled bipartite), it can add new 
  vertices (adding isolated vertices requires having suitable elements 
  in the signature $\Omega$, but even if \(\Omega=\varnothing\) one 
  can always add a pair of vertices joined by an edge), it can add new 
  edges (freely, since there are no restrictions on changing a $0$ to a $1$ 
  in the transference type), and it can identify pairs of non-adjacent 
  vertices from opposite partitions (by means of a feedback, combined 
  with suitable permutations). In this last operation, identifying 
  vertices $u_0$ and $u_1$ actually removes these vertices from the 
  graph (since they no longer correspond to legs of the underlying 
  network), but instead every neighbour of $u_1$ gets an edge to 
  every neighbour of $u_0$ to reflect the increased connectivity. 
  Hence the number of vertices can grow as well as shrink, new vertices 
  can be as isolated as one wishes, and edges can be added at will, so 
  why is there no way to make $q_m$ out of $q_n$ when \(m>n\)? 
  Because the graph corresponding to $q_n$ is a cycle of length $2n$, 
  and there is no way for these operations to increase the length of 
  a cycle, although the identification operation can decrease it. 
  Considering in particular the girth (length of shortest cycle) of 
  these graphs, one finds that $q_m$ has girth $2m$, but any 
  graph that can be made from $q_n$ has girth at most \(2n < 2m\).
\end{example}

Of course, it may still turn out that a nice set of ambiguities exist, 
with the property that every other ambiguity is a shadow of one in 
the nice set, even if the shadow relation in general is not 
well-founded. An alternative approach to finding that set is then to 
start with a characterisation of what its elements should look like, 
and hope to establish that the resulting set is sufficient. This is 
the approach that lead to the definition of terse and decisive 
ambiguity, and it leads to the following diamond lemma.

\begin{theorem}
  [Diamond Lemma for $\mc{R}$-linear \PROPs\ and sharp rewriting 
  systems]
  \label{S:PROP-skarp-DL}
  Let $S$ be a sharp rewriting system for $\mc{R}\{\Omega\}$ and $P$ 
  be a well-founded strict \PROP\ quasi-order on $\Nwt(\Omega)$ which 
  has the strict uncut property and with which $S$ is compatible. 
  Then the following are equivalent:
  \begin{enumerate}
    \item[\parenthetic{a}]
      All decisive ambiguities of $S$ are resolvable.
    \item[\parenthetic{a\textprime}]
      All decisive ambiguities of $S$ are resolvable relative to $P$.
    \item[\parenthetic{b}]
      \(\Red(S)(q) = \mc{M}(q)\) for all \(q \in 
      \B^{\bullet\times\bullet}\), i.e., the rewriting system $S$ is 
      terminating and confluent.
    \item[\parenthetic{c}]
      \(\mc{M}(q) = \Irr(S)(q) \oplus \mc{I}(S)(q)\) for all \(q \in 
      \B^{\bullet\times\bullet}\). In particular,
      \[
        \mc{R}\{\Omega\}(m,n) = 
        \Irr(S)(J_{m \times n}) \oplus \mc{I}(S)(J_{m \times n})
      \]
      for all \(m,n \in \N\).
  \end{enumerate}
\end{theorem}
\begin{proof}
  By letting $q$ range over all of $\B^{\bullet\times\bullet}$ in 
  Lemma~\ref{L:entyps-DL}, \parenthetic{b} implies \parenthetic{c}, 
  and \parenthetic{c} implies that all ambiguities of $S$ are 
  resolvable, which of course includes the decisive ones. Hence 
  \parenthetic{c} implies \parenthetic{a}. Moreover, ordinary 
  resolvability implies relative resolvability 
  by~\cite[Lemma~5.10]{rpaper}, so \parenthetic{a} implies 
  \parenthetic{a\textprime}.
  
  The tricky part is, as always, whether relative resolvability of 
  ambiguities implies~\parenthetic{b}. First, it follows 
  from~\cite[Lemma~6.7]{rpaper} that every montage ambiguity is 
  resolvable relative to $P$, and in combination with 
  \parenthetic{a\textprime} this means all terse ambiguities of $S$ 
  are resolvable relative to $P$. By Lemma~\ref{L:Avskalad}, every 
  ambiguity of $S$ is a shadow of a terse ambiguity of $S$, and hence 
  by~\cite[Lemma~6.3]{rpaper}, all ambiguities of $S$ are resolvable 
  relative to $P$. Hence \parenthetic{a\textprime} of 
  Lemma~\ref{L:entyps-DL} holds for every \(q \in 
  \B^{\bullet\times\bullet}\), and thus \parenthetic{b} holds too.
\end{proof}

After verifying resolvability of the ambiguity in 
Example~\ref{Ex:Associativitet}, one may thus use the above theorem 
to exhibit the unique normal forms of arbitrary associative products. 
Indeed, this argument works also for associative products in operads, 
and that special case is common enough to justify an explicit 
corollary. It is however convenient to first make a few definitions.

\begin{definition}
  An $\N^2$-graded set $\mc{P}$ is said to be \DefOrd{operadic} if 
  \(\omega(x)=1\) for all \(x \in \mc{P}\). A 
  \emDefOrd[*{operadic}]{rewrite rule} $s$ is said to be operadic if 
  \(\omega(\mu_s)=1\), and a rewriting system is operadic if all rules 
  in it are operadic. An \emDefOrd[*{operadic}]{ambiguity} is 
  operadic if it has type $J_{1 \times n}$ for some $n$.
\end{definition}

A direct specialisation to the operadic setting is then the following.

\begin{corollary}
  \label{Kor:operad-DL}
  Let $\Omega$ be operadic. Let $S$ be an operadic rewriting system for 
  $\mc{R}\{\Omega\}$ and $P$ be a well-founded strict \PROP\ quasi-order 
  on $\Nwt(\Omega)$ which has the strict uncut property and with which 
  $S$ is compatible. Then the following are equivalent:
  \begin{enumerate}
    \item[\parenthetic{a}]
      All decisive operadic ambiguities of $S$ are resolvable.
    \item[\parenthetic{a\textprime}]
      All decisive operadic ambiguities of $S$ are resolvable 
      relative to $P$.
    \item[\parenthetic{a\textprime\textprime}]
      All decisive ambiguities of $S$ are resolvable relative to $P$.
    \item[\parenthetic{b}]
      \(\Red(S)(J_{1 \times n}) = \mc{R}\{\Omega\}(1,n)\) for all \(n 
      \in \N\).
    \item[\parenthetic{b\textprime}]
      \(\Red(S)(q) = \mc{M}(q)\) for all \(q \in 
      \B^{\bullet\times\bullet}\).
    \item[\parenthetic{c}]
      \(\mc{R}\{\Omega\}(1,n) = \Irr(S)(J_{1 \times n}) \oplus 
      \mc{I}(S)(J_{1 \times n})\) for all \(n \in \N\).
    \item[\parenthetic{c\textprime}]
      \(\mc{M}(q) = \Irr(S)(q) \oplus \mc{I}(S)(q)\) for all \(q \in 
      \B^{\bullet\times\bullet}\).
  \end{enumerate}
\end{corollary}
\begin{remark}
  If taking the quotient $\mc{R}\{\Omega\}\big/ \equiv\! \pmod{S}$ 
  here, one gets not only the (freest) operad generated by $\Omega$ and 
  satisfying the relations expressed by the rules of $S$, but also 
  the \PROP\ that it generates.
\end{remark}
\begin{proof}
  An observation that should be made is that if $\Omega$ is operadic, 
  then every input leg of every \(G \in \Nw(\Omega)\) is in a unique 
  path from $1$ to $0$, so $\Tr(G)$ has a $1$ in each column. For an 
  operadic rule $s$ this implies that \(\Tr(\mu_s) = J_{1 \times 
  \alpha(\mu_s)}\), and thus the rewriting system $S$ is sharp. This 
  means the conditions of Theorem~\ref{S:PROP-skarp-DL} are 
  fulfilled, and thus \parenthetic{a\textprime\textprime}, 
  \parenthetic{b\textprime}, and \parenthetic{c\textprime} above are 
  known to be equivalent. Moreover \parenthetic{b\textprime} implies 
  \parenthetic{b} and \parenthetic{c\textprime} implies 
  \parenthetic{c}. By letting $q$ in Lemma~\ref{L:entyps-DL} range over 
  $J_{1 \times n}$ for \(n \in \N\), one also gets that 
  \parenthetic{b} implies \parenthetic{c}, and \parenthetic{c} 
  implies that all operadic ambiguities of $S$ are resolvable, which 
  in particular implies the above \parenthetic{a}. Hence everything 
  again boils down to the ambiguities.
  
  As in the proof of Theorem~\ref{S:PROP-skarp-DL}, it follows 
  from~\cite[Lemma~5.10]{rpaper} that ordinary resolvability implies 
  relative resolvability, so \parenthetic{a} implies 
  \parenthetic{a\textprime}. Furthermore \parenthetic{a\textprime} 
  and \parenthetic{a\textprime\textprime} are the same, because every 
  decisive ambiguity of $S$ must be operadic: By definition every 
  edge and inner vertex of the site of a terse ambiguity must 
  correspond to an edge and inner vertex respectively of at least one 
  of the left hand sides of the two rules. These rules are operadic, 
  so their left 
  hand sides are trees and thus connected. By Lemma~\ref{L:Kritisk} 
  they must have at least one edge or inner vertex in common, so the 
  ambiguity site is connected. Taking into account that $\Omega$ is 
  operadic, it thus follows that the site has exactly one output leg, 
  and thus the ambiguity is operadic.
\end{proof}

It is possible to continue along this line of specialisation, and 
restrict to arity $0$ (i.e., consider only $\mc{R}\{\Omega\}(1,0)$) to 
get a diamond lemma for $\mc{R}$-linear $\Omega$-algebras, but both 
that and the operad case may be better served by something shown 
directly using the machinery in~\cite{rpaper}. In particular, it may be 
awkward to have to construct the order $P$ for all components of the 
\PROP, if one only wishes to use it for elements of coarity $1$. The 
entire construction of $\Nwt(\Omega)$ is also something that can be 
simplified if only the $(1,n)$ components are of interest. 
Theorem~\ref{S:PROP-skarp-DL} is primarily intended for problems that 
actually exercise the generality of the network concept. One example of 
that would be the rewriting system of \eqref{Eq:SporadiskaRegler} and 
\eqref{Eq:Regelfamiljer}, but that is a matter for another paper.

Still, this theorem is not quite satisfactory as ``\emph{the} Diamond 
Lemma for $\mc{R}$-linear \PROPs,'' and the reason for this is the 
little word `sharp' that appears in it. This is needed because of 
Lemma~\ref{L:Avskalad}, but why is it needed there? The arguments 
involving Theorem~\ref{S:Inbaddningsuppdelning} only make use of the 
networks themselves, so why would the transference type of the 
underlying rules make a difference? Well, Example~\ref{Ex:Girth} 
should have made clear that the transference type can be 
\emph{precisely} what prevents something from being a shadow. 
Analysing exactly how the method of Lemma~\ref{L:Avskalad} can fail 
is however instructive.

\begin{example} \label{Ex:Brygga}
  Let \(\eta \in \Omega(1,0)\), \(\ve \in \Omega(0,1)\), and consider 
  the rule \(s = \bigl( J_{1 \times 1}, \eta \circ \ve, 
  \phi(\same{1}) \bigr)\). This rule is rather artificial, because 
  its consequences are draconic:
  \begin{multline*}
    \kappa \circ \lambda
    \equiv
    \kappa \circ (\eta \circ \ve) \circ \lambda
    =
    (\kappa \circ \eta) \otimes \phi(\same{0}) \circ
    \phi(\same{0}) \otimes (\ve \circ \lambda)
    = \\ =
    \phi(\same{1}) \otimes \ve \circ
    \kappa \otimes \lambda \circ
    \eta \otimes \phi(\same{1})
    \equiv
    (\eta \circ \ve) \otimes \ve \circ
    \kappa \otimes \lambda \circ
    \eta \otimes (\eta \circ \ve)
    = \displaybreak[0]\\ =
    \ve \otimes (\eta \circ \ve) \circ
    \kappa \otimes \lambda \circ
    (\eta \circ \ve) \otimes \eta
    \equiv
    \ve \otimes \phi(\same{1}) \circ
    \kappa \otimes \lambda \circ
    \phi(\same{1}) \otimes \eta
    = \\ =
    \bigl( \ve \circ \kappa \circ \phi(\same{1}) \bigr) \otimes 
    \bigl( \phi(\same{1}) \circ \lambda \circ \eta \bigr)
    =
    \lambda \circ (\eta \circ \ve) \circ \kappa
    \equiv
    \lambda \circ \kappa
    \pmod{\{s\}}
  \end{multline*}
  for all \(\kappa,\lambda \in \Nwt(\Omega)(1,1)\). On the other 
  hand this rule is at least simple, and will thus to some degree 
  avoid cluttering the example with irrelevant details.
  
  The transference type of $s$ cannot be $0$ since \(a_s = 
  \phi(\same{1}) \notin \mc{M}(0)\), but \(\Tr(\mu_s) = 0\) so it 
  is not sharp, and hence the conclusion of Lemma~\ref{L:Avskalad} 
  need not hold for the rewrite system $\{s\}$. But how does the 
  proof fail?
  Let \(\kappa \in \Nwt(\Omega)(1,1) \setminus \mc{Y}(0)\) be arbitrary 
  and consider the type $0$ ambiguity $(t_{v_1,s}, \mu, t_{v_2,s} )$, 
  where \(\mu = \eta \circ \ve \circ \kappa \circ \eta \circ \ve\), 
  \(v_1(b) = b \circ \kappa \circ \eta \circ \ve\), and \(v_2(b) = 
  \eta \circ \ve \circ \kappa \circ b\). The site is mapped to 
  $\kappa \circ \eta \circ \ve$ by the first reduction and to $\eta 
  \circ \ve \circ \kappa$ by the second.
  
  The method Lemma~\ref{L:Avskalad} would use to simplify this ambiguity 
  is to cut out the middle $\kappa$ part, resulting in the network 
  \(\mu' = \eta \otimes \eta \circ \ve \otimes \ve\) from which the 
  previous is recovered as \(\mu = \bigl( \phi(\cross{1}{1}) 
  \otimes\nobreak \kappa \bigr) \rtimes \mu'\). Perhaps somewhat 
  unexpectedly, the new ambiguity $(t_{v_1',s}, \mu', t_{v_2',s})$ 
  has \(v_1'(b) = \ve \otimes b \otimes \eta\) and \(v_2'(b) = \eta 
  \otimes b \otimes \ve\)\Ldash the first reduction acts on the first 
  $\eta$ and second $\ve$ of $\mu'$, whereas the second acts on the 
  second $\eta$ and first $\ve$\Rdash but a moment of thought makes 
  it clear that one cannot have one of them act simultaneously on the 
  second $\eta$ and second $\ve$, since that would produce a cycle. 
  Another way of arguing this point is that the type of the $\mu'$ 
  ambiguity needs to have a $0$ in position $(2,2)$ for the 
  annexation \(\mu = \bigl( \phi(\cross{1}{1}) \otimes\nobreak \kappa 
  \bigr) \rtimes \mu'\) to be allowed, and that prevents rule $s$ 
  from simultaneously applying to those two vertices since its 
  transference type is too high. But what \emph{is} the type of 
  the $\mu'$ ambiguity? This turns out to be the crux of the matter.
  
  For the rule $s$ to apply in both cases, the type has to be at 
  least \(\left( \begin{smallmatrix} 0 & 1 \\ 1 & 0 \end{smallmatrix} 
  \right) = \phi(\cross{1}{1})\). Unfortunately this contradicts the 
  claim that it should have the $\mu$ ambiguity as a shadow, because 
  \[
    \Tr\bigl( \phi(\cross{1}{1}) \otimes\nobreak \kappa \bigr)
    =
    \begin{pmatrix}
      0 & 1 & 0 \\
      1 & 0 & 0 \\
      0 & 0 & 1
    \end{pmatrix}
    =
    \begin{bmatrix}
      \begin{pmatrix} 0 \end{pmatrix} & 
      \begin{pmatrix} 1 & 0 \end{pmatrix} \\
      \begin{pmatrix} 1 \\ 0 \end{pmatrix} & 
      \begin{pmatrix} 0 & 0 \\ 0 & 1 \end{pmatrix}
    \end{bmatrix}
    \text{,}
  \]
  so
  \[
    (0) + 
    \begin{pmatrix} 1 & 0 \end{pmatrix} 
    \begin{pmatrix} 0 & 1 \\ 1 & 0 \end{pmatrix}
    \biggl( 
      \begin{pmatrix} 0 & 0 \\ 0 & 1 \end{pmatrix}
      \begin{pmatrix} 0 & 1 \\ 1 & 0 \end{pmatrix}
    \biggr)^*
    \begin{pmatrix} 1 \\ 0 \end{pmatrix}
    =
    (1)
  \]
  is the minimal type of a $\mu$ ambiguity that is a shadow of the
  $\mu'$ one, whereas it was stated above that the $\mu$ ambiguity 
  under consideration here has type $0$. This is possible since 
  \(\Tr(\eta \circ\nobreak \ve) = 0\) and thus \(\Tr(\kappa \circ\nobreak 
  \eta \circ\nobreak \ve) = 0 = \Tr(\eta \circ\nobreak \ve \circ\nobreak 
  \kappa)\), even though \(q_s \neq 0\). Thus the ambiguity at $\mu'$ 
  indeed has an ambiguity at $\mu$ as shadow\Dash only it was not the 
  one for which a terser form was sought, because it has the wrong 
  transference type.
\end{example}
  
Another way of looking at this is that when determining the minimal 
type of an ambiguity $(t_{v_1,s_1},\mu,t_{v_2,s_2})$, there are three 
kinds of dependencies that must be considered: vertices in $\mu$, 
entries that are $1$ in $q_{s_1}$, and entries that are $1$ in 
$q_{s_2}$; these may be viewed as 
three different kinds of edges in an abstract dependency graph (the 
vertices of which are roughly the edges of $\mu$), and each directed 
path between two legs of this dependency graph requires a position in 
the type of the ambiguity to be $1$. However, not all paths need to 
be considered; for both rules to apply, it is necessary that all 
paths of $\mu$ and $q_{s_1}$ edges are cared for, and also necessary 
that all paths of $\mu$ and $q_{s_2}$ edges are cared for, whereas it 
is not necessary to consider paths mixing $q_{s_1}$ and $q_{s_2}$ 
edges. If, on the other hand, one works out the conditions for 
$(t_{v_1,s_1},\mu,t_{v_2,s_2})$ to be a shadow of some 
$(t_{v_1',s_1},\mu',t_{v_2',s_2})$ then it is necessary to also 
consider paths made up of all three types of edge; a position that is 
$1$ in the type of the $\mu'$ ambiguity retains no memory of whether 
it is there because of a $q_{s_1}$ or $q_{s_2}$ path, so when $\mu'$ 
is annexed by another network to form $\mu$ then there can arise 
dependency paths mixing $q_{s_1}$ and $q_{s_2}$ dependencies. This is 
precisely what happens in the example above. It does not happen in 
sharp rewriting systems, because in that case the $q_{s_1}$ and 
$q_{s_2}$ dependencies only duplicate what is already in $\mu$ and 
thus common anyway.

So is this the fault of the symmetric joins? In part, but not in the 
sense that one would be better off with convex subexpressions; the 
rule considered in Example~\ref{Ex:Brygga} would have just as 
dramatic consequences in a more traditional formalism with convex 
subexpressions, and if there is any difference it would be that even 
more ambiguities are required to capture them all. This example is 
rather a sign that the transference types are too coarse; keeping 
track of enough of the dependencies to cut out the $\kappa$ above 
requires finer control of the internal structures of the networks. 
Refining the rewriting machinery to provide such control is however 
a problem that will have to be addressed in future research. Also, it 
would be a mistake to expect that these problems only turn up in 
artificial examples such as the above example.

\begin{example}
  Assume \(\mOp \in \Omega(1,2)\) and \(\Delta \in \Omega(2,1)\), and 
  consider the two rules
  \begin{align*}
    s_1 ={}& \bigl( J_{2 \times 2}, 
      \phi(\same{1}) \otimes \mOp \circ \Delta \otimes \phi(\same{1}),
      \Delta \circ \mOp \bigr) \text{,}\\
    s_2 ={}& \bigl( J_{2 \times 2}, 
      \mOp \otimes \phi(\same{1}) \circ \phi(\same{1}) \otimes \Delta,
      \Delta \circ \mOp \bigr) \text{.}
  \end{align*}
  These encode the primary identities satisfied by the product $\mOp$ 
  and coproduct $\Delta$ of a Frobenius algebra (see 
  e.g.~\cite{Lauda}), namely that
  \[
    \phi(\same{1}) \otimes \mOp \circ \Delta \otimes \phi(\same{1})
    \equiv
    \Delta \circ \mOp
    \equiv
    \mOp \otimes \phi(\same{1}) \circ \phi(\same{1}) \otimes \Delta
    \text{;}
  \]
  reducing to the middle expression here seems natural, since picking 
  any of the others as the thing to reduce to would introduce an 
  unintuitive left--right asymmetry. It does however render the 
  rewriting system non-sharp.
  
  An interesting non-decisive ambiguity is that at $\mOp \otimes \mOp 
  \circ \phi(\same{1} \star\nobreak \cross{1}{1} \star\nobreak 
  \same{1}) \circ \Delta \otimes \Delta$, which in network notation 
  can be depicted as follows:
  \[
    \left[ \begin{mpgraphics*}{-102}
      PROPdiagramnoarrow(0,0)
        same(2) circ
        box(1,2)(btex $\mOp$\strut etex) otimes same(1) circ
        same(1) otimes cross(1,1) circ
        ( 
          box(2,1)(btex $\Delta$\strut etex) circ
          same(1) circ
          box(1,2)(btex $\mOp$\strut etex)
        ) otimes same(1) circ
        same(1) otimes cross(1,1) circ
        same(1) otimes box(2,1)(btex $\Delta$\strut etex) 
        circ same(2)
      ;
    \end{mpgraphics*} \right]
    \leftarrow
    \left[ \begin{mpgraphics*}{-101}
      PROPdiagramnoarrow(0,0)
        same(2) circ
        box(1,2)(btex $\mOp$\strut etex) otimes same(1) circ
        same(1) otimes cross(1,1) circ
        frame( 
          same(1) otimes box(1,2)(btex $\mOp$\strut etex) circ
          same(3) circ
          box(2,1)(btex $\Delta$\strut etex) otimes same(1)
        ) otimes same(1) circ
        same(1) otimes cross(1,1) circ
        same(1) otimes box(2,1)(btex $\Delta$\strut etex) 
        circ same(2)
      ;
    \end{mpgraphics*} \right]
    =
    \left[ \begin{mpgraphics*}{-100}
      PROPdiagramnoarrow(0,0)
        same(2) circ
        box(1,2)(btex $\mOp$\strut etex) otimes 
          box(1,2)(btex $\mOp$\strut etex) 
        circ
        same(1) otimes cross(1,1) otimes same(1)
        circ
        box(2,1)(btex $\Delta$\strut etex) otimes
          box(2,1)(btex $\Delta$\strut etex) 
        circ same(2)
      ;
    \end{mpgraphics*} \right]
    =
    \left[ \begin{mpgraphics*}{-103}
      PROPdiagramnoarrow(0,0)
        same(2) circ
        same(1) otimes box(1,2)(btex $\mOp$\strut etex) circ
        same(1) otimes cross(1,1) circ
        frame(
          box(1,2)(btex $\mOp$\strut etex) otimes same(1) circ
          same(1) otimes box(2,1)(btex $\Delta$\strut etex) 
        ) otimes same(1) circ
        same(1) otimes cross(1,1) circ
        box(2,1)(btex $\Delta$\strut etex) otimes same(1)
        circ same(2)
      ;
    \end{mpgraphics*} \right]
    \rightarrow
    \left[ \begin{mpgraphics*}{-104}
      PROPdiagramnoarrow(0,0)
        same(2) circ
        same(1) otimes box(1,2)(btex $\mOp$\strut etex) circ
        same(1) otimes cross(1,1) circ
        (
          box(2,1)(btex $\Delta$\strut etex) circ
          box(1,2)(btex $\mOp$\strut etex) 
        ) otimes same(1) circ
        same(1) otimes cross(1,1) circ
        box(2,1)(btex $\Delta$\strut etex) otimes same(1)
        circ same(2)
      ;
    \end{mpgraphics*} \right]
    \text{.}
  \]
  Clearly, the two reductions produce different results, but both are 
  irreducible, so this is a non-trivial congruence that follows from 
  the rewriting system. On the other hand, there is no overlap of the 
  rules in the site of the ambiguity, so it is easy to believe that 
  this ambiguity should be a mere montage. That it evidently is not 
  (a montage should be resolvable) signals that \PROPs\ may exhibit 
  ambiguities that do not belong to either of the traditional classes 
  overlap and inclusion ambiguity. A suitable name for this kind of 
  ambiguity might be \emph{wrap ambiguity}, since it is the way 
  that the left hand sides wrap over each other that makes it 
  necessary to take this ambiguity into account.
\end{example}

\section*{Acknowledgements}

This text represents a milestone in a line of research that started 
in the spring of 2004, during my postdoc stay at the Mittag-Leffler 
institute as part of the NOG (Noncommutative Geometry) programme, 
funded by the European Science Foundation and the Royal Swedish 
Academy of Sciences. Some important suggestions were given by 
professors Jean-Louis Loday and Michel Van den Berg.
Many of the ideas presented herein were already clear back then but 
have taken this long to be properly worked out, whereas others have 
emerged at a later date. 

The most important later break-through is probably that of viewing 
rules as being equipped with a transference type: this occurred in 
the autumn of 2006, while the author attended the second Algebra, 
Geometry, and Mathematical Physics (AGMF) conference and stayed for a 
couple of weeks at Lund University. Some indirect financial support 
for this came from the STINT Foundation, the Crafoord Foundation, and 
the Swedish Royal Academy of Sciences by way of grants from these to 
Professor Sergei Silvestrov. Subsequent visits to Lund in more recent 
years have also helped further this work.

Another event in 2006 for which the author received some financial 
support was for a month at the Special Semester on Groebner Bases 
organised by RICAM, Austrian Academy of Sciences, and RISC, Johannes 
Kepler University, Linz, Austria; this also saw the first external 
presentation (in which \PROPs\ and networks were the focus) of this 
theory. Some work on \PROP\ orders were carried out at that time (but 
not much of that has made it into this text, as a simpler solution of 
the problems studied in terms of the biaffine \PROP\ were obtained in 
late 2008).

In 2010 and 2011, the author has received support from the G.~S.\@ 
Magnuson Foundation, managed by the Royal Swedish Academy of Sciences.

\expandafter\def \expandafter\theindex \expandafter{\theindex
   \addcontentsline{toc}{section}{Index}%
}

\printindex

\end{document}